\documentclass[a4paper,english,12pt]{smfbook}
\usepackage[latin1]{inputenc}
\usepackage[bottom]{footmisc}
\usepackage{mathrsfs}
\usepackage{amssymb}
\usepackage{amsmath}
\usepackage{amsthm}
\usepackage{commath}
\usepackage{mathtools}
\usepackage{amsfonts}
\usepackage{tikz}
\usepackage{bm}
\usepackage{accents}
\usepackage{tikz-cd}
\usetikzlibrary{matrix}
\usepackage{stmaryrd}
\usepackage{hyperref}
\usepackage[shortlabels]{enumitem}

\newcommand{\PP}{\mathbb{P}}

\newcommand{\QQ}{\mathbb{Q}}
\newcommand{\ddd}{\mathbf d}

\newcommand{\prodjn}{\prod _{j=1} ^n}
\newcommand{\gk}{g_K}

\newcommand {\Gm}{\mathbb{G}_m}

\newcommand {\Gmn}{\mathbb{G}_m ^{n}} 

\newcommand{\GG}{\mathcal G}

\newcommand{\RR}{\mathbb{R}}
\newcommand{\Rgz}{\mathbb{R} _{>0}}

\newcommand{\CC}{\mathbb{C}}

\newcommand{\ZZ}{\mathbb{Z}}

\newcommand{\DD}{\mathcal{D}}

\newcommand{\Dav}{\mathcal {D} ^{\mathbf a} _v}

\newcommand{\doots}{,\dots ,}

\newcommand{\AAF}{\mathbb{A} _F}
\newcommand{\AAFt}{\mathbb{A} _F^\times}
\newcommand{\AAFj}{\mathbb {A}_F^1}

\newcommand{\AAFtn}{(\mathbb A_F^\times)^n}
\newcommand{\AAFta}{(\mathbb A_F^\times)_\aaa}

\newcommand{\acZ}{\accentset{\circ}Z}
\newcommand{\AK}{\mathfrak A_K}

\newcommand{\OO}{\mathcal O}
\newcommand{\OFS}{\mathcal O_{F,S}}

\newcommand{\FFF}{\mathcal F}

\newcommand{\Ovt}{ \mathcal O_v ^\times}
\newcommand{\Ovtn}{(\mathcal O_v^\times) ^{n}}

\newcommand{\Ft}{F ^\times}
\newcommand{\Ftn}{F^{\times n}}

\newcommand{\Fv}{F_v}
\newcommand{\Fvt}{F_v^\times}
\newcommand{\Fvn}{F^n_v}
\newcommand{\Fvtn}{(F_v^{\times})^n}
\newcommand{\Fvnz}{F^n_v-\{0\}}

\newcommand{\Fvj}{F_{v,1}}
\newcommand{\fvt}{f_v^\#}

\newcommand{\Hvt}{H_v^\#}

\newcommand{\qav}{q^\aaa _v}

\newcommand{\OOF}{\mathcal O _F}
\newcommand{\Ov}{\mathcal O _v}

\newcommand{\Ovn}{\mathcal O _v ^{\times n}}
\newcommand{\oPPa}{\overline{\PPP(\aaa)}}

\newcommand{\vMF}{v\in M_F}
\newcommand{\vMFz}{v\in M_F ^0}
\newcommand{\vMFi}{v\in M_F ^\infty}

\newcommand{\omI}{\overline{\mathfrak I}}
\newcommand{\vMFC}{v\in M_F^{\CC}}
\newcommand{\vMFR}{v\in M_F^{\RR}}

\newcommand{\XXX}{\mathscr X}
\newcommand{\YYY}{\mathscr Y}
\newcommand{\ZZZ}{\mathscr Z}

\newcommand{\sss}{\mathbf{s}}

\newcommand{\xxx}{\mathbf{x}}
\newcommand{\rrr}{\mathbf{r}}
\newcommand{\llll}{\boldsymbol{\ell}}

\newcommand{\wx}{\widetilde\xxx}
\newcommand{\wy}{\widetilde\yyy}

\newcommand{\kav}{k^{\aaa}_v}

\newcommand{\aaa}{\mathbf{a}}

\newcommand{\AAA}{\mathbb A}

\newcommand{\AAnz}{\mathbb A^n-\{0\}}

\newcommand{\bbb}{\mathbf{b}}
\newcommand{\yyy}{\mathbf{y}}
\newcommand{\zzz}{\mathbf{z}}

\newcommand{\www}{\mathbf{w}}

\newcommand{\KiRn}{\mathcal K+i\RR^n}
\newcommand{\chij}{\chi ^{(j)}}

\newcommand{\chiv}{\chi_v}
\newcommand{\chivk}{\chi_v^{(k)}}
\newcommand{\uuu}{\mathbf{u}}

\newcommand{\jed}{\mathbf {1}}
\newcommand{\mmm}{\mathbf{m}}

\newcommand{\wH}{\widehat H}

\newcommand{\PPP}{\mathscr P}

\newcommand{\TT}{\mathscr T}
\newcommand{\TTT}{\mathscr T}
\newcommand{\TTa}{\mathscr T ({\mathbf{a}})}

\newcommand{\TTd}{\mathscr T(d)}
\newcommand{\TTm}{\mathscr T(m)}

\newcommand{\piv}{\pi_v}
\newcommand{\pivv}{|\pi _v|_v}
\def\card #1{\mathopen| #1 \mathclose|}
\DeclareMathOperator{\an}{an}
\DeclareMathOperator{\cond}{cond}
\DeclareMathOperator{\coun}{count}
\DeclareMathOperator{\disc}{disc}

\DeclareMathOperator{\Res}{Res}
\DeclareMathOperator{\Pic}{Pic}
\DeclareMathOperator{\wPic}{\widehat{\Pic}}

\DeclareMathOperator{\Tam}{Tam}
\DeclareMathOperator{\coker}{coker}
\DeclareMathOperator{\Hom}{Hom}

\DeclareMathOperator{\Gal}{Gal}

\DeclareMathOperator{\Sing}{Sing}

\DeclareMathOperator{\Div}{Div}

\DeclareMathOperator{\Imm}{Im}
\DeclareMathOperator{\Ideal}{Ideal}
\DeclareMathOperator{\discrete}{discrete}
\DeclareMathOperator{\Jac}{Jac}
\DeclareMathOperator{\lcm}{lcm}
\DeclareMathOperator{\Id}{Id}
\DeclareMathOperator{\prim}{prim}

\DeclareMathOperator{\supp}{supp}

\DeclareMathOperator{\Spec}{Spec}

\DeclareMathOperator{\fppf}{fppf}
\DeclareMathOperator{\fin}{fin}

\DeclareMathOperator{\rk}{rk}
\DeclareMathOperator{\Conrad}{Conrad}

\DeclareMathOperator{\Cl}{Cl}

\DeclareMathOperator{\Tr}{Tr}
\DeclareMathOperator{\Reg}{Reg}

\def\no{n\textsuperscript{0}\,}
\def\sPic{\mathop{\mathscr P\mathit{ic}}}

\def\thesection{\thechapter.\arabic{section}}

\newtheorem{mydef}[equation]{Definition}
\newtheorem{lem}[equation]{Lemma}
\newtheorem{thm}[equation]{Theorem}

\newtheorem{conj}[equation]{Conjecture}

\newtheorem{prop}[equation]{Proposition}

\newtheorem{rem}[equation]{Remark}

\newtheorem{cor}[equation]{Corollary}

\newtheorem{exam}[equation]{Example}

\newtheorem*{theorem*}{Theorem}

\numberwithin{equation}{subsection}
\usepackage{blindtext}

  \DeclareFontFamily{U}{wncy}{}
    \DeclareFontShape{U}{wncy}{m}{n}{<->wncyr10}{}
    \DeclareSymbolFont{mcy}{U}{wncy}{m}{n}
    \DeclareMathSymbol{\Sh}{\mathord}{mcy}{"58} 

\begin{document}
\title{Rational points of bounded height on weighted projective stacks}
\date{\huge \textbf{PhD thesis, submitted}}
\author{Ratko Darda}
\address{Université de Paris, Sorbonne Université, CNRS, Institut de Mathématiques de Jussieu-Paris Rive Gauche, IMJ-PRG, F-75013, Paris, France}
\email{ratko.darda@imj-prg.fr}
\begin{abstract}
A weighted projective stack is a stacky quotient $\PPP(\aaa)=(\AAA^n-\{0\})/\Gm$, where the action of~$\Gm$ is with weights $\aaa\in\ZZ^n_{>0}$. Examples are: the compactified moduli stack of elliptic curves $\PPP(4,6)$ and the classifying stack of $\mu_m$-torsors $B\mu_m=\PPP(m)$. We define heights on the weighted projective stacks. 
The heights generalize the naive height of an elliptic curve and the absolute discriminant of a torsor. 

We use the heights to count rational points. We find the asymptotic behaviour for the number of rational points of bounded heights. 
\end{abstract}

\begin{altabstract}
Un champ projectif à poids est un quotient champêtre $\PPP(\aaa)=(\AAA^n-\{0\})/\Gm$, où l'action de~$\Gm$ est avec des poids $\aaa\in\ZZ^n_{>0}$. Des examples sont: le champ compactifié de modules de courbes elliptiques $\PPP(4,6)$ et le champ classifiant de $\mu_m$-torseurs $B\mu_m=\PPP(m)$. Nous définissons des hauteurs sur ces champs. Les hauteurs généralisent la hauteur na\"ive d'une courbe et le discriminant absolu d'un torseur. 

Nous utilisons les hauteurs pour compter des points rationnels. Nous trouvons le comportement asymptotique pour le nombre de points rationnels de hauteur bornée.
\end{altabstract}
\maketitle
\tableofcontents
\chapter{Introduction}
\section{Notation}
The following notation will be used throughout the thesis. By~$F$ we will denote a number field (which one may fix for the whole article). Let $M_F$, $M_F^0$, $M_F^\infty$, $M_F^\RR$ and $M_F^\CC$ be the set of places, finite places, infinite places, real places and complex places of~$F$, respectively. 
For $v\in M_F$ we let $\Fv$ be the~$v$-adic completion of~$F$.
For $\vMFz$, let $\Ov$ be the ring of integers of $\Fv$, let us fix an uniformizer $\piv\in F_v$ and let $\lvert\cdot\rvert_v$ be the absolute value on $F_v$ normalized by $\pivv=[\Ov:\piv\Ov]^{-1}.$ For $\vMF^{\RR}$, we let $\lvert\cdot\rvert_v$ be the usual absolute value and for $\vMF^\CC$ we let $\lvert\cdot\rvert_v$ be the square of the usual absolute value. 
The normalizations are chosen so that the product formula is valid i.e. for every $x\in F$, one has $$\prod_{\vMF}|x|_v=1.$$
By $\OO_F$ we denote the ring of the integers of~$F$ and for a finite subset $S\subset M_F^0$, we denote by $\OFS$ the ring of~$S$-integers. When $\vMFi$, we will denote by $n_v$ the degree $[\Fv:\RR]$. We denote by $\AAF$ the ring of the adeles of~$F$ and by $\AAFt$ the group of ideles.

For a vector $\xxx\in\CC^n$, we will denote by $|\xxx|$ the sum $x_1+\cdots x_n$.
\section{Manin-Peyre conjecture}Let us recall a conjecture due to Manin and Peyre on the asymptotic behaviour of the number of rational points of bounded ``size".
\subsection{}One of the fundamental questions in Diophantine geometry is the study of the number of solutions to algebraic equations. The conjecture of Manin-Peyre deals with a such a question. It predicts the number of the rational points on algebraic varieties, when there are ``a lot" of them. Let us briefly recall it. 

Let~$X$ be a Fano variety over a number field~$F$ and let $K_X^{-1}$ be its anticanonical bundle. The Fano condition, i.e. that $K_X^{-1}$ is positive, is believed to make, after possibly passing to an extension of~$F$, rational points Zariski dense in~$X$. An adelic metric on $K_X^{-1}$ is a choice of metrics for every topological line bundle $K_X^{-1}(F_v)\to X(F_v)$ for~$v$ in the set of the places $M_F$ of~$F$, subject to certain compatibility conditions. A choice of an adelic metric on $K_X^{-1}$ produces two things. Firstly, it gives a {\it height}, i.e. a function $H:X(F)\to\RR_{>0}$ which satisfies the Northcott's property: for every $B>0$, the set $\{x\in X(F)|H(x)\leq B\}$ is finite. This in essence, generalizes the classical notion of the height on the projective space $\PP^n$ when $F=\QQ,$ which is given by $H(\xxx)=\max|x_j|$, where~$\xxx$ are coordinates which satisfy $\gcd(\xxx)=1$. The height serves as a ``size" of a rational point. Secondly, the choice of the adelic metric produces a measure $\omega_H$ on the adelic space $X(\AAF):=\prod_{\vMF}X(\Fv)$ (see \cite{Peyre}). Let $\tau_H$ be the value $\omega_H(\overline{X(F)})$, where the closure $\overline{X(F)}$ is taken in $X(\AAF)$. The following question is asked by Peyre in \cite{Peyre} and refines the original question posed by Manin: 
\begin{conj}\label{conjofpeyre} Suppose that the rational points $X(F)$ are Zariski-dense in~$X$. Then there exist a closed subvariety $Z\subsetneq X$, such that one has $$|\{x\in (X-Z)(F)| H(x)\leq B\}|\sim_{B\to\infty}\alpha \tau_H B\log(B)^{\rk(\Pic(X))-1},$$where $\alpha=\alpha(X)$ is a positive constant connected to the location of $K_X^{-1}$ in the ample cone of~$X$ and $\rk(\Pic(X))$ is the rank of the Picard group of~$X$.
\end{conj}
One removes a closed subvariety to avoid so-called ``accumulating" subvarieties, which contain more points than the rest of the variety. 

The conjecture has been settled in many different cases. The proof for the case of $\PP^n$ is given by Schanuel in \cite{Schanuel}, long before the conjecture was even formulated. Other important known cases of the conjecture are toric varieties (\cite{Toric}),  equivariant compactifications of vector groups (\cite{Vgps}), certain families of Ch\^atelet surfaces (\cite{Chatelet}, \cite{Destagnol}), etc. The version from \ref{conjofpeyre} does admit counterexamples (e.g. \cite{BTcount}, \cite{LeRudulier}). There exists a version for which no known counterexamples exist: instead of removing closed subvarieties, one removes ``thin" sets (a thin set is a subset of the image of the set of rational points $V(F)$ for a morphism of varieties $V\to X$, which, in a neighbourhood of the generic point of $V,$ is quasi-finite and admits no section).  For a survey on Manin-Peyre conjecture, we refer the reader to \cite{LectHZF}.
\subsection{}Different methods are available to tackle the question: universal torsors, circle method, harmonic analysis, Eisenstein series, etc. We briefly recall the harmonic analysis method, firstly used in \cite{aniso} by Batyrev and Tschinkel to prove Manin-Peyre conjecture on compactifications of anisotropic tori and later developed in \cite{Toric},  \cite{FonctionsZ}, \cite{Vgps}, \cite{integral}, etc. to settle more general and new examples. Let~$X$ be a toric variety and let~$T$ be its torus. Let~$H$ be the height given by an adelic metric on the anti-canonical line bundle. We count the rational points of~$T$ (the divisor at the infinity $X-T$ may, however, accumulate points). We have that $T(F)$ is discrete in the adelic torus $T(\AAF)$. We extend $H|_{T(F)}$ to a ``height" on $T(\AAF)$. We let $\wH(s,\chi)$ be the Fourier transform of $H^{-s}$ (where $s$ is a complex number) at the character $\chi:T(\AAF)\to S^1$ which vanishes at $T(F)$. The global transform is an Euler product of local transforms $\wH_v(s,\chi_v)$ for $\vMF$. The local transforms are Igusa integrals (see \cite{Igusa}) and we can either give exact formulas for them or prove certain bounds. Then global Fourier transform $\wH(s,\chi)$ turns out to be a product of~$L$-functions and a part that is easy to analyse.

The Poisson formula (\ref{formuledepoison}) gives $Z(s)=\int_{(T(\AAF)/T(F))^*}\wH(s,\chi)d\chi,$ where $d\chi$ is suitably normalized Haar measure on the group of the characters $(T(\AAF)/T(F))^*$.  There are methods to analyse the integrals on the right hand side, e.g. method of ``controlled $M$-functions" from \cite{FonctionsZ}. 

One obtains the pole and a meromorphic extension of~$Z$, which, by Tauberian results, gives the wanted asymptotic for the number of rational points of~$T$ of bounded height.
\section{Manin-Peyre conjecture for stacks} In this thesis, we are intending to extend the conjecture of Manin-Peyre to the algebraic stacks.

We present two motivations.
\subsection{} A {\it naive height} $H_N$ of an elliptic curve $E/\QQ$  is defined as follows: write the equation of $E$ as $Y^2=X^3+AX+B$, where $(A,B)\in\ZZ^2$ has the property that for every prime~$p$ one has that $p^4|A\implies p^6\nmid B$ and set $H_N(E):=\max(\lvert A^3\rvert,\lvert B^2\rvert)$. Faltings, in his proof of Mordell conjecture \cite{Faltings}, defines different notions of a height of an elliptic curve called {\it unstable Faltings height} and {\it stable Faltings height}. 
For the naive height and the unstable Faltings height, it turns out that if $B>0$, there are only finitely many isomorphism classes of elliptic curves of height at most $B$. It is not hard to count elliptic curves over $\QQ$ of bounded naive height (and, as we will see later, it is possible to do so over any number field~$F$). For the case $F=\QQ$, Hortsch in \cite{Hortsch} finds the asymptotic behaviour for the number of the isomorphism classes of elliptic curves and bounded unstable Faltings height. Both asymptotics are similar to the asymptotics appearing in Manin-Peyre conjecture. However, there is a distinction: the elliptic curves over a number field are not classified by a variety, but by an {\it algebraic stack}. The stack is usually denoted by $\mathcal M_{1,1}$. 
\subsection{} We present another example where one counts rational points on algebraic stacks. Malle in \cite{Malle} conjectures the following: 
\begin{conj}[Malle, \cite{Malle}] Let~$G$ be a non-trivial finite permutation group on~$n$ letters and let~$F$ be a number field. 
We say that $\Gal (K/F)=G$ if $K/F$ is an extension such that the Galois group of its Galois closure is isomorphic to~$G$. There exists $c(F,G)>0,$ such that $$|\{K/F|\Gal(K/F)=G, \Delta(K/F)\leq B \}|\sim c(F,G)B^{a(G)}\log (B)^{b(F,G)-1},$$ when $B\to\infty$, where $\Delta$ is the absolute discriminant of an extension, and $a(G)$ and $b(F,G)$ are explicit invariants of~$G$ and of~$F$ and~$G$, respectively.
\label{malle}
\end{conj}
The prediction is proved for some cases like the case of abelian groups (\cite{Wright}), some other families of groups (e.g. \cite{Wang}, \cite{dave}) and it admits counter-examples (\cite{Kluners}). An object that one counts in Malle's question determines a point on the stack $BG$ (this is the algebraic stack which classifies~$G$-torsors). Thus, Malle conjecture too, can be studied as counting rational points on an algebraic stack. Moreover, the predictions of Manin and Malle conjectures appear similar. 
The similarities have already been observed by Yasuda in \cite{Yasuda} and by Ellenberg,  Satriano and Zureick-Brown in a forthcoming work. The reason for the similarities of the predictions may be hidden in the geometry of the corresponding $BG$-stack. 

\subsection{}The goal of our work is to formulate and investigate the conjecture of Manin-Peyre in the context of algebraic stacks. More precisely, we are going to so for the {\it weighted projective stacks}. If $n\geq 1$ is an integer and $\aaa\in\ZZ^n_{>0}$, the weighted projective stack~$\PPP(\aaa)$ is the ``stacky" quotient of the scheme $\AAA^n-\{0\}$ by the group scheme~$\Gm$, where the action is given by $t\cdot \xxx:=(t^{a_j}x_j)_j$. 
When all the weights $a_j$ are equal to~$1$, then~$\PPP(\aaa)$ is the projective space $\PP^{n-1}$. One has homogenous coordinates on the weighted projective stacks: a rational point on~$\PPP(\aaa)$ is given by~$n$-tuple of elements of~$F$ and two~$n$-tuples~$\xxx$ and $\xxx'$ represent the same point if there exists $t\in F^{\times}$ such that $t^{a_j}x_j=x_j'$ for $j=1\doots n$. 

The moduli stack of elliptic curves $\mathcal M_{1,1}$ is an open substack of the stack $\PPP(4,6)$ (the stack $\PPP(4,6)$ itself is the classifying stack of the curves of genus~$1$ having at worst ordinary singularities).  Another example is given by the stack $B\mu_m$ (where $\mu_m=\Spec\big(F[X]/(X^m-1)\big)$ is the group scheme of~$m$-th roots of unity), which is
precisely the weighted projective stack $\PPP(m)$. The stack~$\PPP(\aaa)$ is smooth, proper and {\it toric}: it contains the stacky torus $\TT(\aaa)=\Gm^n/\Gm$. Its similarity with toric varieties makes it a great candidate to study the Manin-Peyre conjecture on it. 
\section{Principal results}We state principal results of our thesis. Our goal is to provide a theory similar to the one for the rational points on varieties, rather to give ad-hoc proofs of certain cases. The development of the theory occupies a significant part of our thesis.

If~$X$ is a stack and~$R$ a ring, in order to distinguish between the category $X(R)$ and the set of isomorphism classes of objects of this category, we write $[X(R)]$ for the latter. Let~$F$ be a number field.
\subsection{} Let us firstly explain when counting rational points on the weighted projective stack~$\PPP(\aaa)$ is essentially different from counting rational points on the weighed projective space $\PP(\aaa)$. Recall that the weighted projective space $\PP(\aaa)$ is the quotient $(\AAA^n-\{0\})/\Gm$ in the category of schemes for the same action as above. Let us denote by~$j$ the canonical morphism $j:\PPP(\aaa)\to\PP(\aaa)$. The scheme $\PP(\aaa)$ is a toric variety, and let us denote by $T(\aaa)$ its torus. 
%
One can verify that $$\TTa=j^{-1}(T(\aaa))\cong T(\aaa)\times B\mu_{\gcd(\aaa)}=T(\aaa)\times\PPP(\gcd(\aaa)).$$ 
%
Hence, a rational point $\xxx\in[\TTa(F)]$ is uniquely determined by the pair $(j(\xxx), \xxx')\in T(\aaa)(F)\times [j^{-1}(\xxx)(F)]\cong T(\aaa)(F)\times [\PPP(\gcd(\aaa))(F)].$ If $\gcd(\aaa)=1$, then $\PPP(\gcd(\aaa))$ is the one point scheme. It follows from above that the morphism~$j$ induces a bijection $[\TTa(F)]\xrightarrow{\sim}T(F)$. According to \cite[Proposition 6.1]{Olsint}, the pullback homomorphism $j^*_{\QQ}:\Pic(\PP(\aaa))_{\QQ}\to\Pic(\PPP(\aaa))_{\QQ}$ of the rational Picard groups is an isomorphism. It follows that the counting rational points of~$\TTa$ corresponds to counting the rational points of $\PP(\aaa)$ with respect to a height coming from a certain rational line bundle. When $\gcd(\aaa)>1,$ the set $[\PPP(\gcd(\aaa))(F)]$ is infinite (Corollary \ref{examnotnorth}), and we see that counting the rational points of (the stacky torus of)~$\PPP(\aaa)$ is not the same as the counting the rational points of (the torus of) $\PP(\aaa)$. 
\subsection{}In Chapter \ref{Quasi-toric heights}, we define a notion of {\it quasi-toric height} on the set of rational points $\PPP(\aaa).$ It is a function $H:[\PPP(\aaa)(F)]\to\RR_{\geq 0}$ and we establish that it satisfies the Northcott property. A height depends on the choice of a line bundle on the stack $\oPPa=\AAA^n/\Gm$ (where the action is canonically extended) and an ``adelic metric" on it. For $\vMF$, we define topological spaces $[\PPP(\aaa)(F_v)]:=(\Fvnz)/\Fvt,$ where the action is induced from the action of~$\Gm$ on $\AAA^n-\{0\}$. The product space $\prod_{\vMF}[\PPP(\aaa)(F_v)]$ is a good analogue of the ``adelic space" of a variety. In Chapter \ref{Measures on topological spaces associated to weighted projective stacks}, we define a measure $\omega_H$ on the product space $\prod_{\vMF}[\PPP(\aaa)(F_v)]$ and we set $\tau_H=\omega_H(\prod_{\vMF}[\PPP(\aaa)(F_v)])$. We prove that:
\begin{theorem*}[Theorem \ref{osnovna}, Proposition \ref{moreofosnovna}]
Let~$H$ be a quasi-toric height. One has that $$|\{\xxx\in[\PPP(\aaa)(F)]|H(\xxx)\leq B\}|\sim_{B\to\infty}\frac{\tau_H}{|\aaa|}B.$$
\end{theorem*}
For a particular type of quasi-toric height (that in our work is called {\it toric height}), the result has been established in \cite{NajmanBruin}. For the other heights, the result is new.

We establish furthermore that the rational points of~$\PPP(\aaa)$ are equidistributed in $\prod_{\vMF}[\PPP(\aaa)(F_v)]$ in the following sense. Let $i:[\PPP(\aaa)(F)]\to\prod_{\vMF}[\PPP(\aaa)(F_v)]$ be the diagonal map. If $W\subset\prod_{\vMF}[\PPP(\aaa)(F_v)]$ is an open subset such that $\omega(\partial W)=0$, in Theorem \ref{quasitoricequi}, we prove that $$\lim_{B\to\infty}\frac{|\{\xxx\in[\PPP(\aaa)(F)]| i(\xxx)\in W\text{ and }H(\xxx)\leq B\}|}{|\{\xxx\in [\PPP(\aaa)(F)]|H(\xxx)\leq B\}|}=\frac{\omega_H(W)}{\tau_H}.$$
\subsection{} Let us state the second principal result of our work. We suppose that $n=1$ and that $m\in\ZZ_{>1}$. We count $\mu_m$-torsors over~$F$ (i.e the rational points of $\PPP(m)$) of bounded discriminant. In Chapter \ref{number of torsors}, we define a notion of {\it quasi-discriminant height}. It is a function $H:[\PPP(m)(F)]\to\RR_{\geq 0}$ which satisfies the Northcott property, and is essentially different from a quasi-toric height (one cannot normalize it so that the quotients of the two heights are bounded functions on $[\PPP(m)(F)]$). The normalizations are taken so that $H^{m(1-1/r)}$, where $r$ is the least prime of $m,$ is essentially the absolute discriminant $\Delta$ of a $\mu_m$-torsor (i.e. the local components of $H^{m(1-1/r)}$ are different from the local components of the absolute discriminant $\Delta$ at only finitely many places, consequently $H^{m(1-1/r)}/\Delta$ is bounded on $[\PPP(m)(F)]$). As above, a choice of the quasi-discriminant height~$H$ defines a measure $\omega_H$ on $\prod_{\vMF}[\PPP(m)(F_v)]$ and we set $\tau_H=\omega_H(\prod_{\vMF}[\PPP(m)(F_v)])$. We prove that
\begin{theorem*}[Corollary \ref{countingquasidisc}]
Let~$H$ be a quasi-discriminant height. One has that $$|\{\xxx\in[\PPP(m)(F)]|H(\xxx)\leq B\}|\sim_{B\to\infty}\frac{1}{(r-2)!}\cdot\frac{\tau_H}{m}\cdot B\log(B)^{r-2}.$$ 
\end{theorem*}
Again we prove an equidistribution property of rational points in the space $\prod_{\vMF}[\PPP(m)(F_v)]$ (Theorem \ref{equidisc}). The equidistribution property is used to prove that a positive proportion of $\mu_m$-torsors of bounded quasi-discriminant height are fields and that a positive proportion are not fields. Moreover, when $4\nmid m$ or when $i=\sqrt{-1}\in F$, we are able to give a formula for the proportion of fields (Theorem \ref{licava}).

Suppose that~$F$ contains all~$m$-th roots of unity. In particular, one has that $4\nmid m$ or that $i\in F$. We have that $\mu_m=\ZZ/m\ZZ$. In this case, the asymptotics for the number of $\mu_m=\ZZ/m\ZZ$-torsors of bounded discriminant which are fields, has already been given by Wright in \cite{Wright}. The advantage of our method is that we are able to modify the discriminant at finitely many places. 
%
%
%
%
%
\section{Overview of the thesis}\label{Overview of the thesis}Let us make an overview of our work.
\subsection{}\label{metheit} Let us discuss a difference between heights on varieties and stacks. Let~$X$ be a proper~$F$-variety and let~$L$ be a line bundle on~$X$. It is well known that a choice of an $\OFS$-model of $(X,L)$ (where~$S$ is a finite set of finite places of~$F$) endows~$L$ with a metric for every finite place not in~$S$. Endowing the line bundle~$L$ at the remaining places with a metric, gives an adelic metric on~$L$, and hence a height on $X(F)$.

In the construction of the metric for the finite places not in~$S$, one uses the {\it valuative criterion of properness} which gives that every $F_v$-point of~$X$ extends to an $\Ov$-point of the model.  However, this is not true for stacks (e.g. only the points of $\PPP(4,6)$ corresponding to curves having good reductions at~$v$ do extend to $\Ov$-points).

The valuative criterion of properness for stacks gives only that an $\Fv$-point extends to an~$A$-point of the model, where~$A$ is the normalization of $\Ov$ in a finite extension of $\Fv$. 
Such integral extensions give rise to stable heights (i.e. the height of an~$F$-point stays the same when the point is looked as a~$K$-point, where $K/F$ is a finite extension). A drawback of the stable heights is that they do not satisfy the Northcott property, as one sees in the case of $\PPP(4,6).$ Namely, the stable height of two~$F$-elliptic curves which are not isomorphic over~$F$ is the same if they have the same~$j$-invariant. For a fixed elliptic curve $E$, there are infinitely many such elliptic curves (they are constructed by performing quadratic twists to $E$).
%
 \subsection{} In Chapter \ref{Weighted projective stacks and associated topological spaces}, we recall several results about stacks, with the focus on the weighted projective stacks. We also introduce the stack $\oPPa=\AAA^n/\Gm$ (for the canonical extension of the action of~$\Gm$ on $\AAA^n-\{0\}$). Thus, one has open immersions $\TTa\subset\PPP(\aaa)\subset\oPPa$. The stack~$\oPPa$ is not separated (\ref{oPPanotsep}), hence not proper, yet it exhibits the property that all of its rational points extend to integral points, and hence resolves the problem of lack of the integral points from above. This property will be used in Chapter \ref{Quasi-toric heights} to produce {\it unstable} heights on the weighted projective stacks.

Work of Moret-Bailly from \cite{Moret-BaillyS} provides a notion of a topological space associated to the set of (isomorphism classes of)~$R$-points of stacks, when~$R$ is a certain kind of topological local ring. The association is functorial, that is, for a morphism $X\to Y$ of stacks, the induced map $[X(R)]\to [Y(R)]$ is continuous. A list of other properties that the construction satisfies is given in \cite{Cesnavicius}. We prove the following proposition that enables us understand this topology for certain quotient stacks:
\begin{prop}
Suppose that~$X$ is a quotient stack $Y/G$, with~$G$ {\it special} (its torsors are locally trivial, by Hilbert 90, an example is provided by $G=\Gm$). One has that $[X(R)]$ is the topological quotient $Y(R)/G(R)$. 
\end{prop}
Thus, one has for example that $[\PPP(\aaa)(F_v)]=(\Fv^n-\{0\})/\Fvt$ and $[\TTa(\Fv)]=(\Fvt)^n/\Fvt$, where the action of $\Fvt=\Gm(\Fv)$ is the induced from the action of~$\Gm$ on $\AAA^n-\{0\}$ and $\Gmn$, respectively. In the last part of the chapter, we speak about the adelic space of the torus $\TT(\aaa)$. We define it to be the restricted product $$[\TTa(\AAF)]:=\sideset{}{'}\prod _{v\in M_F}[\TTa(F_v)]$$ with the respect to the compact and open subgroups $[\TTa(\Ov)]\subset[\TTa(\Fv)].$ Using the results of \v Cesnavi\v cius from  \cite{Cesnaviciuss} on cohomology of the adeles, we prove that $[\TTa(\AAF)]$ has similar properties to the adelic torus $\Gmn(\AAF)$ (e.g. the image of the rational points $[\TTa(F)]$ for the diagonal map is discrete). 
\subsection{} We start Chapter \ref{Quasi-toric heights} by recalling facts about line bundles on stacks, in particular that the line bundles on the quotient stack $Y/G$ correspond to~$G$-linearized line bundles on the scheme~$Y$. The Picard groups $\Pic(\PPP(\aaa))$ and $\Pic(\oPPa)$ are calculated. Then, we define metrics on line bundles as follows. Let~$v$ be a place of~$F$, let~$X$ be an $\Fv$-algebraic stack and let~$L$ a line bundle on~$X$. We define an $\Fv$-metric on~$L$ to be the data given by ``compatible" $\Fv$-metrics on $y^*L$ for every morphism $y:Y\to X$ with~$Y$ an $\Fv$-scheme (by an $\Fv$-metric on a line bundle over an $\Fv$-scheme, we mean a ``continuous" choice of norms on all $\Fv$-fibers). Our metric does not need to be stable. For quotient stacks $X=Y/G,$ when~$G$ is assumed to be a special algebraic group, we relate the group of $\Fv$-metrized line bundles $\widehat{\Pic}_v(Y/G)$ with the group $\widehat{\Pic_v^G}(Y)$ of $\Fv$-metrized line bundles on~$X$ which are endowed with a~$G$-linearization and such that the metric is~$G$-invariant:
\begin{prop}[Proposition \ref{gmetmet}]
 Let~$G$ be a special locally of finite type $\Fv$-group scheme acting on locally of finite type $\Fv$-scheme~$Y$. The canonical homomorphism $\widehat{\Pic}_v(Y/G)\to\widehat{\Pic_v^G}(Y)$ is injective, and is an isomorphism if $\widehat{\Pic}_v(Y/G)\to\Pic(Y/G)$ is surjective. 
\end{prop}
The stack $\PPP(\aaa)=(\AAA^n-\{0\})/\Gm$ satisfies this condition on the existence of $\Fv$-metrics on every of its line bundles (\ref{pagfi}). Consequently, as $\Pic(\AAA^n-\{0\})$ is trivial, we deduce that to define an $\Fv$-metric on a line bundle on~$\PPP(\aaa)$, it suffices to define a~$\Gm$-invariant metric on the corresponding~$\Gm$-linearization of the trivial line bundle on $\AAA^n-\{0\}$. Such metric is defined by the norm of the section~$1$ and condition on the linearizations gives a ``homogeneity" condition to the function $\Fv^n-\{0\}\to \RR_{>0}, \xxx\mapsto ||1||_{\xxx}.$ 

Suppose that a line bundle~$L$ on~$\PPP(\aaa)$ is endowed with an $\Fv$-metric for every $\vMF$, subject to a compatibility condition which allows that the norms of a section can be multiplied at any $\xxx\in[\PPP(\aaa)(F)]$ (see the condition in \ref{fahom}). We can define heights by multiplying the inverses of these norms for every~$v$. The generality, leaves possibility of existence of ``essentially different" heights on the same line bundles, i.e. heights such that their quotients are not bounded functions on the set $[\PPP(\aaa)(F)]$. Examples are: the mentioned {\it stable heights}, the {\it quasi-toric heights} (we are going to explain them now) and in the case $n=1$ the {\it quasi-discriminant heights} (they will be explained in the last chapter). Quasi-toric heights are the heights, which come from the families of metrics which arise in the following way for almost every place: extend an $\Fv$-point of the stack~$\oPPa$ to an $\Ov$-point and use the classical method (already discussed in \ref{metheit}) to get a metric (a smaller modification, however, must be done, because~$\oPPa$ is not separated and thus an $\Ov$-extension of an $\Fv$-point is not unique). Contrary to the stable heights, the quasi-toric heights do satisfy the Northcott property: 
\begin{thm}[Theorem \ref{boundppatta}]\label{northintro} 
Let~$H$ be a quasi-toric height on~$\PPP(\aaa)$. 
Let $\epsilon >0$. One has that there exists $C>0$ such that  $$|\{\xxx\in[\PPP(\aaa)(F)]|H(\xxx)\leq B\}|\leq CB^{1+\epsilon}.$$ 
\end{thm}
The idea of the proof of is to separately estimate the finite and the infinite height. The upper bound for the cardinality in the theorem is needed to provide convergence of the corresponding height zeta series. The claim of \ref{northintro} stays valid even when metrics at finitely many places are allowed to have ``logarithmic" singularities along rational divisors (see \ref{nortlogfalt}). The proof of that version follows immediately from \ref{northintro}, after establishing an estimate for the singular height of the form: $H_{\Sing}\geq CH\log^{-\eta}(H)$, where $C, \eta>0$, which we do in \ref{comphh}. 
\subsection{} In Chapter \ref{Measures on topological spaces associated to weighted projective stacks}, we endow the topological spaces associated to~$R$-points with measures. In particular, we define measures on $[\PPP(\aaa)(F_v)]$ (which depend on the choice of the metrics) and on $[\TTa(F_v)]$ (which do not depend on the choice of metrics). The measures are used to define Peyre's constant $\tau_H$.

The last part of the chapter is dedicated to the definition of the measures on the ``adelic torus" $[\TTa(\AAF)]$ and the Tamagawa number of the stacky torus~$\TTa$. We establish that 
\begin{prop}[Proposition \ref{tamagawatta}]\label{tamagawattaaaa}
One has that $\Tam(\TTa)=1$.
\end{prop}
When $\aaa=\jed$, this is the classical result that the Tamagawa number of a split torus is~$1$. The proof of \ref{tamagawattaaaa} uses Oesterl\'e's Euler-Poincar\'e characteristics of complexes of locally compact abelian groups which are endowed with Haar measures.
\subsection{} Chapter \ref{Analysis of characters} studies characters of the ``adelic torus" $[\TTa(\AAF)]$. We introduce ``discrete" norms and ``infinity" norms of these characters. One establish a finiteness result on the number of the characters $\chi\in[\TTa(\AAF)]^*$ which vanish on $[\TTa(F)]$ and on certain subgroups of bounded either of these norms. In the last part of this chapter we recall estimates of Rademacher on~$L$ functions of characters. The results of this chapter will be used in Chapter \ref{Fourier transform of the height function} to prove that the Fourier transform of a height function is integrable.

\subsection{} In Chapter \ref{Fourier transform of the height function} we adapt the method of harmonic analysis of Batyrev and Tschinkel from \cite{Toric} to our situation. We assume the metrics are {\it smooth}. The first part of the chapter is dedicated to the calculation of the Fourier transform of the local height at a finite place~$v$. For almost all~$v$, we can give the exact formula which turns out to be the product of local~$L$ functions of characters and other factors. 
Then, for infinite~$v$, using the smoothness assumptions, we prove suitable decays of the Fourier transform in the two norms of characters. The proof for this claim is an adaptation of the idea of Chambert-Loir and Tchinkel from \cite{Vgps} and \cite{integral}, where the authors apply integration by parts with the respect to invariant vector fields. The global Fourier transform, hence, writes as the product of~$L$ functions and a part that we have control of. 
\subsection{}In Chapter \ref{Analysis of height zeta functions} we use the theory of \cite{FonctionsZ} to analyse height zeta function.

The accent is on the stacky torus $\TTa\subset \PPP(\aaa).$ From the estimate of Theorem \ref{northintro}, one deduces that the height zeta function $Z(s):=\sum_{\xxx\in[\TTa(F)]}H(\xxx)^{-s}$ converges and defines a holomorphic function of $s$ in the domain $\Re(s)>1$. Poisson formula gives that $$Z(s)=\int_{([\TTa(\AAF)]/[\TTa(F)])^*}\wH(s,\chi)d\chi$$ whenever the expressions on both hand side converge. The estimates from Chapter \ref{Analysis of characters} and Chapter \ref{Fourier transform of the height function} and the ``similarity" of the global Fourier transform of the height with~$L$ functions, give the convergence of the integral on the right hand side for $\Re(s)>1$. Moreover, the proof will imply that $s\mapsto Z(s)$ has a meromorphic extension to a domain $\Re(s)>1-\delta$, for some $\delta>0$. The estimates of Rademacher imply that~$Z$ satisfies the growth conditions needed for Tauberian theorems. The residue of~$Z$ at~$1$ is also calculated. As a consequence, Tauberian theorems give the asymptotic behaviour of the number of rational points of the height at most $B$:
\begin{thm}[Theorem \ref{osnovna}, Proposition \ref{moreofosnovna}]\label{osnovnanapocet}
Let~$H$ be a quasi-toric height. One has that $$\{\xxx\in[\PPP(\aaa)(F)]| H(\xxx)\leq B\}\sim\frac{\tau_H}{|\aaa|}B,$$ when $B$ tends to $+\infty$.
\end{thm}
The constant $\frac1{|\aaa|}$ has the same interpretation as in the case of varieties (see \ref{firstpartofpeyre}) and the asymptotic stays the same when one counts rational points of~$\oPPa$ (because $[\oPPa(F)]-[\PPP(\aaa)(F)]$ is a one point set). Thus \ref{osnovnanapocet} can be understood as that Manin-Peyre's conjecture is true for weighted projective stacks~$\oPPa$. The last part of the chapter is dedicated to the understanding {\it equidistribution} of the rational points of the stack~$\PPP(\aaa)$. The idea is to find the asymptotic behaviour for the number of rational points of bounded height which are required for finitely many places~$v$ to belong to certain open subsets of the~$v$-adic space of the stack (e.g. say that the $2$-adic valuation is even). An elegant way to phrase this question has been given by Peyre in \cite{Peyre}, using the measure $\omega_H$ from above: if  $W\subset\prod_{\vMF}[\PPP(\aaa)(F_v)]$ is an open subset of negligible boundary, one expects that: $$\lim_{B\to\infty}\frac{|\{\xxx\in[\PPP(\aaa)(F)]|i(\xxx)\in W\text{ and }H(\xxx)\leq B\}|}{|\{\xxx\in[\PPP(\aaa)(F)]|H(\xxx)\leq B\}|}=\frac{\omega_H(W)}{\tau_H},$$where $i:[\PPP(\aaa)(F)]\to\prod_{\vMF}[\PPP(\aaa)(F_v)]$ is the diagonal map. If this is true for every such $W$, then we say that the rational points are equidistributed. We prove that that: 
\begin{thm}[Theorem \ref{quasitoricequi}] The rational points of~$\PPP(\aaa)$ are equidistributed in the space $\prod_{\vMF}[\PPP(\aaa)(F_v)]$.
\end{thm}
\subsection{} In Chapter \ref{number of torsors}, we use our methods to study a question similar to Malle conjecture. We find the asymptotic for the number of $\mu_m$-torsors over~$F$ of bounded absolute discriminant. When~$F$ contains all~$m$-th roots of~$1$, this question is the Malle conjecture for the cyclic group $\ZZ/m\ZZ,$ but one needs to remove the torsors which are not fields and there may be a positive proportion of them (the counting of the cyclic extensions has been covered by the case of the abelian groups from \cite{Wright}). The $\mu_m$-torsors over~$F$ are classified by the algebraic stack $B\mu_m=\TTT(m)=\PPP(m)$. 

We use the language of the heights developed earlier. We speak about {\it quasi-discriminant} heights, which are similar to the discriminants, with the difference that the local components of these heights at the finitely many places may be different from the local components of the discriminant. Let us note that a quasi-discriminant height is not a quasi-toric height, as the local components of the two heights are different at almost every place. We will define a measure $\omega_H$ on $\prod_{\vMF}[\PPP(m)(F_v)]$ and we will set $\tau_H=\omega_H(\prod_{\vMF}[\PPP(m)(F_v)])$. 

The method of the proof is an adaption of the above harmonic analysis method. There are some simplifications, as the local spaces $\TTT(m)(\Fv)$ are finite, and modifications because of the difference with quasi-toric heights. Eventually, we prove the convergence of the height zeta series and use Poisson's formula as before. For the purpose of having more elegant formula, here we state the final asymptotic it for heights that are ``essentially" $\Delta^{\frac{1}{m(1-1/r)}},$ where $r$ is the least prime of~$m$ (that is for almost all places the local component of~$H$ coincides with the local component of $\Delta^{\frac{1}{m(1-1/r)}}$). 
\begin{thm}[Corollary \ref{countingquasidisc}]\label{quasidiscpocetak}
Let~$H$ be a quasi-discriminant height. One has that $$|\{x\in[\PPP(m)(F)]| H(x)\leq B\}|=\frac{\tau_H}{(r-2)!m}B\log(B)^{r-2}.$$ 
\end{thm}
The asymptotic is very reminiscent of the one from Manin-Peyre conjecture. 
As in the case of quasi-toric heights, we are able to prove a corresponding equidistribution property in $\prod_{\vMF}[\PPP(m)(F_v)]$. We end the chapter by the proof that there is a positive proportion of $\mu_m$-torsors which are fields. This is proven by finding an open subset $W\subset\prod_{\vMF}[\PPP(m)(F_v)]$ of positive volume such that all of $\mu_m$-torsors contained in it are fields. Moreover, when $4\nmid m$ or when $i=\sqrt{-1}\in F$, we give an exact formula for this proportion.
\section{Questions and Remarks} Let us discuss some questions that arise naturally from our work.
\subsection{} The conjecture of Manin-Peyre has been proved for all smooth toric varieties. We would like to know to what generality the proof applies to other toric stacks (c.f. \cite{mann}). We do not know what happens when the ``stacky torus" is not split, i.e. not a quotient of two split tori. 
Manin-Peyre conjecture for toric varieties has also been proved for the case of function fields (\cite{Bourqui}). We would like to know what is the situation for toric stacks. 
\subsection{} It would be interesting to understand to what other stacks one can develop a theory of heights and use it to count rational points. Examples of such stacks could be: the stack $\mathcal M_g$ which classifies the curves of genus $g$, the stack of principally polarized abelian varieties $\mathcal A_g$, etc.

\subsection{}One could ask whether there exists a stack~$X$ with enough integral points, such that $B\mu_m\subset X$ and such that the discriminant arises as a height induced by an $\OFS$-model of~$X$ and a line bundle on it (here~$S$ is a finite set of places). We would like to know then whether the result of \ref{quasidiscpocetak} can be reinterpreted as that Manin-Peyre conjecture is true for~$X$. We may then ask how the prediction of Malle conjecture compares with the prediction of Manin-Peyre conjecture. The same question can be asked for any finite group (scheme)~$G$.
\subsection{}A different notion of a height on stack, defined by vector bundles, has been proposed Ellenberg, Satriano and Zureick-Brown in a forthcoming work. Their height is not additive in the vector bundles. We would like to know how this notion compares with our notion of the height.
\subsection{}A counterexample to Malle conjecture has been constructed by Kl\"uners in \cite{Kluners}. The known counterexamples in the conjecture of Manin-Peyre are avoided if one allows removing ``thin sets". We would like to know whether removing ``thin sets" fixes the prediction of Malle.
\chapter{Weighted projective stacks}
\label{Weighted projective stacks and associated topological spaces}
A weighted projective stack is a stacky quotient $\PPP(\aaa):=(\AAA^n-\{0\})/\Gm$, where the action of~$\Gm$ is weighted with the weights $a_1\doots a_n$, where $a_1\doots a_n$ are positive integers. In the first part of this chapter we will recall several properties of such stacks. It turns out that the weighted projective stacks are proper, however, not all of its rational point extends to an integral point. This is a fundamental feature that enables one to define heights. The stack $\oPPa:=\AAA^n/\Gm$ (that by the abuse of the terminology we may also call a weighted projective stack) has enough of integral points in this sense. The second part of the chapter is dedicated to the topological spaces associated to weighted projective stacks.
\section{Weighted projective stacks}\label{wps} In this section we recall several facts about stacks and weighted projective stacks.
\subsection{} \label{intrstack}In this paragraph we recall some generalities on quotient stacks. We follow \cite{stacks-project}. 

Let~$Z$ be a scheme. Let~$X$ be a~$Z$-scheme and let $a:G\times _ZX\to X$ be a left~$Z$-action of locally of finite presentation flat~$Z$-algebraic group~$G$ on~$X$. Denote by $p_2$ the projection to the second coordinate $G\times _ZX\to X$. One has a commutative diagram \begin{equation}\label{comdiagact} \begin{tikzpicture}
  \matrix (m) [matrix of math nodes,row sep=3em,column sep=4em,minimum width=2em]
  {
      G\times _ZX& G\times _ZX \\
     & X,& \\ };
  \path[-stealth]
    (m-1-1)  edge node [above] {$u$} (m-1-2)
            edge node [below] {$a $} (m-2-2)
    (m-1-2) edge node [right] {$p_2 $} (m-2-2)
            ;
\end{tikzpicture} \end{equation}
where the morphism $u$ is given by $(g,x)\mapsto (g,a(g,x))$. The morphism $u$ is an automorphism as its inverse is provided by $(g,x)\mapsto (g,a(g^{-1},x))$, hence, the morphism $a$ is surjective, flat and locally of finite presentation.

We write $X/G$ for the quotient stack (\cite[\href{https://stacks.math.columbia.edu/tag/044O}{Tag 044O}]{stacks-project}). Recall that if $V$ is a~$Z$-scheme a~$1$-morphism $x:V\to X/G$ is given by a $G_V$-equivariant morphism $T\to X_V,$ where~$T$ is a $G_V$-torsor. The $G_V$-equivariant morphism $T\to X_V$ will be often denoted by $\widetilde x$. A $2$-morphism $\theta:x\to y$ of~$1$-morphisms $x:V\to X/G$ and $y:V\to X/G$ corresponding to $G_V$-equivariant morphisms from $G_V$-torsors $\widetilde x:T\to X$ and $\widetilde y:R\to X$, respectively, is given by a morphism $\theta:R\to T$ of $G_V$-torsors such that $\widetilde x=\widetilde y\circ \theta$. 

In the following proposition we recall some of the properties of a quotient stack. Let $q:X\to X/G$ be the quotient~$1$-morphism, i.e. the one given by the trivial $G_X$-torsor $G\times _ZX\xrightarrow{p_2}X$ and the $G_X$-equivariant morphism $G_X=G\times _ZX\xrightarrow{a} X.$ 
\begin{prop}\label{propofstack}
\begin{enumerate}
\item For every~$1$-morphism $x:V\to X/G$ over~$Z$ with $V$ a scheme, the diagram
\begin{equation*}\label{ojub}
\begin{tikzcd}
T \arrow[r,"\widetilde x"] \arrow[d,""]& X\arrow[d,"q"] \\
 V\arrow[r,"x"] &X/G.
\end{tikzcd}\end{equation*}
is $2$-commutative $2$-cartesian (\cite[\href{https://stacks.math.columbia.edu/tag/04UV}{Section 04UV}]{stacks-project}).
\item The morphism $q:X\to X/G$ is surjective, flat, representable and locally of finite presentation \cite[\href{https://stacks.math.columbia.edu/tag/06FH}{Lemma 06FH}]{stacks-project}.
\item The stack $X/G$ is algebraic (\cite[\href{https://stacks.math.columbia.edu/tag/06FI}{Theorem 06FI}]{stacks-project}).
\item The stack $X/G$ is smooth over~$Z$ if~$X$ is smooth over~$Z$ (\cite[\href{https://stacks.math.columbia.edu/tag/0DLS}{Lemma 0DLS}]{stacks-project}). 
\item If~$Y$ is a~$Z$-scheme, we let $X_Y/G_Y$ be the quotient stack for the induced~$Y$-action $a_Y:G_Y\times _YX_Y\to X_Y$. The canonical~$1$-morphism $X_Y/G_Y\to (X/G)\times _ZY$ is an equivalence (\cite[\href{https://stacks.math.columbia.edu/tag/04WX}{Lemma 04WX}]{stacks-project}).
\end{enumerate}
\end{prop}
Let $\Delta_{X/G}$ be the diagonal morphism of $X/G\to Z$. The following lemma will be quoted several times:
\begin{lem}\label{diagofquot}
The diagram \begin{equation}\label{xgxgxg}
\begin{tikzcd}
G\times _ZX \arrow[r,"a\times p_2"] \arrow[d,"q\circ p_2"]& X\times _ZX\arrow[d,"q\times _{\ZZ}q"] \\
 X/G\arrow[r,"\Delta_{X/G}"] &(X/G)\times_Z(X/G).
\end{tikzcd}\end{equation}is $2$-commutative $2$-cartesian.
\end{lem}
\begin{proof}
The following diagram:
\begin{equation}\label{bubyv}
\begin{tikzcd}
	{G\times_ZX} & {X\times_ZX} & X \\
	X & {X\times_Z(X/G)} & {X/G} \\
	{X/G} & {(X/G)\times_Z(X/G)} & {X/G,}
	\arrow["{p_1}", from=1-2, to=1-3]
	\arrow["{p_2}"', from=3-2, to=3-3]
	\arrow["{(\Id_X,q)}"', from=1-2, to=2-2]
	\arrow["{(a,p_2)}", from=1-1, to=1-2]
	\arrow["{p_2}"', from=1-1, to=2-1]
	\arrow["{\Delta_{X/G}}"', from=3-1, to=3-2]
	\arrow["q"', from=2-1, to=3-1]
	\arrow["{p_2}", from=2-2, to=2-3]
	\arrow["{\text{Id}_{X/G}}"', from=2-3, to=3-3]
	\arrow["q"', from=1-3, to=2-3]
	\arrow["{(q,\text{Id}_{X/G})}"', from=2-2, to=3-2]
	\arrow["{\Gamma_q}"', from=2-1, to=2-2]
\end{tikzcd}\end{equation}
where $p_1,p_2$ are always the evident projections and $\Gamma_q$ is the graph of the~$1$-morphism $q,$ is $2$-commutative. The bottom right square is $2$-commutative $2$-cartesian. The square
\[\begin{tikzcd}
	X && {X/G} \\
	{X/G} && {X/G}
	\arrow["q"', from=1-1, to=2-1]
	\arrow["{\text{Id}_{X/G}}"', from=1-3, to=2-3]
	\arrow["p_2\circ\Gamma_q", from=1-1, to=1-3]
	\arrow["{p_2\circ\Delta_{X/G}}"', from=2-1, to=2-3]
\end{tikzcd}\]
is $2$-commutative $2$-cartesian, because one has isomorphisms $\Id_{X/G}\cong p_2\circ\Delta_{X/G}$ and $q\cong p_2\circ\Gamma_q$. It follows that the $2$-commutative square
\begin{equation}\label{korjc}\begin{tikzcd}
	{} & {X} & X\times_Z(X/G) \\
	\\
	& {X/G} & {(X/G)\times_Z(X/G)}
	\arrow["{\Delta_{X/G}}"', from=3-2, to=3-3]
	\arrow["{\Gamma_q}", from=1-2, to=1-3]
	\arrow["{q}"', from=1-2, to=3-2]
	\arrow["{(q,\text{Id}_{X/G})}"', from=1-3, to=3-3]
\end{tikzcd}\end{equation} is $2$-cartesian. The upper right square in (\ref{bubyv}) is $2$-commutative $2$-cartesian. The square 
\begin{equation*}\begin{tikzcd}
	{G\times_ZX} && X \\
	X && {X/G}
	\arrow["q"', from=1-3, to=2-3]
	\arrow["{p_2}"', from=1-1, to=2-1]
	\arrow["{p_2\circ \Gamma_q}"', from=2-1, to=2-3]
	\arrow["{p_2\circ(a,p_1)}", from=1-1, to=1-3]
\end{tikzcd}\end{equation*}
is $2$-cartesian $2$-commutative, because one has $a=p_1\circ(a,p_2)$ and there exists an isomorphism $q\cong p_2\circ\Gamma_q$. It follows that the square 
\begin{equation}\label{ribo}
\begin{tikzcd}
	{G\times_ZX} & {X\times_ZX} \\
	X & {X\times_Z(X/G)}
	\arrow["{p_2}"', from=1-1, to=2-1]
	\arrow["{(a,p_2)}", from=1-1, to=1-2]
	\arrow["{\text{Id}_X\times q}"', from=1-2, to=2-2]
	\arrow["{\Gamma_q}"', from=2-1, to=2-2]
\end{tikzcd}
\end{equation}
is $2$-commutative and $2$-cartesian. By using that (\ref{korjc}) and (\ref{ribo}) are $2$-commutative and $2$-cartesian square we deduce that (\ref{xgxgxg}) is $2$-commutative $2$-cartesian.
\end{proof}
We say that affine algebraic group~$G$ is \textit{special} (Serre, Section 4.1 in \cite{SerreG}) if every~$G$-torsor $Y\to W$, with $W$ and~$Y$ schemes, is locally trivial for the Zariski topology on $W$. Hilbert 90 theorem states that the general linear groups $GL_d$, for $d\geq 1$ are special (see e.g. \cite[Lemma 4.10, Chapter III]{Milne}). 
\begin{lem}\label{xgtac}
Suppose~$G$ is a flat, locally of finite presentation special algebraic group. Let~$R$ be a local~$Z$-ring. For every~$1$-morphism of~$Z$-stacks $x:\Spec R\to X/G$, there exists a $2$-commutative $2$-cartesian square \begin{equation}\label{ojub}
\begin{tikzcd}
G_R \arrow[r,"\widetilde x"] \arrow[d,""]& X\arrow[d,"q"] \\
 \Spec (R)\arrow[r,"x"] &X/G,
\end{tikzcd}\end{equation}
with $\widetilde x$ being~$G$-equivariant morphism.
\end{lem}
\begin{proof}As~$R$ is local and~$G$ special, every~$G$-torsor over $\Spec R$ is isomorphic to the trivial one. Now the claim follows from part (1) of \ref{propofstack}. 
\end{proof}
One can also see that in the situation of \ref{xgtac}, the category $(X/G)(R)$ is equivalent to the following category: its objects are $G_R$-equivariant morphisms $G_R\to X_R$ and a morphism $t:(x:G_R\to X_R)\to (y:G_R\to X_R)$ is an element $t\in G_R(R)$  such that $x=y\circ t$, when~$t$ is seen as a morphism $t:G_R\to G_R$ by multiplication to the left. We will often by the abuse of notation write $(X/G)(R)$ for the latter category.
\subsection{}
In this paragraph we work over $\Spec(\ZZ)$. Let $n\geq 1$ be an integer. Let $\aaa\in\ZZ^n_{\geq 1}$. The smooth group scheme~$\Gm$ acts on $\AAA^n$ via the formula: \begin{equation}a:\Gm\times\AAA^n\to\AAA^n\hspace{1cm}(t,\xxx)=(t^{a_j}x_j)_j. \label{defactstac}\end{equation}
We will often write $t\cdot\xxx$ instead of $a(t,\xxx).$
Note that $\AAA^n-\{0\}\subset\AAA^n$ and $\Gm^n=(\AAA^1-\{0\})^n\subset\AAA^n$ are~$\Gm$-invariant open subschemes for this action. We have, hence, induced actions of~$\Gm$ on $\Gmn$ and on $\AAA^n-\{0\}$.
\begin{lem}\label{aaaction}
\begin{enumerate}
\item The morphism $(a,p_2):\Gm\times\AAA^n\to\AAA^n\times\AAA^n$ is of finite presentation and affine (hence separated and quasi-compact by \cite[\href{https://stacks.math.columbia.edu/tag/01S7}{Lemma 01S7}]{stacks-project}).
\item The morphisms $$(a|_{\Gm\times(\AAA^n-\{0\})},p_2):\Gm\times(\AAA^n-\{0\})\to(\AAA^n-\{0\})\times(\AAA^n-\{0\})$$ and $$(a|_{\Gm\times\Gm^n},p_2):\Gm\times\Gm^n\to\Gmn\times\Gmn $$ induced from~$\Gm$-invariant open subschemes $\AAA^n-\{0\}\subset \AAA^n$ and $\Gmn\subset\AAA^n$, respectively, are finite.
\end{enumerate}
\end{lem}
\begin{proof}
\begin{enumerate}
\item The morphism $(a,p_2)$ is of finite presentation as both $a$ and $p_2$ are of finite presentation (see Diagram (\ref{comdiagact}) in \ref{propofstack}). The morphism $(a,p_2)$ is affine because it is a morphism of affine schemes. 
\item  Let us verify that $(a|_{\Gm\times(\AAA^n-\{0\})},p_2)$ is proper. It is affine, hence separated, and of finite type, as it is the base change of the affine and finite type morphism $\Gm\times\AAA^n\to\AAA^n\times\AAA^n $ along the open immersion $(\AAA^n-\{0\})\times(\AAA^n-\{0\})\to\AAA^n\times\AAA^n.$ We use the valuative criterion for finite type morphism with target Noetherian to be universally closed \cite[\href{https://stacks.math.columbia.edu/tag/0CM5}{Lemma 0CM5}]{stacks-project}. Let~$R$ be a discrete valuation ring, $v_R$ its valuation and~$K$ its fraction field. Consider the diagram 
\[\begin{tikzcd}
	{\Spec(K)} & {\Gm\times (\mathbb{A}^n-\{0\})} \\
	 {\Spec(R)}& {(\mathbb A^n-\{0\})\times(\mathbb A^n-\{0\}).}
	\arrow["{}"', from=1-1, to=2-1]
	\arrow["{(t,\zzz)}", from=1-1, to=1-2]
	\arrow["{}"', from=1-2, to=2-2]
	\arrow["{(\xxx,\yyy)}", from=2-1, to=2-2]
\end{tikzcd}\]
It follows that $\zzz=\yyy\in(\AAA^n-\{0\})(R)$ and that $t\cdot\zzz=\xxx$. There exists~$i$ such that $v_R(z_i)=v_R(y_i)=0.$ We have that $0\leq v_R(x_i)=v_R(t^{a_i}z_i)=a_iv_R(t)$ and hence $v_R(t)\geq 0$. There exists~$k$ such that $v_R(x_k)=0$. We have $0=v_R(x_k)=v_R(t^{a_k}z_k)=a_kv_R(t)+v_R(z_k)\geq a_kv_R(t)$ and hence $v_R(t)\leq 0$. We deduce $v_R(t)=0$ i.e. $t\in\Gm(R)$. We deduce $(t,\zzz)\in(\Gm\times(\AAA^n-\{0\}))(R)$ and the valuative criterion is verified. It follows that $(a|_{\Gm\times(\AAA^n-\{0\})},p_2)$ is universally closed and we deduce that it is proper. It is affine, and we deduce that it is also finite. Now, the morphism $(a|_{\Gm\times\Gm^n},p_2)$ is the base change of the finite morphism $(a|_{\Gm\times(\AAA^n-\{0\})},p_2)$ along the open immersion $\Gm^n\times\Gm^n\to(\AAA^n-\{0\})\times(\AAA^n-\{0\}),$ hence is finite. 
\end{enumerate}
\end{proof}
\begin{mydef}
We define quotient stacks for the actions from above: \begin{align*}
\overline{\PPP(\aaa)}&:=\AAA^n/\Gm,\\
\PPP(\aaa)&:=(\AAA^n-\{0\})/\Gm,\\
\TTa&:=\Gm^n/\Gm.
\end{align*}
The first two stacks we may call weighted projective stacks.
\end{mydef}
The~$\Gm$-equivariant open immersions $\Gm^n\subset\AAA^n-\{0\}$, $\AAA^n-\{0\}\subset\AAA^n$ induce~$1$-morphisms of stacks $\TTa\to\PPP(\aaa)$ and $\PPP(\aaa)\to\oPPa$ by \cite[\href{https://stacks.math.columbia.edu/tag/046Q}{Lemma 046Q}]{stacks-project}, which are open immersions by 
\cite[\href{https://stacks.math.columbia.edu/tag/04YN}{Lemma 04YN}]{stacks-project}.
\begin{lem}\label{ppaqs} The stacks $\TTa,$~$\PPP(\aaa)$ and $\overline{\PPP(\aaa)}$ satisfy the following:
\begin{enumerate}
\item They are smooth algebraic stacks.
\item They are quasi-compact.
\item Their diagonals for the canonical morphisms to $\Spec(\ZZ)$ are representable and affine (hence, separated and quasi-compact). The diagonals of~$\TTa$ and~$\PPP(\aaa)$ are further finite.
\item They are of finite presentation.
\end{enumerate}
\end{lem}
\begin{proof}
\begin{enumerate}
\item  This claim follows from parts (3) and (4) of \ref{propofstack}. 
\item The quotient~$1$-morphisms $q^\aaa:\AAA^n\to\AAA^n/\Gm=\overline{\PPP(\aaa)},$ $q^{\aaa}|_{\AAA^n-\{0\}}:(\AAA^n-\{0\})\to\PPP(\aaa)$ and $q^{\aaa}|_{\Gm^n}:\Gm^n\to\TTa$ are surjective by \ref{propofstack} and as $\AAA^n,\AAA^n-\{0\}$ and $\Gm^n$ are quasi-compact, we deduce by \cite[\href{https://stacks.math.columbia.edu/tag/04YC}{Lemma 04YC}]{stacks-project} that $\overline{\PPP(\aaa)},$~$\PPP(\aaa)$ and~$\TTa$ are quasi-compact. 
\item All three diagonals are representable, because all three stacks are algebraic. Let us prove that the diagonal morphism $\Delta_{\oPPa}:\oPPa\to\oPPa\times\oPPa$ is affine. By \ref{diagofquot}, we have a $2$-commutative $2$-cartesian square\[\begin{tikzcd}
	{\Gm\times \mathbb{A}^n} & {\mathbb A^n\times\mathbb A^n} \\
	\oPPa & {\overline{\mathscr P(\mathbf a)}\times\oPPa.}
	\arrow["{q\circ p_2}"', from=1-1, to=2-1]
	\arrow["{(a,p_2)}", from=1-1, to=1-2]
	\arrow["{(q,q)}"', from=1-2, to=2-2]
	\arrow["\Delta_{\oPPa}"', from=2-1, to=2-2]
\end{tikzcd}\]
The~$1$-morphism $(q,q):\AAA^n\times\AAA^n\to\oPPa\times\oPPa$ is surjective, flat and locally of finite presentation and the morphism $(a,p_2):\Gm\times\AAA^n\to\AAA^n\times\AAA^n$ is affine by \ref{aaaction}. It follows from \cite[\href{https://stacks.math.columbia.edu/tag/06TY}{Lemma 06TY}]{stacks-project} that $\Delta_{\oPPa}$ is affine. Let us prove that the diagonal $\Delta_{\PPP(\aaa)}:\PPP(\aaa)\to\PPP(\aaa)\times\PPP(\aaa)$ is finite.
By \ref{diagofquot}, we have a $2$-commutative $2$-cartesian square\[\begin{tikzcd}
	{\Gm\times (\mathbb{A}^n-\{0\})} & {(\mathbb A^n-\{0\})\times(\mathbb A^n-\{0\})} \\
	\PPP(\aaa) & {\PPP(\aaa)\times\PPP(\aaa).}
	\arrow["{q\circ p_2}"', from=1-1, to=2-1]
	\arrow["{(a,p_2)}", from=1-1, to=1-2]
	\arrow["{(q,q)}"', from=1-2, to=2-2]
	\arrow["\Delta_{\PPP(\aaa)}"', from=2-1, to=2-2]
\end{tikzcd}\]
The~$1$-morphism $$(q|_{\AAA^n-\{0\}},q|_{\AAA^n-\{0\}}):(\AAA^n-\{0\})\times(\AAA^n-\{0\})\to\PPP(\aaa)\times\PPP(\aaa)$$ is surjective, flat and locally of finite presentation and the morphism $(a,p_2):\Gm\times(\AAA^n-\{0\})\to(\AAA^n-\{0\})\times(\AAA^n-\{0\})$ is finite by \ref{aaaction}. It follows from \cite[\href{https://stacks.math.columbia.edu/tag/06TY}{Lemma 06TY}]{stacks-project} that $\Delta_{\PPP(\aaa)}$ is finite.
 Now, one has that the diagonal $\Delta_{\TTa}$ is just the base changes of~$\Delta_{\PPP(\aaa)}$ along the open immersion $\TTa\subset\PPP(\aaa)$, hence is finite by \cite[\href{https://stacks.math.columbia.edu/tag/045C}{Lemma 045C}]{stacks-project}.
\item Recall that of finite presentation means quasi-compact, quasi-separated and locally of finite presentation. We have seen in (1) and (2) that~$\TTa$,~$\PPP(\aaa)$ and~$\oPPa$ are smooth, thus locally of finite presentation, and quasi-compact. By (3), the diagonals $\Delta_{\TTa},\Delta_{\PPP(\aaa)}$ and $\Delta_{\oPPa}$ are quasi-compact and separated, thus quasi-separated by \cite[\href{https://stacks.math.columbia.edu/tag/050E}{Lemma 050E}]{stacks-project}, i.e. $\TTa, \PPP(\aaa)$ and $\overline{\PPP(\aaa)}$ are quasi-separated. We deduce that~$\TTa$,~$\PPP(\aaa)$ and $\overline{\PPP(\aaa)}$ are all of finite presentation. 
\end{enumerate}
\end{proof} 
\begin{prop}
The stack~$\PPP(\aaa)$ is proper.
\end{prop}
\begin{proof}
Recall that a proper~$1$-morphism is a~$1$-morphism which is of finite type, separated and universally closed. In \ref{ppaqs}, we have verified that~$\PPP(\aaa)$ is of finite presentation, hence of finite type \cite[\href{https://stacks.math.columbia.edu/tag/06Q5}{Lemma 06Q5}]{stacks-project}, and of finite diagonal, hence separated.

We will now apply the valuative criterion for separated~$1$-morphisms with target locally Noetherian algebraic stacks to be universally closed \cite[\href{https://stacks.math.columbia.edu/tag/0CQM}{Lemma 0CQM}]{stacks-project}, to the~$1$-morphism $\PPP(\aaa)\to\Spec\ZZ.$ Let again~$R$ be a discrete valuation ring, let~$K$ be its field of fractions, let $v_R$ be its valuation and $\pi_R$ its uniformizer. We pick an object~$\xxx$ in the groupoid~$\PPP(\aaa)(K)$ and we prove that there exists a finite extension $K'$ of~$K$ and a valuation ring $R'\subset K'$ such that $\mathfrak m_R=\mathfrak m_{R'}\cap K,$ where $\mathfrak m_R$ and $\mathfrak m_{R'}$ are maximal ideals of~$R$ and $R'$ respectively, and such that the restriction $\xxx_{K'}\in\PPP(\aaa)(K'),$ is in the essential image of the functor $$\PPP(\aaa)(R')\to\PPP(\aaa)(K').$$ Let $\widetilde\xxx:(\Gm)_K\to(\AAA^n-\{0\})_K$ be the $(\Gm)_K$-equivariant morphism given by $\xxx.$ We set $\ell:=\lcm(\aaa)$. 
Let us set $K'=K(\pi_R^{1/\ell})$ and let $R'$ be the integral closure of~$R$ in $K'$. By \cite[\href{https://stacks.math.columbia.edu/tag/09EV}{Lemma 09EV}]{stacks-project}, one has that $R'=R[\pi_R^{1/\ell}]$, that $R'$ is a discrete valuation ring, and that $\pi_R^{1/\ell}$ is a uniformizer of $R'$. Thus the maximal ideal of $R'$ is given by $(\pi_R^{1/\ell})$ and its intersection with~$R$ is precisely the maximal ideal $(\pi_R)$ of~$R$. We extend canonically $v_R$ to $R'$ and $K'$. We set $k=-\min_j\frac{v_R(\widetilde x_j(1))}{a_j},$ where $\widetilde x_j(1)$ is the~$j$-th coordinate of $\wx(1)$. Note that one has that $$\pi _R^{k}\cdot \wx (1)=(\pi_R^{a_jk}\widetilde x_j(1))_j\in(\AAA^n-\{0\})(R'),$$ because for every index~$i$ one has $$v_R(\pi_R^{a_ik}\widetilde x_i(1))=-a_i\min_j\bigg(\frac{v_R(\widetilde x_j(1))}{a_j}\bigg)+a_iv_R(\widetilde x_i(1))\geq 0$$ and because for the index~$i$ such that $\frac{\widetilde x_i(1)}{a_i}$ is minimal one has that $$v_R(\pi_R^{a_ik}\widetilde x_i(1))=0.$$ We define a $(\Gm)_{R'}$-equivariant morphism by $$\widetilde\zzz:(\Gm)_{R'}\to(\AAA^n-\{0\})_{R'}\hspace{1cm}1\mapsto \pi^{k}_{R}\cdot\wx(1),$$ and $\widetilde\zzz$ defines a morphism $\zzz:\Spec(R')\to\PPP(\aaa)$.  
One has that $\pi_R^k=(\pi_R^{1/\ell})^{k\ell}\in K'$ and thus $\pi_{R}^k$ defines by the multiplication a morphism $(\Gm)_{K'}\to(\Gm)_{K'}$ which satisfies $\widetilde\zzz=\wx\circ\pi_R^k$. It follows that $\zzz_{K'}$ and $\xxx_{K'}$ are isomorphic. The valuative criterion is verified and $\PPP(\aaa)\to\Spec(\ZZ)$ is universally closed. It follows that the algebraic stack~$\PPP(\aaa)$ is proper.
 \end{proof} 
\section{Models with enough integral points}\label{Modelswithenough} In this section, we will define models of stacks which admit ''enough integral points" in order to define unstable heights on stacks.
\subsection{}In this paragraph we define models of stacks.
\begin{mydef}
Let~$X$ be a finite presentation algebraic stack over a number field~$F$ and let $\OOF\subset A\subset F$ be a ring.  A model of~$X$ over $\Spec(A)$ is a finite presentation~$A$-algebraic stack~$\XXX$ endowed with a~$1$-isomorphism $x:\XXX_F\xrightarrow{\sim}X$.
\end{mydef}
A base change of a~$1$-morphism of finite presentation is of finite presentation \cite[\href{https://stacks.math.columbia.edu/tag/06Q4}{Lemma 06Q4}]{stacks-project}. We deduce that if $(\XXX,x:\XXX_F\xrightarrow{\sim}X)$ is a model of~$X$ over $\Spec(A)$, for some $\OOF\subset A\subset F$, then for every $A'$ such that $A\subset A'\subset F$, one has that $(\XXX_{A'}, x:\XXX_F\xrightarrow{\sim}X)$ is a model of~$X$. The model is unique in the following sense.
\begin{lem}\label{unicitymodel}
Let~$X$ be a finite presentation~$F$-algebraic stack. Let $S_1$ and $S_2$ be finite sets of finite places of~$F$. Let $(\YYY,y:\YYY_F\xrightarrow{\sim}X)$ and $(\ZZZ,z:\ZZZ_F\xrightarrow{\sim} X)$ be models of~$X$ over $\OO_{F,S_1}$ and $\OO_{F,S_2}$, respectively. There exists a finite set $S\supset S_1\cup S_2$ of finite places of~$F$, a~$1$-isomorphism of stacks $f:\YYY_{\mathcal O_{F,S}}\xrightarrow{\sim}\ZZZ_{\mathcal O_{F,S}}$ and a $2$-isomorphism $y\xrightarrow{\sim}z\circ f_F.$ 
\end{lem}
\begin{proof}
Fix~$1$-inverses $y^{-1}:X\to \mathscr Y_F$ and $z^{-1}:X\to\ZZZ_F$. We set $S_0=S_1\cup S_2$ and $T_0=\Spec(\OO_{F,S_0}).$ For every finite subset $\Lambda\supset S_0$ of the set of finite places of~$F$, we set $T_{\Lambda}=\Spec(\OO_{F,\Lambda})$. The schemes $T_{\Lambda}$ form an inverse system and $$\varprojlim _{\Lambda}T_{\Lambda}=\Spec(\varinjlim_{\Lambda}\OO_{F,\Lambda})=\Spec (F).$$ Set $Y_0:=\YYY_{\OO_{F,S_0}}$ and $Z_0:=\ZZZ_{\OO_{F,S_0}}$, and for finite subset $\Lambda\supset S_0$ of finite places of~$F$, we set $Y_{\Lambda}:=Y_0\times_{T_0}T_{\Lambda}$ and $Z_{\Lambda}:=Z_0\times_{T_0}T_{\Lambda}$. Note that by the definition of the model and by the fact that the base change of finite presentation~$1$-morphism is of finite presentation, the stack $Y_0$ is quasi-compact and quasi-separated and the stack $Z_0$ is locally of finite presentation. We have verified the conditions of \cite[Proposition B2]{Rydh}. 
It follows that there exists finite subset $S'\supset S_0$ of the set of finite places of~$F$, a~$1$-morphism of stacks $f':Y_{S'}\to Z_{S'}$ and a $2$-isomorphism $f_F'\xrightarrow{\sim}z^{-1}\circ y.$ Hence, there exists a $2$-isomorphism $y\xrightarrow{\sim }z\circ f_F'.$ For every finite subset $\Lambda\supset S_0$, we set $f_{\Lambda}':\YYY_{\OO_{F,\Lambda}}=Y_{\Lambda}\to Z_{\Lambda}=\ZZZ_{\OO_{F,\Lambda}}$ for the base change morphism $f'\times _{\OO_{F,S'}}\OO_{F,\Lambda}$. For every finite subset $\Lambda\supset S_0$ of the set of finite places of~$F$, the stacks $Y_{\Lambda}$ and $Z_{\Lambda}$ are of finite presentation, thus by \cite[Proposition B3]{Rydh}, there exists $\Lambda$ big enough such that $f_{\Lambda}'$ is a~$1$-isomorphism. We set $S=\Lambda$ and $f=f_{\Lambda}'$. One clearly has that $f_F'=f_F$, thus there exists a $2$-isomorphism $y\xrightarrow{\sim}z\circ f_F$. The statement follows. 
\end{proof}
\begin{exam}
\normalfont The pair $(\PPP(\aaa)_{\OOF},\Id_{\PPP(\aaa)_F})$ is a model over $\Spec(\OOF)$ of the stack $\PPP(\aaa)_F=(\AAA^n-\{0\})_F/(\Gm)_F$. Indeed, it follows from \ref{wps} that $$\PPP(\aaa)_F=(\AAA^n-\{0\})_F/(\Gm)_F=((\AAA^n-\{0\})_{\OOF}/(\Gm)_{\OOF})_F=\PPP(\aaa)_F$$and from \ref{ppaqs} and from \cite[\href{https://stacks.math.columbia.edu/tag/06Q4}{Lemma 06Q4}]{stacks-project} that $\PPP(\aaa)_{\OO_F}=\PPP(\aaa)\times_{\ZZ}\OOF$ is of finite presentation. An analogous argument shows that $\oPPa_F=\AAA^n_F/(\Gm)_F$ admits a model $(\oPPa,\Id_{\oPPa_F})$ over $\Spec(\ZZ)$. 
\end{exam}
\subsection{} 
We propose the following definitions to have sufficiently $\Ov$-integral points to define unstable heights on stacks. 
\begin{mydef} Let~$v$ be a finite place of~$F$ and let $\OOF\subset A\subset \Ov$ be a ring. Let~$X$ be a finite presentation~$A$-algebraic stack. We say that~$X$ admits enough $\Ov$-integral points if the canonical functor $X(\Ov)\to X(F_v)$ is essentially surjective.
\end{mydef}
\begin{mydef} \label{modelenough}Let~$X$ be an~$F$-algebraic stack of finite presentation. Let~$S$ be a finite set of finite places of~$F$ and $(\XXX,x)$ be a model of~$X$. We say that $(\XXX,x)$ has enough integral points, if for every finite place~$v$ of~$F$ which is not in~$S$, the stack $\XXX_{\Ov}$ has enough $\Ov$-integral points. 
\end{mydef}
Note that in the situation of \ref{modelenough}, the property ''has enough integral points" is in fact a property of~$\XXX$. It follows that if $(\YYY,y)$ is another model of~$X$ such that there exists an equivalence $\XXX\xrightarrow{\sim}\YYY,$ then $(\YYY,y)$ admits has enough integral points. For~$v$ not in~$S$, every $\Fv$-point ''extends" to an $\Ov$-point of~$\XXX$ in the following sense: the functor \begin{multline*}\XXX(\Ov)=\XXX_{\Ov}(\Ov)\to\XXX(\Fv)= \XXX_{\Fv}(\Fv)=\XXX_F(\Fv)\xrightarrow{x(\Fv)} X_{\Fv}(F_v)\\=X(F_v) \end{multline*} is essentially surjective (this follows from the fact that $x(F_v):\XXX_F(\Fv)\to X_{\Fv}(F_v)$ is an equivalence).
\begin{lem}\label{notov}
Suppose for some index~$i$, one has $a_i>1$. Let~$v$ be a finite place of~$F$. The $\Ov$-stack $\PPP(\aaa)_{\Ov}$ does not have enough $\Ov$-integral points.
\end{lem}
\begin{proof}
We prove that the point $\xxx:q^{\aaa}_{\Fv}(F_v)(\piv\doots \piv)\in\PPP(\aaa)_{\Fv}(\Fv)$ is not in the essential image of the canonical functor $\PPP(\aaa)_{\Ov}(\Ov)\to\PPP(\aaa)_{\Fv}(\Fv)$. The group scheme~$\Gm$ is special and let $\widetilde\xxx:(\Gm)_{\Fv}:(\Gm)_{\Fv}\to(\AAA^n-\{0\})_{\Fv}$ be the $(\Gm)_{\Fv}$-equivariant morphism defined by~$\xxx$. If $\yyy\in\PPP(\aaa)_{\Fv}(\Fv)$, an isomorphism $\xxx\xrightarrow{\sim}\yyy$ is given by an element $t\in\Gm(\Fv)$ such that $\widetilde\xxx=\widetilde\yyy\circ t$, where $\widetilde\yyy:(\Gm)_{\Fv}\to(\AAA^n-\{0\})_{\Fv}$ is the $(\Gm)_{\Fv}$-equivariant morphism given by~$\yyy$ and~$t$ is seen as a morphism $(\Gm)_{\Fv}\to(\Gm)_{\Fv}$ by multiplication. It follows that if~$\yyy$ is isomorphic to~$\xxx$, then \begin{multline*}\widetilde\yyy(1)\in\{\widetilde \xxx(t)|t\in\Gm(\Fv)\}=\{t\cdot\wx(1)|t\in\Gm(\Fv)\}\\=\{(t^{a_j}\piv)_j|t\in\Gm(\Fv)\}.\end{multline*} On the other side, if~$\yyy$ is the image of an $\Ov$-point for the canonical morphism $\PPP(\aaa)_{\Ov}(\Ov)\to\PPP(\aaa)_{\Fv}(\Fv)$, it follows that $\widetilde\yyy$ extends to a $(\Gm)_{\Ov}$-equivariant morphism $(\Gm)_{\Ov}\to(\AAA^n-\{0\})_{\Ov}$ and in particular that $\widetilde \yyy(1)\in(\AAA^n-\{0\})_{\Ov}(\Ov)$. We will show that the sets $$(\AAA^n-\{0\})_{\Ov}(\Ov)=\{(z_j)_j\in\Ov^n|\exists j: v(z_j)=0\}$$ and $$\{(t^{a_j}\piv)_j|t\in\Gm(\Fv)\}$$ are disjoint. Suppose that $(t^{a_1}\piv\doots t^{a_n}\piv)\in(\AAA^n-\{0\})_{\Ov}(\Ov)$ for some $t\in\Gm(\Fv)$. One has $v(t^{a_i}\piv)=a_iv(t)+1\geq 0$ and as $a_i>1$, we deduce $v(t)\geq 0.$ Now for every index~$j$ one has $v(t^{a_j}\piv)=a_jv(t)+1>0$, a contradiction with the assumption that $(t^{a_1}\piv\doots t^{a_n}\piv)\in(\AAA^n-\{0\})(\Ov)$. We deduce that~$\xxx$ is not in the essential image of $ \PPP(\aaa)_{\Ov}(\Ov)\to\PPP(\aaa)_{\Fv}(\Fv)$ and consequently $\PPP(\aaa)_{\Ov}$ does not have enough $\Ov$-points. 
\end{proof}
\begin{cor} Suppose for some index~$i$, one has $a_i>1$. For any finite subset~$S$ of the set of the finite places of~$F$, there exists no model $(\XXX,x)$ of $\PPP(\aaa)_F$ over $\Spec(\OFS)$ such that the following condition is satisfied: there exists a finite place $v\not\in S$ such that the stack~$\XXX$ admits enough $\Ov$-points.
\end{cor}
\begin{proof}
Suppose there exists finite set~$S$ of finite places of~$F$ and a model $(\XXX,x)$ of $\PPP(\aaa)_F$ such that for every finite place $v\not\in S$ one has that~$\XXX$ admits enough $\Ov$-integral points. By \ref{unicitymodel}, we can increase~$S$ if needed and find a~$1$-isomorphism $f:\XXX\xrightarrow{\sim}\PPP(\aaa)_{\OO_{F,S}}$ and a $2$-isomorphism $x\xrightarrow{\sim}\Id _{\PPP(\aaa)_F}\circ f_F=f_F$. Let $f^{-1}:\PPP(\aaa)_{\OO_{F,S}}\to\XXX$ be an inverse to~$f$. By \ref{notov}, for every finite $v\not\in S$, one has that $$\XXX(\Ov)\xrightarrow{f(\Ov)}\PPP(\aaa)_{\OO_{F,S}}(\Ov)=\PPP(\aaa)_{\Ov}(\Ov)\to \PPP(\aaa)_F(F)\xrightarrow{f_F^{-1}(F)}\XXX_F(F)$$is not essentially surjective. We obtain a contradiction and the claim follows\end{proof}
\begin{prop}\label{whyenough}The stack $\overline{\PPP(\aaa)}_{\OO_F}=\AAA^n_{\mathcal O_F}/(\Gm)_{\mathcal O_F}$ is a model of $\overline{\PPP(\aaa)_F}$ which for every $v\in M_F^{0}$ has enough $\Ov$-integral points.\end{prop} 
\begin{proof}
Let $\xxx\in\overline{\PPP(\aaa)}_{\Fv}(\Fv)$  and let $\wx:(\Gm)_{\Fv}\to\AAA^n_{\Fv}$ be the $(\Gm)_{\Fv}$-equivariant morphism defined by~$\xxx$. 
Let $\vMFz$. By the fact that all $a_j$ are positive, there exists $k\in\ZZ$ such that for every $j=1\doots n$ one has that $v(\widetilde x_j(1)+a_jk)> 0$. 
The $(\Gm)_{\Fv}$-equivariant morphism given by $$(\Gm)_{\Fv}\to\AAA^n_{\Fv}\hspace{1cm} 1\mapsto \piv^k\cdot\wx(1)$$ is isomorphic to~$\wx$ and is the base change of the $(\Gm)_{\Ov}$-equivariant morphism $$\wx_{\Ov}:(\Gm)_{\Ov}\to \AAA^n_{\Ov}\hspace{2cm}1\mapsto \piv^{k}\cdot\widetilde\xxx(1)$$ along $\Spec\Fv\to\Spec\Ov$. The $(\Gm)_{\Ov}$-equivariant morphism defines a morphism $\xxx_{\Ov}:\Spec\Ov\to\overline{\PPP(\aaa)}_{\Ov}$. By construction we have $\xxx_{\Ov}\times_{\ZZ}\Fv\cong \xxx$. 
\end{proof}
The stacks $\overline{\PPP(\aaa)}$ are not proper because they are not separated as the following lemma shows.
\begin{lem}\label{oPPanotsep}
Let~$v$ be a finite place of~$F$. Let $\OOF\subset A\subset\Ov$ be a ring. The canonical morphism $\overline{\PPP(\aaa)}_{A}\to\Spec(A)$ is not separated. 
\end{lem}
\begin{proof}
By the fact that property of being separated is stable for a base changes, one can assume that $A=\Ov$. 
We will verify that the diagonal $\Delta_{\oPPa_{\Ov}}:\oPPa_{\Ov}\to\oPPa_{\Ov}\times_{\Ov}\oPPa$ is not proper. The diagram 
\[\begin{tikzcd}
	{(\Gm)_{\Ov}\times \mathbb{A}_{\Ov}^n} & {\mathbb A_{\Ov}^n\times\mathbb A_{\Ov}^n} \\
	\oPPa_{\Ov} & {\overline{\mathscr P(\mathbf a)}_{\Ov}\times_{\Ov}\oPPa_{\Ov}.}
	\arrow["{(q\circ p_2)_{\Ov}}"', from=1-1, to=2-1]
	\arrow["{(a,p_2)_{\Ov}}", from=1-1, to=1-2]
	\arrow["{(q,q)_{\Ov}}"', from=1-2, to=2-2]
	\arrow["\Delta_{\oPPa_{\Ov}}"', from=2-1, to=2-2]
\end{tikzcd}\]
is $2$-commutative $2$-cartesian. By the fact that being proper is stable for a base change, it suffices to see that $(a,p_2)_{\Ov}:(\Gm)_{\Ov}\times \AAA_{\Ov}^n\to \AAA_{\Ov}^n\times\AAA_{\Ov}^n$ is not proper. We verify that the valuative criterion of properness \cite[\href{https://stacks.math.columbia.edu/tag/0BX5}{Lemma 0BX5}]{stacks-project} is not satisfied for the finite type and quasi-separated morphism $(a,p_2)_{\Ov}$. The diagram 
\[\begin{tikzcd}
	{\Spec (\Fv)} & {(\Gm)_{\Ov}\times_{\Ov}\AAA^n_{\Ov}} \\
	{\Spec(\Ov)} & {\AAA^n_{\Ov}\times_{\Ov}\AAA^n_{\Ov}}
	\arrow["{(\pi_v,(1)_j)}", from=1-1, to=1-2]
	\arrow["{(a,p_2)_{\Ov}}", from=1-2, to=2-2]
	\arrow["{((\pi_v^{a_j})_j,(1)_j)}"', from=2-1, to=2-2]
	\arrow[from=1-1, to=2-1]
\end{tikzcd}\]
does not admit admit an arrow $\Spec(\Ov)\to (\Gm)_{\Ov}\times _{\Ov}\AAA^n_{\Ov}$ so that the diagram commutes. Indeed if $(t,\xxx):\Spec(\Ov)\to(\Gm)_{\Ov}\times_{\Ov}\AAA^n_{\Ov}$ was a such an arrow, then $\xxx=(1)_j$ and $v(t)=0$. One has that $a\circ (t,(1)_j)=(\piv^{a_j})_j$, thus $v(t^{a_1})=a_1v(t)=0\neq a_1=v(\piv^{a_1})$, a contradiction. It follows that $(a,p_2)_{\Ov}$ is not proper, and hence that $\Delta_{\oPPa_{\Ov}}$ is not proper, i.e. that~$\oPPa$ is not separated.
\end{proof}
\section{Topology on~$R$-points of stacks} \label{condring}
We recall a definition, originally due to Moret-Bailly in \cite{Moret-BaillyS}, of a topology that one can put on~$R$-points of a stack. 
Let~$R$ be a local topological ring that satisfies the following conditions:
\begin{enumerate}[(a)]
\item the group of units $\Gm(R)$ is open in~$R$,
\item the inverse map $\Gm (R)\xrightarrow{x\mapsto x^{-1}}\Gm (R)$ is continuous, when $\Gm(R)\subset R$ is endowed with the subspace topology.
\end{enumerate}
We will call such a ring ``topologically suitable". The principal examples are $R=\Fv$ for $v\in M_F$ or $R=\Ov$ for $\vMFz$.
\subsection{}
The following proposition is Proposition 3.1 in \cite{Conrad}. We consider schemes that are locally of finite type over a suitable ring~$R$.
\begin{prop} [Conrad, {\cite[Proposition 3.1]{Conrad}}]  \label{topofsch} Let~$R$ be a topologically suitable ring. There exists a unique way to topologize $Y(R)$ for every scheme~$Y$ locally of finite type over~$R$ subject to the requirements of functoriality, carrying closed (open) immersions into embeddings (open embeddings) of topological spaces, compatibility with fiber products, and giving $Y (R)$ the usual topology when $Y $ is the affine line over~$R$. One also has that if~$Y$ is separated and~$R$ is Hausdorff, then $Y (R)$ is Hausdorff and that if~$R$ is Hausdorff and locally compact, then $Y (R)$ is locally compact. 
\end{prop}
The following suffices to make $Y(R)$ compact.
\begin{lem}\label{ukaln}
Suppose~$R$ is a topologically suitable ring.
\begin{enumerate}
\item Suppose~$R$ is compact. If~$Y$ is an~$R$-scheme of finite type, then $Y(R)$ is compact.
\item $(\Conrad,$ {\cite[Corollary 5.7]{Conrad}}$)$ Suppose~$R$ is a local field. If~$Y$ is a proper~$R$-scheme, then $Y(R)$ is compact.
\end{enumerate}
\end{lem}
\begin{proof}
We prove (1). Take a finite Zariski open covering $\{U_i\}_i$ of~$Y$ with $U_i$ affine. Every affine scheme $U_i$ is a closed subscheme of an affine space $\AAA^{n_i}$. We deduce that $U_i(R)$ is a closed subset of a compact set $\AAA^{n_i}(R)=R^{n_i}$, hence is compact. Now, the sets $\{U_i(R)\}_i$ cover $Y(R)$, because~$R$ is local, and thus $Y(R)$ is compact.
\end{proof}
A direct consequence of \ref{topofsch} is the following corollary.
\begin{cor}\label{topgpagp} Let~$R$ be a topologically suitable locally compact Hausdorff ring. Suppose~$G$ is a locally of finite type algebraic group. Then $G(R)$ is locally compact group. If~$G$ is commutative, then $G(R)$ is commutative. If $a:G\times Y\to Y$ is an action to the left of~$G$ on a locally of finite type scheme~$Y$, then $a(R):G(R)\times Y(R)\to Y(R)$ is a continuous action of $G(R)$ on $Y(R)$.
\end{cor}
\subsection{}
As in the case of schemes, we work only with stacks that are locally of finite type. If~$X$ is an algebraic stack and~$R$ a ring by $[X(R)]$ we denote the set of isomorphism classes of objects in the groupoid $X(R)$. For $x\in X(R)$ we denote by $[x]$ its image in $[X(R)]$. The following definition is firstly given by Moret-Bailly in \cite[Definition 2.2]{Moret-BaillyS} for stacks with separated and quasi-compact diagonals. \v Cesnavi\v cius gives it for stacks without such hypothesis.
\begin{mydef}[\v Cesnavi\v cius,{\cite[Section 2.4]{Cesnavicius}}]
\label{MBdef} Let~$R$ be a topologically suitable ring. Let~$X$ be a locally of finite type~$R$-algebraic stack of separated diagonal. We endow $[X(R)]$ with the finest topology such that for every~$1$-morphism $f:Y\to X$, with~$Y$ a locally of finite type~$R$-scheme, the maps $[f(R)]$ are continuous.
\end{mydef}
The following lemma follows from properties of the finest topology.
\begin{lem}\label{bezmora}
A subset $U\subset[X(R)]$ is open if and only if for every~$1$-morphism $f:Y\to X$ of algebraic stacks, with~$Y$ locally of finite type~$R$-scheme, the preimage $[f(R)]^{-1}(U)$ is open in $Y(R)$. Let~$T$ be a topological space. A map $h:X(R)\to T$ is continuous if and only if for every~$1$-morphism $g:Y\to X$ of algebraic stacks, with~$Y$ locally of finite type~$R$-scheme, the composite map $h\circ g(R)$ is continuous.
\end{lem}
\begin{proof}
Those are consequences of \cite[Proposition 4, \no 4, \S 2, Chapter I]{TopologieGj} and \cite[Proposition 6, \no 4, \S 2, Chapter I]{TopologieGj}.  
\end{proof}
We recall some of properties which are proven in \cite{Cesnavicius} and which we are going to use. 
\begin{prop}[\v Cesnavi\v cius,{\cite[Corollary 2.7]{Cesnavicius}}] \label{ceso}Let $f:X\to W$  be a~$1$-morphism of~$R$-stacks that are locally of finite type, where~$R$ is a topologically suitable ring.
\begin{enumerate}
\item The induced map on~$R$-points $[f(R)]:[X(R)]\to [W(R)]$ is continuous.
\item Suppose~$f$ is an open immersion. Then $[f(R)]:[X(R)]\to [W(R)]$ is an open immersion.
\item Suppose~$R$ is Hausdorff and~$f$ is a closed immersion. Then the map $[f(R)]:[X(R)]\to [W(R)]$ is a closed immersion. 
\item Let $R'$ another topologically suitable ring. Let $h:R\to R'$ a continuous ring homomorphism. The canonical map $[X(R)]\to [X(R')]$ is continuous.
\end{enumerate}
\end{prop}
\subsection{}Let us study the topological spaces $[(Y/G)(R)]$, when the algebraic group~$G$ is special (see \ref{wps}). 
\begin{prop} Let~$R$ be a topologically suitable ring. Let $G=(G,m,e)$ be a flat locally of finite presentation~$R$-algebraic group. Let~$Y$ be a locally of finite type~$R$-scheme endowed with an action of~$G$ and let $\pi :Y\to Y/G$ be the quotient morphism. 
\begin{enumerate}
\item The map $[\pi(R)]:Y(R)\to [(Y/G)(R)]$ is $G(R)$-invariant and continuous. 
\item Assume~$G$ is special. The map $[\pi(R)]$ is surjective and open. The canonical continuous map $$\overline{[\pi(R)]}:Y(R)/G(R)\to [(Y/G) (R)]$$ induced from $G(R)$-invariant map $[\pi (R)]$ is a homeomorphism.
\end{enumerate}
\label{qstop}
\end{prop}
\begin{proof}
\begin{enumerate}
\item We prove that $[\pi(R)]$ is continuous and $G(R)$-invariant.
\begin{itemize} 
\item The fact that the map $[\pi(R)]$ is continuous follows from the functoriality (Proposition \ref{ceso}).
%
\item By \ref{propofstack}, the diagram 
\[\begin{tikzcd}
	{G\times _RY} & Y \\
	Y & {(Y/G),}
	\arrow["a", from=1-1, to=1-2]
	\arrow["\pi", from=2-1, to=2-2]
	\arrow["{p_2}", from=1-1, to=2-1]
	\arrow["\pi"', from=1-2, to=2-2]
\end{tikzcd}\]
where $a$ is the action, is $2$-commutative. It follows that if $(g,x)\in (G\times_R X)(R),$ then $$[\pi(R)](a(R)(g,x))=[\pi(R)](p_2(R)(g,x))=[\pi(R)(x)].$$ Thus $[\pi(R)]$ is $G(R)$-invariant.
\end{itemize}
\item We assume that~$G$ is special.
\begin{itemize}
\item We establish that $[\pi(R)]$ is surjective. Let $x:\Spec R\to Y/G$ be a~$1$-morphism of algebraic stacks. As~$G$ is special, by $\ref{xgtac}$, one has the following $2$-commutative diagram
\begin{equation*}\label{ojub}
\begin{tikzcd}
G_R \arrow[r,"\widetilde x"] \arrow[d,""]& Y\arrow[d,"\pi"] \\
 \Spec R\arrow[r,"x"] &Y/G.
\end{tikzcd}\end{equation*}
By $2$-commutativity, it follows that $\Spec R\xrightarrow{\pi\circ\widetilde x\circ e_R}Y/G$ and $\Spec R\xrightarrow{x} Y/G$ are $2$-isomorphic. We deduce that $[x]$ is the image of $[\pi(R)(\widetilde x(e_R))]$ and it follows that $[\pi(R)]$ is surjective.
\item We establish that the map $[\pi(R)]$ is open. Let $V\subset Y(R)$ be an open subset, we are going to prove that $[\pi(R)](V)$ is open in $[(Y/G)(R)]$.  By Definition \ref{MBdef}, we need to establish that if $s:W\to Y/G$ is a~$1$-morphism of stacks with $W$ a scheme, then $[s(R)]^{-1}([\pi(R)](V))$ is open in $W(R)$. Set $\widetilde W:=W\times _{Y/G}Y$ and set $\widetilde s :\widetilde W\to Y$ to be the base change morphism. The following diagram is commutative:\[\begin{tikzcd}
	{\widetilde W(R)} & {Y(R)} & {Y(R)/G(R)} \\
	{W(R)} & {[(Y/G)(R)]}
	\arrow["{\widetilde s(R)}", from=1-1, to=1-2]
	\arrow["{\pi_W(R)}"', from=1-1, to=2-1]
	\arrow["{s(R)}"', from=2-1, to=2-2]
	\arrow["{[\pi(R)]}", from=1-2, to=2-2]
	\arrow["q", from=1-2, to=1-3]
	\arrow["{\overline{[\pi(R)]}}", from=1-3, to=2-2]
\end{tikzcd}\] The morphism $\pi _{Y}:\widetilde W\to W$ is a~$G$-torsor, hence, as~$G$ is special, it is locally Zariski trivial on $W$. Let $\cup _{i\in I}U_i$ be an open covering of $W$, such that for all~$i$, the morphism $\pi_ W|_{\pi_W^{-1}(U_i)}$ is a trivial~$G$-torsor. For all~$i$, the map $\pi _i:{\pi (W)^{-1}(U_i)(R)}\to U_i(R)$ decomposes as ${\pi (W)^{-1}(U_i)(R)}\xrightarrow{\sim}U_i(R)\times G(R)\to U_i(R),$ where the first morphism comes from an isomorphism of~$G$-torsors ${\pi (W)^{-1}(U_i)}\xrightarrow{\sim} U_i\times G$ and the second map is the projection, and, hence, $\pi_i$ is open and surjective. As $\cup _{i\in I}U_i(R)$ is a covering of $W(R)$, the map $\pi _W(R)$ is open and surjective. We have that \begin{align*}\pi_W(R)^{-1}(s(R)^{-1}([\pi (R)](V)))&=\widetilde s(R)^{-1}([\pi(R)]^{-1}([\pi(R)](V)))\\&=\widetilde s(R)^{-1}(q^{-1}(q(V))),\end{align*} where the last equality follows from the fact that $\overline{[\pi(R)]}$ is a bijection. It follows that $\pi_W(R)^{-1}(s(R)^{-1}([\pi (R)](V)))$ is open in $\widetilde W(R)$, as $q$ is open and continuous and $\widetilde s(R)$ is continuous. Finally, we get that $$\pi_W(R)(\pi_W(R)^{-1}(s(R)^{-1}([\pi (R)](V))))=[s(R)]^{-1}([\pi(R)])$$ is open in $W(R)$. We deduce that $[\pi(R)](V)$ and hence $[\pi(R)]$ are open. 
\item Let us verify that $\overline{[\pi(R)]}$ is a bijection. Denote by $q$ the quotient map $Y(R)\to Y(R)/G(R).$ One has that $[\pi(R)]=\overline{[\pi(R)]}\circ q$. The map $\overline{[\pi(R)]}$ is surjective because $\pi(R)$ is surjective. Let us verify that $\overline{[\pi(R)]}$ is injective. Suppose $x,y$ are such that $[\overline{\pi(R)}(x)]=[\overline{\pi(R)}(y)]$. Let $x', y'\in Y(R)$ be lifts of $x,y$, respectively. We have that $[\pi(R)](x')=[\pi(R)](y'),$ hence, $\pi(R)(x')$ and $\pi(R)(y')$ are isomorphic in the groupoid $(Y/G)(R)$. This means precisely that there exists $g\in G(R)$ such that the $G_R$-equivariant morphisms $\widetilde x:G_R\to Y, e_R\mapsto x'$ and $\widetilde y:G_R\to Y, e_R\mapsto y'$ satisfy that $\widetilde x=\widetilde y\circ g,$ where $g\in G(R)$ is seen as a morphism $G_R\to G_R$ via the left multiplication. As $\widetilde y$ is $G_R$-equivariant, we deduce $x'=\widetilde x(e_R)=g\cdot\widetilde y(e_R)=g\cdot y'$. This means that $x=q(x')=q(g\cdot y')=q(y')=y$.
 It follows that $\overline{[\pi(R)]}$ is injective and hence bijective. It follows that $\overline{[\pi(R)]}$ is a bijection. As $[\pi(R)]$ is open and continuous, it follows that $\overline{[\pi(R)]}$ is a homeomorphism. The statement is now proven.
\end{itemize}
\end{enumerate}
\end{proof}
By Hilbert 90 theorem, the algebraic group~$\Gm$ is special. We can establish that:
\begin{cor}\label{pafvtafv}
Let~$R$ be a topologically suitable ring. 
\begin{enumerate}
\item The map $(\AAA^n-\{0\})(R)\to [\PPP(\aaa)(R)]$ is $\Gm(R)$-invariant, continuous and open, and the induced map
 \begin{equation*}
(\AAA^n-\{0\})(R)/\Gm(R)\xrightarrow{\sim}[\PPP(\aaa)(R)]\\
\end{equation*} is a homeomorphism.
\item The map $$\Gm^n(R)=(\AAA^1-\{0\})^n(R)\to [\TTa(R)]$$ is a $\Gm(R)$-invariant, continuous and open map and the induced map \begin{equation}\Gm^n(R)/\Gm(R)\rightarrow[\TTa(R)]\label{brtv}\end{equation} is a homeomorphism. 
\item The inclusion $[\TTa(R)]\subset[\PPP(\aaa)(R)]$ is an open embedding.
\end{enumerate}
\end{cor}
\begin{proof}
The first two claims are direct consequences of \ref{qstop}. The last claim is a consequence of \ref{ceso}.
\end{proof}
\subsection{}\label{prosclos} To say more about spaces $[\TTa(R)]$ and $[\PPP(R)]$, we will need additional assumptions on~$R$. In \cite[Section 2.12]{Cesnavicius}, \v Cesnavi\v cius defines a notion of proper-closed ring: it is a topologically suitable ring~$R$ such that for every proper morphism $f:X\to Y$ of finite type~$R$-schemes, the induced continuous map $f(R):X(R)\to Y(R)$ is closed. The following examples are presented: a local field or the ring of integers $\Ov$ in the completion $\Fv$ for a finite place~$v$.
\begin{prop}\label{aboutopa} Let~$R$ be a proper-closed integral domain. (e.g. the completion $\Fv$ for some $\vMF$ or the ring of integers $\Ov$ in the completion $\Fv$ for some $\vMFz$).
\begin{enumerate}
\item The topological actions $\Gm(R)\times(\AAA^n-\{0\})(R)\to(\AAA^n-\{0\})(R)$ and $\Gm(R)\times\Gmn(R)\to\Gmn(R),$ deduced from the actions $\Gm\times(\AAA^n-\{0\})\to(\AAA^n-\{0\})$ and $\Gm\times\Gmn\to\Gmn$ by Corollary \ref{topgpagp}, are proper.
\item The map $\Gm(R)\to\Gmn(R), t\mapsto (t^{a_j})_j$ is proper, and the subgroup $\Gm(R)_{\aaa}=\{(t^{a_j})_j|t\in\Gm(R)\}$ is a closed subgroup of $\Gmn(R)$. The canonical map $[\TTa(R)]=\Gmn(R)\to\Gmn(R)/\Gm(R)_{\aaa}$ is $\Gm(R)$-invariant, continuous, open and surjective, and the induced map \begin{equation}\label{ntbrve}[\TTa(R)]=\Gmn(R)/\Gmn(R)\to \Gmn(R)/\Gm(R)_{\aaa}\end{equation} is a homeomorphism.
\item Suppose~$R$ is locally compact (e.g. the completion $\Fv$ for some $\vMF$ or the ring of integers $\Ov$ in the completion $\Fv$ for some $\vMFz$), then $[\PPP(\aaa)(R)]$ and $[\TTa(R)]$ are locally compact and Hausdorff.
\end{enumerate}
\end{prop}
\begin{proof}  
\begin{enumerate}
\item As~$R$ is proper-closed and as $\Gm\times(\AAA^n-\{0\})\to(\AAA^n-\{0\})\times(\AAA^n-\{0\})$ is proper (Lemma \ref{aaaction}), it follows that the map \begin{equation}\label{gmran}\Gm(R)\times(\AAA^n-\{0\})(R)\to(\AAA^n-\{0\})(R)\times(\AAA^n-\{0\})(R)\end{equation} is closed. Let us verify that it's fibers are finite. Suppose that $(t,\xxx)$ is a preimage of $(\yyy,\zzz)$. This means that $\xxx=\zzz$ and that $t\cdot\xxx=\yyy$. Let~$i$ be an index such that $x_i\neq 0$. As~$R$ is an integral domain, there are only finitely many elements $t\in R$ for which $t^{a_i}x_i=z_i$. We deduce that (\ref{gmran}) has finite fibers. Now, it follows from \cite[Theorem 1,  \no 2,  \S 10, Chapter III]{TopologieGj}, that the map (\ref{gmran}) is proper. We deduce from \cite[Example 2, \no 1, \S 4, Chapter III]{TopologieGj}, that the restriction of the action of $\Gm(R)$ to the $\Gm(R)$-invariant subset $\Gmn(R)\subset(\AAA^n-\{0\})(R)$ is proper. The claim is proven.
\item By (1) the action of $\Gm(R)$ on $\Gmn(R)$ is proper. Thus by \cite[Proposition 4, \no 2, \S 4, Chapter III]{TopologieGj}, the induced map $t\mapsto t\cdot (1)_j=(t^{a_j})_j$ is proper and its image $\Gm(R)_{\aaa}$ is closed in $\Gmn(R)$. The map is $\Gmn(R)\to \Gmn(R)/\Gm(R)_{\aaa}$ is continuous, open and surjective, because it is a quotient map. It follows that the induced map $[\TTa(R)]=\Gmn(R)/\Gm(R)\to \Gmn(R)/\Gm(R)_{\aaa}$ is continuous, open and surjective. Moreover, note that for $t\in\Gm(R)$ and $\xxx\in\Gmn(R)$ one has that the image of $t\cdot\xxx=(t^{a_j})_j\xxx$ in $\Gmn(R)/\Gm(R)_{\aaa}$ coincides with the image of~$\xxx$ in $\Gmn(R)/\Gm(R)_{\aaa}$. Observe that if $\xxx,\yyy\in\Gmn(R)$ have the same image in $\Gmn(R)/\Gm(R)_{\aaa}$, then there exists $(t^{a_j})_j\in\Gm(R)_{\aaa}$ such that $(t^{a_j})_j\xxx=\yyy$, hence $t\cdot\xxx=\yyy$. It follows that the induced map $\Gmn(R)/\Gmn(R)\to \Gmn(R)/\Gm(R)_{\aaa}$ is injective. We deduce that it is a homeomorphism.
\item  The action of $\Gm(R)$ on $(\AAA^n-\{0\})(R)$ and $\Gmn(R)$ is proper, hence the spaces $[\PPP(\aaa)(R)]$ and $[\TTa(R)]$ are Hausdorff by \cite[Proposition 3, \no 2, \S4, Chapter III]{TopologieGj}. It follows from \ref{topofsch} that the spaces $(\AAA^n-\{0\})(R)$ and $\Gmn(R)$ are locally compact. Now, 
by \cite[Proposition 9, \no 5, \S 4, Chapter III]{TopologieGj} imply that $[\PPP(\aaa)(R)]$ and $[\TTa(R)]$ are locally compact.  
 \end{enumerate}
\end{proof} 
We finish the paragraph by establishing that $[\PPP(\aaa)(\Ov)]$, $[\TTa(\Ov)]$ and $[\PPP(\Fv)]$ are compact and that $[\TTa(\Fv)]$ is paracompact. First, we prove the following lemma.
\begin{lem}\label{begdav}
Let~$v$ be a place of~$F$. If~$v$ is finite, we define $$\Dav:=(\Ov)^n- (\piv^{a_1}\Ov\times\cdots\times\piv^{a_n}\Ov)$$ and if~$v$ is infinite we define $$\Dav:=\{\xxx\in\Fvn|\hspace{0.1cm} ||\xxx||_{\max}=1\},$$ where $||\xxx||_{\max}=\max_j(|x_j|_v)$.
\begin{enumerate}
\item Suppose~$v$ is finite. The set $\Dav$ is an open, a closed and a compact subset of $\Fv^n-\{0\}$. 
\item Suppose~$v$ is finite and let $\xxx\in\Fvnz$. The set $\{k\in\ZZ|\piv^k\cdot\xxx\in\Ovn\}$ is non-empty and we define $$r_v(\xxx):=\inf\{k\in\ZZ| \piv^{k}\cdot \xxx\in\Ov^n\}.$$ One has that $\piv^{r_v(\xxx)}\cdot\xxx\in\Dav.$
\item Suppose~$v$ is infinite. The set $\Dav$ is a compact subset of $\Fv^n-\{0\}$.
\item Suppose~$v$ is infinite and let $\xxx\in\Fvnz$.  There exists $t\in\Fvt$ such that $||t\cdot\xxx||_{v,\max}=1.$ 
\end{enumerate}
\end{lem}
\begin{proof}
\begin{enumerate}
\item The subset $\DD^{\aaa}_v$ writes as $$\Dav=(\Ov)^n\cap (\piv^{a_1}\Ov\times\cdots\times\piv^{a_n}\Ov)^c. $$ As $(\Ov)^n$ and $(\piv^{a_1}\Ov\times\cdots\times\piv^{a_n}\Ov)^c $ are a ball and a complement of a ball in $F^n_v$, they are both open and closed subsets of $F^n_v$. Hence, the subset $\Dav$ is both open and closed in $F^n_v$ and, as $F^n_v-\{0\}$ is open in $\Fvn$, also in $\Fvnz$. Moreover, $\Dav$ is a closed subset of $(\Ov)^n$, hence $\Dav$ is compact.  
\item As all $a_j$ are strictly positive, there exists a positive integer~$\ell$ such that for all~$j$ one has $a_j\ell>-v(x_j).$ For such~$\ell$ one has $\piv^{\ell}\cdot\xxx\in\Ov^n$, thus $\{k\in\ZZ|\piv^k\cdot\xxx\in\Ov^n\}$ is non-empty. Suppose that $\piv^{r_v(\xxx)}\cdot\xxx\in(\piv^{a_1}\Ov\times\cdots\times\piv^{a_n}\Ov)$. One has that $$\piv^{r_v(\xxx)-1}\cdot\xxx=\piv^{-1}\cdot(\piv^{r_v(\xxx)}\cdot\xxx)\in\piv^{-1}\cdot(\piv^{a_1}\Ov\times\cdots\times\piv^{a_n}\Ov) =\Ov^n,$$ a contradiction. The claim follows.
\item The set $\Dav$ is a sphere for the norm $|\lvert\cdot\rvert|_{v,\max}$ in the finite dimensional $\Fv$-vector space $\Fv^n$, hence is compact.
\item The function $$\Fvt\to\RR_{>0}\hspace{1cm}t\mapsto||t\cdot\xxx||_{v,\max}$$ is continuous. From the fact that all $a_j$ are strictly positive it follows that $$\lim_{|t|_v\to 0}||t\cdot\xxx||_{v,\max}=0$$ and that $$\lim_{|t|_v\to\infty}||t\cdot\xxx||_{v,\max}=+\infty.$$ We deduce that there exists $t\in\Fvt$ such that $t\cdot\xxx\in\Dav$.
 \end{enumerate}
\end{proof} 
When $R=\Fv,$ where~$v$ is a place of~$F$, one can say the following.
\begin{prop}\label{paraap}
Let $\vMF$. One has that:
\begin{enumerate}
\item the space $[\PPP(\aaa)(\Fv)]$ is compact;
\item the space $[\TTa(\Fv)]$ is paracompact;
\item the spaces $[\PPP(\aaa)(\Ov)]$ and $[\TTa(\Ov)]$ are compact.
\end{enumerate}
\end{prop}
\begin{proof}
\begin{enumerate}
\item Suppose firstly that $v\in M_F^0$. It follows from \ref{begdav} that the restriction of the quotient map $[q^{\aaa}(\Fv)]$ to $\Dav$ is surjective and that $\Dav$ is compact. We deduce that $[\PPP(\aaa)(\Fv)]$ is compact. 
\item The group $\Gmn(R)$ is locally compact and $(\Gmn(R))_{\aaa}$ is its closed subgroup. The quotient $[\TTa(R)]=\Gmn(R)/(\Gm(R))_{\aaa}$ is paracompact by \cite[Proposition 13, \no 6, \S4, Chapter III]{TopologieGj}.
\item The spaces $(\AAA^n-\{0\})(\Ov)$ and $\Gmn(\Ov)$ are compact by \ref{ukaln}. Therefore the corresponding quotients by $\Gm(\Ov)$ are compact. 
\end{enumerate}
\end{proof}
\subsection{} \label{gptta} The last paragraph of this section is dedicated to the group structure of $[\TTa(R)]$.

If~$R$ is a proper-closed integral domain, by \ref{aboutopa}, one has a homeomorphism of topological spaces $[\TTa(R)]=\Gmn(R)/\Gmn(R)\to \Gmn(R)/\Gm(R)_{\aaa}$ (induced from $\Gm(R)$-invariant homomorphism $\Gmn(R)\to\Gmn(R)/\Gm(R)_{\aaa}$). We will transfer the structure of an abelian group to $[\TTa(R)]$ using the inverse of this isomorphism and we may write $[\TTa(R)]=\Gmn(R)/(\Gm(R)_{\aaa})$. If, furthermore,~$R$ is assumed to be locally compact, then $[\TTa(R)]$ is a locally compact abelian group. 
\begin{lem}\label{tavissub} \begin{enumerate}
\item Suppose that $h:R\to R'$ is a morphism of rings. The canonical map $[\TTa(h)]:[\TTa(R)]\to[\TTa(R')]$ is a homomorphism.
\item Suppose that $R\to R'$ is an injective morphism of rings. The canonical map $[\TTa(R)]\to[\TTa(R')]$ is injective if and only if $\Gm(R)_{\aaa}=\Gm(R)^n\cap\Gm(R')_{\aaa}$.
\item Suppose that $h:R\to R'$ is a continuous map of topologically suitable rings. The canonical map $[\TTa(h)]:[\TTa(R)]\to[\TTa(R')]$ is a continuous homomorphism.
\item Suppose that $R\to R'$ is an open embedding of proper-closed integral domains, then the canonical map $[\TTa(R)]\to[\TTa(R')]$ is open.
\end{enumerate}
\end{lem}
\begin{proof}
\begin{enumerate}
\item The map $[\TTa(R)]\to[\TTa(R')]$ is the induced map from $\Gm(R)_{\aaa}$-invariant homomorphism \begin{equation}\label{bezovfv}\Gmn(R)\to\Gmn(R')\to[\TTa(R')],\end{equation} hence is a homomorphism. 
\item The kernel of (\ref{bezovfv}) is given by $\Gm(R')_{\aaa}\cap \Gmn(R)$ and it contains $\Gm(R)_{\aaa}$. The induced map $[\TTa(R)]\to[\TTa(R')]$ from $\Gm(R)_{\aaa}$-invariant map (\ref{bezovfv}) is injective if and only if $\Gm(R')_{\aaa}\cap \Gmn(R)=\Gm(R)_{\aaa}$.
\item The map $[\TTa(R)]\to[\TTa(R')]$ is continuous by \ref{ceso} is a homomorphism by (1). 
\item By \cite[Section 2.2, parts (vii) and (x)]{Cesnavicius}, the inclusion $\Gmn(R)\to \Gmn(R')$ is continuous and open. The map $[\TTa(R)]\to[\TTa(R')]$ is the induced map from the $\Gm(R)_{\aaa}$-invariant, continuous and open map (\ref{bezovfv}), thus is open.
\end{enumerate}
\end{proof}
\begin{lem}\label{fvtaovta}
Let $\vMFz$. One has that  $$\Gm(\Fv)_{\aaa}\cap \Gmn(\Ov)=(\Gm(\Ov))_{\aaa}.$$
\end{lem}
\begin{proof}
It is obvious that $\Gm(\Ov)_{\aaa}\subset \Gm(\Fv)_{\aaa}\cap \Gmn(\Ov)$ and let us prove the inverse inclusion. Let $\xxx\in\Gm(\Fv)_{\aaa}\cap \Gmn(\Ov).$ This means that there exists $t\in\Gm(\Fv)$ such that for every~$j$ one has $t^{a_j}=x_j$ and that $x_j\in\Gm(\Ov).$  We deduce $v(t^{a_j})=a_jv(t)=x_j=0,$ and as $a_j>0$ we get $v(t)=0$. It follows that $t\in\Gm(\Ov)$ and hence $\xxx\in\Gm(\Ov)_{\aaa}$. 
\end{proof}
We are ready to prove that:
\begin{prop}\label{identofov}
Let $v\in M_F^0$. The map $[\TTa(\Ov)]\to [\TTa(\Fv)]$, induced from $(\Ov)_{\aaa}$-invariant map $q^{\aaa}_v|_{(\Ovt)^n}$, is continuous, injective and open homomorphism and induces an identification of $[\TTa(\Ov)]$ with an open and compact subgroup of $[\TTa(\Fv)]$.
\end{prop}
\begin{proof}
By applying (1), (3) and (4) of \ref{tavissub} to the inclusion $\Ov\to\Fv$, we obtain that $[\TTa(\Ov)]\to[\TTa(\Fv)]$ is a continuous and open homomorphism. 
Lemma \ref{fvtaovta} gives that $\Gm(\Fv)_{\aaa}\cap \Gmn(\Ov)=(\Gm(\Ov))_{\aaa}$ and thus by (2) of \ref{tavissub}, the map $[\TTa(\Ov)]\to[\TTa(\Fv)]$ is injective. Moreover, by \ref{paraap}, the topological group $[\TTa(\Ov)]$ is compact by \ref{paraap}. The claim is proven.
\end{proof}
\begin{lem}\label{indofov}
Let $\aaa\ZZ$ denotes the subgroup $\{(a_jx)_j|x\in\ZZ\}$ of $\ZZ^n$. The homomorphism \begin{equation}\label{fvtnzza}\Fvtn \to\ZZ^n\to\ZZ^n/(\aaa\ZZ),\end{equation}where the first homomorphism is given by $\xxx\mapsto (v(x_j))_j$ and the second homomorphism is the quotient one, is $(\Fvt)_{\aaa}$-invariant. The kernel of the induced homomorphism $[\TTa(\Fv)]\to\ZZ^n/(\aaa\ZZ)$ is $[\TTa(\Ov)]$.
\end{lem}
\begin{proof}
Note that the image of $(t^{a_j})_j\in (\Fvt)_{\aaa}$ in $\ZZ^n$ under the map $\Fvtn\to\ZZ^n$ from above is $(v(t^{a_j}))_j=(a_jv(t))_j$, and the image of $(a_jv(t))_j$ in $\ZZ^n/(\aaa\ZZ)$ under the quotient homomorphism is $0$. Thus the homomorphism  (\ref{fvtnzza}) is $\Fvt$-invariant. The kernel of $\Fvtn\to\ZZ^n$ is the subgroup $(\Ovt)^n\subset \Fvtn$. The kernel of the induced homomorphism $[\TTa(\Fv)]\to\ZZ^n/(\aaa\ZZ)$ is the image of $(\Ovt)^n$ (as $(\Ovt)^n$ is the kernel of $\Fvtn\to\ZZ^n$) under the quotient map. We have the following ``snake diagram":
\[\begin{tikzcd}
	&&& 0 \\
	1 & (\Ovt)_{\aaa} & (\Fvt)_{\aaa} & \aaa\ZZ & 0 \\
	1 & (\Ovt)^n & \Fvtn & \ZZ^n & 0 \\
	& {[\TTa(\Ov)]} & {[\TTa(\Fv)]} & \ZZ^n/\aaa\ZZ.
	\arrow[from=2-1, to=2-2]
	\arrow[from=3-1, to=3-2]
	\arrow[from=2-2, to=2-3]
	\arrow[from=3-2, to=3-3]
	\arrow[from=2-3, to=2-4]
	\arrow[from=3-3, to=3-4]
	\arrow[from=2-4, to=2-5]
	\arrow[from=3-4, to=3-5]
	\arrow[from=1-4, to=2-4]
	\arrow[from=2-2, to=3-2]
	\arrow[from=3-2, to=4-2]
	\arrow[from=3-3, to=4-3]
	\arrow[from=2-3, to=3-3]
	\arrow[from=2-4, to=3-4]
	\arrow[from=3-4, to=4-4]
	\arrow[from=4-3, to=4-4]
	\arrow[from=4-2, to=4-3]
\end{tikzcd}\]
Snake lemma gives that $$1\to[\TTa(\Ov)]\to[\TTa(\Fv)]\to\ZZ^n/\aaa(\ZZ)\to 0$$is exact.
The statement follows
\end{proof}
We add another fact that will be used later. When $n=1$, the spaces $[\TTa(\Fv)]$ is finite and discrete (finiteness and the fact that $(\Fvt)_m\subset \Fvt$ is closed imply that $(\Fvt)_m$ is open, hence discreteness follows):
\begin{lem}[{\cite[Corollary 5.8, Chapter II]{Neukirch}}]\label{ttafvfinite}
Suppose $n=1$ and that $a=a_1\in\ZZ_{\geq 1}$.
The space $[\TT(a)(\Fv)]=\Fvt/(\Fvt)_a$ is discrete and of the cardinality is $\frac{a}{|a|_v}|\mu_a(F_v)|,$ where $|\mu_a(F_v)|$ is the number of $a$-th roots of~$1$ in $\Fv$.
\end{lem}
\section{Adelic situation}
We define ``adelic space" $[\TTa(\AAF)]$ of the stack~$\TTa$. It is defined as a restricted product of $[\TTa(\Fv)]$ with the respect to open subgroups $[\TTa(\Ov)]\subset [\TTa(\Fv)]$ for $\vMFz$. It turns out that $[\TTa(\AAF)]$ is a locally compact abelian group.
\subsection{}
Let us recall some facts on restricted product homomorphisms.
\begin{lem} \label{mimip}
Let~$I$ be a set. Suppose for every $i\in I$ we are given locally compact abelian groups $G_i$ and $H_i$ and for every $i\in I'$, where $I'\subset I$ is a subset of finite complement, we are given open and compact subgroups $G_i'\subset G_i$ and $H_i'\subset H_i$. Suppose for $i\in I$ we are given continuous homomorphism $\phi_i:G_i\to H_i$ such that for every $i\in I'$ one has $\phi_i (G_i')\subset H_i'$. Let us set $G:=\sideset{}{'}\prod _{i\in I}G_i, $ with the respect to the open and compact subgroups $(G_i'\subset G_i)_{i\in I'}$ and let us set $H:=\sideset{}{'}\prod _{i\in I}H_i, $ with the respect to the open and compact subgroups $(H_i'\subset H_i)_{i\in I'}$
\begin{enumerate}
\item The topological groups~$G$ and~$H$ are locally compact.
\item The canonical inclusion $G\subset \prod_{i\in I}G_i$ is continuous.
\item The image of $G\subset \prod _{i\in I}G_i$ under $\prod _{i\in I}\phi_i$ lies in $H.$ We let $\phi :G\to H$ be the homomorphism induced by $\big(\prod _{i\in I}\phi _i\big)$.
\item The homomorphism $\phi:G\to H$ is continuous.
\item One has that $\ker(\phi)=\big(\prod_i\ker(\phi_i)\big)\cap G$.
\item Suppose for every $i\in I'$ one has $\phi_i(G_i')=H_i'$. Suppose further that for every $i\in I$, the homomorphism $\phi_i$ is surjective (respectively, open, isomorphism of topological groups). Then the homomorphism $\phi$ is surjective (respectively, open, isomorphism of topological groups).
\end{enumerate}  
\end{lem}
\begin{proof}
\begin{enumerate}
\item Let $(x_i)_i\in G$. There exists a subset $I_1\subset I'$ such that $I'-I_1$ is finite and such that for all $i\in I_1$ one has $x_i\in G_i'$. For $i\in I'-I_1$, one can pick a compact neighbourhood $U_i$ of $x_i$ in $G_i,$ because $G_i$ is locally compact. Then $\prod _{i\in I'-I_1}U_i\times\prod _{i\in I_1}G_i' $ is a compact neighbourhood of $(x_i)_i$ in~$G$. It follows that~$G$ is locally compact, and the same is true for~$H$.
\item A basis open subset of $\prod _{i\in I}G_i$ writes as $\prod_{i\in S}U_i\times\prod_{i\in I-S}G_i$, where~$S$ is finite. Its preimage in~$G$ is given by $$\bigcup_{\substack{T\text{ is finite}\\T\supset S}}\prod_{i\in S} U_i\times\prod_{i\in T-S}G_i\times\prod_{i\in I-T}G_i',$$hence is open. It follows that the canonical inclusion is continuous.
\item Suppose $(x_i)_i\in G$. There exists a subset $I_1\subset I'$ such that $I'-I_1$ is finite and such that for all $i\in I_1$ one has $x_i\in G_i'$. For $i\in I_1$, we have $\phi _i(x_i)\in H_i'$, hence $(\phi_i(x_i))_i\in H.$
\item It suffices to verify that the preimage under $\phi$ of a basis open subset of~$H$ is open in~$G$. A basis open subset of~$H$ is given by $\prod_{i\in J}U_i\times\prod_{i\in I-J}H_i'$, for some finite subset $J\subset I$ containing $I-I'$ and some open subsets $U_i\subset H_i$ for $i\in J$. We establish that every point  $(x_i)_i\in \phi^{-1}\big(\prod _{i\in J}U_i\times \prod _{i\in I-J}H_i'\big)$ admits an open neighbourhood contained in $\phi^{-1}\big(\prod _{i\in J}U_i\times \prod _{i\in I-J}H_i'\big)$. There exists a subset $I_1\subset I'$ such that $I'-I_1$ is finite and such that for all $i\in I_1$ one has $x_i\in G_i'$. We note that $$\prod _{i\in (J\cap I)-I_1}\phi _i^{-1}(U_i)\times \prod _{i\in I-I_1-J}\phi_i^{-1}(H_i')\times\prod _{i\in I_1}G_i' $$ is an open neighbourhood of $(x_i)_i$ contained in $\phi^{-1}\big(\prod _{i\in J}U_i\times \prod _{i\in I-J}H_i'\big) .$ The set $\phi^{-1}\big(\prod _{i\in J}U_i\times \prod _{i\in I-J}H_i'\big)$ is thus open and it follows that $\phi$ is continuous.
\item Let $(x_i)_i\in G.$ One has that $0=\phi((x_i)_i)=(\phi_i(x_i))_i$ if and only if $x_i\in\ker(\phi_i)$ for every $i\in I$, i.e. if and only if $(x_i)_i\in\prod_i\ker(\phi_i)$. It follows that $\ker(\phi)=\big(\prod_i \ker(\phi_i)\big)\cap G.$
\item Let us suppose that for every $i\in I'$ one has $\phi_i(G_i')=H_i'$. \begin{enumerate} 
\item Suppose $\phi_i$ are surjective and let $(x_i)_i\in H$. There exists $I_1\subset I'$ such that $I'-I_1$ is finite and such that $x_i\in H_i'$ for every $i\in I_1$. As $\phi _i(G_i)=H_i$ for every $i\in I_1$, we can pick $y_i\in G'_i$ such that $\phi _i(y_i)=x_i$. As maps $\phi _i$ are surjective for every $i\in I-I_1$, we can pick $y_i\in G_i$ such that $\phi _i(y_i)$ for every $i\in I-I_1$. It follows that $(y_i)_i$ is an element of~$G$ such that $\phi ((y_i)_i)=(x_i)_i$. Hence, $\phi$ is surjective.
\item Suppose $\phi _i$ are open. It suffices to prove that the image of a basis open subset of~$G$ is open in~$H$. A basis open subset of~$G$ is given by $\prod _{i\in J}U_i\times \prod_{i\in I-J} G_i',$ for some finite subset $J\subset I$  such that $J\supset I-I'$ and some open subsets $U_i\subset G_i$. We have that \begin{align*}\phi \big(\prod _{i\in J}U_i\times \prod_{i\in I-J} G_i'\big)&=\prod _{i\in J}\phi _i(U_i)\times \prod _{i\in I-J}\phi_i(G'_i)\\&=\prod _{i\in J}\phi _i(U_i)\times \prod _{i\in I-J}H_i'\end{align*} is open in~$H$. It follows that $\phi$ is open.
\item Suppose $\phi_i$ are isomorphisms of topological groups. It follows from above that $\phi$ is injective, surjective, continuous and open. Thus $\phi$ is an isomorphism of topological groups.
\end{enumerate}
\end{enumerate}
\end{proof}
\subsection{}\label{eiiie}
Let $n\geq 1$ be an integer and let $\aaa\in\ZZ^n_{> 0}$. We define an ``adelic space" of the stack~$\TTa$.

Let $\AAF^{\times}$ be the group of ideles of~$F$, that is, it is the restricted product $$\sideset{}{'}\prod _{v\in M_F}\Fvt,$$ with the respect to the family of open and compact subgroups $$(\Ovt\subset\Fvt)_{\vMFz}.$$
It follows from \ref{mimip} that the group $\AAFt$ is locally compact.
\begin{mydef}
We define$$[\TTT ^{\aaa}(\AAF)]:=\sideset{}{'}\prod _{v\in M_F}[\TTT ^{\aaa}(F_v)],  $$ where the restricted product is taken with the respect to the family of compact and open subgroups $$([\TTT ^{\aaa}(\Ov)]\subset[\TTT ^{\aaa}(F_v)])_{v\in M_F^0}. $$\end{mydef} 
For $\vMF$, let $\qav:\Fvtn\to \Fvtn/(\Fvt)_{\aaa}=[\TTa(\Fv)]$ be the quotient morphism. By \ref{identofov}, for every $\vMFz$, one has that $\qav(\Ovtn)=[\TTa(\Ov)]$ is an open and compact subgroup of $[\TTa(\Fv)]$. Lemma \ref{mimip} provides a homomorphism $$q^{\aaa}_{\AAF}=\big(\prod _{\vMF}\qav\big):(\AAF^{\times}) ^n\to [\TTa(\AAF)].$$
We apply \ref{mimip} to our situation an we get the following lemma.
\begin{lem}\label{mapqaf}
The abelian topological group $[\TTa(\AAF)]$ is locally compact. The map $q^{\aaa}_{\AAF}=\big(\prod _{\vMF}\qav\big):(\AAF^{\times}) ^n\to [\TTa(\AAF)]$ is continuous, open and surjective. The kernel of $q^{\aaa}_{\AAF}$ is the group $(\AAFt)_{\aaa}=\{(\xxx^{a_j})_j|\xxx\in\AAFt\}$ (in particular $(\AAFt)_{\aaa}\subset(\AAFt)^n$ is closed).
\end{lem}
\begin{proof}
The abelian topological group $[\TTa(\AAF)]$ is locally compact by \ref{mimip}. For every $v\in M_F$, the quotient map $q^\aaa_v$ is continuous, open and surjective. By \ref{mimip}, one has that $q^{\aaa}_{\AAF}$ is continuous, open and surjective. For $\vMF$, one has that $\ker (q^{\aaa}_v)=(\Fvt)_{\aaa}$ and Lemma \ref{mimip} gives that the kernel of $q^{\aaa}_{\AAF}$ is given by \begin{equation*}\ker(q^{\aaa}_{\AAF})=\big(\prod_v\ker(q^{\aaa}_v)\big)\cap (\AAF^\times)^n=\big(\prod_{v}(\Fvt)_{\aaa}\big)\cap(\AAFt)^n\end{equation*}
By \ref{fvtaovta}, for $x\in\Fvt$ one has that $(x^{a_j})_j\in(\Ovt)_{\aaa}$ if and only if $x\in\Ovt$. It follows that for $\xxx\in\big(\prod_v\Fvt\big)$, one has that $(\xxx^{a_j})_j\in(\AAFt)_{\aaa}$ if and only if $\xxx\in\AAFt$. We deduce that $$ \ker(q^{\aaa}_{\AAF})=\big(\prod_{v}(\Fvt)_{\aaa}\big)\cap(\AAFt)^n=(\AAFt)_{\aaa}.$$
\end{proof}
For $\vMF$, let $i_v:F^{\times}\to F_v^{\times}$ be the canonical inclusion. The homomorphism $$(\Ft)^n\xrightarrow{i_v^n}(\Ft)^n\to[\TTa(\Fv)]$$ is $\Ft$-invariant, and one deduces a homomorphism $[\TTa(i_v)]:[\TTa(F)]\to[\TTa(F_v)]$. 
Let $i:\Ft\to\AAA_F^{\times}$ be the product map $i=\prod_{\vMF}i_v$. It is well known that the image of $F^{\times}$ is discrete in $\AAFt$ (see e.g. \cite[Theorem 5-11]{Ramakrishnan}). The homomorphism $$(F^{\times})^n\xrightarrow{i}(\AAF^{\times} )^n\xrightarrow{q^{\aaa}_{\AAF}}[ \TTa(\AAF)]$$ is $(\Ft)$-invariant and we deduce homomorphism $[\TTa(i)]:[\TTa(F)]\to[\TTa(\AAF)]$. 
For every~$v$, the homomorphism $$\Fvtn\xrightarrow{\xxx\mapsto ((\xxx)_v,(\mathbf 1)_{w\in M_F-\{v\}})}(\AAFt)^n\to [\TTa(\AAF)]$$ is $\Fvt$-invariant, and we deduce a homomorphism $[\TTa(\Fv)]\to[\TTa(\AAF)]$.
\begin{lem}\label{ttaiproduct}
The map $[\TTa(i)]$ coincides with the product map $\prod_v[\TTa(i_v)]:[\TTa(F)]\to\prod_v[\TTa(\Fv)]$.
\end{lem}
\begin{proof}
For every $\vMF$, the following diagram is commutative
\[\begin{tikzcd}
	{(F^\times)^n} & {(\Fvt)^n} & {\hspace{0.1cm}(\AAFt)^n} \\
	{[\TTa(F)]} & {[\TTa(\Fv)]} & {\hspace{0.1cm}[\TTa(\AAF)]}
	\arrow["{i^n_v}", from=1-1, to=1-2]
	\arrow[from=1-2, to=1-3]
	\arrow[from=1-3, to=2-3]
	\arrow[from=2-2, to=2-3]
	\arrow["{[\TTa(i_v)]}", from=2-1, to=2-2]
	\arrow[from=1-1, to=2-1]
	\arrow[from=1-2, to=2-2]
\end{tikzcd}\]
for every $\vMF$. 
It follows that the map $[\TTa(i)]$ coincides with the product map $\prod_v[\TTa(i_v)]:[\TTa(F)]\to\prod_v[\TTa(\Fv)]$.
\end{proof}
\begin{mydef}
We define $[\TTa(\AAF)]_1$ to be the image of $(\AAF^1)^n$ in $[\TTa(\AAF)]$ under the map $q^\aaa_{\AAF}$.
\end{mydef}
\subsection{} In this paragraph we suppose that $n=1$ and that $a=a_1\geq 2$.  

The following Proposition is due to \v Cesnavi\v cius:
\begin{prop}\label{cesusp}One has that:
\begin{enumerate}
\item {\normalfont {\cite[Proposition 4.12]{Cesnaviciuss}}} The image $$[\TT(a)(i)]([\TT(a)(F)])\subset [\TT(a)(\AAF)]$$ is discrete, closed and cocompact. 
\item {\normalfont {\cite[Lemma 4.4]{Cesnaviciuss}}} The group $\Sh^1(F,\mu_a):=\ker ([\TT(a)(i)])$ is finite.
\end{enumerate} 
\end{prop}
\begin{proof}
When~$R$ is a local ring, there exists an identification of abelian groups $H^1_{\fppf}(R,\mu_a)=R^{\times}/(R^{\times})_{\aaa}=[\TT(a)(R)].$ It follows that the space $\TT(a)(\AAF)$ is precisely the space $H^1(\AAF^{\in M_F},\mu_a)$ from \cite[Section 3]{Cesnaviciuss}. Moreover, with these identifications, the map $[\TT(a)(i)]$ becomes the map $\text{loc}^1(\mu_a)$  in the notation of \cite{Cesnaviciuss}. We are thus in the situations of the mentioned statements of \cite{Cesnaviciuss}.
\end{proof}
\begin{lem}\label{ljweq}
One has that $$[\TTT(a)(\AAF)]_1=[\TTT(a)(\AAF)].$$
\end{lem}
\begin{proof}
We will establish that any $(x_v)_v\in[\TTT(a)(\AAF)]$ admits a lift in~$\AAF^{1}$ for the map $q^{a}_{\AAF}=\prod_vq^{a}_v:\AAFt\to[\TTa(\AAF)]$. It follows from \ref{mapqaf} that there exists a lift $(y_v)_v\in\AAFt$ of $(x_v)_v$. 
We set $A:=\prod _{\vMF}|y_v|_v\in\RR_{>0}$. Let $w\in M_F^{\infty}$. Set $z=A^{1/n_w}\in F_w^{\times}$, so that $|z|_w=A$. One has that $((y_v)_{v\neq w},(y_wz^{-1})_w) \in\AAF^1,$ because $$|y_wz^{-1}|_w\prod_{v\neq w}|y_v|=|z^{-1}|_w\prod_v|y_v|_v=A^{-1}\cdot A=1.$$ We note that $z^{-1}=A^{-1/n_v}=(A^{-1/an_v})^a,$ hence $z^{-1}\in(F_{w}^\times)_a$. Now one has that \begin{align*}q^a_{\AAF}((y_v)_{v\neq w},y_wz^{-1})&=((q^a_v (y_v))_{v\neq w}, (q^a_w(y_wz^{-1}))_w)\\&=((x_v)_{v\neq w},(q^a_w(y_w))_w)\\&=(x_v)_{v}. \end{align*}
Thus $((y_v)_{v\neq w},(y_wz^{-1})_w)$ is a lift of $(x_v)_v$ lying in $\AAF^1$. It follows that $[\TT(a)(\AAF)]=q^{a}_{\AAF}(\AAF^1)=[\TT(a)(\AAF)]_1.$
\end{proof}
\subsection{}If~$A$ is an abelian group written additively, we will write $\aaa A$ for the subgroup $\{(a_jm)_j|m\in A\}$ of $A^n$. If the group is written multiplicatively, then we will use notation $A_{\aaa}$ to be consistent with earlier notation. In this paragraph we construct an isomorphism $A/dA\times A^{n-1}\xrightarrow{\sim} A^n/\aaa A$.

Let $n\geq 1$ be an integer and let $\aaa\in\ZZ^n_{>0}$. We denote $d:=\gcd (\aaa).$  For $j=1\doots n$, we set $$b_{j1}:=\frac{a_j}{d}.$$ There exists a matrix $E=(b_{ji})_{ji}\in SL(n,\ZZ)$ which has $(b_{j1})_j^t$ for the first column (see e.g. \cite{Reiner}). 
\begin{lem}\label{catear}
\begin{enumerate}
\item The kernel of the homomorphism \begin{equation}\label{lomnb}\ZZ^n\xrightarrow{E}\ZZ^n\to\ZZ^n/\aaa\ZZ, \end{equation}where the second homomorphism is the quotient homomorphism, is the subgroup $d\ZZ\times\{0\}^{n-1}\subset \ZZ^{n}$.
\item The homomorphism $$\overline E:(\ZZ/d\ZZ)\times\ZZ^{n-1}=\ZZ^n/(d\ZZ\times\{0\}^{n-1})\xrightarrow{\sim}\ZZ^n/\aaa\ZZ,$$  induced from $\aaa\ZZ$-invariant homomorphism \ref{lomnb}, is an isomorphism. Let~$A$ be an abelian group. Let us write $E_{A}$ and $\overline{E}_A$ for the tensor products $E\otimes_{\ZZ}A$ and $\overline{E}\otimes _{\ZZ}A$, respectively.  The following diagram $C(A)$ is commutative, its horizontal sequences are exact and its vertical arrows are isomorphisms of abelian groups:
\begin{equation}\label{exseqtta}\begin{tikzcd}
	{dA\times\{0\}^{n-1}} & A^n& {(A/dA)\times A^{n-1}} & 0 \\
	{\aaa A} & A^n & {A^n/\aaa A} & 0.
	\arrow[from=1-1, to=1-2]
	\arrow[from=1-2, to=1-3]
	\arrow[from=1-3, to=1-4]
	\arrow[from=2-1, to=2-2]
	\arrow[from=2-2, to=2-3]
	\arrow[from=2-3, to=2-4]
	\arrow["{E_{A}}"', from=1-1, to=2-1]
	\arrow["{E_{A}}"', from=1-2, to=2-2]
	\arrow["{\overline{E}_{A}}"', from=1-3, to=2-3]
\end{tikzcd}\end{equation}
Moreover, if $A\to B$ is a homomorphism of abelian groups, the canonical homomorphisms provide a morphism of diagrams $C(A)\to C(B)$.
\end{enumerate}
\end{lem}
\begin{proof}
\begin{enumerate}
\item Obviously, the kernel of (\ref{lomnb}) coincides with $E^{-1}(\aaa\ZZ)$, thus the kernel is a free abelian group of the rank~$1$. Moreover, it contains the vector $(d,0\doots 0)^t.$ We deduce that the generator of (\ref{lomnb}) is given by $(k,0\doots 0)^t$ for some $k|d$. One has that $$E\cdot(k,0\doots,0)^t=(kb_{j1})_j^t=(ka_j/d)_j^t\in\aaa\ZZ,$$ hence $d|k$, hence $d=\pm k$. We deduce that the kernel of (\ref{lomnb}) is precisely the subgroup $d\ZZ\times\{0\}^{n-1}\subset \ZZ^{n}$.
\item The homomorphism $\overline{E}$ is evidently an isomorphism. By (1), the following diagram is commutative, its horizontal sequences are exact and its vertical arrows are isomorphisms:
\begin{equation}\label{nprlj}\begin{tikzcd}
	{d\ZZ\times\{0\}^{n-1}} & {\ZZ^n} & {(\ZZ/d\ZZ)\times \ZZ^{n-1}} & 0 \\
	\aaa\ZZ & {\ZZ^n} & {\ZZ^n/\aaa\ZZ} & 0
	\arrow[from=1-1, to=1-2]
	\arrow[from=1-2, to=1-3]
	\arrow[from=1-3, to=1-4]
	\arrow[from=2-1, to=2-2]
	\arrow[from=2-2, to=2-3]
	\arrow[from=2-3, to=2-4]
	\arrow["E"', from=1-1, to=2-1]
	\arrow["E"', from=1-2, to=2-2]
	\arrow["{\overline{E}}"', from=1-3, to=2-3]
\end{tikzcd}\end{equation}
The diagram (\ref{exseqtta}) is obtained by tensoring the diagram (\ref{nprlj}) by~$A$. Thus (\ref{exseqtta}) is commutative, its horizontal sequences are exact and its vertical arrows are isomorphisms. Moreover, the morphism of diagrams $C(A)\to C(B)$ is deduced by functoriality of the tensor product.
\end{enumerate}
\end{proof}
\subsection{}In this paragraph we prove that the isomorphisms  $E_{\Fvt}$ and $\overline E_{\Fvt}$ are continuous and that they preserve the compact open subgroups from before.
\begin{lem}\label{neis} The following claims are valid:
\begin{enumerate}
\item Let $\vMF$. The homomorphism $E_{\Fvt}:\Fvtn\to\Fvtn$ is an isomorphism of abelian topological groups. The homomorphism $\overline{E}_{\Fvt}:[\TT(d)(\Fv)]\times (\Fvt)^{n-1}\to [\TT(d)(\Fv)]$ is an isomorphism of abelian topological groups.
\item Let $\vMFz$. One has that $$E_{\Fvt}((\Ovt)^{n})=(\Ovt)^n$$and that $$\overline E_{\Fvt}([\TTd(\Ov)]\times(\Ovt)^{n-1})=[\TTa(\Ov)].$$
\end{enumerate}
\end{lem}
\begin{proof}
\begin{enumerate}
\item We have seen in \ref{catear} that the homomorphisms $E_{\Fvt}$ and $\overline{E}_{\Fvt}$ are isomorphisms of abelian groups. The map $E_{\Fvt}:\Fvtn\to\Fvtn$ is given by $$\xxx\mapsto  \big(\prod _{i=1}^nx^{b_{ji}}_i\big)_j$$ and is continuous and open. It follows that $E_{\Fvt}$ is an isomorphism of topological groups. Moreover, by \ref{catear}, the map $\overline {E}_{\Fvt}$ is the induced map from $(\Fvt)_d\times\{1\}^{n-1}$-invariant map $$\Fvtn\xrightarrow{E_{\Fvt}}\Fvtn\to[\TTa(\Fv)],$$ thus is continuous and open. It follows that $\overline{E}_{\Fvt}$ is an isomorphism of topological groups.
\item Let $\vMFz$ and let $\xxx\in(\Ovt)^n$. The~$j$-th coordinate of $E_{\Fvt}(\xxx)$ is equal to $\prod _{i=1}^nx^{b_{ji}}_i$ and is an element of $\Ovt$, thus $E_{\Fvt}(\xxx)$ is an element of $(\Ovt)^n$. We have established $E_{\Fvt}(\Ovtn)\subset\Ovtn$. Let $E^{-1}=(c_{ji})_{ji}$. For $\yyy\in\Fvtn$, one has that $E^{-1}_{\Fvt}(\yyy)=\prodjn y^{b_{ji}}_i$ and it follows that $E^{-1}(\yyy)\in(\Ovt)^n$, hence $E^{-1}((\Ovt)^n)\subset(\Ovt)^n$. We deduce $E_{\Fvt}((\Ovt)^n)=(\Ovt)^n$. Now one has that \begin{align*}\overline{E}_{\Fvt}([\TT(d)(\Ov)]\times (\Ovt)^{n-1})&=\overline{E}_{\Fvt}(q^d_v\times\Id_{(\Fvt)^{n-1}}((\Ovt)^n))\\&=q^\aaa_v(E_{\Fvt}((\Ovt)^{n-1}))\\&=\qav((\Ovt)^n)\\&=[\TTa(\Ov)].\end{align*}
\end{enumerate}
\end{proof}
\subsection{} We will now use properties given in \ref{neis} to define maps $(\AAFt)^n\to (\AAFt)^n$ and $[\TT(d)(\AAF)]\times (\AAFt)^{n-1}\to[\TTa(\AAF)]$.

\begin{lem} \label{razdob} The following claims are valid
 \begin{enumerate}
\item The map $\big(\prod_{\vMF}E_{\Fvt}\big):(\prod_v(\Fvt)^n)\to\prod_v(\Fvt)^n$ is precisely the map $E\otimes_{\ZZ}(\prod_{v}\Fvt)$. The map $\big(\prod_{\vMF}\overline{E}_{\Fvt}\big):(\prod_v([\TT(d)(\Fv)]\times(\Fvt)^{n-1}))\to\prod_v[\TTa(\Fv)]$ is precisely the map $\overline{E}\otimes_{\ZZ}(\prod_{v}\Fvt)$.
 \item One has that $\big(\prod_{\vMF}E_{\Fvt}\big)((\AAFt)^n)\subset(\AAFt)^n$. Moreover, the induced homomorphism $$E_{\AAFt}:(\AAFt)^n\to(\AAFt)^n $$ is an isomorphism of topological abelian groups. 
  \item One has that $\big(\prod _{\vMF}\overline E_{\Fvt}\big)([\TTd(\AAF)]\times (\AAF^{\times})^{n-1})\subset [\TTa(\AAF)].$ Moreover, the induced homomorphism $$\overline{E}_{\AAFt}:([\TTd(\AAF)]\times (\AAF^{\times})^{n-1})\to [\TTa(\AAF)]$$ is an isomorphism of abelian topological groups. 
 \end{enumerate}
 \end{lem}
 \begin{proof}
\begin{enumerate}
\item One has that $E$ and $\overline E$ are homomorphisms of finitely presented $\ZZ$-modules, thus tensoring by $E$ and $\overline E$ commutes with the direct products. The claim follows
\item For every $\vMFz$, Lemma \ref{neis} gives that $ {E}_{\Fvt}$ is a continuous homomorphism which satisfies that ${E}_{\Fvt}((\Ovt)^{n})=(\Ovt)^n,$ and hence by \ref{mimip}, it follows that $$\big(\prod_{\vMF} E_{\Fvt}\big)((\AAFt)^n)\subset (\AAFt)^n.$$ Moreover, by \ref{neis} the maps $E_{\Fvt}$ are isomorphisms of abelian topological groups, thus by \ref{mimip} the homomorphism $\big(\prod_{\vMF} {E}_{\Fvt}\big):(\AAFt)^n\to(\AAFt)^n$ is an isomorphism of abelian topological groups.  
\item For every $\vMFz$, Lemma \ref{neis} gives that $\overline {E}_{\Fvt}$ is a continuous homomorphism which satisfies that $\overline{E}_{\Fvt}([\TT(d)(\Ov)]\times(\Ovt)^{n-1})=[\TTa(\Ov)],$ and hence by \ref{mimip}, it follows that $$\big(\prod_{\vMF}\overline E_{\Fvt}\big)([\TT(d)(\AAF)]\times (\AAFt)^{n-1}\big)\subset [\TTa(\AAF)].$$ Moreover, by \ref{neis} the maps $\overline E_{\Fvt}$ are isomorphisms of abelian topological groups, thus by \ref{mimip} the homomorphism $\big(\prod_{\vMF}\overline {E}_{\Fvt}\big):[\TTd(\AAF)]\times (\AAF^{\times})^{n-1}\to [\TTa(\AAF)]$ is an isomorphism of abelian topological groups.  
\end{enumerate} 
\end{proof}
\begin{lem} The following claims are valid:
\begin{enumerate}
\item One has that  $E_{\AAFt}((\AAF^1)^n)=(\AAF^1)^n.$ 
\item One has that $\overline{E}_{\AAFt}([\TTT(d)(\AAF)]_1\times(\AAF^1)^{n-1})=[\TTa(\AAF)]_1.$
\end{enumerate}
\end{lem}
\begin{proof}
\begin{enumerate}
\item Let $(x_{jv})_v\in(\AAF^1)^n$. The~$j$-th coordinate of its image under $E_{\AAFt}=\prod _{v}E_{\Fvt}$ is $(\prod _{k=1}^nx_{kv}^{b_{jk}} )_v $ and one has $$\prod _{v}\big|\prod _{k=1}^n x_{kv}^{b_{jk}}\big|_v=\prod _{k=1}^n \prod _{\vMF}|x_{kv}|_v^{b_{jk}}=1. $$ It follows that $E_{\AAFt}((\AAF^1)^n)\subset (\AAF^1)^n$. Let $E^{-1}=(c_{ji})_{ji}\in SL_n(\ZZ)$. For $j=1\doots n$, the~$j$-th coordinate of $E^{-1}_{\AAFt}(x_{jv})_v=(E^{-1}_{\Fvt}(x_{jv}))_v$ is equal to $\prod _{k=1}^n x_{kv}^{c_{jk}}$ and one has $$\prod _{\vMF}\big|\prod _{k=1}^n x_{kv}^{c_{jk}}\big|_v=\prod _{k=1}^n \prod _{\vMF}|x_{kv}|_v^{c_{jk}}=1$$ 
It follows that $E^{-1}_{\AAFt}(\AAF^1)^n)\subset (\AAF^1)^n.$ We deduce that $E_{\AAFt}((\AAF^1)^n)= (\AAF^1)^n$.
\item The diagram 
\[\begin{tikzcd}
	{(\AAFt)^{n}} && {(\AAFt)^{n}} \\
	{[\TT(d)(\AAF)]\times (\AAFt)^{n-1}} && {[\TTa(\AAF)]}
	\arrow["{q^d_{\AAF}\times \Id_{(\AAFt)^{n-1}}}"', from=1-1, to=2-1]
	\arrow["{\overline{E}_{\AAFt}}"', from=2-1, to=2-3]
	\arrow["{E_{\AAFt}}"', from=1-1, to=1-3]
	\arrow["{q^{\aaa}_{\AAF}}"', from=1-3, to=2-3]
\end{tikzcd}\] is commutative by \ref{catear}. One has that $[\TTT ^d(\AAF)]_1\times (\AAF^1)^{n-1}$ and $[\TTa(\AAF)]_1$ are images of $(\AAF^1)^n$ in $[\TTT(d)(\AAF)]\times(\AAFt)^{n-1}$ and in $[\TTa(\AAF)]$ under the surjective maps $q^d_{\AAF}\times\Id_{(\AAFt)^{n-1}}$ and $q^{\aaa}_{\AAF}$, respectively. The claim now follows from (1).
\end{enumerate}
\end{proof}
\subsection{}
We will now study the kernel and the image of the map $[\TTa(i)]$ when $n\geq 2$ and $\aaa\in\ZZ^n_{\geq 1}$.
\begin{lem} \label{cirti}
One has that \begin{align*}\overline E_{\Ft}\big(\Sh^1(F,\mu_d)\times\{0\}^{n-1}\big)&=\overline E_{\Ft}\big(\ker ([\TTT(d)(i)])\times \{0\}^{n-1}\big)\\&=\ker ([\TTa(i)])\end{align*} is finite. 
\end{lem} 
\begin{proof}
\item By \ref{catear}, the following diagram is commutative and the vertical arrows are isomorphisms:
\[\begin{tikzcd}
	{[\TTd(F)]\times (F^{\times})^{n-1}} & {\hspace{0.2cm}[\TTd(\AAF)]\times (\AAFt)^{n-1}} \\
	{[\TTa(F)]} & {[\TTa(\AAF)].}
	\arrow["{[\TTd(i)]\times i^{n-1}}"', from=1-1, to=1-2]
	\arrow["{[\TTa(i)]}", from=2-1, to=2-2]
	\arrow["{\overline E_{\Ft}}"', from=1-1, to=2-1]
	\arrow["{E_{\AAFt}}", from=1-2, to=2-2]
\end{tikzcd}\]
We deduce that $$\overline E_{\Ft}(\ker([\TTT(d)(i)]\times i^{n-1}])=\ker([\TTa(i)]).$$ The map $i^{n-1}: (F^{\times})^{n-1}\to (\AAFj)^{n-1}$ given by the diagonal inclusion is injective, hence the kernel of $ [\TTT(d)(i)]\times i^{n-1}$ is precisely the group $\ker([\TTT(d)(i)])\times \{0\}^{n-1}=\Sh^1(F,\mu_d)\times \{0\}^{n-1}$. Finiteness of the Tate-Shafarevich group has been established in \cite[Lemma 1.2]{Sansuc}. The claim is proven
\end{proof}
\begin{prop}\label{cesres}
The image $[\TTa(F)]$ under $[\TTa(i)]$ lies in $[\TTa(\AAF)]_1$. Moreover, the subgroup $[\TTa(i)]([\TTa(F)])\subset[\TTa(\AAF)]_1$ is discrete, closed and cocompact.
\end{prop}
\begin{proof}
The following diagram is commutative by the definition of $[\TTa(i)(F)]$:
\[\begin{tikzcd}
	{(\Ft)^n} & {(\AAFt)^n} \\
	{[\TTa(F)]} & {[\TTa(\AAF)]}
	\arrow["{i^n}", from=1-1, to=1-2]
	\arrow["{q^{\aaa}_{\AAF}}", from=1-2, to=2-2]
	\arrow["{[q^{\aaa}(F)]}"', from=1-1, to=2-1]
	\arrow["{[\TTa(i)]}"', from=2-1, to=2-2]
\end{tikzcd}\]
The image under $i^n$ of $\Ftn$ is contained in $(\AAF^1)^n$. We deduce that $$[\TTa(i)]([\TTa(F)])=(q^\aaa_{\AAF}\circ i^n)(\Ftn)\subset q^{\aaa}_{\AAF}(((\AAF)^1)^n)=[\TTa(\AAF)]_1.$$ The map $\overline{E}_{\AAFt}$ is an isomorphism of abelian topological groups. Thus the subgroup $[\TTa(i)]([\TTa(F)])=E_{\AAFt}^{-1}(([\TTd(i)]\times i^{n-1})([\TTd(F)]\times(\Ft)^{n-1}))$ is discrete, closed and cocompact in $[\TTa(\AAF)]_1=E^{-1}_{\AAFt}([\TTd(\AAF)]_1\times (\AAF^1)^{n-1})$ if and only the subgroup $$[\TTd(i)]\times i^{n-1})([\TTd(F)]\times(\Ft)^{n-1})\subset [\TTd(\AAF)]_1\times (\AAF^1)^{n-1} $$is discrete, closed and cocompact. The subgroup $(F^{\times})^n\subset(\AAF^1)^n$ is discrete, closed and cocompact by \cite[Theorem 5-15]{Ramakrishnan}. Recall that by \ref{ljweq} one has that $[\TT ^d(\AAF)]_1=[\TT ^d(\AAF)]$. Now, by \ref{cesusp}, one has that  $[\TTd(i)]([\TTd(F)])\subset[\TT ^d(\AAF)]_1=[\TT ^d (\AAF)]$ is discrete, closed and cocompact. We deduce that $$[\TT(d)(F)]\times (F^\times)^{n-1}\subset [\TT(d)(\AAF)]_1\times ((\AAF)^1)^n$$ is discrete, closed and compact. The statement follows. 
\end{proof}
\subsection{} In this paragraph we establish some more basic properties of $[\TTa(\AAF)]$.
\begin{lem}\label{countableatinfty}
For $\vMF$, the groups $\Fvt$ are countable at infinity (i.e. are countable unions of compact subsets). The groups $\AAFt,$ $\AAF^1$, $[\TTa(\AAF)]$ and $[\TTa(\AAF)]_1$ are countable at infinity.
\end{lem}
\begin{proof}
For $\vMF-\vMFC$, one has that $\Fvt$ is the countable union of 
of compact balls: $$\Fvt=\bigcup_{q\in\QQ_{\neq 0}}B\big(q,\frac{|q|_v}2\big)$$ and for $v\in M_F^{\CC}$ one has that $\Fvt=\CC-\{0\}$ is the countable union of compact balls $$\Fvt=\bigcup_{q\in\QQ_{\neq 0}}B\big(q,\frac{|q|}2\big).$$
For finite subset $S\subset M_F,$ it follows that $\prod_{v\in S}\Fvt$ is countable at infinity. It follows that the group $\AAFt,$ which writes as the countable union $$\AAFt=\bigcup_{\substack{M_F^\infty\subset S\subset M_F\\S\text{ finite}}}\prod_{v\in S}\Fvt\times\prod_{v\in M_F-S}\Ovt,$$ is countable at infinity. The group $\AAF^1$ is countable at infinity as it is a closed subgroup of $\AAFt$. The group $(\AAFt)^n$ and the group $(\AAF^1)^n$ are countable at infinity, as finite products of groups which are countable at infinity. Finally, the groups $[\TTa(\AAF)]$ and $[\TTa(\AAF)]_1$ are countable at infinity, as they admit surjective continuous maps from groups which are countable at infinity.
\end{proof}
\begin{lem}\label{properaaftn}
The maps $\AAFt\to\AAFtn$, $\AAFj\to (\AAFj)^n$ and $\Rgz\to \RR^n_{>0}$ given by $t\mapsto (t^{a_j})_j$ are proper.
\end{lem}
\begin{proof}
The morphism $$\Gm\to\Gmn\hspace{1cm}t\mapsto (t^{a_j})_j,$$ is the base change along $\Gmn\times\{\jed\}\to\Gmn\times\Gmn$ of the morphism $\Gm\times \Gmn\xrightarrow{(a,p_2)}\Gmn\times\Gmn$ which is proper by \ref{aaaction}, hence itself is proper. By
\cite[Proposition 4.4]{Conrad}, for a proper morphism of separated schemes $X\to Y$, the induced topological map $X(\AAF)\to Y(\AAF)$ is proper. We deduce that the map $\AAFt=\Gm(\AAF)\to\Gmn(\AAF)=\AAFtn$ is proper. 
One has that the preimage of $(\AAFj)^n$ under the map $\AAFt\to\AAFtn$ is $\AAFj,$ thus by \cite[Proposition 3, \no 1, \S 10, Chapter I]{TopologieGj} the map $\AAFj\to(\AAFj)^n$ is proper. Note that under the identification $\log:\RR_{>0}\xrightarrow{\sim}\RR$, the homomorphism $t\mapsto (t^{a_j})_j$ becomes $t\mapsto (a_jt)_j$. As for every~$j$ one has $a_j>0$, the latter morphism is proper. The statement is proven. 
\end{proof}
Let us write $\lvert\cdot\rvert$ for the map $\AAFt\to \RR_{>0}$ given by $\lvert(x_v)_v\rvert=\prod_{\vMF}\lvert x_v\rvert_v,$ and by $\lvert\cdot\rvert^n$ the product map $(\AAFt)^n\to \RR_{>0}^n$. We recall that in \ref{mapqaf} we have established that the map $$q^{\aaa}_{\AAF}:(\AAFt)^n\to[\TTa(\AAF)]\hspace{1cm}(\xxx_v)_v\mapsto (\qav(\xxx_v))_v$$ is open, continuous and surjective. The image of $(\AAFj)^n$ under $q^{\aaa}_{\AAF}$ we have denoted by $[\TTa(\AAF)]_1$. Let $q^{\aaa}_{\RR_{>0}}:\RR^n_{>0}\to\RR^n_{>0}/(\RR_{>0})_{\aaa}$ be the quotient map. The map \begin{equation*} q^\aaa_{\RR_{>0}}\circ\lvert\cdot\rvert^n:\AAFtn\to\RR^n_{>0}/(\RR_{>0})_\aaa 
\end{equation*} is $\AAFta$-invariant, and let \begin{equation}
\lvert\cdot\rvert_{\aaa}:[\TTa(\AAF)]\to\RR^n_{>0}/(\Rgz)_{\aaa}
\end{equation} be the induced map.
\begin{lem}\label{quotttaf}
In the commutative diagram$$\begin{tikzcd}[column sep=small]
  & \{1\} \arrow[d, ""]  & \{1\} \arrow[d, ""] & \{1\} \arrow[d, ""]  &  \\
  \{1\} \arrow[r] & (\AAFj)_{\aaa} \arrow[d, ""] \arrow[r, ""] & (\AAFt)_{\aaa}\arrow[d, ""] \arrow[r, "\lvert\cdot\rvert^n"] & (\RR^n_{>0})_{\aaa} \arrow[d, ""] \arrow[r] & \{1\} \\
  \{1\} \arrow[r] & (\AAFj)^n \arrow[d, "q^{\aaa}_{\AAF}"] \arrow[r, ""] &\AAFtn  \arrow[d, "q^{\aaa}_{\AAF}"] \arrow[r, "\lvert\cdot\rvert^n"] & \RR^n_{>0} \arrow[d, "q^{\aaa}_{\RR_{>0}}"] \arrow[r] & \{1\} \\
  \{1\} \arrow[r] &{[}\TTa(\AAF){]}_1  \arrow[d, ""] \arrow[r, ""] & {[}\TTa(\AAF){]} \arrow[d, ""] \arrow[r, "\lvert\cdot\rvert_{\aaa}"] & \RR^n_{>0}/ (\RR_{>0})_{\aaa}\arrow[d, ""] \arrow[r] & \{1\}\\
    & \{1\} & \{1\} & \{1\} & 
\end{tikzcd},$$ where the maps that are not named are either canonical inclusions or the canonical maps to singletons, all horizontal and all vertical sequences are exact.
\end{lem}
\begin{proof}
\begin{itemize}
\item One has an exact sequence $\{1\}\to\AAFj\to\AAFt\xrightarrow{\lvert\cdot\rvert}\RR_{>0}\to\{1\}$ and its~$n$-th product with itself is the second horizontal exact sequence, which is therefore exact.
\item We establish that the first horizontal sequence is exact. We establish that the map $\lvert\cdot\rvert^n|_{\AAFta}:\AAFta\to(\Rgz)_{\aaa}$ is surjective. An element in $(\Rgz)_{\aaa}$ is of the form $(r^{a_j})_j$. One can find $(x_v)_v\in\AAFt$ such that $\prod _v|x_v|_v=r$. Then one has $((x_v)_v^{a_j})\in\AAFta$ satisfies that its image under $\lvert\cdot\rvert^n$ is precisely $(r^{a_j})_j$. To establish the exactness at~$\AAFta,$ let $((y_v)_v^{a_j})\in\ker (\lvert\cdot\rvert^n)$ and we observe that one must have for $j=1\doots n$ that $\prod_{v}|y_v^{a_j}|_v=\big(\prod_{v}|y_v|_v\big)^{a_j}=1.$ As $a_j>0$, we deduce $(y_v)_v\in\AAFj$. We conclude the first horizontal sequence is exact.
\item The third vertical sequence is exact by the definition.
\item We establish that the second vertical sequence is exact. That the map $q^{\aaa}_{\AAF}$ is surjective is proven in \ref{mapqaf}. We prove that its kernel is $\AAFta$. Suppose $\xxx_v\in\ker q^{\aaa}_{\AAF}=\ker\big(\prod _{\vMF}\qav\big)$. Then for every $v\in M_F$ one has that $\xxx_v\in\ker\qav$ and for almost every~$v$, one has $\xxx_v\in\Ovtn$. Hence, for every $\vMF$, there exists $y_v\in\Fvt$ such that  for almost every~$v$ one has that $y_v\in\Ovt$ and such that for every~$v$ one has $(y^{a_j}_v)=\xxx_v$. We deduce $\xxx=(y_v^{a_j})_j\in \AAFta$. Suppose now $\zzz\in\AAFta$ and pick $\yyy\in\AAFt$ such that $(\yyy^{a_j})_j=\zzz.$ We have for every $v\in M_F$ that $\qav(\zzz_v)=\qav(\yyy_v^{a_j})=1$ and hence $q^{\aaa}_{\AAF}(\zzz)=1$. We have established that the kernel of $q^{\aaa}_{\AAF}$ is $\AAFta$ and we deduce the exactness of the second vertical sequence. 
\item The map $q^{\aaa}_{\AAF}|_{(\AAFj)^n}:(\AAFj)^n\to[\TTa(\AAF)]_1$ is surjective and its kernel is $\ker q^{\aaa}_{\AAF}|_{(\AAFj)^n}=\ker q^{\aaa}_{\AAF}\cap (\AAFj)^n=(\AAFj)_{\aaa}.$ We deduce that the first vertical sequence is exact.
\item The long exact sequence deduced from applying Five lemma on the first two horizontal sequences contains the third horizontal sequence and the statement is proven.
\end{itemize}
\end{proof}
\subsection{}
\label{hrklj} We end the chapter, by observing that as in the classical situation, the short exact sequence $$\{1\}\to [\TTa(\AAF)]_1\to [\TTa(\AAF)]\xrightarrow{\lvert\cdot\rvert_{\aaa}}(\Rgz^n/(\Rgz)_{\aaa})\to\{1\}$$ splits and we give its section.
 
The exact sequence $$1\to \AAF^1\to \AAFt\xrightarrow{\lvert\cdot\rvert}\RR_{>0}\to 1$$ admits for a section the map $$\sigma:\RR_{>0}\to\AAFt\hspace{1cm}x\mapsto ((\rho_v(x)^{\frac1{r_1+r_2}})_{\vMFi},(1)_{\vMFz}),$$ where, for $\vMFi$, we define $\rho_v:\RR_{>0}\to \Fvt $ by $x\mapsto x^{1/n_v},$ and $r_1$ and $r_2$ are the number of real and complex places of~$F$, respectively. We deduce an isomorphism :
\begin{equation}\label{identaafj}
\AAFj\times\RR_{>0}\xrightarrow{\sim}\AAFt\hspace{1cm}(\xxx,r)\mapsto \xxx\sigma(r).
\end{equation}
The map $$\lvert\cdot\rvert^n:\AAFtn\to(\RR_{>0})^n\hspace{1cm}(x_j)_j\mapsto (|x_j|)_j$$ admits a section$$\sigma^n:\RR^n_{>0}\to\AAFtn\hspace{1cm}(x_j)_j\mapsto (\sigma(x_j))_j.$$ Note that $\sigma^n((\RR_{>0})_\aaa)\subset\AAFta$ and let $\sigma^{\aaa}:\RR^n_{>0}/(\RR_{>0})_{\aaa}$ be the map induced from $(\RR_{>0})_{\aaa}$-invariant map $q^\aaa_{\AAF}\circ\sigma^n:\RR^n_{>0}\to[\TTa(\AAF)].$ The map $\sigma^{\aaa}$ is a section to the map $\lvert\cdot\rvert_{\aaa}:[\TTa(\AAF)]\to\RR^n_{>0}/(\RR_{>0})_\aaa$ and 
we deduce an isomorphism \begin{equation}\label{seulment}[\TTa(\AAF)]_1\times \RR^n_{>0}/(\RR_{>0})_{\aaa}\xrightarrow{\sim}[\TTa(\AAF)]\hspace{1cm} (\xxx,\rrr)\mapsto \xxx\sigma^{\aaa}(\rrr).\end{equation}
The image $[\TTa(i)]([\TTa(F)])$ is contained in $[\TTa(\AAF)]_1$ by Proposition \ref{cesres}. The isomorphism (\ref{seulment}) induces, hence, an identification that may be used implicitly \begin{multline}\label{silmo}
\big([\TTa(\AAF)]_1/[\TTa(i)]([\TTa(F)])\big)\times\big(\RR^n/(\Rgz)_{\aaa}\big)\\\xrightarrow{\sim}[\TTa(\AAF)]/[\TTa(i)]([\TTa(F)]).
\end{multline} 
\chapter{Quasi-toric heights}
\label{Quasi-toric heights}
In this chapter we define heights on weighted projective stacks. A height on a stack can be stable or unstable. Stable means that any two rational points of stack that are $\CC$-isomorphic have same heights. Stable heights feature a drawback: they are not ``Northcott heights", i.e. there may exist $B>0$ such that there are infinitely many rational points of the stack having the height less than $B$ (see \ref{examnotnorth} for an example). Hence, such heights cannot be used to count rational points.

We start the chapter by recalling several facts on the line bundles on the weighted projective stacks. We make a formalism of unstable metrics and unstable heights in sections \ref{metricsonstacks} and \ref{heightsonppa}. Examples of unstable heights are {\it quasi-toric} heights and {\it quasi-discriminant heights} (the latter one appears in the last chapter of this article). 

A quasi-toric height is a height arising from a model with enough integral points (see \ref{metricsinducedbymodels}) of the stack~$\oPPa$. We prove in \ref{finitenesstoric}, that they are Northcott heights. These heights will be used in the following chapters to make an estimate of Manin-Peyre for the number of rational points. The last section of the chapter is dedicated to the proof that quasi-toric heights admitting logarithmic singularities are Northcott.
\section{Line bundles on a quotient stack}
In the next paragraphs, we recall that line bundles on quotient stacks correspond to~$G$-linearizations of line bundles on presentations. We use this to determine the line bundles on $\overline{\PPP(\aaa)}$ and on~$\PPP(\aaa)$.
Let~$Z$ be a scheme.
\subsection{}
Let~$X$ be a~$Z$-algebraic stack. By a line bundle on~$X$ we mean a quasi-coherent $\OO_X$-module~$L$ for which there exists a faithfully flat~$1$-morphism of finite presentation $f:X'\to X$ and an isomorphism $f^*L\cong\OO_{X'}$.

We give another presentation of line bundles. The category of~$X$-schemes is the category the objects of which are pairs $(T,t)$, where~$T$ is a~$Z$-scheme and $t:T\to X$ is a~$1$-morphism over~$Z$. A morphism of~$X$-schemes $(T',t')\to (T,t)$ is a pair $(f,f^\#)$, where $f:T'\to T$ is a~$Z$-morphism and $f^\#:t'\xrightarrow{\sim}t\circ f$. The composition of $(g,g^\#):(T'',t'')\to (T',t')$ and  $(f,f^\#):(T',t')\to (T,t)$ is defined to be the pair consisting of $g\circ f$ and the $2$-morphism: $$t''\xrightarrow{g^\#}t'\circ g\xrightarrow{g(f^\#)}t\circ f\circ g=t\circ(f\circ g).$$ 
The big {{fppf} }site of~$X$ (see \cite[Exercise 9.F]{Olsson}) is defined to be the category of~$X$-schemes endowed with the Grothendieck topology defined by coverings: $$\{(T_i,t_i)\to Y\}_i \text{ is a covering of Y, if }\coprod_i T_i\to Y\text{ is surjective  and fppf}. $$
We write $X_{\text{fppf}}$ for this site. Then a line bundle~$L$ on~$X$ is simply an {fppf}-locally trivial quasi-coherent sheaf on $X_{\text{fppf}}$. 
\subsection{}
 We recall the definition of a linearization of a module on a scheme (see e.g. \cite[Definition 1.6, Section 3, Chapter 3]{Fogarty} for the case of line bundle).
\begin{mydef}[Stacks Project, {\cite[\href{https://stacks.math.columbia.edu/tag/03LF}{Definition 03LF}]{stacks-project}}]\label{Fog} Let~$Y$ be a~$Z$-scheme and let $a:G\times Y\to Y$ be a~$Z$-algebraic action of a~$Z$-group scheme $G=(G,m_G,e_G)$ to the left on~$Y$. Denote by $p_2:G\times Y\to Y$ the second projection.
\begin{enumerate}
\item  A~$G$-linearized quasi-coherent $\OO_Y$-module is a pair $(M,\psi)$ where $M$ is a quasi-coherent $\OO_Y$-module and $\psi:p_2^*M\xrightarrow{\sim}a^*M$ is an isomorphism of quasi-coherent $\OO_Y$-modules which satisfies the following cocycle condition $$
\begin{tikzcd}[column sep=small]
(p_2\circ(m_G\times \Id_Y))^*M=(p_2\circ p_{23})^*M\arrow[r,"p_{23}^*\psi"] \arrow[d,"(m_G\times\Id_Y)^*\psi"] & (a\circ(p_{23}))^*M=(p_2\circ(\Id_G\times a))^*M \arrow[d, "(\Id_G\times a)^*\psi"] \\
 (a(m_G\times \Id_Y))^*M\arrow[r,"="] &(a\circ(\Id_G\times a))^*M,
\end{tikzcd}$$
 where $p_{23}:G\times G\times Y\to G\times Y$ is the projection obtained by forgetting the first coordinate. We also say that $\psi$ is a~$G$-linearization of $M$.
\item A morphism $\ell:(M_1,\psi_1)\to (M_2,\psi_2)$ of~$G$-linearized quasi-coherent $\OO_Y$-modules is an isomorphism of $\OO_Y$-modules $\ell :M_1\to M_2$ such that the following diagram is commutative $$ \begin{tikzpicture}
  \matrix (m) [matrix of math nodes,row sep=3em,column sep=4em,minimum width=2em]
  {
      a^*M_1& p_2^*M_1 \\
     a^*M_2& p_2^*M_2.& \\ };
  \path[-stealth]
    (m-1-1) edge node [left] {$a^*\ell$} (m-2-1)
            edge node [above] {$\psi_1$} (m-1-2)
    (m-2-1.east|-m-2-2) edge node [below] {$\psi_2 $}
            node [above] {$$} (m-2-2)
    (m-1-2) edge node [right] {$p_2^*\ell $} (m-2-2)
            ;
\end{tikzpicture} $$
\item The tensor product of~$G$-linearized quasi-coherent $\OO_Y$-modules is defined by $$(M,\psi)\otimes (M',\psi'):=(M\otimes M',\psi\otimes\psi').$$
\item The trivial~$G$-linearized quasi-coherent $\OO_Y$-modules is the one given by the pair $(\OO_Y,\Id_{\OO_{G\times _ZY}}).$ 
\end{enumerate}
\end{mydef}
The inverse of a morphism $\ell:(M,\psi)\to (M',\psi')$ is the morphism given by the isomorphism $\ell^{-1}:M'\to M$. One verifies easily that~$G$-linearized quasi-coherent $\OO_Y$-modules and their morphisms form a category, which is a Picard category  (see \cite[Definition 1.4.2]{Dualite} )with respect to $\otimes.$ 
\subsection{}
In this paragraph we are going to discuss linearizations of the trivial line bundle. Let~$Y$ be a~$Z$-scheme endowed with a left~$Z$-action of a~$Z$-group scheme~$G$. If $\psi$ is a~$G$-linearization of the trivial line bundle on~$Y$, using that $a^*\mathcal O_Y=\mathcal O_{G\times _ZY}$ and that $p_2^*\mathcal O_Y=\mathcal O_{G\times _ZY}$, we deduce that $\psi\in\mathcal O^*_{G\times_Z Y}$. The cocycle condition translates as \begin{equation}\label{coccond}\psi(g'g,x)=\psi (g',g\cdot x)\psi (g, x) \hspace{1cm} g,g'\in G, x\in Y.\end{equation}
The following things are immediate for~$G$-linearizations of the trivial line bundle on~$Y$.
\begin{itemize}
\item A morphism $\ell:(\mathcal O_{Y},\psi_1)\to(\mathcal O_Y,\psi_2)$ of~$G$-linearizations is given by an element $\ell\in\OO_Y^*$ such that $$\ell(g\cdot x)\psi _1(g,x) =\ell(x)\psi_2 (g,x)\hspace{1cm} g\in G, x\in Y.$$
\item The trivial~$G$-linearization of the trivial line bundle is the~$G$-linearization given by $(g,x)\mapsto 1$.
\item If $\ell_1:(\mathcal O_Y,\psi_1)\to(\mathcal O_Y,\psi_2)$ and $\ell_2:(\mathcal O_Y,\psi_2)\to(\mathcal O_Y,\psi_3)$ are morphisms, the morphism $\ell_2\circ \ell_1:(\mathcal O_Y,\psi_1)\to(\mathcal O_Y,\psi_3)$ is given by  $l_2l_1\in\OO_Y^*.$
\item If $\psi_1$ are $\psi_2$ are two~$G$-linearizations of the trivial line bundle, the~$G$-linearization of the trivial line bundle $\psi_1\otimes\psi_2 $ is given by $$(g,x)\mapsto \psi_1(g,x)\psi_2(g,x).$$
\item The inverse of the isomorphism $\ell:(\mathcal O_Y,\psi_1)\to(\mathcal O_Y,\psi_2)$ is given by the element $\ell^{-1}\in\OO_Y^*$.
\end{itemize}
By abuse of the terminology, we will speak about morphisms or tensor products of~$G$-linearizations of the trivial line bundle, when if fact we mean the morphisms or tensor products of corresponding~$G$-linearized line bundles. We have that~$G$-linearizations of $\OO_Y$ form a Picard category $\sPic^G(\OO_Y)$. We let $\Pic^G(\OO_Y)$ be the abelian group formed by the isomorphism classes of objects of $\sPic^G(\OO_Y)$. Let $G^{\vee}$ be the group of characters of~$G$. One has a homomorphism $$G^{\vee}\to \Pic^G(\OO_Y)$$ given by associating to $\chi\in G^{\vee}$ the isomorphism class of the~$\Gm$-linearization of $\OO_Y$ \begin{equation}\label{pocgg}(g,x)\mapsto \chi(g).\end{equation}
\begin{lem}\label{ygmnul}
\begin{enumerate}
\item Suppose that $\Pic (Y)=0$. One has that $\Pic^G(Y)=\Pic^G(\OO_Y)$. 
\item Suppose that there are no non-constant morphisms $Y\to\Gm$. Any two isomorphic~$G$-linearizations of $\OO_Y$ are identical. The homomorphism (\ref{pocgg}) is an isomorphism.
\end{enumerate}
\end{lem}
\begin{proof}
\begin{enumerate}
\item It suffices to prove that $(L,\psi)\in\sPic^G(Y)$ is isomorphic to an element $(\OO_Y,\psi)\in\sPic^G(\OO_Y)$. As the Picard group of~$Y$ is trivial, one can find an isomorphism of line bundles $\ell:\OO_{Y}\xrightarrow{\sim} L$. One readily verifies that $\theta:(a^*\ell)^{-1}\circ\psi\circ(p_2^*\ell)$ is a linearization of $\OO_{Y}$. Now, we have an isomorphism $(\ell,(a^*\ell)^{-1}\circ\psi\circ(p_2^*\ell)):(L,\psi)\to(\OO_{Y},\theta)$. It follows that $\Pic^G(Y)=\Pic^G(\OO_Y)$.
\item An isomorphism $\ell:(\OO_Y,\psi_1)\xrightarrow{\sim}(\OO_Y,\psi_2)$ is given by an element $\ell\in\OO_Y^*$ such that $$\ell(gx)\psi_1(g,x)=\ell(x)\psi_2(g,\xxx)\hspace{1cm}\forall g\in G, \forall x\in Y.$$  As~$\ell$ is a constant morphism and $\ell(x)=\ell(gx)\neq 0$ for every $g$ and~$x$, we get that $\psi_1=\psi_2$. We also get that (\ref{pocgg}) is injective. Let us prove the second claim. As there are no non-constant morphisms $Y\to\Gm$, for fixed $g$, the morphism $\psi(g,-):Y\to\Gm$ is constant. Set $\chi (g):=\psi(g,x)$ for some $x\in Y$. The cocycle condition \ref{coccond} gives that $\chi(g'g)=\chi(g')\chi(g)$ for every $g,g'\in G$, i.e. $\chi$ is a character of~$G$. It follows that (\ref{pocgg}) is surjective, and hence is an isomorphism.
\end{enumerate}
\end{proof}
\subsection{}\label{defoflift}
Let now~$Y$ be a~$Z$-scheme and~$G$ a flat, locally of finite presentation~$Z$-algebraic group scheme. Suppose we are given an action $a:G\times Y\to Y$.  We are going to recall that the category of line bundles on the quotient stack $Y/G$ is equivalent to the category of line bundles on~$Y$ endowed with a~$G$-linearization. 

Let $q:Y\to Y/G$ be the quotient~$1$-morphism. Let $M$ be a $\mathcal O_{Y/G}$-module. Note that the pullback module $q^*M$ is~$G$-linearized as follows. Let $t:q\circ a\xrightarrow{\sim}q\circ pr_2$ be a $2$-morphism making the diagram \[\begin{tikzcd}
	G\times _ZY & Y \\
	Y & Y/G
	\arrow["{pr_2}"', from=1-1, to=2-1]
	\arrow["{a}", from=1-1, to=1-2]
	\arrow["{q}"', from=1-2, to=2-2]
	\arrow["q"', from=2-1, to=2-2]
\end{tikzcd}\]
$2$-commutative. We have an isomorphism $$pr_2^*(q^*M)=(q\circ pr_2)^*M\xrightarrow{t^*M}(q\circ a)^*M=a^*(q^*M),$$
where $t^*M$ is the isomorphism given by \cite[\href{https://stacks.math.columbia.edu/tag/06WK}{Lemma 06WK}]{stacks-project}. We omit the proof that the cocycle condition from \ref{Fog} is satisfied. If $\ell:M\to~M'$ is a morphism of $\OO_{Y/G}$-modules, then a morphism of~$G$-linearized $\OO_Y$-modules is provided by $q^*\ell.$ We have, hence, a functor $\FFF$ from the category of $\OO_{Y/G}$-modules to the category of~$G$-linearized $\OO_Y$-modules given by $\FFF:L\mapsto (M,t^*M)$.  One can verify that $t^*(M\otimes M')=t^*M\otimes t^*M'$ and we deduce that the functor is an additive functor. 
According to \cite[\href{https://stacks.math.columbia.edu/tag/06WT}{Proposition 06WT}]{stacks-project}, $\FFF$ is an equivalence of categories. We consider the restriction of $\FFF$ on the category of the line bundles on~$Y$. It is immediate that if~$L$ is a line bundle on $Y/G$ then $\FFF(Y)$ is a~$G$-linearized line bundle on~$Y$.
\begin{prop}\label{gpic}
The above functor $\FFF$ induces an additive (\cite[Definition 1.4.2]{Dualite}) equivalence of the Picard category $\sPic({Y/G})$ of the line bundles on $Y/G$ and the Picard category $\sPic^G(Y)$ of~$G$-linearized line bundles on~$Y$.
\end{prop}
\begin{proof}
By \cite[\href{https://stacks.math.columbia.edu/tag/06WT}{Proposition 06WT}]{stacks-project} the functor $\FFF$ is fully faithful, thus the restricted functor $\FFF|_{\sPic(Y/G)}$ is an equivalence to its image in $\sPic^G(Y)$. It suffices therefore to verify that an object $(L,\psi)\in\sPic^G(Y/G)$ is isomorphic to an object in the image of $F|_{\sPic(Y/G)}$. It follows from \cite[\href{https://stacks.math.columbia.edu/tag/06WT}{Proposition 06WT}]{stacks-project} that there exists a $\mathcal O_{Y/G}$-module $M$ such that $\FFF(M)\cong(L,\psi)$. Let us prove $M$ is locally fppf trivial. The pullback $\pi^*M$ is a line bundle, thus, there exists Zariski open covering $y:Y'\to Y$ such that $y^*\pi^*M$ is locally trivial. It follows that for the fppf covering $\pi\circ y: Y'\to Y$ one has $(\pi\circ y)^*M\cong\OO_{Y'}.$
\end{proof}
If~$X$ is an algebraic stack, we denote by $\Pic(X)$ the Picard group of~$X$. We denote by $\Pic^G(Y)$ the abelian group the elements of which are isomorphism classes of objects in $\sPic^G(Y)$, and the addition of which is defined by the tensor product in $\sPic^G(Y)$. Functor $\mathcal F$ induces a homomorphism $\Pic(Y/G)\to\Pic^G(Y)$. Proposition \ref{gpic} gives that it is an isomorphism. 
 \subsection{} \label{phiequiv}
Let $\phi:G_1\to G_2$ be a homomorphism of flat, locally of finite type~$Z$-group schemes. Let $a_{X}:G_1\times _ZX\to X$ and $a_Y:G_2\times _ZY\to Y$ be~$Z$-actions on~$Z$-schemes~$X$ and $Y.$ Suppose $f:X\to Y$ is a morphism of~$Z$-schemes and suppose furthermore that the following diagram is commutative:
$$\begin{tikzcd}
	{G_1\times_Z X} & {G_2\times _ZY} \\
	X & Y.
	\arrow["{(\phi,f)}", from=1-1, to=1-2]
	\arrow["{a_X}"', from=1-1, to=2-1]
	\arrow["f"', from=2-1, to=2-2]
	\arrow["{a_Y}", from=1-2, to=2-2]
\end{tikzcd}$$
We say that~$f$ is $\phi$-equivariant. We construct an additive functor of Picard categories $\sPic_{G_2}(Y)\to\sPic_{G_1}(X)$ as follows. Let $(L,\psi)\in\sPic_{G_2}(Y)$. Denote by $pr_{X}:G_1\times_Z X\to X$ and $pr_{Y}:G_2\times _ZY\to Y$ the projections to the second coordinate. The linearization on $f^*L$ is provided by the isomorphism $$(\phi,f)^*\psi\hspace{0.1cm} :\hspace{0.1cm}(pr_{X})^*(f^*L)=(\phi,f)^*(pr_{Y})^*L\xrightarrow{(\phi,f)^*\psi}(\phi,f)^*a_Y^*L=a^*_X(f^*L).$$  If $\ell:(L,\psi)\to (L',\psi')$ is a morphism in $\sPic_{G_2}(Y),$ then $f^*\ell$ is a morphism of $(f^*L,(\phi,f)^*\psi)\to(f^*L',(\phi,f)^*\psi)$. It is a straightforward verification that this construction provides a functor $\sPic_{G_2}(Y)\to\sPic_{G_1}(X),$ and it is moreover an additive functor. The induced map $\Pic_{G_2}(Y)\to\Pic_{G_1}(X)$ is thus a homomorphism.
\begin{lem}\label{gjgdxy} Let $\overline{f}:X/G_1\to Y/G_2$ be the morphism of quotient stacks given by \cite[\href{https://stacks.math.columbia.edu/tag/046Q}{Lemma 046Q}]{stacks-project}. The following diagram is $2$-commutative: 
\begin{equation}\label{tomuchtimefor}\begin{tikzcd}
	{\sPic^{G_2}(Y)} & {\sPic^{G_1}(X)} \\
	{\sPic(Y/G_2)} & {\sPic(X/G_1)}.
	\arrow[from=2-1, to=2-2]
	\arrow[from=2-2, to=1-2]
	\arrow[ from=2-1, to=1-1]
	\arrow[from=1-1, to=1-2]
\end{tikzcd}\end{equation}
The diagram\begin{equation}\label{tomuchtimefo}\begin{tikzcd}
	{\Pic^{G_2}(Y)} & {\Pic^{G_1}(X)} \\
	{\Pic(Y/G_2)} & {\Pic(X/G_1)}.
	\arrow[from=2-1, to=2-2]
	\arrow[from=2-2, to=1-2]
	\arrow[ from=2-1, to=1-1]
	\arrow[from=1-1, to=1-2]
\end{tikzcd}\end{equation}  is commutative.
\end{lem}
\begin{proof}
Fix $2$-isomorphisms $t_X: q_X\circ pr_X\xrightarrow{\sim} q_X\circ a_X,$ $t_Y:q_Y\circ pr_Y\xrightarrow{\sim}q_Y\circ a_Y$ and $t_f:q_Y\circ f\xrightarrow{\sim}\overline{f}\circ q_X$ that make corresponding diagrams $2$-commutative. Let $L\in\sPic(Y/G_2)$. The image of~$L$ in $\sPic^{G_2}(Y)$ is $(q_Y^*L,t_Y^*L)$ and the image of $(q_Y^*L,t_Y^*L)$ in $\sPic^{G_1}(X)$ is $(f^*q_Y^*L,(\phi,f)^*t_Y^*L).$ The image of~$L$ in $\sPic(X/G_1)$ is $\overline{f}^*L,$ and the image of $\overline{f}^*L$ in $\sPic^{G_1}(X)$ is $(q_X^*\overline{f}^*L,t_X^*(\overline{f}^*L))$. It suffices to verify that the two images under the composite functors of~$L$ in $\sPic^{G_1}(X)$ are isomorphic. In fact, we will verify that $$t_f^*L:f^*q_Y^*L=(q_Y\circ f)^*L\xrightarrow{\sim}(\overline{f}\circ q_X)^*L=q_X^*\overline{f}^*L$$ is an isomorphism $(f^*q_Y^*L,(\phi,f)^*t_Y^*L)\xrightarrow{\sim} (q_X^*\overline{f}^*L,t_X^*\overline{f}^*L)$. For that we need to verify the commutativity of the following diagram:
$$\begin{tikzcd}
	{pr_X^*f^*q_Y^*L} & {a_X^*f^*q_Y^*L} \\
	{pr_X^*q_X^*\overline{f}^*L} & {a_X^*q_X^*\overline{f}^*L}
	\arrow["{t_X^*\overline{f}^*L}"', from=2-1, to=2-2]
	\arrow["{pr_X^*t_f^*L}", from=1-1, to=2-1]
	\arrow["{(\phi,f)^*t_Y^*L}", from=1-1, to=1-2]
	\arrow["{a_X^*t_f^*L}"', from=1-2, to=2-2]
\end{tikzcd}$$
This is true as in fact $$(t_f*a_X)\circ(t_Y*(\phi,f))=(\overline{f}*t_X)\circ (t_f* pr_X),$$as one verifies by simple diagram chasing,
where $*$ is the operation given by the structures of $2$-categories. Therefore the diagram (\ref{tomuchtimefor}) is $2$-commutative, and hence (\ref{tomuchtimefo}) is commutative.
\end{proof} 
 \section{Picard group of a weighted projective stack} 
 We are going to calculate Picard groups of weighted projective stacks.
 
 Let~$Z$ be a scheme. All the schemes and stacks are understood to be over~$Z$. Let $n\geq 1$ and let $\aaa\in\ZZ^n_{\geq 1}$.
 \subsection{}
We calculate the Picard group of~$\oPPa$ and of~$\PPP(\aaa)$. 
Proposition \ref{gpic} gives that the canonical homomorphisms $$\Pic(\oPPa)\to\Pic^{\Gm}(\AAA^n)$$ and $$\Pic(\PPP(\aaa))\to\Pic^{\Gm}(\AAA^n-\{0\}) $$are isomorphisms.
\begin{prop}\label{picofoppa}
Let $n\geq 1$ be an integer and let $\aaa\in\ZZ^n_{\geq 1}$.  
\begin{enumerate}
\item One has that $$\Pic(\overline{\PPP(\aaa)})\cong\ZZ$$ and a generator of $\Pic(\overline{\PPP(\aaa)})$ is given by the isomorphism class of line bundles defined by the~$\Gm$-linearization of $\OO_{\AAA^n}:$ $$\psi :\Gm\times\AAA^n\to\Gm\hspace{1cm}(t,\xxx)\mapsto t.$$
\item Suppose that $n\geq 2$. One has that $$\Pic({\PPP(\aaa)})\cong\ZZ$$ and a generator of $\Pic({\PPP(\aaa)})$ is given by the isomorphism class of line bundles defined by the~$\Gm$-linearization of $\OO_{\AAA^n-\{0\}}:$ $$\psi :\Gm\times(\AAA^n-\{0\})\to\Gm\hspace{1cm}(t,\xxx)\mapsto t.$$
\item Suppose that $n=1$. One has that $$\Pic({\PPP(a)})\cong\ZZ/a\ZZ$$ and a generator of $\Pic({\PPP(a)})$ is given by the isomorphism class of line bundles defined by the~$\Gm$-linearization of $\OO_{\AAA^1-\{0\}}:$ $$\psi :\Gm\times(\AAA^1-\{0\})\to\Gm\hspace{1cm}(t,x)\mapsto t.$$
\end{enumerate}
\end{prop}
\begin{proof}
Let us prove the first two statements. We have that $\Pic(\AAA^n)=\Pic(\AAA^n-\{0\})=0$ and thus by \ref{ygmnul}, one has equalities of abelian groups $\Pic^{\Gm}(\AAA^n)=\Pic^{\Gm}(\OO_{\AAA^n})$ and $\Pic^{\Gm}(\AAA^n-\{0\})=\Pic^{\Gm}(\OO_{\AAA^n-\{0\}})$. There are no non-constant morphisms $\AAA^n\to\Gm$ (respectively, non-constant morphisms $\AAA^n-\{0\}\to\Gm$) and thus by \ref{ygmnul} any two~$\Gm$-linearization of $\OO_{\AAA^n}$ (respectively, of $\OO_{\AAA^n-\{0\}}$) are identical. Moreover, by the same lemma, the homomorphisms $\Gm^{\vee}\to\Pic^{\Gm}(\OO_{\AAA^n})=\Pic^{\Gm}(\AAA^n)$ and $\Gm^{\vee}\to\Pic^{\Gm}(\OO_{\AAA^n-\{0\}})=\Pic^{\Gm}(\AAA^n-\{0\})$ given by $$\chi\mapsto ((t,\xxx)\mapsto \chi(t))$$ are isomorphisms. The group $\Gm^{\vee}$ is the infinite cyclic group with generator $t\mapsto t$, therefore the isomorphism class of the linearization $(t,\xxx)\mapsto t$ of the trivial line bundle is a generator of the group of the~$\Gm$-linearizations of the trivial line bundle (in both cases). The first two claim follow.

 Let us prove the third claim.  One has that $\Pic(\AAA^1-\{0\})=0$, thus by \ref{ygmnul}, one has an equality of abelian groups $\Pic^{\Gm}(\AAA^1-\{0\})=\Pic^{\Gm}(\OO_{\AAA^1-\{0\}})$. We study the latter group. Let $\widetilde{\Pic}^{\Gm}(\OO_{\AAA^{1}-\{0\}})$ be the group of~$\Gm$-linearizations of $\OO_{\AAA^1-\{0\}}$. We will split the proof into two parts. First, we prove that $\widetilde{\Pic^{\Gm}}(\OO_{\AAA^1-\{0\}})$ is isomorphic to $\ZZ$ and that its generator is given by $(t,x)\mapsto t$. Then we will prove that the canonical surjective homomorphism $\widetilde{\Pic}^{\Gm}(\OO_{\AAA^{1})-\{0\}}\to\Pic^{\Gm}(\OO_{\AAA^1-\{0\}})$ has for the kernel the subgroup generated by the~$\Gm$-linearization $(t,x)\mapsto t^a$.
 
 The cocycle condition $\psi(t't,x)=\psi(t',t\cdot x)\psi(t,x)$ implies that the degree of $\psi$ in~$x$ must be zero i.e. $\psi(t,-):(\AAA^1-\{0\})\to\Gm$ is a constant morphism. Set $\chi(t):=\psi(t,x)$ for some $x\in\AAA^1-\{0\}$. The cocycle condition gives $\chi(t't)=\chi(t')\chi(t)$, i.e. $\chi$ is a character of~$\Gm$. It follows that the homomorphism $$\Gm^{\vee}\to \widetilde{\Pic^{\Gm}}(\OO_{\AAA^1-\{0\}})\hspace{1cm}\chi\mapsto \hspace{0.1cm}\big((t,x)\mapsto \chi(t)\big)$$ is surjective. This homomorphism is evidently injective, hence is an isomorphism. The group $\Gm^{\vee}$ is isomorphic to $\ZZ$ and its generator is given by $t\mapsto t,$ thus $\widetilde{\Pic^{\Gm}}(\OO_{\AAA^1-\{0\}})$ is isomorphic to $\ZZ$ and its generator is given by $(t,x)\mapsto t$. Let us now prove the second claim. Let $\ell:(\OO_{\AAA^1-\{0\}},(t,x)\mapsto t^{k})\to(\OO_{\AAA^1},(t,x)\mapsto 1)$ be an isomorphism of~$\Gm$-linearizations of $\OO_{\AAA^1-\{0\}}$. This means that $\ell:\AAA^1-\{0\}\to\Gm$ is a morphism such that for every $t\in\Gm$ and every $x\in\AAA^1-\{0\}$ one has that $\ell(t\cdot x)t^{k}=\ell(x)$ i.e. $\ell(t^ax)=\ell(x)t^{-k}.$ Thus the degree of the rational function~$\ell$ must be $(-k)/a$. It follows that if $(t,x)\mapsto t^{k}$ and $(t,x)\mapsto 1$ are isomorphic, then $k\equiv 0\pmod a$. 
Let us see that if $a|k$, then the~$\Gm$-linearizations $(t,x)\mapsto t^{k}$ and $(t,x)\mapsto 1$ are isomorphic. Indeed if one sets $\ell(x)=x^{-k/a},$ we have that $$\ell(t\cdot x)t^{k}=(t\cdot x)^{(-k)/a}t^{k}=(t^ax)^{(-k)/a}t^{n_1}=x^{(-k)/a}t^{n_2}=\ell(x).$$
 The second claim follows. We deduce that $\Pic^{\Gm}(\AAA^1-\{0\})=\Pic^{\Gm}(\OO_{\AAA^1-\{0\}})\cong \ZZ/n\ZZ$ that its generator is given by the isomorphism class of the~$\Gm$-linearization of the trivial line bundle $(t,x)\mapsto t$. As $\Pic(\PPP(a))\xrightarrow{\sim}\Pic^{\Gm}(\AAA^1-\{0\})$ is an isomorphism by \ref{gpic}, the statement follows. 
 \end{proof}
\begin{mydef}
Let~$\mathcal O(1)$ be the isomorphism class of line bundles on $\overline{\mathscr P(\aaa)}$ given by the linearization of the trivial line bundle $$\psi:\Gm\times\AAA^n\to\Gm\hspace{1cm}(t,\xxx)\mapsto t .$$ For $k\in\ZZ$, we write $\mathcal O(k)$ for $\mathcal O(1)^{\otimes k}.$  By abuse of notation, we may also write $\mathcal O(k)$ for a line bundle in the corresponding isomorphism class of line bundles. 
\end{mydef}
\begin{mydef}
Let~$\mathcal O(1)$ be the isomorphism class of line bundles on ${\mathscr P(\aaa)}$ given by the linearization of the trivial line bundle $$\psi:\Gm\times(\AAA^n-\{0\})\to\Gm\hspace{1cm}(t,\xxx)\mapsto t .$$ For $k\in\ZZ$, we write $\mathcal O(k)$ for $\mathcal O(1)^{\otimes k}.$  By abuse of notation, we may also write $\mathcal O(k)$ for a line bundle in the corresponding isomorphism class of line bundles. 
\end{mydef}
When $n=1$ one has that $\OO(a)=\OO(0)=\OO_{\PPP(\aaa)}$.
For $\bbb\in\ZZ^n$, we denote $|\bbb|:=b_1+\cdots +b_n$.
\begin{mydef}\label{defnofanticanon}
We say that a line bundle on~$\oPPa$ with isomorphism class $\mathcal O(|\aaa|)$ is anti-canonical. We may write $(K_{\oPPa})^{-1}=\mathcal O(|\aaa|)$.
\end{mydef} 
\section{Metric on a line bundle on a stack}\label{metricsonstacks} In this section we define metrics on line bundles on algebraic stacks. 
\subsection{} Let~$v$ be a place of~$F$. We present a definition of an $\Fv$-metric on a line bundle on a stack.
\begin{mydef}\label{defofmet}
Let~$X$ be a locally of finite type $\Fv$-algebraic stack and let~$L$ be a line bundle on~$X$. An $\Fv$-metric $|\lvert\cdot\rvert|$ is the following data:
\begin{itemize}
\item for every~$1$-morphism of stacks $x:\Spec(\Fv)\to X$ we give a norm $|\lvert\cdot\rvert|_x$ on $L(x):=x^*L$.
\item for every $2$-morphism $x\xrightarrow{\sim} y$, the canonical morphism $L(x)\to L(y)$ is an isometric isomorphism.
\item for every~$1$-morphism $f:U\to X$ over $F_v$, with $U$ locally of finite type $F_v$-scheme and every section $s$ of $f^*L$ over $U$, the map $$U(\Fv)\to\RR_{\geq 0}\hspace{1cm} z\mapsto ||s(z)||_{f\circ z}$$ is continuous.\end{itemize}
An $\Fv$-metrized line bundle is a pair $(L,|\lvert\cdot\rvert|)$ of a line bundle~$L$ and an $\Fv$-metric on~$L$.
\end{mydef}
Let~$v$ be a place of~$F$ and let~$X$ be a locally of finite type $\Fv$-algebraic stack. 
\begin{itemize}
\item A morphism $r:(L,|\lvert\cdot\rvert|)\to (L',|\lvert\cdot\rvert|')$ of $\Fv$-metrized line bundles on~$X$ is an isomorphism of line bundles $r:L\to L'$ which is isometric i.e. for every $x\in X(\Fv)$, the morphism $r(x):L(x)\to L'(x)$ is isometric. 
\item The trivial line bundle can be endowed with the following metric: set $||1||_x=1$ for each~$1$-morphism $x:\Spec(\Fv)\to X$. The corresponding $\Fv$-metrized line bundle will be called the trivial $\Fv$-metrized line bundle. \item One defines the tensor product of $\Fv$-metrized line bundles on~$X$ as follows: let $(L,|\lvert\cdot\rvert|)\otimes (L',|\lvert\cdot\rvert|')$ be $\Fv$-metrized line bundles on~$X$, endow $L\otimes L'$ with the metric $|\lvert\cdot\rvert|\otimes|\lvert\cdot\rvert|'$ defined by $|\lvert\cdot\rvert|_x\otimes |\lvert\cdot\rvert|'_x$ for every~$1$-morphism $x:\Spec(\Fv)\to X$ (it is immediate that one indeed gets a metric on $L\otimes L'$). 
\end{itemize}
\begin{lem}\label{mthroot}
Let~$X$ be an algebraic stack and let~$L$ be a line bundle on~$X$. 
\begin{enumerate}
\item Suppose $|\lvert\cdot\rvert|$ is an $\Fv$-metric on~$L$. For every $x\in X(\Fv),$ let ${|\lvert\cdot\rvert|}^{-1}_x$ be the metric on $L^{-1}(x)$ defined by $${|\lvert\cdot\rvert|}^{-1}_x(\lambda):={||\ell||^{-1}_x},$$for $0\neq\ell\in L(x)$ and $\lambda\in L^{-1}(x)$ such that $\lambda(\ell)=1$. The metrics ${|\lvert\cdot\rvert|}^{-1}_x$ on $L^{-1}(x)$ for $x\in X(\Fv)$ define an $\Fv$-metric ${|\lvert\cdot\rvert|^{-1}}$ on $L^{-1}.$
\item Suppose $|\lvert\cdot\rvert|$ is an $\Fv$-metric on $L^{\otimes m}$ where~$m$ is a positive integer. For every $x\in X(\Fv)$ and every $\ell\in L(x)$ let $\sqrt[m]{||\ell||}_x$ be the metric on $L(x)$ defined by $$\sqrt[m]{|\lvert\cdot\rvert|}_x(\ell):=\sqrt[m]{||\ell^m||_x}.$$ The metrics $\sqrt[m]{|\lvert\cdot\rvert|}_x$ on $L(x)$ for $x\in X(\Fv)$ define an $\Fv$-metric $\sqrt[m]{|\lvert\cdot\rvert|}$ on $L.$
\end{enumerate}
\end{lem}
\begin{proof}
\begin{enumerate}
\item Suppose that $x\xrightarrow{\sim}y$ is a $2$-morphism, where $x,y\in X(\Fv)$. The induced linear map $L(x)\to L(y)$ is isometric, and it follows that $L^{-1}(x)\to L^{-1}(y)$ is isometric. Let now $g:U\to X$ be a~$1$-morphism, with $U$ is a locally of finite type $\Fv$-scheme and let $s\in g^*L(U)$. The map $U(\Fv)\to\RR_{>0}$ given by $z\mapsto {||s(z)||^{-1}}$ is precisely the composition of the map$$U(\Fv)\to\RR_{>0}\hspace{1cm} z\mapsto ||s(z)||_{g\circ x}$$  and the map $\RR_{>0}\to\RR_{>0}, x\mapsto x^{-1}$ and is thus continuous. It follows that ${|\lvert\cdot\rvert|^{-1}}$ is an $\Fv$-metric on $L^{-1}$.
\item Suppose that $x\xrightarrow{\sim}y$ is a $2$-morphism, where $x,y\in X(\Fv)$. The induced map $L^{\otimes m}(x)\to L^{\otimes m}(y)$ is isometric, and it follows that $L(x)\to L(y)$ is isometric. Let now $g:U\to X$ be a~$1$-morphism, with $U$ is a locally of finite type $\Fv$-scheme and let $s\in g^*L(U)$. The map $U(\Fv)\to\RR_{>0}$ given by $z\mapsto \sqrt[m]{||s(z)||}$ is precisely the composition of the map$$U(\Fv)\to\RR_{>0}\hspace{1cm} z\mapsto ||\ell^m||_{g\circ x}$$  and the map $\sqrt[m]{\cdot}:\RR_{>0}\to\RR_{>0}$ and is thus continuous. It follows that $\sqrt[m]{|\lvert\cdot\rvert|}$ is an $\Fv$-metric on~$L$.
\end{enumerate}
\end{proof}
One sees that the category $\widehat{\sPic}_v(X)$ of $\Fv$-linearized line bundles is a Picard category. The abelian group of isomorphism classes of objects in this category will be denoted by $\widehat\Pic_{v}(X).$ By abuse of terminology, we may call an element of $\widehat\Pic_{v}(X)$ an $\Fv$-metrized line bundle.

For topological spaces $A,B$, let us denote by $\mathscr C^0(A,B)$ the set of continuous functions from~$A$ to $B$. If $B$ is a topological abelian group, then $\mathscr C^0(A,B)$ carries a structure of an abelian group. 
\begin{lem} \label{ftopic}
Let~$v$ be a place of~$F$ and let~$X$ be a locally of finite type $\Fv$-algebraic stack. If $f:[X(\Fv)]\to\RR_{>0}$ is a continuous function, then setting $||\ell||^f_x:=|\ell(x)|_vf([x])$, for $\ell\in\OO_X(x)=\Fv$, defines an $\Fv$-metric $|\lvert\cdot\rvert|^f$ on $\OO_X$. The sequence $$0\to\mathscr C^0([X(\Fv)],\RR_{>0})\to\widehat{\Pic}_v(X)\to \Pic(X)$$ where the first homomorphism is given by $f\mapsto (\OO_X,(|\lvert\cdot\rvert|^f)$ and where the second homomorphism is the one that forgets the metric, is exact.
\end{lem}
\begin{proof}
Let us verify that $|\lvert\cdot\rvert|^f$ is a metric on $\OO_X$. Let us verify that for any $2$-morphism $x\xrightarrow{\sim}y$, where $x,y:\Spec(\Fv)\to X$ are~$1$-morphisms of algebraic stacks, the induced linear map $\OO_X(x)\xrightarrow{=}\Fv\xrightarrow{=}\OO_X(y)$ is an isometry. This is true because $||1||^f_x=f([x])=f([y])=||1||^f_{y}$. We now verify the last condition of \ref{defofmet}. Let $r:V\to X$ be a~$1$-morphism with $V$ a locally of finite type $\Fv$-scheme and let $s$ be a section over $V$ of $r^*\OO_X=\OO_V$ i.e. a morphism $s:V\to\AAA^1$. The map $$V(\Fv)\to\RR_{>0}\hspace{1cm}z\mapsto ||s(z)||_{r\circ z}$$ coincides with the map $$V(\Fv)\to\RR_{>0}\hspace{1cm}z\mapsto |s(z)|_vf([r(z)])$$ and is thus is continuous. We have verified that $|\lvert\cdot\rvert|^f$ is a metric on $\OO_X$.

It is evident that the composite homomorphism $$\mathscr C^0([X(\Fv)],\RR_{>0})\to\widehat{\Pic}_v(X)\to \Pic(X)$$ is the zero homomorphism. Let $(L,|\lvert\cdot\rvert|)$ be in the kernel of $\widehat{\Pic}_v(X)\to\Pic(X)$. We verify that $(L,|\lvert\cdot\rvert|)$ is in the image of $\mathscr C^0([X(\Fv)],\RR_{>0})\to\widehat{\Pic}_v(X).$ Obviously, $L=\OO_X$ in $\Pic(X)$. By the fact that a $2$-morphism $x\xrightarrow{\sim}y$ induces an isometry $\OO_X(x)\xrightarrow{=}\Fv\xrightarrow{=}\OO_X(y)$, it follows that $$[X(\Fv)]\to\RR_{>0}\hspace{1cm} [x]\mapsto ||1||_x$$ is a well defined function. We verify it is continuous. For every~$1$-morphism $g:U\to X$, with $U$ locally of finite type $\Fv$-scheme, the function $$U(\Fv)\xrightarrow{[g(\Fv)]}[X(\Fv)]\xrightarrow{[x]\mapsto ||1||_x}\RR_{>0}$$ is continuous as it coincides with the map $$U(\Fv)\to\RR_{>0}\hspace{1cm} x\mapsto ||1||_{g(x)},$$ which is continuous by the definition of an $\Fv$-metric. We deduce from \ref{bezmora} that the map $[x]\mapsto ||1||_x$ is a continuous function. It follows that $(L,|\lvert\cdot\rvert|)=(\OO_X,|\lvert\cdot\rvert|)$ is in the image of $\mathscr C^0([X(\Fv)],\RR_{>0})\to\widehat{\Pic}_v(X).$ The statement is proven.
 \end{proof}
\subsection{} Let~$v$ be a place of~$F$. Suppose $g:Y\to X$ is a~$1$-morphism of locally of finite type $\Fv$-algebraic stacks. Let $(L,|\lvert\cdot\rvert|_L)$ be an $\Fv$-metrized line bundle on~$X$. We define the pullback metric on $g^*L$.  Suppose $x:\Spec(\Fv)\to Y$ is an $\Fv$-point of~$Y$. For a section $\ell\in(g^*L)(x)=x^*(g^*L)=(g\circ x)^*L=L(g(x))$, we set $g^*||\ell||_x:=||\ell||_{g\circ x}$. 
\begin{lem}\label{ugovn} Let $g:Y\to X$ be a~$1$-morphism of locally of finite type $\Fv$-algebraic stacks.
\begin{enumerate}
\item If $(L,|\lvert\cdot\rvert|)$ is an $\Fv$-metrized line bundle on~$X$. The pair $(g^*L,g^*|\lvert\cdot\rvert|)$ is an $\Fv$-metrized line bundle on~$Y$.
\item If $(L,|\lvert\cdot\rvert|)$ is the trivial $\Fv$-metrized line bundle on $X,$ then $(g^*L,g^*|\lvert\cdot\rvert|)$ is the trivial $\Fv$-metrized line bundle on~$Y$. 
\item If $$r:(L,|\lvert\cdot\rvert|)\to (L',|\lvert\cdot\rvert|')$$ is a morphism of $\Fv$-metrized line bundles on~$X$, then $g^*r$ is isometric. 
\item The functor $g^*:\widehat{\sPic}_v(X)\to\widehat{\sPic}_v(Y)$ given by \begin{align*}(L,|\lvert\cdot\rvert|)&\mapsto (g^*L,g^*|\lvert\cdot\rvert|)\\
r:(L,|\lvert\cdot\rvert|)\to(L',|\lvert\cdot\rvert|')&\mapsto g^*r\hspace{2cm}
\end{align*} is an additive functor. 
\item Let $f:Y\to X$ be another~$1$-morphism and let $t:g\xrightarrow{\sim} f$ be a $2$-isomorphism. Let $(L,|\lvert\cdot\rvert|)$ be an $\Fv$-metrized line bundle on~$X$. The canonical isomorphism $t^*L:g^*L\xrightarrow{\sim}f^*L$ is an isometry.
\end{enumerate}
\end{lem}
\begin{proof}
\begin{enumerate}
\item Let us verify that $g^*|\lvert\cdot\rvert|$ is an $\Fv$-metric on $g^*L$. Let $y:\Spec(\Fv)\to Y$ be such that there exists a $2$-morphism $x\xrightarrow{\sim} y$. The canonical morphism $$(g^*L)(x)=x^*(g^*L)=(g\circ x)^*L= L(g(x))\to L(g(y))=(g^*L)(y)$$ is an isometric isomorphism, as such is $L(g(x))\to L(g(y))$. Let $h:U\to Y$ be a~$1$-morphism of algebraic stacks with $U$ locally of finite type $\Fv$-scheme. Pick $s\in (h^*(g^*L))(U)=((g\circ h)^*L)(U)$. The map $$U(\Fv)\to\RR_{\geq 0}\hspace{1cm}z\mapsto g^*||s(z)||_{h\circ z}$$ is continuous, because it coincides with the continuous map $$U(\Fv)\to\RR_{\geq 0}\hspace{1cm}z\mapsto ||s(z)||_{g\circ (h\circ z)}.$$
\item One has that $g^*\OO_X=\OO_Y$. Let $x\in Y(\Fv)$. One has that $$g^*||1||_x=||1||_{g(x)}=1.$$ The claim follows.
\item Let $x\in Y(\Fv)$. The linear map $(g^*r)(x):(g^*L)(x)\to (g^*L')(g(x))$ is isometric as it coincides with the isometric linear map $L(g(x))\to L'(g(x)).$ It follows that $g^*r$ is isometric.
\item Let $x\in Y(\Fv)$ and pick $\ell\in (g^*L)(x)=L(g(x))$ and $\ell'\in (g^*L')(x)=L'(g(x))$. We have that $$g^*||\ell||_x\cdot g^*||\ell'||_x'=||\ell||_{g(x)}\cdot ||\ell'||_{g(x)}'.$$ It follows that $g^*(|\lvert\cdot\rvert|)\otimes g^*(|\lvert\cdot\rvert|')=g^*(|\lvert\cdot\rvert|\otimes|\lvert\cdot\rvert|').$ It follows that $g$ is an additive functor.
\item Let~$y$ be an $\Fv$-point of~$Y$. The isomorphism $t:g\xrightarrow{\sim}f$ induces the isomorphism $(t*y):g\circ y\xrightarrow{\sim} f\circ y$. The linear map $t^*L$ is precisely the linear map $g^*L(y)=L(g(y))\xrightarrow{(t*y)^*L}L(f(y))=f^*L(y)$ and is thus an isometry by the definition of an $\Fv$-metric. The claim is proven.
\end{enumerate}
\end{proof}
The lemma provides a group homomorphism $\widehat{\Pic}_v(X)\to\widehat{\Pic}_v(Y)$ that we also denote by $g^*$.
\subsection{} In this paragraph we study $\Fv$-metrics which are invariant for an action of an algebraic group.

Let~$v$ be a place of~$F$. Let~$G$ be a locally of finite type $\Fv$-algebraic group acting on a locally of finite type $\Fv$-scheme~$X$. By \ref{topgpagp}, one gets a topological action $G(\Fv)\times X(\Fv)\to X(\Fv).$
\begin{mydef}
Let~$L$ be a~$G$-linearized line bundle on~$X$. An $\Fv$-metric $|\lvert\cdot\rvert|$ on~$L$ is said to be~$G$-invariant if for every $t\in G(\Fv)$ and every $x\in X(\Fv),$ one has that the linear map $L(x)\to L(t\cdot x)$ given by the linearization is an isometric isomorphism. The $\Fv$-metrized line bundle $(L,|\lvert\cdot\rvert|)$ will be said to be~$G$-invariant.\end{mydef}
Let us introduce the category of~$G$-invariant $\Fv$-metrized line bundles.
\begin{itemize}
\item A morphism $\ell:(L,|\lvert\cdot\rvert|)\to (L',|\lvert\cdot\rvert|')$ of~$G$-invariant $\Fv$-metrized line bundles is a morphism of~$G$-linearized line bundles $\ell:L\to L'$ which is isometric.
\item The trivial~$G$-invariant $\Fv$-metrized line bundle is the trivial~$G$-linearized line bundle endowed with the metric: $||1||_x=1$ for every $x:\Spec(\Fv)\to X$ (it is immediate the metric is~$G$-invariant).
\item A morphism of~$G$-invariant $\Fv$-metrized line bundles is a morphism of corresponding~$G$-linearized line bundles which is an isometry.
\item The tensor product of two~$G$-invariant $\Fv$-metrized line bundles is a~$G$-invariant $\Fv$-metrized line bundle. 
\end{itemize}
The~$G$-invariant $\Fv$-metrized line bundles form a Picard category that we denote by $\widehat{\sPic^G_v}(X)$. Let $\widehat{\Pic^G_v}(X)$ be the abelian group given by the isomorphism classes of objects of $\widehat{\sPic^G_v}(X)$. One has a homomorphism of abelian groups $$\widehat{\Pic^G_v}(X)\to\Pic^G(X)$$which forgets the structure of $\Fv$-metrized line bundle.
One also has a canonical morphism $\widehat{\Pic^G_v}(X)\to\Pic_v(X)$ by simply forgetting that $\Fv$-metrized line bundle is~$G$-invariant.

If $E$ is a topological group acting on a topological space~$A$, we denote by $\mathscr C^0_E(A,B)$ the set of $E$-invariant continuous functions $A\to B$. If $B$ has a structure of a topological abelian group, then $\mathscr C^0_E(A,B)$ carries a structure of an abelian group. The map $\mathscr C^0(A/E,B)\to\mathscr C^0(A,B)$ induces an isomorphism $$\mathscr C^0(A,B)\xrightarrow{\sim}\mathscr C^0_E(A,B).$$
\begin{lem}\label{gfvexact}
If $f\in\mathscr C_{G(F_v)}^0(X(\Fv),\RR_{>0})$ then setting $$||\ell||_x^f:=|\ell(x)|_vf(x)$$ for every $x\in X(\Fv)$ and every $\ell\in\OO_X(x)=\Fv$, defines an $\Fv$-metric $|\lvert\cdot\rvert|^f$ on $\OO_X$ which is~$G$-invariant for the trivial~$G$-linearization of $\OO_X$. The sequence $$0\to\mathscr C_{G(F_v)}^0(X(\Fv),\RR_{>0})\rightarrow\widehat{\Pic^G_v}(X)\to\Pic^G(X),$$where $\mathscr C_{G(F_v)}^0(X(\Fv),\RR_{>0})\rightarrow\widehat{\Pic^G_v}(X)$ is given by ${f\mapsto (\OO_X,\Id_{\OO_{G\times X}},|\lvert\cdot\rvert|^f)},$ is exact.
\end{lem}
\begin{proof}
Lemma \ref{ftopic} gives that $|\lvert\cdot\rvert|^f$ is a metric on $\OO_X$. Moreover, for every $x\in X(\Fv)$ and every $t\in G(\Fv)$, the linear map $\OO_X(x)\xrightarrow{=}\Fv\xrightarrow{=}\OO_X(t\cdot x)$ given by the trivial~$G$-linearization, is an isometry, because $||1||^f_x=f(x)=f(t\cdot x)=||1||^f_{t\cdot x}$. It follows that the $\Fv$-metric $|\lvert\cdot\rvert|^f$ is~$G$-invariant. 

The composite homomorphism $\mathscr C^0_{G(\Fv)}(X(\Fv),\RR_{>0})\to\widehat{\Pic^G_v}(X)\to\widehat{\Pic}(X)$ is the zero homomorphism. To complete proof, it suffices to verify that if $(L,|\lvert\cdot\rvert|)$ is in the kernel of $\widehat{\Pic^G_v}(X)\to\widehat{\Pic}(X)$, then $(L,|\lvert\cdot\rvert|)$ is in the image of $\mathscr C^0_{G(\Fv)}(X(\Fv),\RR_{>0})\to\widehat{\Pic^G_v}(X)$. Obviously,~$L$ is the trivial~$G$-linearized line bundle $(\OO_X,\Id_{\OO_{G\times X}})$. For every $x\in X(\Fv)$ and every $t\in\Gm(\Fv)$, the canonical linear map $F_v=\OO_X(x)\xrightarrow{=}\OO_X(t\cdot x)=\Fv,$ given by the trivial~$G$-linearization, is isometry and maps $1\in \OO_X(x)$ to $1\in\OO_X(t\cdot x),$ thus the function $x\mapsto ||1||_x$ is~$G$-invariant. It follows that $(L,|\lvert\cdot\rvert|)=(\OO_X,\Id_{\OO_{G\times X}})$ is the image of $x\mapsto ||1||_x$ for the homomorphism $\mathscr C^0_{G(\Fv)}(X(\Fv),\RR_{>0})\to\widehat{\Pic^G_v}(X)$ from above. The statement is proven.
\end{proof}
In the rest of paragraph we study~$G$-invariant $\Fv$-metrics on the trivial line bundle. Let us denote by $\widehat{\sPic^G_v}(\OO_X)$ the full subcategory of $\widehat{\sPic^G_v}(X)$ given by~$G$-invariant $\Fv$-metrized line bundles $(L,\psi,|\lvert\cdot\rvert|)$ such that $L=\OO_X$. It is immediate that $\widehat{\sPic^G_v}(\OO_X)$ is a Picard category. Let $\widehat{\Pic^G_v}(\OO_X)$ be the abelian group formed by isomorphism classes of objects of $\widehat{\sPic^G_v}(\OO_X)$. The canonical inclusion \begin{equation}\label{canincpic}\widehat{\Pic^G_v}(\OO_X)\to\widehat{\Pic^G_v}(X)\end{equation} is an equality, if $\Pic(X)=0$.
\begin{lem}\label{metribun}
Let $\psi:G\times X\to\Gm$ be a~$G$-linearization of the trivial line bundle on~$X$. An~$G$-invariant $\Fv$-metric $|\lvert\cdot\rvert|$ satisfies  that $x\mapsto (||1||_x)^{-1}$ is a continuous function and that for every $x\in X(\Fv)$ and every $t\in G(\Fv)$ one has $$(||1||_{t\cdot x})^{-1}=|\psi (t,x)|_v (||1||_{ x})^{-1}.$$ Conversely, let $f:X(\Fv)\to\RR_{>0}$ be a continuous function such that for every $x\in X(\Fv)$ and every $t\in G(\Fv)$ one has $f(t\cdot x)=|\psi (t,x)|_vf( x)$. Then setting $||1||_x:=(f(x))^{-1}$ for every $x\in X(\Fv)$ defines a~$G$-invariant $\Fv$-metric on the~$G$-linearized line bundle $(\OO_X,\psi)$.
\end{lem}
\begin{proof}
The function $x\mapsto (||1||_x)^{-1}$ is continuous as $|\lvert\cdot\rvert|$ is an $\Fv$-metric. For $x\in X(\Fv)$ and $t\in G(\Fv)$, the linear map $\Fv=\OO_X(x)\to\OO_X(t\cdot x)=\Fv,$ induced from the~$G$-linearization $\psi$, is given by multiplication by $\psi(t,x).$ Now, the fact that $|\lvert\cdot\rvert|$ is~$G$-invariant gives that for every $x\in X(\Fv)$ and every $t\in G(\Fv)$ one has that $$||1||_x=||\psi(t,x)||_{t\cdot x}=|\psi(t,x)|_v||1||_{t\cdot x},$$i.e. $$(||1||_{t\cdot x})^{-1}=|\psi(t,x)|_v(||1||_x)^{-1}.$$ Suppose now $f:X(\Fv)\to\RR_{>0}$ is a continuous function such that for every $x\in X(\Fv)$ and every $t\in G(\Fv)$ one has $f(t\cdot x)=|\psi (t,x)|_vf( x)$. For every $x\in X(\Fv)$, set $||1||_{x}=(f(x))^{-1}$. The function $x\mapsto ||1||_x$ is continuous and hence defines an $\Fv$-metric on $\OO_X$. We have that $\Fv=\OO_X(x)\to\OO_X(t\cdot x)=\Fv$ maps~$1$ to $\psi(t,x)$ and thus $$(||1||_{t\cdot x})^{-1}=f(t\cdot x)=|\psi(t,x)|_vf(x)=|\psi (t,x)|_v (||1||_{ x})^{-1}.$$
\end{proof}
\subsection{} In this paragraph we compare the $\Fv$-metrized line bundles on the quotient $X/G$ and~$G$-linearized $\Fv$-metrized line bundles on the scheme~$X$. 

Let~$v$ be a place of~$F$ and let~$X$ be a locally of finite type $\Fv$-scheme. Let~$G$ be a locally of finite presentation $\Fv$-group scheme acting on~$X$. The quotient stack $X/G$ is an algebraic stack by \ref{propofstack} and locally of finite type by \cite[\href{https://stacks.math.columbia.edu/tag/06FM}{Lemma 06FM}]{stacks-project}.
\begin{lem}\label{qmetricpull}
Let $q:X\to X/G$ be the quotient morphism. Let $\mathcal F: \sPic(X/G)\to \sPic^G(X)$ be the equivalence defined in \ref{defoflift}. 
\begin{enumerate}
\item Let $|\lvert\cdot\rvert|$ be an $\Fv$-metric on $L\in\sPic(X/G).$ The $\Fv$-metric $q^*|\lvert\cdot\rvert|$ on $q^*L$ makes $\mathcal F(L)$ a~$G$-invariant $\Fv$-metrized line bundle that we denote $\widehat{\mathcal F}(L)$.
\item If $\ell:(L,|\lvert\cdot\rvert|)\to (L',|\lvert\cdot\rvert|)$ is a morphism of $\Fv$-metrized line bundles, then $\mathcal F(\ell):\widehat {\mathcal F}(L)\to \widehat {\mathcal F}(L')$ is an isometry, thus a morphism of~$G$-invariant $\Fv$-metrized line bundles.  In this case we set $\widehat {\FFF}(\ell):={\FFF}(\ell)$. 
\item If~$L$ is the trivial $\Fv$-metrized line bundle on ${X/G}$, then $\widehat{\FFF}(L)$ is the trivial~$G$-invariant $\Fv$-metrized line bundle.
\item The functor $\widehat \FFF:\widehat{\sPic}(X)\to\widehat{\sPic^G_v}(X)$ is an additive functor. 
\end{enumerate}
\label{komzanis}
\end{lem}
\begin{proof}
\begin{enumerate}
\item Let $x\in X(\Fv)$ and $t\in G(\Fv)$. The map $(q^*L)(x)\to (q^*L)(t\cdot x)$ given by the~$G$-linearization defining $\mathcal F(L)$ coincides with the linear map $(q^*L)(x)=L(q(x))\to L(q(t\cdot x))=(q^*L)(t\cdot x)$ given by the isomorphism $q(x)\xrightarrow{\sim} q(t\cdot x)$. The map $L(q(x))\to L(q(t\cdot x))$ is an isometry by the definition of an $\Fv$-metric, hence is $(q^*L)(x)\to (q^*L)(t\cdot x)$ is an isometry. The claim follows.
\item  The morphism $\mathcal F(\ell)$ identifies with the morphism $q^*\ell$, which is isometric by the third part of \ref{ugovn}. 
\item The~$G$-linearized line bundle defining $\widehat{\mathcal F}(\OO_{X/G})$ is the trivial~$G$-linearized line bundle. Moreover, the $\Fv$-metrized line bundle defining $\widehat{\mathcal F}(\OO_{X/G})$ is precisely the trivial $\Fv$-metrized line bundle by \ref{ugovn}. It follows that $\widehat{\mathcal F}(\OO_{X/G})$ is the trivial~$G$-invariant $\Fv$-metrized line bundle.
\item Let $(L,|\lvert\cdot\rvert|)$ and $(L',|\lvert\cdot\rvert|' )$ be two $\Fv$-metrized line bundles on $X/G$. The~$G$-linearized line bundle defining $\widehat{\FFF}((L,|\lvert\cdot\rvert|)\otimes(L',|\lvert\cdot\rvert|' ))$ is the~$G$-linearized line bundle ${\FFF}(L\otimes L')=\FFF(L)\otimes\FFF(L')$. The $\Fv$-metrized line bundle defining $\widehat{\FFF}((L,|\lvert\cdot\rvert|)\otimes(L',|\lvert\cdot\rvert|' ))$ is precisely the $\Fv$-metrized line bundle $(q^*(L\otimes L'),q^*(|\lvert\cdot\rvert|\otimes|\lvert\cdot\rvert|'))=(q^*L\otimes q^*L',q^*|\lvert\cdot\rvert|\otimes q^*|\lvert\cdot\rvert|').$ It follows that $\widehat\FFF$ is an additive functor. 
\end{enumerate}
\end{proof}
If follows from \ref{komzanis}, that we have a homomorphism of abelian groups \begin{equation}\label{picglin}\widehat{\Pic}_v(X/G)\to\widehat{\Pic^G_v}(X).\end{equation}
The map $X(\Fv)\to [(X/G)(\Fv)]$ is $G(\Fv)$-invariant by \ref{qstop}and thus if $f\in \mathscr C^0({[(X/G)(\Fv)],\RR_{>0}})$ its pullback along $[q(\Fv)]$ is an element of $\mathscr C^0_{G(\Fv)}(X(\Fv),\RR_{>0})$.
\begin{lem}\label{cuviiij}
The following diagram is commutative:\[\begin{tikzcd}
	0 & \mathscr C^0({[(X/G)(\Fv)],\RR_{>0}})& \widehat{\Pic}_v(X/G)  & {\Pic}(X/G) \\
	0 & \mathscr C^0_{G(\Fv)}(X(\Fv),\RR_{>0}) & \widehat{\Pic^G_v}(X)&\Pic^G(X).
	\arrow[from=1-1, to=2-1]
	\arrow[from=1-2, to=2-2]
	\arrow[from=1-2, to=2-2]
	\arrow[from=2-2, to=2-3]
	\arrow[from=1-1, to=1-2]
	\arrow[from=2-1, to=2-2]
	\arrow[from=1-2, to=1-3]
	\arrow[from=1-3, to=2-3]
	\arrow[from=1-3, to=1-4]
	\arrow[from=2-3, to=2-4]
	\arrow[from=1-4, to=2-4]
\end{tikzcd}\]
\end{lem}
\begin{proof}
The diagram
\[\begin{tikzcd}
	\widehat{\Pic_v}(X/G) & \Pic(X/G)  \\
	\widehat{\Pic^G_v}(X) & \Pic^G(X)
	\arrow[from=1-1, to=1-2]
	\arrow[from=2-1, to=2-2]
	\arrow[from=1-1, to=2-1]
	\arrow[from=1-2, to=2-2]
\end{tikzcd}\]
is commutative by the construction of (\ref{picglin}). Let $q:X\to X/G$ be the quotient~$1$-morphism. We prove the commutativity of 
\begin{equation}\label{viuby}\begin{tikzcd}
	 \mathscr C^0({[(X/G)(\Fv)],\RR_{>0}})& \widehat{\Pic_v}(X/G)  \\
	 \mathscr C^0_{G(\Fv)}(X(\Fv),\RR_{>0})& \widehat{\Pic^G_v}(X)
	\arrow[from=1-1, to=1-2]
	\arrow[from=2-1, to=2-2]
	\arrow[from=1-1, to=2-1]
	\arrow[from=1-2, to=2-2]
\end{tikzcd}\end{equation}
Let $f\in \mathscr C^0({[(X/G)(\Fv)],\RR_{>0}})$. Its image in $\mathscr C^{0}_{G(\Fv)}(X(\Fv),\RR_{>0})$ is $f\circ[q(\Fv)]$.  The image of $f\circ[q(\Fv)]$ in $\widehat{\Pic^G_v}(X)$ is $(\OO_{X},\Id_{\mathcal O_{G\times X}},|\lvert\cdot\rvert|^{f\circ[q(F_v)]})$, where $|\lvert\cdot\rvert|^{f\circ[q(F_v)]}$ is defined by $||1||^{f\circ[q(F_v)]}_x=f([q(\Fv)](x))$ for $x\in X(\Fv)$. The image of~$f$ in $\widehat{\Pic_v}(X/G)$ is the $\Fv$-metrized line bundle $(\OO_X, |\lvert\cdot\rvert|^f),$ where $|\lvert\cdot\rvert|^f$ is defined by $||1||^f_{y}=f(y)$ for $y\in [X/G](\Fv)$. The image of $(\OO_X, |\lvert\cdot\rvert|^f)$ in $\widehat{\Pic^G_v}(X)$ is the~$G$-invariant $\Fv$-metrized line bundle $(\OO_X,\Id_{\OO_{G\times X}},q^*|\lvert\cdot\rvert|^f). $ Note that $q^*||1||^f_{x}=||1||^f_{[q(\Fv)](x)}=f([q(\Fv)](x))$, i.e. $q^*|\lvert\cdot\rvert|^f= |\lvert\cdot\rvert|^{f\circ[q(Fv)]}.$ We deduce the commutativity of \ref{viuby}. The statement is proven.
\end{proof}
We give a conditions for the homomorphism (\ref{picglin}) to be injective and an isomorphism.
\begin{prop}\label{gmetmet}
Let~$v$ be a place of~$F$. Let~$G$ be a special locally of finite type $\Fv$-group scheme acting on locally of finite type $\Fv$-scheme~$X$. The homomorphism (\ref{picglin}) is injective. If moreover $\widehat{\Pic}_v(X/G)\to\Pic(X/G)$ is surjective, then (\ref{picglin}) is an isomorphism. 
\end{prop}
\begin{proof}
Consider the diagram from \ref{cuviiij}. Both horizontal sequences are exact by \ref{gfvexact} and \ref{ftopic}. Obviously, the first vertical homomorphism is an isomorphism. As~$G$ is special, one can identify $[(X/G)(\Fv)]$ with the topological quotient $X(\Fv)/G(\Fv)$ using \ref{qstop}. Thus, the second vertical homomorphism, given by pulling back continuous functions on $[X/G(\Fv)]=X(\Fv)/G(\Fv)$ to $X(\Fv),$ is an isomorphism. The fourth vertical homomorphism is an isomorphism by \ref{gpic}. By $4$-lemma, we deduce that $\widehat{\Pic}_v(X/G)\to\widehat{\Pic^G_v}(X)$ is injective. Suppose that $\widehat{\Pic}_v(X/G)\to\Pic(X/G)$ is surjective. The following diagram is commutative.
\[\begin{tikzcd}[column sep=small]
 0 & \mathscr C^0({[(X/G)(\Fv)],\RR_{>0}})& \widehat{\Pic}_v(X/G)&\Pic(X/G) & 0 \\
0 & \mathscr C^0_{G(\Fv)}(X(\Fv),\RR_{>0}) & \widehat{\Pic^G_v}(X)& \Pic^G(X)&E,
         \arrow[from=1-1, to=2-1]
	\arrow[from=1-1, to=1-2]
	\arrow[from=1-2, to=1-3]
	\arrow[from=1-3, to=1-4]
	\arrow[from=1-4, to=1-5]
	\arrow[from=1-5, to=2-5]
	\arrow[from=2-1, to=2-2]
	\arrow[from=2-2, to=2-3]
	\arrow[from=2-3, to=2-4]
	\arrow[from=2-4, to=2-5]
	\arrow[from=1-2, to=2-2]
	\arrow[from=1-3, to=2-3]
	\arrow[from=1-4, to=2-4]
\end{tikzcd}\]
where $E$ is the quotient of $\Pic^G(X)$ with the image of $\widehat{\Pic^G_v}(X)$. Obviously, the first and the fifth vertical homomorphisms are an isomorphism and a monomorphism, respectively. Again, the second and the fourth vertical homomorphisms are isomorphisms.  By $5$-lemma, we deduce that $\widehat{\Pic}_v(X/G)\to\widehat{\Pic^G_v}(X)$ is an isomorphism.
\end{proof}
\subsection{}\label{metricsonlx}
We give some conditions that will make $\wPic_v(X)\to\Pic(X)$ surjective for an algebraic stack~$X$. Let~$v$ be a place of~$F$. We are going to construct $\Fv$-metrics on line bundles on algebraic stacks, by pulling back metrics on line bundles on schemes. 
\begin{lem}\label{picmet}
Let~$X$ be a locally of finite type $\Fv$-algebraic stack satisfying the following condition: for every line bundle~$L$ on~$X$, there exist positive integer $m,$ a locally of finite type $\Fv$-scheme~$Z$ and a line bundle $L'$ on~$Z$ which is in the image of the canonical homomorphism $\widehat{\Pic}_v(Z)\to\Pic(Z),$ such that $L^{\otimes m}=g^*(L').$ Then $\widehat{\Pic}_v(X)\to\Pic(X)$ is surjective.
\end{lem}
\begin{proof}
The line bundle $L'$ admits an $\Fv$-metric, hence $L^{\otimes m}=g^*(L')$ admits an $\Fv$-metric by \ref{ugovn}.  It follows that~$L$ admits an $\Fv$-metric by \ref{mthroot}. The statement follows.
\end{proof}
\begin{rem}
\normalfont
In the case~$Z$ is a separated scheme, then every line bundle on~$Z$ admits an $\Fv$-metric. Indeed $Z(\Fv)$ is locally compact and Hausdorff topological space. Moreover, it is a finite union of paracompact spaces (if $U$ is an open affine subset of~$Z$, then $U(\Fv)$ is a closed subspace of $\Fv^n$, thus paracompact). By \cite[Proposition 18, \no 10, \S 9, Chapter I]{TopologieGj}, we deduce that $Z(\Fv)$ is paracompact.
By \cite[Proposition 4, \no 4, \S 4, Chapter IX]{TopologieGd}, the space $Z(\Fv)$ is normal. By \cite[Theorem 3, \no 3, \S 4, Chapter IX]{TopologieGd}, one can find a partitions of unity subordinated to every locally finite open covering and one can use them to construct metrics in a usual way.
\end{rem}
\begin{rem}
\normalfont
The following is true.  A locally of finite type $\Fv$-stack~$X$ of finite diagonal, that admits a separated good moduli space $p:X\to\mathcal X$ (see \cite[Definition 1.2]{Alper}), satisfies that $\wPic(X)\to\Pic(X)$ is surjective. Indeed, by the fact that the diagonal of~$X$ is finite and by \cite[Section 2]{Alper}, every line bundle~$L$ on~$X$ admits a multiple $L^{\otimes m}$ which is a pullback of a line bundle on $\mathcal X$. We can endow $L^{\otimes m}$ with the pullback of a metric on the line bundle on $\mathcal X$, and we can endow~$L$ with the ``$m$-th root" metric by Lemma \ref{mthroot}.
\end{rem}
\subsection{}
We will prove that every line bundle on $\PPP(\aaa)_{\Fv}$ admits an $\Fv$-metric without using remarks from \ref{metricsonlx}. Let us firstly prove the following lemma.
\begin{lem}\label{jlmor}
Let~$\ell$ be a positive integer divisible by $\lcm(\aaa)$. Consider the morphism \begin{align*}J(\ell):\AAA^n-\{0\}&\to\AAA^n-\{0\}\\\xxx&\mapsto (x_j^{\ell/a_j})_j.\end{align*} Endow the first $\AAA^n-\{0\}$ with the~$\Gm$-action with the weights $a_1\doots a_n,$ that is $t\cdot _{\aaa}\xxx=(t^{a_j}x_j)_j,$ and the second $\AAA^n-\{0\}$ with the~$\Gm$-action with the weights $1\doots 1$, that is $t\cdot_{\jed}\xxx=(tx_j)_j$.
\begin{enumerate}
\item The morphism $J(\ell)$ is $t\mapsto t^{\ell}$-equivariant, that is $$J(\ell)(t\cdot _{\aaa}\xxx)=t^{\ell}\cdot_{\jed}(J(\ell)(\xxx))\hspace{0.8 cm}\forall t\in\Gm,\forall\xxx\in \AAA^n-\{0\}.$$ The diagram \begin{equation}\label{ajlpa}\begin{tikzcd}
	\AAnz & \AAnz \\
	{\PPP(\aaa)} & {\mathbb P^{n-1}}.
	\arrow["{q^{\aaa}}"', from=1-1, to=2-1]
	\arrow["{\overline{J(\ell)}}", from=2-1, to=2-2]
	\arrow["{J(\ell)}"', from=1-1, to=1-2]
	\arrow["{q^{\jed}}", from=1-2, to=2-2]
\end{tikzcd}\end{equation} is $2$-commutative, where $\overline{J(\ell)}$ is given by \cite[\href{https://stacks.math.columbia.edu/tag/046Q}{Lemma 046Q}]{stacks-project}.
\item The pullback  $\overline{J(\ell)}^*(\mathcal O(k))$ for $k\in\ZZ$ is the line bundle $\mathcal O(\ell k)$.
\end{enumerate}
\end{lem}
\begin{proof}
\begin{enumerate}
\item One has that \begin{multline*}J(\ell)(t\cdot _{\aaa}\xxx)=J(\ell)((t^{a_j}x_j)_j)=(t^{\ell}x_j^{\ell/a_j})_j\\=t^{\ell}\cdot_{\jed}J(\ell)(\xxx)\hspace{0.5cm}\forall t\in\Gm,\forall\xxx\in \AAA^n-\{0\}.\end{multline*} The $2$-commutativity of \ref{ajlpa} follows from the universal property of $\overline{J(\ell)},$ see \cite[\href{https://stacks.math.columbia.edu/tag/0436}{Lemma 0436}]{stacks-project}.
\item By \ref{gjgdxy}, the pullback $\overline{J(\ell)}^*\mathcal O(k)$ is the line bundle determined by the pullback linearization ${J(\ell)}^*\psi$ of the trivial line bundle, where $\psi:(t,\xxx)\mapsto t^{k}$ is the~$\Gm$-linearization of $\OO_{\AAA^n-\{0\}}$ that defines $\mathcal O(k)$. One has that $$({J(\ell)})^*\psi=\psi(\phi(t),J(\ell)(\xxx))=\phi(t)^k=t^{\ell k}.$$It follows that $\overline{J(\ell)}^*(\mathcal O(k))=\mathcal O(\ell k).$
\end{enumerate}
\end{proof}
\begin{lem}\label{pagfi} The canonical morphism $\widehat{\Pic}_{v}(\PPP(\aaa)_{\Fv})\to \Pic(\PPP(\aaa)_{\Fv})$ is surjective.
The canonical homomorphism \begin{equation}\label{ppapic}\widehat{\Pic}_{v}(\PPP(\aaa)_{\Fv})\to\widehat{\Pic}^{\Gm}_{\Fv}((\AAA^n-\{0\})_{\Fv})\end{equation} is an isomorphism.
\end{lem}
\begin{proof}
Let~$\ell$ be an integer divisible by $\lcm(\aaa)$. 
By Lemma \ref{jlmor}, one has that $\overline{J(\ell)}^*(\mathcal O(1))=\mathcal O(\ell k).$ We deduce that every line bundle on~$\PPP(\aaa)$ admits a non-zero power which is a pullback of a line bundle on $\PP^{n-1}.$ The line bundle~$\mathcal O(1)$ on $\PP^{n-1}$ admits a metric, e.g. one can endow it with the Fubiny-Study metric. It follows from \ref{picmet}, that $\widehat{\Pic}_v(\PPP(\aaa))\to\Pic(\PPP(\aaa))$ is surjective. Proposition \ref{gmetmet} now gives that (\ref{ppapic}) is an isomorphism. 
\end{proof}
Let us dedicate the end of this paragraph to a dictionary given by \ref{metribun} for the case of weighted projective stacks.
\begin{mydef}\label{ahomfun}
Let $\vMF,$ let $d\in\CC$ and let $f:\Fvnz\to\RR_{\geq 0}$ be a function. We say that~$f$ is $\aaa$-homogenous of weighted degree~$d$ if for every $\xxx\in\Fvnz$ and every $t\in\Fvt$ one has $$f(t\cdot\xxx)=|t|^d_vf(\xxx).$$
\end{mydef}
\begin{lem}\label{cormeq}
Let~$v$ be a place of~$F$ and let~$k$ be an integer. If $|\lvert\cdot\rvert|$ is an $\Fv$-metric on the line bundle $\mathcal O(k)$ on $\PPP(\aaa)_{\Fv}$,  the pullback metric $(q^{\aaa}_{\Fv})^*|\lvert\cdot\rvert|$ on $\OO_{\AAA^n-\{0\}}$ is~$\Gm$-invariant and the function $f_{|\lvert\cdot\rvert|}:\xxx\mapsto (q^{\aaa}_{\Fv})^*||1(\xxx)||$ is an $\aaa$-homogenous continuous function $\Fvnz\to\RR_{>0}$ of weighted degree $k.$ 
Conversely, let $f:\Fvnz\to\RR_{>0}$ be an $\aaa$-homogenous continuous function of weighted degree~$k$. 
Then setting $||1(\xxx)||':=(f(\xxx))^{-1}$ defines a~$\Gm$-invariant $\Fv$-metric on the~$\Gm$-linearized line bundle $(\OO_{\AAA^n-\{0\}},(t,\xxx)\mapsto t^k)$. Furthermore, there exists a metric $|\lvert\cdot\rvert|_f$ on $\OO(k)$ such that $(q^{\aaa}_{\Fv})^*|\lvert\cdot\rvert|=|\lvert\cdot\rvert|'$.
\end{lem}
\begin{proof}
The line bundle $\OO(k)$ is given by the linearization $\psi:(t,\xxx)\mapsto t^k$ of $\OO_{\AAA^n-\{0\}}$. Let $|\lvert\cdot\rvert|$ be an $\Fv$-metric on $\OO(k)$. The pullback metric $(q^{\aaa}_{\Fv})^*|\lvert\cdot\rvert|$ is~$\Gm$-invariant by \ref{qmetricpull}. By \ref{metribun}, the function $\xxx\mapsto (q^{\aaa}_{\Fv})^*||1(\xxx)||$ is continuous satisfies that
$$((q^{\aaa}_{\Fv})^*||1(t\cdot \xxx)||)^{-1}=|t^k|_v ((q^{\aaa}_{\Fv})^*||1(\xxx)||)^{-1} ,$$
i.e. it is $\aaa$-homogenous continuous of weighted degree~$k$. Let us prove the converse claim. Note that \ref{metribun} gives that $|\lvert\cdot\rvert|'$ is an $\Gm(F_v)$-invariant metric on $(\OO_{\AAA^n-\{0\}},(t,\xxx)\mapsto t^k)$. By \ref{pagfi}, there exists an $\Fv$-metric $|\lvert\cdot\rvert|$ on~$\mathcal O(1)$ such that $(q^{\aaa}_{\Fv})^*|\lvert\cdot\rvert|=|\lvert\cdot\rvert|'.$ The statement is proven.  
\end{proof}
\begin{rem}
\normalfont
By allowing that~$f$ takes value $0$, we allow ``singular" metrics.
\end{rem}
\section{Heights on $\PPP(\aaa)(F)$}\label{heightsonppa}
In this section we will define heights on $\PPP(\aaa)(F)$. An obvious approach is to define heights to be pullbacks of heights for some morphism to a scheme. However, such heights exhibit a drawback, they do not satisfy the Northcott property and hence are not suitable for countings. For the purpose of counting, we define quasi-toric heights, for which in \ref{finitenesstoric} we establish that they satisfy the Northcott property. 
\subsection{}In this paragraph we define heights on stacks. 
We give a condition that will enable us to define heights.
\begin{mydef}\label{fahom}
Let $k\in\ZZ$. For $v\in M_F$, let $f_v:\Fvnz\to\RR_{\geq 0}$ be an $\aaa$-homogenous function of weighted degree~$d$ and let us set $$E_v:=\{\xxx\in F_v^n-\{0\} | \forall j \hspace{0.1cm}|x_j|_v=1\text{ or } x_j=0.\} $$
We say that a family $(f_v)_v$ is generalized adelic if for almost every $v\in M_F^0$, one has that $$f|_{E_v}=1.$$ 
\end{mydef}
We use the language of $\aaa$-homogenous continuous functions to define heights.
\begin{lem} 
Let $(f_v)_{v\in M_F}$ be a generalized adelic family of $\aaa$-homogenous functions $f_v:\Fvnz\to\RR_{\geq 0}$ of weighted degree $d\in\ZZ$. For $\xxx\in\PPP(\aaa)(F)$, let us denote by $\wx:(\Gm)_{\Fv}\to(\AAA^n-\{0\})_{\Fv}$ the induced $(\Gm)_{\Fv}$-equivariant morphism defined by~$\xxx$. For every $\xxx\in\PPP(\aaa)(F)$, the product \begin{equation}H((f_v)_v)(\xxx):=\prod_{\vMF}f_v(\wx(1))\label{huhub}\end{equation} is a finite product. Moreover, if $\xxx\xrightarrow{\sim}\yyy$ is a $2$-isomorphism, then $H((f_v)_v)(\xxx)=H((f_v)_v)(\yyy)$. 
\end{lem}
\begin{proof}
Let $\xxx\in\PPP(\aaa)(F).$  If $\widetilde x_j(1)\neq 0$, then for almost every $v\in M_F^0$ one has $|\widetilde x_j(1)|_v=1$. We conclude that for almost all $v\in M_F^0$, one has $\wx(1)\in E_v$.  
As $(f_v)_v$ is generalized adelic, we deduce that the product (\ref{huhub}) is indeed finite. Let now $\xxx\xrightarrow{\sim}\yyy$ be a $2$-isomorphism, it is given by an element $t\in\Gm(\Fv)$ such that $t\cdot\wx(1)=\widetilde{\yyy}(1)$. By the product formula, one has that $$\prod_{\vMF}f_v(\widetilde{\yyy}(1))=\prod_{\vMF}f_v(t\cdot\wx(1))=\prod_{\vMF}|t|^k_vf_v(\wx(1))=\prod_{\vMF}f_v(\wx(1)).$$
 \end{proof}
 \begin{mydef}
Let $(f_v)_v$ be a generalized adelic family of $\aaa$-homogenous continuous functions $\Fvnz\to\RR_{\geq 0}$ of weighted degree~$k$. 
\begin{enumerate}
\item The function that associates to $\xxx\in\PPP(\aaa)(F)$ the value of the product (\ref{huhub}), we call the resulting height defined by the family $(f_v)_v$ and we denote it by $H((f_v)_v)$.
\item For $\xxx\in[\PPP(\aaa)(F)],$ we define $H((f_v)_v)(\xxx)$ by setting it to be $H((f_v)_v)(\yyy)$ where $\yyy\in\PPP(\aaa)(F)$ is such that the isomorphism class of~$\yyy$ is~$\xxx$.
\end{enumerate}
\end{mydef}
When there is no confusion, we call it simply the height. When two families consist of functions which are non-vanishing at every place and which coincide at almost ever places, the resulting heights can be compared as follows.
\begin{lem}\label{toricisclos}
Let $d\in\ZZ$, and let $(f_v:\Fvnz\to\RR_{>0})_v$ and $(f_v':\Fvnz\to\RR_{>0})_v$ be two generalized adelic degree~$d$ families of $\aaa$-homogenous functions. Suppose that for almost all $v\in M_F$, one has that $f_v=f_v'$. There exist $C_1,C_2>0$ such that for every $\xxx\in\PPP(\aaa)(F)$ one has that $$C_1<\frac{H((f_v)_v)(\xxx)}{H((f_v')_v)(\xxx)}<C_2.$$
\end{lem}
\begin{proof}
For $v\in M_F$, the function $$\Fvnz\to\RR_{>0}\hspace{1cm}\yyy\mapsto \frac{f_v(\yyy)}{f_v'(\yyy)}$$ is $\Gm(\Fv)$-invariant, thus descends to a positive valued function on $[\PPP(\aaa)(\Fv)],$ which is compact by \ref{paraap}. We deduce that for every $\vMF$, there exist $C_{v,1},C_{v,2}>0$ such that for every $\yyy\in\Fvnz$ one has that  $$C_{v,1}<\dfrac{f_v(\yyy)}{f_v'(\yyy)}<C_{v,2}.$$
Let us write~$H$ for $H((f_v)_v)$ and $H'$ for $H((f_v')_v)$. Let $\yyy\in\Fvnz$ be a lift of~$\xxx$. Let~$S$ be the finite set of places, for which $f_v\neq f_v'$. For every $v\not\in S$, one has that $f_v=f_v'$. We deduce that $$\frac{H(\xxx)}{H'(\xxx)}=\prod_{v\in S}\frac{f_v(\wx)}{f_v'(\wx)}<\prod_{v\in S}C_{v,2}$$ and that  $$\prod_{v\in S}C_{v,1}<\prod_{v\in S}\frac{f_v(\wx)}{f_v'(\wx)}=\frac{H(\xxx)}{H'(\xxx)}. $$
The statement is proven.
\end{proof}
\subsection{} In this paragraph, we define stable heights on $\PPP(\aaa)(F)$, where $\aaa\in\ZZ^n_{\geq 1}$. These heights are pullbacks of the heights on varieties. Such height~$H$ satisfies that  if $\xxx,\yyy\in\PPP(\aaa)(F)$ are such that $\xxx_K\cong\yyy_K$ for some extension $K/F$, then $H(\xxx)=H(\yyy)$. The definition of the height can hence be naturally carried to any $\overline{F}$-point of~$\PPP(\aaa)$. They are called ``stable" because the height of an~$F$-point stays invariant when regarding this point as a~$K$-point of $[\PPP(\aaa)]$. We will see later in \ref{finitenesstoric}, that such heights do not always satisfy the Northcott property.

If $\vMF$ and $|\lvert\cdot\rvert|$ is an $\Fv$-metric on the line bundle $\OO(k)$ on $\PPP(\aaa)_{\Fv}$, for some $k\in\ZZ$, we will denote by $f_{|\lvert\cdot\rvert|}:\Fvnz\to\RR_{>0}$ the $\aaa$-homogenous function of weighted degree~$k$ given by Lemma \ref{cormeq}. 

If $r\geq 0$ is an integer, let us denote by $|\lvert\cdot\rvert|_{v,\max}$ the metric on~$\mathcal O(1)$ on~$\PP^{r}$ given by the $\jed$-homogenous continuous function $f_v^{\#}:\Fvnz\to\RR_{>0}$ of weighted degree~$1$ $$f_v^{\#}:\xxx\mapsto \max_{j}(|x_j|_v).$$

By a stable $\Fv$-metric on a line bundle~$L$ on locally of finite type $\Fv$-scheme~$Z$, we mean a metric which is the restriction of a metric on the analytic line bundle $L^{\an}$ on the analytic space $Z^{\an}_{\Fv}$. The metric $|\lvert\cdot\rvert|_{v,\max}$ on the line bundle~$\mathcal O(1)$ on the projective space is stable.
\begin{mydef}\label{defofstmet}
Let~$X$ be a locally of finite type $\Fv$-algebraic stack and let~$L$ be a line bundle on~$X$. An $\Fv$-metric $|\lvert\cdot\rvert|$ on~$L$ is said to be stable, if there exist an integer $m\neq 0$, a locally of finite type $\Fv$-scheme $Z,$ an $\Fv$-metrized line bundle $(L',|\lvert\cdot\rvert|')$ on~$Z$ with $|\lvert\cdot\rvert|'$ stable and a~$1$-morphism of algebraic stacks $g:X\to Z$ such that $(L^{\otimes m},|\lvert\cdot\rvert|^{\otimes m})=g^*(L',|\lvert\cdot\rvert|')$.
\end{mydef}
Recall that an adelic metric on the line bundle~$\mathcal O(1)$ on $\PP^{n-1}$ is a collection of metrics $(|\lvert\cdot\rvert|_v)_v,$ where each $|\lvert\cdot\rvert|_v$ is a stable metric on the line bundle~$\mathcal O(1)$ on $\PP^{n-1}_{\Fv},$ and for almost all $v,$ one has  $|\lvert\cdot\rvert|_v=|\lvert\cdot\rvert|_{v,\max}$.
\begin{mydef}
Let~$k$ be an integer. Let $(f_v)_v$ be a generalized adelic family of $\aaa$-homogenous continuous functions of weighted degree~$k$. We say that $(f_v)_v$ is stable if there exists an integer $\ell\neq 0$, an integer $r\geq 0,$ an adelic metric $(|\lvert\cdot\rvert|_v)_v$ on the line bundle~$\mathcal O(1)$ on~$\PP^{r}$ and a~$1$-morphism of algebraic stacks $h:\PPP(\aaa)\to\mathbb P^r$ such that the line bundle~$\mathcal O(\ell k)$ is the pullback line bundle $h^*\mathcal O(1)$ and such that for every~$v$ one has that $f^{\ell}_v=f_{h^*|\lvert\cdot\rvert|_v}$. In this case, we also say that the height $H=H((f_v)_v)$ is stable.
\end{mydef}
\begin{lem}\label{propertyofstable}
Let~$k$ be an integer, let $(f_v)_v$ be a stable generalized adelic family of $\aaa$-homogenous continuous functions of weighted degree~$k$. Let $H=H((f_v)_v)$ be the corresponding height. Let $\xxx,\yyy\in\PPP(\aaa)(F)$ be such that there exists an extension $K/F$ and an isomorphism $\xxx|_K\xrightarrow{\sim}\yyy_K$ in $\PPP(\aaa)(K)$. Then $H(\xxx)=H(\yyy)$. 
\end{lem}
\begin{proof}
There exists an $\ell\neq 0$, an integer $r\geq 0,$ an adelic metric $(|\lvert\cdot\rvert|_v)_v$ on the line bundle $\mathcal O(1),$ a~$1$-morphism of algebraic stacks $g:\PPP(\aaa)\to\PP^r$ such that $\mathcal O(\ell k)=g^*\mathcal O(1)$ and such that for each~$v$ one has $f_v^{\ell}=f_{g^*|\lvert\cdot\rvert|_v}$. Let $H_{\PP^{n-1}}$ be the height on $\PP^{n-1}(F)$ given by $(|\lvert\cdot\rvert|_v)_v$. For $\xxx\in\PPP(\aaa)(F)$ that \begin{multline*}H(\xxx)^{\ell}=\prod_v f_{v}(\wx(1))^{\ell}=\prod_vf_{g^*|\lvert\cdot\rvert|_v}(\wx(1))=\prod_v(g^*||1(\xxx)||_{v})^{-1}\\=\prod_v||1(g(\xxx)||^{-1}_v=H_{\PP^{n-1}}(g(\xxx)).\end{multline*}
The image of $\xxx_K$ in $\PP^{n-1}(K)$ is precisely $g(\xxx)_K$ and the image of $\yyy_K$ in~$\PP^{n-1}(K)$ is precisely $g(\xxx)_K$. The existence of an isomorphism $\xxx_K\xrightarrow{\sim}\yyy_K$ gives that $g(\xxx)_K=g(\yyy)_K$. It follows that $g(\xxx)=g(\yyy)$, and hence $$H(\xxx)=H_{\PP^{n-1}}(g(\xxx))=H_{\PP^{n-1}}(g(\yyy))=H(\yyy).$$
 \end{proof}
We give an example of a stable family. 
\begin{lem}\label{brzostab}
For $\vMF$, the functions $f_v:\Fvnz\to\RR_{>0}$ given by $\xxx\mapsto \max(|x_j|^{1/a_j}_v)$ are continuous and $\aaa$-homogenous of weighted degree~$1$. The family $(f_v)_v$ is generalized adelic and stable. Moreover, the resulting height $H^{\max}=H((f_v))$ satisfies for every $\xxx\in\PPP(\aaa)(F)$ that $H(\xxx)\geq 1$.
\end{lem}
\begin{proof}
It is evident that $f_v$ is continuous and that $f_v(t\cdot\xxx)=\max_j(|t^{a_j}x_j|^{1/a_j}_v)=|t|_v\max_j(|x_j|_v^{1/a_j})=|t|_vf_v(\xxx)$, and thus for every $\vMF$ one has that $f_v$ is a continuous $\aaa$-homogenous function of weighted degree~$1$. Moreover, for every $\vMF$, one has that if $\yyy\in E_v=\{\xxx\in\Fvnz|\hspace{0.1cm}|x_j|_v=1\text{ or }x_j=0\}$, then $f_v(\yyy)=1$. Thus, the family $(f_v)_v$ is generalized adelic.

We will verify that $(f_v)_v$ is stable. Let $\vMF$, an $\Fv$-metric on $\OO_{\AAA^n-\{0\}}$ identifies with a continuous function $g:\Fvnz\to\RR_{>0}$ by setting $g(\xxx)=||1(\xxx)||^{-1}$, where $\xxx\in\Fvnz$ and $1(\xxx)$ is the value of $1\in\Gamma(\AAnz,\OO_{\AAnz})$ at~$\xxx$.
 Let us set $\ell=\lcm(\aaa)$. Consider the morphism $J(\ell):\AAA^n-\{0\}\to\AAnz$ given by $$J(\ell):\xxx\mapsto (x_j^{\ell/a_j})_j.$$ In \ref{jlmor}, we have established that $J(\ell)$ is $(t\mapsto t^{\ell})$-invariant and that $J(\ell)^*\mathcal O(1)=\mathcal O(\ell).$ Moreover, the following diagram is $2$-commutative: 
\[\begin{tikzcd}
	\AAnz & \AAnz \\
	{\PPP(\aaa)} & {\mathbb P^{n-1}}.
	\arrow["{q^{\aaa}}"', from=1-1, to=2-1]
	\arrow["{\overline{J(\ell)}}", from=2-1, to=2-2]
	\arrow["{J(\ell)}"', from=1-1, to=1-2]
	\arrow["{q^{\jed}}", from=1-2, to=2-2]
\end{tikzcd}\]
By \ref{ugovn}, it follows that the diagram 
\[\begin{tikzcd}
	{\widehat{\Pic_v}(\AAnz)} & {\widehat{\Pic_v}(\AAnz)} \\
	{\widehat{\Pic_v}(\PPP(\aaa))} & {\widehat{\Pic_v}(\mathbb P^{n-1})}
	\arrow["{(q^{\aaa}_{\Fv})^*}"', from=2-1, to=1-1]
	\arrow["{\overline{J(\ell)}^*}"', from=2-2, to=2-1]
	\arrow["{J(\ell)^*}", from=1-2, to=1-1]
	\arrow["{(q^{\jed}_{\Fv})^*}", from=2-2, to=1-2]
\end{tikzcd}\]
is commutative. Hence, the image of the $\Fv$-metrized line bundle $(\OO(k),|\lvert\cdot\rvert|_{v,\max})$ under $(q^{\aaa}_{\Fv})^*\circ \overline{J(\ell)}^*$ identifies with \begin{align*}J(\ell)^*(q^{\jed}_{\Fv})^*(\OO(k),|\lvert\cdot\rvert|_{v,\max})&=(\OO_{\AAnz},J(\ell)^*(q^{\jed}_{\Fv})^*|\lvert\cdot\rvert|_{v,\max})\\
&=(\OO_{\AAnz},\xxx\mapsto J(\ell)^*(\xxx\mapsto \max_j|x_j|_v))\\
&=(\OO_{\AAnz},\xxx\mapsto \max_j(|x_j|^{\ell/a_j}_v))\\
&=(\OO_{\AAnz},\xxx\mapsto f_v(\xxx)^{\ell}).
\end{align*} 
The~$\Gm$-invariant $\Fv$-linearized line bundle $(\OO_{\AAnz},\xxx\mapsto f_v(\xxx)^{\ell})$ is precisely the image of $\overline {J(\ell)}^*(\OO(1),|\lvert\cdot\rvert|_{v,\max})=(\mathcal O(k),\overline {J(\ell)}^*|\lvert\cdot\rvert|_{v,\max})$ under the pullback $(q^{\aaa}_{\Fv})^*$, so that $f_{v}^{\ell}=f_{J(\ell)^*|\lvert\cdot\rvert|_{v,\max}}.$ It follows that the family $(f_v)_v$ is stable. 

Finally, let us prove the estimate $H(\xxx)\geq 1$. Let $\xxx\in\PPP(\aaa)(F)$ and let~$\wx$ be the~$\Gm$-equivariant~$F$-morphism $\Gm \to \AAA^n-\{0\}$ given by~$\xxx$. Let~$i$ be an index such that $\wx(1)_i\neq 0$. Let $K=F\big(\sqrt[a_i]{\widetilde x_i(1)}\big).$ For $\vMF$, the absolute value $\lvert\cdot\rvert_v$ on~$F$ admits a unique extension to~$K$. By the product formula, we get that \begin{align*}H(\xxx)&=\prod_{\vMF}\max_j(|\widetilde x_j(1)|_v^{1/a_j})\\&=\prod_{\vMF}|\sqrt[a_i]{\widetilde x_i(1)})|_v\max_j\bigg(\bigg|\frac{\widetilde x_j(1)}{(\sqrt[a_i]{\widetilde x_j(1)})^{a_j}}\bigg|^{1/a_j}_v\bigg)\\&=\prod_{\vMF}\max_j\bigg(\bigg|\frac{\widetilde x_j(1)}{(\sqrt[a_i]{\widetilde x_j(1)})^{a_j}}\bigg|^{1/a_j}_v\bigg)\\&\geq 1. \end{align*} 
\end{proof}
\subsection{} \label{defrv} In this paragraph we present an example of a height on~$\PPP(\aaa)$ which is intrinsic to stacks. We will call them quasi-toric heights. In \ref{finitenesstoric}, we are going to show that they satisfy the Northcott property. 
In \ref{begdav}, we have introduced the set $\Dav=\Ov^n-(\piv^{a_1}\Ov)\times\cdots\times(\piv^{a_n}\Ov)$. 
For $v\in M_F^0$, we have defined in \ref{begdav} a function $r_v:\Fvnz\to\ZZ$ by $$r_v(\xxx)=\inf\{k\in\ZZ| \piv^{k}\cdot\xxx\in\Ov^n\}$$ and we have established that for every $\xxx\in\Fvnz$ that $$\piv^{r_v(\xxx)}\cdot\xxx\in\Dav.$$ 
\begin{lem}\label{davdavdav}
Let $\vMFz$. 
\begin{enumerate}
\item For every $\xxx\in\Fvnz$, one has that $$r_v(\xxx)=\sup_{\substack{j=1\doots n\\x_j\neq 0}}\bigg\lceil-\frac{v(x_j)}{a_j}\bigg\rceil.$$
\item For every $t\in\Fvt$ and every $\xxx\in\Fvnz$, one has that $$r_v(t\cdot\xxx)=r_v(\xxx)-v(t). $$
\item For every $k\in\ZZ$ one has that $$\{\xxx\in\Fvnz|r_v(\xxx)=k\}=\piv^{-k}\cdot\Dav.$$
\item Let $k\in\ZZ$. Suppose that for $u\in\Gm(\Fv)$ one has $u\cdot(\piv^{k}\cdot\Dav)\cap(\piv^{k}\cdot\Dav)\neq\emptyset$, then $v(u)=0$.
\item The  function $$f^{\#}_v:\Fvnz\to\RR_{>0}\hspace{1cm}\xxx\mapsto \pivv^{-r_v(\xxx)} $$ is $\aaa$-homogenous of weighted degree~$1$ and is locally constant.
\end{enumerate}
\end{lem}
\begin{proof}
\begin{enumerate}
\item For every index~$j$ such that $x_j\neq 0$, one has that $\piv^{a_jk}x_j\in\Ov$ if and only if $k\geq -\frac{v(x_j)}{a_j}.$ We conclude $$r_v(\xxx)=\bigg\lceil\sup _{\substack{j=1\doots n\\x_j\neq 0}}\frac{-v(x_j)}{a_j}\bigg\rceil=\sup_{\substack{j=1\doots n\\x_j\neq 0}}\bigg\lceil-\frac{v(x_j)}{a_j}\bigg\rceil.$$
\item We observe that $$r_v(t\cdot\xxx)=\sup_{\substack{j=1\doots n\\x_j\neq 0}}\bigg\lceil-\frac{v(t^{a_j}x_j)}{a_j}\bigg\rceil=\sup_{\substack{j=1\doots n\\x_j\neq 0}}\bigg\lceil-\frac{v(x_j)+a_jv(t)}{a_j}\bigg\rceil=r_v(\xxx)-v(t).$$
\item We verify the claim for $k=0$. Pick $\xxx\in\Dav.$  For every index~$j$ one has $v(x_j)\geq 0$ and there exists an index~$i$ such that $x_i\not\in\piv^{a_i}\Ov,$ i.e. such that $v(x_i)<a_i$. This implies that $$0\geq r_v(\xxx)= \sup_{\substack{j=1\doots n\\x_j\neq 0}}\bigg\lceil-\frac{v(x_j)}{a_j}\bigg\rceil\geq 0.$$
Pick now $\xxx\in\Fvnz$ such that $r_v(\xxx)=0$. This means that for every~$j$ such that $x_j\neq 0$, one has that $$\frac{-v(x_j)}{a_j}\leq 0,$$i.e. that $v(x_j)\geq 0$ and that there exists an index~$i$ such that $$-1<\frac{-v(x_i)}{a_i}\leq 0,$$ i.e. such that $0\leq v(x_i)<a_i.$ We deduce that $\xxx\in\Ov^n$ and that $\xxx\not\in\piv^{a_1}\Ov\times\cdots\times\piv^{a_n}\Ov$, i.e. one has $\xxx\in\Dav$. Let now $k\neq 0$ be an integer. It follows from (1) and the case $k=0$ that \begin{align*}\{\xxx\in\Fvnz|r_v(\xxx)=k\}&=\{\xxx\in\Fvnz|r_v(\piv^{k}\cdot\xxx)=0\}\\&=\{\xxx\in\Fvnz|\piv^k\cdot\xxx\in\Dav\}\\&=\piv^{-k}\cdot\Dav.\end{align*} 
\item We have that $$\emptyset\neq\piv^{-k}\cdot (u\cdot(\piv^k\cdot\Dav)\cap(\piv^k\cdot\Dav))=(u\cdot\Dav)\cap\Dav.$$ If $v(u)>0$, then $(u\cdot\Dav)\subset \piv^{a_1}\Ov\times\cdots\times\piv^{a_n}\Ov$ and hence $(u\cdot\Dav)\cap \Dav=\emptyset,$ a contradiction. Suppose $v(u)<0,$ then $(u\cdot\Dav)\subset\piv^{-a_1}\Ov\times\cdots\times\piv^{-a_n}\Ov-\Ov^n,$ and hence $u\cdot\Dav\cap\Dav=\emptyset$, a contradiction. We deduce $v(u)=0$.
\item Let $t\in\Fvt$. One has that $$f^{\#}_v(t\cdot\xxx)=\pivv^{-r_v(t\cdot\xxx)}=\pivv^{-r_v(\xxx)+1}=\pivv\cdot f^{\#}_v(\xxx),$$
hence $f_v^{\#}$ is $\aaa$-homogenous. We have seen in \ref{begdav} that $\Dav$ is open and closed in $\Fv^n$, and hence in $\Fvnz$.  Hence, the sets $\piv^k\cdot\Dav$ are open and closed in $\Fvnz$ for every $k\in\ZZ$. It follows that $r_v$ is locally constant, so is $f_v^{\#}$. 
\end{enumerate}
\end{proof}
\begin{mydef}\label{tordef} Let $\vMFz$. 
We call the $\aaa$-homogenous function $f^{\#}_v:F_v^n-\{0\}\to\RR_{> 0}$ of weighted degree $d\in\ZZ$ given by 
 \begin{align*}\quad
f_v:\xxx\mapsto&
               \begin{cases}
 \pivv^{-dr_v(\xxx)}&\text{if } \vMFz,\\
(\max_{j}(|x_j|^{1/a_j}_v))^d&\text{if }\vMFi,
               \end{cases}
              \end{align*} the toric $\aaa$-homogenous function of weighted degree~$d$. 
\end{mydef}
We remark that when $\aaa=\jed,$ for every $\vMF$, the toric $\jed$-homogenous function of weighted degree~$1$ is given by $f_v^{\#}:\xxx\mapsto \max_{j}(|x_j|_v)$.
\begin{mydef}\label{familyquasitoric}
Let $(f_v:\Fvnz\to\RR_{\geq 0})_v$ be a family of $\aaa$-homogenous continuous functions of weighted degree~$d$. We say that the degree of the family $(f_v)_v$ is~$d$.
\begin{enumerate}
\item We say that $(f_v)_v$ is quasi-toric if for almost all $v\in M_F^0$, the function $f_v$ is toric. 
\item We say that $(f_v)_v$ is toric if for every~$v$ one has that $f_v=f_v^{\#}$, where $f_v^{\#}$ is the toric $\aaa$-homogenous function of weighted degree~$d$. 
\end{enumerate}
\end{mydef}
For every $\vMFz,$ as $E_v\subset\Dav$, one has that $f^{\#}_v|_{E_v}=1$. Thus every quasi-toric family $(f_v)_v$ is automatically generalized adelic.
\begin{mydef}
Let $(f_v)_v$ be a quasi-toric family of $\aaa$-homogenous continuous functions of a fixed weighted degree. We say that the resulting height $H((f_v)_v)$ is quasi-toric. If $(f_v)_v$ is furthermore assumed to be toric, we say that~$H$ is toric and may be denoted by $H^{\#}$.
\end{mydef}
\begin{exam}
\normalfont
We present a formula for the toric height in the case $F=\QQ$. Let~$d$ be a strictly positive integer. Let $\xxx\in[\PPP(\aaa)(\QQ)]$ and let $\wx\in\ZZ^n$ be a lift of~$\xxx$ which satisfies $\xxx\in\DD^{\aaa}_p$ for every prime~$p$. Then for every prime~$p$, one has that $f^{\#}_p(\wx)=1$. It follows that the toric height defined by the degree~$d$ toric family satisfies: $$H^{\#}(\xxx)=(\max(|\widetilde x_j|^{1/a_j}))^d.$$ When $n=2$ and $\aaa=(4,6)$, the toric height defined by the degree $12$ toric family is sometimes called {\it naive height} (e.g. \cite{Brumer}).  
\end{exam}
%
\begin{rem}
\normalfont
Quasi-toric heights are not stable. In fact, we will verify in \ref{finitenesstoric}, that quasi-toric heights satisfy the Northcott heights, while the stable heights do not 
\end{rem}
\begin{lem}\label{htoricbig}
For $\vMF$, let $f_v^{\#}$ be the toric $\aaa$-homogenous function of weighted degree~$1$. 
\begin{enumerate}
\item For every $\vMF$ and every $\yyy\in\Fvnz$, one has that $f_v^{\#}(\yyy)\geq \max(|y_j|^{1/a_j}_v)$. 
\item For every $\xxx\in\PPP(\aaa)(F)$ one has that $H^{\#}(\xxx)\geq H^{\max}(\xxx)\geq 1,$ where~$H^{\max}$ is the height given by  $(\yyy\mapsto\max(|y_j|^{1/a_j}_v))_v.$
\end{enumerate}
\end{lem}
\begin{proof}
For $\vMFi$, we recall that $f_v^{\#}(\yyy)=\max(|y_j|^{1/a_j}_v)$. Let $\vMFz$.  For every $\yyy\in\Dav$ one has that $$r_v(\yyy)\geq \sup_{\substack{j=1\doots n\\x_j\neq 0}}\big(\frac{-v(y_j)}{a_j}\big)=:r^{\max}_v(\yyy)$$and thus $$f_v(\yyy)=\pivv^{-r_v(\yyy)}\geq \pivv^{-r^{\max}_v(\yyy)}\geq (\max(|y_j|^{1/a_j}_v).$$ The function $\yyy\mapsto f_v^{\#}(\yyy)(\max(|y_j|^{1/a_j}_v))^{-1}$ is $\Gm(\Fv)$-invariant. As by \ref{davdavdav} any $\zzz\in\Fvnz$ writes us $t\cdot\yyy$, where $\yyy\in\Dav$ it follows that $f^{\#}_v(\zzz)\geq (\max(|z_j|^{1/a_j}_v)$ for every $\zzz\in\Fvnz.$ The first claim is proven. 

Now let $\xxx\in\PPP(\aaa)(F)$ and let $\wx:\Gm\to\AAA^n-\{0\}$ be the~$\Gm$-equivariant morphism over~$F$ defined by~$\xxx$. We have that $$H^{\#}(\xxx)=\prod_{\vMF}f_v^{\#}(\wx(1))\geq \prod_{\vMF}\max_j(|\wx(1)|_v^{1/a_j})=H^{\max}(\xxx)\geq 1,$$ by \ref{brzostab}. The second claim is proven. 
\end{proof}
Let us end the paragraph by a proof of a lemma that will be used latter for the proof of Proposition \ref{comphh}. It is motivated by \cite[Theorem B.2.5]{HindrySilv}.
\begin{lem}\label{estisilv}
Let $P_1\doots P_{r+1}\in F[X_1\doots X_n]$ be $\aaa$-homogenous polynomials of the same weighted degree~$d$. 
Let $\overline{J}(P_1\doots P_{r+1}):(\mathscr P(\aaa)-\mathcal Z(P_1\doots P_{r+1}))\to\PP^{r}$ be the~$1$-morphism given by the $t\mapsto t^d$-equivariant morphism $(P_1\doots  P_{r+1}):(\AAA^n-\{0\})\to(\AAA^n-\{0\})$ (see \ref{phiequiv}). Let $H^{\#}$ be the toric height on $[\PPP(\aaa)(F)]$ defined by the toric degree~$d$ family. There exists $C>0$ such that for all $\xxx\in[(\PPP(\aaa)-Z(P_1\doots P_{r+1}))(F)]$ one has that $$CH^{\#}(\xxx)\geq H_{\PP^{r}}(\phi(\xxx))^d,$$where $H_{\PP^r}$ is the toric height defined by the toric degree~$1$ family on $\PP^r$.
\end{lem} 
\begin{proof}
The strategy from the proof of \cite[Theorem B.2.5]{HindrySilv} applies here. Let us denote by $X^{\mmm}$ the monomial $X_1^{m_1}\doots X_n^{m_n},$ where $\mmm\in\ZZ^n_{\geq 0}$. We denote by $w(\aaa,d)$ the number of $\mmm\in\ZZ^n_{\geq 0}$ for which $\aaa\cdot\mmm=d$, this is precisely the number of monomials which have $\aaa$-weighted degree equal to~$d$. For $i\in\{1\doots r+1\}$, we write $$P_i=\sum_{\aaa\cdot\mmm=d}A_{i,\mmm}X^{\mmm},$$where the sum runs over $\mmm\in\ZZ^n_{\geq 0}$ for which $\aaa\cdot\mmm=\sum_{j=1}^na_jm_j=d.$ For $\vMF$, we denote by $|P_i|_v=\max_{\mmm}|A_{i,\mmm}|_v$. For almost all $v,$ one has that $|P_i|_v=1.$ For $k\in\ZZ$, we set \begin{align*}
 \quad
\varepsilon_v(k):=&
             \begin{cases}
k&\text{if } \vMFi,\\
1&\text{if }\vMFz.
              \end{cases}
               \end{align*} 
We note that for any $\vMF$, any $k\geq 1$ and any $z_1\doots z_k\in F$, one has that $$|z_1+\cdots +z_k|_v\leq \varepsilon_v(r)\max(|z_1|_v\doots |z_k|_v).$$
For $\xxx\in[\PPP(\aaa)(F)]$, let $\wx\in\Fvnz$ be a lift of~$\xxx$. For $i\in\{1\doots r+1\}$ and $v\in M_F$, we deduce that \begin{align*}
|P_i(\wx)|_v&=\big|\sum_{\aaa\cdot\mmm}A_{i,\mmm}\widetilde x^{m_1}_1\cdots\widetilde x^{m_n}_n\big|_v\\
&\leq \varepsilon(w(\aaa,d))(\max_{\mmm}|A_{i,\mmm}|_v)(\max_{\mmm}|\widetilde x_1^{m_1}\cdots \widetilde x_n^{m_n}|_v)\\&\leq \varepsilon(w(\aaa,d))|P_i|_v\max_{\mmm}\big(\prod _{\ell=1}^n\max_{j}(|\widetilde x_j|^{1/a_j}_v)^{a_{\ell}m_{\ell}}\big)\\
&=\varepsilon(w(\aaa,d))|P_i|_v \max _j(|\widetilde x_j|_v^{1/a_j})^{\aaa\cdot\mmm}\\
&=\varepsilon(w(\aaa,d))|P_i|_v \max _j(|\widetilde x_j|_v^{1/a_j})^{d}.
\end{align*}
Set $C_v=\varepsilon(w(\aaa,d))\max_i |P_i|_v$. For almost all~$v$, one has that $C_v=1$ and we set $C=\prod_{v}C_v$. Let us define $$r_v:\Fvnz\to\ZZ\hspace{1cm}\yyy\mapsto\sup_{\substack{j=1\doots n\\x_j\neq 0}}\bigg\lceil-\frac{v(y_j)}{a_j}\bigg\rceil$$ and $$r^{\max}_v:\Fvnz\to\QQ\hspace{1cm}\yyy\mapsto\sup_{\substack{j=1\doots n\\y_j\neq 0}}\big(\frac{v(y_j)}{a_j}\big).$$ One has that $$f^{\#}_v(\yyy)=\pivv^{-dr_v(\yyy)}\geq\pivv^{-dr^{\max}_v(\yyy)}= \max_j(|y_j|_v^{1/a_j}|_v)^d.$$ We deduce that $$\max_i(|P_i(\wx)|_v)\leq C_v\max_j(|y_j|_v^{1/a_j})^d\leq C_vf^{\#}_v(\wx).$$ By multiplying this inequality for all~$v$ we obtain that for every $\xxx\in[(\PPP(\aaa)-Z(P_1\doots P_n))(F)]$ one has that\begin{multline*}H_{\PP^r}(\overline J(P_1\doots P_{r+1})(\xxx))^d=\prod_{\vMF}(\max_i(|P_i(\wx)|_v))^{d}\leq \prod_vC_vf_v^{\#}(\wx)\\= CH^{\#}(\xxx).\end{multline*}The claim is proven.
\end{proof}
\subsection{} \label{localheightdef} We establish several facts on ``local heights" that will be needed in \ref{finitenesstoric}. 

Let $(f_v:\Fvnz\to\RR_{>0})_v$ be a generalized adelic family of $\aaa$-homogenous continuous function of weighted degree $|\aaa|$. Let $H=H((f_v)_v)$ be the corresponding height on $[\PPP(\aaa)(F)]$. 
For $v\in M_F,$  $t\in\Fvt$ and $\xxx\in(\Fvt)^n$ one has that $$f_v(t\cdot\xxx)\prodjn |t^{a_j}x_j|_v^{-1}=|t|_v^{|\aaa|}f_v(\xxx)\prodjn|t|_v^{-a_j} |x_j|_v^{-1}=f_v(\xxx)\prodjn|x_j|_v^{-1}$$ i.e. for $v\in M_F$,  the continuous function \begin{equation}\label{ovaveo}(\Fvt)^n\to \RR_{>0},\hspace{1cm} \xxx\mapsto f_v(\xxx)\prodjn |x_j|_v^{-1}\end{equation} is $(\Fvt)_{\aaa}$-invariant. Let $H_v:[\TTa(F_v)]\to\RR_{>0}$ be the function induced from $(\Fv)_{\aaa}$-invariant function (\ref{ovaveo}). 
For $\xxx\in[\TTa(F)],$ we write $H_v(\xxx)$ for what is technically $H_v([\TTa(i_v)](\xxx)),$ where $[\TTa(i_v)]:[\TTa(i)(F)]\to[\TTa(i)(F_v)]$ is the induced homomorphism from $(F^{\times})_{\aaa}$-invariant homomorphism $$(F^\times)^n\hookrightarrow (\Fvt)^n\to[\TTa(F_v)].$$ 
\begin{lem}\label{localheightglobal}
Let $\xxx\in[\TTa(F)]$. One has that $H(\xxx)=\prod_vH_v(\xxx)$.
\end{lem}
\begin{proof}
Let $\wx\in (F^{\times})^n$ be a lift of~$\xxx$. By using the product formula, one gets that \begin{align*}
\prod_{\vMF}H_v(\xxx)&=\prod_{\vMF}\big(f_v(\wx)\prodjn|\widetilde x_j|_v^{-1}\big)\\
&=\big(\prod_{\vMF}f_v(\wx)\big)\prod_{\vMF}\prodjn|\widetilde x_j|_v^{-1}\\
&=H(\xxx)
\end{align*}
\end{proof}
\begin{lem}
Let $v\in M_F$ and suppose that $f_v=f_v^{\#}$ is the toric $\aaa$-homogenous function of weighted degree~$1$. Let $\xxx\in[\TTa(F)].$ One has that $H^{\#}_v(\xxx)\geq 1$.
\end{lem}
\begin{proof}Suppose $v\in M_F^0.$  It follows from Lemma \ref{davdavdav} that there exists a lift $\widetilde\xxx$ of $[\TTa(i_v)](\xxx)$ lying in $\Dav$. By using that $f^{\#}_v|_{\Dav}=1$ and that $\Dav\subset(\Ov)^n$, we obtain \begin{equation*}H^{\#}_v(\xxx)=f_v^{\#}(\wx)\prodjn |\widetilde x_{j}|_v^{-1_j}=\prodjn |\widetilde x_{j}|_v^{-1_j}\geq 1.\end{equation*}
Suppose now $\vMFi$. Let $\widetilde\xxx\in(\Fvt)^n$ be a lift of~$\xxx$. One has that \begin{align*}H^{\#}_v(\xxx)&=f^{\#}_v(\wx)\prodjn|\widetilde x_{j}|_v^{-1}\\&=\big(\max _{k}(|\widetilde x_{k}|_v^{1/a_k})\big)^{|\aaa|}\prod _{j=1}^n|\widetilde x_{j}|_v^{-1}\\ &=\prod _{j=1}^n\big(\max _{k=1\doots n}(|\widetilde x_{k}|_v^{1/a_k})\big)^{a_j}\big(|\widetilde x_{j}|_v^{1/a_j}\big)^{-a_j} \\ &\geq \prod _{j=1}^n\big(\max_{k}(|\widetilde x_{k}|_v^{1/a_k})\big)^{a_j}\big(\max_{k}(|\widetilde x_{k}|_v^{1/a_k})\big)^{-a_j}\\&=1.\end{align*}
\end{proof}
\section{Metrics induced by models} \label{metricsinducedbymodels} We use models with enough integral points to define metrics. We establish that the toric metric comes from models of weighted projective stacks from \ref{Modelswithenough}. 
%
\subsection{}We use $\Ov$-points of~$\oPPa$ to define $\Fv$-metrics.

Let $\vMFz$. By an $\Ov$-extension of $\xxx\in\oPPa(\Fv)$ we mean a pair $(\yyy,t)$ where $\yyy\in\oPPa(\Ov)$ and $t:\yyy_{\Fv}\xrightarrow{\sim}\xxx$ is a $2$-isomorphism.  Let $S_{\xxx}$ be the set of $\Ov$-extensions of $\xxx\in\oPPa(\Fv)$. Proposition \ref{whyenough} gives that the set $S_{\xxx}$ is non-empty for any $\xxx\in\oPPa(\Fv).$ 

For $\xxx\in\oPPa(\Fv)$ (respectively, $\xxx\in\oPPa(\Ov)$), we will denote by~$\wx$ the canonical $(\Gm)_{\Fv}$-equivariant morphism $\wx:(\Gm)_{\Fv}\to\AAA^n_{\Fv}$ (respectively, the canonical $(\Gm)_{\Ov}$-equivariant morphism $\wx:(\Gm)_{\Ov}\to\AAA^n_{\Ov}$) induced by $\xxx.$ One has that $q^{\aaa}\circ\wx (1)=\xxx$.

If~$L$ is a line bundle on~$\oPPa$, a $2$-isomorphism $t:\xxx\xrightarrow{\sim}\xxx'$, where $\xxx,\xxx'\in\oPPa(\Fv),$ induces a linear map $L(t):L(\xxx)\xrightarrow{\sim} L(\xxx')$. 
\begin{mydef}\label{defmetmod}
Let~$L$ be a line bundle on~$\oPPa$. Let $\xxx\in\PPP(\aaa)(\Fv)$ and let $\ell\in L(\xxx)$. We define $$||\ell||_{\xxx}:=\sup_{(\yyy,t)\in S_\xxx}\{\inf\{|a|_v\hspace{0.1cm} |\hspace{0.1cm}a\in\Fv^{\times}: \ell\in a (L(t)(\yyy^*L))\}\}.$$
\end{mydef}
We calculate the metrics.
\begin{lem}\label{calmetmod}
Let $\xxx\in\PPP(\aaa)(\Fv)$. Let~$k$ be an integer and let $f_v^{\#}:\Fvnz\to\RR_{>0}$ be the $\aaa$-homogenous toric function of weighted degree~$k$. Let $\ell\in \mathcal O(k)(\xxx)$. One has that $$||\ell||_{\xxx}:=f^{\#}_v(\wx(1))^{-1}|\ell|_v.$$
\end{lem}
\begin{proof}
If $r:\xxx'\to\xxx''$ is a $2$-isomorphism in $\oPPa(\Fv)$, the induced linear map $\mathcal O(k)(r):\Fv=\mathcal O(k)(\xxx')\xrightarrow{\sim}\mathcal O(k)(\xxx'')=\Fv$ is the linear map $x\mapsto t^kx$. It follows that:\begin{align*}||\ell||_{\xxx}&=\sup_{(\yyy,t)\in S_{\xxx}}\{ \inf\{|a|_v\hspace{0.1cm}|\hspace{0.1cm}a\in\Fv^{\times}:\ell\in a t^k(\yyy^*\mathcal O(k))\}\}\\
&=\sup_{(\yyy,t)\in S_{\xxx}}\{\inf\{|a|_v\hspace{0.1cm}|\hspace{0.1cm}a\in\Fv^{\times}:\ell\in at^k((\wy(1))^*(q^{\aaa}_{\Ov})^*\mathcal O(k))\}\}\\
&=\sup_{(\yyy,t)\in S_{\xxx}}\{\inf\{|a|_v\hspace{0.1cm}|\hspace{0.1cm}a\in\Fv^{\times}:\ell\in at^k((\wy(1))^*\OO_{\AAA^n})\}\}\\
&=\sup_{(\yyy,t)\in S_{\xxx}}\{\inf\{|a|_v\hspace{0.1cm}|\hspace{0.1cm}a\in\Fv^{\times}:\ell\in at^k\Ov\}\}\\
&=\sup_{(\yyy,t)\in S_{\xxx}}\{\inf\{|a|_v\hspace{0.1cm}|\hspace{0.1cm}a\in\Fv^{\times}:1\in a\ell^{-1}t^k\Ov\}\}\\
&=\sup_{(\yyy,t)\in S_{\xxx}}\{|t|^{-k}_v|\ell|_v\}\\
&=|\ell|_v\sup_{(\yyy,t)\in S_{\xxx}}\{|t|^{-k}_v\}.
\end{align*}
Note that if $(\yyy,t)\in S_{\xxx}$, then as $t^{-1}\cdot\wx(1)=\wy(1)\in\Ov^n,$ it follows from \ref{davdavdav} that $v(t^{-1})\geq r_v(\wx(1))$, with the equality if and only if $\wy(1)\in\Dav$. We deduce that $$\sup_{(\yyy,t)\in S_{\xxx}}\{|t|^{-k}_v\}=\pivv^{-kr_v(\wx(1))}=f^{\#}_v(\wx(1))^{-1}.$$ The claim follows.
\end{proof}
We can deduce that:
\begin{cor}Let~$v$ be a finite place of~$F$ and let~$k$ be an integer. The metrics $|\lvert\cdot\rvert|_{\xxx}$ from \ref{defmetmod} on $\mathcal O(k)(\xxx)$ for $\xxx\in\PPP(\aaa)(\Fv)$ define an $\Fv$-metric $|\lvert\cdot\rvert|$ on $\mathcal O(k)|_{\PPP(\aaa)_{\Fv}}.$ The $\Fv$-metric $|\lvert\cdot\rvert|$ is the induced $\Fv$-metric from the function $x\mapsto f^{\#}_v({\xxx})$ by \ref{metribun}.
\end{cor}
\begin{proof}
For $\xxx\in \PPP(\aaa)(\Fv)$, let us pick $\ell=1\in \mathcal O(k)(\xxx)$. By \ref{calmetmod}, we obtain that $||1||_{\xxx}=f^{\#}_v(\wx)^{-1}$ for every $\xxx\in\PPP(\aaa)(\Fv)$. As $f_v^{\#}:\Fvnz\to\RR_{>0}$ is continuous (\ref{davdavdav}) and satisfies that $f^{\#}_v(t\cdot\xxx)=|t|^k_vf^{\#}_v(\xxx)$, we deduce that $|\lvert\cdot\rvert|$ is the induced $\Fv$-metric from the function $f_v^{\#}$ using Lemma \ref{cormeq}.
\end{proof}
\section{Finiteness property of quasi-toric heights}\label{finitenesstoric}
We say that a height is a Northcott height if for every $B>0$ there are only finitely many points in $[\PPP(\aaa)(F)]$ having the height less than $B$. We establish that quasi-toric heights are Northcott heights. We improve it further for heights that are degenerate if the singularities of $f_v$ for~$v$ infinite are logarithmic along a rational divisor. 
Let $\aaa\in\ZZ_{\geq 1}^n.$
\subsection{} We define Northcott heights. 
\begin{mydef}
Suppose $(f_v:\Fvnz\to\RR_{\geq 0})_v$ is a generalized adelic family of continuous $\aaa$-homogenous functions. We say that the corresponding height $H((f_v)_v)$ is a Northcott height if for every $B>0$, the set $$\{\xxx\in[\PPP(\aaa)(F)]| H((f_v)_v)(\xxx)<B\}$$ is finite. 
\end{mydef}
\subsection{}We give heights which are not Northcott heights.
\begin{lem}\label{infpaf}
Suppose $n=1$ and $a\in\ZZ_{\geq 2}$. The set $[\PPP(a)(F)]$ is infinite.
\end{lem}
\begin{proof}
One has that $[\PPP(a)(F)]=[\TTa(F)]=\Ft/(\Ft)_a,$ where $(\Ft)_a$ is the subgroup given by the non-zero $a$-th powers in $\Ft$. There are infinitely many non-zero principal prime ideals in $\OO_F$ (because only finitely many prime ideals in $\ZZ$ which ramify in $\OO_F$). For any of those principal prime ideals $\mathfrak p$, let $b_{\mathfrak p}\in F^{\times}$ be a generator. Note that for any principal prime ideals $\mathfrak p_1,\mathfrak p_2$ one has that $b_{\mathfrak p_1}b_{\mathfrak p_2}^{-1}\not\in (\Ft)_a,$ because $a\nmid v_{\mathfrak p_1}(b_{\mathfrak p_1}b_{\mathfrak p_2}^{-1})=1,$ where $v_{\mathfrak p_1}$ is the valuation corresponding to $\mathfrak p_1$. It follows that the images in $\Ft/(\Ft)_a$ of different $b_{\mathfrak p}$ are different. It follows that $[\PPP(a)(F)]=\Ft/(\Ft)_a$ is infinite.
\end{proof}
\begin{cor}\label{examnotnorth} Suppose that $n=1$ and $a\in\ZZ_{\geq 2}$. Let $(f_v)_v$ be a stable generalized adelic family of $a$-homogenous functions $\Fv-\{0\}\to~\RR_{>0}$. For every $x,y\in\PPP(a)(F)$ one has that $H((f_v)_v)(\xxx)=H((f_v)_v)(\yyy)$. The height $H=H((f_v)_v)$ is not a Northcott height. 
\end{cor}
\begin{proof} 
Let $x,y\in\PPP(a)(F)$ and let $\widetilde x,\widetilde{y}:(\Gm)_{\Fv}\to(\AAA^1-\{0\})_{\Fv}$ be the two $(\Gm)_{\Fv}$-equivariant morphisms defined by~$x$ and~$y$. Let $K=F(\sqrt[a]{\widetilde {x}(1)\widetilde{y}(1)^{-1}}).$ Note that for $t:=\sqrt[a]{\widetilde {x}(1)\widetilde{y}(1)^{-1}}$ one has $t^ax=y$, thus $x_K\cong y_K$. By \ref{propertyofstable}, one has that for every $x,y\in\PPP(a)(F)$ one has $H(x)=H(y)$. Thus for every $z,w\in[\PPP(\aaa)(F)]$ one has $H(z)=H(w)$. Let $B>H(w)$, where $w\in[\PPP(a)(F)]$. By \ref{infpaf}, the set $\{z|z\in [\PPP(a)(F)]\text{ and } H(z)<B\}$ is infinite. Thus~$H$ is not a Northcott height.
\end{proof}
\begin{rem}
\normalfont
In fact, stable heights are not Northcott heights unless all $a_j$ are equal to~$1$. Indeed, every rational point of $\PPP(a_j)$ can be seen as a rational point of~$\PPP(\aaa)$ via the closed embedding $\PPP(a_j)\to\PPP(\aaa)$ (induced from~$\Gm$-invariant morphism $\AAA^1-\{0\}\to\AAA^n-\{0\}\xrightarrow{q^{\aaa}}\PPP(\aaa)$). Thus if for some~$j$ one has $a_j>1$, then $\PPP(a_j),$ and hence $\PPP(\aaa),$ both have infinitely many rational points. We omit details.
\end{rem}
\subsection{}We dedicate the next paragraphs to the proof that the toric heights are Northcott heights. By the property of the boundedness \ref{toricisclos} of the quotients, it is then immediate that any quasi-toric height is Northcott height. 

Let $\Div(F)$ be the group of fractional ideals of~$F$. We define $$\Div(F)_{\aaa}:=\{(x^{a_j})_j|\hspace{0.1cm} x\in \Div (F)\}.$$ In this paragraph we will give an estimate to the number of elements of the abelian group $\Div(F)^n/\Div(F)_{\aaa}$ of bounded ``height". This will be useful, as in the next paragraph, we relate the finite part of the height on $\TTa(F)$ with the ``height" on $\Div(F)^n/\Div(F)_{\aaa}.$ If $v\in M_F^0,$ for a fractional ideal~$x$ of $\OO_F$ we define $v(x)$ by setting it to be the exponent of the prime ideal corresponding to $v,$ in the prime factorization of~$x$. 

We define set $$\Div(F)_{\aaa-\prim}:=\{\xxx\in\Div(F)^n|\xxx\subset\OOF^n\text { and }\forall v\in M_F^0, \exists j:v(x_j)<a_j\}.$$
\begin{lem}\label{rvideal}\begin{enumerate}
\item Let us define $$r_v:\Div(F)^n\to\ZZ\hspace{1cm} \xxx\mapsto \sup _{\substack{j=1\doots n\\x_j\neq 0}}\big\lceil\frac{-v( x_j)}{a_j}\big\rceil.$$ Let $\xxx\in\Div(F)^n.$ The fractional ideal $$y=y(\xxx):=\prod_{\vMFz}(\piv^{r_v(\xxx)}\Ov\cap F)$$ satisfies that $$(x_jy(\xxx)^{a_j})_j\in\Div(F)_{\aaa-\prim}.$$
\item The restriction of the quotient map $$q_{\Div(F),\aaa}:\Div(F)^n\to\Div(F)^n/\Div(F)_{\aaa}$$ to $\Div(F)_{\aaa-\prim}$ is a bijection $\Div(F)_{\aaa-\prim}\xrightarrow{\sim} \Div(F)^n/\Div(F)_{\aaa}.$
\end{enumerate}
\end{lem}
\begin{proof}
\begin{enumerate}
\item Let $v\in M_F^0$. For every~$j$ one has that $$v(x_jy^{a_j})=v(x_j)+{a_jr_v(\xxx)}\geq 0. $$ Let~$i$ be the index such that $\frac{-v(x_i)}{a_i}$ is maximal, then $r_v(\xxx)<1+\frac{-v(x_i)}{a_i}.$  We deduce that $$v(x_jy^{a_j})=v(x_j)+a_jr_v(\xxx)<a_j.$$ Hence, $( x_j y^{a_j})_j$ is an element of $\Div(F)_{\aaa-\prim}.$
\item We prove the surjectivity of $q_{\Div(F),\aaa}$. Let $\zzz\in\Div(F)^n/\Div(F)_{\aaa}$ and let $\widetilde{\zzz}\in\Div(F)^n$ be a lift of~$\zzz$. Then $(\widetilde z_j y(\zzz)^{a_j})_j$ is a lift of~$\zzz$ belonging to $\Div(F)_{\aaa-\prim}$. Let us prove the injectivity. 
Suppose that the elements $\xxx,\mathbf r\in\Div(F)_{\aaa-\prim}$ are lying above the same element in $\Div(F)^n/\Div(F)_{\aaa}$. Then there exists $t\in\Div(F)$ such that $x_j=t^{a_j}r_j$ for $j=1\doots n$. We need to establish that $t=(1)$. Suppose on the contrary $t\neq(1).$ One can choose~$v$ such that $v(t)\neq 0$. If $v(t)>0$, then for every~$j$ one has $$v(x_j)=a_jv(t)+v(r_j)\geq a_j\cdot 1+0=a_j,$$ which is a contradiction with the fact that~$\xxx$ is primitive. If $v(t)<0$, then there exists~$j$ such that $a_j>v(r_j)$, which gives that $$v(x_j)=a_jv(t)+v( r_j)<a_j-a_j=0,$$ which is a contradiction with the fact that~$\xxx$ is primitive. We deduce that $t=(1)$, and hence $\xxx=\rrr.$ Therefore $q_{\Div(F),\aaa}|_{\Div(F)_{\aaa-\prim}}$ is injective. The claim is proven. 
\end{enumerate}
\end{proof}
For every $\xxx\in\Div(F)^n/\Div(F)_{\aaa}$, let us denote by $\widetilde {\xxx}$ the unique lift of~$\xxx$ lying in $\Div(F)_{\aaa-\prim}.$ For $\xxx\in\Div (F)^n/\Div(F)_{\aaa}$, we set $${ {H}_{\Ideal}}(\xxx):=\prod _{j=1}^nN(\widetilde x_j) $$ where~$N$ is the ideal norm. 
\begin{lem}\label{hzerui} For $m\in\ZZ_{\geq 1}$, there exists $C_m>0$ such that for every $B>0$ one has that$$\{\yyy\in\Div(F)^m_{\geq 0}|\prodjn N(y_j) \leq B\}\leq C_mB\log(1+B)^{m-1},$$where $\Div_{\geq 0}(F)=\Div(F)\cap\OOF$.
\end{lem}
\begin{proof} 
We will use the following result: there exists $C_1>0$ such that $$|\{y\in\Div (F)|y\subset\OOF, N(y)\leq B\}|\leq C_1 B.$$ This is proven in \cite[Pages 145-150]{Janusz}. We use the induction on~$m$. Let $m\geq 2$. For $k\geq 1$, we set $$g(k)=\big|\big\{\xxx\in\Div_{\geq 0}(F)^{m-1}|\prod _{j=1}^{m-1} N(x_j)=k\big\}\big|.$$ 
For $B\geq 1$, using Abel's summation formula, we get that there exists $C_m', C_m>0$ such that have that \begin{align*}
|\{\yyy\in\Div_{\geq 0}(F)^m|\prod _{j=1}^mN(y_j)\leq B\} |\hskip-6,3cm&\\&= \sum_{\substack{\yyy\in\Div_{\geq 0}(F)^m\\N(y_m)\leq B/\prod_{j=1}^{m-1}N(y_j) }}1\\
&\leq \sum_{\substack{(y_j)_{j=1}^{m-1}\in\Div_{\geq 0}(F)^{m-1}\\\prod_{j=1}^{m-1} N(y_j)\leq B}}\frac{C_1B}{\prod_{j=1}^{m-1}N(y_i)}\\
&=C_1B\sum_{k=1}^{\lfloor B\rfloor}\frac{g(k)}k\\
&=C_1B\bigg(\frac{1}{\lfloor B\rfloor}\sum_{j=1}^{\lfloor B\rfloor }g(k)+\sum_{j=1}^{\lfloor B\rfloor-1}\big(\frac{1}{j}-\frac{1}{j+1}\big)\sum_{k=1}^{j}g(k)\bigg)\\
&\leq 2C_1C_{m-1}\log(\lfloor B\rfloor+1)^{m-2}+C_1B\sum_{j=1}^{\lfloor B\rfloor-1}\big(\frac{1}{j}-\frac{1}{j+1}\big)C_{m-1}j\log(j+1)^{m-2}\\
&\leq 2C_1C_{m-1}\log(B+1)^{m-2}+C_1C_{m-1}B\sum_{j=1}^{\lfloor B\rfloor-1}\frac{\log(j+1)^{m-2}}{j}\\
&\leq C_m'\log(B+1)^{m-2}+C_m'B\log(\lfloor B\rfloor)^{m-2}\sum_{j=1}^{\lfloor B\rfloor-1}\frac{1}{j+1} \\
&\leq C_m'\log(B+1)^{m-2}+C_m'B\log(\lfloor B\rfloor)\log(\lfloor B\rfloor +1)^{m-1}\\
&\leq C_{m}B\log(B+1)^{m-1}.
\end{align*}
The statement follows.
\end{proof}
We can estimate the number of elements of $\Div(F)^n/\aaa\Div(F)$ having $H_{\Ideal}$ less than $B$.
\begin{cor}\label{hzer}
There exists $C>0$ such that for any $B>0$, one has $$|\{\xxx\in\Div(F)^n/\aaa\Div(F)|\hspace{0.1cm}{H_{\Ideal}}(\xxx)\leq B \}|\leq C B\log(1+B)^{n-1}.$$
\end{cor}
\begin{proof}
By \ref{hzerui}, there exists $C_n>0$ such that: \begin{align*}
|\{\xxx\in\Div(F)^n/\Div(F)_{\aaa}|\hspace{0.1cm}{H}_{\Ideal}(\xxx)<B\}|\hskip-3cm&\\&=|\{\yyy\in\Div_{\aaa-\prim}(F)| \prod_{j=1}^nN(y_j)<B \}|\\&\leq |\{\yyy\in\Div(F)^n|{\yyy} \subset\OOF^n,\prod _{j=1}^nN(y_j)<B\} |\\&\leq C_nB\log(1+B)^{n-1}.
\end{align*}
The statement is proven.
\end{proof}
\subsection{}\label{defofibar} For $\vMF$, we let $f_v^{\#}:\Fvnz\to\RR_{>0}$ be the toric $\aaa$-homogenous function of weighted degree $|\aaa|$. Let $H^{\#}$ be the corresponding toric height on $[\PPP(\aaa)(F)]$. For $v\in M_F$, we let $H^{\#}_v$ to be the ``local heights" from \ref{localheightdef} i.e. the functions $[\TTa(\Fv)]\to\RR_{>0}$ induced from $\Fvt$-invariant maps $$(\Fvt)^n\to \RR_{>0},\hspace{1cm} \xxx\mapsto f^{\#}_v(\xxx)\prodjn |x_j|_v^{-1}.$$In this paragraph we relate $\prod_{\vMFz}H^{\#}_v$ and $H_{\Ideal}$. 
For $x\in \Ft$, let us denote by $\mathfrak I(x)\in\Div(F)$ the fractional ideal $x\OOF$. We denote by $\mathfrak I^n$ the product homomorphism $(\Ft)^n\to \Div(F)^n.$ Note that if $\xxx\in (\Ft)_{\aaa}=\{(x^{a_j})_j|x\in\Ft\}$, then $\mathfrak I^n(\xxx)\in\Div(F)_{\aaa}.$ We denote by $\omI$ the induced homomorphism \begin{equation}\omI:[\TTa (F)]=(\Ft)^n/(\Ft)_{\aaa}\to \Div (F)^n/\Div(F)_{\aaa}\end{equation} from $(\Ft)_{\aaa}$-invariant map $$F^{\times n}\xrightarrow{\xxx\mapsto \mathfrak I^n(\xxx)}\Div (F)^n\to \Div (F)^n/\Div (F)_{\aaa},$$where the second homomorphism is the quotient homomorphism. 
\begin{lem}\label{brkica} 
Let $\xxx\in[\TTa(F)]$. One has that $$H_{\Ideal}(\overline{\mathfrak I}(\xxx))=\prod _{\vMFz}H^{\#}_v(\xxx). $$ 
\end{lem}
\begin{proof}
Let $\wx\in(\Ft)^n$ be a lift of~$\xxx$. It follows from \ref{rvideal} that $$\mathfrak I^n(\wx) (y(\mathfrak I^n(\wx))^{a_j})_j=(\mathfrak I(\widetilde x_j)\prod_{\vMFz}(\piv^{a_jr_v(\mathfrak I^n(\wx))}\Ov\cap\OOF))_j$$ is the unique lift of $\overline{\mathfrak I}(\xxx)$ lying in $\Div(F)_{\aaa-\prim}$. 
We can calculate (if $0\neq I\in\Div(F)$, we write $v(I)$ for the exponent of the prime ideal corresponding to~$v$ in the prime factorization of~$I$):\begin{align*}H_{\Ideal}(\overline {\mathfrak I}(\xxx))&=\prod_{j=1}^nN\bigg(\mathfrak I(\widetilde x_j)\prod_{\vMFz}(\piv^{a_jr_v(\mathfrak I^n(\wx))}\Ov\cap\OOF)\bigg)\\&=\prod_{j=1}^nN\bigg(\prod_{\vMFz}\piv^{v(\mathfrak I(\widetilde x_j))+a_jr_v(\mathfrak I^n(\wx))}\Ov\cap\OOF\bigg)\\&=\prodjn\prod _{\vMFz}\pivv^{-( v(\mathfrak I(\widetilde x_j))+a_jr_v(\mathfrak I^n(\wx)))}\\
&=\prod_{\vMFz}\prodjn\pivv^{-(v(\widetilde x_j)+a_jr_v(\wx))}\\&=\prod_{\vMFz}\pivv^{-|\aaa|r_v(\wx)}\prodjn\pivv^{-v(\widetilde x_j)}\\&=\prod_{\vMFz}f_v(\wx)\prodjn|\widetilde x_j|_v^{-1}\\&=\prod_{\vMFz}H^{\#}_v(\xxx).
\end{align*}
\end{proof}
\subsection{}\label{uuakeri} In this paragraph we study the kernel of the homomorphism~$\omI$ defined in \ref{defofibar}. Let $U$ be the group of units of~$F$ and let $U_\aaa:=\{(u^{a_j})_j|u\in U\}.$ As $U_\aaa=(\Ft)_\aaa\cap U^n$, we have a canonical identification of $U^n/U_\aaa$ with a subgroup of $[\TTa(F)]$. Note that $U\subset\ker(\mathfrak I)$ and thus $U^n\subset\ker(\mathfrak I^n).$ It follows that $$U^n/U_{\aaa}\subset\ker(\omI).$$

We prove the following fact:
 \begin{lem}\label{kerpf} Let $d=\gcd(\aaa)$. One has that $$(\ker( {\omI}):(U^n/U_\aaa))=|{\Cl(F)}[{d}]|,$$
where ${\Cl(F)}[{d}]$ is the~$d$-torsion subgroup of the class group $\Cl(F)$. 
\end{lem}
\begin{proof}
The following ``snake diagram" is commutative
\[\begin{tikzcd}
	& U & {U^n} & {\ker(\overline{\mathfrak I})} \\
	& {F^{\times}} & {(F^{\times})^n} & {[\TTa(F)]} & 1 \\
	0 & {\Div(F)} & {\Div(F)^n} & {\Div(F)^n/\Div(F)_{\aaa}} & {} \\
	& {\Cl(F)} & {\Cl(F)^n} & {\coker(\overline{\mathfrak I}),}
	\arrow["{t\mapsto (t^{a_j})_j}", from=2-2, to=2-3]
	\arrow[from=2-3, to=2-4]
	\arrow[from=3-2, to=3-3]
	\arrow[from=3-1, to=3-2]
	\arrow[from=3-3, to=3-4]
	\arrow["{\mathfrak I}", from=2-2, to=3-2]
	\arrow["{\mathfrak I^n}", from=2-3, to=3-3]
	\arrow["{\overline{\mathfrak I}}"', from=2-4, to=3-4]
	\arrow[from=1-4, to=2-4]
	\arrow[from=1-3, to=2-3]
	\arrow[from=1-3, to=1-4]
	\arrow["{t\mapsto (t^{a_j})_j}", from=1-2, to=1-3]
	\arrow[from=1-2, to=2-2]
	\arrow[from=3-2, to=4-2]
	\arrow["{t\mapsto (t^{a_j})_j}", from=4-2, to=4-3]
	\arrow[from=4-3, to=4-4]
	\arrow[from=3-4, to=4-4]
	\arrow[from=3-3, to=4-3]
	\arrow[from=2-4, to=2-5]
\end{tikzcd}\]
where $\Cl(F)$ is the class group of~$F$. All vertical and horizontal sequences are exact. The snake lemma provides an exact sequence
$$U\xrightarrow{t\mapsto (t^{a_j})_j}U^n\rightarrow\ker(\overline{\mathfrak I})\xrightarrow{\sigma}\Cl(F)\xrightarrow{t\mapsto (t^{a_j})_j}\Cl(F)^n\to \coker(\mathfrak I).$$
It follows that $\ker(\sigma)=\Imm\big(U^n\to\ker(\overline{\mathfrak I})\big)=U^n/U_{\aaa}$. The kernel of $$\Cl(F)\to\Cl(F)^n\hspace{1cm}t\mapsto (t^{a_j})_j $$ is given by the subgroup $\Cl(F)[d]$. We deduce that $$|\Cl(F)[d]|=|\Imm(\sigma)|=|\ker(\overline{\mathfrak I})/\ker(\sigma)|=(\ker( {\omI}):(U^n/U_\aaa)).$$
\end{proof}
\subsection{} In this paragraph for fixed $\xxx\in[\TTa(F)]$, we bound the number of elements $\uuu\in U^n/U_\aaa$ for which $\prod_{\vMFi}H_v^{\#}(\uuu\xxx)<B$.

We define an auxiliary function $h_{\infty}: (\RR^{M_F^{\infty}})^{n}\to\RR$ by $$h_{\infty}:(\yyy_v)_v\mapsto \sum _{v\in M_F^{\infty}}\aaa\cdot \big(\max_{k=1\doots n}\bigg(\frac{y_{vk}}{a_k}-\frac{y_{vj}}{a_j}\bigg)\big)_j.$$
We define a homomorphism $$\rho: F^{\times}\to\RR^{M_F^{\infty}}\hspace{1cm}x\mapsto \log(|x|_v), $$ and a homomorphism$$\rho^n:(F^{\times})^n\to(\RR^{M_F^\infty})^n\hspace{1cm}\xxx\mapsto (\rho(x_j))_j.$$
\begin{lem}\label{odhdoh}
Let $\xxx\in[\TTa(F)]$ and let $\wx\in(\Ft)^n$ be its lift. 
For every $\sss\in\RR^n_{>0}$, we have that $$\log\bigg(\prod _{v\in M_F^{\infty}}H^{\#}_v(\xxx)\bigg)=h_{\infty}(\rho^n(\wx)).$$
\end{lem}
\begin{proof}
For $v\in M_F^{\infty}$, we have that \begin{align*}\log(H^{\#}_v(\xxx))&=\log\big(\max_k(|\widetilde x_k|^{1/a_k}_v)^{|\aaa|}\prod_{j=1}^n|\widetilde x_j|_v^{-1}\big)\\ 
&=\log\big(\prod_{j=1}^n\max_{k}(|\widetilde x_k|^{1/a_k}_v)^{a_j}|\widetilde x_j|_v^{-1}\big) \\&=\log\big(\prod_{j=1}^n\max_k(|\widetilde x_k|^{a_j/a_k}_v)|\widetilde x_j|^{-1}_v\big)\\&=\log\big(\prod_{j=1}^n\max_k(|\widetilde x_k|^{a_j/a_k}_v|\widetilde x_j|_v^{-1})\big)\\&=\sum_{j=1}^n\log\big(\max_k\bigg(\frac{|\widetilde x_k|^{1/a_k}_v}{|\widetilde x_j|^{1/a_j}_v}\bigg)^{a_j}\big)\\&=\sum_{j=1}^na_j\max_{k}\bigg(\frac{\log(|\widetilde x_k|_v)}{a_k}-\frac{\log(|\widetilde x_j|_v)}{a_j}\bigg)\\
&=\aaa\cdot \big(\max_{k}\bigg(\frac{\log(|\widetilde x_k|_v)}{a_k}-\frac{\log(|\widetilde x_j|_v)}{a_j}\bigg)\big)_j.
\end{align*}
We deduce:
\begin{multline*}
\log\big(\prod _{v\in M_F^{\infty}}H^{\#}_v(\xxx)\big)=\sum_{v\in M_F^{\infty}}\aaa\cdot \big(\max_{k}\bigg(\frac{\log(|\widetilde x_k|_v)}{a_k}-\frac{\log(|\widetilde x_j|_v)}{a_j}\bigg)\big)_j\\=h_{\infty}(\rho(\wx)).
\end{multline*}
\end{proof}
\begin{lem}\label{oalfa} For $j=1\doots n$, let us define a homomorphism $\alpha_j:(\RR^{M_F^{\infty}})^n\to(\RR^{M_F^{\infty}})^{n}$ by $$\alpha_j:(\yyy_{i})_{i}\mapsto \big(\frac{\yyy_i}{a_i}-\frac{\yyy_j}{a_j}\big)_{i}.$$
Set $\alpha=\prod_{j=1}^n\alpha_j:(\RR^{M_F^{\infty}})^n\to (\RR^{M_F^{\infty}})^{n^2}.$ 
\begin{enumerate}
\item One has that $$\ker (\alpha)=\big\{(a_j\yyy)_{j}|\yyy\in \RR^{M_F^{\infty}}\big\}. $$ 
\item Let $K>0$ and let $(\yyy_j)_j\in(\RR^{M_F^{\infty}})^n$ be such that $h_{{\infty}}(\yyy)<K.$ One has that $$\alpha(\yyy)\in\bigg[\frac{-K}{a_{i}},\frac{ K}{a_j}\bigg]_{v,i,j}.$$
\item Suppose~$L$ is a lattice of $\RR^{M_F^{\infty}}$. Then $\alpha(L^n)$ is contained in a lattice of $(\RR^{M_F^{\infty}})^{n^2}.$
\end{enumerate}
\label{cuvur}
\end{lem}
\begin{proof}
\begin{enumerate}
\item For $k,\ell\in\{1\doots n\}$, one has that \begin{multline*}\ker(\alpha_{\ell})=\{\big(\dfrac{a_i\yyy_{\ell}}{a_{\ell}}\big)_i|\yyy_{\ell}\in\RR^{M_F^{\infty}}\}=\{(a_i\yyy)_i|\yyy\in\RR^{M_F^{\infty}}\}\\=\{\big(\dfrac{a_i\yyy_k}{a_k}\big)_j|\yyy_k\in\RR^{M_F^{\infty}}\}=\ker(\alpha_k).\end{multline*}
Therefore $$\ker(\alpha)=\bigcap_{j=1}^n\ker (\alpha_j)=\ker(\alpha_1)=\{(a_j\yyy)_j|\yyy\in\RR^{M_F^{\infty}}\}.$$ 
\item Let $\ell\in\{1\doots n\}.$  Note that from \begin{equation*}h_{\infty}(\yyy)=\sum_{v\in M_F^{\infty}}\aaa\cdot \big(\max_{k}\bigg(\frac{y_{vk}}{a_k}-\frac{y_{vj}}{a_j}\bigg)\big)_j< K, \end{equation*}it follows that for every $i\in\{1\doots n\}$ and every $w\in M_F^{\infty}$, one has that$$a_\ell \bigg(\frac{y_{wi}}{a_i}-\frac{y_{w\ell}}{a_\ell}\bigg)\leq a_\ell \max_k\bigg(\frac{y_{vk}}{a_k}-\frac{y_{vj}}{a_j}\bigg)\leq\sum_{v\in M_F^{\infty}}\aaa\cdot \big(\max_{k}\bigg(\frac{y_{vk}}{a_k}-\frac{y_{vj}}{a_j}\bigg)\big)_j < K, $$and hence that \begin{equation*}\frac{y_{wi}}{a_i}-\frac{y_{w\ell}}{a_\ell}< \frac{K}{a_\ell }.\end{equation*} Similarly, for every $i\in\{1\doots n\}$ and every $w\in M_F^{\infty}$, one has that\begin{equation*}\frac{y_{w\ell}}{a_\ell}-\frac{y_{wi}}{a_i}< \frac{K}{a_i}.\end{equation*}We deduce that $$\frac{\yyy_i}{a_i}-\frac{\yyy_j}{a_j}\in\bigg[-\frac{K}{a_i},\frac{K}{a_{\ell}}\bigg]_{\vMFi}$$and hence that$$\alpha_{\ell}((\yyy_i)_i)=\big(\frac{\yyy_{i}}{a_i}-\frac{\yyy_{\ell}}{a_\ell}\big)_i\in \bigg[-\frac{K}{a_i} ,\frac{K}{a_\ell }\bigg]_{v,i}.$$ Finally, it follows that $$\alpha((\yyy_i)_i)=(\alpha_j((\yyy_i)_i))_j\in\bigg[-\frac{K}{a_i} ,\frac{K}{a_j}\bigg]_{v,i,j}.$$
\item Let $(\yyy_j)_j\in L^n$. Then for every $i,j$ one has $$\frac{\yyy_{i}}{a_i}-\frac{\yyy_{j}}{a_j}\in\frac{L}{\prodjn a_j}.$$ We deduce $$\alpha_j(\yyy)=\big(\frac{\yyy_i}{a_i}-\frac{\yyy_j}{a_j}\big)_i\in\bigg(\frac{L}{\prodjn a_j} \bigg)_{i=1}^n.$$ It follows that $$\alpha(L^n)=(\alpha_j(L^n))_j\subset \bigg(\frac{L}{\prodjn a_j}\bigg)_{\ell=1}^{n^2}.$$ The claim is proven.
\end{enumerate}
\end{proof}
Let us set $$W:=\big\{(w_v)_v\in \RR^{M_F^{\infty}}|\sum_vw_v=0\big\} $$
For $u\in U,$ we have that $$\rho(u)=(\log |u|_v)_v\in W,$$ 
and hence for $\uuu\in U^n$, we have that $\rho^n(\uuu)=(\rho(u_j))_j\in W^{n}.$
\begin{lem}\label{minir} The following claims are valid:
\begin{enumerate} 
\item The homomorphism $\rho^n$ is of finite kernel and its image is a lattice of $W^n$.
\item One has that $W^n\cap\ker(\alpha)=\{(a_j\www)_{j}|\www\in W\}$.
\item One has that $\rk(\ker(\alpha\circ(\rho^n|_{U^n})))\leq r_1+r_2-1,$ where $r_1$ is the number of the real and $r_2$ is the number of complex places of~$F$.
\item One has that $U_\aaa\subset \ker(\alpha\circ(\rho^n|_{U^n}))$ and the map $\beta:U^n/U_\aaa\to (\RR^{M_F^{\infty}})^{n^2},$ induced from $U_{\aaa}$-invariant map $\alpha\circ(\rho^n|_{U^n})$, is of finite kernel.
\item The image $\beta(U^n/U_{\aaa})$ is contained in a lattice of $(\RR^{M_F^{\infty}})^{n^2}$.
\end{enumerate}
\end{lem}
\begin{proof}
\begin{enumerate}
\item By Dirichlet's unit theorem, one has that $\ker(\rho|_U)$ is finite and that $\rho(U)$ is a lattice of $W$. It follows that $\ker(\rho^n|_{U^n})=\ker(\rho|_U)^n$ is finite and that $\rho^n(U^n)=\rho(U)^n$ is a lattice of $W^n$.
\item By \ref{cuvur} one has that $\ker(\alpha)=\{(a_j\yyy)_j|\yyy\in\RR^{M_F^{\infty}}\}$ and thus $$W^n\cap\ker(\alpha)=W^n\cap\{(a_j\yyy)_j|\yyy\in\RR^{M_F^{\infty}}\}=\{(a_j\www)_{j}|\www\in W\}.$$
\item By (1), the kernel of $\rho^n|_{\ker (\alpha\circ\rho^n|_{U^n})}:\ker(\alpha\circ\rho^n|_{U^n})\to W^n$ is finite. It suffices therefore to show that the image $\rho^n({\ker (\alpha\circ\rho^n|_{U^n})})$ is of rank no more than $r_1+r_2-1$. The image $\rho^n({\ker (\alpha\circ\rho^n|_{U^n})})$ is contained in $\ker(\alpha)$ and by (1), it is also contained in the lattice $\rho^n(U^n)$ of $W^n,$ thus $$\rho^n({\ker (\alpha\circ\rho^n|_{U^n})})\subset \rho^n(U^n)\cap \ker(\alpha)=\rho^n(U^n)\cap W^n\cap\ker(\alpha).$$ The rank of the intersection of the lattice $\rho^n(U^n)$ of $W^n$ and the vector subspace $\ker(\alpha)\cap W^n$ of $W^n$ cannot be more than $\dim (W^n\cap\ker(\alpha))=\dim (\{(a_j\www)_{j}|\www\in W\})=r_1+r_2-1$. The claim follows.
\item Let $(u^{a_j})_j\in U_\aaa$, where $u\in U$. We have that $$\rho^n((u^{a_j})_j)=(a_j\rho(u))_j\in\ker(\alpha).$$ We deduce $U_\aaa\subset\ker(\alpha\circ(\rho^n|_{U^n}))$. Let us establish that $U_{\aaa}$ is of finite index in $\ker(\alpha\circ(\rho^n|_{U^n}))$. The map $$U\to U^{n}\hspace{1cm}u\mapsto (u^{a_j})_j$$ is of finite kernel and thus $$\rk (U_\aaa)=\rk(U)=r_1+r_2-1\geq\rk(\ker(\alpha\circ(\rho^n|_{U^n}))).$$ We conclude that $\rk(U_\aaa)=r_1+r_2-1=\rk(\ker(\alpha\circ(\rho^n|_{U^n}))),$ and hence is $U_{\aaa}$ of finite index in $\ker(\alpha\circ(\rho^n|_{U^n})).$ Now, it follows that the homomorphism $\beta$, which is precisely the composite homomorphism $$\beta:U^n/U_{\aaa}\to U^n/\ker(\alpha\circ(\rho^n|_{U^n}))\hookrightarrow(\RR^{M_F^{\infty}})^{n^2}),$$ is of finite kernel. The claim is proven.
\item  One has that $\beta(U^n/U_\aaa)=\alpha(\rho(U^n)).$ By \ref{cuvur} one has that $\alpha(\rho^n(U^n))$ is contained in a lattice of $(\RR^{M_F^{\infty}})^{n^2}$. The statement follows.
\end{enumerate}
\end{proof}
We now prove the principal result of the paragraph
\begin{lem}\label{boundunits}
There exists $C>0$ such that for every $\xxx\in[\TTa(F)]$ and every $B>2$, one has that $$\{\uuu\in U^n/U_\aaa|H^{\#}_\infty(\uuu\xxx)<B \}\leq C \big(\log(B)\big)^{n^2(r_1+r_2)}.$$
\end{lem}
\begin{proof}
For $\xxx\in[\TTa(F)]$, let $\wx\in(F^{\times})^n$ be a lift of~$\xxx$ and for $\uuu\in U^n/U_{\aaa},$ let $\widetilde\uuu\in U^n$ be a lift of $\uuu$. For $B>1$, by \ref{odhdoh}, one has \begin{align}\begin{split}\label{robuu}\{\uuu\in U^n/U_\aaa|H^{\#}_{\infty}(\uuu\xxx)<B\}&=\{\uuu\in U^n/U_\aaa|h_{\infty}(\rho^n(\widetilde{\uuu}\wx))<B\}\\&=\{\uuu\in U^n/U_\aaa|h_{\infty}(\rho^n(\wx)+\rho^n(\widetilde\uuu))<\log(B)\}.\end{split}\end{align} Further, by using \ref{oalfa} and the fact that $\ker(\beta )$ is finite (\ref{minir}), we obtain for every $B>1$ that
\begin{align}
\begin{split}\label{robou}
|\{\uuu&\in U^n/U_\aaa|h_{\infty}(\rho^n(\wx)+\rho(\widetilde\uuu))<\log(B)\}|\\&=\big|\big\{\uuu\in U^n/U_\aaa|\alpha(\rho^n(\wx)+\rho^n(\widetilde\uuu))\in\big[\frac{-\log(B)}{a_i},\frac{\log(B)}{a_i} \big]_{v,i,j} \big\}\big|\\
&=\big|\big\{\uuu\in U^n/U_\aaa|\alpha(\rho^n(\widetilde\uuu))\in-\alpha(\rho^n(\wx))+\big[\frac{-\log(B)}{a_i},\frac{\log(B)}{a_i} \big]_{v,i,j} \big\}\big|\\  &=\big|\{\uuu\in U^n/U_\aaa|\beta(\uuu)\in-\alpha(\rho^n(\wx))+\big[\frac{-\log(B)}{a_i},\frac{\log(B)}{a_j} \big]_{v,i,j}\}\big|\\
&=|\ker(\beta)|\cdot\big|\{\beta(U^n/U_\aaa)\cap \big(-\alpha(\rho^n(\wx))+\big[\frac{-\log(B)}{a_i},\frac{\log(B)}{a_j} \big]_{v,i,j}\big)\}\big|.
\end{split}
\end{align}
As $\beta(U^n/U_\aaa)$ is a subgroup (\ref{minir}) of a lattice of $(\RR^{M_F^{\infty}})^{n^2}$, there exists $C'>0$ such that \begin{align}
\begin{split}
\bigg|\beta(U^n/U_\aaa) \cap\big(-\alpha(\rho(\wx))+\big[\frac{-\log(B)}{a_i},\frac{\log(B)}{a_j} \big]_{v,i,j}\big)\}\bigg|\hskip-5cm&\\&\leq C'\prod_{i,j}\big(\frac{\log(B)}{a_i}+\frac{\log(B)}{a_j}\big)^{r_1+r_2}\\
&\leq C' \prod_{i,j}\big(\log(B^{1/a_i+1/a_j})\big)^{r_1+r_2}\\
&=C'\big(1/a_i+1/a_j\big)^{n^2(r_1+r_2)}\big(\log(B)\big)^{n^2(r_1+r_2)}.
\label{robo}
\end{split}
\end{align} for every $B>2$.
By combining $(\ref{robuu}), (\ref{robou})$ and $(\ref{robo})$ we deduce that there exists $C>0$ such that for every $B>2$ one has that\begin{align*}\{\uuu\in U^n/U_\aaa|H^{\#}_{\infty}(\uuu\xxx)<B\}&\leq C \big(\log(B)\big)^{n^2(r_1+r_2)}.\end{align*}
\end{proof}
\subsection{} We prove there are only finitely many isomorphism classes of rational points of bounded height on the ``stacky torus"~$\TTa$.
\begin{prop}\label{vnjem}
Let $(f_v:\Fvnz\to\RR_{>0})_v$ be a quasi-toric family of $\aaa$-homogenous functions of weighted degree $|\aaa|$. There exists $C>0$ such that for every $B>0$ one has $$\big|\big\{\xxx\in[\TTa(F)]|H(\xxx)<B\big\}\big|<C B\log(2+B)^{n^2(r_1+r_2)+n-1}. $$
\end{prop}
\begin{proof} Let $(f_v^{\#})_v$ be the family of toric $\aaa$-homogenous continuous functions of weighted degree $|\aaa|$. 
We firstly prove the statement for the family $(f_v=f_v^\#)_v$. Recall that $\omI:[\TTa(F)]\to\Div(F)^n/\Div(F)_{\aaa}$ is the induced homomorphism from $\Div(F)_{\aaa}$-invariant homomorphism $(F^\times)^n\to \Div(F)^n\to\Div(F)^n/\Div(F)_{\aaa}$. We set \begin{align*}
H_0^{\#}(\xxx)&:=\prod_{\vMFz}H_v^{\#}(\xxx)\\
H_{\infty}^{\#}(\xxx)&:=\prod_{\vMFi}H_v^{\#}(\xxx).
\end{align*}Recall that for every $\vMF$ we have $H^{\#}_v(\xxx)\geq 1$.  Using this fact and \ref{brkica}, for $B>0$ we deduce that \begin{align}\begin{split}\label{niub}\big|\big\{\xxx\in[\TTT ^\aaa(F)]|H^{\#}(\xxx)<B\big\}\big|\hskip-2cm&\\&\leq \big|\big\{\xxx\in[\TTa(F)]|H^{\#}_0(\xxx)<B,H^{\#}_{\infty}(\xxx)<B\big\}\big|\\&\leq \big|\big\{\xxx\in\TTT ^\aaa(F)|{H_{\Ideal}}(\omI(\xxx))<B,H^{\#}_\infty(\xxx)<B\big\}\big|.\end{split}\end{align} For $\yyy\in\Div(F)^n/\Div(F)_{\aaa},$ let $\widetilde\yyy\in\Div(F)^n$ be a lift of~$\yyy$. We have that 
\begin{multline*}\big|\big\{\xxx\in[\TTT ^\aaa(F)]|{H_{\Ideal}}(\omI(\xxx))<B,H^{\#}_\infty(\xxx)<B\big\}\big|\\=\big|\{(\yyy,\zzz)\in(\Div(F)^n/\Div(F)_\aaa)\times\ker(\omI)|H_{\Ideal}(\yyy)<B,H^{\#}_{\infty}(\zzz\widetilde\yyy)<B\}\big| \end{multline*}
and thus by (\ref{niub}) we get that \begin{multline}\label{puzovka}\big|\big\{\xxx\in[\TTT ^\aaa(F)]|H(\xxx)<B\big\}\big|\\\leq\big| \big\{(\yyy,\zzz)\in(\Div(F)^n/\Div(F)_\aaa)\times\ker(\omI)|H_{\Ideal}(\yyy)<B,H^{\#}_{\infty}(\zzz\widetilde\yyy)<B\big\}\big|.\end{multline}
Recall from \ref{uuakeri}, that $U^n/U_{\aaa}\subset\ker(\mathfrak I).$ Let $\delta$ be a set theoretical section to the quotient map $\ker(\omI)\to\ker(\omI)/(U^n/U_{\aaa})$. By \ref{kerpf}, the group $\ker(\omI)/(U^n/U_{\aaa})$ is finite, thus is $\delta(\ker(\omI)/(U^n/U_{\aaa}))$ finite. By using the estimate from \ref{boundunits}, we deduce that there exists $C_0>0$ such that for every $\yyy\in\Div(F)^n/\Div(F)_{\aaa}$ and every $B>0$ one has\begin{align}\begin{split}\label{horovka}|\{\zzz\in\ker(\omI)|H^{\#}_{\infty}(\zzz\widetilde\yyy)<B \}|\hskip-4cm&\\&\leq \big|\big\{(\uuu,\ddd)\in (U^n/U_\aaa)\times \delta(\ker(\omI)/(U^n/U_{\aaa}))\big |H^{\#}_{\infty}(\uuu\ddd\widetilde\yyy )<\log(B)\big\}\big| \\ &\leq\sum_{\ddd\in\delta(\ker(\omI)/(U^n/U_{\aaa}))}|\{\uuu\in U^n/U_\aaa|H^{\#}_{\infty}(\uuu\ddd\widetilde\yyy)<\log(B)\}|\\&\leq C_0(\log(B+2))^{n^2(r_1+r_2)}.\end{split} \end{align} 
Corollary \ref{hzer} gives that there exists $C_1>0$ such that for every $B>0$ one has that \begin{align*}|\{\yyy\in\Div(F)^n/\aaa\Div(F)|\hspace{0.1cm}{H_{\Ideal}}(\yyy)<B \}|&\leq C_1 B\log(1+B)^{n-1}.\end{align*}
We deduce that for every $B>0$ one has\begin{align*}\big|\big\{&\xxx\in\TTT ^\aaa(F)|H(\xxx)<B\big\}\big|\\
&\leq\big| \big\{(\yyy,\zzz)|(\yyy,\zzz)\in(\Div(F)^n/\Div(F)_\aaa)\times\ker(\omI);\\&\quad\quad\quad\quad\quad\quad\quad\quad\quad\quad\quad\quad H_{\Ideal}(\yyy)<B,H^{\#}_{\infty}(\zzz\widetilde\yyy)<B\big\}\big|\\
&\leq \big|\big\{\yyy\in\Div(F)^n/\Div(F)_{\aaa}| H_{\Ideal}(\yyy)<B\big\}\big|\big(\log(B)\big)^{n^2(r_1+r_2)}\\&\leq C_0C_1B\log(1+B)^{n-1}\log(2+B)^{n^2(r_1+r_2)}\\&\leq CB\log(2+B)^{n^2(r_1+r_2)+n-1},\end{align*}where $C>0$ is big enough.

Let $(f_v)_v$ now be any quasi-toric family of $\aaa$-homogenous continuous functions of weighted degree $|\aaa|$. By \ref{toricisclos}, there exists $C'>0$ such that $${H^\#(\xxx)}<C'{H(\xxx)} \hspace{1cm}\forall\xxx\in\Fvnz.$$
We deduce that \begin{align*}|\{\xxx\in[\TTa(F)]|H(\xxx)<B\}|&\leq|\{\xxx\in[\TTa(F)]|H^{\#}(\xxx)<C'B\}|\\&\leq C(C'B)\log(2+C'B)^{n^2(r_1+r_2)+n-1}\\
&=C''B\log(2+B)^{n^2(r_1+r_2)+n-1}\end{align*}
for $C''\gg 0$ and every $B>0$. The statement is proven.
\end{proof}
\subsection{}In this paragraph we prove that quasi-toric heights are Northcott heights. 
%
For $j\in\{1\doots n\}$, we will denote by $\{j\}^c$, the set $\{1\doots n\}-\{j\}.$ For $j\in\{1\doots n\}$, we denote by $p^j:\AAA^n\to\prod _{i\in \{j\}^c}\AAA^1$ the canonical projection and by $d_j$ the closed immersion  $$d_j:\prod _{i\in \{j\}^c}\AAA^1\to\AAA^n\hspace{1cm}(x_{i})_{i}\mapsto ((x_i)_{i\neq j}, (0)_j). $$ By abuse of notation, we will shorten $p^j(\ZZ)$ and write simply $p^j$. The morphisms $d_j$ are~$\Gm$-equivariant when $\prod _{i\in \{j\}^c}\AAA^1$ is endowed with the action of~$\Gm$ given by $$t\cdot(x_i)_{i}= (t^{a_i}x_i)_{i}.$$ For every~$j$, one has that $d_j^{-1}(\AAA^n-\{0\})=\big(\prod _{i\in\{j\}^c}\AAA^1\big)-\{0\}$. We deduce~$\Gm$-equivariant morphisms $$d_j|_{(\prod _{i\in\{j\}^c}\AAA^1)-\{0\}}:\big(\prod _{i\in\{j\}^c}\AAA^1\big)-\{0\}\to \AAA^n-\{0\}.$$ Let $\overline{d_j}$ denotes the induced closed immersion of stacks $\PPP(p^j(\aaa))\to\PPP(\aaa)$ from the~$\Gm$-equivariant morphism $d_j|_{{(\prod _{i\in\{j\}^c}}\AAA^1)-\{0\}}$.
\begin{lem} \label{abnu}
Let $j\in\{1\doots n\}$.
\begin{enumerate}
\item Let $\vMF$ and let $f:\Fvnz\to\RR_{>0}$ be an $\aaa$-homogenous continuous function of weighted degree $b\in\ZZ_{\geq 0}$. The function $f\circ (d_j(F_v)): (\prod _{i\in \{j\}^c}F_v) -\{0\}\to \CC$ is $p^j(\aaa)$-homogenous, continuous and of weighted degree $b$. Moreover, if~$f$ is the toric $\aaa$-homogenous function of weighted degree $b$, then $f\circ (d_j(F_v))$ is the toric $\aaa$-homogenous of weighted degree $b$.
\item Let $(g_v:\Fvnz\to\RR_{\geq 0})_v$ be a quasi-toric family of $\aaa$-homogenous continuous functions of weighted degree $b$ and let $H=H((g_v)_v)$ be the corresponding height. The family $ g_v\circ d_j(F_v)$ is a degree $b$ quasi-toric family of $p_I(\ZZ)(\aaa)$-homogenous continuous functions and the corresponding height $H^j=H((g_v\circ d_I(F_v))_v)$ satisfies that $$H^j=H\circ(\overline{d_j}(F)).$$
\item Let $\mathfrak i:\TTa\to\PPP(\aaa)$ be the inclusion induced by the~$\Gm$-invariant subscheme $\Gm^n\subset\AAA^n-\{0\}$. The map $$\big(\coprod_{j=1}^n[\overline{d_{j}}(F)]\coprod[\mathfrak i(F)]\big):\coprod _{j=1}^n[\PPP(p^{j}(\ZZ)(\aaa))(F)]\coprod [\TTa(F)]\to[\PPP(\aaa)(F)]$$ is surjective. 
\end{enumerate}
\end{lem}
\begin{proof}
\begin{enumerate}
\item The function $f\circ d_j(\Fv)$ is continuous, as it coincides with the restriction $f|_{(\prod_{i\in\{j\}^c}\Fv)-\{0\}}$ and~$f$ is continuous. Let $t\in\Fvt$ and let $(x_i)_{i\in\{j\}^c}\in (\prod_{i\in\{j\}^c}\Fv)-\{0\}$. We have that \begin{align*}f(d_j(\Fv)(t\cdot (x_i)_{i\in\{j\}^c}))&=f(d_j(\Fv)((t^{a_i}x_i )_{i\in\{j\}^c}))\\&=f((t^{a_i}x_i)_{i\in\{j\}^c}, (0)_{j})\\&=|t|_v^bf((x_i)_{i\in\{j\}^c},(0)_{j})\\&=|t|_v^b(f\circ d_j(\Fv)) ((x_i)_{i\in\{j\}^c}),\end{align*}thus $f\circ d_j(\Fv)$ is $p^{j}(\aaa)$ homogenous of weighted degree $b$. Suppose~$f$ is the toric $\aaa$-homogenous of weighted degree $b$, that is $$f(\yyy)=\pivv^{-b\cdot r^{\aaa}_v(\yyy)}\text{, where } r^{\aaa}_v(\yyy)={\sup _{\substack{j=1\doots n\\y_j\neq 0}}\bigg\lceil\frac{-v(y_j)}{a_j}\bigg\rceil} ,$$if~$v$ is finite and $$f(\yyy)= ({\max_{j=1\doots n}(|y_j|^{1/a_j}_v)})^{-b} $$ if~$v$ is infinite. 
Suppose~$v$ is finite. Let $\xxx\in (\prod_{i\in\{j\}^c}\Fv)-\{0\}$. One has that $d_j(\Fv)(\xxx)=((x_i)_{i\in\{j\}^c},(0)_j)$ and that $$(r^{\aaa}_v\circ (d_j(\Fv)))(\xxx)=\sup _{\substack{i\in\{j\}^c\\x_i\neq 0}}\bigg\lceil\frac{-v(x_i)}{a_i}\bigg\rceil=:r^{p^j(\aaa)}(\xxx),$$ because the~$j$-th coordinate of $d_j(\Fv)(\xxx)$ is zero.
It follows that the function $f\circ(d_j(\Fv))$ given by$$\xxx\mapsto f(d_j(\Fv)(\xxx))=\pivv^{-b\cdot r^{\aaa}_v(d_j(\Fv)(\xxx))}=\pivv^{-b\cdot r^{p^j(\aaa)}_v(\xxx)}$$ is toric. 

Suppose~$v$ is infinite. The toric $p^{j}(\aaa)$-homogenous function of weighted degree $b$ is the function $$f':(\prod_{i\in\{j\}^c}\Fv)-\{0\} \to \RR_{>0}\hspace{1cm}\xxx\mapsto \max_{i\in\{j\}^c}(|x_i|^{1/a_i}_v).$$ For $\xxx\in(\prod_{i\in\{j\}^c}\Fv)-\{0\},$ we have that $$f(d_j(\Fv)(\xxx))={\max_{i\in\{j\}^c}\big(|x_i|^{1/a_i}_v}\big)=f'(\xxx),$$ because the~$j$-th coordinate of $d_j(\Fv)(\xxx)$ is zero. The claim follows. 
\item By (1), for every $v\in M_F$, the function $g_v\circ d_I(F_v)$ is a continuous $p_I(\ZZ)(\aaa)$-homogenous of weighted degree $b$ and if $g_v$ is toric, then $g_v\circ d^j(F_v)$ is toric of the same weighted degree. It follows that the family $(g_v\circ d^j(F_v))_v$ is quasi-toric of the degree $b$. 
Let $\yyy\in[\PPP(p^j(\aaa))(\Fv)]$ and let $\widetilde\yyy\in\ (\prod _{i\in\{j\}^c}F)-\{0\}$ be a lift. Then $d^j(F)((\widetilde y_i)_{i\in\{j\}^c} )$ is a lift of $\overline{d^j}(F)(\yyy).$ We have that $$H^j(\yyy)=\prod_{v\in M_F}(g_v\circ d^j(F_v))(\widetilde\yyy)=\prod_{v\in M_F}g_v(d^j(F)(\widetilde\yyy))=H((g_v)_v)(\yyy)=H(\yyy).$$ The claim follows.
\item Let $\xxx\in[\PPP(\aaa)(F)]$ and let $\wx\in F^n-\{0\}$ be a lift of~$\xxx$.  If $\wx\in(\Ft)^n$, then $\xxx\cong[\mathfrak i(F)(q^{\aaa}(F)(\wx))],$ where $q^\aaa$ is the quotient~$1$-morphism $\Gmn\to\TTa$. Suppose at least one coordinate, say the~$j$-th coordinate, of~$\wx$ is equal to zero. Then, $$\wx= d_{j}(F)((\widetilde x_i)_{i\in\{j\}^c}) $$ 
and so$$\xxx\cong\overline{d_{j}}( q^{p^j(\aaa)}((\widetilde x_i)_{i\in\{j\}^c})),$$ where $q^{p^j(\aaa)}:\big((\prod _{i\in\{j\}^c}\AAA^1)-\{0\}\big)\to \PPP(p^j(\aaa))$, is the quotient~$1$-morphism. The claim follows.
\end{enumerate}
\end{proof}
\begin{thm}\label{boundppatta}Let $n\geq 1$ be an integer and let $\aaa\in\ZZ^n_{>0}$. Let $(f_v:\Fvnz\to\RR_{>0})_v$ be a quasi-toric family of $\aaa$-homogenous functions of weighted degree $|\aaa|$ and let $H=H((f_v)_v)$. 
\begin{enumerate}
\item  There exists $C>0$ such that for every $B>0$ one has that \begin{multline*}|\{\xxx\in[\PPP(\aaa)(F)]-[\TTa(F)]|H(\xxx)\leq B\}|\\<CB^{\frac{|\aaa|-\min_ia_i}{|\aaa|}}\log(2+B^{\frac{|\aaa|-\min_ia_i}{|\aaa|}})^{n^2(r_1+r_2)+n-1}.\end{multline*}
\item There exists $C>0$ such that for every $B>0$ one has that $$|\{\xxx\in[\PPP(\aaa)(F)]|H(\xxx)\leq B\}|<CB\log(2+B)^{n^2(r_1+r_2)+n-1}.$$
\end{enumerate}
\end{thm}
\begin{proof}
We prove the both statements simultaneously. We apply the induction on~$n$. When $n=1,$ one has that $\PPP(a)=\mathscr T(a)$. Thus the first claim is trivial and the second is proven in \ref{vnjem}. 

Suppose that the both claims are true for some $n-1\geq 0$ and let us prove them for~$n$.
The map
$$\big(\coprod_{j=1}^n[\overline{d_{j}}(F)]\coprod[\mathfrak i(F)]\big):\coprod _{j=1}^n[\PPP(p^{j}(\ZZ)(\aaa))(F)]\coprod [\TTa(F)]\to[\PPP(\aaa)(F)]$$ is surjective by \ref{abnu}. For $j=1\doots n$, by \ref{abnu}, one has that $$H^j=H((f_v\circ (d_j(\Fv)))_v)=H\circ \overline{d}_j,\hspace{0,5cm}\text{where }H^j:=H((f_v\circ (d_j(\Fv)))_v).$$ 
We deduce that \begin{multline}\label{povuin}
|\{\xxx\in[\PPP(\aaa)(F)]| H(\xxx)\leq B\}|\\\leq |\{\xxx\in[\TTa(F)]| H(\xxx)\leq B\}|+\sum_{j=1}^n|\{\xxx\in[\PPP(p^j(\aaa))(F)]| H^j(\xxx)\leq B\}|.
\end{multline}
For $j=1\doots n$, by \ref{abnu}, the height $H^j$ is defined by degree $|\aaa|$ quasi-toric family. The family $$(g_v\circ (d_j(\Fv))^{\frac{|p^j(\aaa)|}{|\aaa|}})_v$$ is quasi-toric of degree $|p^j(\aaa)|$ and the resulting height $H'=H((g_v\circ (d_j(\Fv))^{\frac{|p^j(\aaa)|}{|\aaa|}})_v)$ is related to $H^j$ as follows$$H'=(H^j)^{\frac{|p^j(\aaa)|}{|\aaa|}}.$$ By the induction hypothesis, we deduce that for every $j=1\doots n$ there exists $C_j>0$ such that for every $B>1$ one has that \begin{align*}|\{\xxx\in[\PPP(p^j&(\aaa))(F)]| H^j(\xxx)\leq B\}|\\&\leq |\{\xxx\in[\PPP(p^j(\aaa))(F)]| H^j(\xxx)^{\frac{|p^j(\aaa)|}{|\aaa|}}\leq B^{\frac{|p^j(\aaa)|}{|\aaa|}}\}|\\
&\leq |\{\xxx\in[\PPP(p^j(\aaa))(F)]| H'(\xxx)\leq B^{\frac{|p^j(\aaa)|}{|\aaa|}}\}|\\
&\leq C_jB^{\frac{|p^j(\aaa)|}{|\aaa|}}\log(2+B^{\frac{|p^j(\aaa)|}{|\aaa|}})^{n^2(r_1+r_2)+n-1}.
\end{align*}
It follows that \begin{align*}
|\{\xxx\in[\PPP(\aaa)(F)]-[\TTa(F)]| H(\xxx)\leq B\}|\hskip-8cm&\\&=|\{\xxx\in[\PPP(\aaa)(F)]| H(\xxx)\leq B\}|-|\{\xxx\in[\TTa(F)]| H(\xxx)\leq B\}|\\
&\leq \sum_{j=1}^n|\{\xxx\in[\PPP(p^j(\aaa))(F)]| H^j(\xxx)\leq B\}|\\
&\leq \sum_{j=1}^{n} C_jB^{\frac{|p^j(\aaa)|}{|\aaa|}}\log(2+B^{\frac{|p^j(\aaa)|}{|\aaa|}})\\
&\leq C''B^{\frac{|\aaa|-\min_ia_i}{|\aaa|}}\log(2+B^{\frac{|\aaa|-\min_ia_i}{|\aaa|}})
\end{align*} for $C''\gg 0$. Thus the first claim is proven for~$n$.
By \ref{vnjem}, there exists $C'>0$ such that \begin{equation}\label{sevimb}|\{\xxx\in[\TTa(F)]| H(\xxx)\leq B\}|<C'B\log(2+B)^{n^2(r_1+r_2)+n-1}.\end{equation}
By combining (\ref{povuin}), (\ref{sevimb}) and the first claim, we get that there exists $C>0$ such that for every $B>1$ one has that $$|\{\xxx\in[\PPP(\aaa)(F)]| H(\xxx)\leq B\}|\leq CB\log(2+B)^{n^2(r_1+r_2)+n-1}.$$
The statement is proven.
\end{proof}
\begin{rem}
\normalfont
In Chapter \ref{Analysis of height zeta functions}, we establish that $$|\{\xxx\in[\PPP(\aaa)(F)]|H(\xxx)\leq B\}|\sim C B,$$ for some $C>0$.
\end{rem}
\section{Northcott property for singular heights}\label{Northsing}
In this section we allow finitely many $f_v$ to take values in $0$ but we require that $f_v^{-1}$ admits a ``logarithmic singularity" over a rational divisor. We establish that corresponding heights are Northcott heights. 
\subsection{} Let~$K$ be a field and let $\aaa\in\ZZ_{>0}^n$. Let us define weighted degrees of polynomials in $K[X_1\doots X_n]$. For $j=1\doots n$, we define $\aaa$-weighted degree (or simply weighted degree) of polynomial $X_j$ by setting $\deg_{\aaa}(X_j)=a_j.$ For a monomial $cX_1^{d_1}\cdots X_n^{d_n},$ where $c\in R$,  we define $\deg_{\aaa}(cX_1^{d_1}\cdots X_n^{d_n})=\aaa\cdot\mathbf d$. Finally, if $P=\sum_i Q_i,$ with $Q_i$ monomials, we define $\deg_{\aaa}(P)=\max_i (\deg_{\aaa}(Q_i))$.
\begin{mydef}
We say that a polynomial $P\in K[X_1\doots X_n]$ is $\aaa$-homogenous if it is a sum of monomials of the same $\aaa$-weighted degree. 
\end{mydef}
\begin{lem}
Let $P$ be an $\aaa$-homogenous polynomial of weighted degree $d\geq 1$. 
\begin{enumerate}
\item For $t\in \Gm$ and $\xxx\in\AAA^n$, one has that $P(t\cdot\xxx)=t^d P(\xxx)$.
\item Let $Q$ be an $\aaa$-homogenous polynomial of weighted degree $k\geq 0$. The polynomial $PQ$ is an $\aaa$-homogenous polynomial of weighted degree $d+k$. 
\item Let $\{P_1\doots P_m\}\in F_v[X_1\doots X_n]$ be a set of non-constant $\aaa$-homogenous polynomials. The closed subscheme $Z(\{P_1\doots P_m\})\subset \AAA^n,$  given by the common zero set of $P_1\doots P_m$ is~$\Gm$-invariant for the action of~$\Gm$ on $\AAA^n$ with weights $a_1\doots a_n$. The open subscheme $D(P_1)\subset\AAA^n,$  given by the locus where $P_1$ does not vanish is~$\Gm$-invariant for the same action of~$\Gm$ on $\AAA^n$.
\end{enumerate}
\end{lem}
\begin{proof}
\begin{enumerate}
\item Let $t\in\Gm$ and let $CX_1^{m_1}\cdots X_n^{m_n}$ be a monomial. For $\xxx\in \AAA^n$, we have that $$CX_1^{m_1}\cdots X_n^{m_n}(t\cdot\xxx)=Ct_1^{a_1m_1}x_1^{m_1}\cdots t_n^{a_nm_n}x_n^{m_n}=t^{\aaa\cdot\mathbf m}X_1^{m_1}\cdots X_n^{m_n}(\xxx).$$ Now, $P$ can be written as sum $\sum_i C_iX^{m_{1,i}}_1\cdots X^{m_{n,i}}_n$ where for each~$i$ one has $\aaa\cdot\mathbf m_i=d$ and we deduce that $P(t\cdot\xxx)=t^dP(\xxx).$  
\item The product $PQ$ is a sum of monomials $CX^{m_1}_1\cdots X^{m_n}_n\cdot DX^{r_1}_1\cdots X^{r_n}_n,$ with $\aaa\cdot\mathbf m=d$ and $\aaa\cdot\mathbf r=k$. It follows that $PQ$ is $\aaa$-homogenous of weighted degree $\aaa(\mathbf m+\mathbf r)=k+d$.
\item For every $\xxx\in\AAA^n$ and every $i\in\{1\doots m\}$, we have that $P_i(t\cdot\xxx)=t^dP_i(\xxx)=0$ if and only if $P_i(\xxx)=0$. The claim follows.
\end{enumerate}
\end{proof}
\begin{mydef}Let $\{P_1\doots P_m\}\subset\Fv[X_1\doots X_n]$ be a set of non-constant $\aaa$-homogenous polynomials. We define $\mathcal Z(\{P_1\doots P_m\})$ to be the substack of~$\PPP(\aaa)$ defined by the~$\Gm$-invariant closed subscheme $Z(P_1\doots P_m)-\{0\}\subset\AAA^n-\{0\}$. We define $\mathcal D(P_1)$ to be the substack of~$\PPP(\aaa)$ defined by the~$\Gm$-invariant open subscheme $D(P_1)\subset\AAA^n-\{0\}$.
\end{mydef}
It follows from \cite[\href{https://stacks.math.columbia.edu/tag/04YN}{Lemma 04YN}]{stacks-project} that $\mathcal Z(\{P_1\doots P_k\})$ is a closed substack of~$\PPP(\aaa)$ and that $\mathcal D(P_1\cdots P_m)$ is an open substack of~$\PPP(\aaa)$.

Let $Q_1\doots Q_{r+1}$ be $\aaa$-homogenous polynomials of the same weighted degree~$d$. The morphism $$J(Q_1\doots Q_{r+1}):(\AAA^n-Z(\{Q_1\doots Q_{r+1}\}))\to\AAA^{r+1}-\{0\}\hspace{0.5cm}\xxx\mapsto (Q_i(\xxx))_i$$
is $t\mapsto t^d$-equivariant (see \ref{phiequiv} for the terminology), when the left scheme is endowed with the action $t\cdot_{\aaa}\xxx=(t^{a_j}x_j)_j$ of~$\Gm$ and the right scheme is endowed with the action $t\cdot_{\jed}\xxx=(tx_j)_j$. Let us denote by $$\overline{ J(Q_1\doots Q_{r+1})}:\PPP(\aaa)\to\PP^r$$ the~$1$-morphism of stacks induced by \cite[\href{https://stacks.math.columbia.edu/tag/046Q}{Lemma 046Q}]{stacks-project}.
\subsection{} In this paragraph we study heights that are obtained when $v$-~adic metric is singular. We require the singularities to be ``logarithmic" over a rational divisor. We establish that resulting heights are Northcott.
\begin{mydef} Let $P\in F[X_1\doots X_n]$ be a non-constant $\aaa$-homogenous polynomial and let~$S$ be a finite set of places. A collection of continuous $\Gm(\Fv)$-invariant functions $g_v:D(P)(\Fv)\to\RR_{>0}$ for $v\in S$, will be said to be logarithmically suitable  if the following condition is satisfied: there exists a set $\{Q_i\}_i$ of non-constant $\aaa$-homogenous polynomials $Q_i\in F[X_1\doots X_n]$ such that $Q_i$ and $P$ have no common factors of degree at least $1,$ such that $Z(\{Q_i\}_i)=\{0\}$ and such that for every $v\in S$ and every~$i$ there exist $\alpha_{v,i},\beta_{v,i}>0$ such that$$g_v(\xxx)\leq \alpha_{v,i}\max\bigg(-\log\bigg(\frac{|P(\xxx)|^{\deg_{\aaa}(Q)/\deg_\aaa(P)}_v}{|Q(\xxx)|_v}\bigg)^{\beta_{v,i}},1\bigg)$$ for every $\xxx\in D(PQ_i)(\Fv)$.
\end{mydef}
By the quasi-compactness of the scheme $\AAA^n$, one can always ask for the set $\{Q_i\}_i$ to be finite.

For $\vMF$, let $f_v^\#:F_v^n-\{0\}\to\RR_{>0}$ be the~$v$-adic toric $\aaa$-homogenous continuous function of weighted degree $d\geq 1$.  Let $P\in F[X_1\doots X_n]$ be a non-constant $\aaa$-homogenous polynomial. Let~$S$ be a finite set of places and let $(g_v)_v$ be a logarithmically suitable family of continuous~$\Gm(\Fv)$-invariant functions.  
For $v\in S$ we define $f_v:\Fvnz\to\RR_{\geq 0}$ by
 \begin{align*}\quad
f_v(\xxx):=&
               \begin{cases}
f_v^{\#}g_v^{-1}(\xxx) &\text{if } \xxx\in D(P)(\Fv),\\
0&\text{if }\xxx\in Z(P)(\Fv).
               \end{cases}
              \end{align*} and for $v\in M_F-S$ we let $f_v=f_v^{\#}$. For every $v\in M_F$, the function $f_v:\Fvnz\to\RR_{\geq 0}$ is an $\aaa$-homogenous continuous function of weighted degree $d\geq 1$. 
We define height $H=H((f_v)_v)$ and $H^\#=H((f_v^\#))$ on $\PPP(\aaa)(F)$. We recall that for every $\xxx\in\PPP(\aaa)(F)$, one has that $H^{\#}(\xxx)\geq~1$ by \ref{htoricbig}. 

Motivated by \cite[Proposition 2.1]{HeightsEllipticCurves}, we establish that: 
\begin{prop} \label{comphh}
There exist $C,\beta>0$ such that for all $\xxx\in \DD(P)(F)$ one has $$H(\xxx)\geq CH^\#(\xxx)\log (1+H^\#(\xxx))^{-\beta}$$
\end{prop}
\begin{proof} For $\xxx\in\DD(P)(F)$, let $\wx:\Gm\to\AAA^n-\{0\}$ be the~$\Gm$-equivariant morphism over~$F$ defined by~$\xxx$. For simplicity we will write~$\wx$ for $\wx(1)$.

There exists a finite set of non-constant $\aaa$-homogenous polynomials~$\{Q_{i}\}_i$ which have no common factors with $P_v$ of degree at least~$1$ such that $Z(\{Q_i\}_i)=\{0\},$ and such that for every $v\in S$ and every $i,$ there exist $\alpha_{i,v},\beta_{i,v}$ such that $$g_v(\xxx)\leq \alpha_{i,v}\max\bigg(-\log\bigg(\frac{|P(\xxx)|^{\deg_{\aaa}(Q)/\deg_\aaa(P)}_v}{|Q(\xxx)|_v}\bigg)^{\beta_{i,v}},1\bigg)$$ for every $\xxx\in D(PQ_i)(\Fv)$. 
 Fix an index~$i$. For $\xxx\in \DD(Q_iP)(F)$, using the fact that $f_v=f_v^\#$ for $v\not\in S$ we deduce that \begin{align*}H(\xxx)&=\prod _{v\in M_F}f_v(\wx)\\
&= \prod _{v\in M_F}f_v^{\#}(\wx)\prod _{v\in S}g_v(\wx)^{-1}\\
&\geq H^{\#}(\xxx)\prod_{v\in S} \alpha_{v,i}^{-1}\max\bigg(-\log\bigg(\frac{|P(\wx)|_v^{\deg_\aaa(Q_i)/\deg_{\aaa}(P)}}{|Q_i(\wx)|_v}\bigg),1\bigg)^{-\beta_{v,i}}\\
&\geq H^{\#}(\xxx)\prod_{v\in S} \alpha_{v,i}^{-1}\max\bigg(\log\bigg(\frac{|Q_i(\wx)|_v}{|P(\wx)|_v^{\deg_\aaa(Q_i)/\deg_{\aaa}(P)}}\bigg),1\bigg)^{-\beta_{v,i}}\\
&= H^{\#}(\xxx)\prod_{v\in S} \alpha_{v,i}^{-1}\log\bigg(\max\bigg(\frac{|Q_i(\wx)|_v}{|P(\wx)|_v^{\deg_\aaa(Q_i)/\deg_{\aaa}(P)}}\bigg),e\bigg)^{-\beta_{v,i}}\\
&\geq\alpha_iH^{\#}(\xxx)\prod_{v\in S}\log\bigg(\max\bigg(\frac{|Q_i(\wx)|_v}{|P(\wx)|_v^{\deg_\aaa(Q_i)/\deg_{\aaa}(P)}}\bigg),e\bigg)^{-\beta_{v,i}},
\end{align*}
where $\alpha=\prod_{v\in S}\alpha_{v,i}^{-1}$. 
Lemma \ref{estisilv} gives that for every $v\in S$, there exists $\gamma_v>1$ such that for every $\xxx\in \DD(Q_iP)(F)$ one has that :
\begin{align*}\max\bigg(\frac{|Q_i(\wx)|_v}{|P(\wx)|_v^{\deg_{\aaa}(Q_i)/\deg_{\aaa}(P)}},e\bigg)&\leq e\max \bigg(\frac{|Q_i(\wx)|_v}{|P(\wx)|_v^{\deg_{\aaa}(Q_i)/\deg_{\aaa}(P)}},1\bigg)\\
&\leq e\max\bigg(\frac{|Q_i(\wx)|^{\deg_{\aaa}(P)}_v}{|P(\wx)|_v^{\deg_{\aaa}(Q_i)}},1\bigg)^{1/\deg_{\aaa}(P)} \\
&\leq e\prod_{\vMF}\max \bigg(\frac{|Q_i(\wx)|^{\deg_{\aaa}(P)}_v}{|P(\wx)|_v^{\deg_{\aaa}(Q_i)}},1\bigg)^{1/\deg_{\aaa}(P)}\\
&= e H_{\PP^1}(\overline{J}( Q^{\deg_{\aaa}(P)}_i, P^{\deg_{\aaa}(Q_i)})(\xxx))^{1/\deg_{\aaa}(P)}\\
&\leq \gamma_v H^{\#}(\xxx)^{\deg_{\aaa}(Q_i)}.
 \end{align*}
For $v\in S$ let us set $\delta_v=\max(\log (\gamma_v),{\deg_{\aaa}(Q_i)})$.  We have for every $\xxx\in\mathcal D(Q_iP)(F)$ that
\begin{align*}
H(\xxx)&\geq \alpha_i H^{\#}(\xxx)\prod_{v\in S}\log(\gamma_vH^{\#}(\xxx)^{\deg_{\aaa}(Q_i)})^{-\beta_{v,i}}\\
&=\alpha_iH^{\#}(\xxx)\prod_{v\in S}\big(\log(\gamma_v)+{\deg_{\aaa}(Q_i)}\log(H^{\#}(\xxx))\big)^{-\beta_{v,i}}\\
&\geq \alpha_iH^{\#}(\xxx)\prod_{v\in S}\delta _v^{-\beta_{v,i}}\big(1+\log(H^{\#}(\xxx))\big)^{-\beta_{v,i}}\\
&= C_i H^{\#}(\xxx)\log(1+H^{\#}(\xxx))^{-\beta_i},
 \end{align*}
 where we have set $C_i=\alpha_i\prod_{v\in S} \delta _v^{-\beta_{v,i}}$ and $\beta_i=\sum_{v\in S}\beta_{v,i}$. We have that $\bigcup_iD(Q_iP)=D(P).$ Thus for $C=\min_i C_i$ and $\beta=\max \beta_i$ we get $$H(\xxx)\geq C H^{\#}\log(1+H^{\#}(\xxx))^{-\beta},$$ for every $\xxx\in D(P)(F)$.
\end{proof} 
\begin{cor}\label{nortlogfalt}
Let $(f_v)_v$ be as above and let~$H$ be the corresponding height. The height~$H$ is a Northcott height. Moreover, for every $\epsilon >0$, there exists $C=C(\epsilon)>0$ such that for every $B>1$ one has that $$|\{\xxx\in [\mathcal D(P)(F)] |H(\xxx)\leq B\}|\leq CB^{1+\epsilon}$$ and that \begin{multline*}|\{\xxx\in[\DD(P)(F)]\cap( [\PPP(\aaa)(F)]-[\TTa(F)])| H(\xxx)\leq B\}|\\
\leq CB^{\frac{(1+\epsilon)(|\aaa|-\min_{j}a_j)}{|\aaa|}}.\end{multline*}
\end{cor}
\begin{proof}
By \ref{estisilv}, there exist $C',N>0$, such that for every $\xxx\in[\mathcal D(P)(F)]$ one has that $$H(\xxx)\geq C'H^\#(\xxx)\log (1+H^\#(\xxx))^{-N}.$$ and thus there exists $A>0$ and $\frac \epsilon{1+\epsilon}>\delta >0$ such that $$H(\xxx)\geq A H^{\#}(\xxx)^{1-\delta}\hspace{1cm}\forall\xxx\in[\mathcal D(P)(F)].$$  
Using \ref{boundppatta}, one has hence that there exists $C>0$ such that \begin{align*}
|\{\xxx\in[\mathcal D(P)(F)]| H(\xxx)\leq B\}|\hskip-3cm&\\ &\leq |\{\xxx\in[\mathcal D(P)(F)]| AH^{\#}(\xxx)^{1-\delta}\leq B\}|\\
&=|\{\xxx\in[\mathcal D((P)(F))]| H^{\#}(\xxx)\leq A^{-1/(1-\delta)}B^{1/(1-\delta)}\}|\\
&\leq A^{-1/(1-\delta)}B^{1/(1-\delta)}\log(2+A^{-1/(1-\delta)}B^{1/(1-\delta)})^{n^2(r_1+r_2)+n-1}\\
&\leq CB^{1+\epsilon}
\end{align*}
and such that
\begin{align*}
|\{\xxx\in[\mathcal D(P)(F)]\cap([\PPP(\aaa)(F)]-[\TTa(F)])| H(\xxx)\leq B\}|\hskip-10cm&\\ &\leq |\{\xxx\in[\mathcal D(P)(F)]| AH^{\#}(\xxx)^{1-\delta}\leq B\}|\\
&=|\{\xxx\in[\mathcal D((P)(F))]\cap([\PPP(\aaa)(F)]-[\TTa(F)])| H^{\#}(\xxx)\leq A^{-\frac1{1-\delta}}B^{\frac1{1-\delta}}\}|\\
&\leq (A^{-1}B)^{\frac{(1+\epsilon)(|\aaa|-\min_ja_j)}{(1-\delta)|\aaa|}}\log(2+(A^{-1}B)^{\frac{|\aaa|-\min_ja_j}{(1-\delta)|\aaa|}})^{n^2(r_1+r_2)+n-1}\\
&\leq CB^\frac{(1+\epsilon)(|\aaa|-\min_{j}a_j)}{|\aaa|}.
\end{align*}
The statement is proven.\end{proof}
\chapter{Measures}
\label{Measures on topological spaces associated to weighted projective stacks}
Let $n\geq 1$ be an integer. Let $\aaa\in\ZZ^n_{>0}$. Recall that~$\PPP(\aaa)$ is the quotient stack for the action $$\Gm\times (\AAA^n-\{0\})\to \AAA^n-\{0\}\hspace{1cm}t\cdot\xxx=(t^{a_j}x_j)_j.$$  The open subscheme $\Gmn\subset\AAA^n-\{0\}$ is~$\Gm$-invariant for this action and~$\TTa$ is defined to be the quotient $\Gmn/\Gm$. In this chapter we define measures on $[\PPP(\aaa)(\Fv)]$ and $[\TTa(\Fv)]$ for $\vMF$. We use these measures to define Peyre's constant as in \cite{Peyre}. Later in the section we define a Haar measure on the adelic space $[\TTa(\AAF)]$ and we define and calculate the corresponding Tamagawa number. 
\section{Quotient measures}
\subsection{}
We make several conventions on measures that will be used throughout the chapter. Let~$X$ be a locally compact topological space. Let $\mathscr C^0_c(X,\CC)$ be the set of continuous compactly supported functions on~$X$. We endow $\mathscr C^0_c(X,\CC)$ by the uniform convergence topology.  By a \textit{measure} on~$X$ \cite[Definition 2, \no 3, \S 1, Chapter III]{Integrationj}, we mean a continuous linear functional $\mu:\mathscr C^0_c(X,\CC)\to\CC$ .

 Let~$\mu$ be a measure on~$X$. Let $L^1(X,\mu)$ be the Banach space of absolutely~$\mu$-integrable complex valued functions modulo negligible functions  \cite[Definition 2, \no 4, \S 3, Chapter IV]{Integrationj}. By the abuse of the terminology, we may call an element $f\in L^1(X,\mu)$ a function and for a function~$g:X\to\CC$ which is~$\mu$-absolutely integrable we may write $g\in L^1(X,\mu)$. For $f\in L^1(X,\mu)$, we denote by $\int_Xf\mu$ the integral of~$f$ against~$\mu$ \cite[Definition 1, \no 1, \S 4, Chapter IV]{Integrationj}. If $U\subset X$ is a subset, such that $\mathbf 1_{U,X}\in L^1(X,\mu)$, we write $\mu(U)$ for $\mu (\mathbf 1_{U,X})$ (where $\mathbf 1_{U,X}$ stands for the characteristic function of $U$ in~$X$, written sometimes as $\mathbf 1_{U}$).  Such~$U$ will be said to be~$\mu$-measurable.
 \subsection{}\label{intqm}
We recall some facts about quotient measures from \S 2, Chapter 7 of \cite{Integrationd}.  Let~$X$ be a locally compact Hausdorff topological space. Let~$G$ be a locally compact topological group acting on the right on~$X$ continuously and properly (that is the action $X\times G\to X$ is continuous and proper). The quotient topological space $X/G$ is separated \cite[Proposition 3, \no2, \S 2, Chapter III]{TopologieGj} and locally compact \cite[Proposition 10, \no4, \S 10, Chapter I]{TopologieGj}.

Let $dg$ be a left Haar measure on~$G$. 
Let $\Delta_G:G\to \RR_{>0}$ be the modular function of~$G$ (we recall the definition of the modular function: according to \cite[Formula (11), \no 1, \S 1, Chapter VII]{Integrationd}, for every $h\in G$, the association $A\mapsto dg(Ah)$, for $dg$-measurable subset~$A$ of~$G$, is a left Haar measure on~$G$, hence, by the unicity of the Haar measure there exists a unique positive real number $\Delta_G(h)$ such that for every $dg$-measurable subset~$A$ of~$G$, one has that $dg(Ag)=\Delta_G(g)dg(A)$; further, it does not depend on the choice of the Haar measure $dg$). 
\begin{prop}[{\cite[Proposition 1, \no2, \S 2, Chapter VII]{Integrationd}}] \label{duio}
For $x\in X/G$, let $\widetilde x\in X$ be a lift of~$x$. Let $\phi :X\to \CC$ be a compactly supported continuous function. For every $y\in X$, one has that $g\mapsto \phi(yg)\in L^1(G,\CC).$ Moreover, for $x\in X/G$, the value of $\int _G\phi(\widetilde xg)dg $ does not depend on the choice of $\widetilde x$. The function $\phi^*:x\mapsto \int _G\phi(\widetilde xg)dg$ is continuous and compactly supported.
\end{prop}
\begin{prop}[{\cite[Proposition 4, \no2, \S 2, Chapter VII]{Integrationd}}]\label{intbour}
\begin{enumerate}
\item Let~$\mu$ be a measure on~$X$ such that for every~$\mu$-measurable $U$ one has that $\mu (Ug)=\Delta_G(g)\mu (U)$ (such measures will be called~$G$-invariant measures). There exists a unique measure $\mu/dg$ on $X/G$ such that for every compactly supported function $\phi:X\to\CC$ one has that $$\int _X\phi \mu =\int _{X/G} (\phi^*)(\mu/dg).$$
\item Let $\mu'$ be a measure on $X/G$. There exists a unique~$G$-invariant measure $\mu $ on~$X$ such that $\mu/dg=\mu'$.
\end{enumerate}
\end{prop}
Let $\pi:X\to X/G$ be the quotient map. We quote following two propositions from \cite{Integrationd}.
\begin{prop} [{\cite[Proposition 8, \no 4, \S 2, Chapter VII]{Integrationd}}] \label{jeiui}
Suppose that $X/G$ is paracompact. Let $\lambda>0$. There exists a continuous function $k:X\to\RR_{\geq 0}$, the support of which has a compact intersection with the preimage under $\pi$ of any compact of $X/G$ and such that for any $x\in X$ one has that $$\int _{G}k(xg)dg=\lambda. $$
\end{prop}
Of course, when $X/G$ is compact, the condition on the support of~$k$ becomes that it is compact.
\begin{lem}\label{izbitubit}
Let $f:X\to\CC$ be a continuous~$G$-invariant function. Let $\overline f:X/G\to \CC$ be the function induced from~$f$. One has that $(f\mu)/dg=(\overline f)(\mu/dg).$
\end{lem}
\begin{proof}
Let $\phi:X\to \CC$ be a compactly supported continuous function. The function $\phi f$ is compactly supported. Let $x\in X/G$ and let $\widetilde x\in X$ be its lift. One has that \begin{multline*}(\phi\cdot f)^*(x)=\int_G\phi f(\widetilde xg)dg=\int_G\phi(\widetilde xg)f(\widetilde xg)dg=\overline f(x)\int_G\phi(\widetilde xg)dg\\=(\phi^*)\cdot \overline f)(x) .\end{multline*} It follows that: $(\phi\cdot f)^*=\phi^*\cdot \overline f$. We deduce that: 
\begin{align*}
\int_{X/G}(\phi^*)(f\mu/dg)&=\int_X(\phi)(f\mu)\\&=\int_X(\phi f)\mu\\&=\int_{X/G}((\phi\cdot f)^*)(\mu/dg)\\&=\int_{X/G}(\phi^*\overline f)(\mu/dg)\\&=\int_{X/G}(\phi^*)((\overline f)(\mu/dg)).
\end{align*}
It follows that $(f\mu/dg)=(\overline f)(\mu/dg)$.
\end{proof}
\begin{prop}[{\cite[Proposition 9, \no 4, \S 2, Chapter VII]{Integrationd}}]\label{simint}
Let us suppose that $X/G$ is paracompact. Let $k:X\to\RR_{\geq 0}$ be a continuous function, the support of which has compact intersection with the preimage under $\pi$ of any compact of $X/G$ and such that for every $x\in X$ one has $$\int _{G}k(xg)dg=1.$$ Then for any function $h:X/G\to\CC$ one has that $h\in L^1(X/G,\mu/dg)$ if and only if $k\cdot(h\circ\pi)\in L^1(X,\mu)$ and if $h\in L^1(X/G,\mu/dg)$ then $$\int _{X/G}(h)(\mu/dg)=\int _{X}k\cdot (h\circ\pi)\mu. $$
\end{prop}

\subsection{}
We recall the notion of the quotient measure when we have inclusion of locally compact abelian groups. 

Let~$G$ is an abelian locally compact Hausdorff topological group. The action of $G\times G\to G$ of~$G$ on~$G$ given by the multiplication is proper (because the induced map $G\times G\xrightarrow{(m_G,p_2)}G\times G$ is a homeomorphism, where $m_G$ is the multiplication map and $p_2$ the projection to the second coordinate). Thus if~$A$ a closed subgroup, the action of~$A$ on $G,$ given by the restriction of the action of~$G$, is proper by \cite[Example 1, \no 1, \S 4, Chapter III]{TopologieGj}. Let $dg$ be a Haar measure on~$G$ and let $da$ be a Haar measure on~$A$. By {\cite[Proposition 10, \no 7, \S 2, Chapter VII]{Integrationd}}\label{venr}, the quotient measure $dg/da$ is a Haar measure on $G/A$. 

\subsection{} In this paragraph we recall the theory of \cite{Oesterle} on the Euler characteristics of complexes of locally compact abelian groups endowed with Haar measures. 

A homomorphism of topological groups is said to be {\it strict} (see \cite[Definition 1, \no 8, \S 2, Chapter III]{TopologieGj}), if the induced homomorphism to its image is open. By \cite[Corollary of Proposition 6, \no 3, \S 5, Chapter IX]{TopologieGd}, any morphism of locally compact abelian groups which are countable at infinity is strict if and only if its image is closed. 
Consider a complex $C_{\bullet}$ of locally compact abelian groups which are countable at infinity
\begin{equation*}
\cdots\to C_{n+1}\xrightarrow{d_{n+1}} C_n\xrightarrow{d_n} C_{n-1}\cdots
\end{equation*}
such that the following conditions are satisfied:
\begin{itemize}
\item the homomorphisms $d_n$ are continuous for every $n\in\ZZ$,
\item the complex is bounded and of finite homology (i.e. $\ker (d_n)/\Imm(d_{n+1})$ is a finite group for every $n\in\ZZ$.)
\end{itemize}
Suppose that for every $n\in\ZZ$, we are given a Haar measure $\lambda_n$ on $C_n$ and that for almost all~$n$ the  measure $\lambda_n$ normalized by $\lambda_n(C_n)=1$. For every $n\in\ZZ$, let us choose Haar measures $\nu_n$ and $\theta_n$ on $\ker(d_n)$ and $\Imm(d_{n}),$ respectively. Let $\alpha_n$ be the volume of the finite set $\ker (d_n)/\Imm(d_{n+1})$ for the quotient measure $\nu_n/\theta_{n+1}$. Let $\beta _n>0$ be the unique real number such that $\lambda_n/\nu_n=\beta_n\theta_n$ (after identifying $C_n/\ker (d_n)$ and $\Imm(d_n)$ via the isomorphism induced from $d_n$). Oesterl\'e defines the number$$\chi(C_{\bullet})=\chi(C_{\bullet}, (\lambda_n)_n):=\prod_{n\in\ZZ}(\alpha_n\beta_n)^{(-1)^n},$$ which does not depend on the choice of $\nu_n$ and $\theta_n$. We may say that $C_{\bullet}$ is of the trivial measure Euler-Poincar\'e characteristic if $\chi(C_{\bullet})=1$.
\begin{lem}[{Oesterl\'e, \cite[Examples 2 and 3]{Oesterle}}] The following claims are valid:
\label{lemoesterle}
\begin{enumerate}
\item  Suppose that every $C_n$ is compact. Then $$\chi(C_{\bullet})=\prod_{n\in\ZZ}\lambda_n(C_n)^{(-1)^{n}}.$$
\item Suppose that every $C_n$ is discrete and that $\lambda_n$ are counting measures. One has that $$\chi(C_{\bullet})=\prod_{n\in\ZZ} [\ker(d_n):\Imm(d_{n+1})]^{(-1)^n}.$$
\end{enumerate}
\end{lem}
\begin{prop}[{Oesterl\'e, \cite[Proposition 1, \no 4]{Oesterle}}] \label{complexdoesterle}
Consider a bicomplex 
\[\begin{tikzcd}
	& \vdots & \vdots \\
	\cdots & {C_{n,m-1}} & {C_{n-1,m-1}} & \cdots \\
	\cdots & {C_{n,m}} & {C_{n-1,m}} & \cdots \\
	& \vdots & \vdots
	\arrow[from=2-1, to=2-2]
	\arrow[from=3-1, to=3-2]
	\arrow["{d'_{n,m-1}}", from=2-2, to=2-3]
	\arrow["{d'_{n,m}}", from=3-2, to=3-3]
	\arrow[from=2-3, to=2-4]
	\arrow[from=3-3, to=3-4]
	\arrow["{d''_{n,m}}"', from=3-2, to=2-2]
	\arrow["{d''_{n-1,m}}"', from=3-3, to=2-3]
	\arrow[from=2-2, to=1-2]
	\arrow[from=2-3, to=1-3]
	\arrow[from=4-3, to=3-3]
	\arrow[from=4-2, to=3-2]
\end{tikzcd}\] of locally compact abelian groups endowed with Haar measures, such that $C_{n,m}$ is the trivial group for $|n|+|m|$ big enough and such that for every $n\in\ZZ$ and every $m\in\ZZ$, the complexes $C_{n,\bullet}$ and $C_{\bullet, m}$ satisfy the above conditions. One has that $$\prod_{n\in\ZZ}\chi(C_{n,\bullet})^{(-1)^{n}}=\prod_{m\in\ZZ}\chi(C_{\bullet, m})^{(-1)^{m}}.$$
\end{prop}
We end the paragraph by a lemma that will be used on several occasions.
\begin{lem} \label{rudj}
Let $\epsilon:H\to G$ be a proper continuous homomorphism of locally compact Hausdorff abelian topological groups. Let $dg$ and $dh$ be Haar measures on~$G$ and~$H$. Consider the left action of~$H$ on~$G$ $$H\times G\to G\hspace{1cm}h\cdot g=\epsilon (h)g.$$
\begin{enumerate}
\item The action is continuous and proper. The quotient $G/H$ identifies with the quotient group $G/\epsilon (H)$. 
\item The group $K:=\ker(\epsilon)$ is compact and let $dk$ be the probability Haar measure on~$K$. We have an equality of measures on $\epsilon (H)$: $$\epsilon _*(dh)=dh/dk.$$
\item The subgroup $\epsilon(H)$ is closed in~$G$. The measure $dg/dh$ is the quotient Haar measure $dg/\epsilon_*(dh)$.
\item Suppose that~$H$ and~$G$ are countable at infinity. The exact sequence $1\to K\to H\to G\to G/H\to 1$ is of the trivial measure Euler-Poincar\'e characteristics, when the measures on $K,$~$H$,~$G$ and $G/H$ are $dk$, $dh$, $dg$ and $dg/dh$.
\end{enumerate}
\end{lem}
\begin{proof}
\begin{enumerate}
\item The continuity of the group action follows from the continuity of the multiplication by an element in~$G$. The action is proper, as one has a Cartesian diagram $$\begin{tikzpicture}
  \matrix (m) [matrix of math nodes,row sep=3em,column sep=8em,minimum width=2em]
  {
G\times H & G\times G \\
     H& G,\\ };
  \path[-stealth]
    (m-1-1) edge node [left] {$p_2$} (m-2-1)
            edge node [above] {$(g,h)\mapsto (g,g\epsilon(h))$} (m-1-2) 
    (m-2-1.east|-m-2-2) edge node [below] {$\epsilon$}
            node [above] {$ $} (m-2-2)
    (m-1-2) edge node [right] {$ $} (m-2-2)
               ;
\end{tikzpicture} $$
where the right vertical homomorphism is given by $(g_1,g_2)\mapsto g_1^{-1}g_2$. The canonical homomorphism $G\to G/\epsilon (H)$ is continuous and open and induces a continuous and open map $G/H\rightarrow G/\epsilon (H).$ The map is bijective because one has an equality of sets $G/H=G/\epsilon (H)$. We deduce a topological identification $G/H=G/\epsilon(H)$. 
The claim follows.
\item The map $\epsilon$ is proper and it follows that~$K$ is a compact subgroup of~$H$. The quotient space $\epsilon (H)=H/\ker(\epsilon)$ is paracompact \cite[Proposition 13, \no 6, \S 4, Chapter III]{Integrationj}. Note that the constant function $H\to \CC$ given by $h\mapsto 1,$ satisfies the condition that its support (i.e. the whole of~$H$) has compact intersection with the preimage $\epsilon^{-1}(A)$ for every compact $A\subset\epsilon(H),$ because $\epsilon$ is proper. 
Now, by \ref{simint}, for every $r:\epsilon(H)\to\CC$ with $r\in L^1(\epsilon (H),dh/dk)$ we have
$$\int _{\epsilon(H)}(r)(dh/dk)=\int_{H}(r\circ\epsilon)\cdot 1 dh=\int_{H}(r\circ\epsilon)dh,$$ which means precisely that $\epsilon_*(dh)=dh/dk$.
\item The subgroup $\epsilon (H)$ is closed as $\epsilon$ is a proper map. Let $\phi:G/H\to\CC$ be a compactly supported function. By \cite[Proposition 2, \no 1, \S2, Chapter VII]{Integrationd}, there exists a compactly supported continuous function $\Phi:G\to\CC$, such that for any $x\in G/H$ and any lift $\widetilde x\in G$ of~$x$, we have $$\phi (x)=\int _{H}\Phi (h\widetilde x)dh.$$ Note that $$\phi(x)=\int _{\epsilon(H)}\Phi(h\widetilde x)\epsilon_*(dh).$$ We have that $$\int _{G/H}(\phi) (dg/dh)=\int _{G}\Phi \hspace{0.1cm}dg=\int _{G/\epsilon(H)}(\phi)(dg/\epsilon_*(dh)).$$ It follows that $dg/dh=dg/\epsilon_*(dh)$ as claimed.
\item The short exact sequences
\[\begin{tikzcd}
	1 & K & H & {\epsilon(H)} & 1 \\
	1 & {\epsilon(H)} & G & {G/H} & 1
	\arrow[from=1-1, to=1-2]
	\arrow[from=1-2, to=1-3]
	\arrow[from=1-3, to=1-4]
	\arrow[from=1-4, to=1-5]
	\arrow[from=2-1, to=2-2]
	\arrow[from=2-2, to=2-3]
	\arrow[from=2-3, to=2-4]
	\arrow[from=2-4, to=2-5]
\end{tikzcd}\] are of the trivial measure Euler-Poincar\'e characteristics by (2) and (3). Thus $1\to K\to H\to G\to G/H\to 1$ is of the trivial measure Euler-Poincar\'e characteristics.
\end{enumerate} 
\end{proof}
\section{Measures on $[\PPP(\aaa)(F_v)]$}\label{measuresonppa}
For $\vMF$, the abelian locally compact group $\Fvt$ acts on $\Fvnz$ by $t\cdot\xxx=(t^{a_j}x_j)_j$. This action is proper by \ref{aboutopa}. By \ref{pafvtafv}, the quotient $(\Fvnz)/\Fvt$ identifies with $[\PPP(\aaa)(\Fv)].$ The goal of this section is to define measures on $[\PPP(\aaa)_{\Fv}(\Fv)]=[\PPP(\aaa)(F_v)]$ for $\vMF$. 
The topological spaces $[\PPP(\aaa)(\Fv)]$ are Hausdorff and compact by \ref{paraap}. Let us denote by $q^{\aaa}_v$ the quotient maps $\Fvnz\to[\PPP(\aaa)(\Fv)]$ for $\vMF$. 
\subsection{}\label{pojac}
In this paragraph, we define measure on $F_v$.

For $v\in M_F$, let $dx_v$ be the Haar measure on $F_v$ normalized by \begin{itemize}
\item $dx_v(\Ov)=1$ if~$v$ is finite,\\
\item $dx_v$ is the Lebesgue measure on $F_v\cong\RR$ if~$v$ is real,\\
\item $dx_v=2dxdy$ on $F_v\cong\CC\cong \RR^2$ if~$v$ is complex. 
\end{itemize} 
For $\vMF$, let $d^*x_v$ be the Haar measure $\frac{dx_v}{|x_v|_v}$ on $\Fvt.$ For $\vMFz$, it satisfies that $$d^*x_v(\Ovt)=\int_{\Ovt}1d^*x_v=\int_{\Ovt}1dx_v=1-\pivv.$$  When~$v$ is real, the measure $d^*x_v$ identifies with the measure $$d^*x:=\frac {dx}{|x|}$$ on $\RR^{\times}$ and when~$v$ is complex, the measure $d^*x_v$ identifies with the measure $$\frac{2dxdy}{x^2+y^2} $$ on $\CC^2-\{0\}\cong\RR^2-\{0\}$. Let us set $$\Fvj:=\{x\in\Fv|\hspace{0.1cm} |x|_v=1\}, $$ so that $\Fvj=\{\pm 1\}$ when~$v$ is real and $$\Fvj=S^1=\{(x,y)\in\RR^2| x^2+y^2=1\}$$ when~$v$ is complex.
Recall that we have set $n_v=1$ if~$v$ is a real place and let $n_v=2$ if~$v$ is complex. The exact sequence $$\{1\}\to F_{v,1}\xrightarrow{i_{F_{v,1}}}\Fvt\xrightarrow{x\mapsto |x|_v}\RR_{>0}\to\{1\}, $$ where $i_{F_{v,1}}$ is the inclusion map, admits a section $$\rho _v:\RR_{>0}\to\Fvt\hspace{1cm}r\mapsto r^{1/n_v}.$$ The section induces a continuous isomorphism \begin{equation}\label{splitfv}\widetilde \rho _v:\RR_{>0}\times F_{v,1}\to\Fvt\hspace{1cm} (r,z)\mapsto \rho_v(r)z.\end{equation} The inverse of this homomorphism is given by \begin{equation}\label{lqq}\Fvt\to \RR_{>0}\times\Fvj\hspace{1cm}x\mapsto (|x|_v,x\rho _v(|x|_v))\end{equation} and is also continuous. We deduce that $\widetilde\rho _v$ is an isomorphism of topological groups.  
Let us set $\lambda _{v,1}$ to be the counting measure on $\{\pm 1\}$ when~$v$ is real and let us set $\lambda_{v,1}$ to be the Haar measure on $F_{v,1}$ normalized by $\lambda_{v,1}(F_{v,1})=2\pi$ when~$v$ is complex.
\begin{lem}\label{cyv}
Let $\vMFi$. One has that and that $(\widetilde\rho_v)_*\big(dr\times\lambda_{v,1}\big)=dx_v|_{\Fvt}$ and that $(\widetilde\rho_v)_*\big(d^*r\times\lambda_{v,1}\big)=d^*x_v$.
\end{lem}
\begin{proof}
We will firstly verify that $(\widetilde\rho_v)_*\big(d^*r\times\lambda_{v,1}\big)=d^*x_v$. Suppose~$v$ is real. Both measures are Haar measures on $\Fvt$, so it suffices to check their equality on a single Borel subset of $\Fvt\cong\RR^{\times},$ and we will verify it on $[1,2]$. We have that $[1,2]=\widetilde\rho_v (\{1\}\times[1,2])$ and \begin{align*}{d^*x_v}([1,2])={d^*x}([1,2])&=1\cdot{d^*x}([1,2])\\&=\lambda_{v,1}(\{1\})\cdot d^*r([1,2])\\&=\big(\lambda_{v,1}\times d^*r\big)(\{1\}\times [1,2])\\&=\big((\widetilde\rho_v)_*\big(d^*r\times\lambda_{v,1}\big)\big)(\widetilde\rho_v (\{ 1\}\times [1,2])).\end{align*} Suppose~$v$ is complex. 
Note that $\lambda_{v,1}$ is the pushforward measure for the map $[0,2\pi[\to S^1$ given by \begin{equation*}\theta\mapsto e^{i\theta}=(\cos(\theta),\sin(\theta)).\end{equation*} Hence, the map $$S^1\times\RR_{>0}\to\RR^2-\{0\}\hspace{1cm}(z,r)\mapsto zr^{1/2}$$ is measure preserving if and only if the map $$[0,2\pi[\times\RR_{>0}\to \RR^2-\{0\}\hspace{1cm}(\theta,r)\mapsto(r^{1/2}\cos(\theta),r^{1/2}\sin(\theta))$$ is measure preserving. The corresponding Jacobian matrix is $$\begin{pmatrix}-r^{1/2}\sin(\theta) &\frac 12r^{-1/2}\cos(\theta)\\r^{1/2}\cos(\theta) &\frac 12r^{-1/2}\sin(\theta) \end{pmatrix}$$ and its determinant equals $-\frac 12$. It follows that $$dxdy=(\widetilde\rho_v)_*(\lvert-1/2\big\rvert drd\theta)=(\widetilde\rho_v)_*((1/2) drd\theta),$$ and hence that $$d^*x_v=\frac{2dxdy}{x^2+y^2}=(\widetilde\rho_v)_*\bigg(\frac{2\cdot(1/2)drd\theta}{r}\bigg)=(\widetilde\rho_v)_*(d^*rd\theta).$$ We have proven that $(\widetilde\rho_v)_*\big(d^*r\times\lambda_{v,1}\big)=d^*x_v$. 

Let us now verify that $(\widetilde\rho_v)_*\big(dr\times\lambda_{v,1}\big)=dx_v|_{\Fvt}$. One has that $|\widetilde\rho_v(r,z)|_v=r$ for every $(r,z)\in\RR_{>0}\times F_{v,1}$. For a Borel subset $U\subset\Fvt$, one has that \begin{align*}dx_v|_{\Fvt}(U)=\int_U|x_v|_vd^*x_v&=\int_{U}|\widetilde\rho_v(r,z)|_v(\rho_v)_*(d^*r\times \lambda_{v,1})\\&=\int_{\widetilde {\rho}_v^{-1}(U)}r \hspace{0.1cm}(d^*r\times\lambda_{v,1})\\&=\int _{\widetilde {\rho}_v^{-1}(U)}dr\times\lambda_{v,1}\\&=(dr\times\lambda_{v,1})(U).\end{align*}
It follows that $dx_v|_{\Fvt}=(\widetilde\rho_v)_*(dr\times\lambda_{v,1})$. The statement is proven.
\end{proof}
We will often write $dx$ and $d^*x$ for $dx_v$ and $d^*x_v$, respectively.
\subsection{}\label{kave}
In this paragraph we will define compactly supported continuous functions $\kav:\Fvnz \to\RR_{\geq 0}$ which satisfy that their integrals in every orbit for the weighted action of $\Gm(\Fv)$ on $\Fvnz$ are equal to~$1$. 
These functions will enable us to use \ref{simint}.

For every $\vMFz$, let $$\DD^\aaa_v:=\{\yyy\in(\Ov)^n|\exists j:v(y_j)< a_j\}=(\Ov)^n-\piv^{a_1}\Ov\times\cdots\times\piv^{a_n}\Ov $$
and set $$\kav:=\frac{\mathbf 1 _{\DD^{\aaa}_v,F^n_{v}-\{0\}}}{1-\pivv}.$$
\begin{lem}\label{simchar}Let $\vMFz$ and let $\xxx\in\Fvnz$. The function~$\kav$ is compactly supported, locally constant and one has that $$\int _{\Fvt}\kav(t\cdot\xxx)d^*t=1.$$
\end{lem}
\begin{proof}
We have seen in \ref{begdav} that the subset $\DD^{\aaa}_v$ is open, closed and compact subset of $F^n_v-\{0\}$. We conclude that $\mathbf1_{\Dav}$ is locally constant and compactly supported, hence is such $\kav$.
In \ref{davdavdav} we have defined $r_v:\Fvnz\to\ZZ$ by $$r_v(\yyy)=\sup_{\substack{j=1\doots n\\y_j\neq 0}}\bigg\lceil-\frac{v(y_j)}{a_j}\bigg\rceil.$$ One has that$$t\cdot\xxx\in\DD^\aaa_v\iff \xxx\in t^{-1}\cdot\Dav\iff\xxx\in t^{-1}\piv^{r_v(\xxx)}\cdot(\piv^{-r_v(\xxx)}\cdot\Dav).$$ Lemma \ref{davdavdav} gives that $\xxx\in\piv^{-r_v(\xxx)}\cdot\Dav$ and that if $(u\cdot(\piv^{-r_v(\xxx)}\cdot\Dav))\cap (\piv^{-r_v(\xxx)}\cdot\Dav)\neq\emptyset$ then $v(u)=0.$ We conclude $v(t^{-1}\piv^{r_v(\xxx)})=0$ and hence $$\{t|t\cdot\xxx\in\Dav\}=\{t|v(t)=r_v(\xxx)\}.$$ We deduce
\begin{equation*}\label{olip}d^*t\{t\in F_v^{\times}|t\cdot\xxx\in\DD^\aaa_v\}=d^*t(\piv^{r_v(\xxx)}\Ovt)=1-\pivv.
\end{equation*}
One has further that\begin{equation*}\int _{\Fvt}\kav (t\cdot\xxx)d^*t=\int_{\Fvt}\frac{\jed_{\Dav}(t\cdot\xxx)}{1-\pivv}d^*t=\frac{d^*t(\{t\in F_v^{\times}|t\cdot\xxx\in\DD^{\aaa}_v\})}{1-\pivv}=1.\end{equation*}
The claim is proven.
\end{proof}
The following auxiliary lemma will be used in the definition of $\kav$ for~$v$ infinite.
\begin{lem}\label{kwew}
Let $a_{G_1}$ and $a_{G_2}$ be continuous actions of topological groups $G_1$ and $G_2$ on a topological space $X.$ Suppose the actions are permutable, that is for every $g_1\in G_1,$ every $g_2\in G_2$ and every $x\in X$ one has $g_1g_2 x=g_2g_1x$.
\begin{enumerate}
\item The map \begin{align*}a_{G_1\times G_2}:G_1\times G_2\times X&\to X\\ (g_1,g_2,x)&\mapsto g_1(g_2 x)\end{align*} defines a continuous action of $G_1\times G_2$ on~$X$.
\item There exists a continuous action of $G_2$ on $X/G_1$ such that $$g_2\cdot [x]_{G_1}=[g_2x]_{G_1}$$ for every $x\in X$, where $[x]_{G_1}$ is the image of~$x$ in $X/G_1$ for the quotient map.
\item The canonical map $X/G_1\to X/(G_1\times G_2)$ is open, continuous, surjective and $G_2$-invariant. The induced map $(X/G_1)/G_2\to X/(G_1\times G_2)$ is a homeomorphism.
\end{enumerate}
\end{lem}
\begin{proof}
\begin{enumerate}
\item The map $a_{G_1\times G_2}$ factorizes as $a_{G_2}\circ(\Id_{G_1}\times a_{G_1})$ thus is continuous. If $e_{G_1},e_{G_2}$ are neutral elements of $G_1$ and $G_2$, respectively, by definition one has $(e_{G_1},e_{G_2})\cdot x=e_{G_1}e_{G_2}x=x.$ Moreover, if $(g_1,g_2)$ and $(g_1',g_2')$ are elements of $G_1\times G_2$, then for $x\in X$ one has\begin{multline*}(g_1g_1',g_2g_2')x=(g_1g_1')(g_2g_2' x)=g_1( (g_1'(g_2 g_2'x)))=g_1(g_2 (g_1' g_2'x))\\=(g_1,g_2)(g_1',g_2')x.\end{multline*} We deduce that $a_{G_1\times G_2}$ is a continuous action of $G_1\times G_2$ on~$X$.
\item This is proven in \cite[Remark to Proposition 11, \no 4, \S4, Chapter III]{TopologieGj}.
\item The map $X/G_1\to X/(G_1\times G_2)$ is the induced map from $G_1$-invariant continuous, open and surjective map $X\to X/(G_1\times G_2)$, thus itself is continuous, open and surjective. If $g_2\in G_2$ and $x\in X$, one has that $$[g_2\cdot[x]_{G_1}]_{G_1\times G_2}=[[g_2x]_{G_1}]_{G_1\times G_2}=[g_2x]_{G_1\times G_2}=[x]_{G_1\times G_2},$$ where $[\cdot]_{G_1\times G_2}$ is the image in $X/(G_1\times G_2)$. Hence, $X/G_1\to X/(G_1\times G_2)$ is continuous. Suppose now elements $[x]_{G_1}$ and $[y]_{G_1}$ have the same image in $X/(G_1\times G_2)$. This means precisely that there exists $(g_1,g_2)\in G_1\times G_2$ such that $(g_1,g_2)\cdot x=g_1g_2x=y$. We deduce that $$[y]_{G_1}=[g_2g_1x]_{G_1}=[g_2x]_{G_1}=g_2[x]_{G_1},$$ i.e. $[x]_{G_1}$ and $[y]_{G_2}$ are in the same orbit of $G_2$. We deduce that the induced map $((X/G_1)/G_2)\to (X/(G_1\times G_2))$ is bijective. It follows that it is a homeomorphism.
\end{enumerate}
\end{proof}
\begin{lem}\label{qki}
Let $v\in M_F^\infty$. There exists a continuous compactly supported function $\kav:F^n_v-\{0\}\to\RR_{\geq 0}$, which is $F_{v,1}$-invariant and such that for every $\xxx\in\Fvnz$ one has that $$\int_{\RR_{>0}}\kav(\rho_v(t)\cdot \xxx)d^*t=\frac{1}{\lambda_{v,1}(F_{v,1})},$$where $\rho_v:\RR_{>0}\to\Fvt$ is given by $\rho_v:r\mapsto r^{1/n_v}$. Such function satisfies furthermore for every $\xxx\in\Fvnz$  that $$\int _{\Fvt}\kav (t\cdot\xxx)d^*t={1}.$$
\end{lem} 
\begin{proof}
The isomorphism (\ref{splitfv}) $$\widetilde \rho_v:\RR_{>0}\times F_{v,1}\xrightarrow{\sim}\Fvt\hspace{1cm}(r,x)\mapsto \rho_v(r)x$$ satisfies $(\widetilde\rho_v)_*(d^*r\times \lambda_{v,1})=d^*x_v$ by \ref{cyv}. It induces a topological action of $\RR_{>0}\times F_{v,1}$ on $\Fvnz$. Moreover, as $F_{v,1}=\{1\}\times F_{v,1}$ is a closed subgroup of $\RR_{>0}\times F_{v,1},$ by \cite[Example 1, \no 1, \S 4, Chapter III]{TopologieGj} the restriction of this action to $F_{v,1}$ is continuous and proper. Let $q_{v,1}:\Fvnz\to(\Fvnz)/F_{v,1}$ be the corresponding quotient map, by \cite[Proposition 2, \no 1, \S4, Chapter III]{TopologieGj}, the map $q_{v,1}$ is proper. Let $\RR_{>0}$ acts on $\Fvnz$ via the identification $\RR_{>0}=\RR_{>0}\times \{1\}$. Lemma \ref{kwew} provides an action of $\RR_{>0}$ on $(\Fvnz)/F_{v,1}$ which satisfies $$r\cdot q_{v,1}(\xxx)=q_{v,1}(r\cdot\xxx)=q_{v,1}(\rho_v(r)\cdot\xxx),$$ for $r\in\RR_{>0}$ and $\xxx\in\Fvnz$. Moreover, the canonical map $(\Fvnz)/F_{v,1}\to [\PPP(\aaa)(\Fv)]$ induces an identification $((\Fvnz)/F_{v,1})/\RR_{>0}\xrightarrow{\sim}[\PPP(\aaa)(\Fv)].$ As $[\PPP(\aaa)(F_v)]$ is compact, hence paracompact, Proposition \ref{jeiui} gives that there exists a continuous compactly supported function $k':(\Fvnz)/F_{v,1}\to\RR_{\geq 0}$ such that for every $\yyy\in(\Fvnz)/F_{v,1}$, we have that $$\int _{\RR_{>0}}k'(r\cdot\yyy)d^*r=\frac{1}{\lambda_{v,1}(F_{v,1})}.$$ Let us set $\kav=k'\circ q_{v,1}$. As $q_{v,1}$ is proper and $F_{v,1}$-invariant, the map $\kav$ is compactly supported and $F_{v,1}$-invariant. Let $\xxx\in\Fvnz$, we have that
\begin{align*}
\int_{\RR_{>0}}\kav(\rho_v(r)\cdot \xxx)d^*r&=\int_{\RR_{>0}}\kav(r\cdot\xxx)d^*r\\
&=\int_{\RR_{>0}}k'(q_{v,1}(r\cdot\xxx))d^*r\\
&=\int_{\RR_{>0}}k'(r\cdot q_{v,1}(\xxx))d^*r\\
&=\frac{1}{\lambda_{v,1}(F_{v,1})},
\end{align*}
and that
 \begin{align*}\int _{\Fvt}\kav(t\cdot\xxx)d^*t&=\int _{\RR_{>0}\times F_{v,1}}\kav((r,z)\cdot\xxx){d^*r}d\lambda_{v,1}(z)\\&=\int _{F_{v,1}}d\lambda_{v,1}(z)\int _{\RR_{>0}}\kav(r\cdot(z_j^{a_j}x_{j})_j)d^*r\\
&=\int _{F_{v,1}}\frac{1}{\lambda_{v,1}(F_{v,1})}d\lambda_{v,1}\\
&=1. 
\end{align*}
The statement is proven.
\end{proof}
\subsection{} The goal of this paragraph is to give several equivalent conditions that make $f^{-1}dx_1\dots dx_n$ a measure on $\Fvnz$, where~$f$ is an $\aaa$-homogenous function of weighted degree $|\aaa|$.

Let $v\in M_F.$ We will use the following conventions $a\cdot\infty=\infty$ for $a\in\RR_{>0}$ and $\infty^{-1}=0$. If $t\in\Gm(\Fv)$, we will see it as a function $t:\Fvnz\to\Fvnz$ given by its action on $\Fvnz$. We say that a function $f:F_v^n-\{0\}\to \CC\cup\{\infty\}$  is $\aaa$-homogenous of weighted degree $|\aaa|$ if $f(t\cdot\xxx)=|t|^{|\aaa|}_vf(\xxx)$ for every $t\in\Fvt$ and every $\xxx\in\Fvnz$.  

\begin{lem}\label{xiub}
Let $f:\Fvnz\to\CC\cup\{\infty\}$ be an $\aaa$-homogenous function of weighted degree $|\aaa|$ such that $$dx_1\dots dx_n(\{\xxx\in\Fvnz| f(\xxx)=0\})=0.$$ Let $\phi\in\mathscr C^0_c(\Fvnz,\CC)$. One has that $\phi f^{-1}\in L^1(\Fvnz, dx_1\dots dx_n)$ if and only if for every $t\in\Fvt$ one has that $(\phi\circ t^{-1})f^{-1}\in L^1(\Fvnz,dx_1\dots dx_n)$ and if $\phi f^{-1}\in L^1(\Fvnz,dx_1\dots dx_n)$ then $$\int_{\Fvnz}\phi f^{-1}dx_1\dots dx_n=\int_{\Fvnz}(\phi\circ t^{-1})f^{-1}dx_1\dots dx_n$$ for every $t\in\Fvt$.
\end{lem}
\begin{proof}
Let $t\in\Gm(F_v)$. The action of~$t$ on $F_v^n-\{0\}$ is given by the multiplication to the left by the diagonal matrix $D_t$ which has the diagonal vector $(t^{a_1}\doots t^{a_n})$. By using the formula for the change of the coordinates, we get that \begin{align*}
\int _{\Fvnz}(\phi \circ t^{-1})f^{-1}dx_{1}\dots dx_{n}\hskip-3,4cm&\\&=\int _{\Fvnz} |\det(\Jac (D_t))|_v(\phi \circ t^{-1})(t\cdot \xxx)f_v(t\cdot\xxx)^{-1}dx_{1}\dots dx_{n}\\
&=\int _{\Fvnz} |t^{|\aaa|}|_v\phi(\xxx)\cdot |t|_v^{-|\aaa|}f(\xxx)^{-1}dx_{1}\dots dx_{n}\\
&=\int _{\Fvnz}\phi f^{-1}dx_{1}\dots dx_{n},
\end{align*}
if one, and hence every, integral converges absolutely. The statement follows.
\end{proof}
The following lemma will be needed to treat the case $\vMF$.
\begin{lem}\label{calcintinfv} Suppose $\vMFi$. Let $\rho_v:\RR_{>0}\to\Fvt$ be the map~$r\mapsto r^{1/n_v}$.
\begin{enumerate}
\item Consider the map $A:(\Fvt)^n\to(\Fvt)^n$ given by $$A:\xxx\mapsto ((\rho_v(|x_n|_v^{a_j/a_n})x_j)_{j=1}^{n-1},x_n).$$
One has that $$|x_n|^{\frac{|\aaa|}{a_n}-1}A_*(dx_1\dots dx_n)=dx_1\dots dx_n.$$
\item Let $f:\Fvnz\to\CC\cup\{\infty\}$ be an $\aaa$-homogenous function of weighted degree $|\aaa|$ such that $$dx_1\dots dx_n(\{\xxx\in\Fvnz| f(\xxx)=0\})=0.$$ Let $\phi:\Fvnz\to\CC$ be a compactly supported continuous function. One has that $\phi f^{-1}\in L^1(\Fvnz,dx_1\dots dx_n)$ if and only if $f^{-1}|_{(\Fvt)^{n-1}\times F_{v,1}}\in L^1((\Fvt)^{n-1}\times F_{v,1},dx_1\dots dx_n\times\lambda_{v,1})$ and if $\phi f^{-1}\in L(\Fvnz,dx_1\dots dx_n)$ then \begin{multline}
\int_{\Fvnz}\phi f^{-1}dx_1\dots dx_n=a_n\int_{F_{v,1}}d\lambda_{v,1}(z)\int _{\RR_{>0}}{\phi\big(\rho_v(t)\cdot ((x_{j})_{j=1}^{n-1},z)\big)d^*t}\times\\ \times \int_{(\Fvt)^{n-1}}f(x_1\doots x_{n-1},z)^{-1}dx_1\dots dx_{n-1}.
\end{multline}
\end{enumerate}
\end{lem}
\begin{proof}
\begin{enumerate}
\item The Jacobian of~$A$ is given by the diagonal matrix having for the diagonal vector $((\rho_v(|x_{n}|^{a_j/a_n}_v))_{j=1}^{n-1},1) .$ The~$v$-adic absolute value of the determinant of the Jacobian is equal to $$\big|\prod_{j=1}^{n-1}\rho_v(|x_{n}|_v^{a_j/a_n})\big|_v=\prod _{j=1}^{n-1}\big|\rho_v(|x_{n}|_v^{a_j/a_n})\big|_v=\prod_{j=1}^{n-1}|x_n|_v^{a_j/a_n}=|x_n|_v^{\frac{|\aaa|}{a_n}-1}.$$
Thus, by the formula for the change of variables we get that $$|x_n|_v^{\frac{|\aaa|}{a_n}-1}A_*(dx_1\dots dx_n)=dx_1\dots dx_n.$$
\item  Let us define $B,C:\Fvtn\to\Fvtn$ by \begin{align*}B:\xxx&\mapsto (x_j\rho_v({|x_{n}|_v^{{-a_j}/{a_n}}}))_j,\\
C:\xxx&\mapsto ((x_j)_{j=1}^{n-1},x_n\rho_v(|x_n|_v)^{-1}).\end{align*} 
Note that for $\xxx\in\Fvtn$, one has that $$B(A(\xxx))=B((x_j\rho_v(|x_n|_v^{a_j/a_n}))_{j=1}^{n-1},x_n)=((x_j)_{j=1}^{n-1},x_n\rho_v(|x_n|^{-1}))=C(\xxx)$$and that \begin{multline*}A(\xxx)=((\rho_v(|x_n|_v^{\frac{a_j}{a_n}})x_j)_{j=1}^{n-1},x_n)=\rho_v(|x_n|_v^{\frac{1}{a_j}})\cdot ((x_j)_{j=1}^{n-1},x_n\rho_v(|x_n|_v)^{-1})\\=\rho_v(|x_n|_v^{\frac{1}{a_j}})\cdot C(\xxx). \end{multline*} 
For every $\xxx\in\Fvnz$, we have that$$f_v(\xxx)=|x_{n}|_v^{|\aaa|/a_n}f_v\big(\big({x_{j}}\rho_v({|x_{n}|_v^{{-a_j}/{a_n}}})\big)_j\big)=f_v(B(\xxx)),$$
by the fact that $f_v$ is $\aaa$-homogenous of weighted degree $|\aaa|$.  
 Let~$\phi\in \mathscr C^0_c(\Fvnz,\CC)$. By using that $dx_1\dots dx_n(\Fvnz-\Fvtn)=0,$ because $\Fvnz-\Fvtn$ is contained in a finite union of hyperplanes of $(\Fv)^n$, and the part (1), we get that: \begin{align*}\begin{split}\label{katrlj}\int _{\Fvnz}&\phi f_v^{-1}dx_1\dots dx_n\\&=\int _{\Fvtn}\phi f_v^{-1}dx_1\dots dx_n\\&=\int_{\Fvtn}|x_n|_v^{-|\aaa|/a_n}\phi\cdot (f_v^{-1}\circ B)dx_1\dots dx_n\\&=\int_{\Fvtn}|x_n|_v^{-1}\phi\cdot (f_v^{-1}\circ B)A_*(dx_1\dots dx_n)\\
&=\int_{\Fvtn}|x_n\circ A|_v^{-1}(\phi\circ A)(f_v^{-1}\circ (B\circ A))dx_1\dots dx_n\\
&=\int_{\Fvtn}|x_n|_v^{-1}(\phi\circ A)f_v((x_j)_{j=1}^{n-1},x_n\rho_v(|x_n|_v^{-1}))^{-1}dx_1\dots dx_n\\
&=\int_{\Fvtn}(\phi\circ A)f_v((x_j)_{j=1}^{n-1},x_n\rho_v(|x_n|_v^{-1}))^{-1}dx_1\dots dx_{n-1}d^*x_n\\
&=\int_{\Fvtn}\phi(\rho_v(|x_n|_v^{1/a_n})\cdot C(\xxx))f_v(C(\xxx))^{-1}dx_1\dots dx_{n-1}d^*x_n.
\end{split}
\end{align*}
 if one (and hence every) integral converges absolutely.
It follows from \ref{cyv} that the map $$\widetilde\rho_v^{-1}:\Fv^{\times}\to\RR_{>0}\times F_{v,1}\hspace{1cm}x\mapsto (|x|_v,{x}{\rho_v(|x|_v)^{-1}})$$ satisfies that$(\widetilde\rho_v^{-1})_*d^*x=d^*r\times \lambda_{v,1}$.
 The last integral is hence equal to \begin{multline*}
\int_{F_{v,1}\times\RR_{>0}}\int_{(\Fvt)^{n-1}}\phi(\rho_v(r^{1/a_j})\cdot ((x_j)_{j=1}^{n-1},z))\times\\\times f_v((x_j)_{j=1}^{n-1},z)^{-1}dx_1.. dx_{n-1}d^*rd\lambda_{v,1}(z).
\end{multline*}
Note that setting $u^{a_n}=r$ gives $d^*r=a_nd^*u$ and thus the last integral becomes 
\begin{equation}\label{vuqwq}
a_n\int_{F_{v,1}\times\RR_{>0}}\int_{(\Fvt)^{n-1}}\frac{\phi(\rho_v(u)\cdot ((x_j)_{j=1}^{n-1},z))}{f_v((x_j)_{j=1}^{n-1},z)}dx_1\dots dx_{n-1}d^*ud\lambda_{v,1}(z).
\end{equation}
By Fubini theorem, we get that (\ref{vuqwq}) is equal to
\begin{multline}\label{vhunw}
a_n\int_{F_{v,1}}d\lambda_{v,1}(z)\int_{\RR_{>0}}\phi(\rho_v(u\cdot ((x_j)_{j=1}^{n-1},z)))d^*u\times\\\times \int_{(\Fvt)^{n-1}}f_v((x_j)_{j=1}^{n-1},z)^{-1}dx_1\dots dx_{n-1}.
\end{multline}
%
 %
 For every $((x_j)_j,z)\in(\Fvt)^{n-1}\times F_{v,1}$, the map $\RR_{>0}\to\CC$ given by $u\mapsto\phi(\rho_v(u)\cdot((x_j)_j,z))$ is compactly supported, because $\phi$ is compactly supported, the map $y\mapsto y\cdot ((x_j)_j,z)$ is proper (the action of $\Fvt$ on $\Fvtn$ is proper by \ref{aboutopa}, thus by \cite[Proposition 4, \no 2, \S 4, Chapter III]{TopologieGj}, the map $y\mapsto y\cdot ((x_j)_j,z)$ is proper) and $\rho_v:r\mapsto r^{1/n_v}$ is proper.
It follows that $$\int_{\RR_{>0}}\phi(\rho_v(u)\cdot((x_j)_j,z))d^*u$$converges absolutely. It follows that the integral (\ref{vhunw}) converges absolutely if and only if $\int_{F_{v,1}}d\lambda_{v,1}(z)\int_{\Fvtn}f_v((x_j)_{j=1}^{n-1},z)^{-1}dx_1\dots dx_{n-1}$  converges absolutely. The statement follows.
\end{enumerate}
\end{proof}
In the following proposition, we give equivalent conditions to the condition that $f^{-1}dx_1\dots dx_n$ is a measure on $\Fvnz$.
\begin{prop}\label{oilp}
Let $f:\Fvnz\to\CC\cup\{\infty\}$ be an $\aaa$-homogenous function of weighted degree $|\aaa|$ such that $$dx_1\dots dx_n(\{\xxx\in\Fvnz| f(\xxx)=0\})=0.$$ 
The following are equivalent: \begin{enumerate} 
\item For every $\phi \in\mathscr C^0_c(\Fvnz,\CC)$ one has that $\phi f^{-1}\in L^1(\Fvnz,dx_1\dots dx_n)$ and that $$\mathscr C^0_c(\Fvnz,\CC)\to\CC\hspace{1cm}\phi\mapsto\int_{\Fvnz}\phi \cdot f^{-1}dx_1\dots dx_n$$ is a measure on $\Fvnz$.
\item  For every compactly supported function $\phi :\Fvnz\to\CC$ one has that $\phi f^{-1}\in L^1(\Fvnz,dx_1\dots dx_n)$ and that there exists a unique measure $\omega_v$ on $[\PPP(\aaa)(F_v)]$ such that  \begin{equation}\label{vuwqi}\int _{\Fvnz}\phi f^{-1}dx_1\dots dx_n= \int _{[\PPP(\aaa)(F_v)]}d\omega_v(\yyy)\int _{\Fvt}\phi(t\cdot\widetilde{\yyy})d^*t,\end{equation}where $\widetilde\yyy$ is a lift on an element $\yyy\in[\PPP(\aaa)(\Fv)]$.
\item One has that  $\kav\cdot f^{-1}\in L^1(\Fvnz,dx_1\dots dx_n),$ where the function $\kav$ is defined in $\ref{kave}$.
\item If $v\in M_F^0$, one has that $f^{-1}|_{\Dav}\in L^1(\Dav,dx_1\dots dx_n)$. If $v\in M_F^\infty$, one has that $f^{-1}|_{(\Fvt)^{n-1}\times F_{v,1}}\in L^1((\Fvt)^{n-1}\times F_{v,1}, dx_1\dots dx_{n-1}\times\lambda_{v,1})$.
\end{enumerate}
If any of the conditions is satisfied, one has that \begin{equation}\label{kimj}\omega _v([\PPP(\aaa)(\Fv)])=\int _{\Fvnz}\kav f^{-1}dx_{1}\dots dx_n \end{equation}
Moreover, if $\vMFz$, both quantities in (\ref{kimj}) are equal to $$\frac{1}{1-\pivv}\int _{\Dav}f^{-1}dx_1\dots dx_n $$ and if $\vMFi$, and are equal to 
$$\frac{a_n}{\lambda_{v,1}(\Fvj)}\int_{(\Fvt)^{n-1}\times F_{v,1}}f^{-1}dx_{1}\dots dx_{n-1}\times \lambda_{v,1}. $$
\end{prop} 
\begin{proof}
Note that the implication $(1)\implies (2)$ follows from \ref{intbour}. 
By using the fact that $[\PPP(\aaa)(\Fv)]$ is compact, hence paracompact, and Proposition \ref{simint}, we get the implication $(2)\implies (3)$ and that if the condition (2) is satisfied, then $$\omega_v^{\aaa}([\PPP(\aaa)(\Fv)])=\int _{\Fvnz}\kav f^{-1}dx_1\dots dx_n.$$

$(3)\implies (4)$.  Suppose~$v$ is a finite place. As $\kav f^{-1}=\frac{\mathbf 1_{\Dav}\cdot f^{-1}}{1-\pivv}$, we deduce that the restriction of $\xxx\mapsto f(\xxx)^{-1}$ to $\Dav$ is $dx_1\dots dx_n$-absolutely integrable and that $$\int _{\Fvnz}\kav f^{-1}dx_1\dots dx_n=\frac{1}{1-\pivv}\int_{\Dav}f^{-1}dx_1\dots dx_n.$$ Suppose~$v$ is infinite. By Lemma \ref{calcintinfv} one has that $$f^{-1}|_{(\Fvt)^{n-1}\times F_{v,1}}\in L^1((\Fvt)^{n-1}\times F_{v,1},dx_1\dots dx_{n-1}\times \lambda_{v,1})$$ and that \begin{align*}\int_{\Fvnz}\kav f^{-1} dx_1\dots dx_n\hskip-2cm&\\&=a_n\int_{F_{v,1}}d\lambda_{v,1}(z)\int_{\RR_{>0}}\kav(\rho_v(u)\cdot((x_j)_j,z))d^*u\times\\&\quad\quad\quad\quad\quad\times\int_{(\Fvt)^{n-1}}f(x_1\doots x_{n-1},z)^{-1}dx_1\dots dx_{n-1}.\end{align*}
We have by \ref{calcintinfv} that $$\int_{\RR_{>0}}\kav(\rho_v(u)\cdot ((x_j)_j,z))d^*u=\frac{1}{\lambda_{v,1}(F_{v,1})},$$ for every $((x_j)_j,z)\in (\Fvt)^{n-1}\times F_{v,1}$. It follows that $$\int_{\Fvnz}\kav f^{-1}dx_1\dots dx_n=\frac{a_n}{\lambda_{v,1}(F_{v,1})}\int_{(\Fvt)^{n-1}\times F_{v,1}}f^{-1}dx_1\dots dx_{n-1}\times \lambda_{v,1}. $$
The implication is proven.

$(4)\implies (1).$ If $\phi:\Fvnz\to\CC$ is compactly supported, we set $||\phi||_{\sup}:=\sup_{\xxx\in\Fvnz}|\phi(\xxx)|.$ 
Suppose firstly $\vMFz$ and let $\phi\in \mathscr C^0_c(\Fvnz,\CC)$. It follows from \ref{davdavdav} that $\cup_{t\in\Gm(\Fv)}(t\cdot\Dav)=\Fvnz$. As $\supp(\phi)$ is compact, there exists a finite set $\{t_1\doots t_m\}$ such that $(\cup_{i=1}^m(t_i\cdot\Dav))\supset \supp(\phi)$. For $\xxx\in\Dav$ and $t\in\Gm(\Fv)$ one has that $f(t\cdot\xxx)^{-1}=|t|_v^{-|\aaa|}f(\xxx)^{-1}$. Thus from the fact that $f^{-1}|_{\Dav}\in L^1(\Dav,dx_1\dots dx_n)$, it follows that $f^{-1}|_{t\cdot\Dav}\in L^1(t\cdot\Dav,dx_1\dots dx_n)$ and hence that $$f^{-1}|_{\cup_{i=1}^m(t_i\cdot\Dav)}\in L^1(\cup_{i=1}^m(t_i\cdot\Dav), dx_1\dots dx_n).$$ We deduce that $\phi f^{-1}\in L^1(\Fvnz,dx_1\dots dx_n)$ and that \begin{align*}\bigg|\int_{\Fvnz}\phi f^{-1}dx_1\dots dx_n\bigg|&\leq \int_{\sup(\phi)}||\phi||_{\sup} f^{-1}dx_1\dots dx_n\\
&\leq ||\phi||_{\sup}\int_{(\cup_{i=1}^m(t_i\cdot\Dav))}f^{-1}dx_1\dots dx_n. \end{align*}
It follows furthermore that $\phi\mapsto \int_{\Fvnz}\phi f^{-1}dx_1\dots dx_n$ is a bounded, hence a continuous operator, i.e. $f^{-1}dx_1\dots dx_n$ is a measure on $\Fvnz$. 

Suppose now $\vMFi.$ For any $\phi\in\mathscr C^0_c(\Fvnz,\CC),$ one has by Lemma \ref{calcintinfv} that
$\phi f^{-1}\in L^1(\Fvnz,dx_1\dots dx_n).$ 
We will use \cite[Proposition 6 in \no 3, \S 1, Chapter III ]{Integrationd} to prove the continuity of the operator $\phi\mapsto \int _{\Fvnz}\phi f^{-1}dx_{1}\dots dx_{n}.$ By the mentioned proposition, it suffices to pick a sequence of compacts $\{K_{\alpha}\}_{\alpha}$, the interiors of which cover $\Fvnz$ and to establish that for every $\alpha$ there exists $M_{\alpha}>0$ such that $$\bigg|\int _{\Fvnz}\phi f^{-1}dx_{1}\dots dx_{n}\bigg|\leq M_\alpha ||{\phi}||_{\sup}$$ for $\phi\in\mathscr C^0_c(\Fvnz,\CC)$ with $\supp(\phi)\subset K_{\alpha}$. 
We set $$K_{\alpha}:=\{\xxx\in\Fvnz|\forall j:\hspace{0.1cm} \alpha^{-1}\leq |x_{j}|_v\leq \alpha\}\hspace{1cm}\text{for } \alpha \in{\ZZ_{\geq 2}}.$$ Let $\xi _{\alpha}\in\mathscr C^0_c(\Fvnz,\RR_{\geq 0})$ such that $\xi_{\alpha}(\xxx)\geq 1$ for every $\xxx\in K_{\alpha}$. 
For every $\xxx\in\Fvnz$, the map $\RR_{>0}\to\RR_{\geq 0}$ given by $r\mapsto\xi_{\alpha}(\rho_v(r)\cdot\xxx)$ is compactly supported, because $\xi_{\alpha}$ is compactly supported, the map $y\mapsto y\cdot\xxx$ is proper (the action of $\Fvt$ on $\Fvtn$ is proper by \ref{aboutopa}, thus by \cite[Proposition 4, \no 2, \S 4, Chapter III]{TopologieGj}, the map $y\mapsto y\cdot\xxx$ is proper) and $\rho_v:r\mapsto r^{1/n_v}$ is proper.  Let $\phi\in\mathscr C^0_c(\Fvnz,\CC)$ with $\supp(\phi)\subset K_{\alpha}$.  We have that \begin{align*}\bigg|\int _{\RR_{>0}}{\phi(\rho_v(r)\cdot ((u_{j})_{j=1}^{n-1},z))d^*r}\bigg|\hskip-1cm&\\ &\leq\int _{\RR_{>0}}{||\phi||_{\sup}\cdot\xi_{\alpha}(\rho_v(r)\cdot ((u_{j})_{j=1}^{n-1},z))d^*r}\\&\leq ||\phi||_{\sup}\int _{\RR_{>0}}\xi_{\alpha} (\rho_v(t)\cdot ((u_{j})_{j=1}^{n-1},z))d^*t
\\&\leq ||\phi||_{\sup}||\xi||_{\sup}\int_{\supp(r\mapsto\xi_{\alpha}(\rho_v(r)\cdot\xxx))}1d^*t\\
&=||\phi||_{\sup}||\xi||_{\sup}d^*t(\supp(r\mapsto\xi_{\alpha}(\rho_v(r)\cdot\xxx))).
\end{align*}  
We deduce from (\ref{calcintinfv}) that \begin{align*}\int_{\Fvnz}\phi f^{-1}dx_1\dots dx_n\hskip-3cm&\\&=a_n\int_{F_{v,1}}d\lambda_{v,1}(z)\int _{\RR_{>0}}{\phi\big(\rho_v(t)\cdot ((x_{j})_{j=1}^{n-1},z)\big)d^*t}\times\\ &\quad\quad\quad\quad\times \int_{(\Fvt)^{n-1}}f(x_1\doots x_{n-1},z)^{-1}dx_1\dots dx_{n-1}\\&\leq||\psi||_{\sup}||\xi||_{\sup}d^*t(\supp(\eta))\int_{(\Fvt)^{n-1}\times F_{v,1}}f^{-1}dx_1\dots dx_{n-1} \lambda_{v,1}  .\end{align*}
By \cite[Proposition 6 in \no 3, \S 1, Chapter III]{Integrationj} the operator $$\mathscr C^0_c(\Fvnz,\CC)\to \CC\hspace{1cm}\phi\mapsto\int _{\Fvnz}\phi f^{-1}dx_{1}\dots dx_{n}$$ is continuous. We conclude that $f^{-1}dx_{1}\dots dx_n$ is a measure.  The statement is proven.
\end{proof}
When~$f$ satisfies the equivalent conditions of \ref{oilp}, the measure $f^{-1}dx_1\dots dx_n$ is $\Fvt$-invariant by \ref{xiub}.
\begin{exam}\label{exkojsc}
\normalfont
Suppose~$f$ is continuous and $f(\xxx)\in\CC-\{0\}$ for every $\xxx\in\Fvnz$. From the fact that a product of a measure and a continuous  function is a measure \cite[\no 4, \S1, Chapter III]{Integrationj}, it follows that $f^{-1}dx_1\dots dx_n$ is a measure and hence~$f$ satisfies the equivalent conditions of Proposition \ref{oilp}.
\end{exam}
\subsection{} Proposition \ref{oilp} provides a measure $\omega_v$ on $[\PPP(\aaa)(\Fv)]$. 
\begin{mydef}\label{yumf}
Let $v\in M_F$ and let $f_v:\Fvnz\to\CC\cup\{\infty\}$ be an $\aaa$-homogenous function of weighted degree $|\aaa|$  such that $$dx_1\dots dx_n(\{\xxx|f_v(\xxx)=0\})=0$$ and such that $f_v^{-1}dx_1\dots dx_n$ is a measure on $\Fvnz$. We define $\omega _v$ to be the quotient measure $$\omega_v:=(f_v^{-1}dx_{1}\dots dx_{n})/d^*x$$ on $[\PPP(\aaa)(F_v)].$
\end{mydef}
Recall from \ref{pafvtafv} that $[\TTa(\Fv)]$ is the open subset of $[\PPP(\aaa)(\Fv)]$ given by the image of $q^{\aaa}_v(\Fvtn)$. We prove that the complement of the open subset $[\TTa(\Fv)]\subset[\PPP(\aaa)(\Fv)]$ is $\omega_v$-negligible.
\begin{lem}\label{ttadmz}Let $f_v$ be as in \ref{yumf}.
One has that $$\omega_v([\PPP(\aaa)(\Fv)]-[\TTa(\Fv)])=0.$$
\end{lem}
\begin{proof}
The preimage $(q^{\aaa}_v)^{-1}([\PPP(\aaa)(\Fv)]-[\TTa(\Fv)])$ is $dx_1\dots dx_n$-negligible. Now, \cite[Proposition 6, \no 3, \S 2, Chapter II]{Integrationj} gives that $\omega_v([\PPP(\aaa)(\Fv)]-[\TTa(\Fv)])=0$.
\end{proof}
We explain how to do the integration against $\omega_v$. 
\begin{lem}\label{kilp}
Let $\vMF$. Let $f_v$ be as in \ref{yumf}. Let $h:[\PPP(\aaa)(\Fv)]\to\CC$ be a function. Suppose $\vMFz$.  Then $h\in L^1([\PPP(\aaa)(\Fv)],\omega_v)$ if and only if $(h\circ\qav)\cdot f_v^{-1}|_{\Dav}\in L^1(\Dav,dx_1\dots dx_n)$ and if $h\in L^1([\PPP(\aaa)(\Fv)],\omega_v)$, then $$\int _{[\PPP(\aaa)(\Fv)]}h\omega_v=\frac{1}{1-\pivv}\int_{\Dav}(h\circ\qav) f_v^{-1}dx_{1}\dots dx_{n}.$$Suppose $\vMFi$. Then $h\in L^1([\PPP(\aaa)(\Fv)],\omega_v)$ if and only if $$(h\circ\qav)\cdot f_v^{-1}|_{(\Fvt)^{n-1}\times F_{v,1}}\in L^1((\Fvt)^{n-1}\times F_{v,1},dx_1\dots dx_{n-1}\times \lambda_{v,1})$$ and if $h\in L^1([\PPP(\aaa)(\Fv)],\omega_v)$, then $$\int _{[\PPP(\aaa)(\Fv)]}h\omega_v=\frac{a_n}{\lambda_{v,1}(\Fvj)}\int _{(\Fvt)^{n-1}\times\Fvj}(h\circ\qav)f_v^{-1}dx_1\dots dx_{n-1}\times \lambda_{v,1}. $$
\end{lem}
\begin{proof}
As $[\PPP(\aaa)(\Fv)]$ is compact and hence paracompact, by \ref{simint}, one has that $h\in L^1([\PPP(\aaa)(\Fv)],\omega_v)$ if and only if $\kav\cdot (h\circ\qav)\in L^1(\Fvnz,f_v^{-1}dx_1\dots dx_n)$, and if $h\in L^1([\PPP(\aaa)(\Fv)],\omega_v)$ then $$\int_{[\PPP(\aaa)(\Fv)]}h\omega_v=\int_{\Fvnz}\kav\cdot (h\circ\qav)f_v^{-1}dx_1\dots dx_n.$$ 
The function $\qav$ is $\Fvt$-invariant, hence is such $h^{-1}\circ \qav$. It follows that $(h^{-1}\circ\qav)\cdot f_v$ is an $\aaa$-homogenous function of weighted degree $|\aaa|.$ The set where it vanishes coincides with the set where $f_v$ vanishes, thus this set is $dx_1\dots dx_n$-negligible. We apply Proposition \ref{oilp} for the function $(h^{-1}\circ\qav)\cdot f_v$. It follows that $\kav\cdot((h\circ\qav)f_v^{-1})\in L^1(\Fvnz,dx_1\dots dx_n)$ 
if and only if $$(h\circ\qav)\cdot f_v^{-1}|_{\Dav}\in L^1(\Dav,dx_1\dots dx_n)$$ if~$v$ is finite, and if and only if $$(h\circ\qav)\cdot f_v^{-1}|_{(\Fvt)^{n-1}\times F_{v,1}}\in L^1((\Fvt)^{n-1}\times F_{v,1},dx_1\dots dx_{n-1}\times \lambda_{v,1})$$if~$v$ is infinite. Moreover, \ref{oilp} gives that if $$\kav\cdot(h\circ\qav) f_v^{-1}\in L^1(\Fvnz,dx_1\dots dx_n),$$ then $$\int_{\Fvnz}\kav(h\circ\qav) f_v^{-1}dx_1\dots dx_n=\frac{1}{1-\pivv}\int_{\Dav}(h\circ\qav) f_v^{-1}dx_{1}\dots dx_{n}.$$ if~$v$ is finite, and $$\frac{a_n}{\lambda_{v,1}(\Fvj)}\int _{(\Fvt)^{n-1}\times\Fvj}(h\circ\qav)f_v^{-1}dx_1\dots dx_{n-1}\times \lambda_{v,1}. $$ if~$v$ is infinite. The statement follows.
\end{proof}
\section{Peyre's constant} \label{trobuljak}
In this section from a quasi-toric family of $\aaa$-homogenous functions of weighted degree $|\aaa|$, we will define a measure on the product space $\prod_{\vMF}[\PPP(\aaa)(F_v)],$  and Peyre's constant of the stacks~$\PPP(\aaa)$ and~$\oPPa$.

\subsection{}
In this paragraph, we calculate the volume $\omega_v$ for $\vMFz$ from \ref{yumf} when the function $f_v$ is toric.

Let $\vMFz$. As in \ref{defrv}, we set $r_v:\Fvnz\to\RR_{>0}$ for the function $$r_v(\xxx)=\sup_{\substack{j=1\doots n\\x_j\neq 0}}\bigg\lceil-\frac{v(x_j)}{a_j}\bigg\rceil.$$
Let $f^{\#}_v:\Fvnz\to\RR_{>0}$ be the toric $\aaa$-homogenous function of weighted degree $|\aaa|$. Recall that this means $$f^{\#}_v(\xxx)=\pivv^{-|\aaa|r_v(\xxx)}.$$ In \ref{davdavdav}, we have established that $r_v|_{\Dav}=0$, thus $f^{\#}_v|_{\Dav}=1,$ where $\Dav=(\Ov)^n- (\piv^{a_1}\Ov\times\cdots\times\piv^{a_n}\Ov)$. 
For $v\in M_F^0$ and $s\in\CC$, we denote $$\zeta_v(s)=\frac{1}{1-\pivv^s}.$$
\begin{lem} \label{locomi}
Let $v\in M_F^0$ and let $\fvt:F_v^n-\{0\}\to\RR_{> 0}$ be the toric $\aaa$-homogenous function of weighted degree $|\aaa|$. Let $\omega^\#_v$ be the measure on $[\PPP(\aaa)(F_v)]$ that is given by $$((\fvt)^{-1}dx_{1}\dots dx_n)/d^*x.$$
 One has that $$\omega^\# _v([\PPP(\aaa)(F_v)])=\frac{\zeta_v(1)}{\zeta_v(|\aaa|)}. $$
\end{lem}
\begin{proof}
By applying \ref{kilp} and using the fact that $\fvt|_{\Dav}=1$, we get that \begin{equation*}\int _{[\PPP(\aaa)(F_v)]}1\omega^\#_v=\frac{1}{1-\pivv}\int _{\Dav}{(\fvt)^{-1}}dx_{1}\dots dx_{n}=\zeta_v(1)\int _{\DD^\aaa_v}dx_1\dots dx_{n}.
\end{equation*}
In turn one has that$$dx_{1}\dots dx_{n}(\DD^\aaa_v)=1-\prod _{j=1}^ndx(\pi_v^{a_j}\Ov)=1-\pivv^{|\aaa|}=\zeta_v(|\aaa|)^{-1},$$ and the claim follows.
\end{proof} 
\begin{rem}
\normalfont  We will generalize the calculation of \ref{locomi} in \ref{torfour},when will be calculating the Fourier transform of a local toric height at a non-archimedean place. 
\end{rem}
\subsection{} Let us calculate the volume $\omega_v([\PPP(\aaa)(\Fv)])$ when $\vMFi$ and the function $f_v=f_v^{\#}$ is the toric $\aaa$-homogenous function of weighted degree $|\aaa|$. 

Recall that the toric $\aaa$-homogenous function of weighted degree $|\aaa|$ is the function $$\fvt:\Fvnz\to\RR_{>0}\hspace{1cm}\xxx\mapsto\max_{j=1\doots n}(|x_j|_v^{1/a_j})^{|\aaa|}.$$ 
\begin{lem}\label{locinfomega}
Let $v\in M_F^\infty$ and let $\fvt:F_v^n-\{0\}\to\RR_{> 0}$ be the toric $\aaa$-homogenous function of weighted degree $|\aaa|$. Let $\omega^\#_v$ be the measure on $[\PPP(\aaa)(F_v)]$ that is given by $$((\fvt)^{-1}dx_{1}\dots dx_n)/d^*x.$$
 One has that $$\omega^\# _v([\PPP(\aaa)(F_v)])=2^{n-1}|\aaa| $$if~$v$ is real, and that $$\omega^{\#}_v([\PPP(\aaa)(\Fv)])=(2\pi)^{n-1}|\aaa| $$ if~$v$ is complex.
\end{lem}
\begin{proof}
Lemma \ref{kilp} gives that:
\begin{align*}
\omega^{\#}_v([\PPP(\aaa)(\Fv)])\hskip-4cm&\\&=\frac{a_n}{\lambda_{v,1}(F_{v,1})}\int_{(\Fvt)^{n-1}\times F_{v,1}}(f_v^{\#})^{-1}dx_1\dots dx_{n-1}\times\lambda_{v,1}\\
&=\frac{a_n}{\lambda_{v,1}(F_{v,1})}\int_{(\Fvt)^{n-1}\times F_{v,1}}\max_{j=1\doots n}(|x_j|_v^{{1}/{a_j}})^{-|\aaa|}dx_1\dots dx_{n-1}\times\lambda_{v,1}\\
&=\frac{a_n}{\lambda_{v,1}(F_{v,1})}\int_{(\Fvt)^{n-1}\times F_{v,1}}\max(\max_{j=1\doots n-1}(|x_j|_v^{1/a_j}),1)^{-|\aaa|}dx_1\dots dx_{n-1}\times\lambda_{v,1}\\
&=\frac{a_n}{\lambda_{v,1}(F_{v,1})}\int_{F_{v,1}}\lambda_{v,1}\int_{(\Fvt)^{n-1}}\max(\max_{j=1\doots n-1}(|x_j|_v^{1/a_j}),1)^{-|\aaa|}dx_1\dots dx_{n-1}\\
&=a_n\int_{(\Fvt)^{n-1}}\max(\max_{j=1\doots n-1}(|x_j|_v^{1/a_j}),1)^{-|\aaa|}dx_1\dots dx_{n-1}.
\end{align*}
Lemma \ref{cyv} gives that the homomorphism $$\widetilde\rho_v:\RR_{>0}\times F_{v,1}\to\Fvt \hspace{1cm}(r,z)\mapsto\rho_v(r)z,$$ where $\rho_v(r)=r^{1/[\Fv:\RR]}$, satisfies that  $(\widetilde\rho_v)_*(dr\times\lambda_{v,1})=dx|_{\Fvt}$. One has that $|\widetilde\rho_v(r,z)|_v=r$ and we deduce that \begin{align*}
\omega^{\#}_v(\PPP(\aaa)(\Fv))\hskip-1cm&\\&=a_n\int_{(\Fvt)^{n-1}}\max(\max_{j=1\doots n-1}(|x_j|_v^{1/a_j}),1)^{-|\aaa|}dx_1\dots dx_{n-1}\\
&=a_n\int_{(\RR_{>0}\times F_{v,1})^{n-1}}\max(\max_{j=1\doots n-1}(|\widetilde\rho_v(r_jz)|_v^{1/a_j}),1)^{-|\aaa|}dx_1\dots dx_{n-1}\\
&=a_n\int_{F_{v,1}^{n-1}}\lambda_{v,1}^{\otimes (n-1)}\int_{\RR_{>0}^{n-1}}\max(\max_{j=1\doots n-1}(r_j^{1/a_j}),1)^{-|\aaa|}dr_1\dots dr_{n-1}\\
&=a_n(\lambda_{v,1}(F_{v,1}))^{n-1}\int_{\RR_{>0}^{n-1}} \max(\max_{j=1\doots n-1}(r_j^{1/a_j}),1)^{-|\aaa|}dr_1\dots dr_{n-1}.
\end{align*}
Let us evaluate the last integral. Define $V_0:=\{\xxx\in\RR^{n-1}_{>0}|x_i\leq 1\}.$ For $i=1\doots n-1$, define $$V_i:=\{\xxx\in\RR^{n-1}_{>0}|x_i^{1/a_i}=\max_{j}(x^{1/a_j}_j)\}.$$
Clearly, for every $i,j\in\{0\doots n-1\}$ with $i\neq j$, one has that $V_i\cap V_j$ is $dr_1\dots dr_{n-1}$-negligible. Thus
\begin{align*}
\int_{\RR^{n-1}_{>0}}\max(\max_{j}(r_j^{1/a_j}),1)^{-|\aaa|}dr_1\dots dr_{n-1}\hskip-3cm&\\
&=\sum_{i=0}^{n-1}\int_{V_i}\max(\max_{j}(r_j^{1/a_j}),1)^{-|\aaa|}dr_1\dots dr_{n-1}\\
&=\int_{V_0}1dr_1\dots dr_{n-1}+\sum_{i=1}^{n-1}\int_{V_i}r_i^{-|\aaa|/a_i}dr_1\dots dr_{n-1}\\
&=1+\sum_{i=1}^{n-1}\int_1^{\infty}\bigg(\prod^{n-1}_{\substack{j=1\\j\neq i}}\int_{0}^{r_i^{a_j/a_i}}1dr_j\bigg)\cdot r_i^{-|\aaa|/a_i}dr_i\\
&=1+\sum_{i=1}^{n-1}\int_1^{\infty}\bigg(\prod^{n-1}_{\substack{j=1\\j\neq i}}r_i^{a_j/a_i}\bigg)r_i^{-|\aaa|/a_i}dr_i\\
&=1+\int_1^{\infty}r_i^{-1-a_n/a_i}dr_i\\
&=1+\sum_{i=1}^{n-1}\frac{a_i}{a_n}\\
&=\frac{|\aaa|}{a_n},
\end{align*}
where the third equality follows from Fubini theorem.
We deduce that$$\omega^{\#}_v([\PPP(\aaa)(\Fv)])=a_n\lambda_{v,1}(F_{v,1})^{n-1}\frac{|\aaa|}{a_n}=\lambda_{v,1}(F_{v,1})^{n-1}|\aaa|.$$Thus if~$v$ is real one has that $$\omega^{\#}_v([\PPP(\aaa)(\Fv)])=2^{n-1}|\aaa|$$ and if~$v$ is complex, one has that $$\omega^{\#}_v([\PPP(\aaa)(\Fv)])=(2\pi)^{n-1}|\aaa|.$$
\end{proof}
\subsection{}\label{measureandconstant} We define Peyre's constant for stack~$\PPP(\aaa)$ for quasi-toric families $(f_v)_v$.

We will define a measure on the product space$$\prod_{\vMF}[\PPP(\aaa)(\Fv)].$$ The space $\prod_{\vMF}[\PPP(\aaa)(\Fv)]$ is compact and Hausdorff, as for every $\vMF$ by \ref{paraap} and by \ref{aboutopa}, the spaces $[\PPP(\aaa)(\Fv)]$ are compact and Hausdorff.  Let $(f_v:\Fvnz\to\RR_{\geq 0})_{\vMF}$ be a quasi-toric family of $|\aaa|$-homogenous functions of weighted degree $|\aaa|$ such that for every~$v$, one has that the set $\{\xxx\in\Fvnz|f_v(\xxx)=0\}$ is $dx_1\dots dx_n$-negligible and that $f_v^{-1}dx_1\dots dx_n$ is a measure on $\Fvnz$. For every $\vMF$, we set $\omega_v=(f_v^{-1}dx_1\dots dx_n)/ d^*x$. Using measures $\omega_v$, we define a product measure on $\prod_{\vMF}[\PPP(\aaa)(\Fv)]$ (by \cite[Proposition 9, \no 6, \S 4, Chapter III]{Integrationj} we indeed get a measure on the product).
\begin{mydef}\label{PeyreC} 
We define a measure $\omega=\omega((f_v)_v)$ on $\prod_{\vMF}[\PPP(\aaa)(\Fv)]$ by $$|\mu_{\gcd(\aaa)}(F)|\Delta(F)^{-\frac{n-1}2}\Res(\zeta_F,1)\bigotimes_{\vMFz}\big(\zeta_v(1)^{-1}\omega_v\big)\bigotimes_{\vMFi}\omega_v. $$ We set $$\tau=\tau((f_v)_v)=\omega\big(\prod_{\vMF}[\PPP(\aaa)(F_v)]\big),$$where $\Delta(F)$ is the absolute discriminant of~$F$.
\end{mydef}
We explain how~$\omega$ changes, when the quasi-toric family is changed.
\begin{lem}\label{omegachanges}
Let~$S$ be a finite set of places and for $v\in S$, let $h_{v}:[\PPP(\aaa)(F_v)]\to\RR_{>0}$ be a continuous function. For $v\in M_F-S$, we set $h_v=1$. Let us denote by $h:\prod_{\vMF}[\PPP(\aaa)(F_v)]\to\RR_{>0}$ the function $\otimes_{\vMF}h_v$. One has that $$\omega((h_vf_v)_v)=h^{-1}\omega((f_v)_v).$$
\end{lem}
\begin{proof}
For $\vMF$, it follows directly from Lemma \ref{izbitubit} that $$(((h_{v}\circ\qav) \cdot f_v)^{-1}dx_1\dots dx_n)/d^*x=(h^{-1}_{v})((f_v^{-1}dx_1\dots dx_n)/d^*x)=h^{-1}_{v}\omega_v.$$ It follows that the measure $\omega(((h_{v,\eta}\circ\qav)\cdot f_v)_v)$ on $\prod_{\vMF}[\PPP(\aaa)(\Fv)]$ defined by the quasi-toric family $((h_{v}\circ q^{\aaa}_v)\cdot f_v)_v$ satisfies that \begin{align*}\omega(((h_{v}\circ\qav)\cdot f_v)_v)\hskip-1cm&\\
&=\frac{|\mu_{\gcd(\aaa)}(F)|}{\Delta(F)^{\frac{n-1}2}}\Res(\zeta_F,1)\prod_{\vMFz}\bigg(\zeta_v(1)^{-1}h^{-1}_{v}\omega_v\bigg)\prod_{\vMFi}h^{-1}_{v}\omega_v\\
&=h^{-1}\omega.
\end{align*}
\end{proof}
We give another expression for $\tau$.
\begin{lem}\label{locom}Let~$S$ be the finite set of places~$v$ for which $f_v$ is not toric. One has that 
\begin{multline*}\tau((f_v)_v)\\=\frac{\Res(\zeta_F,1)|\mu_{\gcd(\aaa)}(F)|}{\Delta(F)^{\frac{n-1}2}\zeta_F(|\aaa|)}\prod_{v\in S\cap M_F^0}\frac{\zeta_v(|\aaa|)\omega_v([\PPP(\aaa)(\Fv)])}{\zeta_v(1)}\times\\\times \prod_{\vMFi}\omega_v([\PPP(\aaa)(\Fv)]).\end{multline*}
\end{lem}
\begin{proof}
Lemma \ref{locomi} gives that for every $v\in M_F^0-S$ one has $$\omega_v([\PPP(\aaa)(\Fv)])\zeta_v(1)^{-1}=\zeta_v(|\aaa|)^{-1}$$ and thus:
\begin{align*}
\tau((f_v)_v)|\mu_{\gcd(\aaa)}(F)|^{-1}\Delta(F)^{\frac{n-1}2}\hskip-5cm&\\&=\Res(\zeta_F,1)\prod_{v\in S\cap M_F^0}\frac{\omega_{v}([\PPP(\aaa)(\Fv)])}{\zeta_v(1)}\times\\&\quad\quad\quad\quad\quad\quad\quad\quad\times \prod_{v\in M_F^0-S}\zeta_v(|\aaa|)^{-1}\times \prod_{\vMFi}\omega_v([\PPP(\aaa)(\Fv)])\\
&=\frac{\Res(\zeta_F,1)}{\zeta(|\aaa|)}\prod_{v\in S\cap M_F^0}\bigg(\frac{\zeta_v(|\aaa|)\omega_v([\PPP(\aaa)(\Fv)])}{\zeta_v(1)}\bigg)\times \prod_{\vMFi}\omega_v([\PPP(\aaa)(\Fv)]).
\end{align*}
\end{proof}
\section{Haar measure on $[\TTa(F_v)]$} 
Let $\Fvt$ acts on $\Fvtn$ by $t\cdot\xxx=(t^{a_j}x_j)_j$. This action is proper by \ref{aboutopa}.  The quotient for this action is $[\TTa(\Fv)]$ by \ref{pafvtafv} and is locally compact by \ref{aboutopa}. By \ref{pafvtafv}, one has that the map $[\TTa(\Fv)]\to[\PPP(\aaa)(\Fv)],$ induced from $\Fvt$-invariant map $\Fvtn\hookrightarrow\Fvnz\to [\PPP(\aaa)(\Fv)]$ is an open embedding.  By \ref{aboutopa}, the map 
$$\epsilon :\Fvt\to \Fvtn\hspace{1cm} t\mapsto (t^{a_j})_j$$ is proper, its image $(\Fvt)_{\aaa}:=\epsilon(\Fvt)=\{(t^{a_j})_j|t\in\Fvt\},$ is a closed subgroup of $\Fvtn$ and one has an identification $[\TTa(\Fv)]=\Fvtn/(\Fvt)_{\aaa}$. Using this identification, we endow $[\TTa(\Fv)]$ with a structure of a topological group (which is necessary abelian). The goal of this section is to define a Haar measure on $[\TTa(\Fv)]$ and relate it with the measure $\omega_v$ on $[\PPP(\aaa)(\Fv)]$.
\subsection{}\label{snoo}
Let $v\in M_F$. We are going to define a Haar measure on $[\TTa(\Fv)]$. 
%

By the fact that a product of a continuous function and a measure is a measure \cite[\no 4, \S1, Chapter III]{Integrationj}, one has that $$\mathscr C^0_c(\Fvtn,\CC)\to\CC\hspace{1cm}\phi\mapsto \int_{\Fvtn}\prodjn|x_j|_v^{-1}dx_1\dots dx_n$$ is a measure on $\Fvtn$. The the function $$\Fvnz\to\RR_{\geq 0}\hspace{1cm}\xxx\mapsto \prodjn |x_j|_v$$ is $\aaa$-homogenous of weighted degree $|\aaa|$ and the set where it vanishes is given by $\{\xxx\in\Fvnz|\exists j: x_j=0\}$, thus this set is $dx_1\dots dx_n$-negligible (because it is contained in a finite union of hyperplanes in $(\Fv)^n$.) It follows from \ref{xiub} that 
 $$\prodjn|x_j|_v^{-1}dx_1\dots dx_n=d^*x_{1}\dots d^*x_{n}$$ is $\Fvt$-invariant measure on $\Fvt$.
\begin{mydef}\label{haarttafv}
We define a measure $\mu_v$ on $[\TTa(\Fv)]$ by
\begin{equation*}\mu_v:=(d^*x_1\dots d^*x_n)/d^*x_v.
\end{equation*}
\end{mydef}
 By \ref{rudj} the measure $\mu_v$ is a Haar measure on $[\TTa(F_v)]$. 
 \begin{lem}\label{icicu}
Let $v\in M_F$. Let $f_v:\Fvnz\to\RR_{\geq 0}$ be a continuous $\aaa$-homogenous function of weighted degree $|\aaa|$ such that $$f_v^{-1}dx_1\dots  dx_n$$ is a measure on $\Fvnz$ (this in particular implies that $dx_1\dots dx_n(\{\xxx|f_v(\xxx)=0\})=0$). Let $H_v:[\TTa(F_v)]\to\RR_{\geq 0}$ be the function given by the continuous $\Fvt$-invariant function $$\xxx\mapsto f_v(\xxx)\prod _{j=1}^n|x_{j}|^{-1}_v. $$ 
\begin{enumerate}
\item One has that $\mu_v(\{\yyy|H_v(\yyy)=0\})=0$.
\item 
One has an equality of the measures $\mu_v=(H_v)(\omega_v|_{[\TTa(\Fv)]}).$
\end{enumerate}
\end{lem} 
 \begin{proof}  \begin{enumerate}
 \item By the definition of~$H_v$ one has that $\{\yyy|H_v(\yyy)=0\}=\qav(\{\xxx\in\Fvtn| f_v(\xxx)=0\}).$ Note that as $dx_1\dots dx_n(\{\xxx|f_v(\xxx)=0\})=0,$ it follows that the set $\{\xxx\in\Fvtn|f_v(\xxx)=0\}$ is $dx_1\dots dx_n$-negligible and thus $d^*x_1\dots d^*x_n$-negligible. Hence, by \cite[Proposition 6, \no 3, \S 2, Chapter VII]{Integrationd}, one has that $\mu_v(\qav(\{\xxx\in\Fvtn|f(\xxx)=0\}))=0.$
\item 
Observe that \begin{align*}d^*x_1\dots d^*x_n=\prodjn |x_j|_v^{-1}dx_1\dots dx_n&=f_v(\xxx)\prodjn |x_j|_v^{-1} f_v(\xxx)^{-1}dx_1\dots dx_n\\&=(H_v\circ\qav)f_v^{-1}dx_1\dots dx_n.\end{align*}Now Lemma \ref{izbitubit} gives precisely that $$\mu_v=(d^*x_1\dots d^*x_n)/d^*x=H_v (f_v^{-1}dx_1\dots dx_n)/d^*x=H_v\omega_v.$$
\end{enumerate}
\end{proof}
 \begin{lem}\label{smacor}
Let $h:[\TTa(\Fv)]\to \CC$ be a function.
Suppose $\vMFz$. One has that $h\in L^1([\TTa(\Fv)],\mu_v)$ if and only if $ (h\circ\qav)\in L^1(\Fvtn\cap\Dav,d^*x_1\dots d^*x_n)$, and if $h\in L^1([\TTa(\Fv)],\mu_v)$, one has that:
$$\int_{[\TTa(\Fv)]}h\mu_v=\frac{1}{1-\pivv}\int_{\Fvtn\cap\Dav}(h\circ\qav)d^*x_1\dots d^*x_n.$$
Suppose $\vMFi$. One has that $h\in L^1([\TTa(\Fv)],\mu_v)$ if and only if $ (h\circ\qav)\in L^1((\Fvt)^{n-1}\times F_{v,1},d^*x_1\dots d^*x_{n-1}\times\lambda_{v,1})$, and if $h\in L^1([\TTa(\Fv)],\mu_v)$, one has that:
$$\frac{a_n}{\lambda_{v,1}(F_{v,1})}\int_{(\Fvt)^{n-1}\times F_{v,1}}(h\circ\qav)d^*x_1\dots d^*x_{n-1}\times\lambda_{v,1} .$$
\end{lem}
\begin{proof}
Let $f_v^{\#}:\Fvnz\to\RR_{>0}$ be the toric $\aaa$-homogenous function of weighted degree $|\aaa|$. Let $H^{\#}_v:[\TTa(\Fv)]\to \RR_{>0}$ be the induced function from $\Fvt$-invariant function $\Fvtn\to\RR_{>0}$ given by $\xxx\mapsto f^{\#}_v(\xxx)\prod_{j=1}^n|x_j|_v$. It follows from Lemma \ref{icicu}, that one has an equality of the measures $\mu_v=H^{\#}_v\omega^{\#}_v|_{[\TTa(\Fv)]}$. We deduce that $h\in L^1([\TTa(\Fv)],\mu_v)$ if and only if $hH^{\#}_v\in L^1([\TTa(\Fv)],\omega^{\#}_v),$ and as by \ref{ttadmz} one has $\omega^{\#}_v([\PPP(\aaa)(\Fv)]-[\TTa(\Fv)])=0$, if and only if $hH^{\#}_v\in L^1([\PPP(\aaa)(\Fv)],\omega^{\#}_v)$. Moreover, it follows that if $h\in L^1([\TTa(\Fv)],\mu_v),$ then $$\int_{[\TTa(\Fv)]}h\mu_v=\int_{[\PPP(\aaa)(\Fv)]}(hH^{\#}_v)\omega^{\#}_v.$$

Suppose $\vMFz.$ Recall that by \ref{davdavdav} one has that $f^{\#}_v|_{\Dav}=1$. By \ref{kilp}, one has that $hH^{\#}_v\in L^1([\PPP(\aaa)(\Fv)],\omega^{\#}_v)$ if and only if $$(((hH^{\#}_v)\circ\qav)(f_v^{\#})^{-1})|_{\Dav}=(h\circ\qav)(H_v^{\#}\circ\qav)|_{\Dav}=(h\circ\qav)(|x_1|_v^{-1}\cdots |x_n|_v^{-1})|_{\Dav}$$ is an element of $ L^1(\Dav,dx_1\dots dx_n).$ Moreover, Lemma \ref{kilp} gives that if $hH^{\#}_v\in L^1([\PPP(\aaa)(\Fv)],\omega^{\#}_v)$ then $$\int_{[\PPP(\aaa)(\Fv)]}hH^{\#}_v\omega^{\#}_v=\frac{1}{1-\pivv}\int_{\Dav}(h\circ\qav)|x_1|_v^{-1}\cdots |x_n|_v^{-1}dx_1\dots dx_n.$$ As $dx_1\dots dx_n(\Dav-(\Dav\cap\Fvtn))=0,$ (because $\Dav-(\Dav\cap\Fvtn)$ is contained in a finite union of hyperplanes of $\Fv^n$,) the last integral is equal to $$\frac{1}{1-\pivv}\int_{\Fvtn\cap\Dav}(h\circ\qav)d^*x_1\dots d^*x_n.$$
It follows that one has $h\in L^1([\TTa(\Fv)],\mu_v)$ if and only if $(h\circ\qav)\in L^1(\Fvtn\cap\Dav,d^*x_1\dots d^*x_n)$ and if $h\in L^1([\TTa(\Fv)],\mu_v),$ then  
$$\int_{[\TTa(\Fv)]}h\mu_v=\frac{1}{1-\pivv}\int_{\Fvtn}(h\circ\qav)d^*x_1\dots d^*x_n.$$
Suppose that $\vMFi$. By \ref{kilp}, one has that $hH^{\#}_v\in L^1([\PPP(\aaa)(\Fv)],\omega^{\#}_v)$ if and only if \begin{equation*}((hH^{\#}_v)\circ\qav)(f_v^{\#})^{-1}=((h\circ\qav)\cdot (H^{\#}\circ\qav))(f_v^{\#})^{-1}=(h\circ\qav)\prod_{j=1}^{n-1}|x_j|_v^{-1}\end{equation*} is an element of $L^1((\Fvt)^{n-1}\times F_{v,1},dx_1\dots dx_{n-1}\times \lambda_{v,1})$ i.e. if and only if $$(h\circ\qav)\in L^1((\Fvt)^{n-1}\times F_{v,1},d^*x_1\dots d^*x_{n-1}\lambda_{v,1}) .$$ Moreover, Lemma \ref{kilp} gives that if $hH^{\#}_v\in L^1([\PPP(\aaa)(\Fv)],\omega^{\#}_v)$ then \begin{align*}\int_{[\PPP(\aaa)(\Fv)]}hH^{\#}_v\omega^{\#}_v\hskip-2cm&\\&=\frac{a_n}{\lambda_{v,1}(F_{v,1})}\int_{(\Fvt)^{n-1}\times F_{v,1}}(h\circ\qav)\prod_{j=1}^{n-1}|x_j|_v^{-1} dx_1\dots dx_{n-1}\times\lambda_{v,1}\\
&=\frac{a_n}{\lambda_{v,1}(F_{v,1})}\int_{(\Fvt)^{n-1}\times F_{v,1}}(h\circ\qav)d^*x_1\dots d^*x_{n-1}\times\lambda_{v,1}.
\end{align*}
\end{proof}
\subsection{} Recall that for $\vMFz$ by \ref{identofov}, the group $[\TTa(\Ov)]$ identifies with the image of $\Ovn$ in $[\TTa(\Fv)]$ for the quotient homomorphism $\Fvtn\to[\TTa(\Fv)]$ and that with this identification, becomes open and compact subgroup of $[\TTa(\Fv)]$. We calculate the volume of $[\TTa(\Ov)]$ against $\mu_v$. 
\begin{lem}\label{xiuim}
Let $\vMFz$. The measure $\mu_v$ is normalized by $$\mu_v([\TTa(\Ov)])=(1-\pivv)^{n-1}=\zeta_v(1)^{-(n-1)}$$.
\end{lem}
\begin{proof}
Let us firstly establish that $(\qav)^{-1}([\TTa(\Ov)])\cap \Dav=(\Ovt)^n.$ One has that $$(\qav)^{-1}([\TTa(\Ov)])^{-1}=\bigcup_{t\in\Gm(\Fv)}t\cdot\Ovtn.$$
Note that if $\uuu\in(\Ovt)^n,$ then for every index~$j$ one has $v(t^{a_j}u_j)=a_jv(t)+v(u_j)=a_jv(t)$. Now if $v(t)>0$ it follows that $t\cdot\uuu\in\prodjn(\piv^{a_j}\Ov),$ thus $t\cdot\uuu\not\in\Dav$ and if $v(t)<0$ then $t\cdot\uuu\not\in\Ov^n$, thus $t\cdot\uuu\not\in\Dav$. We have obtained if $v(t)\neq 0$, then $(t\cdot\Ovn)\cap\Dav=\emptyset$ and hence, $$\Dav\cap(\qav)^{-1}([\TTa(\Ov)])=\Dav\cap  \bigcup_{t\in\Gm(\Fv)}t\cdot\Ovn=\Dav\cap\Ovn=\Ovn.$$ 
We deduce that one has equality of functions \begin{equation*}(\jed_{(\qav)^{-1}([\TTa(\Ov)])})|_{\Fvtn\cap \Dav}={\jed_{(\Ovt)^n}}.
\end{equation*}
Now, we can calculate $\mu_v([\TTa(\Ov)]).$ One has by \ref{smacor} that \begin{align*}\int_{[\TTa(\Fv)]}\mathbf 1_{[\TTa(\Ov)]}\mu_v&=\frac{1}{1-\pivv}\int_{\Fvtn\cap\Dav}\mathbf 1_{(\qav)^{-1}([\TTa(\Ov)])} d^*x_1\dots d^*x_n\\
&=\frac{1}{1-\pivv}\int_{\Fvtn}\mathbf 1_{(\Ovt)^n}d^*x_1\dots d^*x_n\\
&=\frac{1}{1-\pivv}\bigg(\int_{\Fvt}\mathbf 1_{\Ovt}{d^*x}\bigg)^n\\
&=(1-\pivv)^{n-1}.
\end{align*}
The statement follows.
\end{proof}
\subsection{}Let us dedicate this paragraph to the definition of a Haar measure on $[\TTa(\AAF)]$. Let $n\in\ZZ_{>0}$ and let $\aaa\in\ZZ^n_{>0}$. We set $[\TTa(\AAF)]$ to be the ``adelic space" of~$\TTa$, i.e. $$[\TTa(\AAF)]= \sideset{}{'}\prod _{v\in M_F}[\TTa(F_v)],$$where the restricted product is taken with respect to the sequence of the open and compact subgroups for $[\TTa(\Ov)]\subset[\TTa(\Fv)]$ for $\vMFz$. 
\begin{prop}\label{haarmofrp}Let~$I$ be a set and let $I'$ be a subset such that $I-I'$ is finite. For $i\in I$, let $G_i$ be a locally compact abelian group endowed with a Haar measure $dg_i$. For $i\in I'$, let $H_i$ be an open and compact subgroup of $G_i$ such that $dg_i(H_i)=1$. Set~$G$ to be the restricted product $\sideset{}{'}\prod _{i\in I}G_i$ with respect to the subgroups $H_i\subset G_i$ for $i\in I'$.
\begin{enumerate}
\item {\normalfont \cite[Proposition 5, \no 5, \S 1, Chapter VII]{Integrationd}}Let~$S$ be a finite subset of~$I$ containing $I-I'$. Set $$G_S:=\prod_{i\in S} G_i\times \prod_{i\in I-S}H_i.$$ The measure $$\bigotimes_{i\in S}dg_i\bigotimes_{i\in I-S}(dg_i|_{H_i})$$ is a Haar measure on $G_S$.
\item{\normalfont \cite[Proposition 5.5]{Ramakrishnan}} There exists a unique Haar measure $dg$ on~$G$ such that for every finite subset $S\subset I$ which contains $I-I'$ one has that$$dg|_{G_S}=\bigotimes_{i\in S}dg_i\bigotimes_{i\in I-S}(dg_i|_{H_i}).$$
\end{enumerate}
\end{prop}
The measure $dg$ will be called the restricted product Haar measure and, by the abuse of the notation, may be denoted as $dg=\otimes dg_i$. We present a way how to calculate the integral of a function.
\begin{prop}[{\cite[Proposition 5-6]{Ramakrishnan}}] \label{rsra} In the situation of \ref{haarmofrp}, let $f_i\in L^1(G_i,dg_i)$ be a continuous complex valued function such that there exists a finite subset $I''\subset I'$ with $I'-I''$ is finite such that $f_i|H_i=1$ for every $i\in I''$. The function $$f:(x_i)_i\mapsto \prod_i f_i(x_i)$$ is continuous. Suppose that $$\prod_i\int_{G_i}f_idg_i$$ converges. Then $f\in L^1(G)$ and $$\int_Gfdg=\prod_{i}\int_{G_i}f_idg_i.$$
\end{prop}
For $\vMFz$, we have established in \ref{xiuim} that $\mu_v([\TTa(\Ov)])=\zeta_v(1)^{-(n-1)}.$ We will apply \ref{haarmofrp} to define a Haar measure on $[\TTa(\AAF)]$.
\begin{mydef}\label{muaaf}
Let $\mu_{\AAF}$ be the restricted product measure $$\mu_{\AAF}=\bigotimes_{\vMFz}\zeta_v(1)^{(n-1)}\mu_v\otimes \bigotimes_{\vMFi}\mu_v.$$
\end{mydef}
The group $\AAFt$ acts on $(\AAFt)^n$ via the proper homomorphism $(x_v)_v\mapsto ((x_v^{a_j})_j)_v$ (Lemma \ref{properaaftn}) and one has an identification $[\TTa(\AAF)]=(\AAFt)^n/(\AAFt)$  (Lemma \ref{quotttaf} together with \ref{rudj}). Endow $\AAFt$ with the Haar measure $$d^*x_{\AAF}:=\bigotimes_{\vMFz}\zeta_v(1)d^*x_v\otimes \bigotimes_{\vMFi}d^*x_v.$$ Let $d^*\xxx_{\AAF}:=d^*x_{\AAF}^{\otimes n}$ be the product Haar measure on $(\AAFt)^n$. 
\begin{lem}\label{muaafisquot}
One has the following equality of the measures on $[\TTa(\AAF)]=(\AAFt)^n/\AAFt:$ $$\mu_{\AAF}=d^*\xxx_{\AAF}/d^*x_{\AAF}.$$
\end{lem}
\begin{proof}The quotient measure $d^*\xxx_{\AAF}/d^*x_{\AAF}$ is a Haar measure on $[\TTa(\AAF)]$ by \ref{rudj}. Therefore, it suffices to verify the equality on a single non-trivial compactly supported function on $[\TTa(\AAF)]$ which takes non-negative values. For $\vMFz$ we set $\phi_v=\mathbf 1_{(\Ovt)^n}:\Fvtn\to\CC$ and for $\vMFi$, we let $\phi _v:(\Fvt)^n\to\RR_{\geq 0}$ be a non-trivial continuous function with compact support. The function $\phi=\bigotimes_v\phi_v$ is continuous by \ref{rsra} and compactly supported (its support is the set $\prod_{\vMFz}(\Ovt)^n\times \prod_{\vMFi}\supp(\phi_v)$). For $\vMF$ and $\yyy\in[\TTa(F_v)],$ let $\widetilde\yyy\in(\Fvt)^n$ be its lift and for $\yyy\in[\TTa(\AAF)],$ let $\widetilde\yyy\in(\AAFt)^n$ be its lift.
For $\vMFz$ we define $$\phi_v^*:[\TTa(\Fv)]\to\RR_{\geq 0}\hspace{1cm}\yyy\mapsto \zeta_v(1)\int_{\Fvt}\phi_v(x\cdot\widetilde\yyy)d^*x$$ and for $\vMFi$ we define $$\phi^*_v:[\TTa(\Fv)]\to\RR_{\geq 0}\hspace{1cm}\yyy\mapsto\int_{\Fvt}\phi_v(x\cdot\widetilde\yyy)d^*x,$$ the functions are well defined, continuous compactly supported and of non-negative values by \ref{duio}. By the definition we have that $(d^*x)^{\otimes n}/d^*x=\mu_v$ and thus $(d^*x)^{\otimes n}/(\zeta_v(1)d^*x)=\zeta_v(1)^{-1}\mu_v$. Now by \ref{intbour}, one has that  
\begin{align*}
\int_{\Fvtn}(\phi) (d^*x)^{\otimes n}&=\int_{[\TTa(\Fv)]}(\mu_v/\zeta_v(1)) \bigg(\yyy\mapsto \int_{\Fvt}\phi(x\cdot\widetilde\yyy)\zeta_v(1)d^*x\bigg)\\
&=\int_{[\TTa(\Fv)]}(\phi^*)(\mu_v/\zeta_v(1))\\
&=\zeta_v(1)^{-1}\int_{[\TTa(\Fv)]}\phi^*\mu_v.
\end{align*}
If $\vMFz$, let us prove that $$\phi^*_v=\mathbf 1_{[\TTa(\Ov)]}.$$ Indeed, if $\zzz\not\in (q^\aaa_v)^{-1}([\TTa(\Ov)]),$ then $\phi_v(x\cdot\zzz)=0$ for every $x\in\Fvt$, thus $\phi^*_v(\qav(\zzz))=0$. If $\zzz\in(q^\aaa_v)^{-1}([\TTa(\Ov)],$ then $x\cdot\zzz\in(\Ovt)^n$ if and only if $v(x)=v(z_1)/a_1$, and thus $d^*x(\{x\in\Fvt| x\cdot\zzz\in(\Ovt)^n\})=\zeta_v(1)^{-1},$ and hence $\zeta_v(1)\int_{\Fvt}\phi_v(x\cdot\zzz)d^*x=1$. It follows that $\phi^*_v=\mathbf 1_{[\TTa(\Ov)]}.$

We also define $$\phi^*:[\TTa(\AAF)]\to\CC\hspace{1cm}\yyy\mapsto \int_{\AAFt}\phi(x_{\AAF}\cdot\widetilde\yyy)d^*x_{\AAF},$$by \ref{duio} it is well defined and continuous function. We are going to verify that the two measures coincide on $\phi^*$. For every $\yyy\in[\TTa(\AAF)]$, one has that \begin{align*}
\phi^*(\yyy)&=\int_{\AAFt}\phi(x_{\AAF}\cdot \widetilde\yyy)d^*x_{\AAF}\\
&=\prod_{\vMFz}\zeta_v(1)\int_{[\TTa(\Fv)]}\phi_v(x\cdot\widetilde\yyy_v)d^*x\times\prod_{\vMFi}\int_{[\TTa(\Fv)]}\phi_v(x\cdot\widetilde\yyy_v)d^*x\\
&=\prod_{\vMF}\phi_v^*(\yyy_v),
\end{align*}i.e. $\phi^*=\otimes_{\vMF}\phi^*_v$.
By \ref{intbour}, we have that \begin{align*}\int_{[\TTa(\AAF)]}(\phi^*)(d^*\xxx_{\AAF}/d^*x_{\AAF})=\int _{(\AAFt)^n}\phi d^*\xxx_{\AAF}.\end{align*}
On the other side, one has that \begin{align*}
\int_{[\TTa(\AAF)]}\phi^*\mu_{\AAF}\hskip-2cm&\\&=\int_{[\TTa(\AAF)]}\big(\otimes_v\phi^*_v\big)\mu_{\AAF}\\\\&=\prod_{\vMFz}\zeta_v(1)^{n-1}\int_{[\TTa(\Fv)]}\phi^*_v\mu_v\times\prod_{\vMFi}\int_{[\TTa(\Fv)]}\phi^*_v\mu_v\\
&=\prod_{\vMFz}\zeta_v(1)^n\int_{\Fvtn}\phi_vd^*x_1\dots d^*x_n\times\prod_{\vMFi}\int_{\Fvtn}\phi_vd^*x_1\dots d^*x_n\\
&=\int_{(\AAFt)^n}(\otimes_v\phi_v)d^*\xxx_{\AAF}\\
&=\int_{(\AAFt)^n}\phi d^*\xxx_{\AAF}.
\end{align*}
We have verified that the Haar measures $d^*\xxx_{\AAF}/d^*x$ and $\mu_{\AAF}$ satisfy that $$(d^*\xxx_{\AAF}/d^*x)(\phi^*)=\mu_{\AAF}(\phi^*)> 0,$$ thus $d^*\xxx_{\AAF}/d^*x=\mu_{\AAF}.$ The statement is proven.
%
\end{proof}
\subsection{} In this paragraph we define and calculate the Tamagawa number of the ``stacky torus"~$\TTa$. 

Denote by $\mu_{\RR}$ the quotient measure $(d^*x)^{\otimes n}/ d^*x$ on the quotient $\RR^n_{>0}/\RR_{>0}$ for the action of $\RR_{>0}$ via the proper map $t\mapsto (t^{a_j})_j$. By \ref{rudj} the quotient identifies with the quotient group $\RR^n_{>0}/(\RR_{>0})_{\aaa}$ and $\mu_{\RR}$ is a Haar measure on it. 
\begin{mydef} 
We define $\mu_1$ to be the Haar measure on $[\TTa(\AAF)_1]$ normalized by the condition $\mu_1\otimes \mu_{\RR}=\mu_{\AAF}$ for the identification $$[\TTa(\AAF)]_1\times(\RR^n_{>0}/(\RR_{>0})_{\aaa})=[\TTa(\AAF)]$$ given by the isomorphism (\ref{seulment}). 
\end{mydef} 
We can write $\mu_1$ as a quotient measure:
\begin{lem}\label{mujedan}
The measure $\mu_{1}$ identifies with the measure $d^*\xxx^1_{\AAF}/d^*x^1_{\AAF}$ on $[\TTa(\AAF)]_1=(\AAFj)^n/\AAFj,$ where the action is given by the proper morphism $$\AAFj\to(\AAFj)^n\hspace{1cm}(x_v)\mapsto((x^{a_j}_v)_j)_v.$$
\end{lem}
\begin{proof}
We have the following bicomplex (we have written the corresponding measures next to the groups)
\[\begin{tikzcd}[column sep=small]
	{K_1} & {K_2} & {K_3} \\
	(\AAFj,dx^1_{\AAF}) & (\AAFt,d^*x_{\AAF}) & {(\RR_{>0},d^*x)} \\
	{((\AAFj)^n,(dx^1_{\AAF})^{\otimes n})} & {((\AAFt)^n,d^*\xxx_{\AAF})} & {(\RR^n_{>0},(d^*x)^{\otimes n})} \\
	{([\TTa(\AAF)]_1,\mu_1)} & {([\TTa(\AAF)],\mu_{\AAF})} & {((\RR^n/(\RR_{>0})_{\aaa},(d^*x)^{\otimes n}/d^*x)},
	\arrow[from=4-1, to=4-2]
	\arrow[from=4-2, to=4-3]
	\arrow[from=3-1, to=3-2]
	\arrow[from=3-2, to=3-3]
	\arrow[from=2-1, to=2-2]
	\arrow[from=2-2, to=2-3]
	\arrow[from=1-2, to=1-3]
	\arrow[from=1-1, to=1-2]
	\arrow[from=1-1, to=2-1]
	\arrow[from=2-1, to=3-1]
	\arrow[from=2-2, to=3-2]
	\arrow[from=3-1, to=4-1]
	\arrow[from=3-2, to=4-2]
	\arrow[from=3-3, to=4-3]
	\arrow[from=2-3, to=3-3]
	\arrow[from=1-3, to=2-3]
	\arrow[from=1-2, to=2-2]
\end{tikzcd}\] where $K_1$, $K_2$ are $K_3$ are the corresponding kernels, which are compact by \ref{properaaftn}, endowed with the probability Haar measures, all terms that are not drawn and corresponding measures are assumed to be the trivial groups endowed with the probability Haar measures. All horizontal sequences and the last vertical sequence are of the trivial measure Euler-Poincar\'e characteristic. By \ref{muaafisquot}, the second vertical sequence is of the trivial measure Euler-Poincar\'e characteristic. By \ref{complexdoesterle}, it follows that the first vertical sequence is of the trivial measure Euler-Poincar\'e characteristic. The statement follows. 
\end{proof}
For every space~$X$, we denote by $\coun_X$ the counting measure on~$X$. Proposition \ref{cesres} gives that the image of $[\TTa(F)]$ under the map $[\TTa(i)]$ (which is induced map from the canonical inclusion $(F^\times)^n \to (\AAFt)^n$) is contained in $[\TTa(\AAF)]_1$ and that, moreover, it is discrete, closed and cocompact subgroup of $[\TTa(\AAF)]_1$.

The kernel of the map of discrete groups $$F^\times\to (F^{\times})^n\hspace{1cm}t\mapsto (t^{a_j})_j$$ is the finite group $\mu_{\gcd(\aaa)}(F).$ We endow $[\TTa(F)]=(F^{\times})^n/F^{\times}$ with the unique Haar measure which makes the complex $$1\to\mu_{\gcd(\aaa)}(F)\to F^{\times}\to (F^{\times})^n\to[\TTa(F)]\to 1$$ to have the trivial measure Euler-Poincar\'e characteristics.  
This measure is precisely $\frac{1}{|\mu_{\gcd(\aaa)}(F)|}\coun_{[\TTa(F)]}.$ 
The kernel of the homomorphism $[\TTa(i)]:[\TTa(F)]\to [\TTa(i)]([\TTa(F)])$ is the finite group $|\Sh^1(F,\mu_{\gcd(\aaa)})|$. We endow the discrete group $[\TTa(i)]([\TTa(F)])$ with the pushforward measure of the measure $\frac{1}{|\mu_{\gcd(\aaa)}(F)|}\coun_{[\TTa(F)]}$ on $[\TTa(F)]$. This measure is precisely the measure$$\frac{|\Sh^1(F,\mu_{\gcd(\aaa)})|}{|\mu_{\gcd(\aaa)}(F)|}\coun_{[\TTa(i)]([\TTa(F)])}.$$
\begin{mydef}\label{deftamnum}
We define the Haar measure on the discrete set $[\TTa(i)]([\TTa(F)])$ by $$\overline{\mu_1}:=\mu_{1}/\bigg(\frac{|\Sh^1(F,\mu_{\gcd(\aaa)})|}{|\mu_{\gcd(\aaa)}(F)|}\coun_{[\TTa(i)]([\TTa(F)])}\bigg).$$
We define \begin{multline*}\Tam(\TTa)\\:=\Res(\zeta_F,1)^{-(n-1)}\Delta(F)^{-\frac{n-1}2}\overline{\mu_1}([\TTa(\AAF)]_1/[\TTa(i)]([\TTa(F)])),\end{multline*}where $\Delta(F)$ is the absolute discriminant of~$F$.
\end{mydef}
\begin{prop}\label{tamagawatta}
One has that $$\Tam(\TTa)=1.$$
\end{prop}
\begin{proof}
Consider the following bicomplex (where $\lambda_{\aaa}$ stands for the map $t\mapsto (t^{a_j})_j,$ whatever the domain is;  $K_1$, $K_2$ and $K_3$ are the corresponding kernels and $E$ is the corresponding quotient; and every term that is not drawn is assumed to be the trivial group):
\[\begin{tikzcd}[column sep=small]
	1 & {\mu_{\gcd(\aaa)}(F)} & {K_2} & {K_3} \\
	1 & {F^{\times}} & \AAFj & {(\AAFj/F^{\times})} \\
	1 & {(F^{\times})^n} & {(\AAFj)^n} & {(\AAFj)^n/(F^{\times})^n} \\
	{\Sh^1(F,\mu_{\gcd(\aaa)})} & {[\TTa(F)]} & {[\TTa(\AAF)]_1} & E,
	\arrow[from=4-2, to=4-3]
	\arrow[from=4-3, to=4-4]
	\arrow[from=3-2, to=3-3]
	\arrow[from=3-3, to=3-4]
	\arrow[from=2-2, to=2-3]
	\arrow[from=2-3, to=2-4]
	\arrow[from=1-3, to=1-4]
	\arrow[from=1-2, to=1-3]
	\arrow[from=1-2, to=2-2]
	\arrow["{\lambda_{\aaa}}", from=2-2, to=3-2]
	\arrow["{\lambda_{\aaa}}", from=2-3, to=3-3]
	\arrow[from=3-2, to=4-2]
	\arrow[from=3-3, to=4-3]
	\arrow[from=3-4, to=4-4]
	\arrow["{\lambda_{\aaa}}", from=2-4, to=3-4]
	\arrow[from=1-4, to=2-4]
	\arrow[from=1-3, to=2-3]
	\arrow[from=4-1, to=4-2]
	\arrow[from=3-1, to=4-1]
	\arrow[from=2-1, to=3-1]
	\arrow[from=1-1, to=2-1]
	\arrow[from=1-1, to=1-2]
	\arrow[from=2-1, to=2-2]
	\arrow[from=3-1, to=3-2]
\end{tikzcd}\]
Endow every finite group in the bicomplex with the probability Haar measure. 
The group $K_2$ is compact by \ref{properaaftn}, we endow it with the probability Haar measure. The group $K_3$ is compact, as it is a closed subgroup of the compact group $\AAFj/F^{\times}$ and we endow it with the probability Haar measure. 
Endow the discrete groups $F^{\times}$ and $(F^{\times})^n$ with the counting measures. 
Endow $\AAFj,$  $(\AAFj)^n$ and $[\TTa(\AAF)]_1$ with the measure $d^*x^1_{\AAF},$  $(d^*x^1_{\AAF})^{\otimes n}$ and $\mu_1$, respectively. Finally, endow $(\AAFj/F^{\times}),$ $(\AAFj)^n/(F^{\times})^n$ and $E$ with the unique Haar measures so that the corresponding rows are of the trivial measure Euler-Poincar\'e characteristics, that is $(\AAFj/F^{\times})$ and $(\AAFj)^n/(F^{\times})^n$ are endowed with the corresponding quotient measures, while $E$ is endowed with the measure $\overline{\mu_{1}}.$ 
 We apply \ref{complexdoesterle}. The measure Euler-Poincar\'e characteristics of every row is~$1$. The measure Euler-Poincar\'e characteristics of the first three columns is~$1$. 
Proposition \ref{complexdoesterle} gives that the measure Euler-Poincar\'e characteristics of the fourth column is~$1$. It follows from Part (2) of \ref{lemoesterle} that 
 \begin{align*}\overline{\mu_{1}}(E)&=\overline{\mu_{1}}([\TTa(\AAF)]_1/[\TTa(i)]([\TTa(F)]))\\
 &= \frac{d^*x^1_{\AAF}(\AAFj/F^{\times})^{n}}{d^*x^1_{\AAF}(\AAFj/F^{\times})}\\
&=d^*x^1_{\AAF}(\AAFj/F^{\times})^{n-1}\\
&=(\Res(\zeta_F,1)\Delta(F)^{\frac12})^{n-1},
\end{align*}
where we have used that $d^*x^1_{\AAF}(\AAFj/F^{\times})=\Res(\zeta_F,1)\Delta(F)^{\frac12}$ 
(see e.g. \cite[Page 116]{AdelesW}). We obtain that $$\Tam(\TTa)=1.$$
\end{proof}
\begin{rem}
\normalfont
When $\aaa=\jed$, the result is a classical result that the Tamagawa number of a split torus is~$1$ (\cite[Theorem 3.5.1]{Ono}).
\end{rem}
\begin{rem}
\normalfont
When $n=1$, Oesterl\'e has calculated in \cite[Proposition 2]{Oesterle} the volume of the fundamental domain for the action of the subgroup $[\TT(a)(i)([\TT(a)(F)])]$ on $[\TT(a)(\AAF)]$. The volume of the fundamental domain is not~$1$, because Oesterl\'e has used a different normalization of the Haar measure.

\end{rem}
\chapter{Analysis of characters of $[\TTa(\AAF)]$}
\label{Analysis of characters}
In this chapter, we will study the characters of $[\TTa(\AAF)]$. Later in the chapter, we recall several facts on the estimates of~$L$-functions.
\section{Characters of $\AAFt$}\label{scaft}
We are going to define two ``norms" for Hecke characters and we are going to compare them for the characters vanishing on certain compact subgroups. Later we establish that there are only finitely many characters vanishing on such subgroup of bounded ``norm". 
The analogues for the characters of $[\TTa(\AAF)]$ are stated and proven in \ref{rintaj}.
\subsection{} In this paragraph we recall several facts about characters of locally compact topological groups.

If~$G$ is a locally compact topological group, by a character of~$G$ we mean a continuous homomorphism $G\to S^1$. Let $G^*$ be the group of characters of~$G$ (\cite[Definition 2, \no 1, \S 1, Chapter II]{TSpectrale}). The group $G^*$ is locally compact by \cite[Corollary 2, \no 1, \S 1, Chapter II]{TSpectrale}. A morphism of topological groups $\phi:G\to G'$ induces a continuous homomorphism $\phi^*:(G')^*\to G^*, \phi^*(\chi)=\chi\circ\phi$ (see \cite[Paragraph \no 7, \S 1, Chapter II]{TSpectrale}). If~$A$ is a subset of~$G$, by $A^{\perp}$ we denote the subgroup of $G^*$ given by the characters vanishing on~$A$. 
\begin{prop}[{\cite[Theorem 4, \no 7, \S 1, Chapter II]{TSpectrale}}] \label{tspgj}
Let~$G$ be a commutative Hausdorff locally compact group. Let $i:G_1\to G$ be the inclusion of a closed subgroup $G_1$ into $G.$ Let $G_2=G/G_1$ and let $p:G\to G_2$ be the quotient map. In the sequence $$G_2^*\xrightarrow{p^*}G^*\xrightarrow{i^*}G_1^*,$$ the homomorphism $p^*$ is an isomorphism of $G_2^*$ onto $G_1^{\perp}$ and $i^*$ is a strict homomorphism $G^*\to G_1^*$ of kernel $G_1^{\perp}$ (it follows that $G_1^{\perp}$ is closed).
\end{prop}
Occasionally, we may identify $G_2^*$ with its image $G_1^{\perp}$ under the homomorphism $p^*$. Recall also that if  $G=H_1\times\cdots H_n$, the canonical homomorphism $G^*\to (H^1)^*\times \cdots\times (H^n)^*$ is an isomorphism of topological groups (\cite[Corollary 5, \no 7, \S 1, Chapter II]{TSpectrale}) and we may identify $G^*$ with $(H^1)^*\times \cdots\times (H^n)^*$ using this isomorphism.
\subsection{}\label{rinta} We define two ``norms" of Hecke characters and we compare them. For $\vMFi$,  in \ref{pojac} we have defined $$F_{v,1}=\{x|x\in\Fvt: |x|_v=1\}.$$
We have furthermore established that $$\widetilde\rho_v:\RR_{>0}\times F_{v,1}\xrightarrow{\sim}\Fvt \hspace{1cm} (r,z)\mapsto \rho_v(r)z,$$where $\rho_v:\RR_{>0}\to \Fvt$ is defined by $\rho_v(r)=r^{1/n_v}$, is an isomorphism of abelian topological groups. For a character $\chi_v\in(\Fvt)^*$, we set $m(\chi_v)$ to be the unique real number~$m$ such that the character $\chi_v\circ\widetilde\rho_v:\RR_{>0}\to S^1$ is given by $r\mapsto r^{im}$. If~$v$ is a real place, we set $\ell(\chi_v)$ to be $0$ if the character $\chi_v\circ\widetilde\rho_v:F_{v,1}\to S^1$ is the trivial character, otherwise we set $\ell(\chi_v)=1$. If~$v$ is a complex place, we set $\ell(\chi_v)$ to be the unique integer~$\ell$ such that $\chi_v\circ\widetilde\rho_v:F_{v,1}\to S^1$ is given by $z\mapsto z^{\ell}$.

Let $\AAFt$ be the group of ideles of~$F$ and let $\AAF^1$ be the subgroup of $\AAFt$ given by $(x_v)_v$ which satisfy that $\prod_v|x_v|_v=1$. For a character $\chi\in (\AAF^\times)^*$, let us define \begin{align*}||\chi||_{\discrete}&:=\max _{\vMFi}(||\ell(\chi _v)||),\\ ||\chi ||_\infty&:=\max_{\vMFi}(||m(\chi _v)||).\end{align*}
\label{linta}
Let $K^0_{\max}$ be the topological group $\prod _{\vMFz}\Ovt. $ For an open subgroup $K\subset K^0_{\max}$, we let $\AK$ be the subgroup of $(\AAF^{1})^*$ given by the characters vanishing on $F^\times$ (technically we mean $i(\Ft)$) and on the compact subgroup $K\times\prod _{\vMFi}\{1\}\subset (\AAF^1)^*$. By the abuse of notation, we may write sometimes~$K$ for what is technically $K\times\prod_{\vMFi}\{1\}$. The group $K^0_{\max}$ is compact, therefore, the subgroup~$K$ is of finite index in $K^0_{\max}$. 

The following lemma will be used on several occasions:
\begin{lem}\label{ovijo}
Let~$G$ be an abelian group and let~$A$ and $B$ be two subgroups such that $A\subset B$. 
\begin{enumerate}
\item Let $H\subset G$ be a subgroup.  The homomorphism $B/A\to (B+H)/(A+H)$ induced from~$A$-invariant homomorphism $B\to (B+H)/(A+H)$ is surjective.
\item Suppose~$G$ is a topological group and that $(B:A)$ is finite. Then $(A^{\perp}:B^{\perp})$ is finite and $(A^{\perp}:B^{\perp})\leq (B:A)$.
\end{enumerate}
\end{lem}
\begin{proof}
\begin{enumerate}
\item We have the following commutative diagram, with first two horizontal and all three vertical sequences exact:
\[\begin{tikzcd}[column sep=small]
	0 & A & {B} & {B/A} & 0 \\
	0 & {A+H} & {B+H} & (B+H)/(A+H) & 0 \\
	& {(A+H)/A} & {(B+H)/B} &E &0,
	\arrow[from=1-1, to=1-2]
	\arrow[from=1-2, to=1-3]
	\arrow[from=1-3, to=1-4]
	\arrow[from=1-4, to=1-5]
	\arrow[from=2-1, to=2-2]
	\arrow[from=2-2, to=2-3]
	\arrow[from=1-2, to=2-2]
	\arrow[from=1-3, to=2-3]
	\arrow[from=2-2, to=3-2]
	\arrow[from=3-2, to=3-3]
	\arrow[from=2-3, to=3-3]
	\arrow[from=3-3, to=3-4]
	\arrow[from=2-3, to=2-4]
	\arrow[from=1-4, to=2-4]
	\arrow[from=2-4, to=2-5]
	\arrow[from=3-4, to=3-5]
	\arrow[from=2-4, to=3-4]
\end{tikzcd}\]
where $E=\coker((B/A)\to (B+H)/(A+H))$. By Snake Lemma, the third horizontal sequence is exact. By the second isomorphism theorem, the homomorphism $(A+H)/A\to (B+H)/B$ identifies with the homomorphism $H/(H\cap A)\to H/(H\cap B)$ induced from the inclusion $(H\cap A)\subset (H\cap B)$, hence is surjective and thus $E=0$. It follows that $B/A\to(B+H)/(A+H)$ is surjective. 
\item The kernel of the homomorphism $$A^\perp\to \Hom_{\ZZ} (B,S^1)\hspace{1cm}\chi\mapsto \chi|_B$$ is the subgroup $B^{\perp}$. We deduce an injective homomorphism \begin{equation}\label{yuvog}A^\perp/B^\perp\to \Hom_{\ZZ}(B,S^1).\end{equation} The image of $(\ref{yuvog})$ is contained in the subgroup of $\Hom_{\ZZ}(B,S^1)$ given by the homomorphisms which vanish on~$A$, i.e. in the image of the canonical homomorphism $\Hom_{\ZZ}(B/A,S^1)\to\Hom_{\ZZ}(B,S^1)$. It follows that $$(A^{\perp}:B^{\perp}) \leq |\Hom_{\ZZ}(B/A,S^1)|=(B:A).$$
\end{enumerate}
\end{proof}
\begin{cor}\label{yrtre} One has that $$(K^0_{\max}:K)\geq ((\Ft K^0_{\max}):(\Ft K))\geq (\AK:\mathfrak A_{K^0_{\max}}).$$
\end{cor}
\begin{proof}
To obtain the first inequality, we apply \ref{ovijo} for $G=\AAF^1$, $A=K$, $B=K^0_{\max}$ and $H=\Ft$. It follows that the homomorphism from \ref{ovijo} $$K^0_{\max}/K\to (\Ft K^0_{\max}/\Ft K)$$ is surjective, thus $(K^0_{\max}:K)\geq ((\Ft K^0_{\max}):(\Ft K))$. The second inequality is the case of \ref{ovijo} for $G=\AAF^1$, $A=\Ft K$ and $B=\Ft K^0_{\max}$.
\end{proof}
\begin{lem}\label{vnogok}
Let $K\subset K^0_{\max}$ be an open subgroup. The group~$\AK$ is finitely generated and its rank is at most $r_2$, where $r_2$ is the number of complex places of~$F$.
\end{lem}
\begin{proof}
The group $\mathfrak A_{K^0_{\max}}$ is the group of Hecke characters $(\AAF^1/\Ft)\to S^1$ which are unramified at the finite places of~$F$ and we deduce that the kernel of the homomorphism $$\phi:\mathfrak A_{K^0_{\max}}\to\prod _{\vMFi}(F_{v,1})^*\hspace{2cm}\chi \mapsto \prod _{\vMFi}\chi_v|_{F_{v,1}} $$ is given by the unramified Hecke characters $(\AAF^1/\Ft)\to S^1$, hence is finite by \cite{Tatesthesis}. The group $(\prod _{\vMFi}(F_{v,1})^*)$ is finitely generated and of rank $r_2$, because $F_{v,1}^*$ is of order $2$ if~$v$ is real and is an infinite cyclic group if~$v$ is complex. It follows that $$\rk(\mathfrak A_{K^0_{\max}})=\rk (\Imm(\phi))\leq r_2.$$ By \ref{yrtre}, one has that $\mathfrak A_{ K^0_{\max}}$ is of finite index in $\AK,$ thus one has that $\rk(\AK)=\rk(\mathfrak A_{K^0_{\max}})\leq r_2.$ The statement is proven.\end{proof}
The following lemma will be used in the proof of Proposition \ref{Ugala}
\begin{lem}
Let~$K$ be an open subgroup of $K^0_{\max}$. Consider the homomorphism $$\ell^{\CC}:(\AAF^1/\Ft)^*\to\ZZ^{M_F^\CC}\hspace{1.5cm}\chi\mapsto(\ell(\chi_v))_{v\in M_F^\CC} $$ The group $\ker (\ell^\CC)\cap\mathfrak A_K$ is finite. \label{rudori}
\end{lem}
\begin{proof} Firstly, let us establish that $\ker (\ell^{\CC})\cap\mathfrak A_{K^0_{\max}}$ is finite. The group $\ker (\ell^\CC)$ (respectively, the group $\mathfrak A_{K^0_{\max}}$) is the group of Hecke characters $(\AAF^1/\Ft)\to S^1$ which are unramified at the complex (respectively, at the finite) places of~$F$. Hence, the kernel of the map $$\ker (\ell^\CC)\cap\mathfrak A_{K^0_{\max}}\to \prod _{v\in M_F^\RR}(F_{v,1})^* \hspace{2cm}\chi\mapsto \prod _{v\in M_F^\RR}\chi_v|_{F_{v,1}} $$ is given by the unramified Hecke characters $(\AAF^1/\Ft)\to S^1$, therefore is finite. As $F_{v,1}^*$ is a cyclic group of order $2$, the group $\prod_{v\in M_F^{\RR}}(F_{v,1})^*$ is finite and we conclude that $\ker (\ell^{\CC})\cap \mathfrak A_{K^0_{\max}}$ is finite. Now let us establish that $\ker(\ell^{\CC})\cap\mathfrak A_{K^0_{\max}}$ is of finite index in $\ker (\ell^{\CC})\cap\AK$. By applying Snake lemma to the ``snake" diagram:
\[\begin{tikzcd}
	&&& 1 \\
	1 & \ker (\ell^{\CC})\cap{\mathfrak A_{K_{\max}^0}}& \mathfrak A_{K_{\max}^0}  & l^\CC({\mathfrak A_{K_{\max}^0}})& 1 \\
	1 & \ker (\ell^{\CC})\cap\AK& \mathfrak A_K & \ell^{\CC}(\mathfrak A_K) & 1 \\
	&E& \AK/ \mathfrak A_{K^0_{\max}},
	\arrow[from=3-1, to=3-2]
	\arrow[from=2-2, to=2-3]
	\arrow[from=2-2, to=3-2]
	\arrow[from=3-2, to=3-3]
	\arrow[from=2-3, to=3-3]
	\arrow["{\ell^{\mathbb C}}", from=3-3, to=3-4]
	\arrow["{\ell^{\mathbb C}}", from=2-3, to=2-4]
	\arrow[from=2-4, to=3-4]
	\arrow[from=3-4, to=3-5]
	\arrow[from=2-4, to=2-5]
	\arrow[from=1-4, to=2-4]
	\arrow[from=3-2, to=4-2]
	\arrow[from=3-3, to=4-3]
	\arrow[from=2-1, to=2-2]
	\arrow[from=4-2, to=4-3]
\end{tikzcd}\]
where $E=(\ker (\ell^{\CC})\cap\AK)/(\ker (\ell^{\CC})\cap{\mathfrak A_{K_{\max}^0}}) $ we get an exact sequence $$1\to E\to \AK/ \mathfrak A_{K^0_{\max}}.$$ By \ref{yrtre}, one has that $\AK/\mathfrak A_{K^0_{\max}}$ is finite, thus  $E=(\ker (\ell^{\CC})\cap\AK)/(\ker (\ell^{\CC})\cap{\mathfrak A_{K_{\max}^0}})$ is finite. Using the fact that $\ker (\ell^{\CC})\cap \mathfrak A_{K^0_{\max}}$ is finite, we deduce that $\ker(\ell^\CC)\cap\AK$ is finite.
\end{proof}
The main proposition of this paragraph is the following one.
\begin{prop} \label{Ugala}
For every open subgroup $K\subset K^0_{\max}$, there exists a constant $C=C(K)>0$ such that $$||\chi||_\infty\leq C||\chi||_{\discrete}$$ for all $\chi\in\mathfrak A_K$.
\end{prop}
\begin{proof} Let~$K$ be an open subgroup of $K^0_{\max}$. 
Let $\ell^\CC:(\AAF^1/\Ft)^*\to \ZZ^{M_F^{\CC}}=\ZZ^{r_2}$ be as in Lemma \ref{rudori}. To simplify notation, in the rest of the proof we will write $\ell^{\CC}$ for $\ell^{\CC}|_{\AK}$. The abelian group $\ell^\CC(\mathfrak A_K)$ is finitely generated and free, let us pick $\chi_1\doots \chi_k\in\mathfrak A_K$ such that $\ell^{\CC}(\chi_1)\doots \ell^{\CC}(\chi_k)$ is a basis of $\ell^{\CC}(\mathfrak A^K)$. Obviously, $k\leq |M_F^{\CC}|=r_2$. One has an isomorphism $$\ZZ^k\to\ell^{\CC}(\mathfrak A_K)\hspace{1.5cm} (d_1\doots d_k)\mapsto d_1\ell^{\CC}(\chi_1)+\cdots +d_k\ell^{\CC}(\chi_k).$$ One can pick a section to the surjective homomorphism $\ell^{\CC}:\mathfrak A_K\to \ell^{\CC}(\mathfrak A_K)$, we obtain an induced splitting $$\ell^{\CC}(\mathfrak A_K)\oplus \ker (\ell^{\CC})\xrightarrow{\sim}\mathfrak A_K.$$  We deduce an identification \begin{equation}\label{nbbg}\ZZ^{k}\oplus \ker (\ell^{\CC})\xrightarrow{\sim} \mathfrak A_K.\end{equation}
A character $\chi\in\mathfrak A_K$ writes as $\chi=p_2(\chi)\chi _1^{q_1(p_1(\chi))}\cdots\chi_k^{q_k(p_1(\chi))}$. Let us firstly estimate $||\chi||_{\infty}$ for $\chi\in\AK$. The group $\ker(\ell^{\CC})$ is finite by \ref{rudori} and let $r$ be its order. For every $\chi\in \ker (\ell^{\CC})$ one has that $$\mathbf 0=(m(1_v))_v=(m(\chi_v^r))_v=(rm(\chi_v))_v,$$ hence $m(\chi_v)=0$ for every $\vMFi$. 
Now, we can make the following estimate for every character $\chi\in\mathfrak A_K$: \begin{align*}||\chi||_\infty&=\max_{\vMFi}|m(\chi_v)|\\
&=\max_{v\in M_F^\infty}|m((p_2(\chi))_v)+m(\chi _{1v})q_1(p_1(\chi))+\cdots m(\chi _{kv}){q_k(p_1(\chi))}|\\
&=\max_{\vMFi}| m(\chi _{1v})q_1(p_1(\chi))+\cdots m(\chi _{kv}){q_k(p_1(\chi))}|\\
&\leq\max_{\vMFi,i=1\doots k}|m(\chi _{iv})|\max _{i=1\doots k}|q_i(p_1(\chi))|\\
&=C_0\max_{i=1\doots k}|q_i(p_1(\chi))|,\end{align*}where we have shortened $C_0= \max_{\vMFi,i=1\doots k}|m(\chi _{iv})|.$ 
Let us now estimate $||\chi||_{\infty}$ for $\chi\in\AK$. For every $\chi\in\AK$ we have that: \begin{align*}||\chi ||_{\discrete}&=\max _{\vMFi}|\ell(\chi_v)|\\
&\geq \max_{v\in M_F^{\CC}}|\ell(\chi_v)|\\
&= \max _{v\in M_F^{\CC}}|\ell\big((p_2(\chi))_v\chi _{1v}^{q_1(p_1(\chi))}\cdots \chi _{kv}^{q_k(p_1(\chi))}\big)|\\
&= \max _{\vMFC}|\ell((p_2(\chi))_v)+q_1(p_1(\chi))\ell(\chi_{1v})+\cdots+q_k(p_1(\chi))\ell(\chi_{kv})|\\
&=\max _{\vMFC}|q_1(p_1(\chi))\ell(\chi_{1v})+\cdots +q_k(p_1(\chi))\ell(\chi _{kv})|.
\end{align*}
Tensoring the injective map $\ell^{\CC}|_{\ZZ^{k}}$ with the flat $\ZZ$-module $\RR$ gives an injective map $\RR^{k}\to\RR^{r_2}$. The pullback of the norm $(x_v)_{\vMFC}\mapsto \max _v|x_v|$ on $\RR^{r_2}$ along this map is given by $$\xxx\mapsto \max_{\vMFC}|x_1\ell(\chi_{1v})+\cdots x_k\ell(\chi_{k,v})|, $$ and is a norm on $\RR^{k}$. One can find, hence, a constant $C_1>0$ such that $$C_1\max_{\vMFC}|x_1\ell(\chi_{1v})+\cdots +x_k\ell(\chi_{k,v})|\geq \max _{i=1\doots k}|x_k| $$ for every $\xxx\in\RR^k$. Thus for every $\chi\in\AK$ one has that $$C_1||\chi||_{\discrete}\geq \max_{i=1\doots k}|q_k(p_1(\chi))|.$$
Let us now prove the wanted inequality. We set $C=C_1C_0$. For every $\chi\in\AK$, we have that \begin{align*}
C||\chi||_{\discrete}\geq C_0C_1||\chi||_{\discrete}\geq C_0\max_{i=1\doots k}|q_i(p_1(\chi))|\geq ||\chi||_{\infty}.
\end{align*}
The statement is proven
\end{proof}
\subsection{} In this paragraph, we bound the number of characters $\chi$ for which $||\chi||_{\discrete}$ is bounded.
\begin{lem} \label{gubi}
Let $h_F=|\Cl(F)|$ be the class number of~$F$. For every open subgroup $K\subset K^{0}_{\max}$ and any $C>0$, there are no more than $$h_F(K^0_{\max}:K)2^{r_1}(2C+1)^{r_2}$$ characters $\chi\in\mathfrak A_K$ for which $||\chi||_{\discrete}\leq C$ (here $r_1$ and $r_2$ are numbers of real and complex places of~$F$, respectively).
\end{lem}
\begin{proof}
Let $K\subset K^0_{\max}$ be an open subgroup.
For $\vMFC$, we define $\widetilde \ell(\chi_v):=\ell(\chi_v)$, and for $\vMF ^{\RR}$ we define $\widetilde \ell(\chi_v)\in\ZZ/2\ZZ$ by \begin{align}
 \quad
\widetilde \ell(\chi_v):=&
               \begin{cases}
0&\text{if $\chi_v|_{F_{v,1}}$ is the trivial character on $F_{v,1}$}\\
1&\text{otherwise.}
               \end{cases}
               \end{align}
(The difference between~$\ell$ and $\widetilde \ell$ is of technical nature, recall that we have defined $\ell(\chi_v)\in\ZZ$ for $\vMFR$.) 
For a character $\chi\in\mathfrak A_K$,  the $r_1+r_2$-tuple given by $(\ell(\chi _v))_{\vMFi}\in(\ZZ/2\ZZ)^{r_1}\times\ZZ^{r_2}$ will be called the signature of $\chi$. 

Let us estimate the number of characters with fixed signature. 
Let $(\ell_v)_{\vMFi}\in(\ZZ/2\ZZ)^{r_1}\times\ZZ^{r_2}$ be a signature of some character $\delta\in\mathfrak A_K$. Then the characters in $\mathfrak A_K$ having $(\ell_v)_{\vMFi}$ for the signature are in a bijection with the characters in $\mathfrak A_K$ having $(0)_{\vMFi}$ for the signature. Indeed, a bijection between the two sets is given by $\chi\mapsto\chi\delta^{-1}$. The group of $\chi\in \mathfrak A_K$ having $(0)_{\vMFi}$ for signature is given by the subgroup $(F^\times ({K}\times \prod _{\vMFi}F_{v,1}))^{\perp}$ of $(\AAF^1)^*.$ 
Lemma \ref{ovijo} gives that 
\begin{align*}(K^0_{\max}:K)&\geq  \bigg(\big(F^\times (K_{\max}^0\times \prod_{\vMFi}F_{v,1})\big):\big( F^\times (K\times \prod_{\vMFi}F_{v,1})\big)\bigg)\\
&\geq \bigg((F^\times ({K}\times \prod _{\vMFi}F_{v,1}))^{\perp}:(F^\times ({K^0_{\max}}\times \prod _{\vMFi}F_{v,1}))^{\perp}\bigg).
\end{align*}
The group $(F^\times ({K^0_{\max}}\times \prod _{\vMFi}F_{v,1}))^{\perp}$ is the group of the unramified Hecke characters $(\AAF^1/\Ft)\to S^1$ and its order is $h_F$. 
It follows that the number of characters $\chi\in\AK$ having $(0)_{\vMFi}$ for the signature is $$|(F^\times ({K}\times \prod _{\vMFi}F_{v,1}))^{\perp}|\leq h_F(K^0_{\max}:K).$$ 
Let us estimate the number of signatures when $|\lvert\cdot\rvert|_{\discrete}$ is bounded. Note that if for $C>0$ and a character $\chi\in\mathfrak A_{K}$ one has $||\chi||_{\discrete}<C$, then the signature of $\chi$ lies in $(\ZZ/2\ZZ)^{r_1}\times ([-C,C]^{r_2}\cap \ZZ^{r_2})$. Thus the number of signatures can be bounded by $2^{r_1}(2C+1)^{r_2}$.

It follows that the number of characters $\chi\in\AK$ such that $||\chi||_{\discrete}\leq~C$ is bounded by $h_F(K^0_{\max}:K)2^{r_1}(2C+1)^{r_2}$. The statement is proven.
\end{proof}
\section{Results for the characters of $[\TTa(\AAF)]$} 
\label{rintaj}
Let $n\geq 1$ be an integer and let $\aaa\in\ZZ^n_{\geq 1}$. We make analogous estimates to those in \ref{scaft} for characters of $[\TTa(\AAF)]$. The results are simple consequences of the corresponding results in \ref{scaft}

\subsection{}\label{arintaj} In this paragraph we explain our notation and define ``norms" for the characters of $[\TTa(\AAF)]$ and compare them for characters vanishing on certain compact subgroups of $[\TTa(\AAF)]$. In this paragraph, we discuss local characters.

Recall that the subgroups $(\Fvt)_{\aaa}\subset\Fvtn$ and $(\AAFt)_{\aaa}\subset(\AAFt)^n$ are closed by \ref{aboutopa} and \ref{mapqaf}, respectively. Recall that $[\TTa(\Fv)]$ identifies with the quotient group $\Fvtn/(\Fvt)_{\aaa}$ by \ref{aboutopa}. Recall that $[\TTa(\AAF)]$ identifies with the quotient group $(\AAFt)^n/(\AAFt)_{\aaa}$ by \ref{mapqaf}.

\begin{lem} \label{senjak}
The subgroup $(\Fvt)_{\aaa}^{\perp}$ is the subgroup $$\{(\chi_j)_j|\prod_j\chi_j^{a_j}=1\}$$ of $((\Fvt)^*)^n$. Moreover, $(\Fvt)_{\aaa}^{\perp}$ is closed in $((\Fvt)^*)^n,$ and it is the image of the pullback homomorphism $[\TTa(\Fv)]^*\to (\Fvtn)^*$. The subgroup $(\AAFt)_{\aaa}^{\perp}$ is the subgroup $\{(\chi_j)_j|\prod_j\chi_j^{a_j}=1\}$ of $((\AAFt)^*)^n$. Moreover, $(\AAFt)_{\aaa}^{\perp}$ is closed in $((\AAFt)^*)^n,$ and it is the image of the pullback homomorphism $[\TTa(\AAF)]^*\to ((\AAFt)^*)^n$.
\end{lem}
\begin{proof}
 Note that if $\chi=(\chi_1\doots \chi_n)$ is a character of $((\Fvt)^*)^n$ which vanishes on $(\Fvt)_{\aaa}=\{(t^{a_j})_j|t\in\Fvt\},$ then for every $t\in\Fvt$, one has that $$1=\chi_1(t^{a_1}) \cdots \chi_n(t^{a_n})=\chi_1(t)^{a_1}\cdots \chi_n(t)^{a_n}.$$ The first claim follows.
 The facts that $(\Fvt)_{\aaa}^{\perp}$ is closed and that it is the image of the pullback homomorphism follow from \ref{tspgj}. Analogous argument shows the claims about $(\AAFt)_{\aaa}^{\perp}$.
\end{proof}
Let~$v$ be a place of $F.$ For $j\in\{1\doots n\},$ we let $\delta^j_v:\Fvt\to (\Fvt)^n$ be the inclusion $$x\mapsto ((x_j)_j,(1)_{i\neq j}).$$ 
If $\chi\in[\TTa(\Fv)]^*$ is a character we let $\chi^{(j)}$ be the character $\chi\circ\qav\circ\delta^j:\Fvt\to S^1$. The lemma from above gives that that $\prodjn (\chi^{(j)})^{a_j}=1$ for every $\chi\in[\TTa(\Fv)]^*$. We define $\mmm(\chi):=(m(\chi^{(j)}))_j$ and $\llll(\chi)=(\ell(\chi^{(j)}))_j.$

If $\chi$ is a character of $[\TTa(\AAF)]$, for $j\in\{1\doots n\},$ we will denote by $\delta^j_{\AAF}$ the inclusion $x\mapsto ((x)_j,(1)_{i\neq j})$ and by $\chij$ the character $\chi\circ q^{\aaa}_{\AAF}\circ \delta ^j_{\AAF}:\AAFt\to S^1$. The lemma from above gives that $\prod _{j=1}^n(\chi^{(j)})^{a_j}=1 $ for every $\chi\in[\TTa(\AAF)]^*$. Moreover, it is immediate that$$\chi|_{[\TTa(i)]([\TTa(F)])}=1\implies \text{ for $j=1\doots n$ one has } \chi^{(j)}|_{i(\Ft)}=1.$$
For a character $\chi\in[\TTa (\AAF)]^*$, we denote by $\chi_v$ the restriction $\chi|_{[\TTa(F_v)]}$. 
\begin{rem}\normalfont
Let $\chi $ be a character of $[\TTa(\AAF)]$. A priori the notation $\chij _v$ can make confusion as it may design $(\chi ^{(j)})|_{F_v^\times}$ and $(\chi |_{[\TTa(F_v)]})^{(j)}$. The commutativity of the following diagram shows that no confusion arises:%
\[\begin{tikzcd}
	\Fvt & \AAFt \\
	{(\Fvt)^n} & {(\AAFt)^n} \\
	{[\TTa(\Fv)]} & {[\TTa(\AAF)].}
	\arrow["\qav"', from=2-1, to=3-1]
	\arrow["{q^{\aaa}_{\AAF}}", from=2-2, to=3-2]
	\arrow["{\delta^j_{\AAF}}", from=1-2, to=2-2]
	\arrow[from=2-1, to=2-2]
	\arrow[from=3-1, to=3-2]
	\arrow["{\delta^j_v}"', from=1-1, to=2-1]
	\arrow[from=1-1, to=1-2]
\end{tikzcd}\]Namely, for a character $\chi$, the character $(\chij)_v$ is the pullback of $\chi$ for the composite by the two vertical homomorphisms on the right and the most upper horizontal homomorphism, while $(\chi_v)^{(j)} $ is the pullback for the lowest horizontal and then two left vertical homomorphisms. 
\end{rem}
For a character $\chi\in[\TTT ^\aaa(\AAF)]^*$, let us define \begin{align*}||\chi||_{\discrete}&:=\max _{\vMFi}(||\llll(\chi _v)||),\\ ||\chi ||_\infty&:=\max_{\vMFi}(||\mmm(\chi _v)||).\end{align*} 
Note that\begin{align*}||\chi ||_{\discrete}=\max _{\vMFi}(||\llll(\chi_v)||)&= \max _{\vMFi} \left(\max _{j=1\doots n }(|\ell(\chi _v^{(j)})|)\right)\\&= \max _{j=1\doots n}\bigg(\max _{\vMFi}|\ell(\chi _v^{(j)})|\bigg)\\&=\max _{j=1\doots n}||\chi ^{(j)}||_{\discrete};\end{align*}
and analogously $$||\chi||_{\infty}=\max_{j=1\doots n}||\chi^{(j)}||_\infty. $$
\begin{lem}\label{njjchid}
Suppose that $n=1$ and $a_1=a\in\ZZ_{\geq 1}$. For every $\chi\in[\TT(a)(\AAF)]^*$, one has that $||\chi||_{\discrete}\leq 1$ and that $||\chi||_{\infty}=0$.
\end{lem}
\begin{proof} Let $\chi\in[\TTa(\AAF)]^*$. To simplify notation, in this proof we write $\widetilde \chi$ for $\chi^{(1)}$. For every $\chi\in[\TT(a)(\AAF)]^*$, one has that $\chi^a=1$.  It follows that $$(0)_{\vMFi}=(m(1))_{\vMFi}=(m((\widetilde\chi_v)^{a}))_{\vMFi}=(am((\widetilde\chi_v)))_{\vMFi},$$hence $(m(\widetilde\chi_v^a))_{\vMFi}=(0)_{\vMFi}$, and thus $||\widetilde\chi||_{\infty}=0$. We obtain that $$||\chi||_{\infty}=||\widetilde\chi||_{\infty}=0.$$
We have that $$||\chi||_{\discrete}=||\widetilde\chi||_{\discrete}\leq \max(1,\max_{v\in M_F^{\CC}}(\ell(\chi_v))). $$ For every $v\in M_F^{\CC}$, one has that $0=\ell(1)=\ell(\widetilde\chi^a)=a\ell(\widetilde\chi_v),$ hence $\ell(\widetilde\chi_v)=0$. We deduce that $||\chi||_{\discrete}\leq 1.$
\end{proof}
\subsection{} In this paragraph we present analogous results to those of \ref{scaft} for the characters of $[\TTa(\AAF)]$.

Let $K^\aaa_{\max}$ be the topological group given by $$K^{\aaa}_{\max}:=\prod _{v\in M_F^0}[\TTa(\Ov)].$$ 
For an open subgroup~$K$ of $K_{\max}^\aaa$, we let $\mathfrak A_K$ be the subgroup of $[\TTa(\AAF)_1]^*$ given by the characters vanishing on $[\TTa(i)]([\TTa(F)])$ and on $K\times\prod _{\vMFi}\{1\}$ (we may simply write~$K$ for what is technically $K\times\prod _{\vMFi}\{1\}$). The group $K^\aaa_{\max}$ is compact, therefore, the subgroup~$K$ is of finite index in $K^\aaa_{\max}$. We present another corollary of \ref{ovijo}.
\begin{cor}Let $K\subset K^{\aaa}_{\max}$ be an open subgroup. One has that 
\begin{align*}(K^{\aaa}_{\max}:K)&\geq ([\TTa(i)]([\TTa(F)])K^{\aaa}_{\max}:[\TTa(i)]([\TTa(F)])K)\\&\geq (\AK:\mathfrak A_{K^0_{\max}}).\end{align*}
\end{cor}
\begin{proof}
The left inequality follows from \ref{ovijo} for $G=[\TTa(\AAF)]_1,$ $A=K,$ $B=K^\aaa_{\max}$ and $H=[\TTa(i)]([\TTa(F)])$. The right inequality follows from \ref{ovijo} for $G=[\TTa(\AAF)]_1,$ $A=[\TTa(i)]([\TTa(F)])K$ and $B=[\TTa(i)]([\TTa(F)])K^{\aaa}_{\max}$.
\end{proof}
Note that $$q^{\aaa}_{\AAF}\big((K^0_{\max})^n\big)=(\qav((\Ovt)^n))_v=([\TTa(\Ov)])_v=K^{\aaa}_{\max}.$$ If $K\subset K^{\aaa}_{\max}$ is an open subgroup, let us set $$\widetilde K=\bigcap _{j=1}^n(\delta ^j_{\AAF})^{-1}(q^{\aaa}_{\AAF}|_{(K^0_{\max})^n})^{-1}(K),$$ it is an open subgroup of $K^0_{\max}$. For a character $\chi\in\AK$, one has that $$\chi^{(j)}|_{\widetilde K}=(\chi\circ q^{\aaa}_{\AAF}\circ \delta ^j_{\AAF})|_{\widetilde K}=1.$$ 
The following statements are simple corollaries of corresponding statements for characters of $\AAF^1$.
\begin{cor}\label{akfgag}
The group $\AK$ is finitely generated and of rank at most $nr_2$.
\end{cor}
\begin{proof}
Recall that by \ref{vnogok}, the abelian group $\mathfrak A_{\widetilde K}$ is finitely generated and of rank at most $r_2$. The image of $\AK$ under the injective homomorphism $$[\TTa(\AAF)]_1^*\to((\AAF^{1})^n)^*=((\AAF^1 )^*)^n\hspace{1.5cm}\chi\mapsto\chi\circ q^\aaa_{\AAF}=(\chi ^{(j)})_j$$ lies in $(\mathfrak A_{\widetilde K})^n.$  It follows that $\AK$ is finitely generated and of rank at most $nr_2$.
\end{proof}
\begin{cor}\label{yuit}
For every compact group $K\subset K^{\aaa}_{\max}$, there exists a constant $C=C(K)>0$ such that for all $\chi\in\AK$ one has that $$||\chi||_{\infty}\leq C||\chi||_{\discrete}. $$
\end{cor}
\begin{proof}
Proposition \ref{Ugala} gives that there exists constant $C=C(\widetilde K)>0$ such that for every $\chi\in\AK$ one has that $$||\chi||_{\infty}\leq C||\chi||_{\discrete}.$$ For $\chi\in\AK$ we deduce that:
$$||\chi||_{\infty}=\max_{j=1\doots n}||\chi^{(j)}||_\infty\leq C\max_{j=1\doots n}||\chi ^{(j)}||_{0}=C||\chi||_{\discrete}.$$
It follows that $C(K)=C(\widetilde K)$ is the wanted constant.
 \end{proof}
\begin{cor}\label{vatu} 
For every open subgroup $K\subset K^{\aaa}_{\max}$ and every $C>0$, there are no more than
$$(h_F(K^0_{\max}:\widetilde K)2^{r_1}(2C+1)^{r_2})^n$$ characters $\chi\in\AK$ for which $||\chi||_{\discrete}\leq C$.
\end{cor}
\begin{proof}
Let $K\subset K^{\aaa}_{\max}$ be an open subgroup. 
Lemma \ref{gubi} gives that for every $C>0$, there are no more than $$h_F(K^0_{\max}:\widetilde K)2^{r_1}(2C+1)^{r_2}$$ characters in $\mathfrak A_{\widetilde K}$ having $||\chi||_{\discrete}\leq C$. A character $\chi\in\mathfrak A_K$ is completely determined by the characters $\chi^{(j)}\in\mathfrak A_{\widetilde K}$ for $j=1\doots n$. It follows that there are no more than $$(h_F(K^0_{\max}:\widetilde K)2^{r_1}(2C+1)^{r_2})^n $$ characters $\chi\in\mathfrak A_{K}$ having $||\chi||_{\discrete}=\max_j(||\chij||_{\discrete})\leq C$.
\end{proof}
By \ref{njjchid}, we have that when $n=1$ and $a=a_1\in\ZZ_{\geq 1}$, we have that $||\chi||_{\discrete}\leq 1.$ We deduce that 
\begin{cor}\label{finmankchar}
Suppose that $n=1$ and that $a=a_1\in\ZZ_{\geq 1}$. For every open subgroup $K\subset K^a_{\max}$, the group $\mathfrak A_K$ is finite.
\end{cor}
\section{Estimates of Rademacher}\label{EORade} In this section we recall some bounds on the growth of~$L$ functions of Hecke characters in vertical strips.
If $v\in M_F^0$, for $s\in\CC$ and $\chi_v\in(\Fvt)^*$, one defines $$L_v(s,\chi_v ):=\frac{1}{1-\pivv ^s\chi_v(\pi _v)}, $$ and one writes $\zeta _v(s)$ for $L_v(s,1)$. For a character $\chi=(\chi_v)_v:\AAFt\to S^1$, we set $$L(s,\chi):=\prod_{\vMFz}L_v(s,\chi_v),$$ and we denote $$\zeta_F(s)=L(s,1).$$
\subsection{}
In this paragraph we restrict ourselves to the characters of $\AAFj$. The corresponding estimates for the characters of $\AAFt$ will be established in \ref{caaftq}. 

Rademacher in \cite[Theorem 5]{Rademacher} establishes that the growth of the~$L$-function of a character in vertical strip is moderate (i.e. bounded by a polynomial). As the notation there is cumbersome, let us quote the variant given as \cite[Theorem 14.A, Chapter III]{Moreno} as part (1) of the next theorem. Part (2) is \cite[Theorem 3]{Rademacher}.
\begin{thm}\label{radem}
Let $0< \eta \leq\frac 12$ and let $\chi=\prod _{v}\chi _{v}:(\AAFj/\Ft)\to~S^1$ be a Hecke character. Let $\cond(\chi)$ be the conductor ideal of $\chi$. We set $d_F=N_{F/\QQ}(\disc_{F/\QQ}).$ 
\begin{enumerate}
\item {\normalfont  (Moreno, \cite[Theorem 14.A, Chapter III]{Moreno})} Suppose $\chi$ is not the trivial character. 
One has that $$|L(s,\chi)|\leq \zeta_F(1+\eta)\bigg(\frac{d_FN_{F/\QQ}(\cond({\chi}))}{(2\pi)^{[F:\QQ]}}\prod _{\vMFi}|{1+s+m(\chi_{v})}|^{n_v}\bigg) ^{\frac{1+\eta-\Re(s)}2} $$ in the strip $-\eta\leq \Re(s)\leq 1+\eta. $
\item {\normalfont (Rademacher, \cite[Theorem 3]{Rademacher})} One has that
$$\bigg|\frac{\zeta_F(s)(1-s)}{1+s}\bigg|\leq 3\zeta _F(1+\eta)^{[F:\QQ]}\bigg(d_F\bigg(\frac{|1+s|}{2\pi}\bigg)^{[F:\QQ]}\bigg)^{\frac{1+\eta-\Re(s)}2} $$ in the strip $-\eta\leq \Re(s)\leq 1+\eta. $
\end{enumerate}
\end{thm}
The following proposition is a corollary of \ref{radem}. A similar version, but only for unramified Hecke characters, has been invoked in the analysis of height zeta functions in \cite[Theorem 3.2.3]{aniso}, \cite[Corollary 4.2.3]{FonctionsZ}, etc. As our metrics at infinite places are not invariant for the maximal compact subgroups, we present the following version. 
\begin{prop}\label{Rade}
Let $K\subset K_{\max}^0=\prod _{\vMFz}\Ovt$ be an open subgroup.
For every $\epsilon>0$, there exist  $C=C(\epsilon)>0$ and $\delta=\delta(\epsilon)>0$ such that the following conditions are satisfied if provided that $\Re(s)\geq 1-\delta:$
\begin{enumerate}
\item for every non trivial Hecke character $\chi:(\AAFj/\Ft)\to S^1$ which vanishes on $K\subset K_{\max}^0$ one has \begin{equation}
\label{pleqa}
|L(s,\chi)|\leq C\big((1+|\Im (s)|)(1+||\chi||_\infty)\big)^\epsilon; 
\end{equation}  
\item  one has \begin{equation}\label{pleqqa}
\bigg|\frac{(s-1)\zeta_F(s)}{s}\bigg|\leq C(1+|\Im (s)|)^\epsilon. 
\end{equation} 
\end{enumerate}
\end{prop}
\begin{proof}
Let $\frac 16>\epsilon >0$. 
We set $$C=27 \zeta_F\big(1+\frac{\epsilon}{[F:\QQ]}\big)^{[F:\QQ]} d_F^{1/[F:\QQ]}(K^0_{\max}:K)^{1/[F:\QQ]}.$$  We are going to verify that $C$ and $\delta =\frac{\epsilon}{2[F:\QQ]}$ verify the above conditions.

Let us prove (\ref{pleqa}) and (\ref{pleqqa}) in the domain $\Re(s)>\frac 43.$ For a non-trivial character $\chi$ one can estimate$$|L(s,\chi)|\leq\zeta _F(\Re(s))\leq\zeta _F\big(\frac 43\big)\leq\zeta _F\big(1+\frac{\epsilon}{[F:\QQ]}\big)\leq C.$$ One also has $$\bigg|\frac{(s-1)L(s,1)}{s}\bigg|\leq \bigg|1-\frac 1s\bigg|\zeta_F\bigg(\frac 43\bigg)\leq \frac74\zeta _F\bigg(\frac 43\bigg)\leq \frac74\zeta_F(1+\frac{\epsilon}{[F:\QQ]})\leq C. $$ It follows that  (\ref{pleq}) and (\ref{pleqq}) are satisfied in the domain $\Re(s)>\frac 43.$

Now we prove (\ref{pleqa}) and (\ref{pleqqa}) in the domain $1-\frac 1{2[F:\QQ]} \epsilon< \Re(s)< \frac 43$.
Let us set $\eta (s)=\frac{1}{[F:\QQ]}\epsilon+\Re(s)-1$, we have that $$0<\frac{\epsilon}{2[F:\QQ]}<\eta (s)< \frac{1}{6[F:\QQ]}+\frac{4}3-1\leq\frac 12$$ and that $-\eta(s)\leq 0\leq \Re(s)< 1+\eta(s)$. We will apply \ref{radem} for $s$ and $\eta=\eta(s)$. 
The following estimate will be used: for every $s$ in the domain ${1}-\frac{1}{2[F:\QQ]}<\Re(s)<\frac43$ one has \begin{equation}1+|s|
\leq 1+|\Re(s)|+|\Im(s)|<3 (1+|\Im(s)|).\label{nabs} \end{equation}  

Let us firstly prove the estimate for the non-trivial characters. 
Using the first part of Theorem \ref{radem}, 
we deduce that for every $\chi\neq 1$ in the domain $1-\frac{\epsilon}{2[F:\QQ]}<\Re(s)<\frac43$ one has \begin{align*}|L(s,\chi)|\hskip-2cm&\\&\leq |\zeta_F(1+\eta(s))|^{[F:\QQ]} \bigg(\frac{d_FN_{F/\QQ}(\cond(\chi))}{(2\pi)^{[F:\QQ]}}\prod _{\vMFi}|{1+s+m(\chi_v)}|^{n_v}\bigg) ^{\frac{\epsilon}{2[F:\QQ]}}\\
&\leq \zeta_F(1+\frac{\epsilon}{2[F:\QQ]})^{[F:\QQ]}\times\\&\quad\quad\quad\quad\quad\times\bigg(\frac{d_FN_{F/\QQ}(\cond(\chi))}{(2\pi)^{[F:\QQ]}}\prod_{\vMFi}(1+|s|)(1+|m(\chi_v)|)\bigg)^{\frac{\epsilon}{[F:\QQ]}} \\
&\leq \zeta_F(1+\frac{\epsilon}{2[F:\QQ]})^{[F:\QQ]}\times\\&\quad\quad\quad\quad\quad \times\bigg(\frac{d_FN_{F/\QQ}(\cond(\chi))}{(2\pi)^{[F:\QQ]}}\prod _{\vMFi}3({1+|\Im (s)|)(1+|m(\chi_v)|})\bigg) ^{\frac{\epsilon}{[F:\QQ]}}\\
&\leq\zeta_F(1+\frac{\epsilon}{2[F:\QQ]})^{{[F:\QQ]}}\times\\&\quad\quad\quad\quad\quad\times\frac{ ((d_FN_{F/\QQ}(\cond(\chi)))^{\frac 1{[F:\QQ]}}3(({1+|\Im (s)|)(1+||\chi||_{\infty}})))^\epsilon}{(2\pi)^\epsilon} .
\end{align*} Moreover, as $\chi_v$ vanishes at~$K$, we have $$ N_{F/\QQ}(\cond(\chi))\leq \prod _{v\in M^0_F}(\Ovt:K_v)= (K^0_{\max}:K),$$ where $K_v$ is the image in $\Ovt$ of the~$v$-adic projection of~$K$.  
The inequality (\ref{pleqa}) now follows from the observation that\begin{multline*}
\zeta_F(1+\frac{\epsilon}{2[F:\QQ]})^{[F:\QQ]}\bigg(\frac{(d_F(K:K^0_{\max}))^{\frac 1{[F:\QQ]}}3}{2\pi}\bigg)^\epsilon\\\leq 27 \zeta_F\big(1+\frac{\epsilon}{[F:\QQ]}\big)^{[F:\QQ]} d_F^{1/[F:\QQ]}(K^0_{\max}:K)^{1/[F:\QQ]}=C.
\end{multline*}
Let us now consider the trivial character. 
When $\Re (s)>1-\frac {\epsilon}{2[F:\QQ]}>\frac 12$, one has that \begin{equation}\bigg|\frac {3s}{s-1}\bigg|\geq \frac{|s|+1}{|s-1|}\geq \bigg|\frac{s+1}{s-1}\bigg|. \label{bropul}\end{equation}
Using the second part of \ref{radem} and (\ref{bropul}), we deduce that \begin{align*}
\bigg|\frac{(s-1)\zeta_F(s)}{s}\bigg|&\leq 9\zeta_F(1+\frac{\epsilon}{2[F:\QQ]})^{[F:\QQ]} \frac{d_F^{\epsilon/[F:\QQ]}3^{\epsilon}(1+||\Im(s)||)^\epsilon}{(2\pi)^{\epsilon}}\\&\leq C(1+|\Im(s)|)^\epsilon.
\end{align*}
The proposition is proven.
\end{proof}
\subsection{}\label{caaftq} We will now present a bound on the growth in the vertical strips of the~$L$-function of a general Hecke character $\chi:(\AAFt/\Ft)\to S^1$. 
In \ref{identaafj}, we have established an identification: $$\AAF^1\times \RR_{>0}\xrightarrow{\sim}\AAFt.$$ For a character $\chi\in(\RR_{>0})^*$ we denote by $m(\chi)$ the unique real number~$m$ such that $\chi(x)=x^{im}$ for every $x\in\RR_{>0}$. For a character $\chi\in(\AAFt)^*$ we denote by $\chi_0$ the restriction $\chi|_{\AAF^1}$ and we write $m(\chi)$ for $m(\chi|_{\RR_{>0}})$ so that $\chi=\chi_0\lvert\cdot\rvert^{im(\chi)}.$ 

For a character $\chi:\AAFt\to S^1$, one has that \begin{align*}L(s,\chi)=L(s,\chi_0\lvert\cdot\rvert^{im(\chi)})&=\prod_{\vMFz}L_v(s,(\chi_{0})_v\lvert\cdot\rvert_v^{im(\chi)})\\&=\prod_{\vMFz}\frac{1}{1-\pivv^s(\chi _{0})_v(\piv)\pivv^{im(\chi)}}\\&=\prod_{\vMFz}\frac{1}{1-\pivv^{s+im(\chi)}(\chi_0)_v(\piv)}\\&=L(s+im(\chi),\chi_0).\end{align*} 
The following proposition deduces easily from \ref{Rade}.
\begin{cor} \label{finrade}
Let $K\subset K^0_{\max}$ be an open subgroup. For every $\epsilon >0$, there exist $C=C(\epsilon)>0$ and $\delta=\delta(\epsilon)>0$ such that the following conditions are satisfied if provided that $\Re(s)\geq 1-\delta:$
\begin{enumerate}
\item for every non trivial Hecke character $\chi:(\AAFt/\Ft)\to S^1$ with $\chi_0\neq 1$ which vanishes on $K\subset K_{\max}^0$ one has \begin{equation}
\label{pleq}
|L(s,\chi)|\leq C\big((1+|\Im (s)|)(1+||\chi _0||_\infty) (1+|m(\chi)|)\big)^\epsilon; 
\end{equation}  
\item for every Hecke character $\chi:(\AAFt/\Ft)\to S^1$ with $\chi_0=1$ one has that \begin{equation}\label{pleqq}
\bigg|\frac{(s+im(\chi)-1)L(s,\chi)}{s+im(\chi)}\bigg|\leq C((1+|\Im (s)|)(1+|m(\chi)|)^\epsilon. 
\end{equation} 
\end{enumerate}
\end{cor}
%
\begin{proof} 
Let $\epsilon >0$ and let $C=C(\epsilon)$ and $\delta=\delta(\epsilon)>0$ be given by \ref{Rade}.
\begin{enumerate}
\item Let $\chi:\AAFt\to S^1$ be a Hecke character which vanishes on~$K$ such that $\chi_0\neq 1$. Then $\chi_0$ vanishes on~$K$ and Proposition \ref{Rade} gives that \begin{align*}L(s,\chi)=L(s+im(\chi),\chi_0)&\leq C((1+|\Im(s)+m(\chi)|)(1+||\chi_0||_\infty))^\epsilon \\
&\leq C((1+|\Im(s)|)(1+|m(\chi)|)(1+||\chi_0||_\infty))^\epsilon.\end{align*}
\item Let $\chi:\AAFt\to S^1$ be a Hecke character with $\chi_0=1$. Proposition \ref{Rade} gives that \begin{align*}\bigg|\frac{(s+im(\chi)-1)L(s,\chi)}{s+im(\chi)}\bigg|&=\bigg|\frac{(s+im(\chi)-1)L(s+im(\chi),1)}{s+im(\chi)}\bigg|\\&=\bigg|\frac{(s+im(\chi)-1)\zeta_F(s+im(\chi))}{s+im(\chi)}\bigg|\\&\leq C((1+|\Im(s)+m(\chi)|))^\epsilon\\&\leq C ((1+|\Im(s)|)(1+|m(\chi)|))^\epsilon.\end{align*}
\end{enumerate}
\end{proof}
\chapter{Fourier transform of the height function}
\label{Fourier transform of the height function}
In this chapter we analyse the Fourier transform of the height function, when the functions $f_v$ are smooth. Let~$n$ be a positive integer and let $\aaa\in\ZZ^n_{>0}$ if $n\geq 2$ and $a=a_1\in\ZZ_{>1}$ if $n=1$. As before, we use notation $f_v^{\#}$ for the toric $\aaa$-homogenous function $\Fvnz\to\RR_{>0}$ of weighted degree $|\aaa|$. For $\vMFz$, we have established in \ref{davdavdav} that $f_v^{\#}$ is locally constant i.e. smooth.  Let $(f_v:\Fvnz\to\RR_{>0})_v$ be a degree $|\aaa|$ quasi-toric $\aaa$-homogenous family of {\it smooth} functions. 
Let~$S$ be the union of the set consisting of the finite places~$v$ at which $f_v$ is not toric and the set of the infinite places.  
Let $H=H((f_v)_v)$ be the corresponding height on $[\PPP(\aaa)(F)]$. If $\vMF$, for a character $\chi_v\in[\TTa(F_v)]^*$ and $j\in\{1\doots n\}$, we denote by $\chi ^{(j)}_v$ the character $\Fvt\to S^1$ given by $$x\mapsto \chi_v (q_v^{\aaa}((1)_{\substack{k=1\doots n\\k\neq j}},(x)_{k})).$$ 
\section{Local transform}In this section we study local Fourier transform.

\subsection{} In this paragraph we define height pairing. 

For $v\in M_F,$ $\sss\in\CC^n,$ $t\in\Fvt$ and $\xxx\in(\Fvt)^n$ one has that \begin{multline*}f_v(t\cdot\xxx)^{\frac{\aaa\cdot\sss}{|\aaa}}\prodjn |t^{a_j}x_j|_v^{-s_j}=|t|_v^{\frac{\aaa\cdot\sss}{|\aaa|}}f_v(\xxx)^{\frac{\aaa\cdot\sss}{|\aaa|}}\prodjn|t|_v^{-a_js_j} |x_j|_v^{-s_j}\\=f_v(\xxx)^{\frac{\aaa\cdot\sss}{|\aaa|}}\prodjn|x_j|_v^{-s_j}\end{multline*} i.e. for $v\in M_F$ and $\sss\in\CC^n$,  the continuous function \begin{equation}\label{ovavo}(\Fvt)^n\to \CC,\hspace{1cm} \xxx\mapsto f_v(\xxx)^{\aaa\cdot\sss}\prodjn |x_j|_v^{-s_j}\end{equation} is $(\Fvt)_{\aaa}$-invariant. Let $H_v(\sss,-):[\TTa(F)]\to\CC$ be the function induced from $(\Fv)_{\aaa}$-invariant function (\ref{ovavo}). 
For $\xxx\in[\TTa(F)],$ we write $H_v(\sss,\xxx)$ for what is technically $H_v(\sss,[\TTa(i_v)](\xxx)),$ where $[\TTa(i_v)]:[\TTa(i)(F)]\to[\TTa(i)(F_v)]$ is the induced homomorphism from $(F^{\times})_{\aaa}$-invariant homomorphism $$(F^\times)^n\hookrightarrow (\Fvt)^n\to[\TTa(F_v)].$$ 
Note that if $\xxx\in[\TTa(\Ov)]$, one can choose $\widetilde\xxx\in\Ovtn$. We deduce that if $f_v=f_v^\#$ is the toric $\aaa$-homogenous function of weighted degree $|\aaa|$ then for $\sss\in\CC^n$, one has $$H_v^\#(\sss,\xxx)=f^\#_v(\xxx)^{\frac{-\aaa\cdot\sss}{|\aaa|}}\prodjn |x_j|_v^{-s_j}=1.$$ 
\begin{lem}\label{hsxhl} 
Let $\sss\in\CC^n$.
\begin{enumerate}
\item Suppose that $\xxx\in[\TTa(\Ov)]$. One has that $H^{\#}_v(\sss,\xxx)=1$.
\item Let $(\xxx_v)_v\in[\TTa(\AAF)]$. The product $$H(\sss,\xxx):=\prod_{\vMF}H_v(\sss,\xxx)$$ is a finite product.
\item Let $\xxx\in[\TTa(F)].$ One has that $$H(\xxx)^{\frac{\aaa\cdot\sss}{|\aaa|}}=H(\sss,[\TTa(i)](\xxx)).$$ Here $[\TTa(i)]:[\TTa(F)]\to [\TTa(\AAF)]$ is the map induced from $(\Ft)_{\aaa}$-invariant map $$\Ftn\to(\AAFt)^n\to[\TTa(\AAF)],$$where the first map is the diagonal inclusion and the second map is the quotient map.
\end{enumerate}
\end{lem}
\begin{proof}
\begin{enumerate}
\item Let $\wx\in(\Ovt)^n$ be a lift of~$\xxx$. One has that $(\Ovt)^n\subset (\Ov^n-\prod_{j=1}^n\piv^{a_j}\Ov)$ and thus by \ref{davdavdav}, one has $f_v^{\#}|_{(\Ovt)^n}=1$. We deduce that$$H^{\#}_v(\xxx)=f_v^{\#}(\wx)^{\frac{\aaa\cdot\sss}{|\aaa|}}\prod_{j=1}^n|\widetilde x_j|_v^{-s_j}=1.$$
\item By definition of $[\TTa(\AAF)]$ for almost every $\vMF$ one has that $\xxx_v\in[\TTa(\Ov)]$. For almost every~$v$, hence, one has that $H_v(\sss,\xxx_v)=H^{\#}_v(\sss,\xxx_v)=1.$ Thus the product defining $H(\sss,\xxx)$ is a finite product.
\item Let $\wx\in F^{\times n}$ be a lift of~$\xxx$. Recall from \ref{ttaiproduct} that~$v$-th coordinate of $[\TTa(i)](\xxx)$ is $[\TTa(i_v)](\xxx)$.  Using the product formula, we get that \begin{align*}H(\sss,[\TTa(i)](\xxx))&=\prod_vH_v(\sss,[\TTa(i_v)](\xxx))\\&=\prod_v\bigg(f_v(\wx)^{\frac{\aaa\cdot\sss}{|\aaa|}}\prodjn|\widetilde x_j|_v^{-s_j}\bigg)\\&=\bigg(\prod_vf_v(\wx)^{\frac{\aaa\cdot\sss}{|\aaa|}}\bigg) \prod_v\prodjn|\widetilde x_j|_v^{-s_j}\\&=\prod_vf^{\#}_v(\wx)^{\frac{\aaa\cdot\sss}{|\aaa|}}\\&=H(\xxx)^{\frac{\aaa\cdot\sss}{|\aaa|}}.\end{align*} 
\end{enumerate}
\end{proof}
We may write $H(\sss,\xxx)$ for what is technically $H(\sss,[\TTa(i)](\xxx))$.
\subsection{} In this paragraph we establish that the functions $H_v^{-1}(\sss,-)$ are integrable and that their Fourier transforms are holomorphic and bounded in $\sss$.

Let $\vMF$. In \ref{haarttafv}, we have defined a Haar measure on $[\TTa(F_v)]=(\Fvt)^n/(\Fvt)_{\aaa}$ by $(d^*x_{1}\dots d^*x_{n})/d^*x$.   
If $v\in M_F^0$, we have established in Lemma \ref{xiuim} that $\mu_v([\TTa(\Ov)])=\zeta_v(1)^{-n+1}$.  
For $\sss\in\CC^n$ and a character $\chi _v\in[\TTa(F_v)]^*$ we define formally 
\begin{align*}
 \quad 
\widehat H_v(\sss,\chi_v ):=&
               \begin{cases}
\zeta_v(1)^{n-1}\int _{[\TTa (F_v)]}H_v(\sss,-)^{-1}\chi_v \mu _v&\text{ if $\vMFz$,}\\
\int _{[\TTa (F_v)]}H_v(\sss,-)^{-1}\chi_v \mu _v&\text{ if $\vMFi$.}
               \end{cases}
\end{align*}  
In this paragraph, we are going to prove that this integral converges absolutely when $\sss\in\Omega_{>0}$ and that it is a holomorphic function of $\sss$ in this domain. 
The following result, in a bit weaker form, has been given as Lemma 8.3 in \cite{Vgps}
\begin{lem}\label{babay}
Let $B>0$. For every $\epsilon>0$, the integral \begin{equation}\label{yuba}\int_{\{x\in\Fv| \hspace{0.1cm}|x|_v\leq B\}}|x|_v^{s-1}dx_v\end{equation} converges absolutely and uniformly in the domain $s\in\RR_{>\epsilon}+i\RR$. The function that associates to $s$ the value of (\ref{yuba}) is holomorphic in the domain $\RR_{>0}+i\RR$.
\end{lem}
\begin{proof}
Suppose $\vMFz$. Let $r$ be the largest integer satisfying that $(\pivv^{-1})^r\leq B$. For every $\epsilon>0$ and every $s\in\RR_{>\epsilon}+i\RR$, we have that \begin{align*}\int_{\{x\in\Fv| \hspace{0.1cm}|x|_v\leq B\}}|x|_v^{s-1}dx_v&=\sum _{k=-\infty}^{r}\int_{|x|_v=(\pivv^{-1})^k}|x|_v^{s-1}dx\\
&=\sum _{k=-\infty}^{r}\int_{|x|_v=(\pivv^{-1})^{k}}\pivv^{-k(s-1)}dx\\
&=\sum _{k=-\infty}^{r}\pivv^{-k(s-1)}dx(\piv^{-k}\Ov)\\
&=\sum_{k=-\infty}^{r}\pivv^{-ks}\\
&=\sum_{k=-r}^{\infty}\pivv^{ks}\\
&=\pivv^{-rs}\sum_{k=0}^{\infty}\pivv^{ks}. \end{align*} The last series converges absolutely and uniformly in the domain $\RR_{>\epsilon}+i\RR.$  Moreover, $s\mapsto \pivv^{-rs}\sum_{k=0}^{\infty}\pivv^{ks}=\frac{\pivv^{-rs}}{1-\pivv^s}$ is a holomorphic function in the domain $\RR_{>0}+i\RR$. Suppose~$v$ is a real place. For every $\epsilon>0$, we have that $$\int _{|x|_v\leq B}|x|_v^{s-1}dx_v=\int _{|x|\leq B}|x|^{s-1}dx=2\int _{0}^Bx^{s-1}dx$$ converges absolutely and uniformly for $s\in\RR_{>\epsilon}+i\RR$. Moreover, $$s\mapsto 2\int _{0}^Bx^{s-1}dx=2 \frac{x^s}s\bigg|_{x=0}^{x=B}=2\frac{B^s}{s}$$is a holomorphic function in $s$ in the domain $\RR_{>0}+i\RR$. Suppose~$v$ is a complex place. For every $\epsilon>0$, we have that \begin{align*}\int_{|x|_v\leq B}|x|_v^{s-1}dx_v&=\int _{x^2+y^2\leq B}(x^2+y^2)^{s-1}2dxdy\\&=\int _0^{2\pi}\int _{r^2\leq B}r^{2(s-1)}2rdrd\phi\\
&=4\pi\int _{0}^{\sqrt{B}}r^{2s-1}dr
\end{align*} converges absolutely and uniformly for $s\in\RR_{>\epsilon}+i\RR$. Moreover, $$s\mapsto 4\pi\int _{0}^{\sqrt{B}}r^{2s-1}dr=4\pi \frac{r^{2s}}{2s}\bigg|_{r=0}^{\sqrt B}=\frac{2\pi B^s}{s}$$ is a holomorphic function in $s$ in the domain $\RR_{>0}+i\RR$.
The statement is proven.\end{proof}
\begin{prop}\label{stml}
For every $\chi_v\in[\TTa(\Fv)]^*$, the integral defining $\wH_v(\sss,\chi_v)$ converges absolutely in the domain $\sss\in\Omega_{>0}$. Moreover, for every compact $\mathcal K\subset\Omega_{>0} $, there exists $C(\mathcal K)>0$ such that for every $\sss\in\mathcal K$ and every $\chi_v\in[\TTa(\Fv)]^*$, one has that $$|\wH_v(\sss,\chi_v)|\leq C(\mathcal K). $$ 
\end{prop}
\begin{proof}
As our characters are assumed unitary (that is with the values in $S^1$), by the triangle inequality, it suffices to prove the statement when $\chi_v=1$. Let $\mathcal K\subset\Omega_{>0}$ be a compact. 
Let $\omega_v$ be the quotient measure $f_v^{-1}dx_1\dots dx_n/d^*x$ on $(\Fvnz)/\Fvt=[\PPP(\aaa)(\Fv)]$ (see \ref{yumf}). By \ref{icicu}, one has an inequality of the measures $H_v(\mathbf 1,-)\omega_v|_{[\TTa(\Fv)]}=\mu_v$. We deduce that $H_v(\sss,-)^{-1}\in L^1([\TTa(\Fv)],\mu_v)$ if and only if $H_v(\sss,-)^{-1}\in L^1([\TTa(\Fv)], H_v(\mathbf 1,-)\omega_v),$ i.e. if and only if $$H_v(\sss,-)^{-1}H_v(\mathbf 1,-)\in L^1([\TTa(\Fv)],\omega_v).$$ Moreover, if $H_v(\sss,-)^{-1}\in L^1([\TTa(\Fv)],\mu_v),$ then \begin{align*}\int_{[\TTa(\Fv)]}H_v(\sss,-)^{-1}\mu_v&=\int_{[\TTa(\Fv)]}H_v(\sss,-)^{-1}H_v(\mathbf 1,-)\omega_v\\&=\int_{[\PPP(\aaa)(\Fv)]}H_v(\sss,-)^{-1}H_v(\mathbf 1,-)\omega_v,\end{align*} where the last equality follows from the fact that $\omega_v([\PPP(\aaa)(\Fv)]-[\TTa(\Fv)])=0$, which we have established in \ref{ttadmz}. 
In \ref{kave}, we have defined 
a compactly supported function $\kav:\Fvnz\to\RR_{\geq 0}$ which satisfies that for every $\xxx\in\Fvnz$ one has that $\int_{\Fvt}\kav(t\cdot\xxx)d^*t=1$. Proposition \ref{simint} gives that $H_v(\sss,-)^{-1}H_v(\mathbf 1,-)\in L^1([\PPP(\aaa)(\Fv)],\omega_v)$ if and only if 
\begin{align*}((H_v(\sss,-)^{-1}H_v(\mathbf 1,-))\circ\qav)\cdot\kav\hskip-1cm&\\&=\big(f_v(\xxx)^{\frac{-\aaa\cdot\sss}{|\aaa|}}\prodjn|x_j|_v^{s_j}\cdot f_v(\xxx)\prodjn|x_j|^{-1}\big)\cdot\kav\\&=f_v(\xxx)^{\frac{-\aaa\cdot\sss}{|\aaa|}}\prodjn|x_j|_v^{s_j-1}\cdot\kav\\&\in L^1(\Fvnz,dx_1\dots dx_n),
\end{align*} and that if $H_v(\sss,-)^{-1}H_v(\mathbf 1,-)\in L^1([\PPP(\aaa)(\Fv)],\omega_v),$ then \begin{align*}\int_{[\PPP(\aaa)(\Fv)]}H_v(\sss,-)^{-1}H_v(\mathbf 1,-)\omega_v\hskip-1cm&\\&=\int_{\Fvnz}f_v(\xxx)^{\frac{-\aaa\cdot\sss}{|\aaa|}}\prodjn|x_j|_v^{s_j-1}\cdot\kav dx_1\dots dx_n\\
&=\int_{\supp(\kav)}f_v(\xxx)^{\frac{-\aaa\cdot\sss}{|\aaa|}}\prodjn|x_j|_v^{s_j-1}\cdot\kav dx_1\dots dx_n.
\end{align*}
For every $\sss\in \mathcal K$, the function $\xxx\mapsto f_v(\xxx)^{-\frac{\aaa\cdot\Re(\sss)}{|\aaa|}}\kav(\xxx)$ is non vanishing and continuous, moreover it can be uniformly bounded for $\sss\in\mathcal K$ and $\xxx\in\supp(\kav)$. Moreover, as $\kav$ is compactly supported, there exists $B>0$, such that $$\supp(\kav)\subset \{\forall j:|x_j|_v\leq B\}.$$ It follows from \ref{babay} that the integral $$\int_{|x|_v\leq B}\prodjn|x_j|_v^{s_j-1}dx_1\dots dx_n=\prodjn \bigg(\int_{|x|_v\leq B}|x|_v^{s_j-1}dx\bigg)$$ converges absolutely and uniformly for $\Re(\sss)\in\mathcal K.$ Hence, \begin{align*}
\int_{\supp(\kav)}f_v(\xxx)^{\frac{-\aaa\cdot\sss}{|\aaa|}}\prodjn|x_j|_v^{s_j-1}\cdot\kav dx_1\dots dx_n\hskip-1cm&\\&=\int_{[\PPP(\aaa)(\Fv)]}H_v(\sss,-)^{-1}H_v(\mathbf 1,-)\omega_v\\
&=\int_{[\TTa(\Fv)]}H_v(\sss,-)^{-1}\mu_v&\\
\end{align*}
converges absolutely and uniformly for $\Re(\sss)\in\mathcal K$. The statement is proven
\end{proof}
\begin{cor}\label{corhlor}
The function $\sss\mapsto\wH_v(\sss,\chi_v)$ is holomorphic.
\end{cor}
\begin{proof}
We apply the Morera's criterion. Let $\Delta\subset\Omega_{>0}$ be a triangle. By Proposition \ref{stml}, the function $H_v(\sss,-)^{-1}$ is absolutely $\mu_v$-integrable and the function $\wH_v(-,-)$ can be uniformly bounded on $\Delta\times[\TTa(\Fv)]^*$, we deduce that $$\int_{\Delta}\wH_v(\sss,\chi_v)d\sss=\int_{\Delta}\zeta_v(1)^{n-1}\int _{[\TTa(\Fv)]}H_v(\sss,-)^{-1}\chi_v \mu_vd\sss,$$where if $\vMFi$ one sets $\zeta_v(1)=1$, converges absolutely. By using Fubini's theorem, we get that $$\int _{\Delta}\int _{[\TTa(\Fv)]}H_v(\sss,-)^{-1}\chi_v\mu_v d\sss=\int_{[\TTa(\Fv)]}\int_{\Delta}H_v(\sss,-)^{-1}\chi_vd\sss \mu_v=0. $$ By Morera's criterion, $\sss\mapsto\wH_v(\sss,\chi_v)$ is holomorphic. \end{proof}
\section{Calculations in non-archimedean case}
We establish some properties of the local transform in the non-archimedean case.  Firstly we treat the case when $f_v$ is the toric $\aaa$-homogenous function of weighted degree $|\aaa|$ and give the exact value of the integral in \ref{torfour}. 
\subsection{}In this paragraph we calculate the Fourier transform at the finite places~$v$ when $f_v=f_v^{\#}$ is toric.

Let $\vMFz$. Let $\fvt:\Fvnz\to\RR_{>0}$ be the toric $\aaa$-homogenous function of weighted degree $|\aaa|$ (see Definition \ref{tordef}). Recall from \ref{davdavdav}, that $f_v^{\#}|_{\Dav}=1$, where $\Dav=\Ov^n-\prodjn\piv^{a_j}\Ov$. 

When $v\in M_F^0$, we observe that $$\chij_v|_{[\TTa(\Ov)]}=1\implies \chij_v|_{\Ovt}=1\text{ for $j=1\doots n$.}$$
\begin{lem}\label{torfour}
Let $v\in M_F^0$ and let $\fvt:\Fvnz\to\RR_{> 0}$ be the toric $\aaa$-homogenous function of weighted degree $|\aaa|$. Let $\sss\in\Omega_{>0}$ and let $\chi _v\in[\TTa(F_v)]^*$ be a character. 
We have that  \begin{align}
 \quad
\wH_v^{\#}(\sss,\chi_v):=&
               \begin{cases}
\dfrac{\prod _{j=1}^nL_v(s_j,\chij_v)}{\zeta _v(\aaa\cdot\sss)}&\text{if $\chi_v|_{[\TTa(\Ov)]}=1$,}\\
0&\text{otherwise.}
               \end{cases}
               \end{align} 
\end{lem}
\begin{proof}
By applying 
\ref{smacor}, we have that 
\begin{align*}
\int _{[\TTa(F_v)]}\Hvt&(\sss,-)^{-1}\chi\mu_v\hskip-3cm&\\&=\zeta_v(1)\int _{\Fvtn\cap\Dav}f_v^{\#}(\xxx)^{\frac{-\aaa\cdot\sss}{|\aaa|}}\prod _{j=1}^n|x_{j}|_v^{s_j}\chi_v^{(j)}(x_{j})\hspace{0.1cm}{d^*x_1\dots d^*x_n}\\
&=\zeta_v(1)\int _{\Fvtn\cap\Dav}\prod _{j=1}^n|x_{j}|_v^{s_j-1}\chi_v^{(j)}(x_{j})\hspace{0.1cm}{dx_1\dots dx_n}.
\end{align*}
We calculate the last integral as the difference of the integrals of $\prodjn |x_j|_v^{s_j-1}\chij_v(x_j)$ over $(\Ov)^n\cap\Fvtn=(\Ov-\{0\})^n$ and $\big(\prodjn \piv^{a_j}\Ov\big)\cap\Fvtn =\prodjn(\piv^{a_j}\Ov-\{0\}):$
\begin{align*}
\int _{\Fvtn\cap\Dav}\prod _{j=1}^n|x_{j}|_v^{s_j-1}\chi_v^{(j)}(x_{j})\hspace{0.1cm}{dx_1\dots dx_n}\hskip-3cm&\\
&=\int _{(\Ov-\{0\})^n}\prod _{j=1}^n|x_{j}|_v^{s_j-1}\chi ^{(j)}_v(x_{j})dx_1\dots dx_n\\&\hspace{0.3cm}-\int_{\prodjn (\piv^{a_j}\Ov-\{0\})}\big(\prodjn |x_j|_v^{s_j}\chij_v(x_j)\big)dx_1\dots dx_n.
\end{align*}
 Let us integrate over $(\Ov-\{0\})^n$. We have that\begin{equation}\label{qqdo}\int _{(\Ov-\{0\})^n}\prod _{j=1}^n|x_{j}|_v^{s_j-1}\chi ^{(j)}_v(x_{j})dx_1\dots dx_n=\prod _{j=1}^n\int _{\Ov-\{0\}}\chi^{(j)}_v(x)|x|_v^{s_j-1}dx.\end{equation}
When $\Re(s)>0$, we have that
\begin{align*}
\int _{\Ov-\{0\}}\chij_v(x)|x|^{s}d^*x&=\sum _{r=0}^\infty\int_{\piv^r\Ovt}\chij_v(x)\cdot|x|^{s}d^*x\\
&=\sum_{r=0}^\infty\int_{\Ovt}\chij_v(\piv^rx)\cdot|\piv^rx|_v^{s}d^*x\\
&=\sum _{r=0}^\infty\chij_v(\piv)^r\pivv^{rs}\int_{\Ovt}\chij_v(x)d^*x.
\end{align*}
The integral of a non trivial character of compact group for a Haar measure on the group is $0$, while it is the volume of the group if the character is trivial. We deduce that, in the case $\chij_v|_{\Ovt}\neq 1$ one has that$$\int_{\Ov-\{0\}}\chij_v(x)|x|^{s}d^*x=0,$$otherwise \begin{align*}\int _{\Ov-\{0\}}\chij_v(x)|x|_v^{s}d^*x&=\sum_{r=0}^\infty\int_{\piv^r\Ovt}\chij_v(x)|x|_v^sd^*x\\&=\sum_{r=0}^{\infty}\chij_v(\piv^r)\pivv^{rs}d^*x(\piv^r\Ovt)\\&=d^*x(\Ovt)\sum_{r=0}^\infty(\chij(\piv)\pivv^s)^r\\&=\frac{1-\pivv}{1-\chij_v(\piv)\pivv^s}\\&=\zeta_v(1)^{-1}L_v(s,\chij).\end{align*}
The integral over $(\Ov-\{0\})^n$ is hence $$\prodjn (1-\pivv)L_v(s,\chij_v)=\zeta_v(1)^{-n}\prodjn L_v(s,\chij).$$ Let us calculate the integral of $\prodjn |x_j|_v^{s_j-1}\chij_v(x_j)$ over $\prodjn(\piv^{a_j}\Ov-\{0\})$. The~$v$-adic absolute value of the determinant of the Jacobian of the map$$(\Ov-\{0\})^n\to (\piv^{a_j}\Ov-\{0\})_j\hspace{1cm} \xxx\mapsto (\piv^{a_j}x_j)_j$$ is equal to $\pivv^{|\aaa|}$. Using the formula for the change of variables and the fact that $\prodjn (\chij_v)^{a_j}=1$, we get that\begin{align*}\int_{\prodjn (\piv^{a_j}\Ov-\{0\})}\big(\prodjn |x_j|_v^{s_j}\chij_v(x_j)\big)dx_1\dots dx_n\hskip-8cm&\\&=\int_{(\Ov-\{0\})^n}\pivv^{|\aaa|}\big(\prodjn|\piv ^{a_j}x_j|_v^{s_j-1}\chij_v(\piv^{a_j}x_j)\big)dx_1\dots dx_n\\&=\pivv^{|\aaa|+\aaa\cdot(\sss-\jed)}\prodjn\chij_v(\piv)^{a_j}\int_{(\Ov-\{0\})^n}\big(\prodjn |x_j|_v^{s_j-1}\chij_v(x_j)\big)dx_1\dots dx_n\\&=\pivv^{\aaa\cdot\sss}\int_{(\Ov-\{0\})^n}\big(\prodjn |x_j|_v^{s_j-1}\chij_v(x_j)\big)dx_1\dots dx_n\\
&=\pivv^{\aaa\cdot\sss}\prodjn\int_{\Ov-\{0\}}|x|_v^{s_j-1}\chiv^{(j)}(x)dx\\
&=\pivv^{\aaa\cdot\sss}\zeta_v(1)^{-n}\prodjn L_v(s_j,\chiv^{(j)}).
\end{align*}
We deduce that 
$$\int _{\Fvtn\cap\Dav}\prod _{j=1}^n|x_{j}|_v^{s_j-1}\chi_v^{(j)}(x_{j})\hspace{0.1cm}{dx_1\dots dx_n}=0 $$if $\chi|_{[\TTa(\Ov)]}\neq 1$ and 
\begin{align*}\int _{\Fvtn\cap\Dav}\prod _{j=1}^n|x_{j}|_v^{s_j-1}\chi_v^{(j)}(x_{j})\hspace{0.1cm}{dx_1\dots dx_n}\hskip-4cm&\\&=\zeta_v(1)^{-n}\prodjn L(s_j,\chiv^{(j)})-\zeta_v(1)^{-n}\pivv^{\aaa\cdot\sss}\prodjn L_v(s_j,\chiv^{(j)})\\
&=\zeta_v(1)^{-n}\zeta_v(\aaa\cdot\sss)^{-1}\prodjn L_v(s_j,\chi_v^{(j)}),\end{align*}if $\chi|_{[\TTa(\Fv)]}$. Finally, it follows that $$\wH_{v}^{\#}(\sss,\chi)=\zeta_v(1)^{(n-1)}\cdot\zeta_v(1)\int _{\Fvtn\cap\Dav}\prod _{j=1}^n|x_{j}|_v^{s_j-1}\chi_v^{(j)}(x_{j})\hspace{0.1cm}{dx_1\dots dx_n}=0$$if $\chi|_{[\TTa(\Ov)]}\neq 1$ and 
\begin{align*}
\wH_{v}^{\#}(\sss,\chi)&=\zeta_v(1)^{n-1}\cdot \zeta_v(1)^{1}\int _{\Fvtn\cap\Dav}\prod _{j=1}^n|x_{j}|_v^{s_j-1}\chi_v^{(j)}(x_{j})\hspace{0.1cm}{dx_1\dots dx_n}\\&=\zeta_v(\aaa\cdot\sss)\prodjn L_v(s_j,\widetilde\chi_v^{(j)})
\end{align*}
if $\chi_v|_{[\TTa(\Ov)]}=1$. 
The statement is proven.
\end{proof}
\subsection{} Let $\vMFz$. When $f_v$ is assumed to be smooth, we establish in \ref{ajvis} that there exists a compact and open subgroup $K_v\subset[\TTa(\Ov)]$ such that $\wH_v(\sss,\chi_v)=0$ whenever $\chi_v\not\in [\TTa(\Fv)]^*$. 

The following lemma will be used.
\begin{lem}\label{vccmu}
Let $\vMFz$ and let $f_v:\Fvnz\to\RR_{>0}$ be a locally constant $\aaa$-homogenous function of weighted degree $b\geq 0$. There exists an open and compact subgroup $\Lambda_v\subset\Fvtn$ such that for every $(\lambda_j)_j\in\Lambda_v$ and every $(x_j)_j\in\Fvnz$ one has that $$f_v((\lambda_jx_j)_j)=f_v(\xxx). $$ 
\end{lem}
\begin{proof}
We set as before $\Dav:=(\Ov^n)-(\piv^{a_1}\Ov)\times\cdots\times(\piv^{a_n}\Ov)$.  By \ref{begdav}, the set $\Dav\subset\Fvnz$ is open and compact subset of $\Fvnz$. There exists a finite set $\{B(\xxx_i,d_i)\}_i\subset\Dav$ of open balls in $\Fvnz$, where $d_i>0$, which cover $\Dav$ and such that for every~$i$, the restriction $f_v|_{B(\xxx_i,d_i)}$ is constant. The open balls in $\Fvnz$ are also closed and hence compact. For every~$i$ and every $\yyy\in B(\xxx_i,d_i)$,  the set of $(\lambda _j)_j\in\Fvtn$ such that $(\lambda _jy_j)_j\in B(\xxx_i,d_i)$ is an open neighbourhood of $\mathbf 1\in\Fvtn$ and thus contains an open subgroup $\Lambda ^i_{\yyy}\subset\Fvtn$. 
For every~$i$, the open sets $\{\Lambda^i_{\yyy}\cdot\yyy\}_{\yyy\in B(\xxx_i,d_i) }$ form an open covering of $B(\xxx_i,d_i)$ and there exists a finite set of points $\yyy^i_1\doots \yyy^i_{m_i}\in B(\xxx_i,d_i) $ such that $\{\Lambda^i_{\yyy^i _e}\cdot\yyy^i_e\}_{e=1\doots m_i}$ is an open covering of $B(\xxx_i,d_i)$. Now $\Lambda_v :=\cap _i\cap _{e=1}^{m_i}\Lambda^i_{\yyy^i_e}$ is an open subgroup of $\Fvtn$ satisfying that for any $\xxx\in\Dav$ and any $(\lambda_j)_j\in\Lambda_v,$ one has $f_v((\lambda_jx_j)_j)=f_v(\xxx)$. Let $\yyy\in\Fvnz$ and $(\lambda_j)_j\in\Lambda_v$. By \ref{begdav}, there exists $t\in\Fvt$ such that $t\cdot\yyy\in\Dav$. We have $$|t|^{|\aaa|}_vf_v((\lambda_jy_j)_j)=f_v((t^{a_j}\lambda_jy_j)_j)=f_v((\lambda_jt^{a_j}y_j))=f_v((t^{a_j}y_j)_j)=|t|^{|\aaa|}_vf_v(\yyy) ,$$ and the statement follows. 
\end{proof}
\begin{lem}\label{ajvis} 
Let $\sss\in\Omega_{>0}$. Let $v\in M_F^0$ and let $f_v:\Fvnz\to\RR_{\geq 0}$ be locally constant $\aaa$-homogenous function of weighted degree $|\aaa|$.
There exists a compact open subgroup $K_v$ of $[\TTa(F_v)]$ such that $\wH_v(\sss,\chi)=0$, for any character $\chi\in[\TTa (F_v)]^*$ not vanishing on $K_v.$ Moreover, if $f_v=f_v^\#$ is toric, one can choose $K_v=[\TTa(\Ov)]$.
\end{lem}
\begin{proof}
By \ref{vccmu}, there exists an open and compact subgroup $\Lambda_v\subset\Fvtn$ such that $f_v$ is $\Lambda_v$-invariant. Let us set $K_v:=\qav(\Lambda_v\cap\Ovn)\subset[\TTa(\Fv)]$. It is open and compact subgroup of $[\TTa(F_v)]$. Let $\xxx\in[\TTa(F_v)]$ and let $\widetilde\xxx\in\Fvtn$ be a lift of~$\xxx$. Let $\yyy\in K_v$ and let $\widetilde\yyy\in\Lambda\cap\Ovn$ be a lift of~$\yyy$. We have that $$H_v(\sss,\yyy\xxx)=f_v((\widetilde y_j\widetilde x_j)_j)^{\frac{\aaa\cdot\sss}{|\aaa|}}\prod _{j=1}^n|\widetilde y_j\widetilde x_j|_v^{-s_j}=f_v(\wx)^{\frac{\aaa\cdot\sss}{|\aaa|}}\prod _{j=1}^n|\widetilde x_j|^{s_j}=H_v(\sss,\xxx).$$ Therefore $\xxx\mapsto H_v(\sss,\xxx)$, and hence $\xxx\mapsto H_v(\sss,\xxx)^{-1}$ are invariant for the open and compact subgroup $\qav(\Lambda _v\cap\Ovn)\subset[\TTa(F_v)].$ We deduce that for any $\chi_v\in[\TTa(F_v)]^*$ which does not vanish on $K_v$ one has $\wH_v(\sss,\chi_v)=0$. Moreover, in the case $f_v$ is toric, by Lemma \ref{torfour}, one has that $\wH_v(\sss,\chi_v)=0$ for every $\chi_v$ not vanishing on $[\TTa(\Ov)]$.  
\end{proof}
\section{Product of transforms over finite places}  \label{mmhu}
Using the results from local analysis, we establish some growth properties on the product of transforms over all finite places. If $\xxx\in\RR^n$, we will denote by $\Omega_{>\xxx}$ the open subset in $\CC^n$ given by $\{\sss\in\CC^n|\forall j: \Re(s_j)>x_j\}$.

The assumption on $\aaa$ is that $\aaa\in\ZZ^n_{>0}$ if $n\geq 2$, and $\aaa=a\in\ZZ_{>1}$ if $n=1$. As before $f_v$ are assumed to be locally constant for $v\in M_F^0$.\subsection{}
For a character $\chi\in[\TTa(\AAF)]^*$ and any $\sss\in\Omega_{>0},$ we set $$\wH _{\fin}(\sss,\chi):=\prod _{\vMFz} \wH _v(\sss,\chi_v).$$ 
One has following proposition. 
\begin{prop}{\label{fintr}}
Let $\aaa\in\ZZ^n_{>0}$ if $n\geq 2$ and let $\aaa=a_1\in\ZZ_{>1}$ if $n=1$. For every character $\chi\in[\TTa(\AAF)]^*$, the infinite product $\wH_{\fin}(\sss,\chi)$ converges for $\sss\in\Omega_{>1}$ and defines a holomorphic function in the domain $\Omega_{>1}$. There exists a unique holomorphic function $\phi _{\fin}(-,\chi)$ on $\Omega_{>\frac 23}$ such that one has an equality of meromorphic functions in the domain $\Omega_{>\frac23}:$ $$\wH_{\fin}(\sss,\chi)=\phi _{\fin}(\sss,\chi) \prod _{j=1}^nL(s_j,\chij).$$ Moreover, for every compact $\mathcal K\subset \RR^n_{>\frac 23}$, there exists $C(\mathcal K)$ such that $|\phi_{\fin}(\sss,\chi)|\leq C(\mathcal K)$ for every character $\chi\in[\TTa(\AAF)]^*$, provided that $\Re(\sss)\in\mathcal K$.
\end{prop}
\begin{proof}
Let~$S$ be the union of the set of finite places~$v$ for which $f_v$ is not toric and the set of the infinite places. By Corollary \ref{corhlor}, for every character $\chi\in [\TTa (\AAF)]^*$, the function $$\Omega_{>0} \to\CC\hspace{1cm}\sss\mapsto\prod _{v\in S\cap M_F^0}\wH _v(\sss,\chi_v) $$ is holomorphic. For every character $\chi\in[\TTa(\AAF)]^*$ and every $v\in M_F^0-S$,  the functions $$\sss\mapsto\prod _{j=1}^nL_v(s_j,\chij_v)=\prod _{j=1}^n\frac 1{1-\pivv^{s_j}\chij_v(\piv)}$$ and $$\sss\mapsto\zeta_v(\aaa\cdot\sss)=\frac 1{1-\pivv^{\aaa\cdot\sss}}$$ are holomorphic and non vanishing in the domain $\Omega_{>0}$. We deduce that 
for every $\sss\in\Omega_{>1}$ and every $\chi\in[\TTa (\AAF)]^*,$ the product \begin{align*}\prod_{\vMFz-S}\wH_v(\sss,\chi_v)&=\prod _{\vMFz-S}\wH_v^{\#} (\sss,\chi _v)\\&=\prod _{v\in M_F^0-S}\frac{\prod _{j=1}^nL_v(s_j,\chij _v)}{\zeta _v(\aaa\cdot\sss)}\end{align*} converges to $$\dfrac{\prod _{j=1}^nL(s_j,\chij_v)}{\zeta _F(\aaa\cdot\sss)\prod_{v\in S\cap M_F^0}\frac{\prod _{j=1}^nL_v(s_j,\chij_v)}{\zeta _v(\aaa\cdot\sss)}}. $$  
We conclude that for $\sss\in\Omega_{>1}$ and $\chi\in[\TTa(\AAF)]^*$, the product $\wH_{\fin}(\sss,\chi)$ converges to $$\bigg(\prod _{v\in S\cap M_F^0}\wH_v(\sss,\chi_v)\bigg)\cdot\frac{\prod _{j=1}^nL(s_j,\chij _v)}{\zeta_F (\aaa\cdot\sss)\prod_{v\in S\cap M_F^0}\frac{\prod _{j=1}^nL_v(s_j,\chij_v)}{\zeta _v(\aaa\cdot\sss)}} $$ and using holomorphicity of $\sss\mapsto\wH_v(\sss,\chi_v)$ for $\sss\in\Omega_{>0}$, that the resulting function is holomorphic in $\sss$ in the domain $\Omega_{>1}$. 
Let us establish the meromorphic extension to the domain $\Omega_{>\frac23}$. 
We set $$\phi_{\fin} (\sss,\chi):=\frac{1}{\zeta_F(\aaa\cdot\sss)}\prod _{v\in S\cap M_F^0}\frac{ \wH _v(\sss,\chi _v)\zeta_v(\aaa\cdot\sss)}{\prod _{j=1}^nL_v(s_j,\chij)}. $$
If $\sss\in\Omega_{>\frac 23}$, then $\Re(\aaa\cdot\sss){>\frac 43}$ (because if $n\geq 2$ then $\aaa\in\ZZ^n_{>0}$ and if $n=1$ then $\aaa\in\ZZ_{>1}$). Using the fact that the function $\zeta _F$ is holomorphic and without zeros in the domain $\Omega_{>\frac 43}$, we deduce that the function $\sss\mapsto \zeta_F(\aaa\cdot\sss)^{-1}$ is holomorphic in the domain $\Omega_{>\frac 23}$. We have already seen that for $v\in M_F^0\cap S$, the function $\sss\mapsto \big(\prodjn L_v(s_j,\chij_v\big)^{-1}$, the function $\sss\mapsto \zeta_v(\aaa\cdot\sss)$ and $\sss\mapsto \wH_v(\sss,\chi_v)$ are holomorphic in the domain $\sss\in\Omega_{>0}.$ It follows that $\phi(-,\chi)$ is holomorphic in the domain $\Omega_{>\frac23}$ and is the unique holomorphic function which satisfies that $\wH_{\fin}(-,\chi)=\phi(-,\chi)\prodjn L(s_j,\chi^{(j)})$ in this domain. 

Let $\mathcal K\subset\RR^n_{>\frac23}$ be a compact. By Proposition \ref{stml}, for every $v\in S\cap M_F^0$, there exists $C_1>0$ such that for every $\sss\in \mathcal K+i\RR^n$ and every $\chi\in[\TTa(\AAF)]^*$, one has $|\wH_v(\sss,\chi_v)|\leq C_1.$ One has that \begin{equation}\bigg|\frac1{\prod _{j=1}^nL_v(s_j,\chij_v)}\bigg|=\prod _{j=1}^n\left|(1-\chij_v(\pi _v)\pivv^{s_j})\right|\leq\prodjn 2= 2^n \end{equation} for every $\sss\in\Omega_{>\frac 23}$. Note that for $v\in S\cap M_F^0$, we have $$\big|\zeta_v(\aaa\cdot\sss)\big|=\bigg|\frac{1}{1-\pivv^{\aaa\cdot\sss}}\bigg|\leq \frac 1{1-\pivv}=\zeta_v(1),$$ whenever $\sss\in\Omega_{>\frac 23}$. 
Finally, for $\sss\in\Omega_{>\frac 23}$, one has $\Re(\aaa\cdot\sss)>\frac43$, and we have \begin{equation}\bigg|\frac {1}{\zeta _F(\aaa\cdot\sss)}\bigg|\leq\frac{1}{\zeta_F(\Re(\aaa\cdot\sss))}\leq\frac{1}{\zeta_F(\frac43)}. \end{equation} 
We conclude that for $\sss\in\KiRn$ and $\chi\in[\TTa(\AAF)]^*$ one has \begin{align*}|\phi(\sss,\chi)|&=\frac{1}{\zeta_F(\aaa\cdot\sss)}\prod _{v\in S\cap M_F^0}\frac{ \wH _v(\sss,\chi _v)\zeta_v(\aaa\cdot\sss)}{\prod _{j=1}^nL_v(s_j,\chij)}\\&\leq{\frac{(2^nC_1\zeta_v(1))^{|S\cap M_F^0|}}{\zeta_F\big(\frac43\big)}}.\end{align*}
We have proven that $\phi$ is uniformly bounded for $\chi\in[\TTa(\AAF)]^*$ when $\Re(\sss)\in\mathcal K$. The proof is completed.
\end{proof}
\section{Calculations in archimedean case}\label{Archimed}
\label{aboutcharacters}
The goal of this section is to analyse the Fourier transforms of local heights at the infinite places.
%
%
%
\subsection{}
In this paragraph we recall some facts about integration by parts. Let $\vMFi.$

Consider the vector field \begin{align}
 \quad
\frac{\partial}{\partial x_v}:=&
               \begin{cases}
\frac{\partial}{\partial x},&\text{if~$v$ is real}\\
\frac{\partial}{\partial z},&\text{if~$v$ is complex}               \end{cases} 
               \end{align} 
defined on $\Fv$. \begin{lem} \label{grisv}
Let $v\in M_F^\infty$. Let $f,g:\Fvt\to\CC$ be smooth functions. We suppose that $\lim _{|x_v|_v\to\infty}fg(x_v)=0$ and that the functions $\Fvt\to\CC$ given by $x_v\mapsto\frac{\partial f(x_v)}{\partial x_v}g(x_v)$ and $x_v\mapsto f(x_v)\frac{\partial g(x_v)}{\partial x_v}$ are absolutely ${dx_v}$-integrable. If~$v$ is real, we suppose further that $\lim _{x_v\to 0}fg(x_v)$ exists and is a finite real number. One has that $$\int _{\Fv}\frac{\partial f}{\partial x_v}gdx_v=-\int _{\Fv}f\frac{\partial g}{\partial x_v}dx_v.$$
\end{lem}
\begin{proof}
Suppose~$v$ is real. By applying the integration by parts, we get that \begin{align*}\int _{\Fv}\frac{\partial f}{\partial x_v}gdx_v\hskip-1cm&\\&=\int _{\RR}\frac{\partial f}{\partial x}gdx\\
&=\int _{\RR_{>0}}\frac{\partial f}{\partial x}gdx +\int _{\RR_{<0}}\frac{\partial f}{\partial x}gdx\\
&=0-\lim _{x\to 0}fg(x)-\int _{\RR_{>0}}f\frac{\partial g}{\partial x}dx+ \lim _{x\to 0}fg(x)-0-\int _{\RR_{<0}}f\frac{\partial g}{\partial x}dx\\
&=-\int _{\RR}f\frac{\partial g}{\partial x}dx.
 \end{align*}
Suppose~$v$ is complex. By Fubini's theorem we have \begin{align}\begin{split}\int _{\Fv}\frac{\partial f}{\partial x_v}gdx_v\hskip-2cm&\\&=\int _{\RR^2}\bigg(\frac{\partial f(x+iy)}{\partial x}-i\frac{\partial f(x+iy)}{\partial y}\bigg)g (x+iy)dxdy\\
&=\int _\RR dy \int _\RR\frac{\partial f(x+iy)}{\partial x}g(x+iy)dx-\int_{\RR}dx\int_{\RR}i\frac{\partial f(x+iy)}{\partial y} g(x+iy)dy.\label{iyywe}
\end{split}
\end{align}
For every $y\in\RR$, by the conditions of our lemma, one has that $$\lim _{x\to\pm\infty}\frac{\partial f(x+iy)}{\partial x}g(x+iy)=\lim _{x\to\pm\infty}\frac{\partial g(x+iy)}{\partial x}f(x+iy)=0. $$ Using that $z\mapsto\frac{\partial f(z)}{\partial z}g(z)$ and of $z\mapsto f(z)\frac{\partial g(z)}{\partial x_v}$ are absolutely $-idzd\overline z=2dxdy$-integrable, we deduce that $x\mapsto \frac{\partial f(x+iy)}{\partial x}g(x+iy) $ and $x\mapsto f(x+iy)\frac {\partial g(x+iy)}{\partial x}$ are absolutely integrable for almost every~$y$. For such~$y$, one gets $$\int _{\RR}\frac{\partial f(x+iy)}{\partial x}g(x+iy)dx=-\int _{\RR}f(x+iy)\frac {\partial g(x+iy)}{\partial x}dx.$$ Similarly, for almost every $x\in\RR$ one has that $$\int _{\RR}\frac{\partial f(x+iy)}{\partial y}g(x+iy)dy=-\int _{\RR}f(x+iy)\frac{\partial g(x+iy)}{\partial y}dy.$$
We deduce that the last integral of (\ref{iyywe}) is equal to \begin{align*}
&=-\int _{\RR}dy\int_{\RR} f(x+iy)\frac{\partial g(x+iy)}{\partial x}dx +\int _{\RR}dx \int _{\RR}if(x+iy)\frac{\partial g(x+iy)}{\partial y}dy\\
&=-\int _{\RR^2}f \bigg(\frac{\partial g}{\partial x}-i\frac{\partial g}{\partial y}\bigg) dxdy\\
&=-\int _{\Fv}f\frac{\partial g}{\partial x_v}dx_v.
\end{align*}
\end{proof}
Note that \begin{equation}\label{jiyu}\frac{\partial \frac{x_v}{|x_v|_v}}{\partial x_v}=0\end{equation} whenever $x_v\neq 0$ (indeed, when~$v$ is real, one has that $\frac{x_v}{|x_v|}$ is a piecewise constant function and when~$v$ is complex, one has that $\frac{\partial \frac{x_v}{|x_v|_v}}{\partial x_v}=\frac{\partial ({\overline z}^{-1})}{\partial z}=0$).

Let $\nabla$ be the vector field on $\Fv$ given by $\frac{x_v\partial}{\partial x_v}$.
\begin{cor}\label{becac}
Let $v\in M_F^\infty$. Suppose that $f,g:\Fv\to\CC$ are continuous functions the restrictions of which to $\Fvt$ are smooth. Suppose further that $\lim _{|x_v|_v\to\infty}fg(x_v)=\lim _{|x_v|_v\to 0}fg(x_v)=0$ and that the functions $\nabla(f)g, f\nabla(g):\Fvt\to\CC$ are absolutely $d^*x_v$-integrable. One has that $$\int_{\Fvt}\nabla(f)g d^*x_v=-\int _{\Fvt}f \nabla(g)d^*x_v.$$
\end{cor}
\begin{proof}
We will apply the previous lemma for~$f$ and $\frac{x_v}{|x_v|_v}g$. We note that  if~$v$ is real, then $$\lim_{x_v\to 0}fg(x_v)\frac{x_v}{|x_v|}=\lim _{x\to 0}fg(x)\frac{x}{|x|}=0.$$ By applying (\ref{jiyu}), we get that $\frac{\partial(\frac{x_v}{|x_v|_v}g)}{\partial x_v}=\frac{x_v}{|x_v|_v}\frac{\partial g}{\partial x_v}.$ It follows from the conditions of the lemma that the functions $\Fvt\to\CC$ given by  $x_v\mapsto \frac{\partial f(x_v)}{\partial x_v} \frac{x_vg(x_v)}{|x_v|_v}$ and $x_v\mapsto \frac{x_v f(x_v)}{|x_v|_v}\frac{\partial g(x_v)}{\partial x_v}$ are absolutely $dx_v$-integrable, and that $\lim _{|x_v|_v\to\infty}\frac{x_v\cdot fg(x_v)}{|x_v|_v}=0.$
Using \ref{grisv}, we get that \begin{align*}
\int_{\Fvt}\nabla(f)gd^*x_v&=\int_{\Fvt}\frac{\partial f}{\partial x_v} \frac{x_v}{|x_v|_v} g{dx_v}\\&=-\int _{\Fvt}f\frac{\partial (\frac{x_v}{|x_v|_v} \cdot g) }{\partial x_v} dx_v\\
&=-\int _{\Fvt}f\frac{x_v}{|x_v|_v}\cdot \frac{\partial g}{\partial x_v} dx_v\\
&=-\int _{\Fvt}f\nabla(g)d^*x_v.
 \end{align*}
\end{proof}
\subsection{}In this paragraph we define and prove properties of auxiliary functions $h_j,$ which will used in \ref{formulaforip} to perform desired integration by parts. 

We start with the following lemma.
\begin{lem} \label{mrljak}Let $U\subset\Fvnz$ be an open and $\Fvt$-invariant subset. Let $g:U\to\CC$ be a smooth, $\aaa$-homogenous function of weighted degree $s\in\CC$ (that is whenever $t\in\Fvt$, one has $g(t\cdot\www)=|t|_v^{s}g(\www)$ for every $\www\in U$). Let $k\in\{1\doots n\}$. The function $\nabla_k(g):U\to\CC$ is a smooth $\aaa$-homogenous function of weighted degree $s$.\end{lem}
\begin{proof}
Let $\www\in U$ and let $t\in\Fvt$. Suppose~$v$ is real. We have \begin{align*}
\frac{x_{kv}\partial g}{\partial x_{kv}}(t\cdot\www)&=\frac{x_k\partial g}{\partial x_k}(t\cdot\www)\\&={t^{a_i}w_k}\lim _{\epsilon\to 0}\frac{g((t^{a_k}w_k+\epsilon)_k,(t^{a_j}\omega_j)_{j\neq k})-g((t^{a_j}w_j)_j)}{\epsilon}\\
&=t^{a_k}w_k\lim _{\epsilon\to 0}\frac{|t|^{s}g((w_k+\epsilon /{t^{a_k}})_k,(w_j)_{j\neq k})-|t|^{s}g((w_j)_j)}{\epsilon}\\&=t^{a_k}w_k\frac{|t|^{s}\frac{\partial g}{\partial x_{k}}(\www)}{t^{a_k}}\\
&=|t|_v^{s}\nabla_k(g)(\www).
\end{align*}
It follows that for~$v$ real, the function $\nabla_k(g)$ is $\aaa$-homogenous and of weighted degree $s$. Moreover, it is smooth. Suppose that~$v$ is complex. We have that:\begin{equation*}\label{xzgh}
\frac{x_{kv}\partial g}{\partial x_{kv}}((t^{a_j}w_j)_j)=\frac{z_k\partial g}{\partial z_k}((t^{a_j}w_j)_j)
=\frac{t^{a_k}w_k}2\bigg(\frac{\partial g}{\partial x_k}-i\frac{\partial g}{\partial y_k}\bigg)(t\cdot\www).
\end{equation*}
We have that \begin{align*}
\frac{\partial g((t^{a_j}w_j)_j)}{\partial x_k}&=\lim _{\epsilon\to 0}\frac{g((t^{a_k}w_k+\epsilon)_k,(t^{a_j}w_j)_{j\neq k})-g(t\cdot\www)}{\epsilon}\\&=\lim _{\epsilon\to 0}\frac{|t|^{2s}\big(g((w_k+\epsilon/t^{a_k})_k,(w_j)_{j\neq k})-g(\www)\big)}{\epsilon}\\
&=\lim _{\epsilon\to 0}\frac{|t|^{2s}|t|^{-a_k}\big(g((w_k+\epsilon/t^{a_k})_k,(w_j)_{j\neq k})-g(\www)\big)}{(\epsilon/t^{a_k})}\\
&=t^{-a_k}|t|^{2s}\frac{\partial g}{\partial x_k}(\www).\end{align*} 
Similarly, \begin{align*}
\frac{\partial g((t^{a_j}w_j)_j)}{\partial y_k}&=\lim _{\epsilon\to 0}\frac{g((t^{a_k}w_k+i\epsilon)_k,(t^{a_j}w_j)_{j\neq k})-g(t\cdot\www)}{\epsilon}\\&=\lim _{\epsilon\to 0}\frac{|t|^{2s}\big(g((w_k+i\epsilon/t^{a_k})_k,(w_j)_{j\neq k})-g(\www)\big)}{\epsilon}\\
&=\lim _{\epsilon\to 0}\frac{|t|^{2s}t^{-a_k}\big(g((w_k+i\epsilon/t^{a_k})_k,(w_j)_{j\neq k})-g(\www)\big)}{(\epsilon/t^{a_k})}\\
&=t^{-a_k}|t|^{2s}\frac{\partial g}{\partial x_k}(\www).\end{align*}
We deduce that:
\begin{align*}\frac{x_{kv}\partial g}{\partial x_{kv}}(t\cdot\www)&= \frac{t^{a_k}w_k}2\bigg(\frac{\partial g}{\partial x_k}-i\frac{\partial g}{\partial y_k}\bigg)(t\cdot\www)\\
&=\frac{w_k|t|^{2s}}{2}\bigg(\frac{\partial g}{\partial x_k}-i\frac{\partial g}{\partial y_k}\bigg)(\www)\\
&=\frac{w_k|t|_v^{s}}{2}\bigg(\frac{\partial g}{\partial x_k}-i\frac{\partial g}{\partial y_k}\bigg)(\www) \\
&=|t|_v^s\nabla_k(g)(\www). \end{align*}
It follows that for~$v$ complex, the function $\nabla_k(g)$ is $\aaa$-homogenous and of weighted degree $s$. Moreover, it is smooth.
The statement is proven.
\end{proof}
 Recall that $f_v:\Fvnz\to\RR_{>0}$ is smooth $\aaa$-homogenous function of weighted degree $|\aaa|$.  For $k\in\{1\doots n\}$, let $\nabla_k$ be the vector field on $\Fv^n$ given by $\frac{x_{kv}\partial}{\partial x_{kv}}.$ For $j=1\doots n,$ let $h_j:\Fv^n-\{x_{j}=0\}\to\RR$ be given by $$\xxx\mapsto-\log \big(|x_{j}|_vf_v(\xxx)^{-a_j/|\aaa|}).$$ Note that $$h_j(t\cdot\xxx)=-\log(|t^{a_j}|_v|x_{j}|_vf_v(t\cdot\xxx)^{-a_j/|\aaa|}) =-\log(|x_j|_vf_v(\xxx)^{-a_j/|\aaa|})=h_j(\xxx)$$ for every $t\in\Fvt$ and every $\xxx\in\Fvn-\{x_{j}=0\}$.
\begin{lem}\label{drbo}
Let $k,j\in\{1\doots n\}$. The function $\nabla_k(h_j)$ extends to a smooth $\Fvt$-invariant function $\Fvnz\to\RR$.
\end{lem}
\begin{proof}
By \ref{mrljak}, $\nabla_k(h_j)$ is smooth and $\Fvt$-invariant on the domain $(\Fv)^n-\{x_{j}=0\}$.
When $x_{j}\neq 0$, we have that \begin{equation*}\nabla_k(h_j)=\nabla_k(\log(f_v^{a_j/|\aaa|})) -\nabla_k(\log(|x_{j}|_v))
\end{equation*}
As $f_v$ is smooth and non-vanishing, the function $\nabla_k(\log(f_v^{a_j/|\aaa|}))$ is a smooth function defined of $\Fvnz$. By \ref{mrljak}, we have for $t\in\Fvt$ and $\yyy\in\Fvnz$ that \begin{align*}\nabla _k(\log(f_v^{a_j/|\aaa|}))(t\cdot\yyy)
=\frac{\nabla_k(f_v^{a_j/|\aaa|})(t\cdot\yyy)}{f_v^{a_j/|\aaa|}(t\cdot\yyy)}&=\frac{|t|_v^{a_j}\nabla_k(f_v^{a_j/|\aaa|})(\yyy)}{|t|_v^{a_j}f_v^{a_j/|\aaa|}(\yyy)} \\&=\nabla_k(\log(f_v^{a_j/|\aaa|}))(\yyy), \end{align*}i.e. $\nabla_k(\log(f_v^{a_j/|\aaa|}))$ is $\Fvt$-invariant. For a real place $v,$ we have that $$\nabla _k(\log( |x_{k}|_v))=\frac{x\partial (\log (|x|) )}{\partial x}= 1,$$ and for a complex place $v,$ we have that $$\nabla_k(\log(|x_{k}|_v)) =\frac{z \partial\log(|z|^2)}{\partial z}=1.$$
We deduce that \begin{align}
 \quad
\nabla_k(\log(|x_{j}|_v))=&
               \begin{cases}
1,&\text{if $k=j$}\\
0,&\text{otherwise.}               \end{cases}
               \end{align} 
Therefore $\nabla_k(\log |x_{j}|_v)$ extends to a smooth $\Fvt$-invariant function $\Fvnz\to\RR.$ Now, we deduce that the function $$\nabla_k(h_j)=\nabla_k(\log(f_v^{a_j/|\aaa|})) -\nabla_k(\log(|x_{j}|_v)) $$extends to $\Fvt$-invariant and smooth function $\Fvnz\to\RR_{>0}$. The statement is proven.
\end{proof}
Continuous $\Fvt$-invariant functions $\Fvnz\to\RR$ descend to continuous functions on the compact $[\PPP(\aaa)(\Fv)].$ We deduce that: 
\begin{cor}\label{remle}
Let $k\in\{1\doots n\}$ and let $N\geq 1$ be an integer. The functions $\nabla _k^N(h_j)$ are bounded.
\end{cor}
\subsection{} \label{formulaforip} In this paragraph we derivate the pullback of the function $H_v(\sss,-)$ for the quotient map $\Fvtn\to[\TTa(\Fv)]$ using the vector fields $\nabla_k$. 

For $\sss\in\CC^n$, let us set $\widetilde H_v(\sss,-)=H_v(\sss,-)\circ\qav:(\Fvt)^n\to\CC.$
We have that $$\widetilde H_v(\sss,\xxx)=f_v(\xxx)^{\frac{\aaa\cdot\sss}{|\aaa|}}\prod _{j=1}^n|x_j|_v^{-s_j}=\prod _{j=1}^n\exp(s_jh_j(\xxx))$$ for $\xxx\in(\Fvt)^{n-1}$ and $\sss\in\CC^n.$
\begin{lem}\label{utui}
Let $k\in\{1\doots n-1\}$. For every $N\in\ZZ_{>0}$, there exists an isobaric polynomial $P_N\in\RR[\{X_{j,d}\}_{\substack{1\leq j\leq n\\1\leq d\leq N}}]$ which is of weighted degree~$N$ (where the degree of $X_{j,d}$ is~$d$) such that $$\nabla_k^N(\widetilde H_v(\sss,-)^{-1})=\widetilde H_v(\sss,-)^{-1}P_N((s_j\nabla_k^d(h_j))_{\substack{1\leq j\leq n\\ 1\leq d\leq N}})$$ for every $\sss\in\CC^n$.
\end{lem}
\begin{proof}
Let $\sss\in\CC^n$. For every $\xxx\in(\Fvt)^n$, we have that $$\widetilde H_v(\sss,\xxx)^{-1}=\exp\bigg(-\sum_{j=1}^n s_jh_j(\xxx)\bigg),$$ and hence that\begin{align*}\nabla_k\big(\widetilde H_v(\sss,\xxx)^{-1} \big)&=\exp\big(-\sum_j s_jh_j(\xxx)\big)\sum _{j=1}^n\frac{x_k\partial h_j}{\partial x_k}\\ &=\widetilde H_v(\sss,\xxx)^{-1}\sum _{j=1}^ns_j\nabla_k(h_j)(\xxx).
\end{align*}
We deduce that when $N=1$, we can take $P_1((X_{j,d})_{j,d})=\sum _{j=1}^nX_{j,1}$.
Suppose the statement is true for some~$N$ and let us verify it for $N+1$. We have \begin{align*}
\nabla_k^{N+1}(\widetilde H_v(\sss,-)^{-1})\hskip-4cm&\\&=\nabla_k\big(\widetilde H_v(\sss,-)^{-1}P_N((s_j\nabla_k^dh_j)_{j,d})\big)\\&
=\widetilde H_v(\sss,-)^{-1}\cdot\bigg(\sum_{j=1}^ns_j\nabla_k(h_j)\bigg)P_N((s_j\nabla_k^dh_j)_{j,d})\\&\quad\quad\quad\quad\quad\quad\quad\quad+\widetilde H_v(\sss,-)^{-1}\nabla_k(P_N((s_j\nabla_k^dh_j)_{j,d})\\
&=\widetilde H_v(\sss,-)^{-1}\cdot\bigg(\big(\sum_{j=1}^ns_j\nabla_k(h_j)\big)P_N((s_j\nabla_k^dh_j)_{j,d})+\nabla_k\big(P_N((s_j\nabla_k^dh_j)_{j,d})\big)\bigg).
\end{align*}
Let $\delta:\RR[\{X_{j,d}\}_{\substack{1\leq j\leq n\\1\leq d\leq N}}]\to\RR[\{X_{j,d}\}_{\substack{1\leq j\leq n\\1\leq d\leq N+1}}]$ be the $\RR$-linear map given by $$X_{j_1,d_1}^{q_1}\cdots X_{j_r,d_r}^{q_r}\mapsto \sum _{e=1}^rq_e\frac{X_{j_e,d_e+1}}{X_{j_e,d_e}}\big(X_{j_1,d_1}^{q_1}\cdots X_{j_r,d_r}^{q_r}\big)\hspace{1cm} q_e\in\ZZ_{\geq 0}.$$ 
Note that if $Q\in\RR[\{X_{j,d}\}_{\substack{1\leq j\leq n\\1\leq d\leq N}}]$ is isobaric of weighted degree~$N$, then $\delta(Q)$ is isobaric of weighted degree $N+1$. As the polynomial $\big(\sum_{j=1}^nX_{j,1}\big)P_N((X_{j,d})_{j,d})$ is isobaric of weighted degree $N+1$, the polynomial $$P_{N+1}=\big(\sum_{j=1}^nX_{j,1}\big)P_N+\delta(P_N)$$ is isobaric of weighted degree $N+1$ and 
from above one has that: $$\widetilde H_v(\sss,-)^{-1}=\widetilde H_v(\sss,-)^{-1}P_{N+1}((s_j\nabla_k^dh_j)_{\substack{1\leq j\leq n\\ 1\leq d\leq N+1}}).$$ 
The statement is proven.
\end{proof}
\subsection{} In this paragraph we calculate several limits that will enable us to perform integration by parts as in \ref{becac} in paragraph \ref{obalj}.
\begin{lem}
Let us fix $(x_{j})_{\substack{j=1\\j\neq k}}^{n} \in \Fv^{n-1}$ and let $\sss\in\Omega_{>0}$. One has that \begin{align*}
\lim _{x_{k}\to 0}\widetilde H_v(\sss,\xxx)^{-1}&=0,\\
\lim _{|x_{k}|_v\to \infty}\widetilde H_v(\sss,\xxx)^{-1}&=0.
\end{align*}
\end{lem}
\begin{proof}
We have \begin{align*}
\lim _{x_{k}\to 0}\widetilde H_v(\sss,\xxx)^{-1}&=\lim _{x_{k}\to 0}\prod _{j=1}^{n}|x_{j}|_v^{s_j}f_v(\xxx)^{\frac{-\aaa\cdot\sss}{|\aaa|}}\\
 &=\lim _{x_{k}\to 0}\prodjn |x_{j}|_v^{s_j}f_v((x_{j})_{j\neq k},(0)_k)^{-\frac{\aaa\cdot\sss}{|\aaa|}}=0.
\end{align*}
Let us calculate the other limit. For every $x_k\in\Fvt$, we have that 
\begin{align*}f_v(\xxx)&=f_v(|x_k|_v^{1/(n_va_k)}\cdot (x_j|x_k|_v^{-a_j/(n_va_k)})_j)\\&=|x_k|_v^{|\aaa|/(n_va_k)}f_v((x_j|x_k|_v^{-a_j/(n_va_k)})_j)\end{align*}and hence that $$f_v(\xxx)^{\frac{\aaa\cdot\sss}{|\aaa|}}=|x_{k}|^{\frac{\aaa\cdot\sss}{a_k}}_vf_v\big((x_{j}|x_{k}|_v^{{-a_j}/({n_va_k})})_j\big)^{\frac{\aaa\cdot\sss}{|\aaa|}}. $$ Thus
\begin{align*}
\lim _{|x_{k}|_v\to\infty}\widetilde H_v(\xxx)^{-1}\hskip-2cm&\\&=\lim _{|x_{k}|_v\to \infty}\bigg(f_v(\xxx)^{\frac{-\aaa\cdot\sss}{|\aaa|}}\prod _{j=1}^{n}|x_{j}|_v^{s_j}\bigg)\\
&=\lim _{|x_{k}|_v\to \infty}\bigg(|x_{k}|^{\frac{-\aaa\cdot\sss}{a_k}}_vf_v\big((x_{j}|x_{k}|_v^{{-a_j}/({n_va_k})})_j\big)^{\frac{-\aaa\cdot\sss}{|\aaa|}}\prod_{\substack{j=1}}^n|x_{j}|_v^{s_j}\bigg)\\
&=\lim _{|x_{k}|_v\to \infty}\bigg(|x_{k}|_v^{\frac{a_ks_k-\aaa\cdot\sss}{a_k}}f_v\big((x_{j}|x_{k}|_v^{{-a_j}/({n_va_k})})_j\big)^{-\frac{\aaa\cdot\sss}{|\aaa|}}\prod _{\substack{j=1\\j\neq k}}^n|x_{j}|_v^{s_j}\bigg).
\end{align*} Note that $\lim _{|x_{k}|_v\to\infty} |x_{k}|_v^{\frac{a_ks_k-\aaa\cdot\sss}{a_k}}=0.$ Let us define $$\mathcal B_k:= \{\yyy\in\Fvnz| \forall j |y_{j}|_v\leq 1\text{ and }|y_{k}|_v=1\}.$$  The set $\mathcal B_k$ is compact. As $f_v$ is strictly positive, there exists $\epsilon _1>0$ such that $f_v(\yyy)>\epsilon _1$ for every $\yyy\in\mathcal B_k$. We deduce that $f_v^{-\frac{\aaa\cdot\sss}{|\aaa|}}$ is bounded above by $\epsilon_1^{-{\aaa\cdot\Re(\sss)}/{|\aaa|}} $ on $\mathcal B_k$. For $|x_{k}|_v\gg 0$, one has $ (x_{j}|x_{k}|_v^{{-a_j}/(n_va_k)})_j\in \mathcal B_k$.   We conclude that $$\lim _{|x_{k}|_v\to 0}\widetilde H_v(\sss,\xxx)^{-1}=0.$$
\end{proof}
By using the formula given in \ref{utui} and the fact that the functions $\nabla_k^d(h_j)$ are bounded, we obtain immediately the following corollary .
\begin{cor}\label{srtto}
Let $\sss\in\Omega_{>0},$ let $N\geq 0$ and let $k\in\{1\doots n\}$. Let us fix $(x_{j})_{\substack{j=1\\j\neq k}}^{n} \in (\Fv)^{n-1}$. One has that
\begin{align*}
\lim _{|x_{k}|_v\to \infty}\nabla^N_k(\widetilde H_v(\sss,\xxx)^{-1})&=0,\\
\lim _{x_{k}\to 0}\nabla^N_k(\widetilde H_v(\sss,\xxx)^{-1})&=0.
\end{align*}
\end{cor}
\subsection{} In this paragraph we present several formulas for the derivation with $\nabla_k$, that will be used in \ref{obalj}. Let $\vMFi$.
\begin{lem}
If $s$ is a complex number, one has that $\nabla(x_v\mapsto |x_v|_v^s)=s|x_v|^s$ in the domain $x_v\in\Fvt.$
\end{lem}
\begin{proof}
If~$v$ is real then $$\nabla(|x_v|_v^s)=\frac{x\partial |x|^{s}}{\partial x}=xs|x|^{s-1}\text{sgn}(x)=s|x|^{s}=s|x_v|_v^{s}.$$ If~$v$ is complex, then \begin{equation*}
\nabla(|x_v|_v^{s})=\frac{z\partial(|z|^{2s})}{\partial z}=sz|z|^{2(s-1)}\overline z=s|z|^{2s}=s|x_v|_v^s.
\end{equation*}
\end{proof}
We set $F_{v,1}:=\{x|\hspace{0.1cm} |x|_v=1\}.$ We have established in \ref{pojac} an identification $$\widetilde\rho_v:\RR_{>0}\times F_{v,1}\xrightarrow{\sim}\Fvt \hspace{1cm} (r,z)\mapsto \rho_v(r)z,$$where $\rho_v:\RR_{>0}\to \Fvt$ is defined by $\rho_v(r)=r^{1/n_v}$. For a character $\chi_v\in(\Fvt)^*$, we have set $m(\chi_v)$ to be the unique real number~$m$ such that the character $\chi_v|_{\RR_{>0}}$ is given by $r\mapsto r^{im}$. If~$v$ is a real place, we set $\ell(\chi_v)$ to be $0$ if the character $\chi_{vF_{v,1}}$ is the trivial character, otherwise we set $\ell(\chi_v)=1$. If~$v$ is a complex place, we have set $\ell(\chi_v)$ to be the unique integer~$\ell$ such that $\chi_v|_{F_{v,1}}$ is given by $z\mapsto z^{\ell}$. 
Let $\chi\in[\TTa(\Fv)]^*$ be a character. We set $\widetilde \chi_v:=\chi_v\circ\qav$. We note that the function $\widetilde\chi_v :\Fvtn\to\CC$ is given by $$\xxx\mapsto \prodjn |x_{j}|_v^{im(\chi_v^{(j)})/n_v}(x_{j}|x_{j}|_v^{-1/n_v})^{\ell(\chi_v^{(j)})},$$
If $\chi_v\in[\TTa(F_v)]^*$ is a character, we set \begin{equation} \label{mchi}\mmm(\chi_v):=(m(\chi ^{(j)}_v))_j\in\RR^n \end{equation} and \begin{equation}\label{lchi}\llll(\chi_v):=(\ell(\chi ^{(j)}_v))_j\in \ZZ^n, \end{equation}
where $\chi _v^{(j)}$ is given by $$x\mapsto \chi_v (q_v^{(\aaa)}((1)_{\substack{k=1\doots n\\k\neq j}},(x)_{k=j})).$$
It follows from the definition that $$\mmm(\chi_v)\in M:=\big\{\xxx\in\RR^n|\sum _{j=1}^na_jx_j=0\big\}.$$ 
\begin{lem}\label{pcg}
Let $k\in\{1\doots n\}$. Suppose $\chi_v\in[\TTa(\Fv)]^*$ is a character. We set $d(k,\chi_v)=(1-\frac1{n_v})\ell(\chi_v^{(k)})+im(\chi_v^{(k)})$. One has that $$\nabla_k(\widetilde\chi_v)=d(k,\chi_v)\cdot\widetilde\chi_v.$$ 
\end{lem}
\begin{proof}
If $k\neq j,$ then we have that $$\nabla_k(|x_{j}|^{im(\chi_v^{(j)})}_v)=0$$ and that $$\nabla_k(({x_{j}}{|x_{j}|_v^{-1/n_v}})^{\ell(\chi_v^{(j)})})=0.$$ By using this and the product rule, we obtain that:\begin{multline*}\nabla_k\bigg(\prodjn |x_{j}|_v^{im(\chi_v^{(j)})}\big({x_{j}}{|x_{j}|_v^{-1/n_v}}\big)^{\ell(\chi_v^{(j)})}\bigg)\\= \nabla_k(x_{k}^{\ell(\chi_v^{(k)})}|x_{k}|_v^{im(\chi_v^{(k)})-\ell(\chi_v^{(k)})/n_v})\prod _{\substack{j=1\\j\neq k}}^n|x_{j}|_v^{im(\chi_v^{(j)})}\big({x_{j}}{|x_{j}|_v^{-1/n_v}}\big)^{\ell(\chi_v^{(j)})}.
\end{multline*}
One has that\begin{align*} \nabla_{k}(x_{k}^{\ell(\chi_v^{(k)})}|x_{k}|_v^{im(\chi_v^{(k)})-\ell(\chi_v^{(k)})/n_v})\hskip-5cm&\\&=\ell(\chi_v^{(k)})x_{k}^{\ell(\chi_v^{(k)})}|x_{k}|_v^{im(\chi_v^{(k)})-\frac{\ell(\chi_v^{(k)})}{n_v}}\\&\quad \quad+\big(im(\chi_v^{(k)})-\frac{\ell(\chi_v^{(k)})}{n_v}\big)x_{k}^{\ell(\chi_v^{(k)})}|x_{k}|^{im(\chi_v^{(k)})-\frac{\ell(\chi_v^{(k)})}{n_v}}_v\\&=\bigg((1-\frac1{n_v})\ell(\chi_v^{(k)})+im(\chi_v^{(k)})\bigg)x_{k}^{\ell(\chi_v^{(k)})}|x_{k}|^{im(\chi_v^{(k)})-\ell(\chi_v^{(k)})/n_v}_v\\
&=d(k,\chi_v)\cdot\widetilde\chi_v.
\end{align*}
It follows that \begin{align*}
\nabla_k(\widetilde\chi_v)&=\nabla_k\bigg(\prodjn |x_{j}|_v^{im(\chi_v^{(j)})}\big({x_{j}}{|x_{j}|_v^{-1/n_v}}\big)^{\ell(\chi_v^{(j)})}\bigg)\\&=d(k,\chi_v)\cdot\prodjn|x_{j}|_v^{im(\chi_v^{(j)})}\big({x_{j}}{|x_{j}|_v^{-1/n_v}}\big)^{\ell(\chi_v^{(j)})}\\
&=d(k,\chi_v)\cdot\widetilde\chi_v.
\end{align*}
The statement is proven.
\end{proof}
\subsection{}\label{obalj}
In this paragraph we make the wanted estimates on the absolute value of the Fourier transform. 
We use the integration by parts with respect to the vector fields $\nabla_k$.
\begin{lem}\label{derki}Let $k\in\{1\doots n-1\},$ let $\sss\in\Omega_{>0}$ and let~$N$ be a non-negative integer. 
The function $\nabla_k^N(\widetilde H_v(\sss,-)^{-1}):(\Fvt)^{n-1}\times \Fvj\to\CC$ is absolutely $dx_{1}\dots dx_{n-1}\times\lambda_{v,1}$-integrable. Moreover, if $\chi_v\in[\TTa(\Fv)]^*$ is a character, one has that 
\begin{multline*}
\wH_v(\sss,\chi_v)\cdot(-d(k,\chi_v))^N\\
=\frac{a_n}{\lambda_{v,1}(F_{v,1})}\int _{(\Fvt)^{n-1}\times\Fvj}\nabla _k^N(\widetilde H_v(\sss,-)^{-1}) \widetilde\chi_v{d^*x_{1}}\dots{d^*x_{n-1}}\lambda_{v,1},
\end{multline*}
where in the case $d(k,\chi_v)=0$ and $N=0$, the quantity $(-d(k,\chi_v))^N$ is understood as~$1$.
\end{lem}
\begin{proof}
Without loss of the generality, we can suppose that $k=1$. Suppose $N=0$. Proposition \ref{stml} gives that the integral defining the Fourier transform $\wH_v(\sss,\chi_v)$ converges absolutely. Now, it follows from Lemma \ref{smacor} that (where $q^{\aaa}_v:\Fvtn\to[\TTa(\Fv)]$ is the quotient map) 
\begin{multline*}(H_v(\sss,-)^{-1}\chi_v)\circ q^{\aaa}_v =\widetilde H_v(\sss,-)^{-1}\widetilde \chi_v\\\in L^1((\Fvt)^{n-1}\times F_{v,1}, d^*x_1\dots d^*x_{n-1} \lambda_{v,1})\end{multline*}and that
\begin{equation*}
\wH(\sss,\chi_v)=\frac{a_n}{\lambda_{v,1}(\Fvj)}\int_{(\Fvt)^{n-1}\times\Fvj}\widetilde H_v(\sss,-)^{-1}\widetilde\chi _vd^*x_{1}\dots d^{*}x_{n-1}\lambda_{v,1}.
\end{equation*}
The statement is thus true when $N=0$ and we suppose it is true for $N-1$, where $N\geq 1$. By Proposition \ref{utui}, we have that $$\nabla^N_1(\widetilde H_v(\sss,-)^{-1})=\widetilde H_v(\sss,-)^{-1}P_N((s_j\nabla _1^dh_j)_{j,d}).$$ By Corollary \ref{remle}, the functions $\nabla ^r_k(h_j)$ are bounded. It follows that $$\nabla_1^N(\widetilde H_v(\sss,-)^{-1}) \in L^1((\Fvt)^{n-1}\times F_{v,1}, d^*x_1\dots d^*x_{n-1} \lambda_{v,1}).$$ 
Using the induction hypothesis and the fact that $\nabla_1(\widetilde\chi)=d(1,\chi_v)\cdot\widetilde\chi_v$ from \ref{pcg}, we obtain that:
\begin{align*}
\wH_v(\sss,\chi_v)\cdot (-d(1,\chi_v))^N\hskip-4cm&\\&=\wH_v(\sss,\chi_v)\cdot (-d(1,\chi_v))^{N-1} (-d(1,\chi_v))\\
&=  \frac{-d(1,\chi_v)a_n}{\lambda_{v,1}(\Fvj)}\int_{(\Fvt)^{n-1}\times\Fvj}\nabla_1^{N-1}(\widetilde H_v(\sss,-)^{-1})\widetilde\chi _v\hspace{0.1cm}d^*x_{1}\dots d^{*}x_{n-1}\lambda_{v,1}\\
&=\frac{-a_n}{\lambda_{v,1}(\Fvj)}\int_{(\Fvt)^{n-1}\times\Fvj}\nabla_1^{N-1}(\widetilde H_v(\sss,-)^{-1})\nabla_1(\widetilde\chi _v)\hspace{0.1cm}d^*x_{1}\dots d^{*}x_{n-1}\lambda_{v,1}.
\end{align*}
The last integral, by Fubini theorem, writes as 
\begin{multline*}
\int _{(\Fvt)^{n-2}\times F_{v,1}}\otimes ^{n-1}_{{j=2}}d^*x_j\otimes d\lambda_{v,1}(u)\times \\ \times\int _{\Fvt}\nabla_1^{N-1}\big(\widetilde H_v(\sss,(x_j)_{{j=1}}^{n-1},u)^{-1}\big)\nabla_1(\widetilde \chi_v((x_j)_{j=1}^{n-1},u)) d^*x_1.
\end{multline*}
As $\nabla_1^{N-1}(\widetilde H_v(\sss,-)^{-1}):(\Fvt)^{n-1}\times F_{v,1}\to\CC$ is absolutely $d^*x_{1}\dots d^*x_{n-1}\lambda_{v,1}$-integrable, we deduce that for almost every $((x_j)_{{j=2}}^{n-1},u)\in (\Fvt)^{n-2}\times F_{v,1},$ we have $\nabla_1^{N-1}(\widetilde H_v(\sss, -,(x_j)_{{j=2}}^{n-1},u)^{-1})$ is absolutely $d^*x$-integrable. Now, for such $((x_j)_{{j=2}}^{n-1},u)\in (\Fvt)^{n-2}\times F_{v,1}$, Lemma \ref{srtto} and Lemma \ref{pcg} give that the functions $x_1\mapsto \nabla_1^{N-1}\widetilde H_v(\sss,(x_{j})_j)^{-1} $ and $ x_1\mapsto\widetilde\chi_v((x_{j})_j)$ satisfy the conditions of \ref{becac} and we can apply the integration by parts with the respect to $\nabla_1$. We get that:\begin{multline*}\int _{\Fvt}\nabla_1^{N-1}\big(\widetilde H_v(\sss,(x_j)^{n-1}_{{j=1}},u)^{-1}\big)\nabla_1\big(\widetilde \chi_v((x_j)_{j=1}^{n-1},u)\big) d^*x_1\\=-\int _{\Fvt}\nabla _1^{N}\big(\widetilde H_v(\sss,(x_j)_{{j=1}}^{n-1},u)^{-1}\big) \widetilde \chi_v((x_j)_{j=1}^{n-1},u)d^*x_1,\end{multline*}
and hence that
\begin{multline*}
\int_{(\Fvt)^{n-1}\times\Fvj}\nabla_1^{N-1}(\widetilde H_v(\sss,-)^{-1})\nabla_1(\widetilde\chi _v)\hspace{0.1cm}d^*x_{1}\dots d^{*}x_{n-1}\lambda_{v,1}\\
=-\int _{(\Fvt)^{n-1}\times\Fvj}\nabla _1^N(\widetilde H_v(\sss,-)^{-1}) \widetilde\chi_v{d^*x_{1}}\dots{d^*x_{n-1}}\lambda_{v,1}.
\end{multline*}
Finally, we deduce that 
\begin{align*}
\wH_v(\sss,\chi_v)\cdot (-d(1,\chi_v))^N\hskip-4cm&\\&=\frac{-a_n}{\lambda_{v,1}(\Fvj)}\int_{(\Fvt)^{n-1}\times\Fvj}\nabla_1^{N-1}(\widetilde H_v(\sss,-)^{-1})\nabla_1(\widetilde\chi _v)\hspace{0.1cm}d^*x_{1}\dots d^{*}x_{n-1}\lambda_{v,1}\\
&=\frac{a_n}{\lambda_{v,1}(F_{v,1})}\int _{(\Fvt)^{n-1}\times\Fvj}\nabla _1^N(\widetilde H_v(\sss,-)^{-1}) \widetilde\chi_v{d^*x_{1}}\dots{d^*x_{n-1}}\lambda_{v,1}.
\end{align*}
The statement is proven.
\end{proof}
\begin{lem}\label{cotka}
Let  $\vMFi$. Let $k\in\{1\doots n-1\}$ and let~$N$ be a positive integer. Let $\mathcal K\subset \RR^n_{>0}$ be a compact. There exists $C=C(k,N,\mathcal K)>0$ such that for every character $\chi_v\in[\TTa(\Fv)]^*$ and every $\sss\in\mathcal K+i\RR^n$, one has that $$(|l(\chi_v^{(k)})|+|m(\chi_v^{(k)})|)^N|\wH_v(\sss,\chi_v) |\leq {C(1+||\Im(\sss)||)}^N. $$
\end{lem}
\begin{proof} 
It follows from  \ref{derki} that \begin{multline*}
|d(k,\chi)|^N|\wH_v(\sss,\chi_v)|\leq  \\
\frac{a_n}{\lambda_{v,1}(\Fvj)}\int _{(\Fvt)^{n-1}\times\Fvj}|\nabla _k^N(\widetilde H_v(\sss,-)^{-1})|{d^*x_{1}}\dots{d^*x_{n-1}}\lambda_{v,1}.
\end{multline*}
By \ref{utui} there exists an isobaric polynomial $P_N$ of weighted degree~$N$ such that $$\nabla_k^N(\widetilde H_v(\sss,-)^{-1})=\widetilde H_v(\sss,-)^{-1}P_N((s_j\nabla_k^d(h_j))_{j,d}). $$
Moreover, by \ref{remle} the functions $\nabla _k^d(h_j)$ are bounded, and we deduce that there exists $C'>0$ such that $$|P_N((s_j\nabla_k^dh_j(\xxx))_{j,k})\leq C'(1+||\Im(\sss)||)^N$$ for every $\xxx\in\Fvnz$ and every $\sss\in\KiRn$. 
Let us note that \begin{align*}|d(k,\chi)|&=\bigg|(1-\frac1{n_v}\big)l(\chi_v^{(k)})+im(\chi_v^{(k)})\bigg|\\&\geq \frac{1}2|l(\chi_v^{(k)})+im(\chi_v^{(k)})|\\&\geq\frac{|l(\chi_v^{(k)})|+|m(\chi_v ^{(k)})|}{4}.
\end{align*}
Using \ref{smacor}, we obtain that:
\begin{align*}(|l(\chi_v^{(k)})+|m(\chi_v^{(k)})|)^N|\wH_v(\sss,\chi_v)|\hskip-5,5cm&\\&\leq4^N|d(k,\chi)|\wH_v(\sss,\chi_v)\\
&\leq{(1+||\Im(\sss)||)^N} \frac{4^NC'a_n}{\lambda_{v,1}(\Fvj)}\int _{(\Fvt)^{n-1}\times F_{v,1}}\big|\widetilde H_v(\sss,-)^{-1}\big| {d^*x_1}\cdots {d^*x_{n-1}}\lambda_{v,1}\\
&\leq(1+||\Im(\sss)||)^N4^NC'\int_{[\TTa(F)]}|H(\sss,-)^{-1}|\mu_v\\
&=(1+||\Im(\sss)||)^N4^NC'\wH(\Re(\sss),1).
\end{align*}
By \ref{stml} there exists $A>0$, such that $|\wH_v(\Re(\sss),1)|\leq A $ for every $\sss\in\KiRn$.
The statement follows.
\end{proof}
We are ready to prove:
\begin{prop}\label{crst}
Let $\vMFi$ and let $\mathcal K\subset\RR^n_{>0}$ be a compact. There exists $C>0$ such that for every $\sss\in\KiRn$ and every $\chi_v\in[\TTa(\Fv)]^*$ one has
$$|\wH_v(\sss,\chi_v)|\leq\frac{{C(1+||\Im(\sss)||)^N}}{((1+||\mmm(\chiv)||)(1+|\llll(\chiv)||))^{N/(2(n-1))}}.$$
\end{prop}
\begin{proof}
We have already seen in \ref{stml}  that there exists $A>0$ such that $$|\wH_v(\sss,\chi_v)|\leq A$$ for every $\chi_v\in[\TTa(\Fv)]$ and every $\sss\in\KiRn$. Using this and \ref{cotka} we get that there exists $M>0$ such that\begin{multline}\label{stvij}
\prod_{\substack{k\in\{1,.., n-1\}\\|\ell(\chivk)|+|m(\chivk)|=0}}|\wH_v(\sss,\chiv)|\times\\\times
\prod _{\substack{k\in\{1,.., n-1\} \\|\ell(\chi_v^{(k))}|+|m(\chi_v^{(k)})|\neq 0 }}(|\ell(\chi_v^{(k)})|+|m(\chi_v^{(k)})|)^N|\wH_v(\sss,\chi_v)|\\ \leq M (1+||\Im(\sss)||)^{N(n-1)}
\end{multline} for every $\chi_v\in[\TTa(\Fv)]$ and every $\sss\in\KiRn$.
Using the fact that $$a_1m(\chi_v^{(1)})+\cdots +a_nm(\chi_v^{(n)})=0$$ we deduce that $$||\mmm(\chi_v)||\leq\max\big( \max _{j=1\doots n-1}|m(\chi_v^{(j)})|,a_1|m(\chi_v^{(1)})|+\cdots +a_{n-1}|m(\chi_v^{(n-1)})|\big)$$ and hence there exists an index $o(\chi_v)\in\{1\doots n-1\}$ such that \begin{equation}\label{oindex}(n-1)\cdot {\max _ja_j}\cdot |m(\chi_v^{(o(\chi_v))})|\geq {||\mmm(\chi_v)||}.\end{equation}
In an analogous way, we conclude that there exists index $r(\chi_v)\in\{1\doots n-1\}$ such that \begin{equation}\label{rindex}(n-1)\cdot {\max _ja_j}\cdot |\ell(\chi_v^{(r(\chi_v))})|\geq {||\llll(\chi_v)||}.\end{equation}
Using the arithmetic-geometric inequality and (\ref{oindex}) and (\ref{rindex}), we conclude that there exists $D>0$ such that \begin{multline}D\prod _{ |l(\chi_v^{(k)})|+|m(\chi_v^{(k)})|\neq 0}(|\ell(\chi_v^{(k)})|+|m(\chi_v^{(k)})|)^N \\\geq \max( ||\llll(\chiv)||^{N/2} ||\mmm(\chiv)||^{N/2},||\mmm(\chiv)||^{N/2},||\llll(\chiv)||^{N/2}).\label{stvi}\end{multline}
Combining (\ref{stvi}) and (\ref{stvij}) and taking the $n-1$-th root gives \begin{multline}|\wH_v(\sss,\chiv)|\max(||\llll(\chiv)|\lvert\cdot\rvert|\mmm(\chiv) ||,||\mmm(\chiv) ||,||\llll(\chiv)||)^{N/(2(n-1))} \\\leq (MD)^{1/(n-1)}(1+||\Im(\sss)||)^N\end{multline} for every $\sss\in\KiRn$ and every $\chiv\in[\TTa(\Fv)]^*$.

For every $\chiv\in[\TTa(\Fv)]^*$ with $\llll(\chiv)=(0)_j$ and $||\mmm(\chiv)||\leq 1$ and every $\sss\in\KiRn$, one has that \begin{align*}|\wH_v(\sss,\chi_v)|&\leq\frac{2^{N/(2(n-1))}A}{(1+||\mmm(\chiv)||)^{N/(2(n-1)}} \\
&\leq\frac{2^{N/(2(n-1))}A(1+||\Im(\sss)||)^N}{((1+||\llll(\chiv)||)(1+||\mmm(\chiv)||))^{N/(2(n-1))}}.\end{align*} 
For every  $\chi_v\in[\TTa(\Fv)]^*$ for which $\llll(\chiv)\neq(0)_j$ and for which $||\mmm(\chiv)||\leq 1$ and for every $\sss\in\mathcal K+i\RR^n$ one has \begin{align*}|\wH_v(\sss,\chi_v)|&\leq\frac{(MD)^{1/(n-1)}(1+||\Im(\sss)||)^N}{||\llll(\chiv)||^{N/(2(n-1))}} \\
&\leq\frac{2^{2N/(2(n-1))}(MD)^{1/(n-1)}(1+||\Im(\sss)||)^N}{((1+||\llll(\chi_v)||)(1+||\mmm(\chi_v)||))^{N/(2(n-1))}}. \end{align*} 
For every $\chiv\in[\TTa(\Fv)]^*$ with $\llll(\chiv)=(0)_j$ and $||\mmm(\chiv)||>1$ and every $\sss\in\KiRn$, one has that \begin{align*}\label{uljip}|\wH_v(\sss,\chiv)|&\leq \frac{(DM)^{1/(n-1)}(1+||\Im(\sss)||)^N}{||\mmm(\chiv)||^{N/(2(n-1))}}\\
&\leq\frac{2^{N/(2(n-1))}(MD)^{1/(n-1)}(1+||\Im(\sss)||)^N}{((1+||\llll(\chiv)||)(1+||\mmm(\chiv)||))^{N/(2(n-1))}}.
\end{align*}
For every $\chiv\in[\TTa(\Fv)]^*$ with $\llll(\chiv)\neq(0)_j$ and $||\mmm(\chiv)||>1$ and every $\sss\in\KiRn$, one has that \begin{align*}|\wH_v(\sss,\chiv)|&\leq\frac{(MD)^{1/(n-1)}(1+||\Im(\sss)||)^N}{||\llll(\chiv)||^{N/(2(n-1))}||\mmm(\chiv)||^{N/(2(n-1))}} \\
&\leq\frac{2^{2N/(2(n-1))}(MD)^{1/(n-1)}(1+||\Im(\sss)||)^N}{((1+||\mmm(\chiv)||)(1+||\llll(\chiv)||))^{N/(2(n-1))}} \end{align*}
Therefore $C=2^{2N/(2(n-1))}\max((MD)^{1/(n-1)},A)$ satisfies the wanted condition.
\end{proof}
\subsection{} In \ref{rinta} and in \ref{arintaj}, we have defined norms of the characters of $\AAFt$ and of $[\TTa(\AAF)]^*$, respectively. In this paragraph we give a corollary of \ref{crst} and we write in the terms on the norms.

For a character $\chi\in (\AAF^\times)^*$, we have defined:
\begin{align*}||\chi||_{\discrete}&=\max _{\vMFi}(||\ell(\chi _v)||),\\ ||\chi ||_\infty&=\max_{\vMFi}(||m(\chi _v)||).\end{align*}
For a character $\chi\in[\TTT ^\aaa(\AAF)]^*$, we have defined \begin{align*}||\chi||_{\discrete}&=\max _{\vMFi}(||\llll(\chi _v)||),\\ ||\chi ||_\infty&=\max_{\vMFi}(||\mmm(\chi _v)||).\end{align*} 
Proposition \ref{crst} implies that:
\begin{cor}\label{izvuciarc}
Let $\mathcal K\subset\RR^n_{>0}$ be a compact and let~$N$ be a positive integer. There exists $C>0$ such that for every $\chi\in[\TTa(\AAF)]^*$ and every $\sss\in\KiRn$ one has $$\prod_{\vMFi}|\wH_v(\sss,\chi_v)|\leq\frac{C(1+||\Im(\sss)||)^{N (r_1+r_2)}}{((1+||\chi||_{\discrete})(1+||\chi||_{\infty}))^{N/(2(n-1))}}. $$
\end{cor}
\begin{proof}
We have that $$(1+||\chi||_{\discrete})\leq \prod _{\vMFi}(1+||\llll(\chiv)||) $$ and that $$(1+||\chi||_{\infty})\leq \prod _{\vMFi}(1+||\mmm(\chiv)||).$$ It follows from Proposition \ref{crst} that there exists $C_1>0$ such that for every $\chi\in[\TTa(\AAF)]^*$ one has that 
\begin{align*}
\prod_{\vMFi}|\wH_v(\sss,\chi_v)|&\leq\prod_{\vMFi}\frac{{C_1(1+||\Im(\sss)||)^N}}{((1+||\mmm(\chiv)||)(1+|\llll(\chiv)||))^{N/(2(n-1))}}\\
&\leq\frac{C_1^{r_1+r_2}(1+||\Im(\sss)||)^{N(r_1+r_2)}}{\big((1+||\chi||_{\discrete})(1+||\chi||_{\infty})\big)^{N/(2(n-1)}}.
\end{align*}
The statement is proven.
\end{proof}
\section{Global transform}
Using the results of \ref{mmhu}, \ref{Archimed} and \ref{EORade}, we obtain estimates for the global Fourier transform. 

Let $\mu_{\AAF}$ be the Haar measure on $[\TTa(\AAF)]$ given by \ref{muaaf} $$\mu_{\AAF}=\bigotimes_{\vMFz}\zeta_v(1)^{n-1}\mu_v\otimes \bigotimes_{\vMFi}\mu_v.$$ 
\subsection{}
In this paragraph we are going to present a Haar measure on the dual group $(\Rgz^n/(\Rgz)_{\aaa})^*.$ 
%
Let us set$$M=M(\aaa):=\{\xxx\in\RR^n|\aaa\cdot\xxx=0\},$$  $$\lambda_{\aaa}:\RR_{>0}\to(\RR_{>0})_{\aaa}\hspace{1cm} x\mapsto (x^{a_j})_j.$$
The subspace $M$ is the kernel of the surjective linear map $\RR^n\to\RR, \xxx\mapsto\aaa\cdot\xxx$ and we will identify $\RR^n/M$ with $\RR$ (for this identification the quotient map $\RR^n\to \RR^n/M$ becomes $\xxx\mapsto \aaa\cdot\xxx$).
\begin{lem}\label{nebrlj}
Let $d\mmm$ be the unique Lebesgue measure on $M$ such that $$(dx_1\dots dx_n)\big/d\mmm=dx.$$
\begin{enumerate}
\item  For $A> 0$, let us define $$\theta^n_A:\RR^n\to\RR^n\hspace{1cm}\xxx\mapsto (Ax_j)_j.$$ We have that $(\theta^n_A|M)_*(d\mmm)=A^{-(n-1)}d\mmm,$ i.e. for any $d\mmm$-integrable function~$f$, one has that $$\int_Mfd\mmm=\frac{1}{A^{n-1}}\int_Mf\circ\theta^n_A|_Md\mmm.$$
\item The homomorphism $$\xi:\RR\to(\RR_{>0})^*\hspace{1cm} t\mapsto \lvert\cdot\rvert^{2i\pi t}$$ is an isomorphism and one has $\xi_*(dx)= d^*x.$ The homomorphism $$\xi^n:\RR^n\to(\RR_{>0}^n)^*\hspace{1cm} \xxx\mapsto \prodjn |pr_j(\cdot)|^{2i\pi x_j},$$ where $pr_j:\RR^n\to\RR$ is the projection to the~$j$-th coordinate, is an isomorphism and one has $\xi^n_*(dx_1\dots dx_n)= d^*x_1\dots d^*x_n.$
\item Let $(\Rgz)_{\aaa}^\perp$ be the group of characters of $(\RR^n_{>0})^*$ which are trivial on $(\Rgz)_{\aaa}.$  One has that $\xi^n(M)=(\Rgz)_{\aaa}^\perp=(\Rgz^n/(\Rgz)_{\aaa})^*.$
\item 
The isomorphism $\xi^n|_M:M\xrightarrow{\sim} (\RR^n_{>0}/(\RR_{>0})_{\aaa})^*$ from (2) satisfies $$(\xi^n|_M)_*\big(d\mmm\big)=\big(d^*x_1\dots d^*x_n/(\lambda_{\aaa})_*({d^*x})\big)^*.$$ 
\end{enumerate}
\end{lem}
\begin{proof}
For $a,b>0,$ we have that $d^*x([a,b])=\log b-\log a$. The homomorphism $\exp:\RR\to\RR_{>0}$ is an isomorphism, for which, hence, one has that $(\exp)_*dx=d^*x.$ In the proof, using the isomorphism $\exp$, we identify $\RR$ and $\Rgz$ and the corresponding measures $dx$ and $d^*x$. 
\begin{enumerate}
\item For $A>0$, let $\theta_A^1:\RR\to\RR$ be the map $x\mapsto Ax$. We observe that $(\theta^n_A)_*(dx_1\cdots dx_n)=A^{-n}dx_1\dots dx_n$  and that $(\theta^1_A)_*(dx)=A^{-1}dx.$ Now, in the commutative diagram$$\begin{tikzcd}
\{0\} \arrow[r] & M \arrow[d, "\theta^n_A|_M"] \arrow[r, ""] & \RR^n\arrow[d, "\theta^n"] \arrow[r, "\xxx\mapsto\aaa\cdot\xxx"] & \RR \arrow[d, "\theta^1_A"] \arrow[r] & \{0\} \\
  \{0\} \arrow[r] &M  \arrow[r, ""] & \RR^n \arrow[r, "x\mapsto\aaa\cdot\xxx"] & \RR \arrow[r] & \{0\}
\end{tikzcd}$$ the vertical maps are isomorphisms, and we deduce \begin{align*}A^{-1}dx=(\theta^1_A)_*(dx)&=(\theta^n_A)_*(dx_1\dots dx_n)/(\theta^n_A|_M)_*d\mmm\\&=A^{-n} dx_1\dots dx_n/(\theta^n_A|_M)_*d\mmm.\end{align*} This gives $(\theta^n_{A}|_M)_*d\mmm= A^{-(n-1)}d\mmm$. For every $d\mmm$-integrable function~$f$ one has that $$\int_Mfd\mmm=\int_M(f\circ(\theta^n_A|_M))(\theta^n_A|_M)_*d\mmm=A^{-n+1}\int _Mf\circ(\theta^n_A|M)d\mmm.$$
\item With this identification, the first claim is given in \cite[Corollary 1, \no 9, \S 1, Chapter II]{TSpectrale}. The second claim we deduce by the fact that the dual of a finite product of Haar measures is the product of the duals. 
\item The isomorphism $\exp^n:\RR^n\to(\Rgz)^n$ identifies $\aaa(\RR)=\{(a_jx)_j| x\in\RR\}$ with $(\Rgz)_{\aaa}$.
Let $0\neq\xxx\in\aaa(\RR)$ and let $0\neq x\in\RR$ such that $a_jx_j=x_j$ for every~$j$. For every $\yyy\in M$, we have that $$\xi^n(\yyy) (\xxx)=\exp\bigg(\sum _{j=1}^n2i\pi y_jx_j\bigg)=\exp\bigg(2i\pi x\sum_{j=1}^na_jy_j\bigg).$$ Hence, $\xi^n(\yyy)\in (\aaa(\RR))^\perp$ if and only if $\yyy\in M$ and the claim follows.
\item 
The dual sequence of the short exact sequence $$0\to\Rgz\xrightarrow{i_{(\Rgz)_{\aaa}}\circ\lambda_\aaa}\Rgz^n\to \Rgz^n/(\Rgz)_{\aaa}\to 0$$ is the short exact sequence sequence $$0\to (\Rgz^n/(\Rgz)_{\aaa})^*\to\RR^n\xrightarrow{\xxx\mapsto \aaa\cdot\xxx}\RR\to 0.$$ Lemma \ref{trcio} gives $$dx_{1}\dots dx_n\big/\big(d^*{x_1}\dots{d^*x_n}/(\lambda_{\aaa})_*({d^*x})\big)^*=dx.$$ By definition $dx_1\dots dx_n/d\mmm=dx.$ Now, the commutativity of the diagram $$\begin{tikzcd}
\{0\} \arrow[r] & M \arrow[d, "\xi^n_M"] \arrow[r, ""] & \RR^n\arrow[d, "\xi^n"] \arrow[r, "\xxx\mapsto\aaa\cdot\xxx"] & \RR \arrow[d, "\xi"] \arrow[r] & \{0\} \\
  \{1\} \arrow[r] &(\Rgz^n/(\Rgz)_{\aaa})^*  \arrow[r, ""] & \RR^n \arrow[r, ""] & \RR \arrow[r] & \{0\},
\end{tikzcd}$$ gives that $$(\xi^n|_M)_*(d\mmm)=(d^*x_1\dots d^*x_n/(\lambda_{\aaa})_*(d^*x))^*.$$
\end{enumerate}
\end{proof}
For a character $\chi\in(\RR_{>0})^*,$ we denote by $m(\chi)$ the unique real number~$m$ such that $\chi$ is given by $x\mapsto x^{im}$. 
For a character $\chi\in (\RR^n_{>0}/(\Rgz)_{\aaa})^*$ we write $$\mmm(\chi):=(m(\chi^{(j)}))_j\in M,$$where $\chi^{(j)}$ is the pullback character $\RR_{>0}$ given by the composite homomorphism $$\RR_{>0}\xrightarrow{x\mapsto ((1)_{i\neq j},x)} \RR^n_{>0}\to\RR^n_{>0}/(\Rgz)_{\aaa}.$$
It follows from \ref{nebrlj}, that $$(\RR^n_{>0}/(\Rgz)_{\aaa})^*\to M\hspace{1cm}\chi\mapsto \mmm(\chi)$$ is an isomorphism and we write $\lvert\cdot\rvert^{i\mmm}$ for the unique character of $(\RR^n_{>0}/(\Rgz)_{\aaa})^*$ such that its image under the isomorphism is $\mmm$. Let $\wx\in\RR^n_{>0}$ be a lift of $\xxx\in(\RR^n_{>0}/(\Rgz)_{\aaa})$. We observe that $$\xi^n(\mmm)(\xxx)=\prodjn\widetilde x_j^{i 2\pi m_j}$$ and that $$|\xxx|^{i\mmm}=\prodjn\widetilde x_j^{im_j}.$$ In other words \begin{equation}\label{relatxim}\xi^n(\mmm)=\lvert\cdot\rvert ^{i 2\pi \mmm}.\end{equation}
\subsection{}
In this paragraph we estimate the global Fourier transform of the height.

In \ref{identaafj}, we have established an identification $$\AAFt\xrightarrow{\sim}{\AAF^1}\times \RR_{>0}.$$ For a character $\chi\in(\AAFt)^*$ we write $m(\chi)$ for $m(\chi|_{\RR_{>0}})$. 
In \ref{seulment}, we have established an identification
$$[\TTa(\AAF)]\xrightarrow{\sim}[\TTa(\AAF)]_1\times (\RR^n_{>0}/(\RR_{>0})_{\aaa}).$$
For $(\xxx_v)_v\in[\TTa(\AAF)],$ let $(\wx_v)_v\in(\AAFt)^n$ be its lift. The morphism to the second coordinate is given by $$\xxx\mapsto q_{\RR_{>0}}\bigg(\big(\prod_{\vMF}|\widetilde x_{jv}|_v\big)_j\bigg),$$where $q_{\RR_{>0}}:\RR^n_{>0}\to \RR^n_{>0}/(\RR_{>0})_{\aaa}$ is the quotient map. For a character $\chi\in[\TTa(\AAF)]^*$, we write $\mmm(\chi)$ for $\mmm(\chi|_{\RR^n_{>0}/(\Rgz)_{\aaa}}).$  
We write $\lvert\cdot\rvert^{i\cdot\mmm}$ for the unique character $\chi \in[\TTa(\AAF)]^*$ which satisfies $\chi|_{[\TTa(\AAF)]_1}=1$ and $\chi|_{\RR^n_{>0}/(\Rgz)_{\aaa}}=\lvert\cdot\rvert^{i\mmm}$. For $(\xxx_v)_v\in[\TTa(\AAF)],$ one has that $$|(\xxx_v)_v|^{i\mmm}=\prod_{j=1}^n\bigg|\prod_{\vMF}|\widetilde x_{jv}|_v\bigg|^{im_j}=\prodjn\prod_{\vMF}|\widetilde x_{jv}|_v^{im_j}.$$
\begin{lem}\label{signpr} 
\begin{enumerate}
\item For every $\mmm\in M$, one has that
$$H(\sss,\cdot)\lvert\cdot\rvert^{i\cdot \mmm}=H(\sss+i\mmm,\cdot).$$
\item Let $\chi\in[\TTa(\AAF)]^*$. Whenever the quantities on the both hand sides converge, one has that $$\wH(\sss,\chi)=\wH(\sss+i\mmm,\chi_0).$$
\end{enumerate}
\end{lem}
\begin{proof}
\begin{enumerate}
\item Let $(\xxx_v)_v\in[\TTa(\AAF)]$ and let $(\wx_v)_v\in(\AAFt)^n$ be its lift. We have that \begin{align*}
H(\sss,(\xxx_v)_v)|(\xxx_v)_v|^{i\mmm}&=\prod_{\vMF}H_v(\sss,\xxx_v)\prodjn|\widetilde x_{jv}|^{im_j}\\
&=\prod_{\vMF}\bigg(\bigg(f_v(\wx_v)^{\frac{-\aaa\cdot\sss}{|\aaa|}}\prodjn |\widetilde x_j|_v^{s_j}\bigg)\prodjn|\widetilde x_{jv}|^{im_j}\bigg)\\
&=\prod_{\vMF}\bigg(f_v(\wx_v)^{\frac{-\aaa\cdot\sss}{|\aaa|}}\prodjn|\widetilde x_j|_v^{s_j+im_j}\bigg)\\
&=H(\sss+i\mmm,(\xxx_v)_v),
\end{align*}as claimed.
\item It follows from (1) that $$H(\sss,\cdot)\chi=H(\sss+i\mmm,\cdot)\chi_0\lvert\cdot\rvert^{i\mmm(\chi)}=H(\sss+i\mmm,\cdot)\chi_0.$$
Now, one has that \begin{align*}
\wH(\sss,\chi)&=\int_{[\TTa(\AAF)]}H(\sss,\cdot)\chi \mu_{\AAF}\\&=\int_{[\TTa(\AAF)]}H(\sss,\cdot)\chi_0|\cdot|^{i\mmm(\chi)}{\mu_{\AAF}}\\&=\int_{[\TTa(\AAF)]}H(\sss+i\mmm(\chi),\chi_0)\mu_{\AAF}\\&=\wH(\sss+i\mmm(\chi),\chi_0),\end{align*}
whenever every quantity converges. The claim is proven
\end{enumerate}
\end{proof}
Given a character $\chi \in[\TTa(\AAF)]^*$, let us denote by $\chi_0$ the character $\chi|_{[\TTa(\AAF)]_1}.$ For every $\sss\in\CC^n$, Lemma \ref{signpr} gives that
\begin{lem}\label{orjuu}
\begin{enumerate}
\item For every $\chi\in[\TTa(\AAF)]^*$, the product $$\wH(\sss,\chi):=\prod _{\vMF}\wH_v(\sss,\chi_v) $$ converges when $\sss\in\Omega_{>1}$.
\item Let $K\subset K^\aaa_{\max}:=\prod _{\vMFz}[\TTa(\Ov)]$ be an open subgroup. For every $\chi\in[\TTa(\AAF)]^*$ vanishing on~$K$, the function 
\begin{align*}\sss&\mapsto \wH(\sss,\chi) \prod _{\chi_0^{(j)}= 1}\frac{s_j+im(\chij)-1}{s_j+im(\chij)}
\end{align*}extends to a holomorphic function in the domain $\Omega _{>\frac23}$.
Moreover, there exists $\frac13>\delta>0$ such that for any compact in the domain $\mathcal K\subset\RR^n_{>1-\delta}$ and any positive integer~$N$, there exists $C(\mathcal K,N)>0$ such that for any $\chi\in[\TTa(\AAF)]^*$ which vanishes on~$K$ and any $\sss\in\mathcal K+i\RR^n$ one has that  \begin{multline}\bigg|\wH(\sss,\chi)\prod _{\chi_0^{(j)}= 1}\frac{s_j+im(\chij)-1}{s_j+im(\chij)}\bigg|\\\leq \frac{C(\mathcal K,N)(1+||\Im(\sss)||)^{N(r_1+r_2)+1}(1+||\mmm(\chi)||)}{(1+||\chi||_{\discrete})^{N/2(n-1)-1}((1+||\chi||_\infty))^{N/2(n-1)-1}}.\end{multline}
\end{enumerate}
\end{lem}
\begin{proof}
\begin{enumerate}
\item Suppose $\sss\in\Omega_{>1}$. For every $\chi\in[\TTa(\AAF)]^*$, by \ref{fintr},  one has that the product $\prod _{\vMFz}\wH_v(\sss,\chiv)$ converges absolutely and by \ref{izvuciarc}, one has that $\sss\mapsto\prod _{\vMFi}\wH_v(\sss,\chi_v) $ is holomorphic in the domain $\Omega_{>0}$. The claim follows.
\item By Proposition \ref{fintr}, there exists a unique holomorphic function $\phi(\cdot,\chi):\Omega_{>\frac23}\to\CC$ such that \begin{align*}\wH(\sss,\chi)\hskip-1cm&\\&=\wH(\sss+i\mmm(\chi),\chi_0)\\&=\phi_{\fin}(\sss+i\mmm(\chi),\chi_0)\prod _{j=1}^nL(s_j+im(\chij),\chi _0^{(j)})\prod _{\vMFi}\wH_v(\sss+i\mmm(\chi),\chi _{0v}). \end{align*}The function $\sss\mapsto \prod _{\vMFi}\wH_v(\sss+i\mmm(\chi),\chi_{0v})$ is holomorphic in the domain $\Omega_{>0}$ by \ref{izvuciarc}. Recall that $L(\cdot,\chi_0)$ is an entire function for every $0\neq\chi_0\in\AAFj$, while $L(\cdot,1)$ is a meromorphic function with the single pole at~$1$ and no other poles. Therefore, the function $$\sss\mapsto \prod _{\chij_0= 1}\frac{s_j+im(\chij)-1}{s_j+im(\chij)} L(s_j+im(\chij),\chij_0)\prod_{\chij_0 \neq 1}L(s_j+im(\chij),\chij_0)$$ extends to a holomorphic function in the domain $\Omega_{>\frac23}$. Hence $$\sss\mapsto 
\wH(\sss+i\mmm(\chi),\chi_0)\prod _{\chij_0=1}\frac{s_j+im(\chij)-1}{s_j+im(\chij)}$$ extends to a holomorphic function in the domain $\Omega_{>1-\delta}$. 
By \ref{Rade}, there exists $\frac 13>\delta>0$ and $C_2>0$ such that \begin{align*}\bigg|\prod _{\chij_0= 1}\frac{s_j+im(\chij)-1}{s_j+im(\chij)}L(s_j+im(\chij),\chij_0)\prod _{\chij_0\neq 1}\cdot L(s_j+im(\chij),\chij_0)\bigg|\hskip-10cm&\\&\leq C_2\big(\prodjn(1+|\Im(s_j)|)(1+||\chij _0||_\infty)(1+|m(\chij)|)\big)^{1/n}\\&\leq C_2(1+||\Im(\sss)||)(1+||\chi_0||_\infty)(1+||\mmm(\chi)||)
 \end{align*} provided that $\Re(s_j)>1-\delta$ for $j=1\doots n$.  Let~$N$ be an integer and let $\mathcal K\subset\RR^n_{>1-\delta}$ be a compact. 
 %
Proposition \ref{fintr} gives that there exists $C_1>0$ such that $|\phi_{\fin}(\sss,\chi)|\leq C_1$ for every $\sss\in\mathcal K+i\RR^n$. 
By Lemma \ref{izvuciarc}, there exists $C_3>0$ such that $$\prod_{\vMFi}|\wH_v(\sss+i\mmm(\chi),\chi_{0v})|\leq\frac{C_3(1+||\Im(\sss)||)^{N(r_1+r_2)}}{(1+||\chi||_{\discrete})^{N/(2(n-1))}(1+||\chi||_\infty)^{N/(2(n-1))}} $$ for every $\sss\in\mathcal K+i\RR^n$ and every $\chi\in[\TTa(\AAF)]^*$ which vanishes on~$K$.
By \ref{yuit}, there exists $C_4>1$ such that $$\frac{1+||\chi_0||_\infty}{1+||\chi_0||_{\discrete}}\leq C_4. $$
We deduce that for every $\sss\in\KiRn$ and every $\chi\in[\TTa(\AAF)]^*$ which vanishes on~$K$, one has \begin{multline*}\bigg|\wH(\sss+i\mmm(\chi),\chi_0)\prod _{\chij_0=1}\frac{s_j+im(\chij)-1}{s_j+im(\chij)}\bigg|=\bigg|\phi_{\fin}(\sss+i\mmm(\chi),\chi_0)\times\\\times\prod_{\chij_0=1}\frac{s_j+im(\chij)-1}{s_j+im(\chij)} \prodjn L(s_j+im(\chij),\chi_0^{(j)})\prod _{\vMFi}\wH_v(\sss+i\mmm(\chi),\chi_{0v})\bigg|\\\leq \frac{C_1C_2C_3(1+||\Im(\sss)||)^{N(r_1+r_2)+1}(1+||\chi_0||_{\infty})(1+||\mmm(\chi)||)}{((1+||\chi_0||_{\discrete})(1+||\chi||_\infty))^{N/(2(n-1))}}
\\\leq \frac{C_1C_2C_3C_4(1+||\Im(\sss)||)^{N(r_1+r_2)+1}(1+||\mmm(\chi)||)}{(1+||\chi_0||_{\discrete})^{N/(2(n-1))-1}(1+||\chi_0||_{\infty})^{N/(2(n-1))}}.\end{multline*}The claim follows.\end{enumerate}
\end{proof}
Note that if $\chi_1,\chi_2\in[\TTa(\AAF)]^*$ are two characters then \begin{align*}
||\chi_1\chi_2||_\infty&=\max _{v\in M_F^\infty}||\mmm(\chi_{1v}\chi_{2v})||\\
&=\max_{\vMFi} ||\mmm(\chi_{1v})+\mmm(\chi_{2v})||\\
&\leq \max_{\vMFi}||\mmm(\chi_{1v})|| +\max_{\vMFi}||\mmm(\chi_{2v})||\\
&=||\chi_1||_{\infty}+||\chi_2||_{\infty}.
\end{align*}
Now, if $\chi_0=\chi\lvert\cdot\rvert^{-i\mmm}$ for $\chi\in[\TTa(\AAF)]^*$ and $\mmm\in M$, we can deduce that
\begin{equation}\label{imppom}\frac{1}{1+||\chi||}=\frac{1}{1+ ||\chi_0\lvert\cdot\rvert^{-i\mmm}||}\leq \frac{1+||\chi_0||}{1+||\mmm||} .\end{equation}  We can establish the following corollary:
\begin{cor}\label{siuyh} Let $K\subset K^\aaa_{\max}$ be an open subgroup. For every $\alpha >0$ there exist $\beta=\beta(\alpha)>0$ and $\delta=\delta (\alpha)>0$ such that for every compact $\mathcal K\subset\RR^n_{>1-\delta},$ one has that there exists $C=C(\alpha,\mathcal K)>0$ such that for every $\sss\in\KiRn$ and every $\chi\in\TTa(\AAF)^*$ which vanishes on~$K$ one has
$$\bigg|\wH(\sss,\chi)\prod_{\chij_0=1} \frac{s_j+im(\chij)-1}{s_j+im(\chij)}\bigg| \leq \frac{C(1+||\Im(\sss)||)^\beta}{((1+||\chi_0||_{\discrete})(1+||\mmm(\chi)||))^{\alpha}}.$$ 
\end{cor}
\begin{proof}
Firstly, let $C_1>0$ be such that for every $\chi_0\in\AK$ one has$$\frac{1+||\chi_0||_\infty}{1+||\chi_0||_{\discrete}} \leq C_1,$$ such $C_1$ exists by \ref{yuit}.
Let~$N$ be an integer bigger than $2(n-1)(2\alpha+2)$ and let $\beta> N(r_1+r_2)+1$. Let $\delta $ be given by Lemma \ref{orjuu}.  It follows from this lemma and from (\ref{imppom}) that for every compact $\mathcal K\subset\RR^n_{>1-\delta}$, there exists  $C(\mathcal K,N)$ such that for every $\chi\in[\TTa(\AAF)]^*$ which vanishes on~$K$ and every $\sss\in\KiRn$ one has \begin{align*}\bigg|\wH(\sss,\chi)\prod_{\chij _0=1}\frac{s_j+im(\chij)-1}{s_j+im(\chij)}\bigg|\hskip-2cm&\\&\leq \frac{C(\mathcal K,N)(1+||\Im(\sss)||)^{N(r_1+r_2)+1}(1+||\mmm(\chi)||)}{(1+||\chi_0||_{\discrete})^{\frac{N}{2(n-1)}-1}(1+||\chi||_{\infty})^{\frac{N}{2(n-1)}}}  \\&\leq \frac{C(\mathcal K, N)(1+||\Im(\sss)||)^\beta(1+||\mmm(\chi)||)}{((1+||\chi_0||_{\discrete})(1+||\chi||_{\infty}))^{2\alpha+1}}.\\
&\leq \frac{C(1+||\Im(\sss)||)^{\beta}(1+||\chi_0||_{\infty})^{\alpha+1}}{(1+||\chi_0||_{\discrete})^{2\alpha +1}(1+||\mmm(\chi)||)^\alpha}\\&\leq \frac{CC_1^{\alpha+1}(1+||\Im(\sss)||)^\beta}{((1+||\chi_0||)(1+||\mmm(\chi)||))^\alpha}.
\end{align*}
The claim follows.
\end{proof}
\subsection{} For an open subgroup $K\subset K^{\aaa}_{\max}=\prod_{\vMFz}[\TTa(\Ov)],$ we denote by $\AK$ the subgroup of $[\TTa(\AAF)]_1^*$ given by the characters that vanish on $[\TTa(i)(F)]K$. In this paragraph we explain that one can sum transforms over the characters of $\AK$. 
%
The following lemma will be used. 
\begin{lem} \label{asioo}
Suppose~$X$ is a discrete set. Let $f:X\to \RR_{\geq 0}$ be such that there exists $A>0$ and $d>0$ such that for every $B>0$ one has $$|\{x\in X| f(x)\leq B\}|\leq AB^d.$$
Let $\epsilon>0$. The series $$\sum _{x\in X}\frac{1}{(1+f(x))^N} $$ and the series $$\sum_{\substack{x\in X\\f(x)>0}}\frac{1}{f(x)^N}$$converge for $N>d+\epsilon$. 
\end{lem} 
\begin{proof}
 For every $i\in\ZZ_{\geq 1}$, let us set $$w(i)=|\{x\in X| i-1\leq f(x)<i\}|.$$
Let $B\geq 1$ be an integer. By Abel's summation formula, for $N>d+\epsilon$ we have that\begin{align*}
\sum_{\substack{x\in X\\f(x)\leq B}}\frac1{(1+f(x))^N}&=\sum _{i=1}^{B+1}\frac{w(i)}{i^N}\\
&=\frac{1}{(1+B)^N}\sum_{r=1}^{B+1}w(r) + \sum_{i=1}^{B}\big(\frac{1}{i^N}-\frac{1}{(i+1)^N}\big)\sum_{j=1}^{i}w(j)\\
&\leq\frac{AB^d}{(1+B)^N}+\sum_{i=1}^{B}\big(\frac{1}{i^N}-\frac{1}{(1+i)^N}\big)Ai^d\\
&\leq A+\sum_{i=1}^{B}Ai^d\frac{2^Ni^{N-1}}{i^N(i+1)^N}\\
&\leq A+\sum_{i=1}^{B}\frac{2^NA}{i^{N+1-d}}\\
&\leq A+2^NA\zeta(N+1-d).
\end{align*} (we have used that $(1+i)^N-i^N\leq 2^{N}i^{N-1}$).
It follows that $\sum_{\substack{x\in X}}\frac1{(1+f(x))^N}$ converges uniformly for $N>d+\epsilon$. Moreover, for $N>d+\epsilon$, one has that \begin{align*}\sum_{\substack{x\in X\\f(x)>0}}\frac{1}{f(x)^N}&=\sum_{\substack{x\in X\\1>f(x)>0}}\frac{1}{f(x)^N}+\sum_{\substack{x\in X\\f(x)>1}}\frac{1}{f(x)^N}\\
&\leq\sum_{\substack{x\in X\\1>f(x)>0}}\frac{1}{f(x)^N}+\sum_{\substack{x\in X\\f(x)>1}}\frac{2}{(1+f(x))^N}.
\end{align*}
As the sum $ \sum_{\substack{x\in X\\1>f(x)>0}}\frac{1}{f(x)^N}$ is a finite sum and as the sum $\sum_{\substack{x\in X\\f(x)>1}}\frac{2}{(1+f(x))^N}$ converges uniformly for $N>d+\epsilon$, we deduce that $\sum_{\substack{x\in X\\f(x)>0}}\frac{1}{f(x)^N}$ converges uniformly for $N>d+\epsilon$. The statement is proven.
\end{proof}
We set $$\wH^*(\sss,\chi):=\wH(\sss,\chi)\prodjn \frac{s_j+im(\chij)-1}{s_j+im(\chi^{(j)})}. $$
\begin{mydef} \label{defgk}We define formally $$g_K(\sss):=\sum _{\chi_0\in\AK}\wH(\sss,\chi_0) $$ and $$g_K^*(\sss):=\sum _{\chi_0\in\AK}\wH^*(\sss,\chi_0).$$
\end{mydef}
\begin{prop}\label{vuvui}
Let $K\subset K^\aaa_{\max}$ be an open subgroup and let $\alpha >0$. There exist $\delta=\delta (\alpha) >0$ and $\beta=\beta (\alpha)>0$ such that the following conditions are satisfied:
\begin{enumerate}
\item The series \begin{equation}\label{nezce}
g^*_K(\sss):=\sum _{\chi_0\in\AK}\wH^* (\sss,\chi_0)
\end{equation}
converges absolutely and uniformly on compacts in the domain $\Omega_{>1-\delta}$ and the function $\sss\mapsto g^*_K(\sss)$ is holomorphic in the domain $\Omega_{>1-\delta}$. 
\item For every compact $\mathcal K\subset\Omega_{>1-\delta}$ one has that there exists $C=C(\alpha,\mathcal K)>0$ such that for every $\sss\in\KiRn$ and every $\mmm\in M$ one has that $$|g_K^*(\sss+i\mmm)|\leq \frac{C(1+||\Im(\sss)||)^\beta}{(1+||\mmm ||)^\alpha}.$$
\end{enumerate} 
\end{prop}
\begin{proof}
Let $\alpha >0$. By \ref{siuyh}, there exist $\frac13>\delta>0 $ and $\beta>0$ such that for any compact $\mathcal K\subset\RR^n_{>1-\delta}$ there exists $C>0$ such that for every $\sss\in\KiRn,$ every $\chi_0\in\AK$ and every $\mmm\in M$ one has for $\chi=\chi_0\lvert\cdot\rvert^{i\mmm}$ that:\begin{align}\begin{split}\label{sveol}\bigg|\wH(\sss+i\mmm,\chi_0)\prod _{\chij_0=1}\frac{s_j+im_j-1}{s_j+im_j}\bigg|&= \bigg|\wH(\sss,\chi_0\lvert\cdot\rvert^{i\mmm})\prod _{\chij_0=1}\frac{s_j+im(\chij)-1}{s_j+im(\chij)}\bigg|\\&\leq\frac{C(1+||\Im(\sss)||)^\beta}{((1+||\chi_0||_{\discrete})(1+||\mmm||))^\alpha}.\end{split}\end{align} We prove that the series (\ref{nezce}) converges absolutely and uniformly on compacts of $\Omega _{>1-\delta}$. Let $\mathcal G\subset \Omega_{>1-\delta}$ be a compact set and let $C(\mathcal G)>0$ be such that for every $\sss\in\mathcal G$ one has $||\Im(\sss)||<C(\mathcal G)$. Let $\mathcal K\subset \RR^n_{>1}$ be a compact such that $\mathcal G\subset\KiRn$. Note that for $j=1\doots n$ and $\sss\in\KiRn$ one has that $$\bigg| \frac{s_j-1}{s_j}\bigg|\leq 3.$$ Given $\alpha>nr_2+1$, it follows from (\ref{sveol}) that for every $\chi_0\in\AK$ and every $\sss\in\GG$ one has that\begin{align}\begin{split}\label{horio}\bigg|\wH(\sss,\chi_0)\prodjn\frac{s_j-1}{s_j}\bigg|&\leq \bigg|\frac{C(\alpha,\mathcal K)(1+||\Im(\sss)||)^\beta}{(1+||\chi_0||_{\discrete})^\alpha}\prod _{\chij\neq 1}3 \bigg|\\
&\leq\frac{(1+ C(\GG))^\beta C(\alpha,\mathcal K) 3^n}{(1+||\chi_0||_{\discrete})^\alpha} .
\end{split}
\end{align}
Now, by \ref{vatu}, there exists $B>0$ such that for every $A>0$ one has that $|\{\chi_0\in\AK| \hspace{0.1cm}||\chi_0||_{\discrete}<A\}|\leq BA^{nr_2}$. The set $\AK$ is discrete (Lemma \ref{vnogok}), and therefore (\ref{horio}) and \ref{asioo} give that the series $\sum _{\chi_0\in\AK}\frac{1}{(1+||\chi_0||)^\alpha}$ converges. We deduce that (\ref{nezce}) converges absolutely for every $\sss\in\GG$. Moreover, for $M>0$ one has that if $||\chi_0||_{\discrete}>M$, then $$|\wH(\sss,\chi_0)|\leq\frac{C(\mathcal G)^\beta C (\alpha,\mathcal K) 3^n}{(1+M)^\alpha}$$ for every $\sss\in\mathcal G$, and, hence, the convergence is uniform on compacts. As for every $\chi_0\in\AK$, the function $\sss\mapsto\wH(\sss,\chi_0)\prodjn\frac{s_j-1}{s_j}$ is holomorphic in the domain $\Omega _{>1-\delta}$, we deduce that $\sss\mapsto g^*_K(\sss)$ is holomorphic in the domain $\Omega _{>1-\delta}$. Let $\mathcal K\subset\RR^n_{>1-\delta}$ be a compact. For every $\sss\in\KiRn$ and every $\mmm\in M$, it follows from (\ref{sveol}) that \begin{align*}
|g_K^*(\sss+i\mmm)|&=\bigg|\sum _{\chi_0\in\AK}\wH(\sss+i\mmm,\chi_0)\prodjn\frac{s_j+m_j-1}{s_j+m_j}\bigg|\\
&\leq \frac{C(1+||\Im(\sss)||)^\beta}{(1+||\mmm||)^\alpha}\sum _{\chi_0\in\AK}\frac{1}{(1+||\chi_0||_{\discrete})^\alpha}.
\end{align*}
The statement follows.
\end{proof}
The function $\sss\mapsto\prodjn\frac{s_j-1}{s_j}$ does not vanish in the domain $\Omega _{>1}$. We deduce that in this domain the series defining $$g_K(\sss)=g^*_K(\sss)\prodjn \frac{s_j}{s_j-1}$$ converges absolutely. Therefore $\sss\mapsto g_K(\sss)$ is a holomorphic function in this domain.
\chapter{Analysis of height zeta functions}
\label{Analysis of height zeta functions}
\section{Analysis of $M$-controlled functions}
In this section we adapt the analysis of \cite{FonctionsZ} to our needs. We recall the definition of $M$-controlled functions and establish properties of their integrals.
\subsection{}
Let $d\geq 1$ be an integer and $U\subset\RR^d$ an open subset. For a vector subspace $M\subset \RR^d$,  we say that a function $f:U+i\RR^d\to\CC$ is $M$-controlled if for every $\alpha>0,$ there exists $\beta>0$ such that for any compact $\mathcal K\subset U,$ there exists $C(\mathcal K)>0$ such that for every $\mmm\in M$ one has \begin{equation}\label{mcont}|f(\sss+i\cdot\mmm)|\leq\frac{C(\mathcal K)(1+||\Im(\sss)||)^{\beta}}{(1+||\mmm||)^{\alpha}}\end{equation} if provided $\Re(\sss)\in\mathcal K$.
\begin{rem}
\normalfont
Note that if~$f$ is $M$-controlled, then~$f$ is $M$-controlled in the sense of \cite[Section 4.3]{FonctionsZ}. There, the condition is that there exists a family of linear forms $(\ell_j)_j$ in $V^*$ such that the $\ell_j|M$ form a basis of $M$ and such that there exists $\beta'>0$ and $1>\epsilon>0$ for which for any compact $\mathcal K\subset U$ there exists $C'(\mathcal K)>0$ such that for every $\mmm\in M$ one has \begin{equation}
\label{mcontcomp}
|f(\sss+i\cdot\mmm)|\leq \frac{C'(\mathcal K)(1+||\Im(\sss)||)^{\beta'} }{(1+||\mmm||)^{1-\epsilon}}\frac{1}{\prod (1+|\ell_j(\sss+\mmm|))}
\end{equation} if provided $\Re(\sss)\in \mathcal K.$ We verify that our condition is stronger. The inequality $$\frac{1}{1+|x+y|}\leq\frac{1+|x|}{1+|y|}\hspace{1cm} x,y\in\CC$$ gives that $$\frac{1+|\ell_j(\sss)|}{1+|\ell _j(\sss+\mmm)|}\geq \frac{1}{1+|\ell_j(\mmm)|}.$$Now there exist $C_0,C_1>0$ such that $$\prod _{\ell _j}{(1+|\ell_j(\sss)|)}\leq C_0{(1+||\Im(\sss))^{\dim M}}$$ and $$\frac{1}{\prod _{\ell_j}(1+|\ell_j(\mmm)|)}\geq \frac{C_1}{(1+||\mmm||)^{\dim M}},$$ provided that $\Re(\sss)\in\mathcal K$. Hence, if~$f$ satisfies (\ref{mcontcomp}) when provided $\Re(\sss)\in\mathcal K$, it satisfies that $$|f(\sss+i\mmm)|\leq\frac{C_0C_1C(\mathcal K)(1+||\Im(\sss)||)^{\beta '+\dim M}}{(1+||\mmm||)^{\dim M+1}}.$$
\end{rem}
The following result is given as Lemma  3.1.6 in \cite{FonctionsZ}.
\begin{lem}[{Chambert-Loir, Tschinkel, \cite[Lemma 3.1.6]{FonctionsZ}}]  \label{lclt}Let $q:\RR^n\to\RR^k$ be a surjective map and let $M\subset\RR^n$ be its kernel. Given $\sss\in\CC^n$, let us pick $\widetilde\sss\in\CC^n$ such that $q_{\CC}(\widetilde \sss)=\sss$. Endow $M$ with the unique Lebesgue measure $d\mmm$ such that $(dx_1\dots dx_n)/d\mmm=dx_1\dots dx_k.$ Suppose that $f:U+i\RR^k\to\CC$ is an $M$-controlled holomorphic function. The integral $$\frac{1}{(2\pi)^{n-k}}\int_{M}f(\sss+i\mmm)d\mmm$$ converges for every $\sss\in U+i\RR^n$ and the value of $$\mathscr S_M(f):\sss\mapsto \frac{1}{(2\pi)^{n-k}}\int_{M}f(\sss+i\mmm)d\mmm$$ does not depend on the choice of $\widetilde\sss$ and the resulting map $\mathscr S_M(f):q(U)+i\RR^k\to\CC$ is holomorphic and $\{0\}$-controlled.
\end{lem}
\begin{rem}
\normalfont
The original Lemma 3.1.6 has been simplified by assuming $M'=M$ (with the notation as in 3.1.6 in \cite{FonctionsZ}).
\end{rem}
The following result is Theorem 3.1.14 in \cite{FonctionsZ}.
\begin{thm}[{Chambert-Loir, Tschinkel, \cite[Theorem 3.1.14]{FonctionsZ}}]\label{CLTsc} Let $q:\RR^n\to\RR^{k}$ be a surjective linear map such that $q(\RR^n_{\geq 0})=\RR^{k}_{\geq 0}$ and let $M=\ker q$. 
Let $f:\Omega_{>0}\to\CC$ be a holomorphic function such that there exists an open ball $B\subset \RR^n$ centred at $0$ such that $\sss\mapsto f(\sss)\prodjn\frac{s_j}{s_j+1}$ extends to an $M$-controlled holomorphic function on $(B+i\RR^n)\cup\Omega_{>0}.$ 
Then there exists an open neighbourhood $B'$ of $0$ in $\RR^k$ such that $\mathscr S_{M}(f)\prod _{j=1}^k\frac{s_j}{s_j+1}$ extends to a holomorphic $M$-controlled function in the domain $(B'+i\RR^k)\cup (\RR^k_{>0}+i\RR^k)$. 
Moreover, if one has for every $\xxx\in\RR^n_{>0}$ that $$\lim _{s\to 0^+}\big(s^n\prodjn x_j\big){f(s\xxx)}=a,$$ for some $0\neq a\in\RR,$ then one has for every $\xxx'\in\RR^{k}_{>0}$ that $$\lim _{s\to 0^+}\big(s^{k}\prod _{j=1}^{k}x'_j \big)\mathscr S_M(f)(s\xxx')=a.$$
\end{thm}
\begin{rem}
\normalfont
The original statement of \ref{CLTsc} is somewhat simplified here. With notation as in Theorem 3.1.14 of \cite{FonctionsZ}, we have supposed that $M=M'$ and $C=\Lambda=\RR^{n}_{\geq 0}$.  We have also added the condition $q(\RR^n_{\geq 0})=\RR^{k}_{\geq 0}$ in order to make calculations of the characteristic functions of cones (\cite[Section 3.1.7]{FonctionsZ}) simple. For every $k\in\ZZ_{\geq 1}$, the cone $\RR^k_{>0}$ is simplicial and its characteristic function (according to \cite[Section 3.1.7]{FonctionsZ}) is given by $$\Omega_{>0}\to\CC\hspace{1cm}\sss\mapsto\frac{1}{\prod _{j=1}^ns_k}.$$
\end{rem}
\subsection{} We recall several facts from abstract Fourier analysis. 

Let~$G$ be an abelian locally compact group and let $dg$ be a Haar measure on~$G$. For $f\in L^1(G)$ and $\chi\in G^*$, we denote the Fourier transform of~$f$ by $$\widehat f(\chi):=\int_{G}(f{\chi})(dg)^*.$$
By $(dg)^*$ or by $dg^*$ we denote the dual measure on the character group $G^*$ (see \cite[Definition 4, \no 3, \S 1, Chapter II]{TSpectrale}). It is characterized by the following property: it is the unique Haar measure on $G^*$ such that {\it Fourier inversion formula} (\cite[Proposition 4, \no 4, \S 1, Chapter II]{TSpectrale}) is valid \begin{equation}\label{FourierInversion}f(x)=\int_{G^*}\overline{\chi(g)}\widehat f(\chi)(dg^*(\chi))\end{equation}for every $x\in G$. 
\begin{lem}[{\cite[Proposition 9, \no 8, \S 1, Chapter II]{TSpectrale}}] \label{trcio} Let~$H$ be a closed subgroup of~$G$ and let $dh$ be a Haar measure on~$H$. The dual measure $(dg/dh)^*$ on $(G/H)^*=H^{\perp}$ of the measure $dg/dh$ on $B/A$ satisfies that $$(dg)^*/((dg/dh)^*)= (dh)^*.$$  
\end{lem}
\begin{lem}[{\cite[Proposition 11, \no 9, \S 1, Chapter II]{TSpectrale}}] \label{dualofdiscrete}
The group $G^*$ is compact if and only if~$G$ is discrete. If $dg$ is the counting measure on~$G$, then $(dg)^*$ is normalized by $(dg)^*(G^*)=1$.
\end{lem}
\begin{prop}[{\cite[Proposition 8, \no 8, \S 1, Chapter II]{TSpectrale}}] \label{formuledepoison}
Let~$H$ be a closed subgroup of an abelian locally compact group~$G$. Let $dh$ and $dg$ be Haar measures on~$H$ and~$G$, respectively. Let $f\in L^1(G)$. We suppose that \begin{enumerate}
\item the restriction of the Fourier transform $\widehat f|_{H^{\perp}}$ is an element of $L^1(H^{\perp})=L^1((G/H)^*)$,
\item for every $x\in G$, one has that $(h\mapsto f(xh))\in L^1(H)$,
\item the function $x\mapsto\int_Hf(xh)dh$ is a continuous function $G\to\CC$.
\end{enumerate}
Then Poisson formula is valid:
$$\int_Hf(h)dh=\int_{H^{\perp}}(\widehat f\hspace{0.1cm}) (dg/dh)^*.$$
\end{prop}
\section{Height zeta function}\label{hzsec}
We define and prove holomorphicity of a height zeta function. Let $(f_v:\Fvnz\to\RR_{>0})_v$ be an adelic toric family of weighted degree~$d$. 
\subsection{}
We define height zeta functions.
\begin{mydef}\label{hzfdef}
For $\sss\in\CC^n$ we formally define series \begin{equation*}\accentset{\circ}Z((f_v)_v)(\sss)=\accentset{\circ}Z(\sss):=\sum _{\xxx\in[\TTa(F)]}H(\sss,\xxx)^{-1} \end{equation*}
and \begin{equation*}
Z((f_v)_v)(\sss)=Z(\sss):=\sum _{\xxx\in[\DD(P)(F)]}H(\sss,\xxx)^{-1}
\end{equation*}
\end{mydef} 
\begin{prop}\label{norco}
Let $(f_v:\Fvnz\to\RR_{\geq 0})_v$ be a degree $|\aaa|$ quasi-toric family of $\aaa$-homogenous functions having at worst logarithmic singularities along a rational divisor (situation of \ref{nortlogfalt}). For $\sss\in\Omega_{>0}$ we set $H(\sss,-):=H^{\frac{\aaa\cdot\sss}{|\aaa|}}$. Let $\epsilon>0$. The height zeta function series defining $Z(\sss)$ and $\accentset{\circ}Z(\sss)$ converge absolutely and uniformly in the domain $\Omega_{>1+\epsilon}$, and defines a holomorphic function in this domain. 
\end{prop}
\begin{proof}
By \ref{nortlogfalt}, one has that there exists $C>0$ such that $$|\{\xxx\in[\DD(P)(F)]| H(\xxx)\leq B\}|\leq CB^{1+\epsilon/2}$$ for every $B>0$. Note that if $\sss\in\Omega_{>1+\epsilon}$ then one has that $\frac{\aaa\cdot\Re(\sss)}{|\aaa|}>1+\epsilon>1+\epsilon/2$. Thus by \ref{asioo}, we have that the series $$\sum _{\xxx\in [\DD(P)(F)]}H(\sss,\xxx)^{-1}=\sum_{\xxx\in [\DD(P)(F)]} H(\xxx)^{-\frac{\aaa\cdot\sss}{|\aaa|}} $$ converges absolutely and uniformly in the domain $\Omega_{>1+\epsilon}$. It follows that the function~$Z$ is holomorphic in this domain. One has that $$Z(\sss)=\accentset{\circ}Z(\sss)+\sum_{\xxx\in[\DD(P)(F)]\cap([\PPP(\aaa)(F)]-[\TTa(F)])}H(\sss,\xxx).$$
Let $\delta>0$ be such that $(1+\delta)(|\aaa|-\min_ja_j)/|\aaa|\leq \max (1/2,1-\epsilon)$. Corollary \ref{nortlogfalt} gives that there exists $C'(\delta)>0$ such that \begin{multline*}|\{\xxx\in[\DD(P)(F)]\cap([\PPP(\aaa)(F)]-[\TTa(F)])| H(\xxx)\leq B\}|\\\leq C'(\delta)B^{(1+\delta)(|\aaa|-\min_ja_j)/|\aaa|}.\end{multline*} Note that if $\sss\in\Omega_{>\max(1/2,1-\epsilon)},$ then $\frac{\aaa\cdot\Re(\sss)}{|\aaa|}>\max(1/2,1-\epsilon).$ Lemma \ref{asioo} gives that \begin{multline*}\sum_{\xxx\in[\DD(P)(F)]\cap([\PPP(\aaa)(F)]-[\TTa(F)])}H(\sss,\xxx)\\=\sum_{\xxx\in[\DD(P)(F)]\cap([\PPP(\aaa)(F)]-[\TTa(F)])}H(\xxx)^{-\frac{\aaa\cdot\Re(\sss)}{|\aaa|}}\end{multline*} converges absolutely and uniformly in the domain $\sss\in\Omega_{>\max(1/2,1-\epsilon)}$. It follows that the function defined by the series is holomorphic. Consequently, the series defining $\accentset{\circ}Z$ converges absolutely and uniformly in the domain $\sss\in\Omega_{>1+\epsilon}$ and defines a holomorphic function in this domain.
\end{proof}
\subsection{}The goal of this paragraph is to apply Poisson formula to understand the analytic behaviour of the height zeta series.

We suppose $n\geq 1$ is an integer and $\aaa\in\ZZ^n_{>0}$ if $n\geq 2$ and $\aaa=a\in\ZZ_{>1}$ if $n=1$.
Recall that in \ref{hzfdef} for $\sss\in \CC^n$, we have defined formally \begin{equation*}\accentset{\circ}{Z}(\sss)=\sum _{\xxx\in\TTa(F)}H(\xxx)^{-\frac{\aaa\cdot\sss}{|\aaa|}}\end{equation*} and in \ref{norco}, we have established that there exists $\gamma>0$ such that the series converges absolutely and uniformly for any $\sss\in\Omega_{>\gamma}$ and defines a holomorphic function in this domain. 

For $\vMFz-S$ (that is the finite places~$v$ for which $f_v$ is toric), by Lemma \ref{ajvis}, one can take an open subgroup $K_v\subset\TTa(\Ov)$, such that if $\chi_v\in\TTa(\Fv)^*$ does not vanish at $K_v$, then $\wH_v(\sss,\chi_v)=0$ for every $\sss\in\Omega_{>0}.$ We set $K_v=\TTa(\Ov)$ for $v\in M_F^0\cap S$. Let us set $K=\prod _{\vMFz}K_v$ and let $\AK$ be the group of characters $\TTa(\AAF)_1\to S^1$ which vanish on $\TTa(i)(\TTa(F))$ and on $K.$ By \ref{akfgag}, the group $\AK$ is a finitely generated abelian group. 

In \ref{defgk}, we have defined $g_K(\sss)=\sum_{\chi_0\in\AK}\wH(\sss,\chi_0)$. By \ref{vuvui} one has that $g_K(\jed+\sss)$ converges absolutely and uniformly in the domain $\Omega_{>0}$, that $\sss\mapsto g_K(\jed+\sss)$ is $M$-controlled and holomorphic function in the domain $\Omega_{>0}$ and that there exists $\delta>0$ such that $\sss\mapsto g_K(\jed+\sss)\prodjn\frac{s_j}{s_j+1}$ extends to a holomorphic $M$-controlled function in the domain $\Omega_{>-\delta}$.

The following lemma will be used to determine the exact constant in Poisson formula.
\begin{lem}\label{calculdualmeas}
The measure $(\mu_{\AAF}/\coun_{[\TTa(i)]([\TTa(F)])})^*$ on \begin{multline*}([\TTa(\AAF)]/[\TTa(i)]([\TTa(F)]))^*\\=\big([\TTa(\AAF)]_1/[\TTa(i)]([\TTa(F)])\big)^*\times(\RR^n_{>0}/(\Rgz)_{\aaa})^*\end{multline*} (the identification follows from the identification \ref{silmo} and \ref{trcio}) satisfies that \begin{multline*}(\mu_{\AAF}/\coun_{[\TTa(i)]([\TTa(F)])})^*\\=\frac{1}{E(\aaa)}\coun_{([\TTa(\AAF)]_1/[\TTa(i)]([\TTa(F)]))^*}\times (d^*\mathbf r_{\aaa})^*,\end{multline*} where $$E(\aaa):=\frac{|\Sh^1(F, \mu_{\gcd(\aaa)})|\Res(\zeta_F,1)^{n-1}\Delta(F)^{\frac{n-1}2}}{|\mu_{\gcd(\aaa)}(F)|}.$$ and where we write $d^*\mathbf r_{\aaa}$ for the measure $(d^*r_1\dots d^*r_n)/(\lambda_{\aaa})_*(d^*r)$ on $\RR^n_{>0}/(\Rgz)_{\aaa}.$
\end{lem}
\begin{proof}
Let $\mu_1$ be the Haar measure on $[\TTa(\AAF)]_1$ normalized by $\mu_{\AAF}/\mu_1=d^*\mathbf r_{\aaa}$. For the above identification, one has that $$\mu_{\AAF}/\coun_{[\TTa(i)]([\TTa(F)])}=\mu_1/\coun_{[\TTa(i)]([\TTa(F)])}\times d^*\mathbf r_{\aaa}.$$We denote $\widetilde{\mu_1}:=\mu_1/\coun_{[\TTa(i)]([\TTa(F)])}.$ This measure satisfies $$\widetilde {\mu_1}=\frac{|\Sh^1(F, \mu_{\gcd(\aaa))}|}{|\mu_{\gcd(\aaa)}(F)|}\overline{\mu_1},$$ where $\overline{\mu_1}$ is the measure from \ref{deftamnum}. Proposition \ref{tamagawatta} gives that \begin{align*}
\widetilde{\mu_1}([\TTa(\AAF)]_1/[\TTa(i)]([\TTa(F)]))\hskip-5cm&\\&=\frac{|\Sh^1(F, \mu_{\gcd(\aaa)})|}{|\mu_{\gcd(\aaa)}(F)|}\overline{\mu_1}([\TTa(\AAF)]_1/[\TTa(i)]([\TTa(F)]))\\
&=\frac{|\Sh^1(F, \mu_{\gcd(\aaa)})|\Res(\zeta_F,1)^{n-1}\Delta(F)^{\frac{n-1}2}}{|\mu_{\gcd(\aaa)}(F)|}\\
&=E(\aaa).\end{align*}
Using \ref{dualofdiscrete}, we obtain that $$\widetilde{\mu_1}^*=\frac{1}{E(\aaa)}\coun_{[\TTa(\AAF)]_1/[\TTa(i)]([\TTa(F)])^*}.$$ The statement follows.
\end{proof}
\begin{prop}
For every $\sss\in\Omega_{>0}$, both sides of $$\accentset{\circ}{Z}(\jed+\sss)=\frac{|\Sh^1(F,\mu_{\gcd(\aaa)})|}{E(\aaa)}\int _{M}\gk(\jed+\sss+i\cdot\mmm)d\mmm$$ converge and the equality is valid.
\end{prop}
\begin{proof}
By \ref{cirti}, the kernel of the map $[\TTa(F)]\to [\TTa(\AAF)]$ is isomorphic to the finite group $\Sh^1(F,\mu _{\gcd (\aaa)}).$ Using this and the fact that $$H(\xxx)^{-\frac{\aaa\cdot\sss}{|\aaa|}}=H(\sss,\TTa(i)(\xxx))$$ (Lemma \ref{hsxhl}), we deduce that$$\accentset{\circ}Z(\jed+\sss)=|\Sh^1(F,\mu_{\gcd (\aaa)})|\sum _{\xxx\in[\TTa(i)]([\TTa(F)])}\wH(\jed+\sss,\xxx)^{-1}.$$
Poisson formula (Proposition \ref{formuledepoison}) applied to the inclusion $$[\TTa(i)]([\TTa(F)])\subset [\TTa(\AAF)]$$ gives that \begin{multline*}\sum_{\xxx\in[\TTa(i)]([\TTa(F)])}\wH(\jed+\sss,\xxx)^{-1}\\=\int _{([\TTa(\AAF)]/[\TTa(i)]([\TTa(F)]))^*}\wH(\jed+\sss,-)(\mu_{\AAF}/\coun_{[\TTa(i)]([\TTa(F)])})^*,\end{multline*} for every $\sss$ for which the both sides converge and, hence, \begin{multline*}\accentset{\circ}Z(\jed+\sss)=|\Sh^1(F,\mu _{\gcd (\aaa)})|\times\\ \times\int _{[\TTa(\AAF)]/[\TTa(i)]([\TTa(F)])^*}\wH(\jed+\sss,-)(\mu_{\AAF}/\coun_{[\TTa(i)]([\TTa(F)])})^* ,\end{multline*} for every $\sss$ that the both sides converge. 
By Lemma \ref{nebrlj}, the homomorphism $$\xi^n:\RR^n\to (\Rgz^n)^*\hspace{1cm} \xxx\mapsto \big(\rrr\mapsto \prodjn r_j^{2i\pi x_j}\big)$$ induces an isomorphism $\xi^n|_M:M\to (\Rgz/(\Rgz)_{\aaa})^*,$ which satisfies $(\xi^n|_M)_*d\mmm=(d^*\mathbf r_{\aaa})^*$. Now, \ref{calculdualmeas} and Fubini theorem give that \begin{multline*}\accentset{\circ}Z(\sss)\\=\frac{|\Sh^1(F,\mu_{\gcd (\aaa)})|}{E(\aaa)}\int _{M}\sum _{\chi_0\in [\TTa(\AAF)]_1/[\TTa(i)]([\TTa(F)])^*}\wH(\sss,\chi_0\xi^n(\mmm))d\mmm.
\end{multline*}Lemma \ref{signpr} gives $\wH(1+\sss,\chi_0\xi^n(\mmm))=\wH(1+\sss+2i\pi\mmm,\chi_0).$ Moreover, it follows from \ref{ajvis} that $\wH(1+\sss+2i\pi,\chi_0)=0,$ for every $\chi_0\in[\TTa(\AAF)]_1/[\TTa(F)]^*-\AK.$ We deduce that \begin{align*}\sum _{\chi_0\in([\TTa(\AAF)]_1/[\TTa(i)]([\TTa(F)]))^*}\wH(\jed+\sss,\chi_0\xi ^n(\mmm))\hskip-2cm&\\&=\sum_{\chi_0\in\AK}\wH(\jed+\sss+2i\pi\mmm,\chi_0)\\&=g_K(\jed+\sss+2i\pi\mmm).\end{align*} 
Therefore, \begin{equation*}\accentset{\circ}Z(\jed+\sss)=\frac{|\Sh^1(F,\mu_{\gcd(\aaa)})|}{E(\aaa)}\int_Mg_K(\jed+\sss+2i\pi\mmm)d\mmm.\end{equation*} We introduce $$G_K(\jed+\sss):=\int _{M}g_K(\jed+\sss+2i\pi\mmm)d\mmm,$$which, if the integral converges, by Part (1) of Lemma \ref{nebrlj}, is the same as $$G_K(\jed+\sss)=\frac{1}{(2\pi)^{n-1}}\int_Mg_K(\jed+\sss+i\cdot\mmm)d\mmm $$The function $\sss\mapsto g_K(\jed+\sss)$ is $M$-controlled and so the integral defining $G_K$ converges for $\sss\in\Omega_{>0}$ by \ref{lclt}.  By \ref{norco}, the series defining $\accentset{\circ} Z(\jed+\sss)$ converges for $\sss\in\Omega_{>0}$. We get \begin{align*}\accentset{\circ}Z(\jed+\sss)&=\frac{ |\Sh^1(F,\mu_{\gcd(\aaa)})| }{E(\aaa)}G_K(\jed+\sss)\\
&=\frac{ |\Sh^1(F,\mu_{\gcd(\aaa)})| }{E(\aaa)}\int _{M}\gk(\jed+\sss+i\cdot\mmm)d\mmm
\end{align*} in the domain $\Omega_{>0}$.  
The claim is proven.\end{proof}
\begin{prop}\label{limofgk}Let $n\geq 1$. Let $\aaa\in\ZZ^n_{\geq 1}$ if $n\geq 2$ and let $\aaa=a\geq 2$ if $n=1$. For every $\xxx\in\RR^n_{>0}$ one has that $$\lim _{s\to 0^+}\big(\prodjn sx_j\big)\cdot \gk(\jed+s\xxx)=\frac{\Delta(F)^{\frac{n-1}2}\Res(\zeta_F,1)^{n-1}\tau}{|\mu_{\gcd(\aaa)}(F)|}.$$
\end{prop}
\begin{proof} Let $\xxx\in\RR^n_{>0}$. The series defining $$\gk(\jed+s\xxx)=\sum _{\chi_0\in\AK}\wH(\jed+s\xxx,\chi_0)$$ converges absolutely and uniformly when $s>0$, so we can exchange the limit and the sum, we get:\begin{equation}\label{posum}\lim_{s\to 0^+}\big(\prodjn  sx_j\big)\cdot \gk(\jed+s\xxx)= \sum _{\chi_0\in\AK}\lim _{s\to 0^+}\big(\prodjn sx_j\big)\wH((1+sx_j)_j,\chi_0).\end{equation}
By Proposition \ref{fintr}, we have that for every $\chi_0\in\AK$, there exists a holomorphic function $\phi(-,\chi_0):\Omega_{>\frac23}\to\CC$ such that one has an equality of the meromorphic functions$$\wH(\sss,\chi_0)=\phi(\sss,\chi_0)\prodjn L(s_j,\chij_0)\prod_{\vMFi}\wH_v(\sss,\chi_{0v}) $$ for $\sss\in\Omega_{>1}.$
Suppose $\chi_0\neq 1$. There exists index~$k$ such that $\chi^{(k)}_0\neq 1,$  and the function $s\mapsto L(s,\chi^{(k)}_0)$ is entire. For $j\neq k$, we have that $$\lim_{s\to 0^+} sx_jL(1+sx_j,\chij_0)=\lim _{s\to 0^+}sx_j\zeta_F(1+sx_j)$$ exists. We conclude that \begin{multline*}\lim _{s\to 0^+}\big(\prodjn sx_j\big)\wH((1+sx_j)_j,\chi_0)\\=\lim _{s\to 0^+}\phi ((1+sx_j)_j,\chi_0)\prodjn sx_jL(1+sx_j,\chij_0) \prod_{\vMFi}\wH_v((1+sx_j)_j,\chi_{0v})\\=0. \end{multline*}
Therefore, the only surviving term in the sum on the right hand side of (\ref{posum}) is, hence, for $\chi_0=1$. Lemma \ref{torfour} gives for every $v\in M_F^0-S$ (that is for every finite~$v$ such that $f_v$ is toric) that$$\wH_v((1+sx_j)_j,1)=\zeta_v(\aaa\cdot(1+sx_j)_j)^{-1}\prodjn \zeta_v(1+sx_j).$$ Hence we have
\begin{multline*}
\wH((1+sx_j)_j,1)\prodjn\zeta_F(1+sx_j)^{-1}\\=\prod_{v\in M_F^0}\bigg(\wH_v((1+sx_j)_j,1)\prodjn\zeta_v(1+sx_j)^{-1}\bigg)\times \prod_{\vMFi}\wH_v((1+sx_j)_j,1)\\
=\prod_{v\in S\cap M_F^0}\bigg(\wH_v((1+sx_j)_j,1)\prodjn\zeta_v(1+sx_j)^{-1}\bigg)\prod_{v\in M_F^0-S}\zeta_v(\aaa\cdot(1+sx_j)_j)^{-1}\times \\ \times\prod _{\vMFi}\wH_v((1+sx_j)_j,1).
\end{multline*}
When $s>0$, one has $\Re(\aaa\cdot (1+sx_j)_j)>1$ (because $\aaa\in\ZZ^n_{\geq 1}$ if $n\geq 2$ and $\aaa=a\in\ZZ^n_{\geq 2}$ if $n=1$), and thus the product $\prod _{v\in M_F^0-S}\zeta_v(\aaa\cdot(1+sx_j)_j)^{-1}$ converges to $$\zeta_F(\aaa\cdot (1+sx_j)_j)^{-1} \prod _{v\in S\cap M_F^0}\zeta_v(\aaa\cdot(1+sx_j)_j).$$
We deduce \begin{multline}\wH((1+sx_j)_j,1)\\
=\bigg(\prodjn\zeta_F(1+sx_j)\bigg)\prod_{v\in S\cap M_F^0}\bigg(\wH_v((1+sx_j)_j,1)\prodjn\zeta_v((1+sx_j)_j)^{-1}\bigg)\times\\ \times \zeta_F(\aaa\cdot (1+sx_j)_j)^{-1}\bigg(\prod_{v\in S\cap M_F^0}\zeta_v(\aaa\cdot(1+sx_j)_j)\bigg)\prod _{\vMFi}\wH_v((1+sx_j)_j,1)
\end{multline}
We will calculate the limit of the last product multiplied by $\prodjn sx_j$, when $s$ goes to zero. 
One has that \begin{equation}\label{lorom}\lim_{s\to 0^+}\big(\prodjn sx_j\big)\zeta_F(1+sx_j)=\Res(\zeta_F,1)^{n}.\end{equation}
Now we calculate: \begin{multline}\label{lorol}\lim _{s\to 0^+}\prod_{v\in S\cap M_F^0}\bigg(\wH_v((1+sx_j)_j,1)\prodjn\zeta_v(1+sx_j)^{-1}\bigg)\times \prod _{\vMFi}\wH_v((1+sx_j)_j,1)\\=\prod _{v\in S\cap M_F^0}\bigg(\wH_v(\jed,1)\zeta_v(1)^{-n}\bigg)\times \prod_{\vMFi}\wH_v(\jed,1).\end{multline}
By Lemma \ref{icicu}, for $\vMFz$ we have that \begin{align*}\wH_v(\jed,1)&=\zeta_v(1)^{n-1}\int_{[\TTa(\Fv)]}H_v(\jed,-)^{-1}\mu_v\\&=\zeta_v(1)^{n-1}\omega_v([\TTa(\Fv)])\\&=\zeta_v(1)^{n-1}\omega_v([\PPP(\aaa)(\Fv)]),\end{align*} where we have used that $\omega([\PPP(\aaa)(F_v)]-[\TTa(\Fv)])=0$ which we have established in \ref{ttadmz}. For $\vMFi$ we have that $$\wH_v(\jed,1)=\int_{[\TTa(\Fv)]}H_v(\jed,-)^{-1}\mu_v=\omega_v([\PPP(\aaa)(\Fv)]). $$
We conclude that the product on the right hand side of (\ref{lorol}) is equal to \begin{equation}\label{loromo}\prod_{v\in S\cap M_F^0}\zeta_v(1)^{-1}\omega_v([\PPP(\aaa)(\Fv)])\times \prod_{\vMFi}\omega_v([\PPP(\aaa)(\Fv)]). \end{equation}
Finally, we can calculate \begin{equation}\label{loromoo}
\lim _{s\to 0^+}\zeta_F(\aaa\cdot(1+sx_j)_j)^{-1}\prod _{v\in S\cap M_F^0}\zeta_v(\aaa\cdot(1+sx_j)_j)=\zeta_F(|\aaa|)^{-1}\prod _{v\in S\cap M_F^0}\zeta_v(|\aaa|).
\end{equation}
Using (\ref{lorom}), (\ref{loromo}) and (\ref{loromoo}), we conclude  \begin{multline*}\lim _{s\to 0^+}\big(\prodjn sx_j\big)\wH((1+sx_j)_j,1)\\=\frac{\Res(\zeta_F,1)^n}{\zeta_F(|\aaa|)}\prod_{v\in S\cap M_F^0}\frac{\zeta_v(|\aaa|)\omega_v(\PPP(\aaa)(\Fv))}{\zeta_v(1)}\times\prod_{\vMFi}\omega_v(\PPP(\aaa)(\Fv)). \end{multline*}
Using the formula given in \ref{locom} for $\tau$ we deduce that the last number is $$\Res(\zeta_F,1)^{n-1}\tau|\mu_{\gcd(\aaa)}(F)|^{-1}\Delta(F)^{\frac{n-1}2},$$ as claimed. 
\end{proof} 
\begin{thm}
There exists $\gamma>0$ such that the series  series defining \begin{equation}\label{zjed}\accentset{\circ}Z((s)_j)=\sum_{\xxx\in[\TTa(F)]}H(\xxx)^{-s}=\sum_{\xxx\in[\TTa(i)]([\TTa(F)])} H((s)_j,\xxx)^{-1}\end{equation} converges for $s\in \RR_{>\gamma}+i\RR$ and the function that associates to $s$ the value of (\ref{zjed}) is holomorphic in the domain $\Re(s)>\gamma$. There exists $1>\delta >0$ such that the function $s\mapsto \accentset{\circ}Z((s)_j)$ extends to a meromorphic function in the domain $\Re(s)>1-\delta$ with the only pole at $s=1$ which is simple, and such that for every compact $\mathcal K\subset\RR_{>1-\delta}$ there exists $C(\mathcal K)>0$ and $\beta(\mathcal K) >0$ such that $$\bigg|\frac{s-1}{s}Z((s)_j)\bigg|\leq C(\mathcal K)(1+|\Im(s)|)^{\beta(\mathcal K)}$$ if provided $\Re(s)\in\mathcal K$. We have further that $$\lim _{s\to 1}(s-1)\accentset{\circ}Z((s)_j)=\frac{\tau}{|\aaa|}.$$ \label{bigthmm}
\end{thm}
\begin{proof}
Let us firstly establish the convergence and holomorphicity. We have seen in \ref{norco} that there exists $\gamma>0$ such that the series defining $\accentset{\circ}Z(\sss)$ converges absolutely for $\sss\in\Omega_{>\gamma}$ and defines a holomorphic function in this domain. We deduce that the series defining $\acZ((s)_j)$ converges in the domain $\RR_{>\gamma}+i\RR$. The function $s\mapsto Z((s_j)_j)$ is holomorphic as it this the composition of the holomorphic map $\RR_{>\gamma}+i\RR\to \Omega_{>\gamma}$ which is given by $ s\mapsto (s)_j$ and the holomorphic map $\Omega_{>\gamma}\to\CC$ which is given by $\sss\mapsto \acZ(\sss).$ 

Let us now prove the meromorphic extension and the bound. The facts that $\sss\mapsto g_K(1+\sss)$ is holomorphic for $\sss\in\Omega_{>0}$ and that $\sss\mapsto g_K(\jed+\sss) \prodjn\frac{s_j}{s_j+1}$ is $M$-controlled in the domain $\Omega_{>-\delta'}$ for some $\delta'>0$ (Proposition \ref{vuvui}), enable us to apply Theorem \ref{CLTsc}. We apply this theorem for the map $\RR^n\to\RR, \xxx\mapsto\aaa\cdot\xxx$ and the function $$\sss\mapsto\frac{|\Sh^1(F,\mu_{\gcd(a_j)_j})|}{E(\aaa)}g_K(\mathbf 1+\sss+i\mmm)d\mmm. $$ We get that\begin{align*}s\mapsto \acZ((1+s)_j)&=\frac{|\Sh^1(F,\mu_{\gcd(\aaa)})|}{E(\aaa)(2\pi)^{n-1}}\int _Mg_K((1+s+m_j)_j)d\mmm\\&=\frac{|\Sh^1(F,\mu_{\gcd(a_j)_j})|}{E(\aaa)}\mathscr S_M(g_K)(1+|\aaa|\cdot s) \end{align*}is holomorphic in the domain $\Re(s)>0$ and that there exists $1>\delta>0$ such that $\frac{s}{s+1}\acZ((1+s)_j)$ extends to a holomorphic functions for $\Re(s)>-\delta$ and is $\{0\}$-controlled in this domain. This means that for  every compact $\mathcal K\subset \RR_{>-\delta}$, there exist $C(\mathcal K),\beta(\mathcal K)>0$ such that $$\bigg|\frac{s\acZ((1+s)_j)}{s+1}\bigg|\leq C(\mathcal K)(1+|\Im(s)|)^{\beta(\mathcal K)}$$ provided that $\Re(s)\in\mathcal K$. Pick a compact $\mathcal K\subset\RR_{>1-\delta}$. For every $s$ with $\Re(s)\in\mathcal K$, one has that $$\bigg|\frac{s-1}s\acZ(s)\bigg|=\bigg|\frac{(s-1)\acZ((1+(s-1))_j)}{(s-1)+1}\bigg|\leq C(\mathcal K-1)(1+|\Im(s)|)^{\beta(\mathcal K-1)} ,$$ where $\mathcal K-1=\{x-1|x\in\mathcal K\}$.

Let us calculate the the limit. By Proposition \ref{limofgk}, for every $\xxx\in\Rgz^n$ one has $$\lim_{s\to 0^+}s^ng_K(1+s\xxx)\prodjn x_j=\Res(\zeta_F,1)^{n-1}|\mu_{\gcd(\aaa)}(F)|^{-1}\Delta(F)^{\frac{n-1}2}\tau.$$ Now, the second part of Theorem \ref{CLTsc} gives that for every $x>0$ one has that \begin{align*}\lim_{s\to 0^+}s\acZ((1+sx)_j)&=\lim _{s\to 0^+}\frac{|\Sh^1(F,\mu_{\gcd (\aaa)})|}{E(\aaa)}sx\mathscr S_M(g_K)(\mathbf 1+|\aaa|sx)\\&=\frac{|\Sh^1(F,\mu_{\gcd (\aaa)})|}{E(\aaa)}\lim _{s\to 0^+}sx\mathscr S_M(g_K)(\mathbf 1+|\aaa|\cdot sx)\\
&=\frac{|\Sh^1(F,\mu_{\gcd (\aaa)})|}{|\aaa|E(\aaa)}\lim _{s\to 0^+}{|\aaa|sx}\mathscr S_M(g_K)(\mathbf 1+|\aaa|\cdot sx)\\
&=\frac{|\Sh^1(F,\mu_{\gcd (\aaa)})|}{|\aaa|E(\aaa)}\lim _{s\to 0^+}s^ng_K(1+s\xxx)\prodjn x_j\\
&= \frac{|\Sh^1(F,\mu _{\gcd (\aaa)})|\Res(\zeta_F,1)^{n-1}\Delta(F)^{\frac{n-1}2}\tau}{|\aaa|\cdot E(\aaa)\cdot |\mu_{\gcd(\aaa)}(F)|}\\
&=\frac{\tau}{|\aaa|}
.\end{align*} We deduce that the function $s\mapsto \acZ((1+s)_j)$ admits a pole of order~$1$ at $0$, which is simple and $$\Res(s\mapsto \acZ((1+s)_j),1)=\frac{\tau}{|\aaa|}.$$
The statement follows.
\end{proof}
By using the Tauberian result given as \cite[Theorem A1]{FonctionsZ} we deduce the following theorem.
\begin{thm}\label{osnovna}
Let $(f_v)_v$ be a quasi-toric degree $|\aaa|$ family of $\aaa$-homogenous smooth functions. One has that $$\{\xxx\in[\PPP(\aaa)(F)]| H(\xxx)\leq B\}\sim\frac{\tau}{|\aaa|}B,$$ when $B$ tends to $+\infty$.
\end{thm}
\begin{proof}
We establish firstly the following fact $$\{\xxx\in[\TTa(F)]| H(\xxx)\leq B\}\sim\frac{\tau}{|\aaa|}B.$$
Let $W\subset[\TTa(F)]$ be the set of points for which $H(\xxx)<1$, it is finite by \ref{nortlogfalt}. We define $\acZ_{\TTa(F)-W}((s)_j)=\sum _{\xxx\in\TTa(F)-W}H(\xxx)^{-s}.$ It follows from \ref{bigthmm} that the series defining $\acZ_{\TTa(F)-W}((s)_j)$ converges absolutely in the domain $\RR_{>\gamma}+i\RR$ and that there exists $\delta>0$ such that the function $s\mapsto \acZ_{\TTa(F)-W}$ extends meromorphically to the domain $\Re(s)>1-\delta$ and has one and only one pole in this domain which is moreover simple with the residue $$\Res(\acZ_{\TTa(F)-W},1)=\frac{\tau}{|\aaa|}.$$  By the absolute convergence of the defining series and the fact that we are summing only for those~$\xxx$ for which $H(\xxx)\geq 1$, the function $s\mapsto \acZ_{\TTa(F)-W}((s)_j)$ is decreasing on $]\gamma, +\infty[$. 
Let us pick $\mathcal K=[1-\delta/2,\gamma +1]$. It follows from \ref{bigthmm} and the triangle inequality, that there exists $C(\mathcal K),\beta(\mathcal K)>0$ such that for $\Re (s)\in\mathcal K$ one has that \begin{equation*} \bigg|\frac{(s-1)\acZ_{\TTa(F)-W}((s)_j)}{s}\bigg|\leq\bigg|\frac{s-1}{s}\bigg|\sum _{\xxx\in W}\big|H(\xxx)^{-s}\big| +C(\mathcal K)(1+|\Im(s)|)^{\beta(\mathcal K)}.\end{equation*} 
The function $s\mapsto \big|\frac{s-1}{s}\big|\sum _{\xxx\in W}\big|H(\xxx)^{-s}\big|$ is bounded for $s\in\mathcal K+i\RR$, say by $A>0.$  By the fact that $\acZ_{\TTa(F)-W}$ is decreasing on $]\gamma,+\infty[$ we deduce that for $\Re(s)>1-\delta/2$ one has that $$ \bigg|\frac{(s-1)\acZ(s)}{s}(\gamma+1)\bigg|\leq(A+C(\mathcal K)+\acZ_{\TTa(F)-W}(\gamma+1))(1+|\Im(s)|)^\beta.$$
Therefore, $\acZ_{\TTa(F)-W}$ satisfies the conditions we need for the Tauberian theorem. Our claim for the rational points of~$\TTa$ follows from the direct application of the theorem.

Let us now prove the statement of the theorem. By \ref{boundppatta}, we have that there exists $A>0$ such that for every $B>0$ one has that \begin{multline*}\{\xxx\in[\PPP(\aaa)(F)]-[\TTa(F)]|H(\xxx)\leq B\}\\\leq AB^{|\aaa|-\min_ja_j}\log(2+B^{|\aaa|-\min_ja_j}{|\aaa|})^{n^2(r_1+r_2)+n-1}.\end{multline*} 
Thus the statement follows.
\end{proof}
\section{Equidistribution of rational points} We study the ``equidistribution" of the set of rational points of a weighted projective stack in its adelic space.
\subsection{} In this paragraph we recall what do we mean by the  equidistribution. 
Let~$X$ be a compact topological space and let~$\mu$ be a measure on~$X$. Let $U$ be a subset of~$X$. Let $H:U\to \RR_{>0}$ be a function such that for every $B>0$, one has that $\{x\in U| H(x)\leq B\}$ is finite. We say that $U$ is equidistributed in $(X,\mu),$ (or simply in~$X$) with respect to~$H$ if for any open~$\mu$-measurable subset $W\subset X$ such that $\mu(\partial W)=0$ one has that $$\lim_{B\to\infty}\frac{|\{x\in U\cap W| H(x)\leq B\}|}{|\{x\in U| H(x)\leq B\}|}\to\frac{\mu(W)}{\mu(X)}.$$
\subsection{} Let $n\in\ZZ_{\geq 1}$ and let $\aaa\in\ZZ^n_{>0}$. We establish that the rational points of~$\PPP(\aaa)$ are equidistributed in the space $\prod_{\vMF}[\PPP(\aaa)(F_v)]$.

 If $\vMF$, we say that a function $f:[\PPP(\aaa)(\Fv)]\to\CC$ is smooth if its pullback $\Fvnz\to\CC$ is smooth. 
If $A\subset \CC$ is open, the sets of smooth functions $[\PPP(\aaa)(F_v)]\to A$ 
will be denoted by $\mathscr C^{\infty}([\PPP(\aaa)(F_v)],A).$ 

\begin{lem} \label{bumpfunctions} Let $\vMF$. 
\begin{enumerate}
\item There exists a unique structure of a~$v$-adic manifold on $[\PPP(\aaa)(\Fv)]$ such that for an open subset $A\subset [\PPP(\aaa)(F_v)]$ and $f\in\mathscr C^{0}(A,\CC)$ one has that $f\in \mathscr C^{\infty}(A,\CC)$ if and only if~$f$ is smooth in the usual sense (i.e.~$f$ is locally constant if $\vMFz$ and~$f$ is an infinitely differentiable function if $\vMFi$).
\item Let $f_v:\Fvnz\to\RR_{>0}$ be an $\aaa$-homogenous continuous function and let $\omega_v:=(f_v^{-1}dx_1\dots dx_n)/d^*x$ be the induced measure on $[\PPP(\aaa)(F_v)]$. Let $W\subset [\PPP(\aaa)(F_v)]$ be an open subset such that $\omega_v(\partial W)=0$. Let $\epsilon >0$. There exist smooth functions $h,g:[\PPP(\aaa)(\Fv)]\to\RR_{>0}$ such that $$0\leq h\leq \mathbf 1_W\leq g\text{ and }\int_{[\PPP(\aaa)(\Fv)]}(g-h)\omega<\epsilon.$$
\end{enumerate}
\end{lem}
\begin{proof}
\begin{enumerate}
\item We have seen in \ref{aboutopa}, that the action $$\Fvt\times(\Fvnz)\to(\Fvnz)\hspace{1cm}(t,\xxx)\mapsto (t^{a_j}x_j)_j$$ is proper, i.e. the canonical morphism $\Fvt\times(\Fvnz)\to(\Fvnz)\times(\Fvnz)$ is proper. Let $e:\Fvt\to(\Fvt)_{\aaa}$ be given by $e(t)=(t^{a_j})_j.$ The group $(\Fv)_{\aaa}=\{(t^{a_j})_j|t\in\Fvt\}$ acts on $\Fvnz$ by the component-wise multiplication. The two actions are compatible in the following sense: $t\cdot\xxx=e(t)\xxx$. Now the proper morphism $\Fvt\times(\Fvnz)\to(\Fvnz)\times(\Fvnz)$ factorizes as $$\Fvt\times(\Fvnz)\xrightarrow{(e,\Id_{\Fvnz})}(\Fvt)_{\aaa}\times(\Fvnz)\to(\Fvnz)\times(\Fvnz).$$The first morphism is surjective, hence by \cite[Proposition 5, \no 1, \S 10, Chapter I]{TopologieGj}, the second morphism is proper, i.e. the action of $(\Fvt)_{\aaa}$ on $\Fvnz$ is proper. By \cite[Paragraph 6.2.3]{VarieteDifj} we deduce that the quotient $[\PPP(\aaa)(\Fv)]=(\Fvnz)/\Fvt=(\Fvnz)/(\Fvt)_{\aaa}$ carries a unique structure of a compact~$v$-adic manifold and with this structure, the smooth functions in our sense are precisely the smooth functions on the~$v$-adic manifold $[\PPP(\aaa)(\Fv)]$.
\item To prove the statement, we use the existence of {\it bump} functions on~$v$-adic manifolds. We recall the proof for $\vMFz$ and for $\vMFi$ we refer the reader to \cite[Lemma 2.15]{SmoothManifolds}. Let $\mathcal K$ be a compact of $[\PPP(\aaa)(F_v)]$. If $U\supset \mathcal K$ is open, for every $\xxx\in \mathcal K$, there exists an open and closed neighbourhood $N_\xxx$ of~$\xxx$ contained in $U$ (because every point in the spaces $F_v^r$ for $r\geq 0$ has a basis of its neighbourhoods  given by open and closed balls). By the compacity of $\mathcal K$, there exist $\xxx_1\doots \xxx_\ell$ such that $N:=\bigcup N_{\xxx_\ell}\supset\mathcal K$. Moreover, $N\subset U$ is open and closed, thus $\mathbf 1_N$ is a bump function which extends the function $\mathbf 1_{\mathcal K}$. 

Now, for every open $U\supset \overline W,$ there exists a smooth function $g_U:[\PPP(\aaa)(\Fv)]\to [0,1]$ such that $g_U|_{\overline W}=1$ and $\supp (g_U)\subset U$. For every compact $K\subset W$ there exists a smooth function $h_K:K\to [0,1]$ such that $h_K|_K=1$ and $\supp(h_K)\subset W$. For every $\epsilon >0$, by the regularity of~$\omega$ and the fact that $\omega(\partial W)=0$, there exists open $U\supset \overline W$ and a compact $K\subset W$ such that $\omega(U)-\omega(K)<\epsilon$. It follows that $$\int_{[\PPP(\aaa)(F_v)]}(g_U-h_K)\omega\leq \int _{[\PPP(\aaa)(F_v)]}\mathbf 1_{U-K}\omega=\omega(U)-\omega(K)<\epsilon.$$
\end{enumerate}
\end{proof}
Let $(f_v:\Fvnz\to\RR_{>0})_v$ be a quasi-toric degree $|\aaa|$-family of $\aaa$-homogenous {\it smooth} functions. Let $H=H((f_v)_v)$ be the resulting height on $[\PPP(\aaa)(F)]$ and let $\omega =\omega((f_v)_v)$ be the resulting measure on $\prod_{\vMF}[\PPP(\aaa)(\Fv)]$. The goal of the rest of the paragraph is to establish that the set $[\PPP(\aaa)(i)]([\PPP(\aaa)(F)])$ is equidistributed in $\prod_{\vMF}[\PPP(\aaa)(\Fv)]$. We will write~$i$ for the map $[\PPP(\aaa)(i)]$ The following theorem is motivated by \cite[Proposition 3.3]{Peyre}.
\begin{thm}[``Equidistribution of rational points"] \label{quasitoricequi} The set $i([\PPP(\aaa)(F)])$ is equidistributed in $\prod_{\vMF}[\PPP(\aaa)(F_v)]$ with the respect to~$H$.
\end{thm}
\begin{proof}
In the case $n=1$ and $a_1=1$, the statement is trivially true. In the rest of the proof we suppose that $n\geq 1$ and that $\aaa\in\ZZ^n_{>0}$ if $n\geq 2$ and that $\aaa=a_1\in\ZZ_{>1}$ if $n=1$. The proof is adaptation of the proof of \cite[Proposition 3.3]{Peyre}. We split it in several parts.
\begin{enumerate}
\item In the first part we establish that the asymptotic for counting points of $[\PPP(\aaa)(F)]$ equals to $|\Sh^1(F,\mu_{\gcd(\aaa)})|$ times the asymptotic for counting the points of $i([\PPP(\aaa)(F)])$. By \ref{osnovna}, one has that \begin{equation*}|\{\xxx\in[\TTa(F)]| H(\xxx)\leq B\}|\sim_{B\to\infty}\frac{\omega(\prod_{\vMF}[\PPP(\aaa)(F_v)])}{|\aaa|}B.\end{equation*}
It follows from \ref{cesusp} that \begin{equation*}|\{\xxx\in i([\TTa(F)])| H(\xxx)\leq B\}|\sim_{B\to\infty}\frac{\omega(\prod_{\vMF}[\PPP(\aaa)(F_v)])}{|\Sh^1(F, \mu_{\gcd(\aaa)})|\cdot |\aaa|}B.\end{equation*}
It follows from \ref{boundppatta}, that for every $B>0$ there exists $C'>0$ such that\begin{multline*}|\{\xxx\in i([\PPP(\aaa)(F)]-[\TTa(F)])|H(\xxx)\leq B\}|\\ \leq C' B^{\frac{|\aaa|-\min_ja_j}{|\aaa|}}\log(2+B)^{n^2(r_1+r_2)+n-1}.
\end{multline*}
We deduce that \begin{equation*}|\{\xxx\in i([\PPP(\aaa)(F)])| H(\xxx)\leq B\}|\sim_{B\to\infty}\frac{\omega(\prod_{\vMF}[\PPP(\aaa)(F_v)])}{|\Sh^1(F, \mu_{\gcd(\aaa)})|\cdot |\aaa|}B.\end{equation*}

\item We are going to prove the claim for the open subsets $W\subset\prod_{\vMF}[\PPP(\aaa)(F_v)]$ of the form $$W=\prod_{v\in S_W} W_v\times \prod_{v\in M_F-S_W}[\PPP(\aaa)(F_v)],$$ where~$S$ is a finite set and for $v\in S$ the set $W_v\subset [\PPP(\aaa)(\Fv)]$ is open satisfying that $\omega_v(\partial W_v)=0$ (we will call such open subset {\it elementary}). For $v\in S_W$, by \ref{bumpfunctions} there exist smooth functions $g_v,h_v:W_v\to \RR_{>0}$ such that $$0\leq h_v\leq \mathbf 1_{W_v}\leq g_v\leq1\text{ and }\int_W(g_v-h_v)\omega_v\leq\frac{\epsilon \omega_v([\PPP(\aaa)(F_v)])}{8 |S_W|}.$$
Let us set $\eta=\epsilon/4.$ For $v\in S$, we define $h_{v,\eta}=(1-\eta)h_v+\eta $ and $ g_{v,\eta}=(1-\eta)g_v+\eta$ and for $v\in M_F-S$, we define $ h_{v,\eta}=g_{v,\eta}=1$. We define $h_\eta=\prod_{\vMF} h_{v,\eta}$ and $g_\eta=\prod_{\vMF}g_{v,\eta}$. For $\xxx\in[\PPP(\aaa)(F)],$ let $\wx\in F^n-\{0\}$ be a lift of~$\xxx$. Observe that \begin{align*}H(((h^{-1}_{v,\eta}\circ\qav)\cdot f_v)_v)(\xxx)&=\prod_{\vMF}((h^{-1}_{v,\eta}\circ\qav)\cdot f_v)(\wx)\\
&=\prod_{\vMF}h^{-1}_{v,\eta}(\qav(\wx))f_v(\wx)\\
&=\bigg(\prod_{\vMF}h^{-1}_{v,\eta}(\xxx)\bigg)H(\xxx)\\
&=h_\eta(i(\xxx))^{-1}\cdot H(\xxx)
\end{align*} for $\xxx\in[\PPP(\aaa)(F)]$. 
By \ref{omegachanges}, one has that $\omega(((h^{-1}_{v,\eta}\circ\qav)\cdot f_v)_v)=h_{\eta}\omega$. Similarly $H(((g^{-1}_{v,\eta}\circ\qav)\cdot f_v)_v)=g_v([\PPP(\aaa)(i_v)](\xxx))^{-1}H(\xxx)$ for $\xxx\in[\PPP(\aaa)(F)]$ and $$\omega(((g_v^{-1}\circ\qav)\cdot f_v)_v)=g_\eta\omega.$$

Now, it follows from (1) that for the quasi-toric degree $|\aaa|$ families of smooth functions $((h^{-1}_{v,\eta}\circ\qav)\cdot f_v)_v$ and $((g^{-1}_{v,\eta}\circ\qav)\cdot f_v)_v$ we have that \begin{multline*}|\{\xxx\in i([\PPP(\aaa)(F)])|H(\xxx)\leq h_\eta(\xxx)\cdot B\}|\\=|\{\xxx\in i([\PPP(\aaa)(F)])| H((( h^{-1}_v\circ\qav)\cdot f_v)_v)(\xxx)\leq B\}|\\
\sim_{B\to\infty} \frac{\omega\big(\prod_{\vMF}[\PPP(\aaa)(F_v)]\big)\cdot\int_{\prod_{\vMF}[\PPP(\aaa)(\Fv)]}h_\eta\omega}{|\Sh^1(F, \mu_{\gcd(\aaa)})|\cdot |\aaa|}B.\end{multline*} and that \begin{multline*}
|\{\xxx\in i([\PPP(\aaa)(F)])|H(\xxx)\leq g_\eta(i(\xxx))B\}|\\
\sim_{B\to\infty} \frac{\omega\big(\prod_{\vMF}[\PPP(\aaa)(F_v)]\big)\cdot\int_{\prod_{\vMF}[\PPP(\aaa)(\Fv)]}g_\eta\omega}{|\Sh^1(F, \mu_{\gcd(\aaa)})|\cdot |\aaa|}B.
\end{multline*}
We deduce that $$\frac{|\{\xxx\in i([\PPP(\aaa)(F)])|H(\xxx)\leq h(i(\xxx))B\}|}{|\{\xxx\in i([\PPP(\aaa)(F)])|H(\xxx)\leq B\}|}\sim_{B\to\infty}\frac{\int_{\prod_{\vMF}[\PPP(\aaa)(F_v)]}h_\eta\omega}{\omega\big(\prod_{\vMF}[\PPP(\aaa)(F_v)]\big)}$$
and analogously for $g_\eta$. Thus, there exists $B_0>0$ such that if $B>B_0$, one has that \begin{align*}\frac{\int_{\prod_{\vMF}[\PPP(\aaa)(\Fv)]}h_\eta\omega}{\omega\big(\prod_{\vMF}[\PPP(\aaa)(F_v)]\big)}-\frac{\epsilon}{8}\hskip-3cm&\\&\leq \frac{|\{\xxx\in i([\PPP(\aaa)(F)])|H(\xxx)\leq B\prod_{v\in S_W}((1-\eta)\mathbf 1_{W_v}+\eta)\}|}{|\{\xxx\in i([\PPP(\aaa)(F)])|H(\xxx)\leq B\}|} \\&\leq\frac{\int_{\prod_{\vMF}[\PPP(\aaa)(\Fv)]}g_\eta\omega}{\omega\big(\prod_{\vMF}[\PPP(\aaa)(F_v)]\big)}+\frac{\epsilon}{8}. \end{align*}
We have that $$\frac{|\{\xxx\in i([\PPP(\aaa)(F)])|H(\xxx)\leq \eta B\}|}{|\{\xxx\in i([\PPP(\aaa)(F)])|H(\xxx)\leq B\}|}\sim_{B\to\infty}\eta,$$thus for $B\gg 0$, we deduce that \begin{align*}\frac{|\{\xxx\in i([\PPP(\aaa)(F)]) |H(\xxx)\leq \mathbf 1_WB\}|}{|\{\xxx\in i([\PPP(\aaa)(F)])|H(\xxx)\leq B\}|}&=\frac{|\{\xxx\in W |H(\xxx)\leq B\}|}{|\{\xxx\in i([\PPP(\aaa)(F)])|H(\xxx)\leq B\}|}\\&\leq 2\frac{\epsilon}{4}+\frac{\epsilon}{8}.\end{align*} Thus the claim follows for elementary open subsets $W\subset\prod_{\vMF}[\PPP(\aaa)(F_v)].$

\item We prove claim for every $W$ which is a finite union of elementary open subsets of $\prod_{\vMF}[\PPP(\aaa)(\Fv)]$. The collection of the elementary open subsets of $\prod_{\vMF}[\PPP(\aaa)(F_v)]$ is stable for finite intersections. 
Suppose the claim is valid for every union of~$k$ elementary open sets. Let $V_1\doots V_{k+1}$ be elementary open sets, we have that \begin{align*}\frac{|\{\xxx\in \bigcup_{j=1}^kV_j\cup V_{k+1})|H(\xxx)\leq B\}|}{|\{\xxx\in i([\PPP(\aaa)(F)])|H(\xxx)\leq B\}|}\hskip-4cm&\\&=\frac{|\{\xxx\in \bigcup_{j=1}^kV_j|H(\xxx)\leq B\}|+|\{\xxx\in V_{k+1}|H(\xxx)\leq B\}|}{|\{\xxx\in i([\PPP(\aaa)(F)])|H(\xxx)\leq B\}|}\\
&\hspace{0.4cm}-\frac{|\{\xxx\in \bigcup_{j=1}^k (V_j\cap V_{k+1})|H(\xxx)\leq B\}|}{|\{\xxx\in i([\PPP(\aaa)(F)])|H(\xxx)\leq B\}|}\\
&\sim_B\frac{\omega\big(\bigcup_{j=1}^kV_j\big)+\omega(V_{k+1})-\omega\big(\bigcup_{j=1}^k(V_j\cap V_{k+1})\big)}{\omega\big(\prod_{\vMF}[\PPP(\aaa)(F_v)]\big)}\\
&=\frac{\omega\big(\bigcup _{j=1}^{k+1}V_j\big)}{\omega\big(\prod_{\vMF}[\PPP(\aaa)(F_v)]\big)}.
\end{align*}
It follows from the induction, that the claim is valid for $W$ which is a union of finitely many elementary open subsets of $\prod_{\vMF}[\PPP(\aaa)(F_v)]$. 

\item Let us now prove the claim for a general open subset $W$ with $\omega(\partial W)=0$. We will establish firstly that for every $\epsilon>0$, there exist $W'$ and $W''$ which are finite unions of the elementary open subsets of $\prod_{\vMF}[\PPP(\aaa)(F_v)]$, such that $W'\subset W\subset W''$ and such that $\omega(W''-W')<\epsilon.$  For $\vMF$, open sets of $[\PPP(\aaa)(F_v)]$ with negligible boundary form a basis of the topologies of $[\PPP(\aaa)(F_v)]$ (the collection of such open sets contains the images of open balls in $\Fvnz,$ and the open balls in $\Fvnz$ form a basis of the topologies and have negligible boundaries). It follows that the elementary open subsets form a basis of the topology of $\prod_{\vMF}[\PPP(\aaa)(\Fv)]$. Let $\epsilon>0$. The space $\overline W$ is a compact, thus can be covered by the finitely many elementary open sets of volume no more than $\epsilon _2$. We let $W''$ be the union of these sets. By the inner regularity of~$\omega$ (e.g. \cite[Theorem 2.5.13]{GeometricMT}), there exists a compact set $K\subset W$ such that $\omega(W)-\omega(K)<\epsilon /2$. Cover the set~$K$ by finitely many elementary open subsets lying completely in $W$. We let $W'$ be the union of this sets. Clearly, $W'\subset W\subset W''.$ By using that $\omega(\partial W)=0$ we get that 
\begin{align*}
\omega(W''-W')&\leq\omega(W''-\overline W)+\omega(\overline W-W')\\
&\leq \omega(W''-(W\cup\partial W))+\omega((W\cup\partial W)-W')\\
&\leq \omega(W''-W)+\omega(W-W')\\
&< \epsilon.
\end{align*}
%
%
Now, for $\delta >0,$ one has that there exists $B_1>0$ such that if $B>B_1,$ then:
\begin{align*}
\frac{\omega(W')}{\omega\big(\prod_{\vMF}[\PPP(\aaa)(F_v)]\big)}-\delta&\leq\frac{|\{\xxx\in W'|H(\xxx)\leq B\}|}{|\{\xxx\in i([\PPP(\aaa)(F)])|H(\xxx)\leq B\}|}\\&\leq \frac{|\{\xxx\in W|H(\xxx)\leq B\}|}{|\{\xxx\in i([\PPP(\aaa)(F)])|H(\xxx)\leq B\}|}.
\end{align*} and that
\begin{align*}
\frac{|\{\xxx\in W|H(\xxx)\leq B\}|}{|\{\xxx\in i([\PPP(\aaa)(F)])|H(\xxx)\leq B\}|}&\leq \frac{|\{\xxx\in W''|H(\xxx)\leq B\}|}{|\{\xxx\in i([\PPP(\aaa)(F)])|H(\xxx)\leq B\}|}\\&\leq\frac{\omega(W'')}{\omega\big(\prod_{\vMF}[\PPP(\aaa)(F_v)]\big)}+\delta.
\end{align*}
It follows that $$\frac{|\{\xxx\in W|H(\xxx)\leq B\}|}{|\{\xxx\in i([\PPP(\aaa)(F)])|H(\xxx)\leq B\}|}\sim_{B\to\infty}\frac{\omega(W)}{\omega\big(\prod_{\vMF}[\PPP(\aaa)(F_v)]\big)}.$$
The statement is proven.
\end{enumerate}
\end{proof}
We deduce the following proposition (it is analogous to parts (a) and (b) of \cite[Proposition 3.3]{Peyre})
\begin{prop}\label{moreofosnovna}The following claims are valid:
\begin{enumerate}
\item Let $f:\prod_{\vMF}[\PPP(\aaa)(F_v)]\to\CC$ be a step function (the sum is assumed to be finite) $\sum \lambda_i\mathbf 1_{W_i},$ where $W_i$ are open sets with negligible boundaries. One has that \begin{align*}\lim_{B\to\infty}\frac{\sum_{\xxx\in [\PPP(\aaa)(F)]}f(\xxx)}{|\{\xxx\in [\PPP(\aaa)(F)]|H(\xxx)\leq B\}|}\hskip-1cm&\\&=\lim_{B\to\infty}\frac{\sum_{\xxx\in i([\PPP(\aaa)(F)])}f(\xxx)}{|\{\xxx\in i([\PPP(\aaa)(F)])|H(\xxx)\leq B\}|}\\
&=\frac{\int_{\prod_{\vMF}[\PPP(\aaa)(F_v)]}f\omega}{\omega\big(\prod_{\vMF}[\PPP(\aaa)(F_v)]\big)}.
\end{align*}
\item  For every continuous function $f:\prod_{\vMF}[\PPP(\aaa)(F_v)]\to\CC,$  the equality from (1) is valid 
\item The asymptotic formula from \ref{osnovna} is valid for any quasi-toric degree $|\aaa|$ family of $\aaa$-homogenous functions $(g_v:\Fvnz\to\RR_{>0})_v$.
\end{enumerate}
\end{prop}
\begin{proof}
For the first two claims, the proofs of the corresponding claims in \cite[Proposition 3.3]{Peyre} work here. For the third one, we make minor modifications.
\begin{enumerate}
\item By \ref{quasitoricequi}, the equality is valid for the characteristic functions of open sets with the negligible boundaries. Clearly, the equality stays valid for any step function (the sum is assumed to be finite) $\sum_{i}\lambda_i\mathbf 1_{W_i},$ where $W_i$ are open sets with negligible boundaries, assumed to be pairwise disjoint. We verify that the same equality stays valid for a step function $\sum_{i}\lambda_i\mathbf 1_{W_i},$ where $W_i$ are open sets with the negligible boundaries (not assumed pairwise disjoint). For every point $\xxx\in\prod_{\vMF}[\PPP(\aaa)(F_v)],$ let $A(\xxx)$ be the set of the indices~$i$ for which $\xxx\in W_i$. We let $W_{A(\xxx)}=\bigcap_{i\in A(\xxx)}W_i$. The function $\sum_{i}\lambda_i\mathbf 1_{W_i}$ coincides with the function $$\sum_{A(\xxx)}\bigg(\sum_{i\in A(x)}\lambda_i\bigg)\mathbf 1_{W_{A(\xxx)}},$$ where the sum is taken over all subsets that appear as $A(\xxx)$ for some $\xxx\in\prod_{\vMF}[\PPP(\aaa)(F_v)]$. A finite intersection of open sets with negligible boundary is an open set with the negligible boundary (because the boundary of an intersection is contained in the union of the boundaries), thus the sets $W_{A(\xxx)}$ are open sets with negligible boundary. Hence, the equality stays valid for described step functions.
\item The open sets with the negligible boundaries form a basis of the compact topological space $\prod_{\vMF}[\PPP(\aaa)(F_v)]$. Now, any continuous function $f:\prod_{\vMF}[\PPP(\aaa)(F_v)]\to\CC$ can be approached uniformly by step functions $\sum_{i}\lambda_i\mathbf 1_{W_i},$ where $W_i$ are open sets with the negligible boundaries. The claim follows.
\item Let $(g_v:\Fvnz\to\RR_{>0})_v$ be a quasi-toric degree $|\aaa|$ family of $\aaa$-homogenous functions and let $H^g$ be the resulting height. Let $S'$ be the finite set of places for which $f_v\neq g_v$. For $v\in S'$ and for $\xxx\in[\PPP(\aaa)(F_v)],$ let $\wx\in\Fvnz$ be a lift of~$\xxx$. The function $$h:\prod_{\vMF}[\PPP(\aaa)(F_v)]\to\RR_{>0}\hspace{1cm}(\xxx_v)_v\mapsto \prod_{v\in S'}\frac{f_v(\wx_v)}{g_v(\wx_v)},$$ does not depend on the choices of $\wx_v,$ is a continuous function (because $f_v$ and $g_v$ are of the same weighted degree and are continuous functions $\Fvnz\to\RR_{>0}$). Note that when $\xxx\in i([\PPP(\aaa)(F)])$, one has that $$h(\xxx)=\frac{H(\xxx)}{H^g(\xxx)},$$ because for every~$v$, one can take $\wx_v$ to be fixed element in $(F^{\times})^n$. 
By \ref{quasitoricequi}, for every open $W$ with the negligible boundary, one has that \begin{multline}\label{nezastoeq}\lim_{B\to\infty}\frac{|\{\xxx\in i(\prod_{\vMF}[\PPP(\aaa)(F_v)])| H(\xxx)\leq \mathbf 1_WB\}|}{|\{\xxx\in i([\PPP(\aaa)(F)])| H(\xxx)\leq B\}|}\\=\frac{\int_{\prod_{\vMF}[\PPP(\aaa)(F_v)]}\mathbf 1_W\omega}{\omega\big(\prod_{\vMF}[\PPP(\aaa)(F_v)]\big)}.\end{multline}
The same equality is valid when $\mathbf 1_W$ is replaced by a step function $\sum_{i}\lambda_i\mathbf 1_{W_i},$ where $W_i$ are open with the negligible boundaries. Let $\epsilon>0$, there exists a step function $\sum_{i}\lambda_i\mathbf 1_{W_i},$ with $W_i$ open with the negligible boundaries, such that $0\leq h-\sum_{i}\lambda_i\mathbf 1_{W_i}\leq \epsilon.$
It follows that \begin{align*}\frac{|\{\xxx\in i([\PPP(\aaa)(F)])|H^g(\xxx)\leq B\}|}{|\{\xxx\in i([\PPP(\aaa)(F)])|H(\xxx)\leq B\}|}\hskip-2cm&\\&=\frac{|\{\xxx\in[\PPP(\aaa)(F)]|H(\xxx)\leq h(\xxx)B\}|}{|\{\xxx\in i([\PPP(\aaa)(F)])|H(\xxx)\leq B\}|}\\
&\leq\frac{|\{\xxx\in i([\PPP(\aaa)(F)])|H(\xxx)\leq (\epsilon+\sum_i\lambda_i\mathbf 1_{W_i}) B\}|}{|\{\xxx\in i([\PPP(\aaa)(F_v)])|H(\xxx)\leq B\}|}\\
&=_{B\to\infty}\frac{\int_{\prod_{\vMF}[\PPP(\aaa)(F_v)]}(\epsilon + \sum_i\lambda_i\mathbf 1_{W_i})\omega}{\omega\big(\prod_{\vMF}[\PPP(\aaa)(F_v)]\big)}\\&=\epsilon+\frac{\int_{\prod_{\vMF}[\PPP(\aaa)(F_v)]} (\sum_i\lambda_i\mathbf 1_{W_i})\omega}{\omega\big(\prod_{\vMF}[\PPP(\aaa)(F_v)]\big)}.
\end{align*}
Similarly, \begin{align*}
\frac{|\{\xxx\in i([\PPP(\aaa)(F)])|H^g(\xxx)\leq B\}|}{|\{\xxx\in i([\PPP(\aaa)(F)])|H(\xxx)\leq B\}|}\hskip-1cm&\\&=\frac{|\{\xxx\in[\PPP(\aaa)(F)]|H(\xxx)\leq h(\xxx)B\}|}{|\{\xxx\in i([\PPP(\aaa)(F)])|H(\xxx)\leq B\}|}\\
&\geq\frac{|\{\xxx\in[\PPP(\aaa)(F)]|H(\xxx)\leq (\sum_i\lambda_i\mathbf 1_{W_i})B\}|}{|\{\xxx\in i([\PPP(\aaa)(F)])|H(\xxx)\leq B\}|}\\
&=_{B\to\infty}\frac{\int_{\prod_{\vMF}[\PPP(\aaa)(F_v)]} (\sum_i\lambda_i\mathbf 1_{W_i})\omega}{\omega\big(\prod_{\vMF}[\PPP(\aaa)(F_v)]\big)}.
\end{align*}
By decreasing $\epsilon$, we deduce that $$\lim_{B\to\infty}\frac{|\{\xxx\in i([\PPP(\aaa)(F)])|H^g(\xxx)\leq B\}|}{|\{\xxx\in i([\PPP(\aaa)(F)])|H(\xxx)\leq B\}|}=\frac{\int_{\prod_{\vMF}[\PPP(\aaa)(F_v)]}h\omega}{\omega\big(\prod_{\vMF}[\PPP(\aaa)(F_v)]\big)}.$$
It follows from \ref{omegachanges} that $$\int_{\prod_{\vMF}[\PPP(\aaa)(F_v)]}h\omega=\omega((g_v)_v)\big(\prod_{\vMF}[\PPP(\aaa)(F_v)]\big).$$ Finally, \ref{osnovna} gives that $$\lim_{B\to\infty}\frac{|\{\xxx\in i([\PPP(\aaa)(F)])|H(\xxx)\leq B\}|}B=\frac{\omega\big(\prod_{\vMF}[\PPP(\aaa)(F_v)]\big)}{|\aaa|},$$ thus $$\lim_{B\to\infty}\frac{|\{\xxx\in i([\PPP(\aaa)(F)])|H^g(\xxx)\leq B\}|}B=\frac{\omega((g_v)_v)\big(\prod_{\vMF}[\PPP(\aaa)(F_v)]\big)}{|\aaa|}$$ as claimed.
\end{enumerate}
\end{proof}
\begin{rem}
\normalfont
Suppose that $(f^{\#}_v)_v$ is the toric degree $|\aaa|$ family of $\aaa$-homogenous functions. Using the expression for $\tau$ from \ref{locom} and the formula for the volumes $\omega_v([\PPP(\aaa)(F_v)])$ for $\vMFi$ from \ref{locinfomega}, we get that
\begin{align*}\tau&=\frac{\Res(\zeta_F,1)|\mu_{\gcd(\aaa)}(F)|(2^{n-1}|\aaa|)^{r_1}((2\pi)^{n-1}|\aaa|)^{r_2}}{\Delta(F)^{\frac{n-1}2}\zeta_F(|\aaa|)}\\
&=\frac{\Reg(F)h_F 2^{r_1}(2\pi)^{r_2}}{\Delta(F)^{\frac12}}\frac{|\mu_{\gcd(\aaa)(F)}|2^{(n-1)r_1}(2\pi)^{(n-1)r_2}|\aaa|^{r_1+r_2}}{\Delta(F)^{\frac{n-1}2}w_F\zeta_F(|\aaa|)}\\
&=\frac{h_F}{\zeta_F(|\aaa|)}\bigg(\frac{2^{r_1}(2\pi)^{r_2}}{\sqrt{\Delta(F)}}\bigg)^{n}|\aaa|^{r_1+r_2}\frac{\Reg(F)|\mu_{\gcd(\aaa)}(F)|}{w_F}.\end{align*} Here, $\Reg(F)$ is the regulator of~$F$ and $w_F$ the number of roots of unity in~$F$.
Hence, \begin{multline*}|\{\xxx\in[\PPP(\aaa)(F)]|H^{\#}(\xxx)\leq B\}|\\\sim_{B\to\infty}\frac{h_F}{\zeta_F(|\aaa|)}\bigg(\frac{2^{r_1}(2\pi)^{r_2}}{\sqrt{\Delta(F)}}\bigg)^{n}|\aaa|^{r_1+r_2-1}\frac{\Reg(F)|\mu_{\gcd(\aaa)}(F)|}{w_F}B.\end{multline*}
The counting result in this case has been obtained by Bruin and Najman (\cite[Theorem 3.7]{NajmanBruin}). They used the method of Deng from \cite{Deng}, which is similar to the original method of Schanuel from \cite{Schanuel} for the case of rational points of the projective space.
\end{rem}
\begin{rem}
\normalfont
Theorem \ref{bigthmm} and Theorem  \ref{osnovna} give that the closed substack $\mathscr Z(\{X_1\cdots X_n\})\subset\PPP(\aaa)$ given by the~$\Gm$-invariant closed subscheme $Z(X_1\cdots X_n)\subset\AAA^n-\{0\}$, is not an ``accumulating" substack (see \cite[Definition 1.3]{Peyre} for the terminology). 
\end{rem}
The same estimate as in \ref{osnovna} is true for the rational points of the stack~$\oPPa$ (because the stack $ \oPPa-\PPP(\aaa)$ has  only one rational point).
\begin{rem}\label{firstpartofpeyre}
\normalfont
For $L\in\Pic(\oPPa)\otimes_{\ZZ}\RR=\Pic(\oPPa)_{\RR}$ and $\lambda\in\RR$, we define a measure $\theta_L$ on the set $$\mathcal H_{L}(\lambda)=\{y\in\Pic(\oPPa)^*_{\RR}|y(L)=\lambda\}= \begin{cases}
\RR&\text{if $L=0$ and $\lambda=0$,}\\
\big\{\frac{\lambda}{L}\big\}&\text{if $L\neq 0$}\\
\emptyset&\text{if $L=0$ and $\lambda\neq 0$.}
               \end{cases},$$ by setting it to be the Lebesgue measure, the Haar measure which is normalized by $\theta_L\big(\big\{\frac{\lambda}{L}\big\}\big)=\frac1L$ and by $\theta_0(\emptyset)=1,$ respectively.
We define $$\alpha=\alpha(\oPPa)=\theta_{|\aaa|}(\{y>0\}\cap \mathcal H_{|\aaa|}(1))=\theta_{|\aaa|}\big(\frac{1}{|\aaa|}\big)=\frac{1}{|\aaa|}.$$ (This definition is \cite[Definition 2.4]{Peyre} for the case $\rk(\Pic)=1$). Thus the leading constant in the asymptotics of \ref{osnovna} and in Part (3) of \ref{moreofosnovna}, writes as $\alpha\tau$, as predicted by Peyre in \cite{Peyre} for Fano varieties.
\end{rem}
\chapter{Number of $\mu_m$-torsors of bounded discriminant}
\label{number of torsors}
In this chapter~$F$ will be a number field and $\OOF$ its ring of integers. Let $m\geq 2$ be an integer. The goal of this chapter is to give the asymptotic behaviour for the number of $\mu_m$-torsors over~$F$ of bounded discriminant. We are going to use the language of heights on weighted projective stacks from \ref{Quasi-toric heights}.

\section{Calculations of the discriminant}
In this section we will be calculating the discriminants of $\mu_m$-torsors over $\Fv$ where~$v$ is a finite place of~$F$.
\subsection{}
We use \cite{Neukirch} as principal reference for the definition and basic properties of discriminants. Let~$R$ be a Dedekind domain and let~$K$ be its field of fractions. Let~$L$ be a finite extension of~$K$. Let $R'$ be the integral closure of~$R$ in~$L$. If $x_1\doots x_n$ is a basis of~$L$ over~$K$, we set $$\Delta(x_1\doots x_n):=\det((\Tr (x_ix_j))_{ij}),$$where $\Tr:L\to K$ is the trace map. We say that $\Delta$ is the discriminant of the basis $x_1\doots x_n$. We define $\Delta (R',R)$ to be the ideal of~$R$ generated by all $\Delta(x_1\doots x_n)$ when $x_1\doots x_n$ range over all bases of $L/K$ which are contained $R'$. By the abuse of notation, we may write $\Delta (L/K)$ for $\Delta (R'/R)$ if~$R$ and $R'$ are understood from the context. 
\begin{prop}[{\cite[Corollary 2.10, Chapter III]{Neukirch}}] \label{tower}
For a tower of fields $K\subset L\subset M$ one has that:
$$\Delta(M/K)=\Delta(L/K)^{[M:L]}N_{L/K}(\Delta(M/L)).$$
\end{prop}
\begin{prop}[{\cite[Corollary 2.12, Chapter III]{Neukirch}}]
Suppose that $L/K$ is unramified. One has that $\Delta (L/K)=(1)$.
\end{prop}
Let $\mathfrak D(L/K)$ denotes the different ideal (\cite[Definition 2.1, Chapter III]{Neukirch}). The following statement is true for any {\it{tamely ramified}} primes (\cite[Definition 7.6, Chapter II]{Neukirch}), but for our needs the following version suffices (clearly, it makes primes automatically tamely ramified):
\begin{prop}[{\cite[Theorem 2.6]{Neukirch}}] \label{tamrame}
Suppose that $\mathfrak p$ is a non-zero prime ideal of~$R$ such that $R/\mathfrak p$ is a finite field of the characteristic coprime to $[L:K]$ and let $\mathfrak q$ be a prime of $R'$ lying over $\mathfrak p$. One has that $v_{\mathfrak q}(\mathfrak D(L/K))=e_{\mathfrak q|\mathfrak p}-1$ where $v_{\mathfrak q}(\Delta (L/K))$ is the exponent of $\mathfrak q$ in the prime factorisation of the ideal $\mathfrak D(L/K)$ and $e_{\mathfrak q|\mathfrak p}$ is the degree of the ramification of the prime ideal $\mathfrak p$ in $\mathfrak q$.
\end{prop}
If $A=K_1\times\cdots\times K_r$ be a product of finite separable extensions $K_i/K$, we define $\Delta(A/K):=\prod_{i}\Delta(K_i/K)$. The following proposition is given in \cite[Corollary 2.11, Chapter III]{Neukirch} when~$A$ is a field, nonetheless, it is true when~$A$ is a finite product of finite extensions of~$F$.
\begin{prop}\label{discloc}
Let $A/F$ be a finite product of finite extensions of~$F$. Let $v\in M_F^0$. One has that $$v(\Delta(A/F))=v\bigg(\Delta\big((A\otimes_FF_v)/\Fv\big)\bigg).$$
\end{prop}
\begin{proof}
Let $A=K_1\times\cdots \times K_r$, with $K_i/F$ finite extensions of~$F$. For every~$i$, by \cite[Corollary 2.11, Chapter III]{Neukirch}, one has that $$v(\Delta(K_i/F))=v\bigg(\prod_{w^i|v}\Delta(K_{w^i}/F_v)\bigg),$$ where $w^i$ are places of $K_i$ lying above~$v$. For every~$i$, by \cite[Proposition 8.3, Chapter II]{Neukirch}, one has that $$\prod_{w^i|v}K_{w^i}=K_i\otimes _FF_v$$ and hence $$\prod_{w^i|v}\Delta(K_{w^i}/\Fv)=\Delta(K_i\otimes_FF_v).$$ 
We deduce that $$v(\Delta(K_i/F))=v\bigg(\Delta(K_i\otimes _FF_v/\Fv)\bigg).$$ Hence, \begin{align*}v(\Delta(A/F))=v\bigg(\prod_{i=1}^{r}\Delta(K_i/F)\bigg)&=v\bigg(\prod_{i=1}^r\Delta\big((K_i\otimes _FF_v)/F_v\big)\bigg)\\&=v\bigg(\Delta\big((A\otimes _FF_v)/\Fv\big)\bigg).\end{align*}
The statement is proven.
\end{proof}
\subsection{} Let $\vMFz$ such that $v(m)=0$.  In this paragraph we calculate the discriminant of a $\mu_m$-torsor over $\Fv$. We have not find a reference for these calculations.
\begin{lem}\label{disctotram}
Let~$n$ be an integer such that $v(n)=0$. Let $\sqrt[n]{\piv}$ be a formal~$n$-th root of the uniformizer $\piv$. One has that $$\Delta\big(\Fv(\sqrt[n]{\piv})/\Fv\big)=\piv^{n-1}\Ov.$$
\end{lem}
\begin{proof}
Let us write~$K$ for $\Fv(\sqrt[n]{\piv})$. By Eisenstein's criterion, the polynomial $X^n-\piv$ is irreducible. It follows that the degree of the extension $K/\Fv$ is equal to~$n$ and that the elements $\sqrt[n]{\piv}, \sqrt[n]{\piv^2}\doots \sqrt[n]{\piv^{n-1}}$ do not belong to $\Ov$. We deduce that $(\sqrt[n]{\piv})\cap\Ov=(\piv)$. Thus the degree of the ramification of $\piv$ in~$K$ is at least~$n$ and hence is equal to~$n$ (because the degree of the extension is~$n$). Proposition \ref{tamrame} gives that $(\mathfrak D(K/\Fv))=(\sqrt[n]{\piv})^{n-1}$. We deduce that 
\begin{align*}\Delta(K/F_v)&=N_{K/F_v}(\mathfrak D({\Fv(\sqrt[n]{\piv})/\Fv}))\\
&=N_{K/F_v}((\sqrt[n]{\piv})^{r-1})\\
&=N_{K/F_v}((\sqrt[n]{\piv}))^{r-1}\\
\end{align*}
As $N_{K/F_v}(\piv)=(\piv^n)$, it follows that $N_{K/F_v}((\sqrt[n]{\piv}))=(\piv),$ and hence $$\Delta(\Fv(\sqrt[n]{\piv})/\Fv)=\Delta(K/F_v)=N_{K/F_v}((\sqrt[n]{\piv}))^{r-1}=(\piv)^{r-1}.$$The statement is proven.
\end{proof}
\begin{lem} \label{citavbozdan} Let $a\in\Fvt$ and let~$n$ be an integer such that $v(n)\neq 0$. We set $d=\gcd(v(a),n)$. Let $\sqrt[n]{a}$ be an~$n$-th root of $a$ (lying in an algebraic closure of $\Fv$). One has that $$\Delta(\Fv(\sqrt[n]a)/\Fv)=\piv^{[\Fv(\sqrt[n]a):\Fv](1-d/n)}\Ov.$$
\end{lem}
\begin{proof}
Let us set $r=v(a)$ and $d=\gcd(v(a),n)$. We write $a=\piv^ru$ for some $u\in\Ovt.$ There exist integers $b,c$ such that $br+cn=d$ so that $(\piv^{r}u)^b(\piv^c)^n=\piv^du$ for some $u\in\Ovt$. As $\piv^{cn}$ is an~$n$-th power of an element in $\Fvt$, one has that $F_v({a}^{1/n})=F_v({\piv^{1/(n/d)}}{u}^{1/n})$, where $\piv^{1/(n/d)}$ and $u^{1/n}$ are formal $n/d$-th and $1/n$-th roots of $\piv$ and $u$, respectively.  Thus we have the following towers of extensions:
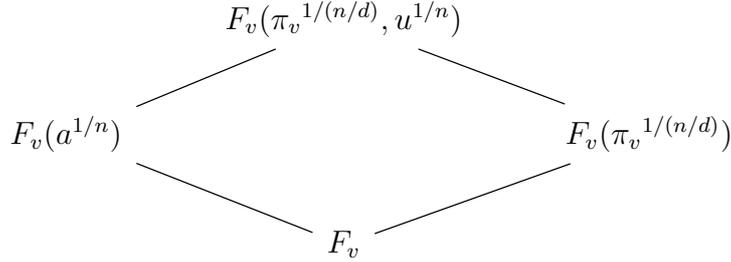
\[\begin{tikzcd}
	& {F_v({\pi_v}^{1/(n/d)},{u}^{1/n})} \\
	{F_v({a}^{1/n})} && {F_v({\pi_v}^{1/(n/d)})} \\
	& {F_v}
	\arrow[no head, from=1-2, to=2-1]
	\arrow[no head, from=2-1, to=3-2]
	\arrow[no head, from=1-2, to=2-3]
	\arrow[no head, from=2-3, to=3-2]
\end{tikzcd}\]
Let us set $M={F_v({\pi_v}^{1/(n/d)},{u}^{1/n})}$, $K=F_v(a^{1/n})$ and $L=F_v(\piv^{1/(n/d)})$. The extensions $M/K$ and $M/L$ are unramified. Now, Proposition \ref{tower} gives that $$\Delta(K/\Fv)^{[M:K]}=\Delta(L/\Fv)^{[M:L]}.$$
One has that $$[K:\Fv]\cdot [M:K]=[L:\Fv]\cdot [M:L]=(n/d)[M:L],$$ thus $$\Delta(K/\Fv)=\Delta(L/\Fv)^{[M:L]/[M:K]}\Delta(L/\Fv)^{(d/n) [K:\Fv]}.$$
Recall that by Lemma \ref{disctotram}, one has that $$\Delta(L/\Fv)=\Delta(\Fv(\piv^{1/(n/d)})/\Fv)=\piv^{(n/d)-1}\Ov,$$ thus 
\begin{multline*}\Delta(K/\Fv)=\piv^{[K:\Fv]\cdot (d/n)((n/d)-1)}\Ov=\piv^{[K:\Fv](1-d/n)}\Ov\\=\piv^{[F(\sqrt[n]{a}):\Fv](1-d/n)}\Ov.\end{multline*}
\end{proof}
\begin{prop}\label{detaxm}
Let $a\in \Fvt$. Let us set $d=\gcd(v(a),n)$. One has that $$\Delta\big((\Fv[X]/(X^m-a))/\Fv\big)=\piv^{d-n}\Ov.$$
\end{prop}
\begin{proof}
Let $X^m-a=\prod_{j=1}^{\ell} b_j(X),$ be the composition of $X^m-a$ into a product of irreducible unitary polynomials (repetitions are allowed). For every $j=1\doots \ell$, let $\eta_{j}$ be a root of $b_j(X)$ so that $\Fv(b_j)$ are fields and the homomorphism $$\Fv[X]/(X^m-a)\to\big(\prod _{j=1}^{\ell}\Fv(\eta_j)\big)$$ induced from the homomorphism $$F[X]\to\prod_{j=1}^{\ell}\Fv(b_j)\hspace{1cm}X\mapsto (b_j)_j,$$ is an isomorphism. For every~$j$, one has that $\eta_j^m=a$ and by \ref{citavbozdan}, we have that $$\Delta(\Fv(\eta_j)/\Fv)=\piv^{\deg(b_j)(1-(d/n))}\Ov.$$ We deduce that \begin{multline*}\Delta((F_v[X]/(X^m-a))/F_v)=\prod_{i=1}^{\ell}(\piv^{-\deg(b_j)(1-d/n)}\Ov)=\piv^{n(d/n-1)}\Ov\\=\piv^{d-n}\Ov.\end{multline*}
\end{proof}
\subsection{} In this paragraph, we define the heights that will be used for the counting. The following notation will be used in the rest of the Chapter: $r$ will be the smallest prime of~$m$ and $\alpha(m):=m^2-m^2/r$. We will use the terminology from \ref{heightsonppa}. 

\begin{lem}\label{defvdisc} For $\vMFz,$ let us define $f^{\Delta}_v:\Fvt\to\RR_{>0}$ by $$f^{\Delta}_v(y)=|y|^{1/m}_v\bigg(N\big(\Delta\big((\Fv[X]/(X^m-y))/\Fv\big)\big)^{1/\alpha(m)}\bigg),$$where~$N$ stands for the ideal norm. For $\vMFi$, we set $f^{\Delta}_v(y)=|y|^{1/m}_v$.
\begin{enumerate}
\item For every $v\in M_F$, the function $f^{\Delta}_v$ is~$m$-homogenous and of weighted degree~$1$.
\item For every $\vMFz$, the function $f^{\Delta}_v$ is locally constant.
\item Let $\vMFz$ such that $v(m)=0$. For every $y\in\Fvt$, one has that $$f^{\Delta}_v(y)= |y|^{1/m}_v\pivv^{(\gcd(v(y),m)-m)/\alpha(m)}=\pivv^{v(y)/m+(\gcd(v(y),m)-m)/\alpha(m)}.$$For every $y\in\Ovt$, one has that $f^{\Delta}_v(y)=1$.
\end{enumerate}
\end{lem}
\begin{proof}
\begin{enumerate}
\item The function $y\mapsto |y|^{1/m}_v$ is~$m$-homogenous of weighted degree~$1$. It follows that for $\vMFi$ one has that $f_v$ is~$m$-homogenous of weighted degree~$1$. Let $\vMFz$ and let $t\in\Fvt$. The image of the ideal $(X^m-y)$ under the isomorphism $$\Fv[X]\to\Fv[X]\hspace{1cm}X\mapsto t^{-1}X $$ is the ideal $(t^{-m}X^m-y)=(X^m-t^my)$. It follows that 
 $\Fv[X]/(X^m-y)$ and $\Fv[X]/(X^m-t^my)$ are isomorphic, hence the norms of the corresponding discriminants are the same. It follows that $$y\mapsto \big(N\big(\Delta\big((\Fv[X]/(X^n-y))/\Fv\big)\big)\big)^{1/\alpha(m)}$$ is $\Fvt$-invariant. We deduce that $f_v$ is~$m$-homogenous of weighted degree~$1$. The claim is proven.
\item The function $y\mapsto \big(N\big(\Delta\big((\Fv[X]/(X^n-y))/\Fv\big)\big)\big)^{1/\alpha(m)}$ is $(\Fvt)_m$-invariant by (1). The subgroup $(\Fvt)_m\subset\Fvt$ is of the finite index in $\Fvt$ by \ref{ttafvfinite}, thus open in $\Fvt$ by \cite[Exercice 4, Chapter II]{Neukirch}. The function $y\mapsto |y|_v$ is invariant for the open subgroup $\Ovt\subset\Fvt$. We deduce that $f_v$ is invariant for the open subgroup $(\Fvt)_m\cap\Ovt$ of $\Fvt$. It follows that $f_v$ is locally constant.
\item As $v(m)=0$, Proposition \ref{detaxm} gives that $$\Delta\big((\Fv[X]/(X^n-y))/\Fv\big)=\piv^{m-\gcd(v(y),m)}\Ov.$$ We deduce that for every $y\in\Fvt$ one has that\begin{align*}f_v^{\Delta}(y)&=|y|_v^{1/m}\bigg(N\big(\Delta\big((\Fv[X]/(X^n-y))/\Fv\big)\big)\bigg)^{1/\alpha(m)}\\&=|y|_v^{1/m}\pivv^{(\gcd(v(y),m)-m)/\alpha(m)}\\&=\pivv^{v(y)/m+(\gcd(v(y),m)-m)/\alpha(m)}.\end{align*}If $y\in\Ovt$, one has that $$f^{\Delta}_v(y)=|y|^{1/m}_v\pivv^{(\gcd(v(y),m)-m)/\alpha(m)}=1.$$The claim is proven.
\end{enumerate} 
\end{proof} 
\begin{mydef}
Let $\vMF$ and let $k\in\ZZ.$ Let $f_v^{\Delta}$ be as in \ref{defvdisc}. 
\begin{itemize}
\item The function $x\mapsto (f_v^{\Delta}(x))^k$ will be called the discriminant~$m$-homogenous function of weighted degree~$k$. 
\item A degree~$k$ family $(f_v:\Fvt\to\RR_{\geq 0})_v$ of~$m$-homogenous continuous functions, will be said to be quasi-discriminant if for almost all~$v$, one has that $f_v=(x\mapsto (f_v^{\Delta}(x))^k)$. It follows from Lemma \ref{defvdisc} that quasi-discriminant families are generalized adelic (see \ref{fahom} for the definition) and the resulting height $H=H((f_v)_v)$ on $\PPP(m)(F)$ will be said to be quasi-discriminant height. 
\item If for every $\vMF$ one has that $f_v=(x\mapsto (f_v^{\Delta}(x))^k)$, then the family $(f_v)_v$ will be said to be the discriminant degree~$k$ family. The resulting height $H^{\Delta}=H((f_v)_v)$ will be said to be the discriminant height.
\end{itemize}
\end{mydef}
As usual, we will write~$H$ for the resulting heights on the set of the isomorphism classes $[\PPP(m)(F)]$.
\begin{rem}
\normalfont
Note that by \ref{defvdisc} and by \ref{davdavdav}, a ``quasi-discriminant" and a ``quasi-toric" family are different notions.
\end{rem}
\begin{rem}
\normalfont
The calculations from \ref{detaxm} and \ref{defvdisc} may be well known, however we have not found an adequate reference.
\end{rem}
 \begin{lem}\label{discvsheight}
Let $(f^{\Delta}_v)_v$ be the discriminant degree~$1$ family of~$m$-homogenous continuous functions and let $H^{\Delta}$ be the resulting height. Let $y\in\Ft$. One has that $$H^{\Delta}(q^m(y))=N\bigg(\Delta\big(F[X]/(X^m-y))/F\big)\bigg)^{1/\alpha(m)},$$ where $q^m:(\AAA^1-\{0\})\to \PPP(m)$ is the quotient~$1$-morphism.
\end{lem}
\begin{proof}
By the product formula, one has that \begin{align*}H^{\Delta}(y)&=\prod_{\vMF}f^{\Delta}_v(y)\\&=\bigg(\prod_{\vMFi}|y|^{1/m}_v\bigg)\prod_{\vMFz}|y|^{1/m}_vN\bigg(\Delta \big((F_v[X]/(X^m-y))/F_v\big)\bigg)^{\frac1{\alpha(m)}}\\
&=\prod_{\vMFz}N\bigg(\Delta \big((F_v[X]/(X^m-y))/F_v\big)\bigg)^{\frac1{\alpha(m)}}.
\end{align*}
Proposition \ref{discloc} gives that \begin{multline*}\prod_{\vMFz}N\bigg(\Delta \big((F_v[X]/(X^m-y))/F_v\big)\bigg)^{\frac1{\alpha(m)}}\\=N\bigg(\Delta\big((F[X]/(X^m-y))/F\big)\bigg)^{\frac1{\alpha(m)}}.\end{multline*}
We deduce that $$H^{\Delta}(y)=N\bigg(\Delta\big((F[X]/(X^m-y))/F\big)\bigg)^{1/\alpha(m)}.$$
\end{proof}
\subsection{} In this paragraph, we prove the Northcott property. Our proof does not involve Hermite-Minkowski theorem (\cite[Theorem 2.13]{Neukirch}) and relies on a comparison with toric heights that we establish in \ref{gvfvtor}.

Let $(f_v:\Fvt\to \RR_{>0})_v$ be a degree~$1$ quasi-discriminant family of~$m$-homogenous functions and let $H=H((f_v)_v)$. For $\vMF$, let~$H_v$ be the function $H_v:[\TT(m)(\Fv)]\to\RR_{>0}$ induced from $\Fvt$-invariant function $\Fvt\to\RR_{>0}, y\mapsto |y|^{-1/m}_vf_v(y)$ (we have studied such functions in \ref{localheightdef} for a general generalized adelic family). If $f_v=f_v^{\Delta}$, we may write $H_v^{\Delta}$ for~$H_v$. By \ref{localheightglobal}, one has that if $x\in[\TT(m)(F)]$, then \begin{equation}\label{htmfhv} H(x)=\prod_{\vMF}H_v([\TT(m)(i_v)](x)),\end{equation}where for $\vMF$, the map $[\TT(m)(i_v)]:[\TT(m)(F)]\to[\TT(m)(\Fv)]$ is the induced map from the $\Fvt$-invariant inclusion $i_v:(\Ft)^n\to(\Fvt)^n$. 

\begin{lem}\label{hdtoi}The following claims are valid:
\begin{enumerate}
\item Suppose that $v(m)\neq 0$. For $y\in\Fvt$, one has that $$H^{\Delta}_v(q^{m}_v(y))=\pivv^{(\gcd(v(y),m)-m)/\alpha(m)}.$$ 
The function $H^{\Delta}_v:[\TT(m)(\Fv)]\to\RR_{>0}$ is $[\TT(m)(\Ov)]$-invariant. If $x\in[\TT(m)(\Ov)],$ then $H^{\Delta}_v(x)=1$.
\item Suppose that $\vMFi$. One has that $H_v^{\Delta}=1.$
\end{enumerate}
\end{lem} 
\begin{proof} 
\begin{enumerate}
\item As $v(m)\neq 0$, by Lemma \ref{defvdisc}, one has that $$H^{\Delta}_v(q^m_v(y))=|y|_v^{-1/m}f_v^{\Delta}(y)=\pivv^{(\gcd(v(y),m)-m)/\alpha(m)}$$ for $y\in\Fvt$. Let us prove that $H^{\Delta}_v$ is $[\TT(m)(\Ov)]$-invariant. Let $x\in[\TT(m)(\Fv)]$ and let $u\in[\TT(m)(\Ov)].$ Let $\widetilde x$ and $\widetilde u$ be its lifts in $\Fvt$ and $\Ovt$, respectively. We have that $$H_v^{\Delta}(x)=\pivv^{(\gcd(v(\widetilde x),m)-m)/\alpha(m)}=\pivv^{(\gcd(v(\widetilde u\widetilde x),m)-m)/\alpha(m)}=H_v^{\Delta}(xu).$$ It follows that $H^{\Delta}_v$ is $[\TT(m)(\Ov)]$-invariant. Suppose that $x\in[\TT(m)(\Ov)]=q^m(\Ovt).$ Then $\widetilde x$ can be taken in $\Ovt.$ Thus $$H^{\Delta}_v(x)=\pivv^{(\gcd(v(\widetilde x),m)-m)/\alpha(m)}=1.$$
\item The function $f_v^{\Delta}$ is the function $x\mapsto |x|^{1/m}_v$. The function $H_v^{\Delta}$ is the induced function from the constant function $x\mapsto |x|^{-1/m}_vf_v^{\Delta}(x)=1$, hence $H_v^{\Delta}=1$. 
\end{enumerate}
\end{proof}
\begin{lem} \label{gvfvtor}
There exists $C>0$ such that for every $x\in[\TT(m)(F)]$ one has that $$CH^{\Delta}(x)\geq H^{\#}(x)^{\frac{1}{\alpha(m)}},$$where $H^{\#}$ is the height defined by the degree~$1$ toric family $(f_v^{\#})_v$ of~$m$-homogenous functions.
\end{lem}
\begin{proof}
For $\vMF$, let $H_v^{\Delta}$ be the function induced from $\Fvt$-invariant function $y\mapsto |y|_v^{-1/m}f_v^{\Delta}(y)$ and let $H^{\#}_v$ be the function induced from $\Fvt$-invariant function $y\mapsto |y|^{-1/m}_vf_v^{\#}(y)$. Recall that for $\vMFi$,  by the definitions of $f_v^{\Delta}$ and $f_v^{\#}$ (see \ref{defvdisc} and \ref{tordef}), one has $f_v^{\Delta}=f_v^{\#}$, thus by \ref{hdtoi}, one has that $H_v^{\Delta}=H_v^{\#}=1.$ Using this and using \ref{localheightglobal}, for $x\in[\TT(m)(F)],$ we get that \begin{align*}H^{\Delta}(x)&=\prod_{\vMFz}H_v^{\Delta}([\TT(m)(i_v)](x))\\
H^{\#}(x)&=\prod_{\vMFz}H_v^{\#}([\TT(m)(i_v)](x)),\end{align*} where the maps $[\TT(m)(i_v)]:[\TT(m)(F)]\to[\TT(m)(\Fv)]$ are the induced maps from $(F^\times)_m$-invariant inclusions $i_v:(\Ft)^n\to(\Fvt)^n$. For every finite~$v$ such that $v(m)=0$, by the finiteness of the space $[\TT(m)(\Fv)]$ one has that there exists $C_v>0$ such that for every $x\in[\TT(m)(\Fv)]$ one has that $$H_v^{\Delta}(x)\geq C_v H^{\#}_v(x)^{1/\alpha(m)}.$$
Let $v\in M_F^0$ be such that $v(m)=0$. By \ref{defvdisc}, one has for $y\in\Fvt$ that $$H_v^{\Delta}(q^m_v(y))=\pivv^{-(m-\gcd(m,v(y)))/\alpha(m)}.$$ 
On the other side, by \ref{davdavdav}, one has for $y\in\Fvt$ that $$H_v^{\#}(q^m_v(y))=\pivv^{-(\frac{v(y)}{m}-\lfloor\frac{v(y)}{m}\rfloor)}.$$ For every $k\in\ZZ$ one has that $$m-\gcd(k,m)\geq \frac{k}m-\bigg\lfloor\frac km\bigg\rfloor$$(if~$k$ is divisible by~$m$, then the quantities on both hand sides are equal to zero, and if~$k$ is not divisible by~$m$, then the quantity on the left hand side is at least $1,$ hence is bigger than the quantity on the right hand side). We deduce that  $$H_v^{\Delta}\geq (H_v^{\#} )^{1/\alpha(m)}.$$

It follows that $$H^{\Delta}\geq \bigg( \prod_{v(m)\neq 0}C_v\bigg) (H^{\#}\big)^{1/\alpha(m)}.$$The statement is proven.
\end{proof}
\begin{prop}\label{northdisc} Let $(f_v:\Fvt\to\RR_{>0})_v$ be a degree~$1$ quasi-discriminant family of~$m$-homogenous functions.
The height $H=H((f_v)_v)$ is a Northcott height. Moreover, for every $B>0$, there exists $C>0$ such that $$|\{x\in[\PPP(m)(F)]| H(x)<B\}|<CB^{m\alpha(m)}.$$
\end{prop}
\begin{proof}
Let $(f_v^{\Delta}:\Fvt\to\RR_{>0})_v$ be the discriminant degree~$1$ family of~$m$-homogenous functions. The families $(f^{\Delta}_v)_v$ and $(f_v)_v$ are degree~$1$ families of~$m$-homogenous continuous functions and for almost all~$v$ one has $f^{\Delta}_v=f_v$ (because $(f_v)_v$ is quasi-discriminant), thus by \ref{toricisclos}, there exists a constant $C_1>0$ such that $C_1H^{\Delta}\leq H(y)$ for every $y\in[\PPP(m)(F)]$.  Let $(f_v^{\#}:\Fvt\to\RR_{>0})_v$ be the toric family of~$m$-homogenous functions of weighted degree~$1$ and let $H^{\#}=H((f_v^{\#})_v)$ be the resulting height. Lemma \ref{gvfvtor} gives that there exists $C_2>0$ such that $$H^{\Delta}(y)\geq C_2H^{\#}(y)^{1/\alpha(m)},$$ for every $y\in[\PPP(m)(F)]$.  Now, for every $B>0$, one has that \begin{align*}|\{x\in[\PPP(m)(F)]|H(x)<B\}|\hskip-2cm&\\&\leq |\{x\in[\PPP(m)(F)]|C_1H^{\Delta}(x)<B\}|\\&\leq|\{x\in[\PPP(m)(F)]|C_2C_1(H^{\#}(x))^{1/\alpha(m)}<B\}|\\&=|\{x\in[\PPP(m)(F)]|H^{\#}(x)<C_2^{-1}C_1^{-1}B^{\alpha(m)}\}|.\end{align*} 
Theorem \ref{osnovna} implies that there exists $C_3>0$ such that $$|\{x\in[\PPP(m)(F)]|H^{\#}(x)<C_0^{-1}C_1^{-1}B^{\alpha(m)}\}| <C_3C_2^{-1}C_1^{-m}B^{m\alpha(m)}$$for every $B>0$ (recall that in \ref{osnovna}, the degree of the toric family is~$m$ and our toric family is of the degree~$1$). Thus for every $B>0,$ we have that $$|\{x\in[\PPP(m)(F)]|H(x)<B\}|\leq C_3C_2^{-1}C_1^{-m}B^{m\alpha(m)}.$$The statement is proven.
\end{proof}
\begin{rem}
\normalfont In the next sections we establish the precise asymptotic behaviour of $|\{x\in[\PPP(m)(F)]|H(x)<B\}|$ when $B\to\infty$.
\end{rem}
\section{Analysis of height zeta function} The goal of this section is to establish the asymptotic behaviour of $|\{x\in[\PPP(m)(F)]|H(x)\leq B\}$. Using the ``Tauberian dictionary", the task translates into the study the convergence of the height zeta series. For that purpose we use Fourier analysis. 

Let $(f_v:\Fvt\to\RR_{>0})_v$ be a quasi-discriminant {\bf degree~$m$} family of~$m$-homogenous functions. For $\vMF$, we will denote by $f_v^{\Delta}$ the discriminant~$m$-homogenous function of weighted degree~$m$.  In the entire section we will denote by~$S$ the finite set\begin{multline*}S:=\\\{v\in M_F^0|\text{$f_v$ is not the discriminant~$m$-homogenous function or $v(m)\neq 0$}\}.\end{multline*}
\subsection{} In this paragraph we study the local Fourier transform of a local height.

Let $\vMF$. Let $q^m_v:\Fvt\to(\Fvt)/(\Fvt)_m=[\PPP(m)(\Fv)]=[\TT(m)(\Fv)]$ be the quotient map. By Lemma \ref{xiub}, the measure $d^*x$ on $\Fvt$ is $\Fvt$-invariant for the action $t\cdot y=t^my$ of $\Fvt$ on $\Fvt$. We set $\mu_v$ to be the quotient Haar measure $d^*x/d^*x$ on $[\TT(m)(\Fv)]$  (see \ref{haarttafv}).
Recall that the sets $[\TT(m)(\Fv)]$ are finite by \ref{ttafvfinite}, hence $\mu_v([\TT(m)(\Fv)])$ are finite positive numbers.
\begin{lem}\label{normmmu}
The measure $\mu_v$ is normalized by $$\mu_v([\TT(m)(\Fv)])=m.$$
\end{lem}
\begin{proof}
Suppose that $\vMFz$.  Recall from \ref{identofov} that $[\TT(m)(\Ov)]$ identifies with the open and compact subgroup $q^{m}_v(\Ovt)$ of $[\TT(m)(\Fv)]$ and is of index~$m$ by \ref{indofov}.  By \ref{xiuim}, the measure $\mu_v$ is normalized by $\mu_v([\TT(m)(\Ov)])=1$ i.e. by $\mu_v([\TT(m)(\Fv)])=m$. Suppose that $\vMFi$. One has by \ref{smacor} that \begin{equation*}\int_{[\TT(m)(\Fv)]}1\mu_v=\frac{m}{\lambda_{v,1}(\Fv)}\int_{F_{v,1}}1\lambda_{v,1}=m.
\end{equation*}
\end{proof}
If $\chi\in[\TT(m)(\Fv)]^*$ is a character, we denote by $\widetilde\chi$ the pullback character $(q^m_v)^*(\chi):\Fvt\to S^1$.
Lemma \ref{senjak} gives that $\widetilde\chi^m=1$ and that $\chi\mapsto \widetilde\chi$ is an isomorphism of $[\TT(m)(\Fv)]^*$ to the closed subgroup $(\Fv)_m^{\perp}\subset(\Fvt)^*$. For a complex number $s$ and a character $\chi\in[\TT(m)(\Fv)]^*$ we define formally $$\wH_v(s,\chi):=\int_{[\TT(m)(\Fv)]}H_v^{-s}\chi\mu_v.$$
\begin{lem}\label{wustani}Let $\vMF$.
\begin{enumerate}
\item For every $s\in\CC$ one has that $H_v^{-s}\in L^1([\TT(m)(\Fv)],\mu_v)$. For every $\chi\in[\TT(m)(\Fv)]^*$, one has that $s\mapsto \wH_v(s,\chi)$ is an entire function. Moreover, for every compact $\mathcal K\subset\RR$, there exists $C(\mathcal K)>0$ such that for every $s\in\mathcal K+i\RR,$ one has that $\wH_v(s,\chi)\leq C.$
\item Suppose that $v\in M_F^0-S$. Let $s\in\CC$ and let $\chi\in[\TT(m)(\Fv)]^*$ be a character. One has that \begin{align}
 \quad
\wH^{\Delta}_v(s,\chi):=&
               \begin{cases}
\sum_{j=0}^{m-1}\pivv^{(s(m^2-m\gcd(j,m)))/\alpha(m)}\widetilde\chi(\piv^j)&\text{if $\chi_v|_{[\TT ^m(\Ov)]}=1$,}\\
0&\text{otherwise.}
               \end{cases}
               \end{align} 
\item Suppose that $\vMFi$. For every $s\in\CC$ and every $\chi\in[\TT(m)(\Fv)]^*$, one has that  \begin{align}
 \quad
\wH_v(s,\chi):=&
               \begin{cases}
m&\text{if $\chi_v=1$,}\\
0&\text{otherwise.}
               \end{cases}
               \end{align} 
\end{enumerate}               
\end{lem} 
\begin{proof}
\begin{enumerate}
\item The group $[\TT(m)(\Fv)]$ is finite by \ref{ttafvfinite}. Moreover, $\mu_v$ is a Haar measure on $[\TT(m)(\Fv)]$, hence a non-zero multiple of the counting measure. It follows that $H_v^{-s}\in L^1([\TT(m)(\Fv)],\mu_v)$. Let $\chi\in[\TT(m)(\Fv)]^*$ be a character. For every $x\in[\TT(m)(\Fv)]$, one has that $s\mapsto H_v(x)^{-s}\chi(x)$ is an entire function. We deduce that $s\mapsto \wH_v(s,\chi)$ is an entire function. Moreover, for every $s\in\CC$, by the triangle inequality, one has that $|\wH_v(s,\chi)|\leq\wH_v(\Re(s),1)$. For every compact $\mathcal K\subset \RR$, by the fact that $\mu_v$ is a multiple of the counting measure, there exists $C'(\mathcal K)>0$ such that if $\Re(s)\in\mathcal K$ then $$|\wH_v(s,\chi)|\leq \wH(\Re(s),1)\leq C'(\mathcal K)\sup_{y\in\mathcal K}\sum_{x\in[\TT(m)(\Fv)]}H_v(x)^{-y}. $$ The claim is proven.
\item It follows from Lemma \ref{hdtoi} that $$[\TT(m)(\Ov)]\to\CC\hspace{1cm} x\mapsto (H_v^{\Delta}(x))^{-s}$$ is $[\TT(m)(\Ov)]$-invariant. We deduce that if $\chi|_{[\TT(m)(\Ov)]}\neq 1$, then $$\wH^{\Delta}_v(s,\chi)=\int_{[\TT(m)(\Fv)]}(H_v^{\Delta})^{-s}\chi\mu_v=0.$$
Suppose that $\chi|_{[\TT(m)(\Ov)]}=1$. It follows that $\widetilde\chi|_{\Ovt}=1$. The function $(H_v^{\Delta})^{-s}$ is $[\TT(m)(\Ov)]$-invariant, hence is $((H_v^{\Delta})^{-s}\circ\qav)=x\mapsto (|x|_v^{-1}f_v^{\Delta}(x))^{-s}$ is $[\TT(m)(\Ov)]$-invariant.  Using \ref{smacor}, we get that \begin{align*}\wH^{\Delta}_v(s,\chi)&=\int_{[\TT(m)(\Fv)]}(H_v^{\Delta})^{-s}\chi\mu_v\\
&=\int_{[\TT(m)(\Fv)]}((H_v^{\Delta})^{-s}\circ\qav)\cdot (\chi\circ\qav)\mu_v\\
&=\zeta_v(1)\int_{\Fvt\cap\DD^m_v}(|x|_v^{-1}f_v^{\Delta}(x))^{-s}\widetilde\chi(x)d^*x\\
&=\zeta_v(1)\int_{\Fvt\cap\DD^m_v}(\pivv^{(m\gcd(v(x),m)-m^2)/\alpha(m)})^{-s}\widetilde\chi(x)d^*x\\
&=\zeta_v(1)\sum_{j=0}^{m-1}\int_{\piv^j\Ov}\pivv^{s(m^2-m\gcd(v(x),m))/\alpha(m)}\widetilde\chi(x)d^*x\\
&=\zeta_v(1)\sum_{j=0}^{m-1}\pivv^{s(m^2-m\gcd(j,m))/\alpha(m)}\widetilde\chi(\piv^j)\cdot d^*x(\piv^j\Ovt)\\
&=\sum_{j=0}^{m-1}\pivv^{s(m^2-m\gcd(j,m))/\alpha(m)}\widetilde\chi(\piv^j).
\end{align*}The claim is proven.
\item For every $s\in\CC,$ the function $(H_v^{\Delta})^{-s}$ is the constant function~$1$. Thus for $\chi\in[\TT(m)(\Fv)]^*$ such that $\chi_v\neq 1$ one has that $\wH_v(s,\chi)=0$. Suppose that $\chi=1$. We have that $$\int_{[\TT(m)(\Fv)]}(H_v^{\Delta})^{-s}\mu_v=\int_{[\TT(m)(\Fv)]}1\mu_v=m$$ by Lemma \ref{normmmu}. 
\end{enumerate}
\end{proof}
\subsection{} 
We will compare the Fourier transform of the local ``discriminant" height with a product of certain local~$L$-functions.
\begin{lem}\label{lfhdisc}
Let $v\in M_F^0-S$. Let $s\in\CC$ with $\Re(s)>0$ and let $\chi\in[\TT(m)(\Fv)]^*.$ One has that
$$\bigg|\frac{\wH^{\Delta}_v(s,\chi)}{\prod_{j=1}^{r-1}L_v(s,\widetilde\chi^{mj/r})}\bigg|\leq \bigg(\frac{\zeta_v\big(\Re(s)(1+1/\alpha(m))\big)}{\zeta_v\big(2\Re(s)(1+1/\alpha(m))\big)}\bigg)^{2^{r-1}m^3}.$$
\end{lem}
\begin{proof}
Suppose firstly $\chi|_{[\TT(m)(\Ov)]}\neq 1$. Lemma \ref{wustani} gives that $\wH^{\Delta}(s,\chi)=0,$ hence the inequality is trivially verified. Suppose now that $\chi|_{[\TT(m)(\Ov)]}=1$. As $v\in M_F^0-S$, Lemma \ref{wustani} gives that $$\wH^{\Delta}_v(s,\chi)=\sum_{j=0}^{m-1}\pivv^{s(m^2-m\gcd(j,m))/\alpha(m)}\widetilde\chi(\piv^j).$$ 
We have that:
\begin{align}\label{watonji}
\begin{split}
\bigg|\frac{\wH^{\Delta}_v(s,\chi)}{\prod_{j=1}^{r-1}L_v(s,\widetilde\chi^{mj/r})}\bigg|\hskip-3cm&\\&=\bigg|{\sum_{j=0}^{m-1}\pivv^{s(m^2-m\gcd(j,m))/\alpha(m)}\widetilde\chi(\piv^j)}\bigg|\cdot\bigg|\prod_{j=1}^{r-1}(1-\pivv^{s}\widetilde\chi^{mj/r}(\piv))\bigg|\\
&=\bigg|\bigg({\sum_{j=0}^{m-1}\pivv^{s(m^2-m\gcd(j,m))/\alpha(m)}\widetilde\chi(\piv^j)}\bigg)\prod_{j=1}^{r-1}(1-\pivv^{s}\widetilde\chi(\piv^{mj/r}))\bigg|.
\end{split}
\end{align}
Whenever $\gcd(m,j)=m/r$, we have $m^2-m\gcd(m,j)=m^2-m\gcd(m,j)=\alpha(m),$ and we rewrite the sum: \begin{align*}&\sum_{j=0}^{m-1}\pivv^{s(m^2-m\gcd(j,m))/\alpha(m)}\widetilde\chi(\piv^j)\\&=1+\sum_{\gcd(j,m)=\frac mr}\pivv^{s}\widetilde\chi(\piv^j)+\sum_{\gcd(m,j)<\frac{m}r}\pivv^{s(m^2-m\gcd(m,j))/\alpha(m)}\widetilde\chi(\piv^{j})\\
&=1+\pivv^{s}\bigg(\sum_{j=1}^{r-1}\widetilde\chi(\piv^{\frac{mj}r})\bigg)+\sum_{\gcd(m,j)<\frac mr}\pivv^{s(m^2-m\gcd(j,m))/\alpha(m)}\widetilde\chi(\piv^{j}).
\end{align*}
The last product of (\ref{watonji}) becomes $$\bigg|1+\sum_{k=1}^{m^2r}A_k(\chi)\pivv^{sk/\alpha(m)}\bigg| ,$$with $$A_1(\widetilde\chi)=A_2(\widetilde\chi)=\cdots=A_{\alpha(m)}(\widetilde\chi)=0.$$Moreover, each $A_k(\chi)$ is a sum of no more than $2^{r-1}m$ numbers of the absolute value~$1$, thus for every $k=1\doots m$, one can estimate that $|A_k(\chi)|\leq 2^{r-1}m$. Now, using the triangle inequality and the fact that $\pivv<1$ we deduce that \begin{align*}\bigg|\frac{\wH^{\Delta}_v(s,\chi)}{\prod_{j=1}^{r-1}L_v(s,\widetilde\chi^{mj/r})}\bigg|&=\bigg|1+\sum_{k=1}^{m^2r}A_k(\chi)\pivv^{sk/\alpha(m)}\bigg|\\&\leq 1+2^{r-1}m\sum_{k=\alpha(m)+1}^{m^2r}\pivv^{\Re(s)k/\alpha(m)}\\&\leq 1+2^{r-1}m(m^2r-\alpha(m)-1)\pivv^{\Re(s)(\alpha(m)+1)/\alpha(m)}\\&\leq1+2^{r-1}m^3\pivv^{\Re(s)(1+(1/\alpha(m)))}\\
&\leq(1+\pivv^{\Re(s)(1+(1/\alpha(m)))})^{2^{r-1}m^3}\\
&= \bigg(\frac{\zeta_v\big(\Re(s)(1+(1/\alpha(m)))\big)}{\zeta_v(2\Re(s)(1+(1/\alpha(m)))\big)}\bigg)^{2^{r-1}m^3}.\end{align*}
The statement is proven.
\end{proof}
\subsection{}In this paragraph we study the global Fourier transform.

Let $[\TT(m)(\AAF)]$ be the restricted product $$[\TT(m)(\AAF)]=\sideset{}{'}\prod _{v\in M_F}[\TTT ^{m}(F_v)],$$where the restricted product is taken with the respect to the open and compact subgroups $[\TT(m)(\Ov)]\subset[\TT(m)(\Fv)]$ for $\vMFz$. Let $[\TT(m)(i)]:[\TT(m)(F)]\to[\TT(m)(\AAF)]$ be the diagonal map. If $(x_v)_v\in[\TT(m)(\AAF)]$, by \ref{hdtoi}, the product $$H(x):=\prod_{\vMF}H_v(x)$$ is finite. By \ref{rsra}, the function $H:[\TT(m)(\AAF)]\to\RR_{>0}$ is continuous.  By (\ref{htmfhv}), for $x\in [\TT(m)(F)],$ one has that  \begin{equation}\label{hxhttmx}H(x)=H([\TT(m)(i)](x)).\end{equation} 
For $\vMF$, let $K_v$ be the maximal subgroup of the finite group $[\TT(m)(F_v)]$ such that~$H_v$ is $K_v$-invariant. By \ref{hdtoi}, for every $v \in M_F^0-S,$ one has that $K_v=[\TT(m)(\Ov)].$ 
\begin{lem}\label{doslema}
Let $x\in[\TT(m)(\AAF)].$ We denote by $S_x$ the finite set of places of~$F$ given by the union of the set $M_F^{\infty}$, of the set of places~$v$ for which $v(x)\neq 0$ and of the set of the places~$v$ for which~$H_v$ is not $[\TT(m)(\Ov)]$-invariant. Let us set $$\prod_{v\in S_x}\frac{\max _{z\in[\TT(F_v)]}H_v(z)}{\min _{z\in[\TT(F_v)]}H_v(z)}.$$
For every $y\in[\TT(m)(\AAF)]$, one has that $$C(x)^{-1}H(y)\leq H(xy)\leq C(x)H(y).$$
\end{lem}
\begin{proof}
For $v\in M_F-S_x$, one has that $H_v(x)=1$.  For $v\in S_x$, let us set $$C_v(x):=\frac{\max _{z\in[\TT(F_v)]}H_v(z)}{\min _{z\in[\TT(F_v)]}H_v(z)}$$ so that $C(x)=\prod_{v\in S_x}C_v(x)$. 
For every $y\in[\TT(m)(\AAF)]$, one has that \begin{align*}H(xy)&=\prod_{\vMF}H_v(xy)\\
&=\prod_{v\in S_x}H_v(xy)\cdot \prod_{v\in M_F-S_x}H_v(xy)\\
&\leq \prod_{v\in S_x}C_v(x) H_v(y) \cdot \prod_{M_F-S_x}H_v(y)\\
&=C(x)H(y).
\end{align*}
Analogously, one verifies that $C(x)^{-1}H(xy)\leq H(y)$.

For $z\in\RR$, we denote by $\Omega_{>z}$ the ``tube": $$\Omega_{>z}:=\RR_{>z}+i\RR\subset \CC.$$
\end{proof}
\begin{lem}\label{disckchi} Let $\chi\in([\TT(i)]([\TT(m)(F)]))^{\perp}$. 
\begin{enumerate}
\item Suppose $\chi$ does not vanish on~$K$. Then for every $s\in\CC$, one has that $\wH(s,\chi)=0$.
\item Suppose $\chi$ vanishes on~$K$. For $s\in\Omega_{>\frac{\alpha(m)}{\alpha(m)+1}},$ the product $$\prod_{\substack{\vMFz}}\frac{\wH_v(s,\chi_v)}{\prod_{j=1}^{r-1}L_v(s,\widetilde \chi^{mj/r})}$$converges uniformly on compacts in the domain $\Omega_{>\frac{\alpha(m)}{\alpha(m)+1}}.$ The function $$\gamma(-,\chi):s\mapsto \prod_{\vMFz}\frac{\wH_v(s,\chi_v)}{\prod_{j=1}^{r-1}L_v(s,\widetilde \chi^{mj/r})},$$is a holomorphic function $\Omega_{>\frac{\alpha(m)}{\alpha(m)+1}}\to\CC$ which satisfies that for every compact $\mathcal K\subset \RR_{>\frac{\alpha(m)}{\alpha(m)+1}}$ there exists $C=C(\mathcal K)>0$ such that $$|\gamma(s,\chi)|\leq C,$$ for $s\in \mathcal K+i\RR$.
\item One has that $\gamma(1,1)>0$.
\end{enumerate}
\end{lem}
\begin{proof}
\begin{enumerate}
\item  Let $v\in M_F^0$ such that $\chi_v|_{K_v}\neq 1$. By \ref{wustani}, one has that $\wH_v(s,\chi_v)=0$ for every $s\in\CC$. It follows that $\wH(s,\chi)=0$. 
\item For $\vMFz$, let us denote by $$\gamma_v(s,\chi_v):=\frac{\wH_v(s,\chi_v)}{\prod_{j=1}^{r-1}L_v(s,\widetilde \chi^{mj/r})}.$$ For every $\vMFz,$ the function $\gamma_v(-,\chi_v)$ is an entire function, because by \ref{wustani}, the function $\wH_v(s,\chi_v)$ is an entire function and because for $j=1\doots r-1$, the function $(L_v(s,\widetilde\chi^{mj/r}))^{-1}=(1-\pivv^{s}\widetilde\chi^{mj/r}(\pivv))$ is an entire function. Moreover, by \ref{lfhdisc}, for $v\in M_F^0-S,$ there exists a positive integer~$A$ such that
\begin{equation}|\gamma_v(s,\chi)|\leq \bigg(\frac{\zeta_v\big(\Re(s)(1+\frac1{\alpha(m)})\big)}{\zeta_v\big(2\Re(s)(1+\frac1{\alpha(m)})\big)}\bigg)^A. \end{equation}For every $y>\alpha(m)/(\alpha(m)+1)$, the product $$\prod_{v\in M_F^0-S}\bigg(\frac{\zeta_v\big(y(1+\frac1{\alpha(m)})\big)}{\zeta_v\big(2y(1+\frac1{\alpha(m)})\big)}\bigg)^A $$ converges uniformly in the domain $\RR_{>\frac{\alpha(m)}{(\alpha(m)+1)}}$ to $$ \bigg(\frac{\zeta(y\big(1+\frac{1}{\alpha(m)})\big)}{\zeta\big(2y(1+\frac1{\alpha(m)})\big)}\times\prod_{v\in S}\frac{\zeta_v\big(2y(1+\frac1{\alpha(m)})\big)}{\zeta_v\big(y(1+\frac1{\alpha(m)})\big)}\bigg)^A.$$ It follows that the product $$\prod_{v\in M_F^0-S}\gamma_v(s,\chi_v)$$converges absolutely and uniformly in the domain $s\in\Omega_{>\frac{\alpha(m)}{\alpha(m)+1}}$. We deduce the product $\gamma(s,\chi)=\prod_{v\in M_F^0}\gamma_v(s,\chi_v)$ converges absolutely and uniformly in the domain $s\in\Omega_{>\frac{\alpha(m)}{\alpha(m)+1}}.$ Moreover, $\gamma(-,\chi):\Omega_{>\frac{\alpha(m)}{\alpha(m)+1}}\to\CC$ is a holomorphic function and it satisfies that \begin{multline*}|\gamma(s,\chi)|\\\leq\bigg(\frac{\zeta\big(\Re(s)(1+\frac{1}{\alpha(m)})\big)}{\zeta\big(2\Re(s)(1+\frac{1}{\alpha(m)})\big)}\times\prod_{v\in S}\frac{\zeta_v\big(2\Re(s)(1+\frac{1}{\alpha(m)})\big)}{\zeta_v\big(\Re(s)(1+\frac{1}{\alpha(m)})\big)}\bigg)^A\times\prod_{v\in S}|\gamma_v(s,\chi)|.\end{multline*}We deduce that for a compact $\mathcal K\subset \RR_{>1/(\alpha(m)+1)}$ one has that $$|\gamma(s,\chi)|\leq \sup_{y\in\mathcal K}\bigg(\bigg(\frac{\zeta\big(y(1+\frac{1}{\alpha(m)})\big)}{\zeta\big(2y(1+\frac{1}{\alpha(m)})\big)}\times\prod_{v\in S}\frac{\zeta_v\big(2y(1+\frac{1}{\alpha(m)})\big)}{\zeta_v\big(y(1+\frac{1}{\alpha(m)})\big)}\bigg)^A\prod_{v\in S}\gamma_v(y,\chi)\bigg).$$ The statement is proven.
\item For $\vMFz$, one has that \begin{align*}\gamma_v(1,1)&=\frac{\wH_v(1,1)}{\zeta_v(1)^{r-1}}\\
&=\bigg(\sum_{j=0}^{m-1}\pivv^{(m^2-m\gcd(j,m))/\alpha(m)}\bigg)(1-\pivv)^{r-1}\\
&=\bigg(1+(r-1)\pivv+O(\pivv^{\frac{m^2-m^2/r+1}{\alpha(m)}})\bigg)\times\\&\hspace{0.3cm}\times\bigg(1-(r-1)\pivv+O(\pivv^2)\bigg)\\
&=\bigg(1-(r-1)^2\pivv+O(\pivv^{\frac{m^2-m^2/r+1}{\alpha(m)}})\bigg),
\end{align*}where by $O(\pivv^k)$ is meant a quantity that for all~$v$ is bounded by a constant (independent of~$v$) times $\pivv^k$.
It follows that $$\gamma(1,1)=\prod_{\vMFz}\gamma_v(1,1)>0.$$ 
\end{enumerate}
\end{proof}
\subsection{} In this paragraph we estimate the global Fourier transform.

In \ref{identaafj}, we have provided an identification $\AAF^1\times\RR_{>0}\xrightarrow{\sim}\AAFt$. If $\chi$ is a character of $[\TT(m)(\AAF)]=\AAFt/(\AAFt)_m$, then $\widetilde\chi$ is of order dividing~$m$, thus $\widetilde\chi|\RR_{>0}=0.$ 
Recall that by \ref{ljweq}, one has that $[\TT(m)(\AAF)]=[\TT(m)(\AAF)]_1$. 

Let $K^0_{\max}:=\prod_{\vMFz}[\TT(m)(\Ov)]$. The group $K^0_{\max},$ as well as its any open subgroup $K,$ is compact. By \ref{finmankchar}, for an open subgroup $K\subset K^0_{\max}:=\prod_{\vMFz}[\TTa(\Ov)],$ one has that  $$\AK:=(K[\TT(m)(i)]([\TT(m)(F)]))^{\perp}$$ is finite. 
We are ready to estimate the global Fourier transform. Let $\mu_{\AAF}$ be the restricted product measure $$\mu_{\AAF}=\bigotimes_{\vMF}\mu_v$$ on $[\TTa(\AAF)]$ (as in \ref{muaaf}).
For $s\in\CC$ and a character $\chi\in[\TT(m)(\AAF)]^*$, we define formally $$\wH(s,\chi)=\int_{[\TT(m)(\AAF)]}H^{-s}\chi\mu_{\AAF}.$$
Lemma \ref{rsra} gives that if the product $\prod_{\vMF}\wH_v(s,\chi_v)$ converges, then $$\wH(s,\chi)=\prod_{\vMF}\wH_v(s,\chi_v).$$
\begin{lem}\label{whdiscam} For $\chi\in\AK$ let us define $d(\chi)=0$ if $\chi^{m/r}\neq 1$ and $d(\chi)=r-1$, otherwise. For every $\chi\in\AK$, the function $\wH(s,\chi)$ is holomorphic in the domain $s\in\Omega_{>1}$. Moreover, there exists $\delta>0$ such that $$\bigg(\frac{s-1}{s}\bigg)^{d(\chi)}\wH(s,\chi)$$extends to a holomorphic function in the domain $\Omega_{>1-\delta}$ and such that for every compact $\mathcal K\subset \RR_{>1-\delta}$ there exists $C=C(\mathcal K)>0$, such that $$\bigg|\bigg(\frac{s-1}{s}\bigg)^{d(\chi)}\wH(s,\chi)\bigg|\leq C(\mathcal K)(1+|\Im(s)|)$$for every $\chi\in\AK.$
\end{lem}
\begin{proof}
 Set $\gamma_v(s,\chi_v)=\frac{\wH_v(s,\chi_v)}{\prod_{j=1}^{r-1}L_v(s,\widetilde \chi^{mj/r})}$ for $\vMFz$. By \ref{disckchi}, for every $\chi\in\AK$, the product $\prod_{\vMFz}\gamma_v(s,\chi_v)$ converges absolutely and uniformly on the compacts in the domain $\Omega_{>\frac{\alpha(m)}{\alpha(m)+1}}$ to a holomorphic function $\gamma(-,\chi):\Omega_{>\frac{\alpha(m)}{\alpha(m)+1}}\to\CC$ and, moreover as $\AK$ is finite, there exists $C_0>0$ such that $|\gamma(s,\chi)|\leq C_0$. 
 
Let us define $\widetilde K:=(q^m|_{K^0_{\max}})^{-1}(K)$. It is an open subgroup of the compact group $K^0_{\max},$ hence $\widetilde K$ is of the finite index in $K^0_{\max}$. Note that if $\chi|_K=1,$ then $\widetilde\chi|_{\widetilde K}=1$.  Now, by \ref{Rade}, we get that there exist $\frac1{\alpha(m)+1}>\delta>0$ and $C>0$, such that for every $\chi\in\AK$ one has that \begin{multline*}\bigg(\prod_{\substack{j=1\\\widetilde\chi^{mj/r}=1}}^{r-1}\frac{s-1}{s}\bigg)\bigg(\prod_{j=1}^{r-1}L(s,\widetilde\chi^{mj/r})\bigg)\\=\bigg(\prod_{\substack{j=1\\\widetilde\chi^{mj/r}=1}}^{r-1}\frac{(s-1)L(s,\widetilde\chi^{mj/r})}{s}\bigg)\bigg(\prod_{\substack{j=1\\\widetilde\chi^{mj/r}\neq 1}}^{r-1}L(s,\widetilde\chi^{mj/r})\bigg) \end{multline*}extends to a holomorphic function in the domain $\Omega_{>1-\delta}$ and satisfies that $$\bigg|\bigg(\prod_{\substack{j=1\\\widetilde\chi^{mj/r}=1}}^{r-1}\frac{s-1}{s}\bigg)\bigg(\prod_{j=1}^{r-1}L(s,\widetilde\chi^{mj/r})\bigg)\bigg|\leq C(1+|\Im(s)|)$$in this domain. 
We deduce that in the domain $\Omega_{>1-\delta}$ one has that\begin{align*}
 \wH(s,\chi)&=\prod_{\vMFz}\wH_v(s,\chi_v)\times\prod_{\vMFi}\wH_v(s,\chi_v)\\&=\bigg(\prod_{\vMFz}\gamma_v(s,\chi)\prod_{j=1}^{r-1}L_v(s,\widetilde\chi^{mj/r}_v)\bigg)\times \prod_{\vMFi}\wH_v(s,\chi)
\end{align*}converges to $\gamma(s,\chi)\prod_{j=1}^{r-1}L(s,\widetilde\chi^{mj/r})\prod_{\vMFi}\wH_v(s,\chi)$. Moreover, using \ref{wustani},  if $\mathcal K\subset\RR_{>1-\delta}$ is a compact, we deduce that there exists $C>0$ such that
 \begin{align*}\bigg|\bigg(\prod_{\substack{j=1\\\widetilde \chi^{mj}=1}}^{r-1}\frac{s-1}{s}\bigg)\wH(s,\chi)\bigg|\hskip-4,2cm&\\&=\bigg|\bigg(\prod_{\substack{j=1\\\widetilde \chi^{mj}=1}}^{r-1}\frac{s-1}{s}\bigg)\gamma(s,\chi)\bigg(\prod_{j=1}^{r-1}L(s,\widetilde\chi^{mj/r})\bigg)\times \prod_{\vMFi}\wH_v(s,\chi)\bigg|\\
 &= |\gamma(s,\chi)|\bigg(\prod_{\substack{j=1\\\widetilde\chi^{mj/r}=1}}^{r-1}\frac{s-1}{s}L(s,1)\bigg)\bigg(\prod_{\substack{j=1\\\widetilde\chi^{mj/r}\neq 1}}^{r-1}|L(s,\widetilde\chi^{mj/r})|\bigg) \bigg|\prod_{\vMFi}\wH_v(s,\chi)\bigg|\\
 &\leq C(1+|\Im(s)|).
 \end{align*}  
To complete the proof, it suffices to see that for every $\chi\in\AK$, one has that \begin{equation}\label{obvty}|\{j|\hspace{0.1cm} 0\leq j\leq r-1\text{ and } \widetilde\chi^{mj/r}=1\}|=d(\chi).\end{equation} Suppose that $\chi^{m/r}=1$. For every $j\in\{1\doots r-1\}$ one has that $\widetilde\chi^{mj/r}=((q^m_{\AAF})^*\chi)^{mj/r}=(q^m_{\AAF})^*(\chi^{mj/r})=1$, hence $$|\{j|\hspace{0.1cm} 0\leq j\leq r-1\text{ and } \widetilde\chi^{mj/r}=1\}|=r-1=d(\chi).$$ Suppose that $\chi^{m/r}\neq 1$. Suppose that for some $j\in\{1\doots r-1\}$ one has that $$\widetilde\chi^{mj/r}=((q^m_{\AAF})^*\chi^{mj/r})=1.$$ Let $k\in\{1\doots r-1\}$ be such that $kj=1+\ell r$ for some $\ell\in\ZZ_{>0}$. Then, $$1=\widetilde\chi^{mjk/r}=\widetilde \chi^{(m\ell r+1)/r} =\widetilde\chi^{m\ell}\widetilde\chi^{m/r}=\widetilde\chi^{m/r},$$because $\widetilde\chi^{m\ell}=(q^m_{\AAF})^*(\chi^{\ell})=1$ (because $[\TT(m)(\AAF)]$ is an~$m$-torsion group).This is a contradiction. We deduce that $$|\{j|\hspace{0.1cm} 0\leq j\leq r-1\text{ and } \widetilde\chi^{mj/r}=1\}|=0=d(\chi).$$ In either of the two cases, we obtain that \ref{obvty} is valid. We deduce that $$\bigg|\bigg(\frac{s-1}{s}\bigg)^{d(\chi)}\wH(s,\chi)\bigg| \leq C(1+|\Im(s)|)$$
 The statement is proven.
\end{proof} 
\subsection{} In this paragraph we define height zeta function, establish its convergence and the meromorphic extension of the function it defines.  
For $s\in\CC$, we define formally $$Z(s):=\sum_{x\in[\TT(m)(F)]}H(x)^{-s}.$$ The following lemma verifies some of the conditions that are needed to be satisfied in order to apply Poisson formula. We will write $i([\TT(m)(F)])$ for $[\TT(m)(i)]([\TT(m)(F)])$.
\begin{lem} \label{zconvunifd} The following claims are valid. 
\begin{enumerate}
\item Let $\epsilon>0$. For $s\in\Omega_{>\alpha(m)+\epsilon}$, the series defining $Z(s)$ converges absolutely and uniformly. The function $s\mapsto Z(s)$ is holomorphic in the domain $\Omega_{>\alpha(m)}$.
\item Let $s\in\Omega_{>\alpha(m)}$.  For every $x\in[\TT(m)(\AAF)]$, the series $$\sum_{y\in i([\TT(m)(F)])}H(xy)^{-s}$$ converges absolutely. The function $$x\mapsto \sum_{y\in i([\TT(m)(F)])}H(xy)^{-s} $$ is continuous.
\end{enumerate}
\end{lem}
\begin{proof}
\begin{enumerate}
\item It follows from \ref{northdisc} that there exists $C_1>0$ such that for every $B>0$ one has that$$|\{y\in[\PPP(m)(F)]| H(y)<B\}|<C_1B^{\alpha(m)}$$(recall that in \ref{northdisc} we had a quasi-discriminant degree~$1$ family, while now it is a degree~$m$ quasi-discriminant degree~$m$ family). Now, it follows from \ref{asioo}, that the series defining $Z(s)$ converges absolutely and uniformly for $s\in \Omega_{>\alpha(m)+\epsilon}$. Thus~$Z$ is holomorphic in the domain $\Omega_{>\alpha(m)+\epsilon}$, and by decreasing $\epsilon$ we deduce that~$Z$ is holomorphic in the domain $\Omega_{>\alpha(m)}$. 
\item For $x\in[\TT(m)(\AAF)]$, we denote by $S_{x}$ the finite set of places of~$F$ given by the union of the set $M_F^{\infty}$, of the set of places~$v$ for which $v(x)\neq 0$ and of the set of the places~$v$ for which~$H_v$ is not $[\TT(m)(\Ov)]$-invariant. Let us set $U_{x,v}:=\{x_v\}
\subset [\TT(m)(\Fv)]$ (where $x_v$ is the~$v$-adic component of~$x$). We define $$U_x=\prod_{v\in S_x}U_{x,v}\times\prod_{v\in M_F^0-S_x}[\TT(m)(\Ov)].$$ We are going to prove that the series converges absolutely and uniformly on $U_x$. For every $x'\in U_x$ and every $y\in i([\TT(m)(F)])$, it follows from Lemma \ref{doslema},  that there exists $C(x)>0$ such that 
\begin{align*}
\big|H(x'y)^{-s}\big|=H(x'y)^{-\Re(s)}\leq C(x)^{-\Re(s)}H(y)^{-\Re(s)}.
%
\end{align*}
The kernel of the homomorphism $[\TT(m)(F)]\to[\TT(m)(\AAF)]$ is $\Sh^1(F, \mu_m)$ by \ref{cesusp}. As $\sum _{y\in i([\TT(m)(F)])}H(y)^{-s}=\frac{Z(s)}{|\Sh^1(F, \mu_m)|}$ converges absolutely, it follows that \begin{align*}\sum _{y\in i([\TT(m)(F)])}\big|H(x'y)^{-s}\big|&= \sum _{y\in i([\TT(m)(F)])}H(x'y)^{-\Re(s)}\\&\leq C(x)^{-\Re(s)}\frac{Z(\Re(s))}{|\Sh^1(F, \mu_m)|},\end{align*}and hence the series $\sum_{y\in[\TT(m)(F)]}H(x'y)^{-s}$ converges absolutely and uniformly in the domain $x'\in U_x$. It follows that $x'\mapsto \sum _{y\in i([\TT(m)(F)])}H(x'y)^{-s}$ is continuous on $U_x$. We deduce that $x\mapsto\sum _{y\in i([\TT(m)(F)])}H(xy)^{-s}$ is continuous on $[\TT(m)(\AAF)]$.
\end{enumerate}
\end{proof}
Let $\Xi\subset[\TT(m)(\AAF)]^*$ be the group of the characters $\chi$ which satisfy that $\chi^{m/r}=1$. Note that $\Xi^{\perp}$ is given by $$[\TT(m)(\AAF)]_{m/r}=\{x^{m/r}|x\in[\TT(m)(\AAF)]\}.$$ For an open subgroup $K\subset K^0_{\max}=\prod_{\vMFz}[\TT(m)(\Ov)],$ we denote by $\Xi_K=\mathfrak A_K\cap \Xi.$ By \ref{finmankchar}, the subgroups $\AK$ are finite, thus discrete, hence the groups $\Xi_K$ are finite and discrete. To simplify notation, in the rest of the paragraph we may write $[\TT(m)(F)]$ for what is technically $[\TT(m)(i)]([\TT(m)(F)])$.
\begin{lem}\label{lemakojunerazumi} The following claims are valid:
\begin{enumerate} 
\item For every open subgroup $K\subset K^0_{\max},$ the group $\Xi^{\perp}_K$ is the kernel of the homomorphism $[\TT(m)(\AAF)]\to\Xi_K^*,$ which is induced from the inclusion $\Xi_K\subset [\TT(m)(\AAF)]^*$. Moreover, it is open, closed and of the index $|\Xi_K|$ in $[\TT(m)(\AAF)]$. One has that $\Xi^{\perp}_K=[\TT(m)(F)]K[\TT(m)(\AAF)]_{m/r}.$
\item The group $\Xi^{\perp}_{\infty}:=\bigcap \Xi^{\perp}_{K},$ where the intersection is over all open subgroups~$K$ of $K^0_{\max},$ identifies with the closure $$\overline{[\TT(m)(F)][\TT(m)(\AAF)]_{m/r}}\subset[\TT(m)(\AAF)].$$
\item Let $K\subset K^0_{\max}$ be an open subgroup and let~$f$ be a~$K$-invariant continuous complex valued function lying in $ L^1([\TT(m)(\AAF)],\mu_{\AAF})$. One has that $$\frac1{|\Xi_K|}\sum_{\chi\in\Xi_K}\widehat f(\chi)=\int_{\Xi_K^{\perp}}f\mu_{\AAF}.$$
\item There exists a unique Haar measure $\mu_{\infty}^{\perp}$ on $\Xi^{\perp}_{\infty}$ such that for every open subgroup $K\subset K^0_{\max}$, any~$K$-invariant continuous function $f:[\TT(m)(\AAF)]\to\CC$ one has that $f\in L^1(\Xi^{\perp}_K,|\Xi_K|\mu_{\AAF})$ if and only if $f\in L^1(\Xi^{\perp}_{\infty},\mu_{\infty}^{\perp})$ and if $f\in L^1(\Xi^{\perp}_K,|\Xi_K|\mu_{\AAF})$, then $$\int_{\Xi^{\perp}_{\infty}}f\mu_{\infty}^{\perp}=|\Xi_K|\int_{\Xi_K^{\perp}}f\mu_{\AAF}.$$
\end{enumerate}
\end{lem}
\begin{proof}
\begin{enumerate}
\item The claim that $\Xi_K^{\perp}$ is the kernel of the homomorphism $[\TT(m)(\AAF)]\to \Xi_K^*$ follows from Proposition \ref{tspgj}. The same proposition gives that the homomorphism $[\TT(m)(\AAF)]\to \Xi_K^*$ is surjective. The group $\Xi_K^*$ is finite and discrete, of the order $|\Xi_K|$, thus $\Xi_K^{\perp}$ is an open, closed and of the index $|\Xi_K|$, as claimed. Clearly, \begin{multline*}\Xi_K^{\perp}=([\TT(m)(F)]^{\perp}K^{\perp}\cap [\TT(m)(\AAF)]_{m/r})^{\perp}\\=\overline{ [\TT(m)(F)]K[\TT(m)(\AAF)]_{m/r}}.\end{multline*}The subgroup $[\TT(m)(F)]K\subset[\TT(m)(\AAF)]$ is closed (because it is a product of a discrete subgroup $[\TT(m)(F)]$ and a compact subgroup~$K$ in $[\TT(m)(\AAF)]$), hence, is equal to $\AK^{\perp}=(([\TT(m)(F)]K)^{\perp})^{\perp}=\overline{[\TT(m)(F)]K}$, which is of the finite index in $[\TT(m)(\AAF)]$. It follows that the subgroup $$[\TT(m)(F)]K[\TT(m)(\AAF)]_{m/r}\subset[\TT(m)(\AAF)]$$ is closed, as it contains $\AK^{\perp}$ as a finite index subgroup. The claim follows.

\item Firstly, let us observe that $\overline{[\TT(m)(F)][\TT(m)(\AAF)]_{m/r}}\subset \Xi^{\perp}_{\infty},$ because $\Xi^{\perp}_{\infty}$ is a closed subgroup of $[\TT(m)(\AAF)]$ and because for every compact and open $K\subset [\TT(m)(\AAF)]$ one has that $$[\TT(m)(F)][\TT(m)(\AAF)]_{m/r}\subset [\TT(m)(F)][\TT(m)(\AAF)]_{m/r}K=\Xi_K^{\perp}.$$ Let now $$y\in[\TT(m)(\AAF)]-\overline{([\TT(m)(F)])[\TT(m)(\AAF)]_{m/r}}.$$ The open subgroups of $K^0_{\max}$ form a basis of neighbourhoods of $1\in[\TT(m)(\AAF)]$, thus there exists an open subgroup $K\subset K^0_{\max}$ such that $$y\not\in   ([\TT(m)(F)][\TT(m)(\AAF)]_{m/r})K=\Xi^{\perp}_K.$$It follows that $y\not\in\Xi^{\perp}_{\infty}$ and the claim is proven.
\item We apply the Poisson formula for the inclusion $\Xi_K\subset[\TT(m)(\AAF)]^*,$ where $\Xi_K$ is endowed with the counting measure and $[\TT(m)(\AAF)]^*$, with the dual measure $\mu_{\AAF}^*$ of the measure $\mu_{\AAF}$. Every $x\in[\TT(m)(\AAF)]$ will be regarded as a character of $[\TT(m)(\AAF)]^*$ by the evaluation map. By (1), the group $\Xi_K^{\perp}$ identifies with the kernel of the homomorphism $[\TT(m)(\AAF)]\to \Xi_K^*,$ given by $x\mapsto x|_{\Xi_K},$ and is an open subgroup of the index $|\Xi_K|$ in $[\TT(m)(\AAF)]$. The dual measure of the measure $\coun_{\Xi_K}$ on the dual group $\Xi_K^*$ is given by $\frac1{|\Xi_K|}\cdot\coun_{\Xi_K^*}$, thus we have an equality of the measures on $([\TT(m)(\AAF)]^*/\Xi_K)^*=\Xi_K^{\perp}:$ $$(\mu^*_{\AAF}/\coun_{\Xi_K})^*=|\Xi_K|\cdot\mu_{\AAF}|_{\Xi_K^{\perp}}.$$ The Fourier transform of $\chi\mapsto \widehat f(\chi)$ at the character~$x$, by the Fourier inversion formula (\ref{FourierInversion}),  is equal to~$f$. By the finiteness of $\Xi_K$ and the continuity of $\chi\mapsto \widehat f(\chi)$ (\cite[Proposition 2, \no 2, \S 1, Chapter II]{TSpectrale}), the conditions (2) and (3) of Poisson formula (\ref{formuledepoison}) are satisfied, and applying it gives that \begin{equation*}\sum_{\chi\in\Xi_K}\widehat f(\chi)=|\Xi_K|\int_{\Xi_K^{\perp}}f\mu_{\AAF},\end{equation*}as claimed.
\item For an open subgroup $K\subset K^0_{\max}$, let $dk$ be the probability Haar measure on~$K$. For every~$K$, one has that \begin{multline*}\Xi_{\infty}^\perp K=\overline{([\TT(m)(F))[\TT(m)(\AAF)]_{m/r}}K=([\TT(m)(F))[\TT(m)(\AAF)]_{m/r}K
\\=\Xi_K^{\perp},\end{multline*} where the second equality follows from the fact that the product subset of a closed subset and a compact subgroup is closed (\cite[Corollary 1 of Proposition 1, \no 1, \S 4, Chapter III]{TopologieGj}). Let us denote by $g_K$ the canonical morphism $$g_K:\Xi^{\perp}_{\infty}\times K\to \Xi_\infty^{\perp}K=\Xi_K^{\perp}\hspace{1cm}(y,k)\mapsto yk.$$

The group $[\TT(m)(\AAF)]$ is countable at infinity by \ref{countableatinfty}. 
For every open $K\subset K^0_{\max}$, it follows from \cite[Corollary of Proposition 13, \no 9, \S 2, Chapter VII]{Integrationd} that there exists a unique Haar measure $\mu^{\perp K}_{\infty}$ on $\Xi_{\infty}^{\perp}$ such that for every continuous positive valued function~$f$ on $\Xi_K^{\perp}$ one has that $f\in L^1(\Xi^{\perp}_K,{\lvert\Xi_K\rvert}\mu_{\AAF})$ if and only if $f\circ g_K\in L^1(\Xi^{\perp}_{\infty}\times K,\mu^{\perp K}_{\infty}\times dk)$ and if $f\in L^1(\Xi^{\perp}_K,{|\Xi_K|}\mu_{\AAF})$ then $${|\Xi_K|}\int_{\Xi_K^{\perp}}f\mu_{\AAF}=\frac{1}{f((1)_v)}\int_{\Xi_{\infty}^{\perp}\times K}(f\circ g_K)\mu_{\infty}^{\perp K}\times dk.$$ We deduce that if $f\in L^1(\Xi_K^{\perp},{|\Xi_K|}\mu_{\AAF})$ is moreover assumed to be~$K$-invariant, then $${|\Xi_K|}\int_{\Xi^\perp_K} f\mu_{\AAF}=\int_{\Xi_{\infty}^{\perp}}f\mu_\infty^{\perp K}.$$ We now prove that the measures $\mu_\infty ^{\perp K}$ are independent of the choice of~$K$. Firstly, we prove it for open subgroups of $K'\subset K$ (such subgroups are compact and of the finite index in~$K$). Let $h:[\TT(m)(\AAF)]\to\RR_{\geq 0}$ be~$K$-invariant function in $L^1([\TT(m)(\AAF)],\mu_{\AAF})$. 
Using (3) and the fact that for a character $\chi\in[\TT(m)(\AAF)]^*$ if $\chi|_{K'}\neq 1,$ then $\widehat h(\chi)=0$ (because $h$ is $K'$-invariant), we deduce that \begin{multline*}
\int_{\Xi_\infty^{\perp}}h\mu_{\infty}^{\perp K}=|\Xi_K|\int_{\Xi_K^{\perp}}h\mu_{\AAF}=\sum _{\chi\in\Xi^{\perp}_K}\widehat h(\chi)=\sum _{\chi\in \Xi^{\perp}_{K'}}\widehat h(\chi)\\=|\Xi_{K'}|\int_{\Xi_{K'}^{\perp}}h\mu_{\AAF}
=\int_{\Xi_\infty^{\perp}}h\mu_{\infty}^{\perp K'}.
\end{multline*}
As $\mu_{\infty}^{\perp K} $ and $\mu_{\infty}^{\perp K'}$ are Haar measures on $\Xi_{\infty}^\perp$, it follows that $\mu_{\infty}^{\perp K}=\mu_{\infty}^{\perp K'}$. Now we prove the claim for general $K'$. As for any open subgroup $K'\subset K^0_{\max}$, we have that $K\cap K'$ is open in~$K$ and $K'$, we deduce $\mu_{\infty}^{\perp K}=\mu_{\infty}^{\perp K\cap K''}=\mu _{\infty}^{\perp K''}$. Thus $\mu_{\infty}^{\perp}:=\mu_{\infty}^{\perp K}$ is the wanted measure.
\end{enumerate}
\end{proof}
Denote by $j:\Xi_\infty^{\perp}\to \prod_{\vMF}[\PPP(m)(F_v)]$ the canonical inclusion. Note that~$j$ is the composite of the closed embedding $\Xi_\infty^{\perp}\hookrightarrow [\TT(m)(\AAF)]$ and the canonical inclusion $[\TT(m)(\AAF)]\hookrightarrow\prod_{\vMF}[\PPP(m)(F_v)].$ The later map is continuous by \ref{mimip}, hence~$j$ is continuous.

For any compactly supported continuous $\phi:\prod_{\vMF}[\PPP(m)(F_v)]\to\RR_{\geq 0},$ it follows from \ref{whdiscam} that the limit $$\lim_{s\to 1^{+}}(s-1)^{r-1}\int_{\Xi^{\perp}_{\infty}}(\phi\circ j) H^{-s}\mu_\infty=\lim_{s\to 1^{+}}(s-1)^{r-1}|\Xi_K|\int_{\Xi^{\perp}_{K}}(\phi\circ j) H^{-s}\mu_{\AAF}$$exists and is a non-negative number. It follows that $$\mathscr C^0_c\big(\prod_{\vMF}[\PPP(m)(F_v)],\RR_{\geq 0}\big)\to\RR_{\geq0} \hspace{0,5cm}\phi\mapsto \lim_{s\to 1^{+}}(s-1)^{r-1}\int_{\Xi_{\infty}^{\perp}}(\phi\circ j)H^{-s}\mu_\infty $$ is a non-negative linear form, hence by \cite[Theorem 1, \no 6, \S1, Chapter III]{Integrationj} extends to a measure on $\prod_{\vMF}[\PPP(m)(F_v)]$
\begin{mydef}\label{peyredisc}
Let $(f_v:\Fvt\to\RR_{>0})_v$ be a quasi-discriminant degree~$m$ family of~$m$-homogenous functions. We define a measure~$\omega$ on $\prod_{\vMF}[\PPP(m)(F_v)]$ by $$\omega=\omega((f_v)_v)=|\mu_m(F)|\lim_{s\to1^{+}}(s-1)^{r-1} j_*(H^{-s}\mu_{\infty}).$$
We set $$\tau=\tau((f_v)_v)=\omega(\prod_{\vMF}[\PPP(m)(F_v)]).$$
\end{mydef}
\begin{lem}\label{positivityofpd}
Assuming the conditions of \ref{peyredisc}, the following claims are valid
\begin{enumerate}
\item One has that $$\lim_{s\to 1^+}(s-1)^{r-1}\wH(s,1)=\lim_{s\to 1^{+}}(s-1)^{r-1}\int_{[\TT(m)(\AAF)]}H^{-s}\mu_{\AAF}>0.$$
\item One has that $\tau>0$.
\end{enumerate}
\end{lem}
\begin{proof}
\begin{enumerate}
\item Recall that $\wH_v(s,1)=m$ for every $\vMFi$ by \ref{wustani}. Now, by \ref{disckchi}, one has that \begin{align*}
\wH(s,1)=\gamma(s,1)m^{r_1+r_2}\prod _{j=1}^{r-1}L(s,1)=\gamma(s,1)m^{r_1+r_2}\zeta(s)^{r-1},
\end{align*}
where $\gamma(-,1)$ is a holomorphic function in the domain $\Omega_{>\frac{\alpha(m)}{\alpha(m)+1}}$ and $\gamma(1,1)>0$. Thus \begin{align*}\lim_{s\to 1^{+}}(s-1)^{r-1}\wH(s,1)&=\gamma(1,1)m^{r_1+r_2}\lim_{s\to 1^+}(s-1)^{r-1}\zeta(s)^{r-1}\\&=\gamma(1,1)m^{r_1+r_2}\Res(\zeta,1)^{r-1}\\&>0.\end{align*}
\item Let $K\subset K^0_{\max}$ be an open subgroup such that $H=H((f_v)_v)$ is~$K$-invariant.  By \ref{lemakojunerazumi} one has that 
\begin{align*}
\omega\big(\prod_{\vMF}[\PPP(m)(F_v)]\big)&=|\mu_m(F)|\lim _{s\to 1^{+}}(s-1)^{r-1}\int_{\Xi_{\infty}^{\perp}}H^{-s}\mu_{\infty}\\
&=|\mu_m(F)|\lim_{s\to 1^{+}}(s-1)^{r-1}|\Xi_K|\int_{\Xi_K^{\perp}}H^{-s}\mu_{\AAF}.
\end{align*}
Let $\{x_1\dots x_{|\Xi_K|}\}$ be a set of elements of $[\TT(m)(\AAF)]$ such that for any $i\neq j\in\{1\doots |\Xi_K|\}$, one has that $x_ix_j^{-1}\not\in\Xi^{\perp}_K$. Using \ref{doslema}, we obtain that for $s>1$, one has that\begin{align*}\int_{[\TT(m)(\AAF)]}H^{-s}\mu_{\AAF}&=\sum_i\int_{\Xi_K^{\perp}}H(x_iy)^{-s}d\mu_{\AAF}(y)\\
&\leq\sum_iC(x_i)^{-s}\int_{\Xi_K^\perp}H^{-s}\mu_{\AAF},
\end{align*} for certain $C(x_i)>0.$ 
 It follows that \begin{align*}0&<\lim_{s\to 1^+}(s-1)^{r-1}\int_{[\TT(m)(\AAF)]}H^{-s}\mu_{\AAF}\\&=\lim_{s\to 1^+}(s-1)^{r-1}\sum_iC(x_i)^{-s}\int_{\Xi_K^{\perp}}H^{-s}\mu_{\AAF}
 \\&\leq \big(\sum_i C(x_i)^{-1}\big)\lim_{s\to 1^+}(s-1)^{r-1}\int_{\Xi_K^{\perp}}H^{-s}\mu_{\AAF},\end{align*}
 and hence that $\lim_{s\to 1^+}(s-1)^{r-1}\int_{\Xi_K^{\perp}}H^{-s}\mu_{\AAF}>0$. We deduce that  $$\tau=|\mu_m(F)|\lim_{s\to 1^+}(s-1)^{r-1}|\Xi_K|\int_{\Xi_K^{\perp}}H^{-s}\mu_{\AAF}>0,$$ as claimed.
 \end{enumerate}
\end{proof}
\begin{thm} \label{principaldisc}
There exists $\delta>0$, such that~$Z$ extends to a meromorphic function in the domain $s\in\Omega_{>1-\delta}$ with the only pole in this domain at~$1$ which is of order $r-1$ and such that for every compact $\mathcal K\subset\RR_{>1-\delta}$  one has that there exists $C(\mathcal K)>0$ such that $$\bigg|\bigg(\frac{s-1}{s}\bigg)^{r-1}Z(s)\bigg|\leq C(\mathcal K)(1+|\Im(s)|)^{r-1}$$if provided $\Re(s)\in\RR_{>1-\delta}$. The principal value at the pole $s=1$ is equal to $$\frac{\tau}{m}.$$
\end{thm}
\begin{proof}
By (\ref{hxhttmx}), one has that $H(x)=H([\TT(m)(i)](x))$ and by \ref{cesusp} the group $\ker([\TT(m)(i)])=\Sh^1(F,\mu_m)$ is finite. It follows that formally one has that\begin{equation}\label{zajedanzs}Z(s)=|\Sh^1(F,\mu_m)|\sum_{x\in[\TT(m)(i)]([\TT(m)(F)])}H(x)^{-s}.\end{equation}By \ref{zconvunifd} for every $\epsilon>0$, the series defining $Z(s)$ converges absolutely and uniformly to a holomorphic function in the domain $\Omega_{>\alpha(m)+\epsilon}$ and it follows from \ref{whdiscam} that for $s\in\Omega_{>1}$ the sum on the right hand side converges and is a holomorphic function in $s$ in this domain. Thus, the equality (\ref{zajedanzs}) is valid in $\Omega_{>1}$ as an equality of holomorphic functions. We apply Poisson formula (\ref{formuledepoison}) to the inclusion $$[\TT(m)(i)]([\TT(m)(F)])\subset [\TT(m)(\AAF)]$$(we have already verified the conditions (2) and (3) of Proposition \ref{formuledepoison} in \ref{zconvunifd}). We have formally \begin{multline*}\sum_{x\in[\TT(m)(i)]([\TT(m)(F)])}H(x)^{-s}\\=\int_{([\TT(m)(i)]([\TT(m)(F)]))^{\perp}}\wH(s,\chi) (\mu_{\AAF}/\coun_{[\TT(m)(i)]([\TT(m)(F)])})^*.\end{multline*}
We use \ref{calculdualmeas} to understand the measure $(\mu_{\AAF}/\coun_{[\TT(m)(i)]([\TT(m)(F)])})^*$. A volume of a subset of $(\RR_{>0})_m=\Rgz$ when $(\RR_{>0})_m$ is endowed with the pushforward measure of the measure $d^*r$ for the map $\Rgz\to\Rgz, x\mapsto mx$ is $1/m$ times it was for the measure $d^*r$. Thus the Haar measure $d^*r_m$ from \ref{calculdualmeas} is normalized by $(d^*r_m)(\RR_{>0}/(\RR_{>0})_m)=m$. Hence, the dual measure of $d^*r_m$ satisfies that $(d^*r_m)^*((\RR_{>0}/(\RR_{>0})_m)^*)=\frac{1}m$. Now, Lemma \ref{calculdualmeas} gives that \begin{multline*}(\mu_{\AAF}/\coun_{[\TT(m)(i)]([\TT(m)(F)])})^*\\=\frac{|\mu_m(F)|}{m |\Sh^1(F, \mu_{m})|}\coun_{([\TT(m)(\AAF)]_1/[\TT(m)(i)]([\TT(m)(F)]))^*}.\end{multline*} Whenever $\chi\not\in\AK$, we have by \ref{disckchi} that $\wH(s,\chi)=0$. We deduce that formally one has\begin{equation}\label{zdiscchi}Z(s)=\frac{|\Sh^1(F,\mu_m)|\cdot |\mu_m(F)|}{m|\Sh^1(F,\mu_m)|}\sum_{\chi\in\AK}\wH(s,\chi)=\frac{|\mu_m(F)|}m\sum_{\chi\in\AK}\wH(s,\chi).\end{equation} 
For every $\chi\in\AK$, by \ref{whdiscam} one has that $s\mapsto\wH(s,\chi)$ is a holomorphic function in the domain $\Omega_{>1}$. As the group $\AK$ is finite, we deduce that $s\mapsto\sum_{\chi\in\AK}\wH(s,\chi)$ is a holomorphic function in the domain $\Omega_{>1}.$ It follows that (\ref{zdiscchi}) is valid as an equality of holomorphic functions in the domain $\Omega_{>1}$. Moreover, (\ref{zdiscchi}) is valid as an equality of the maximal meromorphic extensions of the functions from the both hand sides. Moreover, \ref{whdiscam} gives that there exists $\delta>0$ such that $$\bigg(\frac{s-1}{s}\bigg)^{d(\chi)}\wH(s,\chi)$$extends to a holomorphic function in the domain $\Omega_{>1-\delta}$ and that for every compact $\mathcal K\subset \RR_{>1-\delta}$ there exists $C'(\mathcal K)>0$, such that $$\bigg|\bigg(\frac{s-1}{s}\bigg)^{d(\chi)}\wH(s,\chi)\bigg|\leq C'(\mathcal K)(1+|\Im(s)|)$$for every $\chi\in\AK$ (here $d(\chi)=r-1$ if $\chi^{m/r}=1$ otherwise $d(\chi)=0$). By the finiteness of $\AK$ and the fact that $d(\chi)\leq r-1$, we deduce that $$\sum_{\chi\in\AK}\bigg(\frac{s-1}{s}\bigg)^{r-1}\wH(s,\chi)=\bigg(\frac{s-1}{s}\bigg)^{r-1}\sum_{\chi\in\AK}\wH(s,\chi)$$extends to a holomorphic function in the domain $\Omega_{>1-\delta}$ and such that for every compact $\mathcal K\subset \RR_{>1-\delta}$ there exists $C(\mathcal K)>0$, such that $$\bigg|\bigg(\frac{s-1}{s}\bigg)^{r-1}\sum_{\chi\in\AK}\wH(s,\chi)\bigg|\leq C(\mathcal K)(1+|\Im(s)|).$$
We deduce that~$Z$ extends to a meromorphic function in the domain $s\in\Omega_{>1-\delta}$ with the only possible pole at $s=1$ in this domain which is of order at most $r-1$. Moreover, for every compact $\mathcal K\subset\RR_{>1-\delta}$  one has that $$\bigg|\bigg(\frac{s-1}{s}\bigg)^{r-1}Z(s)\bigg|\leq C(\mathcal K)(1+|\Im(s)|)^{r-1}$$if provided $\Re(s)\in\mathcal K$.

The last part of the proof we dedicate to the proving that~$Z$ indeed has a pole at $1,$ that this pole is of order exactly $r-1$ and to the calculation of the principal value. We calculate the limit $$\lim_{s\to{1^+}}\bigg(\frac{s-1}{s}\bigg)^{r-1}\sum_{\chi\in\AK}\wH(s,\chi)=\sum_{\chi\in\AK}\lim_{s\to1^+}\bigg(\frac{s-1}{s}\bigg)^{r-1}\wH(s,\chi).$$ Recall that by $\Xi_K$ we have denoted the subgroup of $\chi\in\AK$ such that $\chi^{m/r}=1$. If $\chi\in\AK-\Xi_K,$ then by \ref{whdiscam}, one has that $s\mapsto\wH(s,\chi)$ is holomorphic in the domain $\Omega_{>1-\delta}$, and thus $$\lim_{s\to 1^+}\bigg(\frac{s-1}{s}\bigg)^{r-1}\wH(s,\chi)=0.$$ 
 We deduce that \begin{align*}\lim_{s\to 1^+}\bigg(\frac{s-1}{s}\bigg)^{r-1}\sum_{\chi\in\AK}\wH(s,\chi)&=\sum_{\chi\in\Xi_K}\lim_{s\to 1^+}\bigg(\frac{s-1}{s}\bigg)^{r-1}\wH(s,\chi)\end{align*}
Using \ref{lemakojunerazumi}, we have that \begin{equation}\label{xiproblem}\sum_{\chi\in\Xi_K}\wH(s,\chi)=|\Xi_K|\int_{\Xi_K^{\perp}}H^{-s}\mu_{\AAF}=\int_{\Xi_{\infty}^{\perp}}H^{-s}\mu_{\infty}^{\perp}\end{equation}whenever the quantities on both hand sides converge. By the fact that $H^{-s}$ is absolutely integrable over $[\TT(m)(\AAF)]$ for $s\in\Omega_{>1}$ and by \ref{lemakojunerazumi}, we deduce that (\ref{xiproblem}) is valid in the domain $\Omega_{>1}$. Moreover, (\ref{xiproblem}) is valid, as an equality of the meromorphic functions in the domain $\Omega_{>1-\delta}$. We deduce that $$\lim_{s\to 1^+}\bigg(\frac{s-1}s\bigg)^{r-1}\int_{\Xi^{\perp}_{\infty}}H^{-s}\mu^{\perp}_\infty.$$ We recognize this quantity from \ref{peyredisc} as $\frac{\omega(\prod_{\vMF}[\PPP(m)(F_v)])}{|\mu_m(F)|},$ and hence \begin{align*}\lim_{s\to1^{+}}\bigg(\frac{s-1}{s}\bigg)^rZ(s)&=\lim_{s\to 1^{+}}(s-1)^{r-1}Z(s)\\&=\frac{|\mu_m(F)|}m\sum_{\chi\in\AK}\wH(s,\chi)\\&=\frac{|\mu_m(F)|\cdot\omega(\prod_{\vMF}[\PPP(m)(F_v)])}{|\mu_m(F)| m}\\&=\frac{\omega(\prod_{\vMF}[\PPP(m)(F_v)])}m\\&=\frac{\tau}m.\end{align*}
The statement is proven.
\end{proof}
In \ref{principaldisc}, we have verified the conditions of the Tauberian result \cite[Theorem A1]{FonctionsZ}. We deduce that:
\begin{cor}\label{countingquasidisc}
Let $(f_v:\Fvt\to\RR_{>0})_v$ be a quasi-discriminant degree~$m$ family of~$m$-homogenous functions and let~$H$ be the resulting height on $[\PPP^m(F)]$. One has that $$|\{x\in[\PPP(m)(F)]|H(x)\leq B\}|\sim \frac{\tau}{(r-2)!\cdot m}B\log(B)^{r-2},$$when $B\to\infty$.
\end{cor}
If for every~$v$ one has that $f_v=f_v^{\Delta}$, where $f_v^{\Delta}$ is the discriminant~$m$-homogenous functions of weighted degree~$m$, then by \ref{discvsheight}, we get for $y\in[\PPP(m)(F)]$ that $$H^{\Delta}(y)=N\bigg(\Delta\big(F[X]/(X^m-\widetilde y))/F\big)\bigg)^{m/\alpha(m)},$$ where $\widetilde y$ is a lift of~$y$. Let us write $|\Delta|(y)$ for $H^{\Delta}(y)^{\alpha(m)/m}$. It is precisely the norm of the discriminant of a torsor corresponding to~$y$. We deduce that:
\begin{cor}
One has that \begin{multline*}|\{x\in[\PPP(m)(F)]|\hspace{0.1cm}|\Delta|(x)\leq B\}|\\\sim_{B\to\infty} \frac{r^{r-2}\tau^{\Delta}}{m^{r-1}(r-1)^{r-2}\cdot(r-2)!}B^{\frac{m}{\alpha(m)}}\log(B)^{r-2},\end{multline*}
where $\tau^{\Delta}=\tau((f^{\Delta}_v)_v).$ 
\end{cor}
\subsection{} In this paragraph we explain the equidistribution of rational points in $\prod_{\vMF}[\PPP(m)(F_v)]$. We will write~$i$ for the map $[\PPP(m)(i)]$. 
\begin{thm}\label{equidisc} The set $i([\PPP(m)(F)])$ is equidistributed in $\prod_{\vMF}[\PPP(m)(F_v)]$ with respect to~$H$.
\end{thm} 
\begin{proof}
The proof follows the proof of \ref{quasitoricequi} with some modifications and simplifications. Corollary \ref{countingquasidisc}, together with the fact that $\ker(i)=\Sh^1(F, \mu_m)$ from \ref{cesusp}, gives that \begin{multline*}|\{x\in i([\PPP(m)(F)])| H(x)\leq B\}|\\\sim_{B\to\infty}\frac{\omega(\prod_{\vMF}[\PPP(m)(F_v)])}{(r-2)! m |\Sh^1(F, \mu_m)|}B\log(B)^{r-2}.\end{multline*}

We say that an open subset $W$ is elementary if $W$ writes as $W=\prod_{\vMF}W_v,$ where $W_v\subset [\PPP(m)(F)]$ is open, and for almost all $v,$ one has that $W_v=[\PPP(m)(F_v)]$. As for every~$v$, the spaces $[\PPP(m)(F_v)]$ are finite and discrete by \ref{ttafvfinite}, we deduce that $W$ is open and closed, hence $\partial W=\emptyset.$ We prove that for every elementary $W$, one has that $$\frac{|x\in W| H(x)\leq B\}|}{|x\in i([\PPP(m)(F)]| H(x)\leq B\}|}\sim_{B\to\infty}\frac{\omega(W)}{\omega(\prod_{\vMF}[\PPP(m)(F_v)])}.$$For $\vMF$, we define $g_v=\mathbf 1_{W_v}=h_v$. Let $\epsilon>0$ and set $\eta=\epsilon/4.$ For $\vMF$, we set $g_{\eta, v}=(1-\eta)\mathbf 1_{W_v}+\eta$ and $h_{\eta, v}=g_{\eta, v}$. The proof that $H((f_v\cdot (g^{-1}_v\circ q^m_v))_v)=g_\eta^{-1}H$ is identical to the proof of the corresponding claim in  Part (2) proof of \ref{quasitoricequi}. Let us establish that $$\omega((f_v\cdot (g_{\eta, v}^{-1}\circ q^m_v))_v)=g_{\eta}^{-1} \omega.$$ For a compactly supported continuous function $\phi:\prod_{\vMF}[\PPP(m)(\Fv)]\to\CC$ we say that it is decomposable, if it can be written as $\otimes_{\vMF}\phi_v,$ where for almost all~$v$ one has that $\phi_v=1$. 
Let $\phi:[\PPP(m)(F_v)]\to \RR_{\geq 0}$ be decomposable. One has that \begin{align*}\omega((f_v\cdot (g_{\eta, v}^{-1}\circ q^m_v))_v)(\phi)\hskip-1cm&\\&=\lim_{s\to 1^+}(s-1)^{r-1}\int_{\Xi_\infty^{\perp}}(\phi) H(((f_v\cdot (g_{\eta, v}^{-1}\circ q^m_v))_v))^{-s}\mu_{\infty}^{\perp}\\
&=\lim_{s\to 1^+}(s-1)^{r-1}\int_{\Xi_{\infty}^{\perp}}(\phi g_{\eta}^{-s})H^{-s}\mu_{\infty}^{\perp}\\
&\geq\lim_{s\to 1^+}(s-1)^{r-1}\int_{\Xi_{\infty}^{\perp}}(\phi g_{\eta}^{-1})H^{-s}\mu_{\infty}^{\perp}\\
&=\omega(\phi g^{-1}_{\eta}),
\end{align*} where the only inequality follows from the fact that $g_{\eta}^{-s}\geq g_{\eta}^{-1}$ (which is true because $g_\eta$ takes values in the interval $]0,1[$). On the other side, for every $\delta>1$, by taking in the limit only those $s$ contained in the domain $]1,\delta[$, we deduce that \begin{align*}\omega((f_v\cdot (g_{\eta, v}^{-1}\circ q^m_v))_v)(\phi)&\leq \lim_{s\to 1^+}(s-1)^{r-1}\int_{\Xi_{\infty}^{\perp}}\phi g_{\eta}^{-\delta}H^{-s}\mu_{\infty}^{\perp}\\
&=\omega(\phi g^{-\delta}_{\eta}).\end{align*}
From the fact that~$\omega$ is a measure, it follows that $\lim_{\delta\to 1^{+}}\omega(\phi g^{-\delta})=\omega(\phi g^{-1}),$ and hence that \begin{equation}\label{korisnomega}\omega((f_v\cdot (g_{\eta, v}^{-1}\circ q^m_v))_v)(\phi)=\omega(\phi g_{\eta}^{-1})\end{equation} for any  non-negative decomposable function $\phi$. Clearly, any real valued decomposable function is a difference of two non-negative decomposable functions, and the equality (\ref{korisnomega}) is hence valid for all real valued decomposable functions. Any decomposable $\phi$ writes as $\phi_1+i\phi_2,$ for some decomposable real valued functions $\phi_1$ and $\phi_2$. The equality (\ref{korisnomega}) is thus valid for any decomposable $\phi$. Moreover, (\ref{korisnomega}) is valid for finite sums of decomposable functions. The finite sums of decomposable continuous compactly supported functions are dense in the set $\mathscr C_c^0(\prod_{\vMF}[\PPP(m)(F_v)], \CC)$ by \cite[Lemma 3, \no 5, \S4, Chapter III]{Integrationj}, hence $\omega((f_v\cdot (g_{\eta, v}^{-1}\circ q^m_v))_v)=g^{-1}_{\eta}\omega$. To prove the claim for the elementary open subset $W$, we use the same steps as in the part (2) of \ref{quasitoricequi}, with the only change in the function of $B$ (that is replace $\frac{\omega(\prod_{\vMF}[\PPP(\aaa)(F_v)])}{|\Sh^1(F, \mu_{\gcd(\aaa)})|\cdot |\aaa|}B$ with $\frac{\omega(\prod_{\vMF}[\PPP(\aaa)(F_v)])}{(r-2)! m |\Sh^1(F, \mu_m)|}B\log(B)^{r-2}$). 

We have hence established the claim for every elementary subset of $[\PPP(m)(F)]$. The rest of the proof is identical to the proofs of parts (3) and (4) in the proof of \ref{quasitoricequi}.
\end{proof}
\subsection{}In the last paragraph, we prove that $\mu_m$-torsors which are fields are of positive proportion among all $\mu_m$ torsors of bounded quasi-discriminant height.
\begin{prop}\label{notgreataboutfields}
There exists $C(m, (f_v)_v)>0$ such that $$|\{x\in[\PPP(m)(F)]|\text{\normalfont~$x$ is a field, } H(x)\leq B\}|\sim_{B\to\infty}C(m, (f_v)_v)B\log(B)^{r-2}.$$
\end{prop}
\begin{proof}
Let $w$ be a finite place of~$F$ which does not extend the place $2$ of $\QQ$. Let us denote by $w_m$ the canonical map $$[\TT(m)(F_v)]\to\ZZ/m\ZZ,$$given in \ref{indofov}. Recall that if $\widetilde y\in F_w^{\times}$ is a lift of $y\in[\TT(m)(F_w)]$, then the image of $w(\widetilde y)$ under the quotient map $\ZZ\to \ZZ/m\ZZ$ is $w_m(y)$. We let $$W=\{y\in[\PPP(m)(F_w)]:w_m(y)=1\}\times\prod_{\vMF-\{w\}}[\PPP(m)(F_v)].$$ The set $\{y\in[\PPP(m)(F_v)]|w_m(y)=1\}$ is open and closed in the finite discrete set $[\PPP(m)(F_v)]$, thus $W$ is open and closed in $\prod_{\vMF}[\PPP(m)(F_v)]$. Hence, $\partial W=\emptyset.$ We prove two claims that will imply the statement of the proposition.
\begin{enumerate}
\item Let us prove that if $x\in [\PPP(m)(F)]$ satisfies that $i(x)\in W,$ then~$x$ is a field.  One has that~$x$ is field if and only if $X^m-\widetilde x$ is irreducible, where $\widetilde x\in F^{\times}$ is a lift of~$x$. By \cite[Theorem 9.1, Chapter VI]{Lang}, this is true if and only if $\widetilde x$ is not an element of $(F^{\times})_p$ for~$p$ prime divisor of~$m$ and if $4|m$, then also $\widetilde x\not \in -4(F^\times)_4$. 
Let us write $i_v:[\PPP(m)(F)]=[\TT(m)(F)]\to [\TT(m)(F_v)]=[\PPP(m)(F_v)]$ for the map induced from $(\Fvt)_m$-invariant map $F^{\times}\to\Fvt\xrightarrow{q^m}[\TT(m)(F_v)]$. For every prime $p|m$, one has that $p\nmid 1=w_m(i_v(x))=w(\widetilde x),$ hence $\widetilde x\not\in (F^{\times})_p.$ One has that $2\nmid 1=w_m(i_v(x))=w(\widetilde x),$ thus $\widetilde x\not\in (-4F^{\times})_4$. It follows that~$x$ is a field.
\item Let us prove that $W$ has a strictly positive volume. By \ref{wustani}, one has that $$\wH_w(s,1)=\sum_{j=0}^{m-1}\pivv^{\frac{s(m^2-m\gcd(j,m))}{\alpha(m)}},$$for $s>0$. Hence, $\wH_w(-,1)$ does not vanish for $s>0$. Using \ref{positivityofpd}, we deduce that
\begin{align*}
\lim_{s\to 1^+}(s-1)^{r-1}\int_{\substack{[\TT(m)(\AAF)]\\w_m=1}}H^{-s}\mu_{\AAF}&=\lim_{s\to 1^{+}}(s-1)^{r-1}\wH(s,1)\frac{\int_{w_m=1}H_w^{-s}\mu_w}{\wH_w(s,1)}\\
&=\lim_{s\to 1^{+}}(s-1)^{r-1}\wH(s,1)\frac{\int_{w_m=1}H^{-1}\mu_w}{\wH_w(s,1)}
\end{align*} is positive. By \ref{lemakojunerazumi} one has that 
\begin{align*}
\omega(W)&=|\mu_m(F)|\lim _{s\to 1^{+}}(s-1)^{r-1}\int_{\substack{\Xi_{\infty}^{\perp}\\w_m=1}}H^{-s}\mu_{\infty}\\
&=|\mu_m(F)|\lim_{s\to 1^{+}}(s-1)^{r-1}\mathopen|\Xi_K\mathclose|\int_{\substack{\Xi_K^{\perp}\\w_m=1}}H^{-s}\mu_{\AAF}.
\end{align*}
Let $\{x_1\dots x_{|\Xi_K|}\}$ be a set of representatives of classes of $\Xi_K$ in $[\TT(m)(\AAF)].$ 
Using \ref{doslema}, we obtain that for $s>1$, one has that\begin{align*}\int_{\substack{[\TT(m)(\AAF)]\\w_m=1}}H^{-s}\mu_{\AAF}&=\sum_i\int_{\substack{\Xi_K^{\perp}\\w_m=1}}H(x_iy)^{-s}d\mu_{\AAF}(y)\\
&\leq\sum_iC(x_i)^{-s}\int_{\substack{\Xi_K^\perp\\w_m=1}}H^{-s}\mu_{\AAF},
\end{align*} for certain $C(x_i)>0.$ 
 It follows that \begin{align*}0&<\lim_{s\to 1^+}(s-1)^{r-1}\int_{\substack{[\TT(m)(\AAF)]\\w_m=1}}H^{-s}\mu_{\AAF}\\&=\lim_{s\to 1^+}(s-1)^{r-1}\sum_iC(x_i)^{-s}\int_{\substack{\Xi_K^{\perp}\\w_m=1}}H^{-s}\mu_{\AAF}
 \\&\leq \big(\sum_i C(x_i)^{-1}\big)\lim_{s\to 1^+}(s-1)^{r-1}\int_{\substack{\Xi_K^{\perp}\\w_m=1}}H^{-s}\mu_{\AAF},\end{align*}
 and hence that $\lim_{s\to 1^+}(s-1)^{r-1}\int_{\substack{\Xi_K^{\perp}\\w_m=1}}H^{-s}\mu_{\AAF}>0$. We deduce that  $$\omega(W)=|\mu_m(F)|\lim_{s\to 1^+}(s-1)^{r-1}\card{\Xi_K}\int_{\substack{\Xi_K^{\perp}\\w_m=1}}H^{-s}\mu_{\AAF}>0,$$ as claimed.
\end{enumerate}
Now, as $\omega(\partial W)=0$, by \ref{equidisc}, when $B\to \infty$, one has that \begin{align*}|\{x\in[\PPP(m)(F)]|\text{\normalfont~$x$ is a field, } H(x)\leq B\}|\hskip-1cm&\\&\geq|\{x\in[\PPP(m)(F)]|i(x)\in W\}|\\&\sim_{B\to\infty}\frac{\omega(W)}{m}B\log(B)^{r-2}.
\end{align*} The statement follows.  
\end{proof}
Let us give the exact formula for $C(m, ((f_v)_v))$, under the condition that $4\nmid m$. For a divisor $d|m$, we denote by $j^d$ the homomorphism $$j^d:[\TT(d)(F)]\to[\TT({m})(F)]$$ induced from $(F^{\times})_d$-invariant homomorphism $$\Ft\to\Ft\to[\TTm(F)]\hspace{1cm}y\mapsto [q^{m}(F)](y^{m/d}),$$where $[q^{m}(F)]:F^{\times}\to[\TTm(F)]$ is the quotient map. 
The map $j^d$ induces an isomorphism $[\TT(d)(F)]\to[\TTm(F)]_{d/m}$, where, as usual $[\TTm(F)]_d$ denotes the subgroup given by~$d$-th powers of the elements of $[\TTm(F)]$. For $v\in M_F$ and $d|m$ we set $$f^{*d}_v=\bigg(\big(y\mapsto |y|_v^{1-\frac{m}{d}}\big)\cdot (f_v\circ (x\mapsto x^{\frac md}))\bigg)^{\big(1-\frac1r\big)/\big(1-\frac1q\big)},$$where $q$ is the smallest prime of~$d$.
\begin{lem}\label{kljucnazakraj} Let~$d$ be a divisor of~$m$ and let $q$ be the smallest prime of~$d$. The following claims are valid:
\begin{enumerate}
\item The family $(f^{*d}_v)_v$ is a quasi-discriminant degree~$d$ family of~$d$-homogenous functions.
\item One has equality of functions $[\TTd(F)]\to\RR_{\geq 0}$: $$H((f_v)_v)\circ j^d=\big(H((f^{*d}_v)_v)\big)^{(1-\frac1q)/(1-\frac1r)}.$$
\end{enumerate}
\end{lem}
\begin{proof}
\begin{enumerate}
\item Let $t,x\in\Fvt.$ For $\vMF$, one has that \begin{align*}
f^{*d}_v(t\cdot x)=f^{*d}_v(t^dx)=|t^dx|_v^{1-\frac {m}d} f_v(t^mx^{\frac md})&=|t^d|^{1-\frac {m}d}_v|x|_v^{1-\frac{m}d}|t|_v^mf_v(x^{\frac md})\\&=|t|_v^{d}f^{*d}_v(x).
\end{align*}
It follows that $f^{*d}_v$ is~$d$-homogenous of weighted degree~$d$. For almost all~$v$, the function $f_v$ is the discriminant~$m$-homogenous function of the weighted degree~$m$. Recall from \ref{defvdisc} that for~$v$ finite such that $v(m)=0,$ this means that $$f_v(y)=|y|_v\pivv^{\frac{m\gcd(v(y),m)-m^2}{\alpha(m)}}$$ for every $y\in\Fvt.$ The direct calculation gives that
\begin{align*}
\big(f^{*d}_v(y)\big)^{(1-\frac1q)/(1-\frac1r)}=|y|_v^{1-\frac{m}d}f_v(y^{\frac md})&=|y|_v^{1-\frac{m}d}\cdot |y|_v^{\frac md}\pivv^{\frac{m\gcd((m/d)v(y),m)-m^2}{\alpha(m)}}\\
&=|y|_v\pivv^{\frac{(m/d)m\gcd(v(y),d)-m^2}{m^2(1-\frac1r)}}\\
&=|y|_v\pivv^{\frac{(1/d)\gcd(v(y),d)-1}{1-\frac1r}}\\
&=|y|_v\pivv^{\frac{d\gcd(v(y),d)-d^2}{d^2(1-\frac1r)}}\\
&=|y|_v\pivv^{\frac{d\gcd(v(y),d)-d^2}{\alpha(d)}\cdot\frac{1-\frac1q}{1-\frac1r}},
\end{align*}
i.e. $$f^{*d}_v(y)=|y|_v\pivv^{\frac{d\gcd(v(y),d)-d^2}{\alpha(d)}}.$$ In other words $f^{*d}_v$ is the discriminant~$d$-homogenous function of the weighted degree~$d$.
It follows that $(f^{*d}_v)_v$ is a quasi-discriminant degree~$d$ family of~$d$-homogenous functions.
\item 
Let $x\in[\TTd(F)]$ and let $\widetilde x\in F^{\times}$ be its lift.  
One has that $\widetilde x^{\frac md}\in F^{\times}$ is a lift of $j^d(x)\in[\TTm(F)]$. Hence,
\begin{align*}
H((f_v)_v)(j^d(x))&=\prod_{\vMF}f_v(\widetilde x^{\frac md})\\
&=\bigg(\prod_{\vMF}|\widetilde x|_v^{1-\frac{m}{d}}\bigg)\cdot \prod_{\vMF}f_v(\widetilde {x}^{\frac md})\\
&=\prod_{\vMF}|\widetilde x|_v^{1-\frac{m}{d}}f_v(\widetilde x^{\frac md})\\
&=\prod_{\vMF}(f^{*d}_v(\widetilde x))^{(1-\frac1q)/(1-\frac1r)}\\
&=(H((f^{*d}_v))(x))^{(1-\frac1q)/(1-\frac1r)}.
\end{align*}
We deduce that $$H((f_v)_v)\circ j^d=\big(H((f^{*d}_v)_v)\big)^{(1-\frac1q)/(1-\frac1r)},$$as claimed.
\end{enumerate}
\end{proof}
\begin{thm}\label{licava}
Suppose that $4\nmid m$ or that $i=\sqrt{-1}\in F$. One has that $$C(m,((f_v)_v))=\frac{\bigg(\sum_{\substack {d|m\\r|d}}{d\cdot\mu(d)\cdot\tau((f^{*(m/d)}_v)_v)}\bigg)}{(r-2)!m},$$where~$\mu$ stands for the M\"obius function (here, the sum is taken over divisors~$d$ of~$m$, which are divisible by the prime $r$).
\end{thm}
\begin{proof}
We introduce notation $$[\TTm(F)]^0=\{x\in [\TTm(F)]|\text{$x$ is a field}\}.$$
For $x\in[\TTm(F)],$ let $\widetilde x\in F^{\times}$ be its lift. Suppose for instant that $4\nmid m$. It follows from \cite[Theorem 9.1, Chapter VI]{Lang} that $x\in[\TTm(F)]$ is a field if and only if $\widetilde x$ is not an element of $(F^{\times})_p=\{y^p|y\in F^{\times}\}$ for primes $p|m$. Suppose now that $4|m$, by the hypothesis $i\in F$, hence $-4(F^{\times})_4\subset (F^{\times})_2$. Therefore the same conclusion of \cite[Theorem 9.1, Chapter VI]{Lang} applies. 

We deduce that~$x$ is a field if and only if $x\not\in[\TTm(F)]_p=\{y^p|y\in[\TTm(F)]\}$ for primes $p|m,$ i.e. we can write $$[\TTm(F)]^0=[\TTm(F)]-\bigcup_{\substack{p\text{ prime}\\p|m}}[\TTm(F)]_p.$$We are going to use the inclusion-exclusion principle. For that purpose, we verify that for positive integers $k,\ell$ such that $\gcd(k,\ell)=1,$ one has that $$[\TTm(F)]_k\cap[\TTm(F)]_{\ell}=[\TTm(F)]_{k\ell}.$$ Indeed, the inclusion ``$\supset$" is clear and let us prove the reverse inclusion. Write $u_1k+u_2\ell=1$. If $y\in[\TTm(F)]_k\cap[\TTm(F)]_{\ell},$ then there exists $y'$ and $y''$ such that $y=(y')^k$ and $y=(y'')^{\ell}$. Thus $$y=y^{u_1k+u_2\ell}=y^{u_1k}y^{u_2\ell}=(y'')^{\ell u_1k}(y')^{k u_2\ell}=((y'')^{u_1}(y')^{u_2})^{k\ell},$$and the claim is verified. We deduce that for $B>0$ one has that:
\begin{multline*}
\big|\big\{x\in\bigcup_{\substack{p\text{ prime}\\p|m}}[\TTm(F)]_p|H(x)\leq B\big\}\big|\\=\sum _{\substack{j\geq 1\\p_1<\cdots <p_j\text { primes of~$m$}}}(-1)^{j+1}|\{x\in[\TTm(F)]_{p_1\cdots p_j}|H(x)\leq B\}|,
\end{multline*}
and thus that
\begin{align*}
\big|\big\{x\in[\TTm(F)]^0|H(x)\leq B\big\}\big|\hskip-5cm&\\
&=\big|\big\{x\in[\TTm(F)]|H(x)\leq B\big\}\big|\\&\hspace{0.5cm}-\sum _{\substack{j\geq 1\\p_1<\cdots <p_j\text { primes of~$m$}}}(-1)^{j+1}|\{x\in[\TTm(F)]_{p_1\cdots p_j}|H(x)\leq B\}|.
\end{align*}
Using the M\"obius function~$\mu$, we write the last equality as
\begin{multline*}\big|\big\{x\in[\TTm(F)]^0|H(x)\leq B\big\}\big|\\=\sum_{d|m}\mu(d)\big|\{x\in[\TTm(F)]_{d}|H(x)\leq B\}\big|.\end{multline*}
We write $r(k)$ for the smallest prime of an integer~$k$. Lemma \ref{kljucnazakraj} for $d|m$ gives that \begin{align*}\{x\in[\TTm(F)]_{d}|H(x)\leq B\}\hskip-4cm&\\&=\{x\in[\TT(m/d)(F)]|(H((f^{*(m/d)}_v)_v)(j^{\frac md}(x)))^{(1-\frac1{r(m/d)})/(1-\frac1r)}\leq B\}\\
&=\{x\in[\TT(m/d)(F)]|(H((f^{*(m/d)}_v)_v)(j^{\frac md}(x))))\leq B^{(1-\frac1r)/(1-\frac1{r(m/d)})}\}
\end{align*}
Now, it follows from \ref{countingquasidisc} that \begin{multline*}|\{x\in[\TTm(F)]_{d}|H((f^{*(m/d)}_v)_v)(x)\leq B\}|\\\sim_{B\to\infty} \frac{\tau((f^{*(m/d)}_v)_v)}{({r(m/d)}-2)!(m/d)}B^{(1-\frac1r)/(1-\frac1{r(m/d)})}\log(B^{(1-\frac1r)/(1-\frac1{r(m/d)})})^{{r(m/d)}-2},\end{multline*}where $\tau((f^{*(m/d)}_v)_v)=\omega((f^{*(m/d)}_v)_v)(\prod_{\vMF}[\PPP(m/d)(F_v)])$. Thus, for a divisor $d|m$ for which ${r(m/d)}>r$, the term $|\{x\in[\TTm(F)]_{d}|H(x)\leq B\}|$ does not influence the leading constant. We deduce that
\begin{multline*}|\{x\in[\TTm(F)]^0|H(x)\leq B\}|\\=\frac{\bigg(\sum_{\substack {d|m\\r|d}}{d\cdot\mu(d)\cdot\tau((f^{*(m/d)}_v)_v)}\bigg)}{(r-2)!m}B\log(B)^{r-2}.\end{multline*}
The theorem has been proven.
\end{proof}
\begin{rem}
\normalfont
When~$m$ is not a prime the proof of \ref{licava} gives that there exists a positive proportion of $\mu_m$-torsors, which are not fields. 
Indeed, let $d>1$ be a divisor of~$m$ such that $r|(m/d)$. We have that any $x\in [\PPP(m)(F)]_d$ is not a field. It follows from above that \begin{align*}|\{x\in[\PPP(m)(F)]_d|H(x)\leq B\}|\hskip-1cm&\\&=|\{x\in[\PPP(m/d)(F)]|H((f^{*(m/d)}_v)_v)(x)\leq B\}|\\&\sim_{B\to\infty}\frac{d\cdot\tau((f^{*(m/d)}_v)_v}{(r-2)!m}B\log(B)^{r-2}.\end{align*}
\end{rem} 
\begin{rem}
\normalfont Suppose that~$F$ contains all~$m$-th roots of~$1$ (in particular if $4|m$ then $i\in F$, so Theorem \ref{licava} applies and gives the leading constant). One has that $\mu_m=\ZZ/m\ZZ$. The result of \ref{licava} has been established by Wright in \cite{Wright} finds the asymptotic behaviour for the number of abelian extensions. The proof there also gives the precise leading constant, however, we find it is difficult to compare it with our constant.
\end{rem}

\bibliography{bibliografija}
\bibliographystyle{acm}

\end{document}